\documentclass[a4paper,10pt]{book}

\usepackage[english]{babel}

\usepackage{amssymb,amsmath,amsthm}
\usepackage{amsfonts}
\usepackage{color}

\newtheorem{theorem}{Theorem}[chapter]
\newtheorem{lemma}[theorem]{Lemma}

\theoremstyle{definition}
\newtheorem{definition}[theorem]{Definition}

\newtheorem{proposition}[theorem]{Proposition}
\newtheorem{corollary}[theorem]{Corollary}
\theoremstyle{remark}
\newtheorem{remark}[theorem]{Remark}

\numberwithin{section}{chapter}
\numberwithin{equation}{chapter}

\newcommand{\mB}{\mathcal{B}}

\newcommand{\mE}{\mathcal{E}}
\newcommand{\mF}{\mathcal{F}}
\newcommand{\mG}{\mathcal{G}}
\newcommand{\mH}{\mathcal{H}}

\newcommand{\mL}{\mathcal{L}}

\newcommand{\mN}{\mathcal{N}}

\newcommand{\mR}{\mathcal{R}}
\newcommand{\mS}{\mathcal{S}}

\newcommand{\mZ}{\mathcal{Z}}
\newcommand{\Dom}{ \om \! \cdot \! \pa_\vphi}
\newcommand{\CS}{| {\mathbb S}^+|}

\renewcommand{\a}{\alpha}
\renewcommand{\b}{\beta}
\newcommand{\g}{\gamma}
\renewcommand{\d}{\delta}
\newcommand{\D}{\Delta}
\newcommand{\e}{\varepsilon}
\newcommand{\ph}{\varphi}
\newcommand{\vphi}{\varphi}

\newcommand{\Om}{\Omega}
\newcommand{\om}{\omega}
\newcommand{\p}{\pi}
\newcommand{\s}{\sigma}
\renewcommand{\t}{\tau}

\renewcommand{\t}{\tau }

\newcommand{\teta}{\theta}

\newcommand{\be}{\begin{equation}}
\newcommand{\ee}{\end{equation}}
\newcommand{\norma}{|\!\!|}

\newcommand{\ii}{{\mathrm i} }

\newcommand{\DC}{{\mathtt D \mathtt C}}

\newcommand{\tOm}{{\mathtt \Omega} }
\newcommand{\tLm}{{\mathtt \Lambda} }

\newcommand{\gr}{\nabla}

\newcommand{\intp}{\int_{0}^{2\p}}

\newcommand{\ov}{\overline}
\newcommand{\wtilde}{\widetilde}

\newcommand{\R}{\mathbb R}
\newcommand{\C}{\mathbb C}
\newcommand{\Z}{\mathbb Z}
\newcommand{\N}{\mathbb N}
\newcommand{\T}{\mathbb T}

\newcommand{\pa}{\partial}

\newcommand{\fracchi}{{\mathfrak{I}}}

\DeclareMathOperator{\sign}{sign}

\renewcommand{\a}{\alpha}
\renewcommand{\d}{\delta}
\renewcommand{\t}{\tau}

\newcommand{\odd}{\text{odd}}
\newcommand{\even}{\text{even}}

\def\ba{\begin{aligned}}
\def\ea{\end{aligned}}
\def\beginm{\begin{multline}}
\def\endm{\end{multline}}

\def\B{B}

\def\eps{\varepsilon}

\def\xR{\R} 

\begin{document}

\title{\textbf{Quasi-periodic  standing wave solutions of \\ 
gravity-capillary water waves}}

\date{}

\author{Massimiliano Berti
\footnote{SISSA,  Via Bonomea 265, 34136, Trieste, Italy, e.mail: berti@sissa.it}, 
Riccardo Montalto 
\footnote{University of Z\"urich, Winterthurerstrasse 190, 
CH-8057, Z\"urich, Switzerland,
e-mail: riccardo.montalto@math.uzh.ch}}

\maketitle

\newpage

\noindent
{\bf Abstract} We prove the existence and the linear stability of small amplitude time {\it quasi-periodic}  standing wave
solutions (i.e. periodic and even in the space variable $ x $) 
of a $ 2 $-dimensional ocean with infinite depth 
under the action of gravity and surface tension. Such an existence result is obtained for all the values of the surface tension belonging to a Borel set of asymptotically full Lebesgue measure.

\medskip

\noindent
{\bf Acknowledgements} This research was supported by PRIN 2012 ``Variational and perturbative aspects of nonlinear differential problems" and partially by the Swiss National Science Foundation. 



\tableofcontents




\chapter{Introduction and main result}\label{sec:main result}

In this paper we prove the existence of non trivial, small amplitude, {\it quasi-periodic} in time, {\it linearly stable}
gravity-capillary standing water waves of a $2$-d perfect, incompressible, irrotational  fluid with infinite depth, 
under periodic boundary conditions, and which occupies the free boundary region
$$
{\mathcal D}_\eta  := \big\{(x,y)\in \T \times \xR\,:\, y<\eta(t,x) \, , \quad \T := \R \slash (2 \pi \Z) 
 \big\} \, .
$$ 
 More precisely we find quasi-periodic in time solutions of the system
\begin{equation}\label{water}
\begin{cases}
\partial_t \Phi + \frac12 | \nabla \Phi |^2 + g \eta = \kappa \frac{\eta_{xx}}{(1 + \eta_x^2)^{3/2}} \ \ 
\qquad   {\rm at} \ y = \eta (x)  \cr
\Delta \Phi =0 \qquad \qquad \qquad \  \  \    \qquad  \qquad \ \qquad {\rm in} \ {\mathcal D}_{\eta}  \cr
\nabla \Phi \to 0  \qquad  \qquad \qquad \quad \qquad \qquad \qquad \,  {\rm as} \  
y \to - \infty  \cr
\partial_t \eta = \partial_y \Phi - \partial_x \eta \cdot
\partial_x \Phi \qquad \qquad \qquad \quad  {\rm at} \  y = \eta (x) 
\end{cases}
\end{equation}
where $ g $ is the acceleration of gravity, $ \kappa \in [\kappa_1, \kappa_2 ] $, $ \kappa_1 >  0 $,  is the surface tension coefficient  and 
\begin{equation*}
\frac{\eta_{xx}}{(1 + \eta_x^2)^{3/2}} 
= \partial_x \bigg( \frac{\eta_x}{\sqrt{1+ \eta_x^2}} \bigg) 
\end{equation*}
is the mean curvature of the free surface. 
The unknowns of the problem 
are the free surface $ y = \eta (x) $
and the  velocity potential $ \Phi : {\mathcal D}_\eta \to \R $, 
i.e. the irrotational  velocity field   
$ v =\nabla_{x,y} \Phi   $ 
of the fluid.
The first equation in \eqref{water}  is the Bernoulli condition according to which
 the jump of pressure across the free surface is proportional to the mean curvature. 
 The last equation in \eqref{water} 
expresses that the velocity of the free surface coincides with the one of
the fluid particles.

In the sequel we shall assume (with no  loss of generality) that the gravity constant $ g = 1 $.

\smallskip

Following Zakharov \cite{Zakharov1968} and Craig-Sulem \cite{CrSu}, the evolution problem \eqref{water} 
may be written as an infinite dimensional Hamiltonian system. 
At each time $ t \in \R $ the profile $ \eta ( t,  x ) $ of the fluid and the value 
$$
\psi(t,x)=\Phi(t,x,\eta(t,x))
$$
of the velocity potential $\Phi$ restricted to the free boundary
uniquely determine the velocity potential  $ \Phi $ in
the whole $ {\mathcal D}_\eta $, solving (at each $ t $) the elliptic problem (see e.g.  \cite{ABZ1}, \cite{LannesLivre})
\be
\begin{aligned} \label{BVP}
\Delta \Phi =0 \quad \text{in } {\mathcal D}_\eta , \quad  \Phi(x + 2 \pi, y) = \Phi (x, y) \, , \\
 \quad \Phi \arrowvert_{y=\eta}=\psi,\quad 
\nabla \Phi (x, y) \rightarrow 0 \text{ as } y\rightarrow -\infty \, . 
\end{aligned}
\ee
As proved in  \cite{Zakharov1968}, \cite{CrSu}, 
system \eqref{water} is then equivalent to the system
\begin{equation}\label{WW}
\left\{
\begin{aligned}
&\partial_t \eta = G(\eta)\psi,\\
&\partial_t\psi + \eta+  \frac{1}{2} \psi_x^2  
-\frac{1}{2} \frac{\bigl(G(\eta)\psi + \eta_x \psi_x \bigr)^2}{1+ \eta_x^2}
= \kappa \frac{\eta_{xx}}{(1 + \eta_x^2)^{3/2}} 
\end{aligned}
\right.
\end{equation}
where $ G(\eta) $ is the so-called Dirichlet--Neumann operator  defined by
\be\label{D-N}
G(\eta)\psi (x) :=
\sqrt{1+\eta_x^2}\,
\partial _n \Phi\arrowvert_{y=\eta(x)}
=(\partial_y \Phi)(x,\eta(x))-\eta_x (x) \, (\partial_x \Phi)(x,\eta(x))
\ee
(we denote by $ \eta_x $ the space derivative $ \partial_x \eta $.) The operator $ G(\eta) $ is linear in $ \psi $, 
self-adjoint with respect to the $ L^2 $ scalar product and semi positive definite, actually its Kernel 
are only the constants. It is well known since 
Calderon that the Dirichlet-Neumann operator is a {\it pseudo-differential} operator with principal symbol 
$ |D|$, actually $ G(\eta) - |D| \in OPS^{-\infty} $, see section \ref{subDN}. 

\smallskip

Furthermore  the equations \eqref{WW} are the Hamiltonian system (see \cite{Zakharov1968}, \cite{CrSu})
$$
\pa_t \eta = \nabla_\psi H (\eta, \psi) \, , \quad  \pa_t \psi = - \nabla_\eta H (\eta, \psi) 
$$
\be\label{HS}
\pa_t u  = J \nabla_u H (u)  \, , \quad u := \begin{pmatrix} 
\eta \\
\psi \\
\end{pmatrix} \, , \quad J := 
\begin{pmatrix} 
0 & {\rm Id} \\
- {\rm Id}  & 0 \\
\end{pmatrix} \, , 
\ee
where $ \nabla $ denotes the $ L^2 $-gradient,  
and  the Hamiltonian 
\be\label{Hamiltonian}
H(\eta, \psi ) := \frac12 (\psi, G(\eta) \psi)_{L^2(\T_x)}  +
\int_{\T} \frac{\eta^2}{2}  \, dx  +  \kappa \int_{\T}  \sqrt{1 + \eta_x^2} \, dx 
\ee
is the sum of the kinetic energy 
$$ 
K := \frac12 (\psi, G(\eta) \psi)_{L^2(\T_x)} = \frac12 \int_{{\mathcal D}_\eta} | \nabla \Phi |^2 (x,y) dx dy \, , 
$$ 
the potential energy  and the energy of the capillary forces (area surface integral)  expressed in terms of the variables 
$ (\eta, \psi ) $. 

The symplectic structure induced by  \eqref{HS} is the standard Darboux $ 2 $-form 
\be\label{2form tutto}
{\mathcal W}(u_1, u_2):= ( u_1, J u_2 )_{L^2(\T_x)} = ( \eta_1, \psi_2 )_{L^2 (\T_x)} - ( \psi_1, \eta_2 )_{L^2 (\T_x)}   
\ee
for all $ u_1 = (\eta_1, \psi_1) $, $ u_2 = (\eta_2, \psi_2) $.

\smallskip

The water-waves system \eqref{WW}-\eqref{HS} exhibits several symmetries. 
First of all, the mass $ \int_\T \eta \, dx $ is a prime integral of \eqref{WW}. 
Moreover 
$$
\pa_t \int_\T \psi \, dx  =  -  \int_{\T} \eta \, dx - \int_{\T} \nabla_\eta K \, dx =  -  \int_{\T} \eta \, dx   
$$
because $  \int_{\T} \nabla_\eta K \, dx  = 0 $.  
This follows because $ \R \ni c \mapsto K (c+ \eta, \psi ) $ is constant (the bottom of the ocean is at $ - \infty $)
and so $  0 = d_\eta K (\eta, \psi) [1] = (\nabla_\eta K, 1 )_{L^2(\T)} $.  
As a  consequence the subspace   
\be\label{eta0psi0}
\int_{\T} \eta \, dx = \int_{\T} \psi \, dx = 0   
\ee
is invariant under the evolution of \eqref{WW} and we shall restrict to solutions satisfying \eqref{eta0psi0}.
 
In addition, the subspace of functions  which are even in $ x $, 
\be\label{even in x}
\eta (x) = \eta (-x) \, , \quad \psi (x) = \psi (- x )  \, , 
\ee
is invariant under \eqref{WW}. 
Thanks to this property and  \eqref{eta0psi0},  we shall restrict  $ (\eta, \psi )$ to 
the phase space of $ 2 \pi $-periodic functions  which admit the Fourier expansion
\be\label{phase space}
\eta ( x)  = {\mathop \sum}_{j \geq 1} \eta_j  \cos (j x ) \, , \quad \psi ( x)  = {\mathop \sum}_{j \geq 1} \psi_j  \cos (j x )  \, . 
\ee 
In this case also the velocity potential $ \Phi(x,y) $ is even  and $ 2 \pi $-periodic in $ x $ and so the 
$ x$-component  of the velocity field $ v = (\Phi_x, \Phi_y) $ vanishes  at $ x = k \pi $, $ \forall k \in \Z $. 
Hence   there is no flux of fluid through the lines $ x = k \pi $, $ k \in \Z $, and 
a solution of \eqref{WW} satisfying \eqref{phase space} describes 
the motion of a liquid confined  between two walls. 

\smallskip

Another important symmetry 
of the capillary water waves system  is reversibility, namely the equations
 \eqref{WW}-\eqref{HS} are reversible with respect to the involution  $ \rho : (\eta, \psi) \mapsto (\eta, - \psi) $,
 or, equivalently, the Hamiltonian  is even in $ \psi $: 
\be\label{defS}
H \circ \rho = H \, , \quad H( \eta, \psi) = H ( \eta, - \psi ) \, , \quad \rho : (\eta, \psi) \mapsto (\eta, - \psi) \, .
\ee
As a consequence it is   natural to look  for solutions of \eqref{WW} satisfying 
\be\label{odd-even}
u(-t ) = \rho u(t) \, , \quad i.e. \quad \eta (-t, x) = \eta(t, x) \, , \ \psi(-t,x ) = - \psi (t,x ) \, , \ \forall t, x \in \R \, , 
\ee
namely $ \eta  $ is even  in time and $ \psi $ is odd in time. 
Solutions of the water waves equations \eqref{WW} satisfying 
\eqref{phase space} and \eqref{odd-even} are called gravity-capillary {\it standing water waves}. 

This is a small divisors problem. 
Existence of small amplitude time periodic pure gravity (without surface tension) standing wave solutions  has been proved 
by Iooss, Plotnikov, Toland in \cite{IPT},  see also \cite{IP-SW2}, \cite{IP-SW1}, and in \cite{PlTo}
in finite depth. Existence of time periodic  gravity-capillary standing wave solutions 
has been recently proved by Alazard-Baldi \cite{AB}. 
The above results are proved via a Lyapunov Schmidt decomposition combined with a Nash-Moser iterative scheme.

In this paper we  extend the latter result proving the existence of time {\it quasi-periodic}  gravity-capillary
standing wave solutions of \eqref{WW}, see Theorem \ref{thm:main0}, as well as 
their linear stability. 
The reducibility of the linearized equations at the quasi-periodic solutions is not only an interesting dynamical information 
but it is also the key for the existence proof in  Theorem \ref{thm:main0}. 

We also mention that existence of small amplitude  $ 2 $-d traveling 
gravity water wave solutions dates back to Levi-Civita \cite{LC} 
(standing waves are not traveling because they are even in space, see \eqref{even in x}). 
Existence of small amplitude $ 3 $-d traveling gravity-capillary water wave solutions 
with space periodic boundary conditions  has been proved by Craig-Nicholls 
\cite{CN} (it is not a small divisor problem) 
and by Iooss-Plotinikov \cite{IP-Mem-2009}-\cite{IP2} in the case of zero surface tension
 (in such a case it is a small divisor problem).

\smallskip

Existence of quasi-periodic solutions of PDEs (that we shall call in a broad sense KAM theory) with 
unbounded  perturbations (i.e. the nonlinearity contains derivatives) has been developed 
by Kuksin \cite{Kuksin-Oxford} for KdV, see also Kappeler-P\"oschel \cite{KaP}, by 
Liu-Yuan \cite{Liu-Yuan}, Zhang-Gao-Yuan \cite{Zhang-Gao-Yuan} for
derivative NLS, by Berti-Biasco-Procesi 
\cite{Berti-Biasco-Procesi-Ham-DNLW}-\cite{Berti-Biasco-Procesi-rev-DNLW} for derivative NLW.
All these previous results still refer 
 to semilinear perturbations, i.e. the order of the derivatives
in the nonlinearity is strictly lower than the order of the constant coefficient (integrable) linear differential 
operator. 

For quasi-linear (either fully nonlinear) nonlinearities 
the first KAM results have been recently proved by Baldi-Berti-Montalto  in \cite{BBM-Airy}, 
\cite{BBM-auto}, \cite{BBM-mKdV} (see also \cite{BBMLincei}, \cite{BBMCRAS}) for perturbations of 
Airy, KdV and mKdV equations. 
These techniques  have been extended by Feola-Procesi \cite{FP} 
 for  quasi-linear perturbations of Schr\"odinger equations and by Montalto \cite{riccardo-kirchhoff} for the Kirchhoff equation.

The gravity-capillary water waves system \eqref{WW} is indeed a quasi-linear PDE. 
In suitable complex coordinates it can
be written in the symmetric form
$ {\mathtt u}_t = \ii T(D) {\mathtt u} + N({\mathtt u}, \bar {\mathtt u} ) $, $ {\mathtt u} \in \C $,
 where $ T(D) :=  |D|^{1/2} (1 - \kappa \pa_{xx})^{1/2}  $
is the Fourier multiplier which describes the linear dispersion relation of the water waves equations linearized at $ (\eta, \psi) = 0 $ 
(see \eqref{Lom}-\eqref{linear frequencies}), and the nonlinearity 
$ N ({\mathtt u}, \bar {\mathtt u} ) $ depends on the highest order term $ | D |^{3/2} {\mathtt u} $ as well, see 
sections \eqref{sec:linearized operator}-\eqref{sec:changes} for the complex form 
 of the linearized system. 
 
We have not the space to 
report the huge literature concerning KAM theory for semilinear PDEs  in one and also higher space dimension, 
for which we refer to \cite{Kuksin-Oxford}, 
\cite{B5},  \cite{EK}, \cite{BB14}, \cite{BCP}.
 
\smallskip

Let us present rigorously our main result.
As already said we look for small amplitude quasi-periodic 
solutions of \eqref{WW}. It is therefore  of main importance the dynamics of the  
system obtained linearizing \eqref{WW} at the equilibrium $(\eta, \psi) = (0,0)$ (flat ocean and fluid at rest), namely 
\begin{equation} \label{Lom}
\left\{
\begin{aligned}
&\partial_t \eta = G(0)\psi,\\
&\partial_t\psi + \eta 
= \kappa  \eta_{xx} 
\end{aligned}
\right.
\end{equation}
where $ G(0) = |D_x|  $ is the Dirichlet-Neumann operator for the flat surface $ \eta = 0  $, namely  
$$
|D_x| \cos(jx) = |j| \cos(jx), \quad 
|D_x| \sin(jx) = |j| \sin(jx) \, ,  \  
\forall j \in \Z \, .
$$  
In compact Hamiltonian form,  the system \eqref{Lom} reads 
\begin{equation}\label{definizione Omega}
\pa_t u  = J \Omega u  \, , \quad 
\Om  :=  \begin{pmatrix} 
1  -  \kappa  \pa_{xx}  & 0 \\
0 &  G(0) \\
\end{pmatrix},  
\end{equation}
which is  the Hamiltonian system generated by the quadratic Hamiltonian (see \eqref{Hamiltonian})
\be\label{Hamiltonian linear}
H_{L}  := \frac12 ( u, \Omega u )_{L^2 (\T_x)} = 
\frac12 (\psi, G(0) \psi)_{L^2(\T_x)}    +
 \frac12 \int_{\T} \big( \eta^2  + \kappa \eta_x^2 \big) \, dx \, . 
\ee
The standing wave solutions of  the linear system \eqref{Lom}, i.e. \eqref{definizione Omega}, 
are 
\begin{equation}  \label{eta psi sin cos series}
\begin{aligned} 
\eta(t,x) & = {\mathop \sum}_{ j \geq 1}  a_j  \cos (\om_j t)  \cos(jx), \\
\psi(t,x) & = - {\mathop \sum}_{ j \geq 1}  a_j j^{-1} \om_j    \sin (\omega_j t)  \cos(jx) \, , 
\end{aligned}
\end{equation} 
$ a_j \in \R $, 
with linear frequencies of oscillations
\be\label{linear frequencies}
\om_j := \om_j (\kappa) := \sqrt{ j (1+ \kappa j^2) } \, , \quad j \geq 1 \, . 
\ee
The main result of the paper proves that most of the standing wave solutions \eqref{eta psi sin cos series} 
of the linear system \eqref{Lom}
can be continued to standing wave solutions of the nonlinear water-waves Hamiltonian system \eqref{WW} for most 
values of the surface tension parameter $ \kappa  \in [\kappa_1, \kappa_2] $. 
More precisely, fix an arbitrary finite subset $ {\mathbb S}^+ \subset \N^+ := \{1,2, \ldots \} $ (called ``tangential sites") and consider the 
linear standing wave solutions
(of \eqref{Lom})
\begin{equation}  \label{eta psi sin cos series-QP}
\begin{aligned} 
\eta(t,x) & = \sum_{ j \in {\mathbb S}^+}  \sqrt{\xi_j} \cos (\om_j t)  \cos(jx), \\
\psi(t,x) & = - \sum_{ j \in {\mathbb S}^+} \sqrt{\xi_j} j^{-1}  \om_j    \sin (\omega_j t)  \cos(jx) \, , \  \xi_j > 0 \, , 
\end{aligned}
\end{equation}
which are Fourier supported in $ {\mathbb S}^+ $. 
In Theorem \ref{thm:main0} below we prove the existence of quasi-periodic solutions
$ u (\tilde \om t, x) = (\eta, \psi)( \tilde \om t, x) $ of  \eqref{WW}, 
with frequency $ \tilde \om := ( \tilde \om_j )_{j \in {\mathbb S}^+} $ (to be determined),  close to  the 
solutions \eqref{eta psi sin cos series-QP} of \eqref{Lom}, 
for most values of the surface tension parameter $ \kappa \in [ \kappa_1, \kappa_2 ] $.

Let $ \nu := | {\mathbb S}^+| $ denote the cardinality of $ {\mathbb S}^+ $. 
The function $ u (\vphi, x) = (\eta, \psi)(\vphi, x)  $, $ \vphi \in \T^\nu $,  belongs 
to the Sobolev spaces of $ (2\pi)^{\nu+1}$-periodic real functions  
$$
H^s(\T^{\nu+1}, \R^2) := \big\{ u = (\eta, \psi) : \eta, \psi \in H^s \big\}
$$
\begin{equation} \label{unified norm}
\begin{aligned}
H^s := H^s(\T^{\nu+1}, \R)
= \Big\{ & f  = \sum_{(\ell,j) \in \Z^{\nu+1}} \widehat f_{\ell, j} \, e^{\ii(\ell \cdot \ph + jx)} : \ \\
& \| f \|_s^2 := \sum_{(\ell,j) \in \Z^{\nu+1}} | \widehat f_{\ell, j}|^2 \langle \ell,j \rangle^{2s} < + \infty \Big\}
\end{aligned}
\end{equation}
where  $\langle \ell,j \rangle := \max \{ 1, |\ell|, |j| \} $ with $ | \ell | := \max_{i = 1, \ldots , \nu} |\ell_i|$. For 
\be\label{def:s0}
s \geq s_0 := \Big[ \frac{\nu +1}{2} \Big] +1 \in \N  
\ee 
the Sobolev spaces  $ H^s \subset L^\infty ( \T^{\nu+1})$ are an algebra with respect to the product of functions.

\begin{theorem} \label{thm:main0}  {\bf (KAM for capillary-gravity water waves)}
For every choice of finitely many tangential sites $ {\mathbb S}^+ \subset \N^+   $,
there exists $ \bar s >  s_0 $,  $ \e_0 \in (0,1) $ such that for every $ |\xi |   \leq \e_0^2  $, $  \xi := (\xi_j)_{j \in {\mathbb S}^+ } $, $\xi_j > 0$ for any $j \in \mathbb S^+$,  there exists 
a Borel set  $ \mG \subset [\kappa_1,\kappa_2] $ 
with asymptotically full measure as $ \xi \to 0 $, i.e. 
$$
\lim_{\xi \to 0} | \mG |  = \kappa_2- \kappa_1 \, ,  
$$
such that, for any
surface tension coefficient $ \kappa \in \mG $,  the capillary-gravity system \eqref{WW}
has a time quasi-periodic  standing wave solution 
$$ 
u( \tilde \om t, x ) = (\eta ( \tilde \om t, x), \psi ( \tilde \om t, x) ) \, , 
$$
with Sobolev regularity $ (\eta, \psi)   \in H^{\bar s} ( \T^\nu \times \T, \R^2) $,   
of the form
\begin{equation} \label{QP:soluz}
\begin{aligned}
\eta( \tilde \om t,x) & = {\mathop \sum}_{ j \in {\mathbb S}^+}  \sqrt{\xi_j} \cos ({\tilde \om}_j t)  \cos(jx) + 
r_1 ( \tilde \om t,x ) , \\  
\psi(\tilde \om t ,x)   & = - {\mathop \sum}_{ j \in {\mathbb S}^+} \sqrt{\xi_j} j^{-1} \om_j   \sin ({\tilde \omega}_j t)  \cos(jx)+ 
r_2 ( \tilde \om t,x ) 
\end{aligned}
\end{equation}
with a diophantine frequency vector $ \tilde \omega := \tilde \omega (\kappa, \xi ) \in \R^\nu $ satisfying 
$ {\tilde \omega}_j - \omega_j (\kappa) \to 0 $, $ j \in {\mathbb S}^+ $,  as $ \xi \to 0 $, and 
the functions $ r_1 ( \vphi ,x ), r_2(\vphi, x)$ are   $o( \sqrt{|\xi |} )$-small in $ H^{\bar s} ( \T^\nu \times \T, \R) $, that is $\| r_j \|_{\bar s}/ \sqrt{|\xi|}$ tends to $0$ as $|\xi| \to 0$ for $j = 1, 2$. In addition these quasi-periodic solutions are 
 linearly stable. 
\end{theorem}

Theorem \ref{thm:main0}  follows by Theorems \ref{MAINTHEOREM} and \ref{Teorema stima in misura}. 
This result has been announced in \cite{BM16}. 
Let us make some comments.

\begin{enumerate}
\item 
No global in time existence results concerning the initial value problem of the water waves equations 
\eqref{WW} under {\it periodic} boundary conditions are known so far. 
The present Nash-Moser-KAM iterative procedure selects many values of the  surface tension parameter
$ \kappa \in [\kappa_1, \kappa_2] $  which give rise
to the quasi-periodic solutions  \eqref{QP:soluz}, which are defined for all times.
Clearly, by a Fubini-type argument it also results that, for most values of $ \kappa \in [\kappa_1, \kappa_2] $,
there exist quasi-periodic solutions of \eqref{WW} for most values of the amplitudes $ | \xi | \leq \e_0^2 $. 
The fact that we find quasi-periodic solutions restricting to a proper subset of 
parameters is not a technical issue.
The gravity-capillary water-waves equations \eqref{WW} are not expected to be 
integrable  (albeit a rigorous proof is still lacking): yet  the third order Birkhoff normal form possesses 
multiple resonant triads (Wilton ripples),  
see Craig-Sulem \cite{CS15}. 
\item
In the proof of Theorem \ref{thm:main0} all the estimates depend on the surface tension coefficient 
$ \kappa >  0 $ and the result 
does not hold at the limit of 
zero surface tension  $ \kappa \to 0 $. 
Because of  capillarity the  
linear frequencies \eqref{linear frequencies}
grow asymptotically  $ \sim \sqrt{\kappa} j^{3/2} $ as $ j \to + \infty $. 
Without surface tension  the linear frequencies grow asymptotically as $ \sim  j^{1/2} $
and a different proof is required. 
\item
The quasi-periodic solutions  \eqref {QP:soluz} are mainly supported in Fourier space on the tangential sites 
$ {\mathbb  S}^+ $.
The dynamics of the water waves equations \eqref{WW} restricted 
to the symplectic  subspaces 
\begin{equation}\label{splitting S-S-bot}
\begin{aligned}
& H_{{\mathbb S}^+} := \Big\{ v = \sum_{j \in {\mathbb S}^+} \begin{pmatrix} 
 \eta_j  \\
\psi_j \\
\end{pmatrix} \cos (jx)   \Big\} \, , \\
& H_{{\mathbb S}^+}^\bot := \Big\{ 
z = \sum_{j \in \N \setminus {\mathbb S}^+} \begin{pmatrix} 
 \eta_j  \\
\psi_j \\
\end{pmatrix} \cos (jx) \in H^1_0(\T_x)  \Big\} ,
\end{aligned}
\end{equation}
is quite different. We call $ v \in H_{{\mathbb S}^+} $  the {\it tangential} variable and $ z \in H_{{\mathbb S}^+}^\bot $ the {\it normal} one. 
On the finite dimensional subspace $ H_{{\mathbb S}^+} $
we describe the dynamics  by introducing  the action-angle variables $ (\theta, I) \in  \T^\nu \times \R^\nu $, see \eqref{Ac-An}. 

This is a difference with respect to the previous papers \cite{PlTo}, \cite{IP-SW2}, \cite{IP-SW1}, \cite{IP-Mem-2009}, 
\cite{IP2}, \cite{IPT}, \cite{AB},
that follow the Lyapunov-Schmidt decomposition.  
The present formulation 
enables, among other advantages,  to prove the linear stability of the quasi-periodic solutions.
\item
{\bf Linear stability.}
The quasi-periodic solutions $ u( \tilde \om t) = (\eta ( \tilde \om t), \psi ( \tilde \om t) ) $ found in Theorem \ref{thm:main0}  are linearly stable.
This  is not only a dynamically relevant information 
but also an essential ingredient of the existence proof 
(it is not necessary 
for time periodic  solutions as in \cite{AB},  \cite{IP-SW2}, \cite{IP-SW1}, \cite{IPT}).
Let us state precisely the result. Around each invariant torus 
there exist symplectic coordinates 
$$ 
(\phi, y, w) = (\phi, y, \eta, \psi) \in \T^\nu \times \R^\nu \times  H_{{\mathbb S}^+}^\bot   
$$ 
(see \eqref{trasformazione modificata simplettica} and \cite{BB13}) 
in which the water waves Hamiltonian  reads
\begin{equation}\label{weak-KAM-normal-form}
\begin{aligned} 
 \om \cdot y  & + \frac12 K_{2 0}(\phi) y \cdot y +  \big( K_{11}(\phi) y , w \big)_{L^2(\T_x)}  \\
&  + \frac12 \big(K_{02}(\phi) w , w \big)_{L^2(\T_x)} + K_{\geq 3}(\phi, y, w)  
\end{aligned}
\end{equation}
where $ K_{\geq 3} $ collects the terms at least cubic in the variables $ (y, w )$
(see  \eqref{KHG} and note that at a solution $ \partial_\phi K_{00} = 0 $, $ K_{10} = \omega $, 
$ K_{01} = 0 $ by Lemma \ref{coefficienti nuovi}). 
In these coordinates the quasi-periodic solution reads 
$ t \mapsto (\om t , 0, 0 ) $ (for simplicity we denote the frequency $ \tilde \om $ of the quasi-periodic solution by $ \om $) and  the corresponding  linearized water waves equations are
\begin{equation}\label{linear-torus-new coordinates}
\begin{cases}
\dot{ \phi} = K_{20}(\omega t)[ y] + K_{11}^T(\omega t)[ w] \\
\dot{ y} = 0 \\
\dot{ w} = J K_{02}(\omega t)[ w] + J K_{11}(\omega t)[ y]\, .
\end{cases}
\end{equation}
Thus the actions $ y (t) = y(0) $ do not evolve in time and 
the third equation reduces to  the PDE
\begin{equation}\label{San pietroburgo modi normali}
\dot{w} = J K_{02}(\omega t)[ w] + J K_{11}(\omega t)[ y (0)] \, .
\end{equation}
The self-adjoint operator $  K_{02} (\om t) $ (defined in \eqref{KHG}) 
turns out to be the restriction to $ H_{{\mathbb S}^+}^\bot $ of the linearized water-waves  operator 
$ \pa_u \nabla H (u (\om t ))  $, explicitly computed in \eqref{linearized vero}, 
 up to a finite dimensional remainder, see Lemma \ref{thm:Lin+FBR}.

In
sections \ref{linearizzato siti normali} and  \ref{sec: reducibility} we prove the existence of 
bounded and invertible ``symmetrizer" maps, see \eqref{defW1W2}, such that $ \forall \vphi \in \T^\nu $, $m = 1, 2$
\begin{align}\label{W12-ben-poste}
& {\bf W}_{m, \infty}(\vphi)  \! : \!  
 H^s(\T_x, \C^2)  \cap H_{{\mathbb S}^+}^\bot
\! \to \! \Big( H^s(\T_x, \R) \times H^{s - \frac12}(\T_x, \R) \Big) \cap H_{{\mathbb S}^+}^\bot , \\
& \label{W12-ben-poste-inverse}
{\bf W}_{m, \infty}^{-1}(\vphi) \! : \!
\Big( H^s(\T_x, \R) \times H^{s - \frac12}(\T_x, \R) \Big) \cap H_{{\mathbb S}^+}^\bot \! \to \!
 H^s(\T_x, \C^2)  \cap H_{{\mathbb S}^+}^\bot ,  
\end{align}
and, under the change of variables 
$$
 w = (\eta, \psi) = {\bf W}_{1, \infty} (\omega t) w_\infty \, ,  \quad
 w_\infty = ( {\mathtt w}_\infty, \overline {\mathtt w}_\infty ) \, ,
 $$
the equation \eqref{San pietroburgo modi normali} transforms into the diagonal system
\begin{equation}\label{san pietroburgo siti normali ridotta}
\begin{aligned}
& \partial_t{ w_\infty} = - \ii {\bf D}_\infty  w_\infty + f_\infty(\omega t)\,, \\ 
& f_\infty(\omega t) := 
{\bf W}_{2, \infty}(\vphi) (\omega t)^{- 1} J K_{11}(\omega t)[ y(0)] =
\begin{pmatrix}
{\mathtt f}_\infty (\omega t) \\ 
\overline {\mathtt f}_\infty (\omega t) 
\end{pmatrix}
\end{aligned}
\end{equation}
where, denoting ${\mathbb S}_0 := {\mathbb S}^+ \cup (- {\mathbb S}^+) \cup \{ 0 \} \subseteq \Z$,
\be\label{def: D-infty}
{\bf D}_\infty := \begin{pmatrix}
D_\infty & 0 \\
0 & - D_\infty
\end{pmatrix}\,, \quad D_\infty := {\rm diag}_{j \in {\mathbb S}_0^c} \{ \mu_j^\infty\}\,, \quad \mu_j^\infty \in \R \, , 
\ee
is a Fourier multiplier operator of the form (see \eqref{autovalori finali riccardo}) 
\be\label{Floquet-exp}
  \mu_j^{\infty} := \mathtt m_3^{\infty} \sqrt{|j|(1 + \kappa j^2)} + {\mathtt m}_1^{\infty} |j|^{\frac12} + r_j^\infty \, , \ j  \in {\mathbb S}_0^c \, , 
  \quad r_j^\infty = r_{- j}^\infty\,, 
 \ee
 where, for some $ {\mathtt a} >  0 $, 
$$
\mathtt m_3^{\infty} = 1 + O(\e^{\mathtt a}) \, , \quad {\mathtt m}_1^{\infty} = O(\e^{\mathtt a} ) \, , \quad
\sup_{j \in {\mathbb S}_0^c} |  r_j^{\infty}| = O(\e^{\mathtt a} ) \, .
$$
Actually by \eqref{autovalori infiniti}-\eqref{stime autovalori infiniti} and \eqref{relazione tau k0} we also have a 
control of the derivatives of $ \mathtt m_3^{\infty}  $, $ {\mathtt m}_1^{\infty} $ and 
$ r_j^{\infty} $ with respect to $ (\omega, \kappa)$.
The $  \ii \mu_j^{\infty} $ are the {\it Floquet exponents} of the quasi-periodic solution. 
The second equation of system  \eqref{san pietroburgo siti normali ridotta} is actually the complex conjugated of the first one,
and  \eqref{san pietroburgo siti normali ridotta} reduces to the infinitely many decoupled  scalar equations
$$
\partial_t {\mathtt w}_{\infty, j} = - \ii \mu_j^\infty {\mathtt w}_{\infty, j} + {\mathtt f}_{\infty, j} (\omega t)\,, \quad \forall j \in {\mathbb S}_0^c \, . 
$$
By variation of constants the solutions are  
\be\label{soluz-lin-forc}
\begin{aligned}
& {\mathtt w}_{\infty, j} (t) = c_j e^{ - \ii \mu_j^\infty t} + {\mathtt v}_{\infty, j} (t)  \qquad {\rm where} \\ 
& {\mathtt v}_{\infty, j} (t)  := \sum_{\ell \in \Z^\nu} \frac{ {\mathtt f}_{\infty, j, \ell } \, e^{\ii \om \cdot \ell t } }{ \ii (\om \cdot \ell + \mu_j^\infty) } \,, \quad \forall j \in {\mathbb S}^c_0 \, .
\end{aligned}
\ee
Note that the first Melnikov conditions  \eqref{Cantor set infinito riccardo} hold at a solution so that 
$ {\mathtt v}_{\infty, j} (t) $ in  \eqref{soluz-lin-forc} is well defined.
Moreover \eqref{W12-ben-poste} implies 
$$ 
\| f_\infty(\omega t)\|_{H^s_x \times H^s_x} \leq C |y (0)| \, . 
$$ 
As a consequence  the Sobolev norm of the solution  of \eqref{san pietroburgo siti normali ridotta} 
 with initial condition $ w_{\infty} (0)  \in H^{{\mathfrak s}_0} (\T_x) $,
  for some $ s_0 < {\mathfrak s}_0 < s $ (in a suitable range of values),  
 satisfies 
$$
\|  w_\infty(t) \|_{H^{{\mathfrak s}_0}_x \times H^{{\mathfrak s}_0}_x}   
\leq C(s) (| y (0) | + \| w_\infty (0) \|_{H^{{\mathfrak s}_0}_x \times H^{{\mathfrak s}_0}_x})\, ,  
$$
and, for all $ t \in \R $, using \eqref{W12-ben-poste}, \eqref{W12-ben-poste-inverse}, we get 
$$
\| (\eta, \psi)(t) \|_{H^{{\mathfrak s}_0}_x \times H^{ {\mathfrak s}_0 -\frac12}_x} \leq C
 \| (\eta(0),\psi(0))  \|_{H^{{\mathfrak s}_0}_x \times H^{ {\mathfrak s}_0 - \frac12}_x} 
$$
which proves the linear stability of the torus. Note that  the profile $\eta \in  $ $ H^{{\mathfrak s}_0} (\T_x) $ is more regular than 
the velocity potential $ \psi \in H^{ {\mathfrak s}_0 - \frac12}(\T_x)  $, as it is expected in presence of  surface tension, see \cite{ABZ1}. 

Clearly a crucial  point is 
the diagonalization of \eqref{San pietroburgo modi normali} into \eqref{def: D-infty}.  
With respect to \cite{AB} this requires to analyze more in detail the pseudo-differential  nature of the operators obtained after each conjugation
and to implement a KAM scheme 
with second order Melnikov non-resonance conditions, as we shall explain in detail below. 
\item
{\it Hamiltonian and reversible structure.} 
It is well known that the existence of quasi-periodic motions is possible just for 
systems with some algebraic structure
which excludes ``secular motions" and friction phenomena. 
The most common ones are the Hamiltonian and the reversible structure.
The water-waves system  \eqref{WW} exhibits  both of them and 
we shall use both. 
The Hamiltonian structure is used in particular in section \ref{costruzione dell'inverso approssimato}
to introduce the symplectic coordinates $ (\phi, y, w) $ in \eqref{trasformazione modificata simplettica} 
adapted to an approximately-invariant torus. 
On the other hand, for solving the
second equation of the linear system \eqref{operatore inverso approssimato proiettato} we use reversibility
(we could exploit just the Hamiltonian structure as done in \cite{BBM-auto}-\cite{BBM-mKdV}, \cite{BB13}-\cite{BB14}). 
Moreover the transformations $ {\bf W}_{1, \infty} $, $ {\bf W}_{2, \infty} $ 
which reduce the linearized operator to constant coefficients preserve
the reversible structure
(it is  slightly simpler 
than to preserve 
the Hamiltonian one).  Reversibility  implies  that several averaged vector fields
are zero, for example a constant coefficient 
operator of the form $ h \mapsto a \pa_x h $, $ a \in \R $, is not compatible with the 
reversible structure of the water waves,  and therefore it is zero. This leads to the 
asymptotic expansion of the Floquet exponents $ \ii \mu_j^\infty $ with $ \mu_j^\infty  $ as in \eqref{Floquet-exp}, 
in particular to the fact that they are purely imaginary. The linear stability of the  quasi-periodic standing wave
solutions of Theorem \ref{thm:main0} is a consequence of the  
reversible structure of the water waves equations. 
\end{enumerate}

We  prove the existence of quasi-periodic solutions by a Nash-Moser iterative scheme in Sobolev spaces
formulated as a `Th\'eor\'eme de conjugaison hypoth\'etique" \'a la Herman (section \ref{sec:conjugation-hyp}). 
In order to perform 
effective measure estimates in the surface tension parameter $ \kappa \in [\kappa_1, \kappa_2 ] $ (section \ref{sec:measure})
we use  degenerate KAM theory for PDEs (section \ref{sec:degenerate KAM}).
For the convergence of the Nash-Moser scheme (section \ref{sec:NM}) 
it is sufficient to have an ``almost approximate" inverse of the linearized operators at each step of the iteration. 
 We use the adjectives ``almost" and ``approximate" in the following sense.  
An  ``approximate" inverse  is an operator that is an exact inverse at an  exact invariant torus, 
following the terminology of Zehnder \cite{Z1}.
The adjective ``almost" refers to the fact that at the $ n $-th step of the Nash-Moser iteration we 
shall require only finitely many non-resonance conditions of diophantine type (ultraviolet cut-off)
and therefore remain terms which are Fourier supported on high frequencies of magnitude larger than $ c N_n $ 
and thus can be estimated as $ O( N_n^{-a} ) $ for some $a > 0$ (in suitable norms).
We follow (section \ref{costruzione dell'inverso approssimato}) the scheme  
proposed in \cite{BB13}-\cite{BB14}, and implemented in \cite{BBM-auto}-\cite{BBM-mKdV}, 
 which reduces the problem to ``almost approximately" invert the linearized operator restricted to  the normal directions. 
The  crucial PDE analysis is the 
reduction in sections \ref{linearizzato siti normali}-\ref{sec: reducibility}  of the 
linearized operator to constant coefficients. 

\section{Ideas of proof}

Let us present more in details some key ideas of the paper. 

 \begin{enumerate}
\item {\it Bifurcation analysis and Degenerate KAM theory. } 
A first key observation 
 is that, for most values of the surface tension parameter $ \kappa \in [ \kappa_1, \kappa_2 ] $, 
the unperturbed  linear frequencies \eqref{linear frequencies} are diophantine 
and satisfy also first and second order Melnikov non-resonance conditions. More precisely 
the unperturbed tangential frequency vector $ \vec \om (\kappa) := (\om_j (\kappa))_{j \in {\mathbb S}^+} $ 
satisfies 
$$
| \vec \om (\kappa) \cdot \ell  | \geq \g \langle \ell \rangle^{-\tau}, \ \   \forall \ell \in \Z^{\nu} \setminus \{ 0 \}, \ \ 
\langle \ell \rangle := {\rm max}\{ 1, |\ell| \} \, ,  
$$
and it is non-resonant with
the  normal  frequencies 
$$ 
\vec \Om (\kappa)  := ( \Om_j (\kappa ) )_{j \in \N^+ \setminus {\mathbb S}^+} = 
 ( \om_j (\kappa ) )_{j \in \N^+ \setminus {\mathbb S}^+},  
 $$ 
 i.e.   
\begin{align*}
&  
\qquad \qquad  | \vec \omega (\kappa)  \cdot \ell  +  \Om_j (\kappa )  |  \geq  \gamma j^{\frac32} \langle \ell  \rangle^{- \tau}, \, 
 \forall \ell   \in \Z^\nu, \, j \in \N^+ \setminus {\mathbb S}^+ \, ,   \\
& \qquad \qquad  | \vec \omega (\kappa) \cdot \ell  + 
 \Omega_j (\kappa ) \pm  \Omega_{j'} (\kappa ) |  \geq 
\gamma | j^{\frac32} \pm  j'^{\frac32}| \langle \ell  \rangle^{- \tau}, 
 \forall \ell   \in \Z^\nu,j, j' \in \N^+ \setminus {\mathbb S}^+  \, .   
\end{align*}
This is a problem of diophantine approximation on submanifolds as in \cite{Py}. It can be solved 
 by degenerate KAM theory (explained below) exploiting that 
the linear frequencies $  \kappa \mapsto \om_j (\kappa) $ 
are {\it analytic}, 
simple,  grow asymptotically as  $ j^{3/2} $ and are {\it non-degenerate} in the sense of 
Bambusi-Berti-Magistrelli \cite{BaBM}
(another proof can be given by  the tools  of subanalytic geometry in Delort-Szeftel \cite{DS}). 
For such values of $  \kappa \in $ $ [\kappa_1, \kappa_2 ]  $, the  solutions \eqref{eta psi sin cos series-QP}  of the linear equation \eqref{Lom}
are already sufficiently good approximate quasi-periodic solutions of  the nonlinear water waves system \eqref{WW}. 
Since the parameter space $ [\kappa_1, \kappa_2 ] $ is fixed, 
the  small divisor constant $ \g $ can be taken $ \gamma  = o( \e^a) $ with $ a > 0 $ small as needed,
see \eqref{relazione tau k0}. As a consequence for proving the continuation of \eqref{eta psi sin cos series-QP} 
to solutions of the nonlinear water waves system \eqref{WW},  
all the terms  which are at least quadratic in  \eqref{WW}  are yet perturbative
(in \eqref{WW-riscalato} it is sufficient to regard the vector field $\e X_{P_{\e}} $  
as a perturbation of the linear vector field  $ J \Omega $).

Actually along the Nash-Moser-KAM iteration we need to verify that the perturbed frequencies 
are diophantine and satisfy first and second order Melnikov 
non-resonance conditions.  It is actually for that
we find convenient to develop degenerate KAM theory as in \cite{BaBM} 
and we formulate the problem as a 
Th\'eor\'eme de conjugaison hypoth\'etique {\`a} la  Nash-Moser as we explain below. 
\item{\it A Nash-Moser Th\'eor\'eme de conjugaison hypoth\'etique.}
The  expected quasi-periodic solutions of the autonomous Hamiltonian system \eqref{WW}
will have  shifted frequencies $ \tilde \om_j $ -to be found- close to the linear frequencies $  \om_j (\kappa) $ 
in \eqref{linear frequencies}, which depend on the nonlinearity and the amplitudes $ \xi_j $.
Since the Melnikov non-resonance conditions are naturally imposed on $ \om $, 
it is convenient to use the functional setting formulation 
of Theorem \ref{MAINTHEOREM} where the {\it parameters} are the {\it frequencies}  $ \om \in \R^\nu $ and the {\it surface tension} $\kappa \in [\kappa_1, \kappa_2]$   
and we introduce a counter term $ \a \in \R^\nu $ in the family of Hamiltonians $ H_\a $  defined in \eqref{H alpha}.

Then the goal is 
to prove that,  for $ \e $ small enough, for ``most" parameters $ (\om, \kappa) \in {\mathcal C}^\gamma_\infty $, there exists 
a value of the constants $ \a := \a_\infty (\om, \kappa, \e) = \om + O(\e \g^{-k} ) $ 
and a $ \nu $-dimensional embedded  torus  $ {\mathcal T} = i(\T^\nu )$
close to $ \T^\nu \times \{0\} \times \{ 0 \} $,   
 invariant for the Hamiltonian vector field $ X_{H(\a_\infty (\omega, \kappa, \e), \cdot )}$ and 
 supporting quasi-periodic solutions with frequency $ \om $. 
This is equivalent  to look for a zero of the nonlinear operator
$ {\mathcal F} (i, \a, \omega, \kappa, \varepsilon ) = 0 $ defined in \eqref{operatorF}. This equation is solved in Theorem \ref{MAINTHEOREM} by
a Nash-Moser iterative scheme.
The value of $ \a := \a_\infty (\om, \kappa, \e) $ is adjusted along the iteration 
in order to
control  the average of the first component of the Hamilton equation \eqref{operatorF}, 
in particular for solving the linearized equation \eqref{operatore inverso approssimato}, \eqref{equazione psi hat}. 

The set of parameters $ (\om, \kappa) \in {\mathcal C}^\gamma_\infty $ for which the invariant 
torus exists is the explicit set \eqref{Cantor set infinito riccardo}.
We require that $ \omega $ satisfies the diophantine property  
\be\label{dioph}
| \om \cdot \ell | \geq \gamma \langle \ell \rangle^{- \tau} \, , \quad \forall \ell \in \Z^\nu \setminus \{0\} \, ,
\ee
and, in addition, the  first and second  Melnikov 
non-resonance conditions. 

Note that the set $ {\mathcal C}^\gamma_\infty $ is defined
  in terms of the ``final torus" $ i_\infty $ (see \eqref{stima toro finale}) and the 
``final eigenvalues" in \eqref{autovalori infiniti} which are defined for {\it all} the values of the frequency $ \om \in \R^\nu $ and $ \kappa \in [\kappa_1, \kappa_2]$ by a Whitney-type extension
argument, see the sentences after \eqref{omega diofanteo troncato}.
This formulation completely decouples the Nash-Moser 
iteration (which provides the torus $ i_\infty (\omega, \kappa, \e) $ and the constant $ \a_\infty (\om, \kappa,  \e) \in \R^\nu $) 
from the discussion about the measure of the set of parameters 
where all the non-resonance conditions are indeed verified. This simplifies 
the  measure estimates which are no longer imposed at each step but only once, see section \ref{sec:measure}.
This formulation follows  that of \cite{BB06} (in a Lyapunov-Schmidt context) and \cite{BBi10}
(in a KAM theorem) and \cite{BCP} (in a Nash-Moser context).   
The measure estimates are done in section \ref{sec:measure}. 

In order to prove the existence of quasi-periodic solutions of the  water waves equations \eqref{WW}, and not only
of the system with modified Hamiltonian $ H_\a $ with $ \a := \a_\infty (\om, \kappa, \e) $, 
we have then to prove that the curve  of the unperturbed linear frequencies 
$$ 
[\kappa_1, \kappa_2] \ni \kappa \mapsto \vec \om (\kappa ) := ( \sqrt{j(1+ \kappa j^2) } )_{j \in {\mathbb S}^+} \in \R^\nu 
$$
intersects the image $ \a_\infty ({\mathcal C}^\gamma_\infty ) $, under the map $ \a_\infty $
 of the   set $ {\mathcal C}^\gamma_\infty  $, for ``most" values of $ \kappa \in [\kappa_1, \kappa_2] $.  
 This is proved in Theorem \ref{Teorema stima in misura} by degenerate KAM theory. 
For such values of $ \kappa $ we have found a quasi-periodic solution of \eqref{WW} with 
diophantine frequency $ \om_\e (\kappa) := \a_\infty^{-1} ( \vec \om (\kappa ), \kappa ) $, where $\alpha_\infty^{- 1}(\cdot, \kappa)$ is the inverse of the function $\alpha_\infty(\cdot, \kappa)$ at a fixed $\kappa \in [\kappa_1, \kappa_2]$. 

The above functional setting  perspective 
is in the spirit of the so called ``Th\'eor\'eme de conjugaison hypoth\'etique" of Herman proved by 
Fejoz \cite{HF} for finite dimensional Hamiltonian systems, see also the discussion in \cite{BB13}.
A relevant  difference is that in \cite{HF}, in addition to $ \a $,   also the normal frequencies  are introduced as 
independent parameters, 
unlike in Theorem \ref{MAINTHEOREM}. Actually
for PDEs  it seems more convenient the present formulation:  it is 
a major point of the work to know the asymptotic expansion \eqref{Floquet-exp} 
 of  the 
Floquet exponents.
 \item {\it Degenerate KAM theory and measure estimates.} 
In Theorem \ref{Teorema stima in misura}  we prove that 
for all the  values of 
$ \kappa  \in [\kappa_1, \kappa_2] $ except a set of small measure 
$ O(\g^{1/k_0}) $ 
(the value of $ k_0 \in \N $ is fixed once for all in section \ref{sec:degenerate KAM}) 
the vector 
$ (\alpha_\infty^{- 1}(\vec \om (\kappa), \kappa), \kappa) $ belongs 
to  the set ${\mathcal C}^\gamma_\infty$, see the set $ {\mathcal G}_\e $ in \eqref{defG-ep}.
As already said, we use in an essential way that the unperturbed frequencies 
$ \kappa \mapsto \om_j (\kappa) $  
are {\it analytic}, are simple (on the subspace of the even functions), 
 grow asymptotically as 
$ j^{3/2} $ and are {\it non-degenerate} in the sense of \cite{BaBM}. 
This is  verified in Lemma \ref{non degenerazione frequenze imperturbate} as in \cite{BaBM} by a Van der Monde determinant. Then we 
develop degenerate KAM theory which reduces this qualitative non-degeneracy condition into a quantitative one,
which is sufficient to estimate effectively the measure of the  set 
$ {\mathcal G}_\e $ by the classical R\"ussmann lemma. 
We deduce in Proposition \ref{Lemma: degenerate KAM} that $ \exists k_0 > 0 $, $ \rho_0 > 0 $  such that,
for all $ \kappa \in [\kappa_1, \kappa_2] $,
\be\label{unperturbed measure}
\begin{aligned}
\max_{0 \leq k \leq k_0} \big| \pa_\kappa^k  \big( {\vec \om} (\kappa) \cdot \ell + { \Om}_j (\kappa) 
- { \Om}_{j'} (\kappa)  \big) \big| 
\geq \rho_0 \langle \ell \rangle  \, ,   \\ 
\forall (\ell, j, j') \neq (0,j,j), \  j, j' \in \N^+ \setminus {\mathbb S}^+ \, , 
\end{aligned}
\ee
and similarly for the $ 0 $-th,  
$ 1$-th and the $ 2$-th order Melnikov non-resonance condition with the sign $ +$. Note that 
the restriction to the subspace \eqref{eta0psi0}, see also \eqref{phase space}, of functions with zero average 
in $ x $  eliminates the zero frequency $ \om_0 = 0 $, which is trivially resonant
(this is used also in  \cite{Craig-Worfolk}).
Property \eqref{unperturbed measure} 
implies 
that for ``most" parameters $ \kappa \in [\kappa_1, \kappa_2] $ the unperturbed linear frequencies
$ ( \vec \om (\kappa), \vec \Omega (\kappa) ) $ satisfy the Melnikov conditions of $ 0,1,2 $ order (but we do not use it explicitly). 
Actually, the condition  \eqref{unperturbed measure} is stable under
perturbations which are small in $ {\mathcal C}^{k_0}$-norm, see Lemma \ref{Lemma: degenerate KAM perturbato}.
Since the perturbed Floquet exponents in \eqref{mu j infty kappa} 
are  small perturbations of the unperturbed linear frequencies
$ \sqrt{j(1+ \kappa j^2)} $ in $ {\mathcal C}^{k_0}$-norm 
(see \eqref{stima omega epsilon kappa} and \eqref{stime coefficienti autovalori in kappa})
the `transversality" property \eqref{unperturbed measure} still holds   
for the  perturbed frequencies $ \om_\e (\kappa) $ 
defined in \eqref{omega epsilon kappa}. 
As a consequence, 
by applying the classical R\"ussmann lemma (Theorem 17.1 in \cite{Ru1}) 
 we prove that the set of non-resonant parameters $ {\mathcal G}_\e $
has a large measure, see Lemma \ref{stima risonanti Russman} and the end of the 
proof of Theorem \ref{Teorema stima in misura}.   
\end{enumerate}

\noindent
{\it Analysis of the linearized operators.} 
The other crucial analysis  for the Nash-Moser iterative scheme is to prove that the 
{\it  linearized operator} obtained at any approximate solution 
is, for most values of the parameters, invertible, and that its inverse satisfies {\it tame} estimates in Sobolev spaces.
We implement in section \ref{costruzione dell'inverso approssimato} 
the  procedure developed in Berti-Bolle \cite{BB13} and \cite{BBM-auto}-\cite{BBM-mKdV} for autonomous PDEs.
It consists in introducing a convenient set of symplectic variables (see \eqref{trasformazione modificata simplettica}) near the 
approximate torus such that
the linearized equations 
become  (approximately)  decoupled in the action-angle components and the normal ones, see  
\eqref{operatore inverso approssimato}.
As a consequence,  
 the problem is reduced
to ``almost-approximately" invert  the linearized operator 
$ {\mathcal L}_\om $ defined in   \eqref{Lomega def}.  
Actually, since the 
symplectic change of variables \eqref{trasformazione modificata simplettica} 
 modifies, up to a translation, only the finite dimensional action component, 
the linear operator $ {\mathcal L}_\om $  is nothing but the linearized water-waves 
operator $ {\mathcal L}$ computed in \eqref{linearized vero} -in the original coordinates- up to a finite dimensional remainder
and restricted to the normal directions. 
Thus the key part of the analysis consists in
(almost) reducing  the quasi-periodic linear operator $ {\mathcal L}$  to constant coefficients,
via linear changes of variables close to the identity, which map Sobolev spaces into itself and satisfy  tame estimates, see
Theorem \ref{inversione parziale cal L omega}. We refer to this result as ``almost invertibility" of ${\mathcal L}_\omega$, because we get an inverse of this operator up to the small remainders ${\bf R}_\omega$ (which is of order $O(\e \gamma^{- 1} N_{n - 1}^{- \mathtt a})$, $\mathtt a > 0$) and ${\bf R}_\omega^\bot$ (which is of order $O(K_{n }^{-  b})$, $ b > 0$), see \eqref{splitting cal L omega}-\eqref{stima R omega bot corsivo bassa}. 

\noindent
This is achieved in sections \ref{linearizzato siti normali} and \ref{sec: reducibility}
by making full use of pseudo-differential operator theory 
that we present  in section \ref{sec:pseudo} in a formulation convenient to our purposes. 
\\[1mm]
{\it Pseudo-differential operators. }
We underline that all the coefficients 
of  the linearized operator 
$ {\mathcal L}$ in \eqref{linearized vero} are $ {\mathcal C}^\infty $ in $ (\vphi, x) $ because 
each approximate solution $  (\eta (\vphi, x), \psi (\vphi, x)) $ 
at which we linearize along the Nash-Moser iteration 
is a  trigonometric  polynomial in $ (\vphi, x) $ (at each step we apply the projector
$ \Pi_n $ defined in \eqref{truncation NM}) and the water waves vector field is analytic. 
This allows to work in the usual framework of  $ {\mathcal C}^\infty $ pseudo-differential symbols.
 
In this paper we only use the class $ S^m $ of  (classical) symbols introduced in Definition
\ref{def:Ps2}. We do not explicitly make use  of  
pseudo-differential operators in the  class $ OPS^m_{\frac12, \frac12} $  
used by  Alazard-Baldi in \cite{AB} (called semi-Fourier integral operators). 
Actually we shall produce similar transformations 
as {\it flow}s of   pseudo-PDEs (see \eqref{pseudo-PDE}).
The advantage is that the invertibility of such transformations, 
  as well as the fact that they satisfy tame estimates  in Sobolev spaces
  together with its inverses, 
follows easily 
by proving energy estimates for the flow, see Appendix \ref{AppendiceA}.  

For the Nash-Moser convergence we clearly need to perform quantitative 
estimates in Sobolev spaces. Then, given a pseudo-differential operator 
$$ 
A =  {\rm Op} (a (\vphi,x, \xi) ) \in  OPS^m \, , 
$$
we introduce the norm  $ \norma A \norma_{m,s,\a} $  defined in
\eqref{norm1} (more generally $ \norma A \norma_{m,s,\a}^{k_0, \gamma} $ in Definition \ref{def:pseudo-norm}),  
which is inspired to the para-differential norm in Metivier \cite{Met}, chapter 5.
Note that $ \norma A \norma_{m,s,\a} $  controls the regularity in $ (\vphi, x)$ of the symbol $ a (\vphi, x, \xi) \in S^m $ 
only up to a limited smoothness.  
\\[1mm]
We now explain the main steps for the reduction of the quasi-periodic linear operator $ {\mathcal L} $ in \eqref{linearized vero}. 
 \begin{enumerate}
\item 
{\it Reduction of $ {\mathcal L} $ to constant coefficients in decreasing symbols}.  
The goal of section \ref{linearizzato siti normali} (Proposition \ref{prop: sintesi linearized}) is to reduce  $ {\mathcal L} $ to 
a quasi-periodic linear operator of the form  
\be\label{forma-lineare-preliminary}
(h, \bar h) \mapsto  \big( \Dom   + \ii \mathtt m_3  T(D) + \ii {\mathtt m}_1 |D|^{\frac12} \big) h  + {\mathcal R}  h + {\mathcal Q}  \bar h  \, ,
\quad h \in \C \, , 
\ee
where $ \mathtt m_3, {\mathtt m}_1 \in \R $ are constants satisfying  $ \mathtt m_3 \approx 1 $, $ {\mathtt m}_1 \approx 0 $, the 
principal symbol operator is 
$$ 
T(D) := |D|^{1/2} (1 - \kappa \pa_{xx})^{1/2}  \, , 
$$  
and the remainders 
 $ {\mathcal R} :=  {\mathcal R} (\vphi) $,  $ {\mathcal Q} :=  {\mathcal Q} (\vphi) $
 are small bounded operators acting in the Sobolev spaces $ H^s $, which  
satisfy  tame estimates. More precisely, in view of 
the KAM reducibility scheme of section \ref{sec: reducibility}, 
we need that all the operators in \eqref{operatori-tames},  
together 
 with its derivatives $ \pa_{\om, \kappa}^k {\mathcal R}  $, $ \pa_{\om, \kappa}^k {\mathcal Q}  $, $ |k| \leq k_0 $,  
satisfy tame estimates, see \eqref{stima finale resti prima del KAM}. 
We  neglect in \eqref{forma-lineare-preliminary} smoothing operators 
which are supported on high  Fourier frequencies (ultra-violet cut-off) and therefore satisfy  
\eqref{stima bf R N (3) bot bassa}-\eqref{stima bf R N (3) bot alta}. 
Note that \eqref{forma-lineare-preliminary} is an operator which acts 
on $ (h, \bar h )$. We shall deal in a quite different way the   operators
$$ 
h \mapsto ( \Dom   + \ii \mathtt m_3  T(D) + \ii {\mathtt m}_1 |D|^{\frac12}) h  + {\mathcal R}  h \quad
{\rm and} \quad  
\bar h \mapsto {\mathcal Q}  \bar h \, . 
$$
  We shall call the first operator
``diagonal",   and the latter  ``off-diagonal", with respect to the variables $ (h, \bar h ) $.  
\item
{\it Symmetrization and space-time reduction  of ${\mathcal L} $ at the highest order}. 
The first part of the analysis (sections \ref{sec:linearized operator}-\ref{sec:changes}) 
is similar to Alazard-Baldi \cite{AB}. A  difference is that 
we  reduce the linear operator $ {\mathcal L} $
in \eqref{linearized vero} 
to constant coefficients up to  $ OPS^0 $ remainders (Lemma \ref{lemma:S}),  
while  in \cite{AB} the remainders are $ O(\pa_x^{-3/2}) $. 
The reason of this difference is  that we will not 
invert the linearized operator in \eqref{forma-lineare-preliminary} simply by a Neumann-argument, as done 
for the periodic solutions in \cite{AB}, \cite{IPT},  \cite{IP-SW2}, \cite{IP-SW1},  \cite{PlTo}.
This approach  does not  work in the quasi-periodic case. 
The key difference is that,
in the periodic problem, 
a sufficiently regularizing operator in the space variable is also regularizing in the  time  variable,  
on the characteristic Fourier indices which correspond to 
the small divisors. This is  clearly not true for quasi-periodic solutions.

Our strategy will be to
diagonalize, actually it is sufficient to ``almost diagonalize", 
the linearized operator 
 in \eqref{forma-lineare-preliminary}
by  the KAM scheme of section \ref{sec: reducibility}.  
The expression ``almost diagonalize" refers to the fact that in Theorem \ref{Teorema di riducibilita}
the remainders $ {\bf R}_n $ and $ {\bf Q}_n $ that are left in \eqref{cal L infinito} 
are not zero, but small as $ O( \e \gamma^{-1} N_{n-1}^{- \mathtt a})$ (and this is because we require
just the finitely many diophantine conditions \eqref{Cantor set}).
This requires to analyze more in detail the pseudo-differential  nature of the remainders after all the 
conjugation steps -a key difference concerns the nature of 
 the block-off diagonal operators in $ (h, \bar h )$ with respect to the diagonal ones-
and to be able to impose the second Melnikov non-resonance conditions.

In  section \ref{complex-coordinates} 
we introduce complex coordinates $ (h, \bar h ) $, which  are  convenient 
to reduce the off-diagonal blocks of the linear system  to a very negative order (section \ref{sec:decoupling}). 
We could have introduced the complex variables $ (h, \bar h )$ right after section 
\ref{sec:linearized operator} performing the symmetrization procedure and the space reduction 
of the highest order (section \ref{sec:changes}) in the variables $ (h, \bar h) $. This way, however, 
would require to use an Egorov type argument to estimate the remainders unlike in section \ref{sec:changes}
we use (as in \cite{AB}) only the simple change of variables \eqref{change-of-variable}.

Then in section \ref{sec: time-reduction highest order}, using a time-reparametrization as in \cite{AB},  
we  obtain a quasi-periodic linear operator of the form (see \eqref{cal L 4})
\be\label{forma-lineare-preliminary-prima Egorov}
\begin{aligned}
 (h, \bar h ) \mapsto  \big( \Dom   & + \ii \mathtt m_3  T(D) + a_{11} (\vphi, x) \pa_x 
+ \ii a_{12} (\vphi, x) {\mathcal H} |D|^{\frac12} \big) h  \\ 
& + 
\ii b (\vphi, x) {\mathcal H} |D|^{\frac12}  \bar h + \ldots   \, .
\end{aligned}
\ee 
From this point we have to proceed quite differently with respect to   \cite{AB}.
\item {\it Block-decoupling}.  In view of the transformations used in the next Egorov-step 
and the KAM reducibility scheme of section \ref{sec: reducibility},
we  first  reduce the order of  the off-diagonal term $ \ii b (\vphi, x) {\mathcal H} |D|^{\frac12}  \bar h $
to a very negative order $ OPS^{-M} $. 
In section \ref{sec:decoupling} we  
 conjugate \eqref{forma-lineare-preliminary-prima Egorov}  to a quasi-periodic linear operator of the form  
(Proposition \ref{Lemma finale decoupling}) 
$$
\begin{aligned}
( h, \bar h)  \mapsto  \Dom h & + \ii \mathtt m_3 T(D) h + 
a_{11}(\vphi, x) \partial_x h + \ii a_{12}(\vphi, x)  {\mathcal H} |D|^{\frac12} h \\ 
& + {\mathcal R}_M h + {\mathcal Q}_M \bar h
\end{aligned}
$$
where $  {\mathcal R}_M \in OPS^0 $ and $  {\mathcal Q}_M \in OPS^{- M} $, for some $ M $ large enough
which is fixed by the KAM reducibility scheme, see  \eqref{relazione mathtt b N}.

\item {\it 
Egorov analysis. Space reduction of the order $ \pa_x $.}
The goal of section  \ref{egorov} is to eliminate the first order 
vector field $ a_{11}(\vphi, x) \partial_x   $. For that Alazard-Baldi \cite{AB} used a 
 semi-Fourier integral operator like
$ {\rm Op} ( e^{\ii a(\vphi, x) \sqrt{|\xi|}} ) \in OPS_{\frac12, \frac12}^{0} $. 
We shall  use instead  
the flow $ \Phi (\vphi) :=  \Phi (\vphi, \omega, \kappa ) $ of the pseudo-PDE 
\be\label{flow:PSD0}
u_t = \ii a(\vphi, x, \omega, \kappa) |D|^{1/2} u \, . 
\ee
The proof  that $ \Phi $, as well as 
its inverse $ \Phi^{-1} $, is well posed in Sobolev spaces $ H^s $
and satisfies tame estimates, 
follow by the energy estimates of  Appendix  \ref{AppendiceA}
(the vector field $ \ii a(\vphi, x, \om, \kappa) |D|^{1/2} $ 
is skew-adjoint at the highest order). 
We think that this is conceptually simpler than proving directly the invertibility and the tame estimates of $ {\rm Op} ( e^{\ii a(\vphi, x) \sqrt{|\xi|}} ) $
as in \cite{AB}.

However the main advantage in order to use the present 
flow approach consists  in the Egorov analysis of the pseudo-differential nature of the conjugated
operator.  
The flow
has a very different effect on the operator
$ h \mapsto ( \ii a_{12}(\vphi, x)  {\mathcal H} |D|^{\frac12}  + {\mathcal R}_M) h $ 
and  the off-diagonal one 
$ \bar h \mapsto {\mathcal Q}_M \bar h $:  
the first 
remains a classical pseudo-differential operator in $ OPS^0 $ (Egorov analysis),   
but the off-diagonal one becomes a pseudo-differential operator 
in the class $ OPS_{\frac12, \frac12}^{-M} $.

Let us roughly explain why this is a  relevant information.
The flow 
$  \Phi (\vphi) \sim {\rm Op} ( e^{\ii a(\vphi, x) \sqrt{|\xi|}} ) $
maps Sobolev spaces in itself. However each derivative 
$$
\pa_\vphi  \Phi (\vphi) \sim {\rm Op} \big( e^{\ii a(\vphi, x) \sqrt{|\xi|}} \, \ii \pa_\vphi  a(\vphi, x) \sqrt{|\xi|} \big)
$$
is an unbounded operator which loses $ |D|^{1/2} $ derivatives. In the Appendix we actually prove that 
$ \pa_{\om,\kappa}^k \pa_\vphi^\beta  \Phi (\vphi ) $ satisfies tame estimates with a loss of $| D|^{\frac{|\b|+|k|}{2}} $ derivatives.
 
The main idea of the Egorov analysis in section \ref{egorov} is that, given a scalar classical 
pseudo-differential operator $ P_0 \in OPS^{m} $,
the conjugated operator 
\be\label{coniugazione-Egorov}
P_+ (\vphi) := \Phi (\vphi) P_0 \Phi (\vphi)^{-1} = {\rm Op} ( c(\vphi, x, \xi )) \, , \quad c(\vphi, x, \xi ) \in S^m \, , 
\ee
remains as well a classical pseudo-differential operator. 
Therefore, the differentiated operator $ \pa_\vphi P_+ (\vphi) =  {\rm Op} ( \pa_\vphi c(\vphi, x, \xi )) \in OPS^{m} $ 
is a pseudo-differential operator of the same order of $ P_0 $ with a symbol 
$ \pa_\vphi c $ which is just less regular in $ \vphi $.
Then the loss of regularity for $ \pa_\vphi c $ is compensated by the  
usual Nash-Moser smoothing procedure in $  \vphi $. 
The property \eqref{coniugazione-Egorov} is due to the fact that  $ P_+ $ is ``transported" 
 under the flow of \eqref{flow:PSD0} 
according to  the Heisenberg equation \eqref{equazione Egorov}. 

This is the reason 
why we  require that the diagonal remainder $ {\mathcal R} \in OPS^0 $ is just  of order zero.

On the other hand, the off-diagonal term $ {\mathcal Q}_M \in OPS^{-M} $ evolves, under the flow of \eqref{flow:PSD0},  according to the 
``skew-Heisenberg" equation obtained 
replacing in \eqref{equazione Egorov} the commutator with  the skew-commutator. 
As a consequence  the symbol  of $ {\mathcal Q}_M^+ :=  \Phi (\vphi) {\mathcal Q}_M \Phi (\vphi)^{-1} $ 
  assumes the form $ e^{\ii a(\vphi, x) \sqrt{|\xi|}} q (\vphi, x, \xi  ) $
where $ q (\vphi, x, \xi  ) \in S^{-M} $ is a classical symbol 
(actually we do not prove it explicitly because it is not needed). 
Thus the action of each $  \pa_\vphi   $ on $ {\mathcal Q}_M^+ $ produces an operator 
which loses $ | D |^{\frac12} $ derivatives  in space more than $ {\mathcal Q}_M $.
 This is why we  perform in section \ref{sec:decoupling} a large number $ M $ 
 of regularizing steps for the off-diagonal components  $ {\mathcal Q} $. The constant $ M $ is fixed later in \eqref{relazione mathtt b N}. 
 The precise tame estimates of  $ \pa_\vphi^\beta {\mathcal Q}_M^+$ are given in Proposition \ref{Prop:Egorov} 
 for $M \geq  \beta + k_0 + 4 $.
In section \ref{sec: reducibility} we take $\beta \sim \mathtt b$, see \eqref{relazione mathtt b N}. 
 
\item {\it Space reduction of the order $ |D|^{1/2} $.}
In section \ref{sec:lineare} we reduce to constant coefficients also the  diagonal operator term of order $ |D|^{1/2} $. This  concludes
(section \ref{coniugio cal L omega}) 
the conjugation of $ {\mathcal L}_\om $ to a quasi-periodic linear 
operator like  \eqref{forma-lineare-preliminary}.

 \item {\it KAM-reducibility scheme.} 
We apply the KAM diagonalization  scheme of section \ref{sec: reducibility} to a linear 
operator  as in \eqref{forma-lineare-preliminary} 
where 
\be\label{operatori-tames}
\begin{aligned}
&  {\mathcal R} \, ,  \ [ {\mathcal R}, \pa_x ] \, , \ \partial_{\vphi_m}^{s_0} {\mathcal R} \, , \
\partial_{\vphi_m}^{s_0}  [ {\mathcal R}, \pa_x ] \, , \\
& \partial_{\vphi_m}^{s_0 +  {\mathtt b}} {\mathcal R} \, , 
 \partial_{\vphi_m}^{s_0 +  {\mathtt b}}  [ {\mathcal R}, \pa_x ] \, , \
m =1, \ldots, \nu \, , 
\end{aligned}
 \ee  
 and similarly $ {\mathcal Q} $,  satisfy tame estimates for some $ \mathtt b := \mathtt b (\tau, k_0) \in \N $ large enough, fixed in \eqref{alpha beta}, see 
\eqref{tame cal R0 cal Q0}, \eqref{tame norma alta cal R0 cal Q0}, \eqref{def:costanti iniziali tame}. 
Such condition is proved in  Lemma \ref{lem:tame iniziale}, having assumed that $ M $ 
(= number of regularizing steps for the off-diagonal operators performed in section \ref{sec:decoupling})
is taken large as in \eqref{relazione mathtt b N} (essentially $ M = O({\mathtt b})$).  It is   the property 
which compensates, along the KAM iteration, the loss of derivatives in $ \vphi $ produced by the small divisors 
(this condition  is strictly weaker than assuming a polynomial off-diagonal decay  of  $ {\mathcal R} $, $ {\mathcal Q} $, 
 as in \cite{BBM-Airy}-\cite{BBM-auto}). 

The core of the KAM reducibility scheme
of section \ref{sec: reducibility} is to prove that the class of operators which 
are $ {\mathcal D}^{k_0} $-modulo-tame (Definition \ref{def:op-tame}) 
is closed under the operations involved by a KAM iteration, namely
\begin{enumerate}
\item 
composition (Lemma \ref{interpolazione moduli parametri}), 
\item
solution of the homological equation 
(Lemma \ref{Homological equations tame}), 
\item projections (Lemma \ref{lemma:smoothing-tame}). 
\end{enumerate}
We recall that we have to control that 
the KAM transformations (and all the operators) 
are $ k_0 $-times differentiable  with respect to the parameters $ (\om, \kappa) \in \R^\nu \times [\kappa_1, \kappa_2]  $
to prove that the Floquet exponents 
$ (\om, \kappa) \mapsto \mu_j^\infty (\om, \kappa) $ in \eqref{autovalori infiniti} 
are  small perturbations of the linear frequencies 

\noindent
$  \sqrt{j(1 + \kappa j^2)} $ in $ {\mathcal C}^{k_0} $-norm. 

The reason why 
we implement the KAM reducibility scheme for
$ {\mathcal D}^{k_0} $-modulo-tame operators and not only for $ {\mathcal D}^{k_0} $-tame operators is that 
for  a  $ {\mathcal D}^{k_0} $-tame operator the second estimate in Lemma \ref{lemma:smoothing-tame} for 
the projector $ \Pi_N^\bot $  does not hold
(majorant like norms have been used also in 
\cite{Berti-Biasco-Procesi-Ham-DNLW}-\cite{Berti-Biasco-Procesi-rev-DNLW}). 
The fact that the initial majorant  operators $ | {\mathcal R} | $,
$ | {\mathcal Q} | $ (see Definition \ref{def:maj})  fulfill tame estimates (which is stronger that requiring tame estimates  just for  
$ {\mathcal R} $ and $  {\mathcal Q} $) is verified  
in  Lemma \ref{lem: Initialization} 
thanks to the assumption that $ [\pa_x, {\mathcal R}] $ and $ \pa_{\vphi_m}^{s_0} {\mathcal R} $, as  
well as all the operators in \eqref{operatori-tames},  satisfy tame estimates, 
see Lemma \ref{lem:tame iniziale}. 
Note that the commutator $ [\pa_x, r(x, D)] = r_x (x, D ) $ is a pseudo-differential operator 
with the same order of $ r(x, D)$
(this is used in particular in Proposition \ref{Prop:Egorov}). 
This  is another  reason for which it is sufficient that 
the pseudo-differential remainder which acts  on the diagonal (i.e. on $ h $) is just in $ OPS^0 $.

The key (quadratic + super-exponentially small) inductive estimates required for the convergence of the iteration 
are provided by Lemma \ref{estimate in low norm}. 
More precisely 
\eqref{schema quadratico tame} and \eqref{M+Ms} 
allow to prove the convergence of the scheme up to the Sobolev index $ s $, by choosing 
$ {\mathtt b}  := {\mathtt b}(\tau)  $ large enough as fixed  
in \eqref{alpha beta}.
The inductive relation \eqref{M+Ms} provides an a priori bound for the divergence of the modulo-tame 
constants  $ {\mathfrak M}_\nu^\sharp (s,{\mathtt b}) $ of the operators 
$ \langle \partial_\vphi \rangle^{ {\mathtt b}} {\mathcal R}_{\nu + 1} $ and 
$ \langle \partial_\vphi \rangle^{ {\mathtt b}} {\mathcal Q}_{\nu + 1}  $
along the iteration.  
Then \eqref{schema quadratico tame} shows that $ {\mathfrak M}_\nu^\sharp (s) $
converges very rapidly to $ 0$ as $ \nu \to + \infty $, see  \eqref{stima cal R nu}.

Note that the iterative KAM Theorem \ref{ITERAZIONERIDUCIBILITA} requires only the smallness condition
 \eqref{KAM smallness condition1} which involves just  the low norm 
$ \| \ \|_{s_0 + {\mathtt b}} $ but implies also tame estimates up to the Sobolev scale $ s $, see \eqref{stima cal R nu}. 
The important consequence is that, in Theorem \ref{Teorema di riducibilita}, 
only the  condition \eqref{ansatz riducibilita} in low norm, implies
the  tame estimates \eqref{stima Phi infinito} for the transformations up to any $ s \in [s_0, S ] $. 
The smallness condition  \eqref{ansatz riducibilita} 
will be verified  inductively along the nonlinear Nash-Moser scheme of section \ref{sec:NM}. The 
tame property \eqref{stima Phi infinito} (at any scale) is used in the convergence of the Nash-Moser iteration
of section \ref{sec:NM}. 
\end{enumerate}

After the above analysis of the linearized operator, in section \ref{sec:NM}, 
we implement a differentiable Nash-Moser iterative scheme 
to find better and better approximate quasi-periodic solutions
up to the scales
\begin{equation}\label{definizione Kn}
K_n := K_0^{\chi^{n}} \,,\quad \chi := 3/ 2\, ,
\end{equation}
which lead, at the limit,  to an embedded  torus invariant under the flow of the Hamiltonian PDE, see Theorem
\ref{iterazione-non-lineare} and section \ref{proof theorem 4.1}.   

\smallskip

We conclude the introduction with some other comment.

\begin{enumerate}
\item {\it Whitney extension.}
At each iterative step of the Nash-Moser iteration
-and correspondingly for the
reduction of the linearized operator in sections \ref{costruzione dell'inverso approssimato}, \ref{linearizzato siti normali}, 
\ref{sec: reducibility}- 
we only require that the frequency vector $ \om \in \R^\nu $ satisfies  finitely many non-resonance diophantine conditions.
More precisely we assume at the $ n $-th step that 
$  \om $ belongs to  
\begin{equation}\label{omega diofanteo troncato}
{\mathtt D \mathtt C}_{K_n}^{\g} := \big\{ \om \in \tOm \subset \R^\nu \, : \,  | \omega \cdot \ell | 
\geq \gamma \langle \ell \rangle^{- \tau}\,,\  \forall | \ell | \leq K_n \big\}
\end{equation}
and similarly we require finitely many first and second order Melnikov non-resonance conditions, see 
\eqref{prime di melnikov} and \eqref{Omega nu + 1 gamma}
(the set $ \tOm $ is the neighborhood \eqref{unperturbed-frequencies}
of the curve $ \vec \om ([\kappa_1, \kappa_2]) $ described by the 
unperturbed linear frequencies $ {\vec \omega} $). 
This allows to perform 
a constructive Whitney extension  of the solution, 
with respect to the parameters $ (\om, \kappa ) $ in a way similar to 
\cite{BB06}.
We find this construction convenient in order to estimate the $ k$-derivatives 
$ \pa_{\omega, \kappa}^k $   of the approximate solutions (and of the eigenvalues) 
 which, on a subset with a not empty interior (like $ {\mathtt D \mathtt C}_{K_n}^{\g} $) are well defined in the usual sense
(instead of  
introducing the notion of Whitney derivatives on closed subsets, possibly with an empty interior). 
The  quantitative estimates that we shall obtain (see for example
 \eqref{stima toro finale} and  \eqref{stime coefficienti autovalori in kappa}) 
are similar to those which are satisfied  by the solution 
\be\label{basic-KAM0}
h := (\om \cdot \pa_\vphi)^{-1} g = 
\sum_{\ell \in \Z^\nu \setminus \{ 0 \}} \frac{g_\ell}{ \ii \om \cdot \ell} \, e^{\ii \ell \cdot \vphi } \, ,   \qquad 
g := \sum_{\ell \in \Z^\nu \setminus \{0\} } g_\ell e^{\ii \ell \cdot \vphi } \, , 
\ee
of the basic linear equation of KAM theory $ \om \cdot \pa_\vphi h = g $, namely  
\be\label{basic-KAM}
\| \pa_\om^k h \|_s \leq C \g^{-|k|} \| g  \|_{s+ \tau + |k| \tau} \, . 
\ee
We note that each  derivative $ \pa_\om $ produces a factor $ \g^{-1}$  and 
a loss of $ \t $-derivatives in the Sobolev index. 
This is the phenomenon described by 
P\"oschel in \cite{Po82} as ``anisotropic differentiability" of the families of KAM tori
with respect to $ \om $. Actually 
when solving the homological equations, see \eqref{shomo1}-\eqref{shomo2},   
we also have denominators which depend on both $ (\omega, \kappa) $ 
and we have to estimate  the regularity of the solution also with respect to 
$ \kappa $, see Lemma \ref{Homological equations tame}. 
\item
{\it Dirichlet-Neumann operator.} In section \ref{subDN} we use  a self-contained 
proof of the representation  of the Dirichlet-Neumann operator  $ G(\eta) $ as a pseudo-differential operator, due to Baldi
\cite{Baldi}. 
The conformal change of variables \eqref{conf-diffeo}-\eqref{defUV} transforms
the elliptic problem \eqref{BVP},  which is defined in the variable fluid domain 
 $  \{ y \leq \eta (x) \} $,  
into the elliptic problem \eqref{BVP-new} which is defined 
on the straight strip $ \{ Y \leq 0 \} $ and  can be solved by an explicit integration. 
By  conjugating back such solution,  it turns out that (Lemma \ref{G=mH}) 
the principal symbol of $ G(\eta) $ is just 
$ |D| $ (see  \eqref{sviluppo Geta})  
up to  
a small remainder  $ {\mathcal R}_G (\eta)  \in  OPS^{-\infty} $ (recall that the profile $ \eta  \in {\mathcal C}^\infty $). 
Actually  $ \psi \mapsto {\mathcal R}_G (\eta) [\psi ] $ is a regularizing linear operator which satisfies 
tame estimates (with loss of derivatives) in $ \eta $, see e.g. \eqref{stima tame dirichlet neumann}. 
For obtaining such quantitative estimates  it is convenient to 
represent  $ {\mathcal R}_G $  as an integral operator (see \eqref{Geta intermedia} and Lemma \ref{coniugio Hilbert})
and to use the fact an integral operator
transforms into another integral operator under changes of variable, see Lemma \ref{lemma cio}. 
\end{enumerate}

\noindent
{\it Acknowledgements}. We  thank P. Baldi, L. Biasco, W. Craig and J. M. Delort,  for many useful discussions. 

\section{Notation}

We organize in this subsection the most important notation used in the paper.

\smallskip

We denote by $ \N := \{0,1, 2,  \ldots \} $ the natural numbers including $ \{0\} $ and 
$ \N^+ := \{ 1, 2,  \ldots \} $. 
We denote the ``tangential" sites by 
\be\label{def:S0}
{\mathbb S}^+ \subset \N^+ \quad {\rm and \ we \ set } \quad 
{\mathbb S} := {\mathbb S}^+ \cup (-{\mathbb S}^+) \, , \quad 
{\mathbb S}_0 := {\mathbb S}_+ \cup (- {\mathbb S}_+) \cup \{ 0 \} \subseteq \Z \, . 
\ee
The cardinality of $ {\mathbb S}^+ $ is $ |{\mathbb S}^+| = \nu $, and we look for quasi-periodic solutions
with frequency $ \omega \in \R^\nu $. The surface tension parameter $ \kappa $ is in the interval 
$ [\kappa_1, \kappa_2 ] $ with $ \kappa_1 >  0 $.
In the paper all the functions, operators, transformations, etc \ldots, depend on  
the parameter 
$$ 
\lambda = (\om, \kappa ) \in  \tLm_0 \subset \R^\nu \times [\kappa_1, \kappa_2 ] \, ,
$$
 in a $ k_0$-differentiable way. We will often not specify the domain $ \tLm_0  $ which is understood from the context. 
We use the multi-index notation $ k = ( k_1, \ldots , k_{\nu+1}) \in \N^{\nu+1} $ with $ | k  | := k_1 + \ldots +  k_{\nu+1} $ and  
we denote  the derivative $ \pa_\lambda^k := \pa_{\lambda_1}^{k_1} \ldots \pa_{\lambda_{\nu+1}}^{k_{\nu+1}} $.

For a scalar valued function $ \mu : \tLm_0 \subset \R^{\nu+1}  \to  \R $ (for example the Floquet exponents), or valued in $\R^d$, $d \in \N$, which is $ k_0 $-times differentiable with respect to $ \lambda $, we define 
$$
|\mu|^{k_0, \gamma} := |\mu|^{k_0, \gamma}_{\tLm_0} := {\mathop \sum}_{|k| \leq k_0} \g^{|k|} \sup_{\lambda 
\in \tLm_0} | \pa_\lambda^k \mu  (\lambda) | \, . 
$$
This norm extends the Lipschitz-weighted norm introduced in \cite{K1}, \cite{Po2} and used 
in \cite{BBi10},  \cite{BBM-Airy}, \cite{BBM-auto}. 

Given a set $ B $ we denote by $ {\mathcal N}(B, \eta) $ 
the open neighborhood of $ B  $ 
of width $ \eta $ (which is empty if $ B $ is empty) in $\R^\nu \times [\kappa_1, \kappa_2]$, namely 
\begin{equation}\label{definizione incicciottamento insieme}
{\mathcal N}(B, \eta) := \big\{ \lambda \in  \R^\nu \times [\kappa_1, \kappa_2] : {\rm dist}(B, \lambda) \leq \eta \big\}\,.
\end{equation}
Given $j \in \Z$, we set $\langle j \rangle := {\rm max}\{1, |j| \}$ and for any vector  
$\ell = (\ell_1, \ldots, \ell_\nu) \in \Z^\nu $, 
$$
\langle \ell \rangle := {\rm max}\{ 1, |\ell| \} \, , \quad  |\ell| = {\rm max}_{i = 1, \ldots, \nu} |\ell_i| \, .
$$ 
With a slight abuse of notation, 
given $\ell \in \Z^\nu, j \in \Z$, we write $\langle \ell, j \rangle \! := \! {\rm max}\{ 1, |\ell|, |j| \}$. 
\\[1mm]
{\bf Sobolev spaces.} We denote by $H^s(\T^{\nu + 1})$ the Sobolev space of both real and complex valued functions defined by 
$$
\begin{aligned}
H^s := H^s(\T^{\nu + 1}) 
:= \Big\{ & u(\vphi, x)  = \sum_{
\ell \in \Z^\nu, j \in \Z} u_{\ell, j} e^{\ii (\ell \cdot \vphi + j x)} : \\
& \| u \|_s^2 := \sum_{
\ell \in \Z^\nu,
j \in \Z} \langle \ell, j \rangle^{2 s} |u_{\ell, j}|^2 < + \infty \Big\}\,,
\end{aligned}
$$
see \eqref{unified norm}. In the paper we shall use $ H^s $ Sobolev spaces with index $ s $ in 
a finite range of values 
$$ 
s \in [s_0, S ] \, , \quad {\rm where} \quad s_0  := \Big[ \frac{\nu +1}{2} \Big] +1 \in \N \, ,  
$$ 
see \eqref{def:s0}, and 
the largest possible value of $ S $ is fixed 
in the Nash-Moser iteration in section \ref{sec:NM}, see \eqref{valore finalissimo S}.  

In section \ref{sezione operatori tame} we state some abstract lemmata  (for instance Lemmata \ref{lemma:utile}, \ref{Moser norme pesate}) for a Sobolev space $H^s(\T^d)$ of generic dimension $d \in \N $, that we define as 
$$
H^s(\T^d) := \Big\{ u(y) = \sum_{k \in \Z^d} u_k e^{\ii k \cdot y} : \| u \|_s^2 := \sum_{k \in \Z^d} \langle k \rangle^{2 s} |u_k|^2 <  + \infty \Big\}
$$
where  $k = (k_1, \ldots, k_d) \in \Z^d$, $\langle k \rangle := {\rm max}\{ 1 , |k| \}$, $|k| := {\rm max}_{i = 1, \ldots, d} |k_i|$. 
We shall also use the notation $H^s_x : = H^s(\T_x)$ for Sobolev spaces of functions of 
the space-variable $ x \in \T $,  and $H^s_\vphi = H^s(\T^\nu_\vphi)$ for Sobolev spaces of the periodic variable
$ \vphi \in \T^\nu $. 
Moreover we also define the subspace $H_0^1(\T_x)$ of $H^1(\T_x)$ of  functions depending only on the
space variable $ x $ with zero average, i.e. 
\begin{equation}\label{H 01 Tx}
H_0^1(\T_x) := \Big\{ u \in H^1(\T) : \int_\T u(x)\, d x = 0 \Big\}\,. 
\end{equation}
Along the paper we  consider  families of functions $ u(\lambda ) $ in $ H^s $ that are
$ k_0 $-times differentiable with respect to the parameter 
$ \lambda = (\om, \kappa)  \in \mathtt \Lambda_0 \subset  \R^{\nu + 1}  $, and for which we introduce the following 
weighted Sobolev norm (see \eqref{norma pesata derivate funzioni}): 
for $ \g \in (0,1) $,
\be\label{simplex-norm}
\| u \|_s^{k_0, \gamma} :=
 {\mathop \sum}_{| k | \leq k_0} \gamma^{|k |} {\rm sup}_{\lambda \in \mathtt \Lambda_0}  
\| \partial_{\lambda}^k u (\lambda) \|_s \, .
\ee
The meaning of the indices $ k_0, \gamma, s $ is the following:
\begin{enumerate}
\item
The index $ k_0 \in \N $ denotes that $ u(\lambda )$ 
is  $ {k_0} $-times differentiable 
with respect to the parameter $ \lambda  $. The index $ k_0 $
is {\it fixed} in section \ref{sec:degenerate KAM}.
It depends only on properties  of the linear frequencies $ \omega_j (\kappa ) $   
in \eqref{linear frequencies}, and the choice of the tangential sites $ {\mathbb S}^+ $,  
and it does {\it not} vary along the whole paper.
When used in other contexts the index $ k_0 $ always indicates  that the operators, functions, frequencies, eigenvalues, 
etc.,  are $ {k_0} $-times differentiable  with respect to the parameter $ \lambda  $. 
\item 
The parameter $ \gamma \in (0,1) $ is the diophantine constant of the frequencies
$ |\om \cdot \ell | \geq \g \langle \ell \rangle^{-\tau} $, $ \forall \ell \in \Z^\nu \setminus \{0\} $, and similarly for the
first and second order Melnikov non-resonance conditions. Such quantities enter at the denominators 
in the  solutions of homological equations like \eqref{basic-KAM0}, and therefore 
any derivative $ \pa_\omega $ produces the appearance of a factor $ \gamma^{-1} $, as 
explained for \eqref{basic-KAM}. This motivates the use of the weights $ \gamma^{|k|} $
in  \eqref{simplex-norm}, and  similarly, in other contexts,  before a $ \pa_\lambda^k  $ 
derivative of  operators, functions, frequencies, eigenvalues, etc.... 
Along the paper $ \gamma = O( \e^a) $ with $  a > 0 $ as small as wanted (actually we could  take just $ \g = o(1) $ as $ \e \to 0 $).  
\item  The index $ s $  denotes the Sobolev index  of the norm $ \| \ \|_{s} $.
\end{enumerate}
{\bf Pseudo-differential operators and norms.} 
A pseudo-differential operator  with symbol $ a(x, \xi ) $ is denoted by $ {\rm Op} (a) $ or $ a(x, D )$, see 
 Definitions \ref{def:Ps1}, \ref{def:Ps2}. The set of symbols $ a(x, \xi) $ of order $ m $ is denoted by $ S^m $
and the class of the corresponding pseudo differential operators  by  $OPS^m $. We also set
$$ 
OPS^{-\infty} =  \cap_{m \in \R} OPS^m \, .
$$ 
Along the paper we have to consider  symbols $ a(\lambda, \vphi, x, \xi ) $ that depend on 
$ \vphi \in \T^\nu $ and on a parameter  $ \lambda \in {\mathtt \Lambda}_0 \subset \R^{\nu+1} $. 
The symbol $ a $ is $ k_0 $-times differentiable with respect to $ \lambda $ and $ {\mathcal C}^\infty $ with respect to 
$ (\vphi, x, \xi) $.  
For the corresponding family of pseudo differential operators $ A(\lambda) = a(\lambda, \vphi, x, D) $
we introduce in Definition  \ref{def:pseudo-norm} the norms 
\be\label{pseudo-simple}
\norma A \norma_{m, s, \alpha}^{k_0, \gamma} := \sum_{|k| \leq k_0} \gamma^{|k|} 
{\rm sup}_{\lambda \in {\mathtt \Lambda}_0} \norma\partial_\lambda^k A(\lambda) \norma_{m, s, \alpha}  
\ee
 indexed by $ k_0 \in \N $, 
$ \gamma \in (0,1) $, $ m \in \R $, $ s \geq s_0  $, $ \a \in \N $,
where
$$
\norma A (\lambda) \norma_{m, s, \a} := {\rm max}_{0 \leq \beta  \leq \a} \sup_{\xi \in \R} \|  \partial_\xi^\beta 
a(\lambda, \cdot, \cdot, \xi )  \|_{s} \langle \xi \rangle^{-m + \beta} \, . 
$$
The meaning of the indices $ k_0 $, $ \gamma $, $m$, $s $, $ \alpha $ is the following: 
\begin{enumerate}
\item
The index $ k_0 \in \N $ denotes that the operators $ A(\lambda) $ (i.e. the symbols $ a(\lambda, \cdot )$) are $ {k_0} $-times differentiable 
with respect to the parameters $ \lambda = (\om, \kappa) $. 
\item 
The parameter $ \gamma \in (0,1) $ is the diophantine constant of the frequencies
$ |\om \cdot \ell | \geq \g \langle \ell \rangle^{-\tau} $, $ \forall \ell \in \Z^\nu \setminus \{0\} $, and similarly for the
first and second order Melnikov non-resonance conditions. 
\item The parameter $ m \in \R $ denotes the order of the pseudo-differential operator $ A \in OPS^m $. 
\item  The constant $ s $ 
denotes the Sobolev index  of the norm $ \|  \partial_\xi^\beta a(\lambda, \cdot, \cdot, \xi )  \|_{s} $
which measures the regularity of the function
$ (\vphi, x ) \mapsto  \partial_\xi^\beta a(\lambda, \vphi, x, \xi )  $. 
It varies  in a finite range $ s \in [s_0, S ] $ where $ s_0 $ 
is fixed in \eqref{def:s0} and  the largest $ S $ is fixed 
in  section \ref{sec:NM}, see \eqref{valore finalissimo S}.  
\item The constant $ \a \in \N $  is the number of $ \pa_\xi $ derivatives 
that we estimate of a symbol $ a(x, \xi) $. 
 In section \ref{linearizzato siti normali} we take   $ \a \approx M $ where $ M $ is the number of
decoupling steps performed in section \ref{sec:decoupling}. The constant $ M $ is fixed in \eqref{relazione mathtt b N}.
The important point is that the largest values of $ \a, M $ used along the paper do not depend on the Sobolev index $ s $. 
\end{enumerate}
{\bf $ {\mathcal D}^{k_0}$-tame and $ {\mathcal D}^{k_0}$-modulo-tame operators.} In  Definition \ref{def:Ck0} we introduce the class of linear operators $A = A(\lambda)$ satisfying tame estimates of the form
$$
\sup_{|k| \leq k_0} \sup_{\lambda \in \mathtt \Lambda_0} \gamma^{ |k|}
\| (\partial_\lambda^k A(\lambda)) u \|_s \leq  {\mathfrak M}_A(s_0) \| u \|_{s+\s} + 
{\mathfrak M}_A (s) \| u  \|_{s_0+\s} \,,
$$
that we call ${\mathcal D}^{k_0}$-$\sigma$-tame operators. The constant ${\mathfrak M}_A (s)$ is called the tame constant of the operator $A$. When the ``loss of derivatives"  $ \sigma = 0 $  we simply call a $ {\mathcal D}^{k_0} $-$ 0 $-tame operator to  be $ {\mathcal D}^{k_0} $-tame. 

In Definition \ref{def:op-tame}  we introduce the subclass of ${\mathcal D}^{k_0}$-modulo tame operators $A = A(\lambda)$ such that for any $k \in \N^{\nu+1}$, $|k| \leq k_0$, the majorant operator $|\partial_\lambda^k A|$ satisfies  the tame estimates  
$$
\sup_{|k| \leq k_0} \sup_{\lambda \in \mathtt \Lambda_0} \gamma^{ |k|}
\| | \partial_\lambda^k A | u\|_s \leq  
{\mathfrak M}_{A}^\sharp (s_0) \| u \|_{s} +
{\mathfrak M}_{A}^\sharp (s) \| u \|_{s_0}\,. 
$$
The majorant operator $ |A| $ is introduced in Definition \ref{def:maj}-1, by 
taking the modulus of the matrix entries of the matrix which represents the operator 
$ A $ with respect to the exponential basis. 
We refer to ${\mathfrak M}_{A}^\sharp (s)$ as the modulo tame constant of the operator $A$. 

\smallskip

Finally we use the following notation:  
\begin{enumerate}
\item 
$ a \leq_{s, \a, M} b $ means 
that $a \leq C(s, \a, M ) b$ for some constant $C(s, \a, M ) > 0$ depending on 
the Sobolev index $ s $, and the constants $ \a, M $.  Sometimes, along the paper, we omit to write the dependence $ \leq_{s_0, k_0} $ with respect to $ s_0, k_0 $, 
 because $ s_0 $ (defined in \eqref{def:s0}) and $ k_0 $
 (determined  in section \ref{sec:degenerate KAM})  are considered as fixed constants. 
\item 
$ a \lessdot  b $ means that 
 $a \leq C b $ for some absolute constant which depends only on the data of the problem.
\end{enumerate}

\chapter{Functional setting} \label{sec:prelim}

We regard a function  $ u(\varphi , x) \in L^2 (\T^\nu \times \T, \C) $ of space-time also as a 
 $ \vphi $-dependent family of  functions $ u(\vphi, \cdot ) \in L^2 (\T_x, \C) $ that we expand in Fourier series as
\be\label{function-Fourier}
u(\vphi, x ) =   \sum_{j' \in \Z} u_{j'} (\vphi) e^{\ii j' x } =
\sum_{\ell' \in \Z^\nu, j' \in \Z}  u_{\ell',j'}  e^{\ii (\ell' \cdot \vphi + j' x)}   \, . 
\ee 
Along the paper we  denote the Fourier coefficients  $ u_{\ell,j} $, $ u_{j}(\vphi) $ 
of  the function 
$ u(\vphi, x) $ (with respect to the space variables $ (\vphi, x) $ or $ x $, respectively) also 
as   $ \widehat u_{\ell,j} $,  $  \widehat u_{j} (\vphi)  $.
We also consider real valued functions $ u(\vphi, x)  \in \R $.  
When no confusion appears we will denote simply by 
$ L^2 $, $ L^2 (\T^\nu \times \T ) $, $ L^2_x := L^2 (\T_x) $  either the spaces of real or complex valued 
$ L^2 $-functions. 

The Sobolev norm $ \| \ \|_s $ defined in \eqref{unified norm} 
is equivalent to 
\be\label{Sobolev norm}
\| u \|_s \simeq \| u \|_{H^s_\ph L^2_x} + \| u \|_{L^2_\ph H^s_x} \, . 
\ee
\begin{definition}
Given a function  $ u \in L^2 (\T^\nu \times \T ) $  as in \eqref{function-Fourier}, we define the majorant function
\begin{equation}\label{funzioni modulo fourier}
\norma u \norma (\vphi, x) :=  \sum_{\ell \in \Z^\nu, j \in \Z} |u_{\ell, j}| e^{\ii (\ell \cdot \vphi + j x)} \, .
\end{equation}
\end{definition}
Note that  the Sobolev norms of $ u $ and $ \norma u \norma $ are the same, 
i.e. 
\be\label{Soboequals}
\| u \|_s = \| \norma u \norma \|_s\,.
\ee
We consider also family of Sobolev functions  $ \lambda  \mapsto u(\lambda) \in H^s $ which are 
$ {k_0} $-times differentiable with respect to a parameter  
$$
 \lambda := (\om, \kappa)  \in \mathtt \Lambda_0 \subset  \R^{\nu + 1}  \,.
 $$ 
For $ \g \in (0,1) $ we define the weighted Sobolev norm 
\begin{equation}\label{norma pesata derivate funzioni}
\| u \|_s^{k_0, \gamma} :=
 {\mathop \sum}_{| k | \leq k_0} \gamma^{|k |} {\rm sup}_{\lambda \in \mathtt \Lambda_0}  
\| \partial_{\lambda}^k u (\lambda) \|_s 
\end{equation}
and we use the same notation $ \| u \|_s^{k_0, \gamma} $ for a Sobolev 
function $ u \in H^s_{\vphi} $ of the $ \vphi $ variable only.  

For a family of functions $u(\lambda, \cdot ) : \T^d \to \C $,  which is $ k_0 $-times differentiable with respect
to $ \lambda $,  we define the $ {\mathcal C}^s $-weighted norm
\begin{equation}\label{norma pesata C^s}
\| u \|_{{\mathcal C}^s}^{k_0, \gamma} :=
 {\mathop \sum}_{| k | \leq k_0} \gamma^{|k |} {\rm sup}_{\lambda \in \mathtt \Lambda_0}  
\| \partial_\lambda^k u (\lambda) \|_{{\mathcal C}^s} 
\end{equation}
(we use it in section \ref{sub:integral-op} to  functions $ K (\lambda, \cdot ) $ with $ d = \nu + 1 $). 

We have  the following interpolation lemma. 
\begin{lemma}\label{interpolazione fine}
Let $ a_0, b_0 \geq 0$ and $ p,q >  0 $. For all $\epsilon > 0 $ there exists a constant 
$ C(\epsilon) := C(\epsilon, p, q ) > 0 $, which satisfies $ C(1) < 1 $, such that 
\begin{align} 
& \| u \|_{a_0 + p} \| v \|_{b_0 + q} \leq  \epsilon \| u \|_{a_0 + p + q} \| v \|_{b_0} + C(\epsilon)\| u \|_{a_0} \| v \|_{b_0 + p + q} \label{interpolation estremi fine} \\
& \| u \|_{a_0 + p}^{k_0,\g} \| v \|_{b_0 + q} \leq \epsilon \| u \|_{a_0 + p + q}^{k_0,\g} \| v \|_{b_0} + C(\epsilon)
\| u \|_{a_0}^{k_0, \g} \| v \|_{b_0 + p + q} \, . \label{interpolation estremi fine Ck0}
\end{align}
\end{lemma}

\begin{proof}
By interpolation
$$
\| u \|_{a_0+p} \leq \| u \|_{a_0}^\mu  \| u \|_{a_0+p+q}^{1-\mu} \, , \  \mu := \frac{q}{p+q} \, , \quad 
\| v \|_{b_0+q} \leq \| v \|_{b_0}^\eta  \| v \|_{b_0+p+q}^{1-\eta} \, , \  \eta := \frac{p}{p+q} \, . 
$$
Hence, noting that $ \eta + \mu = 1  $, we have  
$$
\| u \|_{a_0+p} \| v \|_{b_0+q}  \leq 
( \| u \|_{a_0+p+q}  \| v \|_{b_0} )^{\eta}  ( \| u \|_{a_0}  \| v \|_{b_0+p+q} )^\mu  \, .  
$$
By the asymmetric Young inequality we get, for any $ \epsilon > 0  $, 
$$
\| u \|_{a_0+p} \| v \|_{b_0+q}  \leq \epsilon \| u \|_{a_0+p+q}  \| v \|_{b_0} + C(\epsilon, p, q)   \| u \|_{a_0}  \| v \|_{b_0+p+q}
$$
where $ C(\epsilon, p, q ) := \mu ( \eta \slash \epsilon )^{\frac{\eta}{\mu}} = \frac{q}{p + q}  
\big( \frac{p}{\epsilon (p + q)} \big)^{ p / q }  $. 
Note that for $ \epsilon = 1 $ the constant $ C(1, p, q) < 1 $. 

The estimate  \eqref{interpolation estremi fine Ck0} follows by \eqref{interpolation estremi fine} recalling 
\eqref{norma pesata derivate funzioni}.
\end{proof}

For any $ K \in \N^+  $,   we  introduce the smoothing operators, 
\begin{equation}\label{definizione smoothing operators}
(\Pi_K u)(\vphi, x) := \sum_{|(\ell, j)| \leq K} u_{\ell j} e^{\ii (\ell \cdot \vphi + j x)}\,, \qquad \Pi_K^\bot := {\rm Id }- \Pi_K\, , 
\end{equation}
which satisfy the usual smoothing properties 
\begin{align}\label{smoothing-u1}
\|\Pi_{K} u \|_{s + b}^{k_0, \gamma} 
\leq K^{b} \| u \|_{s}^{k_0, \gamma} \, , 
\qquad \ \,   \|\Pi_K^\bot u\|_{s}^{k_0, \gamma} 
& \leq K^{- b} \| u \|_{s + b}^{k_0, \gamma} \,   , \qquad \forall s, b \geq 0\,.
\end{align}

\noindent
{\bf Linear operators.} Let $ A : \T^\nu \mapsto {\mathcal L}( L^2(\T_x))  $, 
$ \vphi \mapsto A(\vphi) $,   be a $ \vphi $-dependent family of linear 
operators acting on $ L^2 (\T_x) $. We regard $ A $  also 
 as an operator (that for simplicity we denote by $A $ as well)
 which acts on functions $ u(\varphi , x) $ of space-time, i.e. 
 we consider the operator 
$ A \in {\mathcal L}(L^2(\T^\nu \times \T ) )$ 
defined by
$$
( A u) (\varphi , x) := (A(\varphi) u(\varphi, \cdot ))(x) \, .  
$$
We say that an operator $ A $ is {\it real} if it maps real valued functions into real valued functions. 

We represent 
a real  operator acting on $ (\eta, \psi) \in L^2(\T^{\nu + 1}, \R^2) $ by a matrix 
\begin{equation}\label{cal R eta psi}
{\mathcal R} \begin{pmatrix}
\eta \\
\psi
\end{pmatrix} = \begin{pmatrix}
A & B \\
C & D
\end{pmatrix}
 \begin{pmatrix}
\eta \\
\psi
\end{pmatrix} 
\end{equation}
where $A, B, C, D$ are real operators acting on  the scalar valued components 
$ \eta, \psi \in L^2(\T^{\nu + 1}, \R)$.

The action of an operator $ A \in {\mathcal L} (L^2(\T^\nu  \times \T )) $ on a function $ u $ as in \eqref{function-Fourier} 
is 
\begin{equation}\label{matrice operatori Toplitz}
\begin{aligned}
A u (\vphi, x) & 
= {\mathop \sum}_{j , j' \in \Z} A_j^{j'}(\vphi) u_{j'}(\vphi) e^{\ii j x} \\ 
& = 
 \sum_{\ell \in \Z^\nu, j \in \Z} \sum_{\ell' \in \Z^\nu, j' \in \Z} A_j^{j'}(\ell - \ell') u_{\ell', j'} e^{\ii (\ell \cdot \vphi + j x)} \, . 
 \end{aligned}
\ee
We shall identify an operator $ A $ with the matrix $ \big( A^{j'}_j (\ell- \ell') \big)_{j, j' \in \Z, \ell, \ell' \in \Z^\nu } $. 

Note that the differentiated 
operator $ \pa_{\vphi_m} A (\vphi)  $, $ m = 1, \ldots, \nu $,  is represented by the 
matrix elements $ \ii (\ell_m- \ell_{m}') A^{j'}_j (\ell- \ell') $,
and the commutator $ [\pa_x, A ] := \pa_x \circ A - A \circ \pa_x $ is represented by the matrix 
with entries $ \ii (j - j') A^{j'}_j (\ell- \ell') $.

\begin{definition}\label{def:maj} 
Given a linear operator $ A $ as in \eqref{matrice operatori Toplitz} we define the operator
\begin{enumerate}
\item $ | A | $  {\bf (majorant operator)}   
whose matrix elements are $  | A_j^{j'}(\ell - \ell')| $,  
\item $ \Pi_N A $, $ N \in \N  $ {\bf (smoothed operator)} 
whose matrix elements are
\be\label{proiettore-oper}
 (\Pi_N A)^{j'}_j (\ell- \ell') := 
 \begin{cases}
 A^{j'}_j (\ell- \ell') \quad {\rm if} \quad   |\ell - \ell' | \leq N \\
 0  \qquad \qquad \quad {\rm  otherwise} \, .
 \end{cases} 
\ee
We also denote $ \Pi_N^\bot := {\rm Id} - \Pi_N $, 
\item $ \langle \pa_\vphi \rangle^{b}  A $, $ b \in \R $,  
whose matrix elements are  $  \langle \ell - \ell' \rangle^b A_j^{j'}(\ell - \ell') $.
\end{enumerate}
\end{definition}

\begin{lemma}\label{compositionAB} 
Given linear operators $ A $, $ B$ we have 
\begin{align}\label{disuguaglianza importante moduli} 
\| | A + B| u \|_s \leq \||  A | \, \norma u \norma  \|_s + \|| B| \, \norma u \norma \|_s  \, , \quad
\|| A B| u \|_s \leq \|| A| | B| \, \norma u \norma \|_s \,.
\end{align}
\end{lemma}

\begin{proof} 
The first inequality in \eqref{disuguaglianza importante moduli} follows by
$$
\begin{aligned}
\| |A + B| u \|_s^2 
& \leq  \sum_{\ell, j} \langle \ell, j \rangle^{2 s} \Big( \sum_{\ell', j'} 
| A_j^{j'} (\ell - \ell' )|  |u_{\ell', j'}| + |B_j^{j'}(\ell - \ell')| |u_{\ell', j'}|  \Big)^2 \\
& = \big\| |A | [ \norma u \norma ] + |B | [ \norma u \norma ] \big\|_s^2 \, . 
 \end{aligned}
 $$
The second inequality in \eqref{disuguaglianza importante moduli} follows by
\begin{align}
\| |A B| u \|_s^2 & \leq  \sum_{\ell, j} \langle \ell, j \rangle^{2 s} \Big( \sum_{\ell', j'} | (A B)_j^{j'}(\ell - \ell')| |u_{\ell', j'}|  \Big)^2 \nonumber\\
& = \sum_{\ell, j} \langle \ell, j \rangle^{2 s} \Big( \sum_{\ell', j'} \Big| \sum_{\ell_1, j_1} A_j^{j_1}(\ell - \ell_1) B_{j_1}^{j'}(\ell_1 - \ell') \Big| |u_{\ell', j'}|  \Big)^2 \nonumber\\
& \leq \sum_{\ell, j} \langle \ell, j \rangle^{2 s} \Big(  \sum_{\ell_1, j_1} | A_j^{j_1}(\ell - \ell_1)| \sum_{\ell', j'}  | B_{j_1}^{j'}(\ell_1 - \ell')|  
|u_{\ell', j'}|  \Big)^2 \nonumber\\
& = \sum_{\ell, j} \langle \ell, j \rangle^{2 s} \Big(  \sum_{\ell_1, j_1} | A_j^{j_1}(\ell - \ell_1)| 
\widehat{\big( | B|[\norma u \norma] \big) }_{\ell_1, j_1}  \Big)^2  = \| |A| \big(| B| [\norma u \norma] \big) \|_s^2 \, . \nonumber
\end{align}
The lemma is proved. 
\end{proof}

\begin{definition}\label{def:even}
{\bf (Even operator)}  A linear operator $ A $ as in \eqref{matrice operatori Toplitz} 
is  {\sc even} if each $ A(\vphi)  $,  $ \vphi \in \T^\nu $,  leaves invariant  the space of  functions even in $  x $.
\end{definition}

Since the Fourier coefficients of an even function satisfy $ u_{- j}  = u_j  $, $ \forall j \in \Z $, 
we have that 
\be\label{even operators Fourier}
A \ \text{is \ even} \quad \Longleftrightarrow \quad   
 \forall \vphi \in \T^\nu \, , \ A_j^{j'}(\vphi) + A_j^{- j'}(\vphi) = A_{- j}^{j'}(\vphi) + A_{- j}^{- j'}(\vphi),  \ \forall j, j' \in \Z \, . 
\ee

\begin{definition} {\bf (Reversibility) }  
An operator ${\mathcal R}$ as in \eqref{cal R eta psi} is
\begin{enumerate}
\item {\sc reversible} if 
$ {\mathcal R}(- \vphi ) \circ \rho = - \rho \circ {\mathcal R}(\vphi ) $,  $ \forall \vphi \in \T^\nu $,   where the involution $ \rho $ is defined in \eqref{defS}, 
\item
{\sc reversibility preserving} if  $ {\mathcal R}(- \vphi ) \circ \rho =  \rho \circ {\mathcal R}(\vphi ) $,  $ \forall \vphi \in \T^\nu $.
\end{enumerate}
\end{definition}
Conjugating the linear operator $ {\mathcal L} := {\om \cdot \pa_\vphi} + A(\vphi) $ by a family of invertible linear maps $ \Phi(\vphi) $ we get
the transformed operator 
$$
\begin{aligned}
& {\mathcal L}_+ :=  \Phi^{-1}(\vphi) {\mathcal L} \Phi (\vphi)  = {\om \cdot \pa_\vphi} + A_+ (\vphi) \, , \\ 
& A_+ (\vphi) := 
 \Phi^{-1}(\vphi) (\om \cdot \pa_\vphi \Phi (\vphi) )  + \Phi^{-1}(\vphi) A (\vphi) \Phi (\vphi) \, .  
 \end{aligned}
$$
It results that the conjugation of an even and reversible operator with an operator 
$ \Phi (\vphi) $ which is even and reversibility preserving is even and  reversible.
An operator $ {\mathcal R} $ as in \eqref{cal R eta psi} is
\begin{enumerate}
\item
 reversible if and only if $ \vphi \mapsto A (\vphi), D (\vphi) $ are odd and $  \vphi \mapsto B(\vphi), C(\vphi) $ are even.
\item 
reversibility preserving if and only if
$  \vphi \mapsto A (\vphi), D (\vphi) $ are even and $  \vphi \mapsto B(\vphi), C(\vphi) $ are odd. 
\end{enumerate}
From section \ref{complex-coordinates} on, it is convenient to consider a
real operator $ {\mathcal R} $ as in \eqref{cal R eta psi}, which acts on the real variables
$ (\eta, \psi) \in \R^2 $, as a 
linear  operator which acts on the complex variables
\be\label{complex-coor}
u := \eta + \ii \psi \, , \quad   \bar u := \eta - \ii \psi \, , \qquad i.e. \ 
\ \eta = (u + \bar u) \slash 2 \, , \quad \psi = (u - \bar u) \slash (2 \ii ) \, . 
\ee
We get that a  {\it real} operator acting in the complex coordinates $ (u, \bar u) $ has the form  
\be
\begin{aligned}\label{operatori in coordinate complesse}
& {\bf R} := \begin{pmatrix}
{\mathcal R}_1 & {\mathcal R}_2 \\
\overline{\mathcal R}_2 & \overline{\mathcal R}_1
\end{pmatrix} \, , \\
& 
{\mathcal R}_1 := \frac12 \big\{(A + D) - \ii (B - C) \big\} \, , \quad   {\mathcal R}_2 := \frac12 \big\{ (A - D) + \ii(B + C) \big\} 
\end{aligned}
\ee 
where the operator $ \overline{A} $ is defined by 
\be\label{def:barA} 
\overline{A}(u) :=  \overline{A( \bar u)} \, .
\ee
It holds $ \overline{AB}  =  \overline{A} \ \overline{B} $. 

The composition of real operators is another real operator. 

A real  operator ${\bf R}$ as in  \eqref{operatori in coordinate complesse} is even 
if the operators ${\mathcal R}_1$, ${\mathcal R}_2$ are even. 

In the complex coordinates \eqref{complex-coor}
the involution $ \rho $ defined in \eqref{defS} is the map $ u \mapsto \bar u $. Thus 

\begin{lemma}\label{R-op:rev}
The real operator $ {\bf R} $ in \eqref{operatori in coordinate complesse}  is
\begin{enumerate}
 \item reversible if and only if $ {\mathcal R}_1 (- \vphi) = - \overline{{\mathcal R}_1} ( \vphi) $, 
 $ {\mathcal R}_2(- \vphi) = - \overline{{\mathcal R}_2} ( \vphi) $,  $ \forall \vphi \in \T^\nu $,  
 \item reversibility preserving if and only if 
 $ {\mathcal R}_1 (- \vphi) = \overline{{\mathcal R}_1} ( \vphi) $, 
 $ {\mathcal R}_2(- \vphi) =  \overline{{\mathcal R}_2} ( \vphi) $, $ \forall \vphi \in \T^\nu $.
\end{enumerate}
\end{lemma}

\section{Pseudo-differential operators and norms}\label{sec:pseudo}

Pseudo-differential operators on the torus may be seen as a particular case (see Definition \ref{def:Ps2})
of pseudo-differential operators on $ \R^n $, as developed for example in \cite{Ho1}. 
It is also convenient to define them also  through Fourier series, see Definition \ref{def:Ps1},  
for which we refer  to \cite{SV}.  

Given a function $ a : \Z \to \C $ 
we denote the discrete derivative by $ 
(\Delta_j a) (j) := a(j + 1 ) - a(j) $.
For $ \b \in \N $ we denote by $ \D_j^\b := \D_j \circ \ldots \circ \D_j $ 
 the composition of $  \b $-discrete derivatives. 

\begin{definition}\label{def:Ps1} {\bf ($\Psi {\rm DO}$$1 $)}
Let $ u = \sum_{j \in \Z} u_j e^{\ii j x } $. 
A linear operator $ A $ defined by
\be\label{action-A}
(Au) (x) 
:= {\mathop \sum}_{j \in \Z} a(x,j) u_j e^{\ii j x }  
\ee
is  called {\sl pseudo-differential} of order $ \leq m $ if its symbol $ a (x,j) $ is 
$ 2 \pi $-periodic and $ {\mathcal C}^\infty $-smooth in $ x $, and satisfies the inequalities 
\be\label{symbol-pseudo1}
\big| \pa_x^\a \Delta_j^\b a (x,j) \big| \leq C_{\a,\b} \langle j \rangle^{m-\b}  \, , 
\quad \forall \a, \b \in \N \, . 
\ee
\end{definition}


We also remark that, given an operator $ A $, we recover its symbol  by 
\be\label{definizione simbolo}
a (x,j) = e^{- \ii j x} (A [e^{\ii jx}]) .
\ee
When the symbol $ a (x) $ is independent of $ j $, the operator $ A = {\rm Op} (a) $ is 
the  multiplication operator for the function $ a(x)$, i.e $ A : u (x) \mapsto a ( x) u(x )$. 
In such a case we shall also denote   $ A = {\rm Op} (a)  = a (x)  $.

\begin{definition} \label{def:Ps2} {\bf ($\Psi {\rm DO}$$2 $)}
A linear operator $ A $ is called {\sl pseudo-differential} of order $ \leq m $ if its symbol $ a (x,j) $ is 
the restriction to $ \R \times \Z $ of a function $ a (x, \xi ) $ which is $ {\mathcal C}^\infty $-smooth on $ \R \times \R $,
 $ 2 \pi $-periodic in $ x $, and satisfies the inequalities
\be\label{symbol-pseudo2}
\big| \pa_x^\a \pa_\xi^\b a (x,\xi ) \big| \leq C_{\a,\b} \langle \xi \rangle^{m - \b} \, , \quad \forall \a, \b \in \N \, .
\ee
We call $ a(x, \xi ) $ the symbol of the operator $ A $, that we denote 
$$ 
A = {\rm Op} (a) = a(x, D) \, , \quad D := D_x := \frac{1}{\ii} \pa_x \, . 
$$ 
We denote by $ S^m $ 
the class of all the symbols $ a(x, \xi ) $ satisfying \eqref{symbol-pseudo2}, and by $ OPS^m $ 
the set of pseudo-differential  operators of order $ m $.  We set $ OPS^{-\infty} := \cap_{m  \in \R} OPS^{m} $.
\end{definition}

Definitions \ref{def:Ps1} and \ref{def:Ps2} are equivalent because  
any discrete symbol $ a : \R \times \Z \to \C $ satisfying \eqref{symbol-pseudo1} 
can be extended to a $ {\mathcal C}^\infty $-symbol $ \wtilde a : \R \times \R  \to \C  $ satisfying  
\eqref{symbol-pseudo2}, see section 7.2 in \cite{SV}. 
It is sufficient to proceed as follows.
Given a function $ \s : \Z \to \C $  we define the  $ {\mathcal C}^\infty $-extension 
\begin{equation}\label{simbolo esteso discreto continuo}
\wtilde \s : \R \to \C \, , \quad \wtilde \s (\xi) := {\mathop \sum}_{j \in \Z} \s ( j) \zeta (\xi - j) \, , \ \ \forall \xi \in \R \, \, , 
\end{equation}
where $ \zeta := \widehat \theta \in {\mathcal S}(\R) $ (Schwartz class) is the Fourier transform of a function $ \teta \in {\mathcal D}(\R) $
(test functions) such that
$ {\rm supp} (\teta) \subset [-2/3, 2/3] $, $ \theta (x) + \theta (x-1) = 1 $, $ \forall x \in [0,1] $, and $ \sum_{j\in \Z } \theta (x + j) = 1 $. 
It results that $ \zeta (k ) = \delta_{0k} $, $ \forall k \in \Z $, namely $ \zeta (0 ) = 1 $ and $ \zeta (k ) = 0 $, $ \forall k \neq 0 $,
so that $ \wtilde \s (k) = \s (k)  $, $ \forall k \in \Z $. 
Moreover 
 there are positive constants $ c_\b' > 0 $, independent of $ \sigma $, such that (see Lemma 7.1.1 in \cite{SV})
\begin{equation}\label{simbolo esteso discreto continuo stima}
|\D^\b_j \s (j)| \leq c_\b \langle j \rangle^{m- \b} \quad \Longleftrightarrow \quad 
| \pa_\xi^\b {\wtilde \s} (\xi)| \leq c_\b' c_\b \langle \xi \rangle^{m- \b} \, .  
\end{equation}
Definition \ref{def:Ps2} 
 is more convenient to get  basic results 
concerning composition, asymptotic expansions, 
\ldots of pseudo-differential operators, that we recall below. 
We underline that, in the sequel, also when we use of the continuous symbol $  a (x, \xi )$, 
we think $ {\rm Op}(a)  $ to act only on $ 2 \pi $-periodic functions $ u(x) $ as in \eqref{action-A}. 

\smallskip

We shall use 
the following notation, used also in \cite{AB}. 
For any $m \in \R \setminus \{ 0\}$, we set
\begin{equation}\label{definizione |D| m}
|D|^m := {\rm Op}(\chi(\xi) |\xi|^m)\,,
\end{equation}
where $\chi \in {\mathcal C}^\infty(\R, \R)$ is an even and positive cut-off function such that 
\begin{equation}\label{cut off simboli 1}
\chi(\xi) = \begin{cases}
0 & \quad \text{if } \quad |\xi| \leq \frac13 \\
1 & \quad \text{if} \quad  \ |\xi| \geq \frac23\,,
\end{cases} \qquad \partial_\xi \chi(\xi) > 0 \quad \forall  \xi \in \Big(\frac13, \frac23 \Big) \, . 
\end{equation}

\begin{lemma}\label{even:pseudo}
Let $ A := {\rm Op} (a) $ be a pseudo-differential operator. Then the following holds:
\begin{enumerate}
\item 
If the symbol $a$ satisfies $ a(-x, - \xi) = a(x,  \xi) $, then $A$ is even. 
\item
Let $g(\xi)$ be a Fourier multiplier satisfying $g(\xi) = g(- \xi)$. Then if $A = {\rm Op}(a)$ is even, the operator ${\rm Op}(a(x, \xi) g(\xi)) = {\rm Op}(a) \circ {\rm Op}(g)$ is an even operator. 
\item 
$A$ is real if and only if  the symbol  $ \overline{a(x, - \xi)} = a(x,  \xi) $. 
\item 
The operator $ \overline{A} $ defined in \eqref{def:barA} is pseudo-differential and its   symbol
is $ \overline{ a (x, -\xi)} $. 
\end{enumerate}
\end{lemma}
We first recall some fundamental properties of  pseudo-differential operators. 

\smallskip

\noindent
{\bf Composition of pseudo-differential operators.} 
If $ A = a( x, D) \in OPS^{m} $, $ B = b(x, D) \in  OPS^{m'}  $, $ m , m' \in \R $,  are pseudo-differential operators 
with symbols $ a \in S^m $, $ b \in S^{m'} $ then the composition operator 
$ A B := A \circ B = \sigma_{AB} ( x, D) $ is a pseudo-differential operator
with  symbol 
\be\label{lemma composition}
\sigma_{AB}( x, \xi) = \sum_{j \in \Z} a(x, \xi + j ) \widehat b(j, \xi) e^{\ii j x  } = 
\sum_{j, j' \in \Z} \widehat a(j' - j, \xi + j) \widehat b(j, \xi) e^{\ii j' x}
\ee
where $ \widehat \cdot $ denotes the Fourier coefficients of the symbols $ a(x, \xi) $ and $  b(x, \xi)  $ with respect to $  x $.  
The symbol $\sigma_{AB} $ has the following asymptotic expansion
\be\label{composition pseudo}
\sigma_{AB}(x, \xi) \sim {\mathop \sum}_{\b \geq 0} \frac{1}{\ii^\b \b !} \partial_\xi^\b a(x, \xi)  \partial_x^\b b(x, \xi) \, , 
\ee
that is, $ \forall N \geq 1  $,  
\be\label{expansion symbol}
\begin{aligned}
\s_{AB} (x, \xi) & = \sum_{\b =0}^{N-1} \frac{1}{  \b ! \ii^\b }  \pa_\xi^\b a (x, \xi) \, \pa_x^\b b (x, \xi) + r_N (x, \xi)
\qquad {\rm where} \\ 
 r_N & := r_{N, AB} \in S^{m + m' - N }  \, .
\end{aligned}
\ee
The remainder $ r_N $ has the explicit formula 
\be\label{rNTaylor}
r_N (x, \xi) := \frac{1}{(N-1)! \, \ii^N} \int_0^1 (1- \t )^{N-1}  
\sum_{j \in \Z} (\pa_\xi^N a)(x, \xi + \t j ) \widehat {\pa_x^N b} (j , \xi) e^{\ii j x } \, d \tau  \, .
\ee

\smallskip

\noindent
{\bf Adjoint of a pseudo-differential operator.}
If  $ A =  a(x, D)  \in OPS^m $ is a pseudo-differential operator with symbol $ a \in S^m $, then its $L^2$-adjoint is the
pseudo-differential operator 
\begin{equation}\label{simbolo aggiunto senza tempo}
A^* = {\rm Op}(a^*) \qquad {\rm with \ symbol}  
\qquad a^*(x, \xi) := \overline{{\mathop \sum}_{j \in \Z} \widehat a(j, \xi - j) e^{\ii j x }}\,.
\end{equation}
\noindent
{\bf Families of pseudo-differential operators.} 
We consider $ \vphi $-dependent families of pseudo-differential operators
\be\label{come-agisce-A}
(A u)(\vphi, x) = {\mathop \sum}_{j \in \Z} a(\vphi, x, j) u_j(\vphi) e^{\ii j x } 
\ee
where  the symbol  $ a(\vphi, x, \xi ) $ is $ {\mathcal C}^\infty $-smooth also in $ \vphi $. 
We still denote $ A :=  $ $ A(\vphi) = $ $ {\rm Op} (a(\vphi, \cdot) ) = {\rm Op} (a ) $. 

By \eqref{lemma composition} and  a Fourier expansion also in $ \vphi \in \T^\nu $,
the symbol of the composition operator  $ A  B $ is 
\be\label{composition with phi}
\begin{aligned}
\sigma_{AB}(\vphi, x, \xi) & = 
\sum_{j \in \Z} a(\vphi, x, \xi + j) \widehat b(\vphi, j, \xi) e^{\ii j x} \\
& = \sum_{\begin{subarray}{c}
j', j \in \Z\\
\ell, \ell_1 \in \Z^\nu
\end{subarray}} \widehat a(\ell - \ell_1, j' - j, \xi + j) \widehat b(\ell_1, j , \xi) e^{\ii(\ell \cdot \vphi +   j' x)} \, . 
\end{aligned}
\ee
By  \eqref{simbolo aggiunto senza tempo}
the symbol of the adjoint operator $ A(\vphi)^* = {\rm Op}(a^*(\vphi, \cdot))$ is 
\begin{equation}\label{simbolo aggiunto}
a^*(\vphi, x, \xi) = \overline{\sum_{j \in \Z} \widehat a(\vphi, j, \xi - j) e^{\ii j x}} 
= \overline{ \sum_{\ell \in \Z^\nu, j \in \Z} \widehat a(\ell, j, \xi - j) e^{\ii (\ell \cdot \vphi + j x)}} \, . 
\end{equation}
Along the paper we also consider families of pseudo-differential operators  
$$ 
A(\lambda) := {\rm Op}(a(\lambda, \vphi, x, \xi)) 
$$ 
 which  are ${k_0}$-times differentiable with respect to 
a parameter  
$$ 
\lambda := (\omega, \kappa) \in \mathtt \Lambda_0 = \tOm_0 \times [\kappa_1, \kappa_2] 
\subset \R^\nu \times [\kappa_1, \kappa_2] \, , 
$$ 
where  the regularity constant $ k_0 \in \N $ is fixed once for all in section \ref{sec:degenerate KAM}. Note that 
$$
\partial_{\lambda}^k A = {\rm Op}(\partial_{\lambda}^k a)\, , \quad \forall k \in \N^{\nu + 1} \, , \ |k| \leq k_0 \,  .
$$
We now introduce a norm (inspired to 
Metivier \cite{Met}, chapter 5)
which controls the regularity in $ (\vphi, x)$,  and the decay in $ \xi $,  of the symbol $ a (\vphi, x, \xi) \in S^m $, 
together with its derivatives 
$ \pa_\xi^\b a \in S^{m - \b}$, $ 0 \leq \b \leq \a $, in the Sobolev norm $ \| \ \|_s $.  

\begin{definition}\label{def:pseudo-norm} {\bf (Weighted $\Psi DO$ norm)}
Let $ A(\lambda) := $ $ a(\lambda, \vphi, x, D) \in $ $ OPS^m $ 
be a family of pseudo-differential operators with symbol $ a(\lambda, \vphi, x, \xi) \in S^m $, $ m \in \R $, which are 
$k_0$-times differentiable with respect to $ \lambda \in \mathtt \Lambda_0 \subset \R^{\nu + 1} $. 
For $ \g \in (0,1) $, $ \a \in \N $, $ s \geq 0 $, we define  the weighted norm 
\begin{equation}\label{norm1 parameter}
\norma A \norma_{m, s, \alpha}^{k_0, \gamma} := \sum_{|k| \leq k_0} \gamma^{|k|} 
{\rm sup}_{\lambda \in {\mathtt \Lambda}_0}\norma\partial_\lambda^k A(\lambda) \norma_{m, s, \alpha}  
\end{equation}
where we use the multi-index notation $ k = ( k_1, \ldots , k_{\nu + 1}) \in \N^{\nu + 1} $ with 
$ | k  | := $ $ | k_1| + \ldots + | k_{\nu + 1} | $, and  
\be\label{norm1}
\norma A( \lambda )  \norma_{m, s, \a} := {\rm max}_{0 \leq \beta  \leq \a} \sup_{\xi \in \R} \|  \partial_\xi^\beta 
a(\lambda, \cdot, \cdot, \xi )  \|_{s} \langle \xi \rangle^{-m + \beta} \, . 
\ee
\end{definition}

For each $ k_0, \gamma, m  $ fixed, the norm \eqref{norm1 parameter} is 
non-decreasing both in $ s $ and $ \a $, 
namely  
\be\label{norm-increa}
\forall s \leq s', \, \a \leq  \a' \, ,  \qquad 
\norma \  \norma_{m, s, \a}^{k_0, \gamma} \leq \norma \ \norma_{m, s', \a}^{k_0, \gamma} \, ,  \ \ 
\norma \  \norma_{m, s, \a}^{k_0, \gamma}  \leq 
\norma \ \norma_{m, s, \a'}^{k_0, \gamma} \, .
\ee
Note also that the norm \eqref{norm1 parameter}  is non-increasing in  $ m $, i.e. 
\be\label{crescente-m-neg}
m \leq m'  \qquad  \Longrightarrow \qquad \norma \  \norma_{m', s, \a}^{k_0, \gamma}  \leq 
\norma \ \norma_{m, s, \a}^{k_0, \gamma} \, . 
\ee
Given a function $ a(\lambda, \vphi, x ) \in {\mathcal C}^\infty $ which is ${k_0}$-times differentiable with respect
to $ \lambda  $,  the weighted norm of the corresponding multiplication operator is 
\be\label{norma a moltiplicazione}
\norma {\rm Op} (a)  \norma_{0, s, \a}^{k_0, \gamma} = \| a \|_s^{k_0,\gamma} \, , \quad  \forall \a \in \N  \,  , 
\ee
where the weighted Sobolev norm $ \| a \|_s^{k_0, \gamma}$ is defined  in \eqref{norma pesata derivate funzioni}.

For a Fourier multiplier $   g(D)  $ with symbol $ g \in S^m $, we simply have 
\be\label{Norm Fourier multiplier}
\norma  g(D)  \norma_{m, s, \a} \leq C(m, \a, g ) \, , \quad \forall s \geq 0 \, . 
\ee
The norm $ \norma \ \norma_{0,s,0}$  controls the action of a pseudo-differential 
operator on the Sobolev spaces $ H^s $ as we shall prove in 
 Lemma \ref{lemma: action Sobolev}.

\begin{remark} 
The  norm of Definition \ref{def:pseudo-norm} is introduced in view of section \ref{egorov} 
 where we have to estimate
the norm $ \norma R_M \norma_{ 1- \frac{M}{2}, s, 0}^{k_0, \g} $ in \eqref{stima RN diagonale Egorov}. The remainder
$ R_M $ depends on $ \norma {\rm Op}(q_M) \norma_{1- \frac{M}{2},s,0}^{k_0, \g} $. 
The terms $ q_1, \ldots, q_M $ are 
obtained iteratively, and each $ q_{k+1} $ depends on $ \pa_\xi q_k $. Thus we need to control the Sobolev norm
in $ (\vphi, x) $ of $ \pa_\xi^M q_0 $. 
This is made precise by estimating the norm $ \norma {\rm Op}(q_0) \norma_{- \frac{3}{2}, s,M}^{k_0, \g}  $.
\end{remark}

The norm $  \norma \  \norma_{m, s, \alpha}^{k_0, \gamma}  $ 
 is closed under composition and  satisfies tame estimates. 
 
 \begin{lemma}\label{lemma stime Ck parametri}
{\bf (Composition)} 
Let $ A = a(\lambda, \vphi, x, D) $, $ B = b(\lambda, \vphi, x, D) $ be pseudo-differential operators
with symbols $ a (\lambda, \vphi, x, \xi) \in S^m $, $ b (\lambda, \vphi, x, \xi ) \in S^{m'} $, $ m , m' \in \R $. Then $ A(\lambda) \circ B(\lambda) \in OPS^{m + m'} $
satisfies,   for all $ \a \in \N $, $ s \geq s_0 $, 
\begin{equation}\label{estimate composition parameters}
\begin{aligned}
\norma A B \norma_{m + m', s, \alpha}^{k_0, \gamma} 
& \leq_{m,  \alpha, k_0} C(s) \norma A \norma_{m, s, \alpha}^{k_0, \gamma} 
\norma B \norma_{m', s_0 + \alpha + |m|, \alpha}^{k_0, \gamma}  \\
& \qquad \ \  + C(s_0) \norma A  \norma_{m, s_0, \alpha}^{k_0, \gamma}  
\norma B \norma_{m', s + \alpha + |m|, \alpha}^{k_0, \gamma} \, . 
\end{aligned}
\end{equation}
Moreover, for any integer $ N \geq 1  $,  
the remainder $ R_N := {\rm Op}(r_N) $ in \eqref{expansion symbol} satisfies
\begin{align}
\norma R_N  \norma_{m+ m'- N, s, \alpha}^{k_0, \gamma} & \leq_{m, N,  \alpha, k_0} \!\!
\frac{1}{N!} \big( C(s) 
\norma A \norma_{m, s, N + \alpha}^{k_0, \gamma} 
\norma B  \norma_{m', s_0 + 2 N + |m| + \alpha, \alpha }^{k_0, \gamma} \nonumber \\
& \qquad \qquad  \ \ \, + C(s_0)\norma A \norma_{m, s_0 , N + \alpha}^{k_0, \gamma}
\norma B  \norma_{m', s + 2 N + |m| + \alpha, \alpha }^{k_0, \gamma} \big) \, .
\label{stima RN derivate xi parametri} 
\end{align}
Both  \eqref{estimate composition parameters}-\eqref{stima RN derivate xi parametri} hold  
with the constant $ C(s_0) $ interchanged with $ C(s) $. 
\end{lemma}

\begin{proof}
As a first step we  prove the estimates with no dependence on $ \lambda $:
\begin{align}
  \norma A B \norma_{m + m', s, \alpha}  & \leq_{m, \a} C(s) \norma A  \norma_{m, s, \alpha} 
\norma B  \norma_{m',s_0 + \alpha + |m|, \alpha}    \nonumber\\
& \ \  \quad + C(s_0) 
\norma A \norma_{m, s_0, \alpha} \norma B \norma_{m',s + \alpha + |m|, \alpha} \, ,  \label{tame norm derivate xi} \\
 \norma R_N  \norma_{m + m'- N, s, \alpha}   & \leq_{m, N, s, \a} 
\frac{1}{N!} \big(
\norma A  \norma_{m, s, N + \alpha} \, \norma B  \norma_{m', s_0 + 2 N + |m| + \alpha, \alpha } 
  \nonumber\\
& \qquad \qquad  \quad + \norma A  \norma_{m, s_0 , N + \alpha} \norma B  \norma_{m', s + 2 N + |m| + \alpha, \alpha }\big) \,. \label{stima RN derivate xi} 
\end{align}
We first prove  \eqref{tame norm derivate xi} for $ \a = 0 $. 
Denote by $ \s := \s_{AB} $ the symbol in  \eqref{composition with phi}. For all $ \xi \in \R $ we have
\be\label{simbolo-pseudo-composto}
\begin{aligned}
& \| \sigma(\cdot , \xi) \|_{s}^2  \langle \xi \rangle^{- 2(m + m')}  \\
&  = \sum_{j', \ell} \langle \ell, j' \rangle^{2s} \Big| \sum_{j, \ell_1} 
 \widehat a(\ell - \ell_1, j' - j, \xi + j) \widehat b(\ell_1, j, \xi) \Big|^2 \langle \xi \rangle^{- 2(m + m')}  \\
 & \leq S_1 + S_2 
 \end{aligned}
\ee  
where
\begin{align}
& S_1 :=  \nonumber\\ 
& \sum_{j', \ell}  \Big( \!\!  \! \sum_{\langle \ell, j' \rangle \leq 2^{1/s} \langle \ell_1, j  \rangle } \!\!\!\!\!\!\!\!\!\!\!\!
 \frac{\langle \ell_1, j \rangle^{s} \langle \ell, j' \rangle^{s} | \widehat a(\ell - \ell_1, j' - j, \xi + j )| \langle \ell - \ell_1, j' - j \rangle^{s_0} |\widehat b(\ell_1, j, \xi)|}{\langle \ell_1, j \rangle^{s} \langle \ell - \ell_1, j' - j \rangle^{s_0} \langle \xi \rangle^{m + m'}}  \Big)^2 
 \nonumber \\
& S_2 :=  \nonumber\\  
&  \sum_{j', \ell}  \Big( \!\! \!  \sum_{\langle \ell, j' \rangle > 2^{1/s} \langle \ell_1, j  \rangle } \!\!\!\!\!\!\!\!\!\!\!\!
 \frac{ \langle \ell_1, j \rangle^{s_0} \langle \ell, j' \rangle^{s}  | \widehat a(\ell - \ell_1, j' - j, \xi + j )| \langle \ell - \ell_1, j' - j \rangle^{s} |\widehat b(\ell_1, j, \xi)| }{\langle \ell_1, j \rangle^{s_0} \langle \ell - \ell_1, j' - j \rangle^{s} \langle \xi\rangle^{m + m'}} \Big)^2 
\, . \nonumber 
\end{align}
Now,  by Cauchy-Schwartz inequality and denoting 
$ \zeta (s_0) := \sum_{\ell \in \Z^\nu, j \in \Z} \frac{1}{\langle \ell, j \rangle^{2s_0} } $, we get 
 \begin{align}
S_1 & \leq
 \sum_{j', \ell}  \Big( \sum_{\langle \ell, j' \rangle \leq 2^{1/s} \langle \ell_1, j  \rangle } \!\!\!\!\!\!\!\!\!\!\!\! 
 \frac{2 \langle \ell_1, j \rangle^{s} | \widehat a(\ell - \ell_1, j' - j, \xi + j )| \langle \ell - \ell_1, j' - j \rangle^{s_0} |\widehat b(\ell_1, j, \xi)|}{ \langle \ell - \ell_1, j' - j \rangle^{s_0} \langle \xi \rangle^{m + m'}} \Big)^2 
 \nonumber \\
& \leq 4 \zeta (s_0) 
 \sum_{j', \ell}  \sum_{ \ell_1, j} 
\frac{ | \widehat a(\ell - \ell_1, j' - j, \xi + j )|^2 \langle \ell - \ell_1, j' - j \rangle^{2s_0} 
 |\widehat b(\ell_1, j, \xi)|^2  \langle \ell_1, j \rangle^{2s}}{\langle \xi \rangle^{2 (m + m')}} \nonumber \\ 
& \leq 4 \zeta (s_0) 
 \sum_{ \ell_1, j}  \frac{ |\widehat b(\ell_1, j, \xi)|^2  \langle \ell_1, j \rangle^{2s}}{\langle \xi \rangle^{ 2m'}}   \sum_{j', \ell} 
 \frac{| \widehat a(\ell - \ell_1, j' - j, \xi + j )|^2 \langle \ell - \ell_1, j' - j \rangle^{2s_0}}{ \langle \xi \rangle^{2 m}}
  \, . \label{Stima:S1}
\end{align}
For each $ j $, $ \ell_1 $ fixed, we apply  Peetre's inequality 
\be\label{Peetre's inequality}
\langle \xi + \eta \rangle^{m} \leq  C_m \langle \xi \rangle^m \langle \eta \rangle^{|m|}, \quad 
\forall m \in \R, \, \eta \in \R \, , \xi \in \R 
\ee
(where $ C_m = 4^{|m|} $) with $ \eta = j $, and we estimate, for any $ s \geq s_0 $,    
\begin{align}
& \sup_\xi 
\sum_{j' , \ell} \frac{|\widehat a(\ell- \ell_1, j' - j, \xi + j )|^2  \langle \ell - \ell_1, j' - j \rangle^{2 s}}{\langle \xi \rangle^{ 2 m}}  
 = \sup_\xi  \frac{\| a(\cdot, \xi + j ) \|^2_s}{\langle \xi \rangle^{ 2 m}}  	\nonumber \\ 
& = \Big( \sup_\xi  \frac{\| a(\cdot, \xi + j ) \|^2_s}{ \langle \xi + j  \rangle^{ 2 m}}   \Big) 
\frac{\langle \xi + j \rangle^{2 m}}{\langle \xi \rangle^{ 2 m}} 
 \leq C_m^2  \norma A \norma_{m,s,0}^2  \langle j \rangle^{ 2 |m|} \label{stima-tameAs0}
\end{align}
and therefore we get, by  \eqref{Stima:S1} and \eqref{stima-tameAs0} for $ s = s_0 $, 
\be
\begin{aligned}
S_1 & \leq 
4 \zeta (s_0)  C_m^2   \norma A \norma_{m,s_0,0}^2 \sum_{ \ell_1, j}  |\widehat b(\ell_1, j, \xi)|^2  \langle \ell_1, j \rangle^{2s} \langle j \rangle^{ 2 |m|} \ 
\langle \xi \rangle^{- 2m'}   \\ 
& \leq
4 \zeta (s_0)  C_m^2  \norma A \norma_{m,s_0,0}^2  \norma B \norma_{m', s+ |m|,0}^2  \, . \label{stima-simbolo-sigma} 
\end{aligned}
\ee
For the estimate of $ S_2 $ note that, since the indices satisfy 
$ \langle \ell, j' \rangle  > 2^{1/s} \langle \ell_1, j \rangle $ we have 
$ \langle \ell, j' \rangle \leq $ $ \langle \ell_1, j \rangle + 
\langle \ell - \ell_1, j' -  j \rangle \leq $ $ 2^{-1/s} \langle \ell, j' \rangle + 
\langle \ell - \ell_1, j' - j \rangle $ 
and therefore
$$
\langle \ell, j' \rangle \leq \big( 1 - 2^{- 1/s} \big)^{-1} \langle  \ell - \ell_1, j' - j \rangle \, .
$$
As a consequence, arguing as above, we deduce that, for some constant $ C(s) > 0 $, we have  
\be\label{stima:S2-primo}
S_2 \leq_m C(s)  \norma A \norma_{m,s,0}^2 \norma B \norma_{m', s_0+ |m|,0}^2 \, . 
\ee
By \eqref{simbolo-pseudo-composto} and \eqref{stima-simbolo-sigma}, \eqref{stima:S2-primo} we deduce 
the estimate \eqref{tame norm derivate xi} for $ \a = 0 $, i.e. 
\begin{equation}\label{tame norm}
\norma A B \norma_{m+m', s,0} \leq_{m} 
C(s) \norma A  \norma_{m, s,0} \norma B \norma_{m', s_0  + |m|,0} 
+ C(s_0)  \norma A \norma_{m, s_0,0} \norma B \norma_{m', s + |m|,0} \, . 
\end{equation}
Now we prove \eqref{tame norm derivate xi} for $ \a \geq 1 $. 
By differentiating  \eqref{composition with phi}  we get,  for all $1 \leq \beta \leq \a $, 
$$
\partial_\xi^\b \sigma_{AB} (\vphi, x, \xi)  = \sum_{\b_1 + \b_2 = \b} C(\b_1, \b_2) 
\sum_{j \in \Z} \partial_\xi^{\b_1} a(\vphi, x, \xi + j ) \partial_\xi^{\b_2} \widehat b(\vphi, j, \xi) e^{\ii j x} \, .
$$
Therefore,  since $\partial_\xi^{\b_2} \widehat b(\vphi, j, \xi) = 
\widehat{\partial_\xi^{\b_2}} b(\vphi, j, \xi)  $ and, again by \eqref{composition with phi}, we get 
\be \label{scomp1}
{\rm Op}(\partial_\xi^\b \sigma_{AB}) = {\mathop \sum}_{\b_1 + \b_2 =\b} 
C(\b_1, \b_2) {\rm Op}(\partial_\xi^{\b_1} a) \circ {\rm Op}(\partial_\xi^{\b_2} b)\,. 
\ee
Since $\partial_\xi^{\b_1} a \in S^{m - \b_1}$, 
$\partial_\xi^{\b_2} b \in S^{m' - \b_2}$,  $\b_1 + \b_2 = \b $,   the estimate \eqref{tame norm} implies 
\begin{equation} \label{scomp2}
\begin{aligned}
& \norma {\rm Op}(\partial_\xi^{\b_1} a) {\rm Op}(\partial_\xi^{\b_2} b) 
 \norma_{m + m' - \beta, s, 0}  \\
 &  \leq_{m,  \beta} C(s) \norma {\rm Op}(\partial_\xi^{\b_1} a) 
 \norma_{m - \beta_1, s, 0} \norma 
{\rm Op}(\partial_\xi^{\b_2} b)  \norma_{m' - \beta_2, s_0 + \b_1 + |m|, 0}   \\
& \quad \  +  C(s_0) \norma {\rm Op}(\partial_\xi^{\b_1} a)  \norma_{m - \beta_1, s_0, 0} 
\norma {\rm Op}(\partial_\xi^{\b_2} b)  \norma_{m' - \beta_2 , s + \b_1 + |m|, 0}\, .
\end{aligned}
\end{equation}
Therefore,  for all $ 1 \leq \b \leq \a $, 
by \eqref{scomp1}, \eqref{scomp2} and the definition \eqref{norm1} we get
\begin{align*}
\norma {\rm Op}(\partial_\xi^{\b} \sigma_{AB}  )\norma_{m +m' - \beta, s, 0} & \leq_{m, \beta} 
C(s) \norma A  \norma_{m, s, \alpha} \norma B  \norma_{m', s_0 + \alpha + |m|, \alpha}  \\
& \quad \ \  + 
C(s_0) \norma A  \norma_{m, s_0, \alpha} \norma B  \norma_{m', s + \alpha + |m|, \alpha} 
\end{align*}
which proves \eqref{tame norm derivate xi}. 

Now we prove  \eqref{stima RN derivate xi}.  Recalling \eqref{rNTaylor} it is sufficient to estimate each 
\begin{equation}\label{r N tau}
r_{N, \t}(\vphi, x, \xi):= {\mathop \sum}_{j \in \Z}  (\pa_\xi^N a)(\vphi, x, \xi + \t j ) \widehat {\pa_x^N b} (\vphi, 
j, \xi) e^{\ii j x } \, , \    \t \in [0,1] \, . 
\end{equation}
Arguing as above (to prove \eqref{tame norm}) we get
\begin{align*}
& \| r_{N, \t}(\cdot, \xi)  \|_s \langle \xi \rangle^{N - (m + m')}  \\
&  \leq_{m, N} 
C(s) \norma {\rm Op}(\partial_\xi^N a) \norma_{m - N, s, 0} \norma {\rm Op}(\partial_x^N b)  \norma_{m' , s_0 + N + |m|, 0} \\
& \quad \ \  + C(s_0) 
\norma {\rm Op}(\partial_\xi^N a) \norma_{m - N, s_0, 0} \norma {\rm Op}(\partial_x^N b)  \norma_{m' , s + N + |m|, 0} \\
& \leq_{m, N} C(s) \norma {\rm Op}(\partial_\xi^N a)  \norma_{m - N, s, 0} \norma {\rm Op}( b) \norma_{m' , s_0 + 2N + |m|, 0} \\
& \qquad  + C(s_0) \norma {\rm Op}(\partial_\xi^N a) 
\norma_{m - N, s_0, 0} \norma {\rm Op}(  b)  \norma_{m', s + 2N + |m|, 0}
\end{align*}
which gives (recall \eqref{rNTaylor} and \eqref{norm1}) 
\be\label{stima RN0} 
\begin{aligned}
\norma R_N  \norma_{m + m' - N, s, 0}  & \leq_{m, N} 
\frac{1}{N!} \big(
C(s) \norma A  \norma_{m, s, N} \norma B  \norma_{m', s_0 + 2 N + |m|, 0}  \\
& \qquad \quad \ + C(s_0) \norma A  \norma_{m, s_0 , N} \norma B  \norma_{m', s + 2 N + |m|, 0}\big) 
\end{aligned}
\ee
namely \eqref{stima RN derivate xi} for $ \a = 0 $.  We now prove \eqref{stima RN derivate xi} for
$ \alpha \geq 1 $. By differentiating \eqref{r N tau} we get, $ \forall 1 \leq \beta \leq \a $,  
 \begin{align*}
\partial_\xi^\b r_{N, \t}(\vphi, x, \xi) & = \sum_{\b_1 + \b_2 = \b} 
C(\b_1, \b_2) \sum_{j \in \Z}  (\pa_\xi^{N + \b_1} a)(\vphi, x, \xi + \t j ) \widehat {\pa_x^N \partial_\xi^{\b_2} b} 
(\vphi, j, \xi) e^{\ii j x }
\end{align*}
and so, arguing as for \eqref{scomp2}, 
\begin{align*}
& \| \partial_\xi^{\b}r_{N, \t}(\cdot, \xi)  \|_s   \langle  \xi \rangle^{N + \b - (m + m')}  \\
& \leq_{m, N,\alpha} 
\! \! \! \! \! \! \! \! 
\sum_{\b_1 + \b_2 = \b } \! \! \! \! \! \! 
 \Big( C(s) \norma {\rm Op}(\partial_\xi^{N + \b_1} a)  \norma_{m - N- \b_1, s, 0} \norma {\rm Op}(\partial_\xi^{\b_2}\partial_x^N b)  \norma_{m' - \b_2, s_0 + N + |m| + \b_1, 0 }  \\
&  \quad \quad \quad \quad  + C(s_0) \norma {\rm Op}(\partial_\xi^{N + \b_1} a)  \norma_{m - N- \b_1, s_0, 0} \norma {\rm Op}(\partial_\xi^{\b_2}\partial_x^N b)   \norma_{m' - \b_2, s + N + |m| + \b_1, 0 } \Big) \\
&  \stackrel{ \eqref{norm1} }{\leq_{m, N, \alpha}} 
C(s) \norma A  \norma_{m, s, N + \alpha} \norma B  \norma_{m', s_0 + 2 N + |m| + \alpha, \alpha } \\
& \qquad \ + C(s_0) \norma A  \norma_{m, s_0 , N + \alpha} \norma B \norma_{m', s + 2 N + |m| + \alpha, \alpha } 
\end{align*}
and  \eqref{stima RN derivate xi} is proved. 

Finally we prove \eqref{estimate composition parameters}, \eqref{stima RN derivate xi parametri}
including the dependence on $ \lambda $. 
For all $ k \in \N^{\nu + 1} $, $|k| \leq k_0$, the derivative
$$
\partial_\lambda^{k} \{ A(\lambda) \circ B(\lambda) \} = {\mathop \sum}_{k_1, k_2 \in \N^{\nu + 1},  k_1 + k_2 = k} C(k_1, k_2) \partial_\lambda^{k_1} A(\lambda) \circ \partial_\lambda^{k_2} B(\lambda)\, .
$$
Then (we have $ |k| = |k_1| + | k_2 |$)
\begin{align*}
& \gamma^{|k|} \norma \partial_\lambda^{k} \{ A(\lambda) \circ B(\lambda) \}   \norma_{m + m', s, \alpha}  \leq_{k_0}  
\sum_{k_1 + k_2 = k} \gamma^{|k_1|} \gamma^{|k_2|}\norma \partial_\lambda^{k_1} A(\lambda) \circ \partial_\lambda^{k_2} B(\lambda)\norma_{m + m', s, \alpha} \\
& \stackrel{\eqref{tame norm derivate xi}}{\leq_{k_0, m, \alpha}} \!\! \sum_{k_1 + k_2 = k} \big( C(s)\gamma^{|k_1|}
\norma \partial_\lambda^{k_1}A  \norma_{m, s, \alpha}  \gamma^{|k_2|}\norma 
\partial_\lambda^{k_2} B  \norma_{m', s_0 + \alpha + |m|, \alpha}  
 \\
 & \qquad \qquad \qquad  + C(s_0)\gamma^{|k_1|}\norma \partial_\lambda^{k_1} A  \norma_{m, s_0, \alpha} \gamma^{|k_2|} \norma \partial_\lambda^{k_2} B   \norma_{m', s + \alpha + |m|, \alpha}  \big)
\end{align*}
and  \eqref{estimate composition parameters} follows by the definition \eqref{norm1 parameter}. 
The estimate \eqref{stima RN derivate xi parametri} follows since for all $|k| \leq k_0$
\begin{align*}
& \gamma^{|k|} \norma \partial_\lambda^k {\rm Op}(r_{N, \tau})  \norma_{m + m' - N, s, \alpha}  \\
& \leq_{k_0, m, N, \alpha} 
\!\!\!\! \sum_{k_1 + k_2 = k} \big( C(s)\gamma^{|k_1|} \norma \partial_\lambda^{k_1}  A  \norma_{m, s, N + \alpha} \gamma^{|k_2|} \norma \partial_\lambda^{k_2} B \norma_{m', s_0 + 2 N + |m| + \alpha, \alpha } \\
& \qquad \qquad \  \qquad + C(s_0)\gamma^{|k_1|} \norma \partial_\lambda^{k_1} A  \norma_{m, s_0 , N + \alpha} \gamma^{|k_2|} \norma \partial_\lambda^{k_2} B  \norma_{m', s + 2 N + |m| + \alpha, \alpha }  \big)\,.
\end{align*}
The proof is complete. 
\end{proof}

When $ B = g(D) $ is a  Fourier multiplier, then $ {\rm Op}(a ) \circ g(D) = {\rm Op} ( a(x, \xi) g(\xi)) $ and 
we have a simpler estimate. 

\begin{lemma}\label{lemma composizione multiplier}
Let $ A = a(\lambda, \vphi, x, D) \in OPS^m $, $ m \in \R $, and let $ g(D) \in OPS^{m'} $ be a Fourier multiplier (independent of $ \lambda $). 
Then  $ \norma A \circ g(D) \norma_{m+m', s, \a}^{k_0, \gamma} 
\leq_{m,  \alpha} \norma A  \norma_{m, s,\a}^{k_0, \gamma} $. 
\end{lemma}

By \eqref{expansion symbol} the  commutator between two pseudo-differential operators 
$$ 
A = a(x, D)  \in OPS^{m} \quad {\rm and} \quad 
 B = b (x, D)  \in OPS^{m'} 
 $$ 
is a pseudo-differential operator $ [A, B] \in OPS^{m + m' - 1} $ 
with symbol  $ a \star b $ (sometimes called the Moyal parenthesis of $ a $ and $ b $), namely 
\be\label{symbol commutator}
[A, B] = {\rm Op} ( a \star b ) \, .  
\ee
By \eqref{expansion symbol} the symbol  $ a \star b  \in S^{m + m' - 1} $ admits the expansion 
\be\label{Expansion Moyal bracket}
a \star b = - \ii \{ a, b \} + {\mathtt r_{\mathtt 2}} (a, b )  \qquad {\rm where}
\quad \{a, b\} := \pa_\xi a \, \pa_x b - \pa_x a \, \pa_\xi b  
\ee
is the Poisson bracket between $ a(x, \xi) $ and $ b(x, \xi)  $, and  
$$ 
{\mathtt r_{\mathtt 2}} (a, b ) := r_{2, AB} - r_{2, BA} \in S^{m + m' - 2} \, . 
$$

\begin{lemma}{\bf (Commutators)} \label{lemma tame norma commutatore}
Let $ A = a(\lambda, \vphi, x, D) $, $ B = b(\lambda, \vphi, x, D) $ be pseudo-differential operators
with symbols $ a (\lambda, \vphi, x, \xi) \in S^m $, $ b (\lambda, \vphi, x, \xi ) \in S^{m'} $, $ m , m' \in \R $. Then  
the commutator $ [A, B] := AB - B A \in OPS^{m + m' - 1} $   satisfies 
\begin{align}
\norma [A,B]  \norma_{m+m' -1, s, \alpha}^{k_0, \gamma} & 
\leq_{m,m',   \alpha, k_0} 
\big( C(s)  \norma A \norma_{m, s + 2 + |m' | + \alpha, \alpha + 1}^{k_0, \gamma} 
\norma B  \norma_{m', s_0 + 2 + |m| + \alpha , \alpha + 1}^{k_0, \gamma}  \nonumber\\
& \qquad \qquad  \, 
+ C(s_0) \norma A  
\norma_{m, s_0 + 2 + |m' | + \alpha, \alpha + 1}^{k_0, \gamma} 
\norma B \norma_{m', s + 2 + |m| + \alpha , \alpha + 1}^{k_0, \gamma} 
\big) \, .  \label{stima commutator parte astratta}
\end{align}
Moreover  the Poisson bracket $ \{ a, b \}  \in S^{m + m' - 1} $ satisfies 
\be\label{stima Poisson}
\begin{aligned}
\norma {\rm Op}( \{ a, b \} )  \norma_{m+m' -1, s, \alpha}^{k_0, \gamma} 
&  \leq_{ \alpha, k_0} C(s)  \norma A \norma_{m, s + 1, \alpha + 1}^{k_0, \gamma} 
\norma B  \norma_{m', s_0 + 1, \alpha + 1}^{k_0, \gamma} \\
& \ \ \quad + C(s_0) \norma A  \norma_{m, s_0 + 1, \alpha + 1}^{k_0, \gamma} \norma B \norma_{m', s + 1, \alpha + 1}^{k_0, \gamma}  \, .  
\end{aligned}
\ee
\end{lemma}

\begin{proof}
The estimate \eqref{stima commutator parte astratta} follows by \eqref{expansion symbol},  
\eqref{stima RN derivate xi parametri}  for $ N = 1$, and \eqref{norm-increa}. 
The estimate \eqref{stima Poisson} follows 
by \eqref{Expansion Moyal bracket}, Definition \ref{def:pseudo-norm},
the tame estimates for the product of two functions \eqref{interpolazione C k0} and \eqref{norm-increa}. 
\end{proof}

Note that in \eqref{stima Poisson} the loss of regularity in $ s $ is smaller than in \eqref{stima commutator parte astratta}.

The adjoint $ A^* $ of a pseudo-differential operator $ A = {\rm Op}(a) \in OPS^m $ is a pseudo-differential
operator of the same order $ A^*  = {\rm Op}(a^*) \in OPS^m $ and the symbol 
$ a^* $ is defined  in \eqref{simbolo aggiunto senza tempo}.  
 
\begin{lemma}{\bf (Adjoint)} \label{stima pseudo diff aggiunto}
Let $A= a(\lambda, \vphi, x, D)$ be a pseudo-differential operator with symbol $a(\lambda, \vphi, x, \xi) \in S^m,  m \in \R $. Then the adjoint 
$A^*  \in OPS^m $ 
satisfies 
$$
\norma A^* \norma_{m, s, 0}^{k_0, \gamma} \leq_{m} \norma A \norma_{m, s + s_0 + |m|, 0}^{k_0, \gamma} \, . 
$$
\end{lemma}
\begin{proof}
Recalling Definition \ref{def:pseudo-norm} and \eqref{simbolo aggiunto}
we have to estimate 
\begin{align}
\norma A^* \norma_{m,s,0}^2 = \sup_{\xi \in \R} \| a^* (\cdot, \cdot, \xi)\|_s^2 \langle \xi \rangle^{- 2 m} & 
= \sum_{\ell, j}\langle \ell, j \rangle^{2 s} |\widehat a(\ell, j, \xi - j)|^2 \langle \xi \rangle^{- 2 m}  \label{pippone 0 aggiunto}\,.
\end{align}
Since
$$
\begin{aligned}
\norma A \norma_{m,s + s_0 + |m|,0}^2 & 
:= \sup_{\xi \in \R} \| a( \cdot, \cdot, \xi ) \|_{s+s_0+|m|}^2 \langle \xi \rangle^{-2m} \\ 
& =  
\sup_{\xi \in \R} 
{\mathop \sum}_{\ell, j} | \widehat a( \ell, j , \xi )|^2 \langle \ell, j \rangle^{2(s+s_0+|m|)} \langle \xi \rangle^{-2m}  
\end{aligned}
$$
we derive the  bound, for all $ \xi \in \R $, $ \ell \in \Z^\nu $,  $ j \in \Z $,  
\begin{equation}\label{claim per l'aggiunto}
|\widehat a(\ell, j, \xi - j)| \leq \frac{\norma A \norma_{m , s + s_0 + |m|, 0}}{\langle \ell, j \rangle^{s + s_0 + |m|}}  \langle \xi - j \rangle^{m} \, . 
\end{equation}
Then by \eqref{pippone 0 aggiunto}, \eqref{claim per l'aggiunto} and Peetre's inequality \eqref{Peetre's inequality} we get
\begin{align}
\norma A^* \norma_{m,s,0}^2 & \leq \sum_{\ell , j} \frac{1}{\langle \ell, j \rangle^{2(s_0 + |m|)}} \frac{\langle \xi - j \rangle^{2 m}}{\langle \xi \rangle^{2 m}} \norma A \norma_{m, s + s_0 + |m|, 0}^2 \nonumber\\
& \leq_m \sum_{\ell, j} \frac{\langle j \rangle^{2 |m|}}{ \langle \ell, j \rangle^{2 (s_0 + |m|)}} \norma A \norma_{m, s + s_0 + |m|, 0}^2  \leq_m \norma A \norma_{m, s + s_0 + |m|, 0}^2 \, . 
\end{align}
The estimate for the derivatives with respect to $\lambda$ follows analogously, since 
$ \partial_\lambda^k A^* = {\rm Op} ( \partial_\lambda^k a^* ) $. 
\end{proof}

\begin{lemma} {\bf (Invertibility)} \label{Neumann pseudo diff}
Let  $ \Phi :=  {\rm Id} + A $ where $ A := {\rm Op}(a(\lambda, \vphi, x, j)) \in OPS^0 $.  
There exist constants  $ C(s_0, \a, k_0) $, $ C(s, \a, k_0) \geq 1 $, $ s \geq s_0 $,  such that,  if 
\begin{equation}\label{condizione di piccolezza Neumann norma pseudo}
C(s_0,\a, k_0)\norma A \norma_{0, s_0 + \alpha, \alpha}^{k_0, \gamma}   \leq 1/2  \, , 
\end{equation}
then, for all $ \lambda  $,  the operator 
$ \Phi $ is invertible, $ \Phi^{- 1}  \in OPS^0 $ and, for all $ s \geq s_0 $,    
$$ 
\norma \Phi^{- 1} - {\rm Id}  \norma_{0,s,\alpha}^{k_0, \gamma} \leq C(s,\a, k_0)  \norma A \norma_{0,s + \alpha, \alpha}^{k_0, \gamma} \, . 
$$  
\end{lemma}

\begin{proof} 
Iterating  \eqref{estimate composition parameters} (for $m = 0$) we deduce that 
there exist constants $ C(s_0, \a, k_0), C(s, \a, k_0) \geq 1 $ such that, $ \forall n \in \N^+ $,  
\begin{align}
& \norma A^n \norma_{0, s_0, \alpha}^{k_0, \g} 
\leq ( C(s_0, \a, k_0) )^{n-1}  \big( \norma A \norma_{0, s_0 + \alpha, \alpha}^{k_0, \g} \big)^n\,, \nonumber \\
& \norma A^n \norma_{0, s, \alpha}^{k_0, \g} \leq 
n C(s, \a, k_0) \big( C(s_0,\a, k_0) \norma A \norma_{0, s_0 + \alpha, \alpha}^{k_0, \g} \big)^{n - 1}  
\norma A \norma_{0, s + \alpha, \alpha}^{k_0, \g} \,.  \label{stima A n}
\end{align}
By \eqref{condizione di piccolezza Neumann norma pseudo} 
 the operator $ \Phi $ is invertible and the inverse $ \Phi^{-1} $ may be expressed by the Neumann series
$ \Phi^{- 1} = {\rm Id} + B $ with $ B:= \sum_{n \geq 1} (- 1)^n A^n $. 
Moreover, since
$$
\| a(\cdot, j )\|_{L^\infty} \leq C(s_0) \| a(\cdot, j )\|_{s_0} \leq C(s_0) \norma A \norma_{0,s_0, 0} \, , \quad \forall j \in \Z  \, ,
$$
the symbol  of $ \Phi $ 
satisfies $ 1 + a(\lambda, \vphi, x, j ) \geq 1/ 2 $, $ \forall j \in \Z $, $ \forall \lambda $,  i.e it is elliptic.  
Hence  the inverse operator 
 $ B $ is pseudo-differential  by the parametrix theorem (see \cite{Ho1}-Theorem 18.1.9). 
Moreover by \eqref{stima A n}
\begin{align*}
\norma B \norma_{0, s, \alpha}^{k_0, \gamma} & \leq {\mathop \sum}_{n \geq 1} 
\norma A^n \norma_{0, s, \alpha}^{k_0, \gamma} \\
& \leq 
\Big( {\mathop \sum}_{n \geq 1} n ( C(s_0,\a, k_0) \norma A \norma_{0, s_0 + \alpha, \alpha}^{k_0, \gamma} )^{n-1} \Big) 
C(s,\a, k_0) \norma A \norma_{0, s + \alpha, \alpha}^{k_0, \gamma} \\
& \leq C'(s,\a, k_0) \norma A \norma_{0, s + \alpha, \alpha}^{k_0, \gamma}
\end{align*}
by the smallness condition \eqref{condizione di piccolezza Neumann norma pseudo}. 
\end{proof}

\section{ $ {\mathcal D}^{k_0}$-tame and $ {\mathcal D}^{k_0}$-modulo-tame operators}\label{sezione operatori tame}

Let $ A := A(\lambda) $ be a linear operator $k_0$-times differentiable with respect to the parameter 
$ \lambda  \in \mathtt \Lambda_0  \subset  \R^{\nu + 1} $.

\begin{definition} { \bf ($ {\mathcal D}^{k_0} $-$ \s $-tame)}\label{def:Ck0}
A linear operator  $ A := A(\lambda)  $ is  
$ {\mathcal D}^{k_0}$-$ \s $-tame  if the following weighted tame estimates hold: there exists $ \s \geq 0 $ such that,
for all $ s_0 \leq s \leq S $, with possibly $S = + \infty $, $ \forall u \in H^{s+\s} $,  
\be\label{CK0-sigma-tame}
\sup_{|k| \leq k_0} \sup_{\lambda \in \mathtt \Lambda_0} \gamma^{ |k|}
\| (\partial_\lambda^k A(\lambda)) u \|_s \leq  {\mathfrak M}_A(s_0) \| u \|_{s+\s} + 
{\mathfrak M}_A (s) \| u  \|_{s_0+\s} 
\ee
where  the functions $ s \mapsto  {\mathfrak M}_A(s) \geq 0  $ are non-decreasing in $ s $. 
We call $ {\mathfrak M}_{A}(s) $ the {\sc tame constant} of the operator $ A $.
The constant $ {\mathfrak M}_A(s) := {\mathfrak M}_A (k_0, \s, s) $
depends also on $ k_0, \s $ but,
 since $ k_0, \s $ are considered in this paper absolute constants, 
we shall often omit to write them. 

When the ``loss of derivatives"  $ \sigma = 0 $  we simply call a $ {\mathcal D}^{k_0} $-$ 0 $-tame operator to  be $ {\mathcal D}^{k_0} $-tame. 
\end{definition}
\begin{remark}
In sections \ref{linearizzato siti normali}, \ref{sec: reducibility} we work with 
$ {\mathcal D}^{k_0}$-$\sigma$-tame operators with  a finite $ S < + \infty$,  whose tame constants ${\mathfrak M}_A(s)$ may depend also on $C(S)$, for instance  
$ {\mathfrak M}_A(s) \leq C(S) (1 + \| \fracchi_0 \|_{s + \mu}^{k_0, \gamma}) $, $ \forall s_0 \leq s \leq S $.
\end{remark}
An immediate consequence of \eqref{CK0-sigma-tame} (with $ k = 0 $, $ s = s_0 $) is that 
\begin{equation}\label{norma operatoriale costante tame}
\| A \|_{{\mathcal L}(H^{s_0 + \sigma}, H^{s_0})} \leq 2 {\mathfrak M}_A(s_0)\,.
\end{equation}
Note also that representing the operator $ A $ by its matrix elements 
$$ \big(A_j^{j'} (\ell - \ell') \big)_{\ell, \ell' \in \Z^\nu, j, j' \in \Z} $$ 
as in \eqref{matrice operatori Toplitz} we have, for all
$ |k| \leq k_0 $, $ j' \in \Z $, $ \ell' \in \Z^\nu $,  
\be\label{tame-coeff}
\begin{aligned}
& \gamma^{2 |k|} {\mathop \sum}_{\ell , j} \langle \ell, j \rangle^{2 s} |\partial_\lambda^k A_j^{j'}(\ell - \ell')|^2 \\
& \leq 2 \big({\mathfrak M}_A(s_0) \big)^2 \langle \ell', j' \rangle^{2 (s+\s)} + 2 \big({\mathfrak M}_A(s) \big)^2 \langle \ell', j' \rangle^{2 (s_0+\s)} 	\, . 
\end{aligned}
\ee
The class of $ {\mathcal D}^{k_0} $-$ \s $-tame operators is closed under composition. 

\begin{lemma}\label{composizione operatori tame AB} {\bf (Composition)}
 Let $ A, B $ be respectively $ {\mathcal D}^{k_0} $-$\sigma_A$-tame and $ {\mathcal D}^{k_0} $-$\sigma_B$-tame operators with tame 
 constants respectively 
$ {\mathfrak M}_A (s) $ and $ {\mathfrak M}_B (s) $. 
Then 
the composed operator 
$  A  \circ B $ is $ {\mathcal D}^{k_0} $-$(\sigma_A + \sigma_B)$-tame with tame constant
$$
 {\mathfrak M}_{A B} (s) \leq  C(k_0) \big( {\mathfrak M}_{A}(s) 
 {\mathfrak M}_{B} (s_0 + \sigma_A) + {\mathfrak M}_{A} (s_0) 
{\mathfrak M}_{B} (s + \sigma_A) \big)\,.
$$
\end{lemma}

\begin{proof}
As for the analogous inequality \eqref{modulo tame constant for composition} below. 
\end{proof}

Pseudo-differential operators are tame operators. We shall use in particular the following lemma.

\begin{lemma}\label{lemma: action Sobolev}
Let $ A = a(\lambda, \vphi, x, D) \in OPS^0 $  be a family of pseudo-differential operators
which are ${k_0}$-times differentiable with respect to  $\lambda $. 
If $ \norma A \norma_{0, s, 0}^{k_0, \gamma} < + \infty $, $ s \geq s_0 $, then  $ A $ is ${\mathcal D}^{k_0}$-tame with tame constant 
\begin{equation}\label{interpolazione parametri operatore funzioni}
{\mathfrak M}_A(s) \leq C(s)  \norma A \norma_{0, s, 0}^{k_0, \gamma}\,.
\end{equation}
\end{lemma}

\begin{proof} 
By expanding \eqref{come-agisce-A} in Fourier, we have 
$$
A u ( \vphi, x)  = 
\sum_{\ell \in \Z^\nu, j \in \Z} \Big(\sum_{\ell', j' } 
\widehat a (\ell - \ell' , j - j', j') u_{\ell', j'} \Big)  e^{\ii (\ell \cdot \vphi + j x)} \, .  
$$
Hence
\begin{align*}
\| A u \|_s^2 & = \sum_{\ell \in \Z^\nu, j \in \Z} \Big(\sum_{\ell' \in \Z^\nu, j' \in \Z} 
\widehat a (\ell - \ell' , j - j', j')  u_{\ell', j'} \Big)^2 \langle  \ell, j \rangle^{2s} \\
& \leq \sum_{\ell \in \Z^\nu, j \in \Z} \Big(\sum_{\ell' \in \Z^\nu, j' \in \Z} 
|\widehat a (\ell - \ell' , j - j', j')|  |u_{\ell', j'}| \langle  \ell, j \rangle^{s} \Big)^2  = S_1 + S_2
\end{align*}
where 
\begin{align*}
& S_1  :=  \\
&  \sum_{\ell \in \Z^\nu, j \in \Z} \! \Big( \! \sum_{\langle  \ell, j \rangle \langle \ell', j' \rangle^{-1} \leq 2^{1/s}} 
\!  \frac{\langle  \ell, j \rangle^{s} \langle \ell - \ell', j - j' \rangle^{s_0} |\widehat a (\ell - \ell' , j - j', j')|  \langle \ell', j' \rangle^s  |u_{\ell', j'}| }{ \langle \ell - \ell', j - j' \rangle^{s_0}  \langle \ell', j' \rangle^s} \Big)^2 \\
& S_2  := \\
&  \sum_{\ell \in \Z^\nu, j \in \Z} \! \Big( \! \sum_{\langle  \ell, j \rangle \langle \ell', j' \rangle^{-1} > 2^{1/s}} 
\!   \frac{\langle  \ell, j \rangle^{s} \langle \ell - \ell', j - j' \rangle^{s} |\widehat a (\ell - \ell' , j - j', j')|  \langle \ell', j' \rangle^{s_0}  |u_{\ell', j'}| }{ \langle \ell - \ell', j - j' \rangle^{s}  \langle \ell', j' \rangle^{s_0}} \Big)^2 \, .
\end{align*}
By  Cauchy Schwartz inequality, and denoting  
$ \zeta (s_0) := \sum_{\ell \in \Z^\nu, j \in \Z} \frac{1}{\langle \ell, j \rangle^{2s_0} } $ (which is $ < + \infty $), we have
\begin{align}
S_1  \leq \!  \sum_{\ell \in \Z^\nu, j \in \Z} \! \Big( \!  
\sum_{\langle  \ell, j \rangle \langle \ell', j' \rangle^{-1} \leq 2^{1/s}} 
\!  \frac{2 \langle \ell - \ell', j - j' \rangle^{s_0} |\widehat a (\ell - \ell' , j - j', j')|  \langle \ell', j' \rangle^s  |u_{\ell', j'}| }{ \langle \ell - \ell', j - j' \rangle^{s_0}  } \Big)^2 \nonumber  & \\
\leq 4 \zeta (s_0)  \sum_{\ell \in \Z^\nu, j \in \Z} \sum_{\ell' \in \Z^\nu, j' \in \Z } 
|\widehat a (\ell - \ell' , j - j', j')|^2 \langle \ell - \ell', j - j' \rangle^{2s_0}  |u_{\ell', j'}|^2 \langle \ell', j' \rangle^{2s} 
  \nonumber & \\
   \leq 4 \zeta (s_0)   \sum_{\ell' \in \Z^\nu, j' \in \Z }   |u_{\ell', j'}|^2 \langle \ell', j' \rangle^{2s}
 \!     \sum_{\ell \in \Z^\nu, j \in \Z}  |\widehat a (\ell - \ell' , j - j', j')|^2 \langle \ell - \ell', j - j' \rangle^{2s_0}  
  \nonumber & \\
   = 4 \zeta (s_0)   \sum_{\ell' \in \Z^\nu, j' \in \Z }   |u_{\ell', j'}|^2 \langle \ell', j' \rangle^{2s}
     \sum_{\ell \in \Z^\nu, j \in \Z}  |\widehat a (\ell , j , j')|^2 \langle \ell , j  \rangle^{2s_0}  
  \nonumber & \\
   = 4 \zeta (s_0)   \sum_{\ell' \in \Z^\nu, j' \in \Z }   |u_{\ell', j'}|^2 \langle \ell', j' \rangle^{2s}
     \sum_{\ell \in \Z^\nu, j \in \Z}  \|  a (\cdot , \cdot , j') \|_{2 s_0}^2  \nonumber & \\
  \leq 4 \zeta (s_0)   \| u \|_s^2  \norma A \norma_{0,s_0,0}^2  \, . \label{stimaS1} & 
 \end{align}
For the estimate of $ S_2 $ note that, since the indices satisfy 
$ \langle \ell, j \rangle  > 2^{1/s} \langle \ell', j' \rangle $ we have 
$ \langle \ell, j \rangle \leq $ $ \langle \ell', j' \rangle + \langle \ell' - \ell, j' -  j \rangle \leq $ $ 2^{-1/s} \langle \ell, j \rangle + 
\langle \ell - \ell', j - j' \rangle $ 
and therefore
$$
\langle \ell, j \rangle \leq \big( 1 - 2^{- 1/s} \big)^{-1} \langle  \ell - \ell', j - j' \rangle \, .
$$
As a consequence, repeating the same argument used for estimating $ S_1 $, we get
\be \label{stimaS2}
S_2 \leq  C(s) \norma A \norma_{0,s,0}^2  \| u \|_{s_0}^2 \, .
\ee
By \eqref{stimaS1}, \eqref{stimaS2}, we deduce 
that 
$$
\| A u \|_{s} \leq 2 (\zeta (s_0))^{1/2}   \norma A \norma_{0,s_0,0} \, \| u \|_{s} + (C(s))^{1/2} \norma A \norma_{0,s,0}  \| u \|_{s_0} 
$$
and therefore  $ A $ is a tame operator with tame constant $ {\mathfrak M}_A(s) \leq C(s) \norma A \norma_{0,s,0} $ 
(for a different $ C(s)$).

\smallskip

Since $ \partial_\lambda^k A = {\rm Op}(\partial_\lambda^k a ) $ for any $k \in \N^{\nu + 1}$, $|k| \leq k_0$, the general case of 
\eqref{interpolazione parametri operatore funzioni} follows. . 
\end{proof}

We now discuss the action of a $ {\mathcal D}^{k_0} $-$ \s $-tame operator  $ A(\om) $ on 
Sobolev functions  $ u (\lambda) \in H^s $ which are 
${k_0}$-times differentiable with respect to  $ \lambda  \in  \Lambda_0 \subset \R^{\nu + 1} $. 
Recall the weighted norm $ \| \ \|_s^{k_0, \gamma}$ in \eqref{norma pesata derivate funzioni}.


\begin{lemma}\label{lemma operatore e funzioni dipendenti da parametro}
Let $ A := A(\lambda) $ be a $ {\mathcal D}^{k_0} $-$ \s $-tame operator. 
Then, $ \forall s \geq s_0 $, for any family of Sobolev functions $ u := u(\lambda) \in H^{s+\s} $ 
which is $k_0$-times differentiable with respect to $ \lambda $,  the following tame estimate holds
$$
\| A u \|_s^{k_0, \gamma} \leq_{k_0}  {\mathfrak M}_A(s_0) \| u \|_{s + \sigma}^{k_0, \gamma} 
+ {\mathfrak M}_A(s) \| u \|_{s_0 + \sigma}^{k_0, \gamma}  \,.
$$
\end{lemma}

\begin{proof}
For all  $|k| \leq k_0 $, $ \lambda \in \mathtt \Lambda_0 $, we have, by \eqref{CK0-sigma-tame}, \eqref{norma pesata derivate funzioni}
\begin{align*}
\| \partial_\lambda^k \big(A(\lambda) u(\lambda) \big) \|_s & \leq_{k_0} 
{\mathop \sum}_{k_1 + k_2 = k} \| (\partial_\lambda^{k_1} A (\lambda) )[\partial_\lambda^{k_2} u (\lambda) ] \|_s \\
& \leq_{ k_0} {\mathop \sum}_{k_1 + k_2 = k} \gamma^{- |k_1|} \big( {\mathfrak M}_A(s_0) 
\| \partial_\lambda^{k_2} u \|_{s + \sigma} + {\mathfrak M}_A(s) \| \partial_\lambda^{k_2} u \|_{s_0 + \sigma} \big) \\
& \leq_{ k_0} \gamma^{- |k|} \big( {\mathfrak M}_A(s_0) \| u \|_{s + \sigma}^{k_0, \gamma} 
+ {\mathfrak M}_A(s) \| u \|_{s_0 + \sigma}^{k_0, \gamma} \big) 
\end{align*}
and the lemma
follows by the definition of the norm $ \| \ \|_{s}^{k_0, \gamma} $ in \eqref{norma pesata derivate funzioni}. 
\end{proof}

Lemma \ref{lemma operatore e funzioni dipendenti da parametro}, 
 \eqref{norma a moltiplicazione} and \eqref{interpolazione parametri operatore funzioni} imply  
  tame estimates for the product of two functions in weighted Sobolev norm:   
for all $s \geq s_0$, 
\begin{equation}\label{interpolazione senza C k0}
\| u v \|_s \leq  C(s) \| u \|_s\| v \|_{s_0} + 
C(s_0) \| u \|_{s_0} \| v \|_s
\end{equation}
\begin{equation}\label{interpolazione C k0}
\| u v \|_s^{k_0, \gamma} \leq_{k_0} C(s) \| u \|_s^{k_0, \gamma} \| v \|_{s_0}^{k_0, \gamma} + 
C(s_0) \| u \|_{s_0}^{k_0, \gamma} \| v \|_s^{k_0, \gamma}\,,
\end{equation}
as well as 
the algebra estimate $\| u v \|_s^{k_0, \gamma} \leq_{k_0} C(s) \| u \|_s^{k_0, \gamma} \| v \|_{s}^{k_0, \gamma}$.
In view of the KAM reducibility scheme of section \ref{sec: reducibility} we also consider
the stronger notion of $ {\mathcal D}^{k_0} $-modulo-tame operator, that we  need  only for 
operators with loss of derivatives $ \s = 0 $.

\begin{definition}\label{def:op-tame} 
{ \bf ($ {\mathcal D}^{k_0} $-modulo-tame)} A  linear operator $ A := A(\lambda) $, $\lambda \in \mathtt \Lambda_0$ is  
$ {\mathcal D}^{k_0}$-modulo-tame  if, for all $ k \in \N^{\nu + 1} $, $ |k| \leq k_0 $,  the majorant operators $  | \partial_\lambda^k A | $ 
(Definition \ref{def:maj})
satisfy the following weighted tame estimates:   for all $ s \geq s_0 $, $ u \in H^{s} $,  
\be\label{CK0-tame}
\sup_{|k| \leq k_0} \sup_{\lambda \in \mathtt \Lambda_0}\gamma^{ |k|}
\| | \partial_\lambda^k A | u\|_s \leq  
{\mathfrak M}_{A}^\sharp (s_0) \| u \|_{s} +
{\mathfrak M}_{A}^\sharp (s) \| u \|_{s_0} 
\ee
where  the functions $ s \mapsto  {\mathfrak M}_{A}^\sharp (s)  \geq 0  $ are non-decreasing in $ s $. 
The constant $ {\mathfrak M}_A^\sharp (s) $ is called the {\sc modulo-tame constant} of the operator $ A $. 
\end{definition}

\begin{lemma}\label{A versus |A|}
An operator  $ A $ which is $ {\mathcal D}^{k_0}$-modulo-tame is also $ {\mathcal D}^{k_0}$-tame and
$ {\mathfrak M}_A (s) \leq  {\mathfrak M}_A^\sharp (s) $. 
\end{lemma}
\begin{proof}
For all $|k| \leq k_0$ one has 
\begin{align}
\| (\partial_\lambda^k A) u \|_s^2 & = {\mathop \sum}_{\ell, j} \langle \ell, j \rangle^{2 s} \big| 
{\mathop \sum}_{\ell', j'} \partial_\lambda^k A_j^{j'}(\ell - \ell') u_{\ell', j'} \big|^2 \nonumber\\
& \leq {\mathop \sum}_{\ell, j} \langle \ell, j \rangle^{2 s} \big( {\mathop \sum}_{\ell', j'} 
|\partial_\lambda^k A_j^{j'}(\ell - \ell')| |u_{\ell', j'} |\big)^2 
= \| |\partial_\lambda^k A| [\norma u \norma]\|_s^2 \nonumber
\end{align}
where $ \norma u \norma $ is the function defined in \eqref{funzioni modulo fourier}.
Then the lemma follows by \eqref{CK0-tame}, \eqref{Soboequals} and  Definition \ref{def:Ck0}. 
\end{proof}

The class of operators which are $ {\mathcal D}^{k_0} $-modulo-tame is closed under
sum and  composition.

\begin{lemma} \label{interpolazione moduli parametri} {\bf (Sum and composition)}
Let $ A, B $ be $ {\mathcal D}^{k_0} $-modulo-tame operators with modulo-tame constants respectively 
$ {\mathfrak M}_A^\sharp(s) $ and $ {\mathfrak M}_B^\sharp(s) $. Then 
$ A+ B $ is $ {\mathcal D}^{k_0} $-modulo-tame with modulo-tame constant
\be\label{modulo-tame-A+B}
{\mathfrak M}_{A + B}^\sharp (s) \leq {\mathfrak M}_A^\sharp (s)  + {\mathfrak M}_B^\sharp (s)  \,.
\ee
The composed operator 
$  A  \circ B $ is $ {\mathcal D}^{k_0} $-modulo-tame with modulo-tame constant
\begin{equation}\label{modulo tame constant for composition}
 {\mathfrak M}_{A B}^\sharp (s) \leq  C(k_0) \big( {\mathfrak M}_{A}^\sharp(s) 
 {\mathfrak M}_{B}^\sharp (s_0) + {\mathfrak M}_{A}^\sharp (s_0) 
{\mathfrak M}_{B}^\sharp (s) \big)\,.
\end{equation}
Assume in addition that  $ \langle \partial_\vphi \rangle^{\mathtt b} A $, 
$ \langle \partial_\vphi \rangle^{\mathtt b}  B $ are $ {\mathcal D}^{k_0}$-modulo-tame with modulo-tame constant 
respectively $ {\mathfrak M}_{\langle \partial_\vphi \rangle^{\mathtt b} A}^\sharp (s) $ and  
 $ {\mathfrak M}_{\langle \partial_\vphi \rangle^{\mathtt b} B}^\sharp (s) $, then 
$ \langle \partial_\vphi \rangle^{\mathtt b} (A  B) $ is $ {\mathcal D}^{k_0}$-modulo-tame with 
modulo-tame constant satisfsying
\begin{align}\label{K cal A cal B}
{\mathfrak M}_{\langle \partial_\vphi \rangle^{\mathtt b} (A  B)}^\sharp (s) & \leq
 C({\mathtt b}) C(k_0)\Big( 
{\mathfrak M}_{\langle \partial_\vphi \rangle^{\mathtt b} A}^\sharp (s) 
{\mathfrak M}_{B}^\sharp (s_0) + 
{\mathfrak M}_{\langle \partial_\vphi \rangle^{\mathtt b} A }^\sharp (s_0) 
{\mathfrak M}_{B}^\sharp (s) \nonumber \\ 
& \qquad \qquad \qquad \quad + {\mathfrak M}_{A}^\sharp (s) {\mathfrak M}_{ \langle \pa_\vphi \rangle^{\mathtt b} B}^\sharp (s_0) 
+ {\mathfrak M}_{A}^\sharp (s_0) {\mathfrak M}_{ \langle \pa_\vphi \rangle^{\mathtt b} B}^\sharp (s)\Big) \,.
\end{align}
The constants $ C(k_0) , C( {\mathtt b} ) \geq 1 $. 
\end{lemma}
\begin{proof}
The bound \eqref{modulo-tame-A+B} follows by \eqref{disuguaglianza importante moduli} and \eqref{Soboequals}. 

\noindent
{\sc Proof of \eqref{modulo tame constant for composition}.}
For all $|k| \leq k_0$ we have 
\begin{align}
\gamma^{|k|}\| |\partial_\lambda^k (A B)| u \|_s & \leq C(k_0) 
 \gamma^{|k|} {\mathop \sum}_{k_1 + k_2 = k} \| |(\partial_\lambda^{k_1}  A)(\partial_\lambda^{k_2} B)| u\|_s \nonumber\\
& \stackrel{\eqref{disuguaglianza importante moduli}} \leq C(k_0)   {\mathop \sum}_{k_1 + k_2 = k} 
\gamma^{|k_1|} \gamma^{|k_2|} \| |\partial_\lambda^{k_1} A| |\partial_\lambda^{k_2} 
B| [\norma u \norma] \|_s \nonumber\\
& \stackrel{\eqref{CK0-tame}} \leq
C(k_0)   {\mathop \sum}_{|k_2| \leq |k|}  {\mathfrak M}_A^\sharp (s_0) \gamma^{|k_2|} \| 
 |\partial_\lambda^{k_2} B|[\norma u \norma] \|_s 
 \nonumber\\
 & \quad \quad + C(k_0)   {\mathop \sum}_{|k_2| \leq |k|} {\mathfrak M}_A^\sharp (s) \gamma^{|k_2|}
 \| |\partial_\lambda^{k_2} B| [\norma u \norma] \|_{s_0}  \nonumber\\
&  \stackrel{\eqref{CK0-tame}, \eqref{Soboequals}} \leq
C(k_0)   {\mathfrak M}_A^\sharp (s_0) {\mathfrak M}_B^\sharp (s_0) \|  u \|_s 
 \nonumber \\
 & \quad \quad \quad   + C(k_0)  \big( {\mathfrak M}_A^\sharp (s) {\mathfrak M}_B^\sharp (s_0) + 
{\mathfrak M}_A^\sharp (s_0) {\mathfrak M}_B^\sharp (s) \big) \|  u  \|_{s_0}   \nonumber
\end{align}
and  \eqref{modulo tame constant for composition}  follows by recalling Definition \ref{def:op-tame}.

\noindent
{\sc Proof of \eqref{K cal A cal B}.} For all $|k| \leq k_0$ we have (use 
the first inequality in \eqref{disuguaglianza importante moduli})
\begin{equation}\label{udine - 1}
\big\| \big| \langle \partial_\vphi \rangle^{\mathtt b}
\big[ \partial_\lambda^k ( A B) \big] \big| u \big\|_s \leq C(k_0) \sum_{k_1 + k_2 = k} 
\big\| \big| \langle \partial_\vphi \rangle^{\mathtt b} 
\big[ (\partial_\lambda^{k_1} A)(\partial_\lambda^{k_2}  B)\big] \big| \norma u \norma \big\|_s\,.
\end{equation}
Next, recalling  the Definition \ref{def:maj} of the operator $ \langle \partial_\vphi \rangle^{\mathtt b} $ and 
\eqref{funzioni modulo fourier}, we have 
\begin{align}
 \Big\| \big|\langle \partial_\vphi \rangle^{\mathtt b} \big[(\partial_\lambda^{k_1} 
A)(\partial_\lambda^{k_2} B)\big] \big| \norma u \norma \Big\|_s^2  = &  \label{udine 0} \\
  \sum_{\ell, j} \langle \ell, j \rangle^{2 s} \Big(\sum_{\ell', j'} 
| \langle \ell - \ell' \rangle^{\mathtt b}[( \partial_\lambda^{k_1} A) 
( \partial_\lambda^{k_2} B)]_j^{j'}(\ell - \ell') | |u_{\ell', j'}| \Big)^2  & \nonumber  \\
\leq \sum_{\ell, j} \langle \ell, j \rangle^{2 s} \Big(\sum_{\ell', j', \ell_1, j_1}  \langle \ell - \ell' \rangle^{\mathtt b} |(\partial_\lambda^{k_1} A)_j^{j_1}(\ell - \ell_1)| 
|(\partial_\lambda^{k_2} B)_{j_1}^{j'}(\ell_1 - \ell')| |u_{\ell', j'}| \Big)^2\,. \nonumber & 
\end{align}
Since 
$ \langle \ell - \ell' \rangle^{\mathtt b} \leq C(\mathtt b) (\langle \ell - \ell_1 \rangle^{\mathtt b} + \langle \ell_1 - \ell' \rangle^{\mathtt b}) $, 
we deduce that
\begin{align}
\eqref{udine 0} & \leq C(\mathtt b)^2 
\sum_{\ell, j} \langle \ell, j \rangle^{2 s} \Big(\sum_{\ell', j', \ell_1, j_1} 
| \langle \ell - \ell_1 \rangle^{\mathtt b} (\partial_\lambda^{k_1} A)_j^{j_1}(\ell - \ell_1)| \times \nonumber \\
& \qquad \qquad \qquad  \qquad \qquad \qquad 
\times |(\partial_\lambda^{k_2} B)_{j_1}^{j'}(\ell_1 - \ell')| |u_{\ell', j'}| \Big)^2 \nonumber\\
& \quad + C(\mathtt b)^2 \sum_{\ell, j} \langle \ell, j \rangle^{2 s} \Big(\sum_{\ell', j', \ell_1, j_1}  |(\partial_\lambda^{k_1} 
A)_j^{j_1}(\ell - \ell_1)| \times \nonumber \\ 
&  \qquad \qquad \qquad  \qquad \qquad \qquad \times | \langle \ell_1 - \ell' \rangle^{\mathtt b}(\partial_\lambda^{k_2} 
B)_{j_1}^{j'}(\ell_1 - \ell')| |u_{\ell', j'}| \Big)^2 \nonumber\\
& \leq C(\mathtt b)^2 \Big(
 \Big\| \big| \langle \partial_\vphi \rangle^{\mathtt b} (\partial_\lambda^{k_1} A) 
 \big| \big[| \partial_\lambda^{k_2} B| \norma u \norma \big] \Big\|_s^2   + 
 \Big\| \big| \partial_\lambda^{k_1} A \big| 
 \big[ | \langle \partial_\vphi \rangle^{\mathtt b} (\partial_\lambda^{k_2} B)
 \big| \norma u \norma 	\big] \Big\|_s^2 \Big) \, . \label{udine 1}
 \end{align}
Hence \eqref{udine - 1}-\eqref{udine 1}, \eqref{CK0-tame} and \eqref{Soboequals} imply 
\begin{align*}
& \big\| \big| \langle \partial_\vphi \rangle^{\mathtt b}
 \big[ \partial_\lambda^k ( A B) \big] \big| u \big\|_s  \nonumber\\
&  \leq C(\mathtt b) C(k_0) \gamma^{- |k|} \big( {\mathfrak M}_{\langle \partial_\vphi \rangle^{\mathtt b} A}^\sharp(s_0 ) 
{\mathfrak M}_{B}^\sharp(s_0) + {\mathfrak M}_{A}^\sharp(s_0) {\mathfrak M}_{\langle \partial_\vphi \rangle^{\mathtt b} B}^\sharp(s_0) \big) \| u\|_s \nonumber\\
&  \quad + C(\mathtt b) C(k_0)
\gamma^{- |k|} \Big( {\mathfrak M}_{ \langle \partial_\vphi \rangle^{\mathtt b} A}^\sharp(s) {\mathfrak M}_{B}^\sharp(s_0) + 
{\mathfrak M}_{ \langle \partial_\vphi \rangle^{\mathtt b}A}^\sharp(s_0) {\mathfrak M}_{B}^\sharp(s)  \nonumber\\
& \quad  + {\mathfrak M}_{A}^\sharp(s) {\mathfrak M}^\sharp_{\langle \partial_\vphi \rangle^{\mathtt b} B}(s_0) + {\mathfrak M}_{ A}^\sharp(s_0) {\mathfrak M}_{\langle \partial_\vphi \rangle^{\mathtt b} B}^\sharp(s) \Big) \| u \|_{s_0} 
\end{align*}
which proves \eqref{K cal A cal B}.  
\end{proof}

As a consequence of \eqref{modulo tame constant for composition}, if $ A $ is $ {\mathcal D}^{k_0}$-modulo-tame, then, 
for all $n \geq 1$, each  $ A^n  $ is $ {\mathcal D}^{k_0}$-modulo-tame and 
\begin{equation}\label{M Psi n}
{\mathfrak M}_{A^n}^\sharp (s) \leq \big( 2 C (k_0) {\mathfrak M}_{A}^\sharp (s_0) \big)^{n - 1} {\mathfrak M}_{A}^\sharp(s)\,.
\end{equation}
Moreover, by \eqref{K cal A cal B},  if $ \langle \pa_\vphi  \rangle^{\mathtt b} A $ is $ {\mathcal D}^{k_0} $-modulo-tame, then,  
for all $ n \geq 2 $, each $ \langle \partial_\vphi \rangle^{\mathtt b}  A^n $ 
is $ {\mathcal D}^{k_0}$-modulo-tame with
\begin{equation}\label{K Psi n}
\begin{aligned}
{\mathfrak M}_{\langle \pa_\vphi \rangle^{ \mathtt b} A^n}^\sharp (s) & \leq  (4 C(\mathtt b) C(k_0))^{n - 1} \Big( 
{\mathfrak M}^\sharp_{\langle \partial_\vphi \rangle^{\mathtt b} A}(s) \big[ {\mathfrak M}^\sharp_A(s_0) \big]^{n - 1} \\  
 & \quad + {\mathfrak M}^\sharp_{\langle \partial_\vphi \rangle^{\mathtt b} A}(s_0) {\mathfrak M}_A^\sharp(s) 
 \big[ {\mathfrak M}_A^\sharp(s_0) \big]^{n - 2} \Big)\,.
 \end{aligned}
\end{equation}

\begin{lemma}[{\bf Invertibility}]\label{serie di neumann per maggioranti}
Let $\Phi := {\rm Id} + A $ where $ A  := A (\lambda)$ is 
${\mathcal D}^{k_0}$-modulo-tame with modulo-tame constant $ {\mathfrak M}_{A}^\sharp (s)  $. Assume 
the smallness condition 
\begin{equation}\label{piccolezza neumann tame}
4 C(\mathtt b) C(k_0)  {\mathfrak M}_{A}^\sharp (s_0)  \leq 1/ 2\,.
\end{equation}
Then the operator $ \Phi $ is invertible,  
$\check A :=   \Phi^{- 1} - {\rm Id}  $  is 
${\mathcal D}^{k_0}$-modulo-tame with modulo-tame constant
\begin{equation}\label{disuguaglianza constante tame A tilde A}
 {\mathfrak M}_{\check A}^\sharp (s) \leq  2 {\mathfrak M}_A^\sharp (s) \, . 
 \end{equation}
Moreover 
$ \langle \partial_\vphi \rangle^{\mathtt b}  \check A $ is $ {\mathcal D}^{k_0}$-modulo-tame with 
tame-constant 
\begin{equation}\label{Psi tilde alta Neumann moduli}
{\mathfrak M}_{\langle \partial_\vphi \rangle^{\mathtt b}  \check A}^\sharp (s)  \leq 
2 {\mathfrak M}_{ \langle \partial_\vphi \rangle^{\mathtt b}A}^\sharp (s)  + 
8  C(\mathtt b) C(k_0)  {\mathfrak M}_{ \langle \partial_\vphi \rangle^{\mathtt b}A}^\sharp (s_0)\, {\mathfrak M}_A^\sharp(s) \, .  
\end{equation}
\end{lemma}

\begin{proof}
By \eqref{norma operatoriale costante tame} and  \eqref{piccolezza neumann tame} 
the operatorial norm $ \| A \|_{{\mathcal L}(H^{s_0})} \leq 2 {\mathfrak M}_A^{\sharp} (s_0) \leq 1 / 2 $. 
Then $\Phi$ is invertible and the inverse operator 
$ \Phi^{- 1} =  {\rm Id} + \check A  $ with $ \check A  := \sum_{n \geq 1} (- 1)^n A^n $ satisfy  
the estimate  \eqref{disuguaglianza constante tame A tilde A}  
 by \eqref{modulo-tame-A+B}, \eqref{M Psi n}, \eqref{piccolezza neumann tame}. 
Similarly \eqref{Psi tilde alta Neumann moduli} follows by \eqref{modulo-tame-A+B}, \eqref{K Psi n}
and \eqref{piccolezza neumann tame}. 
\end{proof}

\begin{lemma} \label{lemma:smoothing-tame} {\bf (Smoothing)}
Suppose that $ \langle \pa_\vphi \rangle^{\mathtt b} A $, $ {\mathtt b} \geq  0 $, is $ {\mathcal D}^{k_0} $-modulo-tame. Then 
the operator $ \Pi_N^\bot A $ is $ {\mathcal D}^{k_0} $-modulo-tame with tame constant
\be\label{proprieta tame proiettori moduli}
{\mathfrak M}_{\Pi_N^\bot A}^\sharp (s) \leq N^{- {\mathtt b} }{\mathfrak M}_{ \langle \pa_\vphi \rangle^{\mathtt b} A}^\sharp (s) \, ,
\quad
{\mathfrak M}_{\Pi_N^\bot A}^\sharp (s) \leq  {\mathfrak M}_{ A}^\sharp (s) \, . 
\ee
\end{lemma}

\begin{proof}
For all $ |k| \leq k_0$ one has, recalling \eqref{proiettore-oper},  
\begin{align}
 \| |\Pi_{N}^\bot \partial_\lambda^k A| u \|_s^2 & 
= \sum_{\ell, j} \langle \ell, j \rangle^{2 s} \Big( \sum_{j' ,|\ell - \ell'| > N} |\partial_\lambda^k A_{j}^{j'}(\ell - \ell')| |u_{\ell' j'}| \Big)^2 \nonumber\\
& \leq N^{- 2 { {\mathtt b} }}  {\mathop \sum}_{\ell, j} \langle \ell, j \rangle^{2 s}
\big( {\mathop \sum}_{j', \ell'}  
| \langle \ell - \ell' \rangle^{ {\mathtt b}} \partial_\lambda^k A_{j}^{j'}(\ell - \ell')| |u_{\ell' j'}| \big)^2 \nonumber\\
& =  N^{- 2 {\mathtt b}} \| |  \langle \partial_\vphi \rangle^{ {\mathtt b}}(\partial_\lambda^k A) | \, [\norma u \norma] \|_s^2 \nonumber
\end{align}
and, using \eqref{CK0-tame}, \eqref{Soboequals}, we deduce the first inequality in \eqref{proprieta tame proiettori moduli}. 
Similarly  we get
$ \| |\Pi_{N}^\bot \partial_\lambda^k A| u \|_s^2 
\leq \| | \partial_\lambda^k A | \, \norma u \norma \, \|_s^2 $ which implies the second inequality in \eqref{proprieta tame proiettori moduli}.  
\end{proof}

The next two lemmata 
will be used in the proof of Theorem \ref{ITERAZIONERIDUCIBILITA}-$({\bf S3})_\nu$. 

\begin{lemma}\label{proprieta norme operatoriali}
Let $A$ and $B$ be linear operators such that 
$ |A|, |\langle \partial_\vphi\rangle^{\mathtt b} A|, $ $|B|, $ $ |\langle \partial_\vphi\rangle^{\mathtt b} B| \in {\mathcal L}(H^{s_0})$. 
Then 
\begin{enumerate}
\item \begin{align}
& \| |A + B| \|_{{\mathcal L}(H^{s_0})} \leq \| |A| \|_{{\mathcal L}(H^{s_0})} + \| |B| \|_{{\mathcal L}(H^{s_0})}\,,  \nonumber\\
& \| |A B| \|_{{\mathcal L}(H^{s_0})} \leq \| |A| \|_{{\mathcal L}(H^{s_0})} \| |B| \|_{{\mathcal L}(H^{s_0})}\,, \nonumber
\end{align}
\item 
 \begin{align*}
 \| | \langle \partial_\vphi \rangle^{\mathtt b} (A B) | \|_{{\mathcal L}(H^{s_0})}& \leq_{\mathtt b} \| | \langle \partial_\vphi\rangle^{\mathtt b} A | \|_{{\mathcal L}(H^{s_0})} \| |B| \|_{{\mathcal L}(H^{s_0})} \\
 & \quad \quad + \| |A| \|_{{\mathcal L}(H^{s_0})} \| |\langle \partial_\vphi\rangle^{\mathtt b} B| \|_{{\mathcal L}(H^{s_0})}, 
 \end{align*}
\item
\begin{align*}
 \| |\Pi_N^\bot A| \|_{{\mathcal L}(H^{s_0})} & \leq N^{- \mathtt b} \| |\langle \partial_\vphi \rangle^{\mathtt b} A|\|_{{\mathcal L}(H^{s_0})}\,,  \\
   \quad \| |\Pi_N^\bot A|\|_{{\mathcal L}(H^{s_0})} & \leq \||A| \|_{{\mathcal L}(H^{s_0})}. 
\end{align*}
\end{enumerate}
\end{lemma}
\begin{proof}
Item 1  is a direct consequence of \eqref{disuguaglianza importante moduli} and \eqref{Soboequals}. 
Items 2-3 are proved  arguing as in Lemmata \ref{interpolazione moduli parametri} and \ref{lemma:smoothing-tame}.
\end{proof}

\begin{lemma}\label{Delta 12 Phi inverso}
Let $ \Phi_i := {\rm Id} + \Psi_i$, $i = 1,2$, satisfy,   
\begin{equation}\label{piccolezza Psi i1 i2}
\| | \Psi_i |  \|_{{\mathcal L}(H^{s_0})} \leq 1/ 2 \,, \quad i = 1,2 \,.
\end{equation}
Then $\Phi_i^{- 1} = {\rm Id} + \check \Psi_i$, $i = 1, 2$, satisfy
$ \| |\check \Psi_1 - \check \Psi_2| \|_{{\mathcal L}(H^{s_0})} \leq 4 \| |\Psi_1 - \Psi_2| \|_{{\mathcal L}(H^{s_0})} $ and 
\begin{align*}
& \|  \langle \partial_\vphi \rangle^{\mathtt b} |\check \Psi_1 - \check \Psi_2 | 
\|_{{\mathcal L}(H^{s_0})}  
\leq_{\mathtt b} \| \langle \partial_\vphi \rangle^{\mathtt b} | \Psi_1 - \Psi_2 | \|_{{\mathcal L}(H^{s_0})} \\
& \quad + 
\big( 1 + \| |\langle \partial_\vphi \rangle^{\mathtt b} \check \Psi_1|\|_{{\mathcal L}(H^{s_0})}  +
\| |\langle \partial_\vphi \rangle^{\mathtt b} \check \Psi_2| \|_{{\mathcal L}(H^{s_0})}\big) \| |\Psi_1 - \Psi_2| \|_{{\mathcal L}(H^{s_0})} \, . 
\end{align*}
\end{lemma}
\begin{proof}
Use  $ \check \Psi_1 - \check \Psi_2 = \Phi_1^{- 1} - \Phi_2^{- 2} = \Phi_1^{- 1} (\Psi_2 - \Psi_1) \Phi_2^{- 1} $
and apply Lemma \ref{proprieta norme operatoriali}-1-2,  using \eqref{piccolezza Psi i1 i2}. 
\end{proof}

The composition operator $ u(y) \mapsto  u(y+p(y)) $
induced by a diffeomorphism of the torus $ \T^{d} $ is tame.

\begin{lemma} {\bf (Change of variable)}  \label{lemma:utile} 
Let $p:= p( \lambda, \cdot ):\R^d \to \R^d$ be a family of $2\p$-periodic functions   
which is ${k_0}$-times differentiable with respect to $\lambda \in \mathtt \Lambda_0 \subset \R^{\nu + 1} $,  satisfying   
\begin{equation}\label{mille condizioni p}
  \| p \|_{{\mathcal C}^{s_0 + 1}} \leq 1/2\,,\quad  \| p \|_{s_0}^{k_0, \gamma} \leq 1\,. 
\end{equation}
Let $ g(y) := y + p(y) $,
$ y \in \T^d $. 
Then the composition operator 
$$
A : u(y) \mapsto (u\circ g)(y) = u(y+p(y))
$$ 
satisfies the tame estimates
\begin{equation}\label{stima cambio di variabile dentro la dim}
\| A u\|_{s_0} \leq_{s_0} \| u \|_{s_0}\,, \quad \| A u\|_s \leq C(s)  \| u \|_s +  C(s_0)\| p \|_s \| u \|_{s_0 + 1}\,, \  \forall  s \geq s_0 + 1\, , 
\end{equation}
and for any $ |k| \leq k_0 $, 
\begin{align}
& \| (\partial_\lambda^k A) u \|_{s_0} \leq_{s_0, k} \gamma^{- |k|} \| u \|_{s_0 + |k|}\,, \label{stima tame cambio di variabile pietro s0}\\
\label{stima tame cambio di variabile pietro}
    & \| (\partial_\lambda^k A)u \|_s \leq_{s, k} \gamma^{- |k|} \big( \| u \|_{s + |k|}  + 
\| p \|_{s}^{|k|, \gamma} \| u \|_{{s_0 + |k| + 1}} \big)\,, \quad \forall  s \geq s_0 + 1 \,.
\end{align}
The map $ g $ is invertible with inverse  $ g^{- 1}(z) = z + q(z) $. Suppose $\partial_\lambda^k p(\lambda, \cdot) \in {\mathcal C}^\infty(\T^{d})$ for all $|k| \leq k_0$. There exists a constant $\delta := \delta(s_0, k_0) \in (0,1) $ such that, if  $ \| p \|_{2 s_0 + k_0 + 1}^{k_0, \gamma} \leq \d$, then  
\begin{equation}\label{stime-lipschitz-q}
\| q \|_{s}^{k_0, \gamma}  \leq_{s,k_0}  \| p \|_{s + k_0}^{k_0, \gamma} \,, \quad \forall s \geq s_0\,. 
\end{equation}
The composition operators $A$ and $A^{- 1}$ are ${\mathcal D}^{k_0}$-$(k_0 + 1)$-tame with tame constants satisfying for any $S > s_0$,  
\begin{equation}\label{tame-lipschitz-cambio-di-variabile}
{\mathfrak M}_A(s) \leq_{S, k_0} 1 + \|p \|_{s }^{k_0, \gamma}\,,\quad {\mathfrak M}_{A^{- 1}}(s) \leq_{S,  k_0} 1 + \| p \|_{{s + k_0 }}^{k_0, \gamma} \, , \quad \forall   s_0 \leq s \leq S\, . 
\end{equation}
\end{lemma}

\begin{proof}

\noindent
{\sc Proof of \eqref{stima cambio di variabile dentro la dim}.}
  By Lemma B.4-$(ii)$ in \cite{Baldi-Benj-Ono} and \eqref{mille condizioni p}, we have 
\begin{equation}\label{ppooll}
\| A u\|_{s_0} \leq_{s_0} \| u \|_{s_0} + \| p\|_{{\mathcal C}^{s_0}} \| u\|_1 \leq_{s_0} \| u \|_{s_0}  \quad {\rm and}
\quad 
\| A u\|_{s_0+1} \leq_{s_0} \| u \|_{s_0+1} \, . 
\end{equation}
Thus the the first inequality in \eqref{stima cambio di variabile dentro la dim}, and the second one for 
$ s = s_0 + 1 $, are proved. 
Now we prove the second inequality in \eqref{stima cambio di variabile dentro la dim}, arguing by induction on $ s $.
We assume that 
it holds for $ s \geq s_0 + 1 $ and we prove it for $s + 1$.
As a notation we denote by $ \nabla u := ( u_{x_1}, \ldots, u_{x_d} ) $ the gradient of the function $ u $
and $ A(\nabla u) := ( A u_{x_1}, \ldots, A u_{x_d} ) $. 
By the definition of the $ \| \ \|_{s+1}$  norm and  \eqref{interpolazione senza C k0} we have 
\begin{align}
\| A u\|_{s + 1} & \leq \| A  u\|_{L^2} + \max_{|\alpha| = 1}\| \partial_x^\alpha (A u) \|_s  \nonumber\\
& \leq \| A u\|_{L^2} + 
C(s) \| A (\nabla u)\|_s +  C(s) \| A(\nabla u)\|_s \| p\|_{s_0 + 1}  \nonumber\\
& \quad + C(s_0) \| A(\nabla u)\|_{s_0} \| p\|_{s + 1}\, . \nonumber 
\end{align}
Hence, by the inductive hyphothesis and using \eqref{mille condizioni p}, \eqref{ppooll}, we get
\begin{align}
\| A u\|_{s + 1} & \leq C_1(s) \| u \|_{s + 1} + C_1(s)\| p\|_s \| u\|_{s_0 + 2} +  C_0 (s_0) \| p \|_{s + 1} \| u \|_{s_0 + 1} \label{http}
\end{align}
for some constants $ C_1 (s), C_0 (s_0) > 0 $.  
Applying \eqref{interpolation estremi fine} with $a_0 = b_0 = s_0 + 1$, $q = 1$, $p = s - s_0 - 1 $, 
 $ \epsilon = 1 / C_1(s) $, we estimate
$$
C_1 (s) \| p\|_s \| u \|_{s_0 + 2} \leq  \| p\|_{s + 1} \| u \|_{s_0 + 1} + C_2(s) \| p\|_{s_0 + 1} \| u \|_{s + 1}\,,
$$
and, by \eqref{http}, using again that $\| p\|_{s_0 + 1} \leq 1$, we get 
$$
\| A u\|_{s + 1} \leq C(s + 1)\| u \|_{s + 1} + C(s_0)\| p\|_{s + 1} \| u \|_{s_0 + 1}\,,
$$
with $ C(s+1) = C_1 (s) + C_2 (s) $ and $ C(s_0) = 1 + C_0 (s_0) $.
This is \eqref{stima cambio di variabile dentro la dim} for the Sobolev index $s + 1$. 

  \noindent
 {\sc Proof of \eqref{stima tame cambio di variabile pietro s0}-\eqref{stima tame cambio di variabile pietro}.} We prove the estimate \eqref{stima tame cambio di variabile pietro}. 
 We argue by induction on $|k| \leq k_0$. For 
 $ k = 0$, the estimate \eqref{stima tame cambio di variabile pietro} follows by 
 \eqref{stima cambio di variabile dentro la dim}. 
Now we assume that  \eqref{stima tame cambio di variabile pietro} holds for any $|k| \leq n < k_0$ and 
we prove it for $n + 1$. Let $\alpha  \in \N^{\nu + 1}$ such that $|\alpha| = 1$. One has 
\begin{equation}\label{formula derivata in omega cambio di variabile}
(\partial_\lambda^{k + \alpha}A) u = \partial_\lambda^k ( A ( \nabla u) \cdot \pa_\lambda^\alpha p ) = 
{\mathop \sum}_{k_1 + k_2 = k} C(k_1, k_2) (\partial_\lambda^{k_1} A)(\nabla u) \cdot \partial_\lambda^{k_2 + \alpha} p\,.
\end{equation}
For any $k_1, k_2 \in \N^{\nu + 1}$, with $k_1 + k_2 = k$, we have, using \eqref{interpolazione senza C k0},  
\begin{align}
& \| (\partial_\lambda^{k_1} A)(\nabla u) \cdot \partial_\lambda^{k_2 + \alpha} p \|_s   \nonumber \\ 
&  \leq_s \| (\partial_\lambda^{k_1} A)(\nabla u) \|_s \| \partial_\lambda^{k_2 + \alpha} p \|_{s_0} + \| (\partial_\lambda^{k_1} A)(\nabla u) \|_{s_0} \| \partial_\lambda^{k_2 + \alpha} p \|_{s} \nonumber\\
& \stackrel{\eqref{stima tame cambio di variabile pietro s0}, \eqref{stima tame cambio di variabile pietro}}{\leq_{s, k_1}}  \gamma^{- |k_1|}\big( \| u \|_{s + |k_1| + 1} + \| p \|_s^{|k_1|, \gamma} \| u \|_{s_0 + |k_1| + 2} \big) \gamma^{- (|k_2| + 1)}\| p \|_{s_0}^{|k_2| + 1, \gamma} \nonumber\\
& \quad \quad \quad \quad + \gamma^{- |k_1|}\|u  \|_{s_0 + |k_1| + 2} \gamma^{- (|k_2+ 1)} \|p \|_s^{|k_2|+1, \gamma} \nonumber\\
& \stackrel{\eqref{mille condizioni p}}{\leq_{s, k_1}} \gamma^{- (|k| + 1)} \big( \| u \|_{s + |k| + 1} + \| p \|_s^{|k| + 1, \gamma} \|u \|_{s_0 + |k| + 2} \big) \nonumber
\end{align}
and recalling \eqref{formula derivata in omega cambio di variabile} we get 
the estimate \eqref{stima tame cambio di variabile pietro} for $ |k| + 1 $. 

\noindent
{\sc Proof of \eqref{stime-lipschitz-q}.} Since     
$ y + p(\lambda, y) = z \iff z + q(\lambda, z) = y $ the function $q(\lambda, z)$ satisfies 
\begin{equation}\label{identita p q}
 q(\lambda, z) + p(\lambda, z + q(\lambda, z)) = 0 .
 \end{equation}
If $ p \in {\mathcal C}^1$ with respect to $ (\lambda, y)$, then, by the standard implicit function theorem,  $ q $ is ${\mathcal C}^1$ with respect to $ (\lambda, z)$ and  by differentiating the identity \eqref{identita p q} one gets, denoting by $D_\lambda$, $D_y$, $D_z$ the Fr\'echet derivatives with respect to the variables $\lambda$, $y$, $z$, 
\begin{align}
D_\lambda q (\lambda, z) & = -  \big( {\rm Id} + D_y p(\lambda, z + q (\lambda, z))  \big)^{- 1} D_\lambda p(\lambda, z + q (\lambda, z))\,,  \nonumber\\
D_z q (\lambda, z) & = -  \big( {\rm Id} + D_y p(\lambda, z + q (\lambda, z))  \big)^{- 1} D_x p(\lambda, z + q (\lambda, z))\,. \nonumber
\end{align}
It then follows by usual bootstrap arguments that if $p$ is $k_0$-times differentiable with respect to $\lambda$ and $\partial_\lambda^k p(\lambda, \cdot) \in {\mathcal C}^\infty$ for any $|k| \leq k_0$, then $q$ is $k_0$-times differentiable with respect to $\lambda$ and $\partial_\lambda^k q(\lambda, \cdot) \in {\mathcal C}^\infty$ for any $|k| \leq k_0$.  We now prove
\begin{equation}\label{claim p q}
\| \partial_\lambda^k q\|_s \leq_s \gamma^{- |k|} \| p \|_{s + |k|}^{|k|, \gamma}\,, \qquad \forall k \in \N^{\nu + 1}\,, \   |k| \leq k_0\,,
\end{equation}
which,  recalling \eqref{norma pesata derivate funzioni},  implies \eqref{stime-lipschitz-q}. 
Denote by $ A_q $ the composition operator 
$$
A_q : h (x) \mapsto h(x + q(x))
$$
so that $  q = - A_q [p] $. 
By differentiating the equation $q(\lambda, z) + p(\lambda, z + q(\lambda, z)) = 0 $, $(s_0 + 1)$-times, one gets that $\| q \|_{{\mathcal C}^{s_0 + 1}} \leq C(s_0) \| p \|_{{\mathcal C}^{s_0 + 1}} \leq 1/2 $, provided $\| p \|_{{\mathcal C}^{s_0 + 1}}$ is small enough and $\| q \|_{s_0}^{k_0, \gamma} \leq C(s_0) \| p \|_{s_0 + k_0}^{k_0, \gamma} \leq 1/2$, provided $\| p \|_{s_0 + k_0}^{k_0, \gamma}$ small enough. Therefore, we can apply the estimates \eqref{stima cambio di variabile dentro la dim}-\eqref{stima tame cambio di variabile pietro} to the operator $A_q$.
By \eqref{stima cambio di variabile dentro la dim}, one has  
$$
\| q \|_s = \| A_q ( p ) \|_s \leq C(s) \| p \|_s + C(s_0) \| q \|_s \| p \|_{s_0 + 1}\,,
$$
which, for $C(s_0)\| p \|_{s_0 + 1} \leq 1/2 $,  implies  \eqref{claim p q} for $k = 0$. Now we 
assume that \eqref{claim p q} holds up to $|k| = n$ and 
we prove it for $n + 1$. Let $\alpha \in \N^{\nu + 1}$ such that $|\alpha| = 1$. We have  
\begin{align}
\partial_\lambda^{k + \alpha} q & = - \partial_\lambda^{k + \alpha} (A_q( p )) = 
- \partial_\lambda^k \big(A_q(\nabla p) \cdot \partial_\lambda^\alpha q  + A_q(\partial_\lambda^\alpha p) \big) = - A_q(\nabla p)\cdot \partial_\lambda^{k + \alpha } q\nonumber\\
&   -  \sum_{
k_1 + k_2 = k, |k_2 | < |k|} C_{k_1, k_2} \partial_\lambda^{k_1} \big( A_q(\nabla p)\big) \cdot \partial_\lambda^{k_2 + \alpha} q 
-  \partial_\lambda^k \big(A_q(\partial_\lambda^\alpha p)\big) \,. \nonumber
\end{align}
Using \eqref{interpolazione senza C k0} we get 
\begin{align}
 \| \partial_\lambda^{k + \alpha} q \|_s  
& \leq C(s_0) \| A_q(\nabla p) \|_{s_0} \| \partial_\lambda^{k + \alpha} q\|_s + C(s) \| A_q(\nabla p)\|_s \| \partial_\lambda^{k + \alpha} q\|_{s_0} \nonumber\\
& + \| \partial_\lambda^k (A_q(\partial_\lambda^\alpha p)) \|_s + C(k, s) \!\!\!\!  \sum_{
k_1 + k_2 = k, |k_2| < |k|} \!\! \| \partial_\lambda^{k_1} ( A_q(\nabla p) ) \|_s \| \partial_\lambda^{k_2 + \alpha} q \|_{s_0} \nonumber\\
& \  + C(k, s) \!\!\!\!  \sum_{
k_1 + k_2 = k, |k_2| < |k|}  \| \partial_\lambda^{k_1} ( A_q(\nabla p) ) \|_{s_0} \| \partial_\lambda^{k_2 + \alpha} q \|_{s} \nonumber\\
& \stackrel{\eqref{stima cambio di variabile dentro la dim}, \eqref{claim p q}, \| p\|_{s_0 + 2} \leq 1}{\leq} C_1(s_0) \| p \|_{s_0 + 1} \| \partial_\lambda^{k + \alpha} q\|_s + C_1(s)  \| p \|_{s + 1} \| \partial_\lambda^{k + \alpha} q \|_{s_0}  \nonumber\\
& \qquad \qquad \qquad \qquad + \gamma^{- |k|} \| A_q(\partial_\lambda^\alpha p)\|_s^{|k|, \gamma} \nonumber\\
& \qquad \qquad \qquad \qquad  + \gamma^{- (|k| + 1)} C_1(k, s) \sum_{
\begin{subarray}{c}
k_1 + k_2 = k\\
|k_2| < |k|
\end{subarray}} \| A_q(\nabla p )\|_{s}^{|k_1|, \gamma} \| p \|_{s_0 + |k_2| + 1}^{|k_2| + 1, \gamma}  \nonumber\\
& \qquad \qquad \qquad \qquad + \| A_q(\nabla p )\|_{s_0}^{|k_1|, \gamma} \| p \|_{s + |k_2| + 1}^{|k_2| + 1 , \gamma} \nonumber\\
& \leq
 C_1(s_0) \| p \|_{s_0 + 1} \| \partial_\lambda^{k + \alpha} q\|_s + C_1(s)  \| p \|_{s + 1} \| \partial_\lambda^{k + \alpha} q \|_{s_0}  \nonumber\\
 & \quad + C_2(s, k) \gamma^{- (|k| + 1)}\| p \|_{s + |k| + 1}^{|k| + 1, \gamma}   \label{stima-a-sinistra}
\end{align} 
using \eqref{stima tame cambio di variabile pietro s0}, \eqref{stima tame cambio di variabile pietro}, \eqref{claim p q}, 
Lemma \eqref{lemma operatore e funzioni dipendenti da parametro} and $\| p\|_{s_0 + k_0 + 1}^{k_0, \gamma} \leq 1$. Then, for $s = s_0$, one has 
\begin{equation}\label{partial omega k alpha q s0}
\| \partial_\lambda^{k + \alpha} q \|_{s_0}  \leq 2 C_1(s_0) \| p \|_{s_0 + 1} \| \partial_\lambda^{k + \alpha} q \|_{s_0} + C_2(s_0, k) \gamma^{- (|k| + 1)}\| p \|_{s_0 + |k| + 1}^{|k| + 1, \gamma}\,,
\end{equation}
implying  \eqref{claim p q} for $k + \alpha $ and $s = s_0$, by taking $2 C_1(s_0) \| p \|_{s_0 + 1} \leq 1/ 2$. 
Then the estimate for $s > s_0$, follows by \eqref{stima-a-sinistra}, \eqref{partial omega k alpha q s0}, 
\eqref{mille condizioni p}.
Finally  \eqref{tame-lipschitz-cambio-di-variabile} follows by \eqref{stima cambio di variabile dentro la dim}-\eqref{stima tame cambio di variabile pietro}, \eqref{stime-lipschitz-q}. 
\end{proof}

We finally state the following generalized Moser tame estimates for the 
composition operator
$$
u(\vphi, x) \mapsto {\mathtt f}(u)(\vphi, x) := f(\vphi, x, u(\vphi, x)) 
$$ 
which can be proved arguing as in the previous lemma. 
Since the variables $ (\vphi, x) := y $ have the same role, we present it for a  generic Sobolev space  $ H^s (\T^d ) $.  

\begin{lemma}\label{Moser norme pesate} {\bf (Composition operator)}
Let $ f \in {\mathcal C}^{\infty}(\T^d \times \R, \R )$.  
If   $u(\lambda) \in H^s(\T^d)$  is a family of Sobolev functions
satisfying $\| u \|_{s_0}^{k_0, \gamma} \leq 1$, then, $ \forall  s > s_0 := (d + 1)/2 $,  
$$
  \| {\mathtt f}(u) \|_s \leq C(s,  f ) ( 1 + \| u \|_{s}) \, , \quad 
  \| {\mathtt f}(u) \|_s^{k_0, \gamma} \leq C(s, k_0, f ) ( 1 + \| u \|_{s}^{k_0, \gamma}) \, . 
$$
\end{lemma}

\section{Integral operators and Hilbert transform}\label{sub:integral-op}

We now consider integral operators with 
a $ {\mathcal C}^\infty $ Kernel. 

\begin{lemma}  {\bf (Integral operators)} \label{lem:Int}
Let $K := K( \lambda, \cdot ) \in {\mathcal C}^\infty(\T^\nu \times \T \times \T)$. Then the integral operator 
\be\label{integral operator}
({\mathcal R} u ) ( \vphi, x) := \int_\T K(\lambda, \vphi, x, y) u(\vphi, y)\,d y 
\ee
is in $ OPS^{- \infty}$ and, for all $ m, s,  \a \in \N $,  
\begin{equation}\label{stima pseudo diff in forma di nucleo derivate xi}
\norma {\mathcal R}  \norma_{-m, s, \a}^{k_0, \gamma} \leq C(m, s, \alpha, k_0)  \| K \|_{{\mathcal C}^{s + m + \alpha}}^{k_0, \gamma}\,.
\end{equation}
\end{lemma}

\begin{proof}
By \eqref{definizione simbolo} the symbol associated 
to the integral operator $ {\mathcal R} $ is 
\be\label{simbolo-discreto}
a(\lambda, \vphi, x, j) = \int_\T K(\lambda, \vphi, x, y) e^{\ii (y - x) j}\,dy \, , \quad \forall j \in \Z \, . 
\ee
The function $ a $ is $ {\mathcal C}^{\infty} $ in $(\vphi, x )$  and $ k_0 $-times differentiable with respect to $ \lambda $. 
For all $ m, \beta, p \in \N$, $ n \in \N^\nu\,,\, k \in \N^{\nu + 1} $,  one has 
\begin{align}
& \partial_\lambda^k \partial_\vphi^n\partial_x^p \Delta_j^\b 
a(\lambda, \vphi, x, \xi) (\ii j)^{m + \beta} \nonumber\\
&  =  {\mathop \sum}_{p_1 + p_2 = p} C_{p_1,p_2} 
\Delta_j^\beta \int_\T  (\partial_\lambda^k \partial_\vphi^n \partial_x^{s_1}  K)(\lambda, \vphi, x, y) \partial_y^{p_2 + m + \beta}(e^{\ii (y - x)j})\,dy \nonumber\\
& = {\mathop \sum}_{p_1 + p_2 = p} C_{p_1,p_2, m, \beta} 
 \int_\T  (\partial_\lambda^k \partial_\vphi^n \partial_x^{p_1} \partial_y^{p_2 + m + \beta} K)(\lambda, \vphi, x, y) 
 \Delta_j^\beta(e^{\ii (y - x)j})\,dy
     \nonumber
\end{align}
integrating by parts. Using that  $|\Delta_j^\beta (e^{\ii x j})| = | e^{\ii x \b} (e^{\ii x} -1)^\b) | \leq 2^\beta $, $\forall \beta \in \N$, 
$ x \in \R $,  
and recalling \eqref{norma pesata C^s},  we deduce that, for all $|k| \leq k_0 $,  
\begin{equation}\label{disuguaglianza simbolo discreto nucleo integrale}
|\partial_\lambda^k \partial_\vphi^n\partial_x^p \Delta_j^\b 
a(\lambda, \vphi, x, j) |  \leq C(p, m, \beta)  \gamma^{- |k|}\|  K\|_{{\mathcal C}^{p + m + \beta + |n|}}^{k_0, \gamma} 
\langle j \rangle^{-m - \beta} \,.
\end{equation} 
Now we construct an extension $\widetilde a(\lambda, \vphi, x, \xi )$ of the symbol $a(\lambda, \vphi,x, j)$ as in \eqref{simbolo esteso discreto continuo}, namely we define 
\be\label{extension-sy}
\widetilde a(\lambda, \vphi, x, \xi) :={\mathop \sum}_{j \in \Z} a(\lambda, \vphi,x, j) \zeta (\xi - j) \, ,  \quad \forall \xi \in \R \, . 
\ee
Since $\widetilde a(\cdot, j) = a(\cdot, j)$ for all $j \in \Z$ one has that $ {\rm Op}(\widetilde a) = {\rm Op}(a) = {\mathcal R} $. 
By \eqref{simbolo esteso discreto continuo stima} and \eqref{disuguaglianza simbolo discreto nucleo integrale} 
 it results that for all $ m, \b, p \in \N $, $ n \in \N^\nu$, $k \in \N^{\nu + 1}$ with $|k| \leq k_0 $, there exist constants $ C'(p, m, \beta) > 0 $ such that 
\begin{equation}\label{simbolo esteso nucleo integrale stima}
|\partial_\lambda^k \partial_\vphi^n\partial_x^p \partial_\xi^\b 
\widetilde a(\lambda, \vphi, x, \xi) |  \leq C'(p, m, \beta)  \gamma^{- |k|}\|  K\|_{{\mathcal C}^{p + m + \beta + |n|}}^{k_0, \gamma}
 \langle \xi \rangle^{-m - \beta} \,.
\end{equation}
By \eqref{Sobolev norm} and \eqref{simbolo esteso nucleo integrale stima} we get: for all $ m, s,  \b \in \N $,  $|k| \leq k_0 $,  
\begin{align}
\|   \partial_\xi^\beta  \partial_\lambda^k \widetilde a (\lambda, \cdot, \xi)\|_s \langle \xi \rangle^{m + \beta} & 
\simeq \Big(\|  \partial_\xi^\beta  \partial_\lambda^k \widetilde a (\lambda, \cdot, \xi)\|_{L^2_\vphi L^2_x} + 
\| \partial_x^s  \partial_\xi^\beta  \partial_\lambda^k \widetilde a (\lambda, \cdot, \xi)\|_{L^2_\vphi L^2_x}  
\nonumber\\
& \quad + \sup_{n \in \Z^\nu, |n| = s}\| \partial_\vphi^n  \partial_\xi^\beta  \partial_\lambda^k
\widetilde a (\lambda, \cdot, \xi)\|_{L^2_\vphi L^2_x} \Big) \langle \xi \rangle^{m + \beta} \nonumber  \\
& \leq_{m, s, \beta} \gamma^{- |k|} \| K \|_{{\mathcal C}^{s + m + \beta}}^{k_0, \gamma} \nonumber 
\end{align}
that,  recalling \eqref{norm1} and \eqref{norm1 parameter}, proves  \eqref{stima pseudo diff in forma di nucleo derivate xi}. 
\end{proof}

\begin{remark}
The extended symbol $ \widetilde a $ in \eqref{extension-sy}
can be  explicitly written, using  \eqref{simbolo-discreto} and the Poisson summation formula, 
 as 
 $$ \widetilde a(\lambda, \vphi, x, \xi)  = \int_{\R} K(\lambda, \vphi, x, y) \theta (y) e^{\ii \xi y} dy 
 $$  
where the test function $ \theta \in {\mathcal D}(\R) $ is defined after \eqref{simbolo esteso discreto continuo}. 
 This expression can be used as well to prove the estimate \eqref{stima pseudo diff in forma di nucleo derivate xi}. 
\end{remark}

An integral operator transforms into another integral operator under a changes of variables
\be\label{change of variable A}
P u(\vphi, x) := u(\vphi, x + p(\vphi, x)) \, . 
\ee
 
\begin{lemma}\label{lemma cio}
Let $K(\lambda, \cdot ) \in {\mathcal C}^\infty(\T^\nu \times \T \times \T)$ 
and $p(\lambda, \cdot ) \in {\mathcal C}^\infty(\T^\nu \times \T, \R) $. 
There exists $\delta := \delta(s_0, k_0) > 0$ such that if $\| p\|_{2 s_0 + k_0 + 1}^{k_0, \gamma} \leq \delta$, then 
the integral operator $ {\mathcal R}$ as in \eqref{integral operator} transforms into the integral operator
\be\label{Kernel trasformato}
\big( P^{- 1} {\mathcal R} P \big) u(\vphi,  x) = \int_\T \tilde K(\lambda, \vphi, x, y) u(\vphi, y)\,dy 
\ee
with a  $ {\mathcal C}^\infty $ 
Kernel $\widetilde K(\lambda, \cdot, \cdot, \cdot)$ which satisfies 
\be\label{tildeK}
\| \tilde K\|_{s}^{k_0, \gamma} \leq C(s, k_0)  \big( \| K \|_{s + k_0}^{k_0, \gamma} 
+ \| p \|_{{s + k_0 + 1}}^{k_0, \gamma} \| K \|_{s_0 + k_0 + 1}^{k_0, \gamma}\big) \qquad \forall s \geq s_0\,.
\ee
\end{lemma}

\begin{proof}
We denote by $ z \mapsto z + q(\lambda, \vphi, z) $ the inverse diffeomorphism of 
$  x \mapsto x + p(\lambda, \vphi, x) $, for all $ \vphi \in \T^\nu $, $ \lambda \in \mathtt \Lambda_0 $.
We have 
$ ( {\mathcal R} P ) u(\vphi, x) = \int_\T K(\lambda, \vphi, x, y) u(\vphi, y + p(\lambda, \vphi, y))\, d y $
and making the change of variable $z = y + p(\lambda, \vphi, y)$ we get \eqref{Kernel trasformato} with  Kernel
$$ 
 \tilde K(\lambda, \vphi, x, z) := \big( 1 + \partial_z q (\lambda, \vphi, z) \big) K(\lambda, \vphi, x + q(\lambda, \vphi, x), 
 z + q(\lambda, \vphi, z)) \, . 
 $$ 
 Since $p \in {\mathcal C}^\infty$, by Lemma \ref{lemma:utile} also $q \in {\mathcal C}^\infty$, therefore $ \tilde K$ is ${\mathcal C}^\infty$.
The estimate \eqref{tildeK} for $ \tilde K $ then follows by \eqref{interpolazione C k0}, \eqref{stima tame cambio di variabile pietro s0}, \eqref{stima tame cambio di variabile pietro}, \eqref{stime-lipschitz-q}
and by Lemma \ref{lemma operatore e funzioni dipendenti da parametro}.
\end{proof}

We now study the properties of the Hilbert transform
$ \mH $. It can be defined through Fourier series by
\be\label{Hilbert T}
\begin{aligned}
\mH \cos(jx)  & := \sign(j) \sin(jx), \quad \forall j \in \Z \setminus \{ 0 \}\,,  \\
\mH \sin(jx) & := - \sign(j) \cos(jx) \, , \ \  
\forall j \in \Z \setminus \{0\} ,  \\
& \quad \mH (1) := 0  \, , 
\end{aligned}
\ee
or in exponential basis 
\be\label{Hilbert transform: complex}
\mH e^{\ii jx} := - \ii \sign(j) e^{\ii jx} \, , \    \forall j \neq 0 \, ,  \quad \mH (1) := 0  \, .
\ee
The Hilbert transform admits also an  integral representation. Given a   $ 2\pi $-periodic function $u$ its Hilbert transform is  
\be\label{HT}
\begin{aligned}
\mH u(x) 
& := \frac{1}{2\pi} \, {\rm  p. v. }  \int \frac{u(y)}{\tan (\frac12 \,(x-y)) }\,dy \\
& := \lim_{\e \to 0} \frac{1}{2\pi} \Big\{ \int_{x-\pi}^{x-\e} + \int_{x+\e}^{x+\pi} \Big\} \frac{u(y)}{\tan (\frac12 \,(x-y)) }\,dy .
\end{aligned}
\ee
The commutator between the Hilbert transform $ \mH $ and the multiplication operator for a smooth function $ a $ is a 
regularizing operator in $ OPS^{-\infty } $.  

\begin{lemma}\label{lem: commutator aH}
Let $ a( \lambda, \cdot, \cdot ) \in {\mathcal C}^{\infty} (\T^\nu \times \T, \R)$. Then the commutator $[a, \mH ] \in OPS^{-\infty }$ and, for all $ m,  s,  \a \in \N $, 
\be\label{norma commutator}
\norma  [a, \mH  ] \norma_{-m, s, \a}^{k_0, \gamma} \leq C(m, s, \alpha, k_0) \| a \|_{{s + s_0 + 1+ m + \alpha}}^{k_0, \gamma} \, . 
\ee
\end{lemma}

\begin{proof}
By \eqref{HT} the commutator 
$$
(\mH a - a \mH) u =    \frac{1}{2\pi} {\rm  p. v. } \int \frac{ (a(y) - a(x))u(y) }{\tan (\frac12 \,(x-y)) }\,dy =
\frac{1}{2\pi}  \int_{\T} K(x,y) u(y) \, dy 
$$
is an integral operator with $ {\mathcal C}^\infty $ Kernel (note that the integral is no longer a principal value)
$$
\begin{aligned}
K(\lambda, \vphi, x,y) & := \frac{ a(\lambda, \vphi, y) - a(\lambda, \vphi, x) }{\tan ( (x-y) / 2) } \\
& = 
\Big( \int_0^1 a_x ( \lambda, \vphi, x + t (y-x)) dt \Big) \frac{y-x}{\tan ( (x-y) / 2)}\, .
\end{aligned} 
$$
Then  \eqref{norma commutator} follows by Lemma \ref{lem:Int} and the bound 
$ \| K \|_{{\mathcal C}^s}^{k_0, \g} \leq_s \| K\|_{s + s_0}^{k_0, \gamma} \leq_s  \| a \|_{s + s_0 + 1}^{k_0, \gamma}$ for all $s \geq 0 $. 
\end{proof}

We now conjugate the Hilbert transform 
by a family of changes of variables  as in \eqref{change of variable A}, see also
the Appendices H and I in \cite{IPT} and 
\cite{Baldi-Benj-Ono}-Lemma B.5. 

\begin{lemma} \label{coniugio Hilbert}
Let $p = p(\lambda, \cdot) \in {\mathcal C}^\infty(\T^{\nu + 1})$. There exists $\delta(s_0, k_0) > 0$ such that, if $\|p \|_{2 s_0 + k_0 + 1}^{k_0, \gamma} \leq \delta(s_0, k_0)$, then 
the operator $P^{- 1} \mH P - \mH$ is an integral operator of the form
\be \label{int repres}
( P^{-1} \mH P - \mH) u (\vphi, x) 
= \int_\T  \,K(\lambda, \vphi, x,z) u(\ph, z)\,dz
\ee
where $ K = K(\lambda, \cdot) \in {\mathcal C}^\infty (\T^\nu \times \T \times \T ) $ satisfies 
\be\label{stima kernel}
\| K \|_{s}^{k_0, \gamma} \leq C(s, k_0) \| p \|_{s +  k_0 + 2}^{k_0, \gamma} \,, \quad \forall s \geq s_0\,.
\ee
\end{lemma}

\begin{proof}
The inverse diffeomorphism of 
$ x \mapsto x + p(\vphi, x) $ has the form 
$ z \mapsto  z + q(\vphi, z) $.
Changing the variable $z =y + p(\vphi, y)$ in the integral \eqref{HT} 
gives 
$$
P^{-1} \mH P u(\vphi, x) = 
 \frac{1}{2\p} \, {\rm p. v.} \int \frac{u(\vphi, z) (1 + \partial_z q(\lambda, \vphi, z))}
{\tan(\frac12\,[ x - z + q(\lambda, \vphi, x) - q(\lambda, \vphi, z)]) } \,dz \, . 
$$
 As a consequence we get \eqref{int repres}  (which is no longer a principal value) with  Kernel 
\begin{align} \label{kernel DN}
K(\lambda, \vphi, x, z) & := \frac{1}{2\p} \Big( 
\frac{1 + \partial_z q(\lambda,  \vphi, z)}{\tan(\frac12[x - z + q(\lambda,  \vphi, x) - q (\lambda,  \vphi, z)])} \, 
- \frac{1}{\tan(\frac12[x - z])} \Big) \nonumber \\ 
& = -\frac{1}{\p}\, \partial_z \log \Big( 
\frac{\sin(\frac12[x - z + q(\lambda,  \vphi, x) - q(\lambda,  \vphi, z)])}
{\sin(\frac12[x-z])} \Big)  \nonumber \\
& =  - \frac{1}{\pi} \partial_z \log \big( 1 + g(\lambda,  \vphi, x, z) \big) 
\end{align}
(note that $ q $ is small) 
where the family of $ {\mathcal C}^\infty$  functions  
$$ 
\begin{aligned}
g(\lambda, \vphi, x, z)  & := \cos \Big( \frac{q(\lambda, \vphi, x) - q(\lambda, \vphi, z)}{ 2}  \Big) - 1 \\ 
& \ \, + 
\cos \Big( \frac{x-z}{2}  \Big) 
\frac{\sin(\frac12[ q(\lambda,  \vphi, x) - q(\lambda,  \vphi, z)])}
{\sin(\frac12[x-z])}
\end{aligned}
$$ 
satisfies 
the estimate 
$ \| g \|_{s}^{k_0, \gamma} \leq_{s, k_0} \|  q \|_{s + 1}^{k_0, \gamma} \leq_{s, k_0} \| p \|_{s +  k_0 + 1}^{k_0, \gamma} $
 using \eqref{stime-lipschitz-q}.
Lemma \ref{Moser norme pesate} implies \eqref{stima kernel}.
\end{proof}

\section{Dirichlet-Neumann operator}\label{subDN}

We now present some fundamental properties of the Dirichlet-Neumann operator $ G $ defined in \eqref{D-N} that are used in the paper.
There is a huge literature about it for which we refer to the recent work of Alazard-Delort \cite{AlDe}-\cite{AlDe1}
and the book of Lannes \cite{LannesLivre}, and references therein.
We remark that for our purposes it is sufficient to work in the class of smooth $ {\mathcal C}^\infty $ profiles $ \eta (x) $ 
because at each step of the  Nash-Moser iteration we perform a 
$ {\mathcal C}^\infty $-regularization. 

\smallskip

The mapping $(\eta,\psi)\rightarrow G(\eta)\psi$ is linear  with respect to $\psi$ and nonlinear with respect to $\eta$. 
The derivative with respect to $\eta$ (``shape derivative'') is given by 
(see e.g. \cite{LannesLivre})
\begin{equation} \label{formula shape der}
G'(\eta) [\hat {\eta} ] \psi 
= \lim_{\eps \rightarrow 0} \frac{1}{\e} \{ G (\eta+\eps\hat \eta ) \psi - G(\eta) \psi \}
= - G(\eta) (B \hat \eta ) -\partial_x (V \hat \eta )
\end{equation}
where 
\begin{equation} \label{def B V}
\B := \B(\eta,\psi) := \frac{\eta_x \psi_x + G(\eta)\psi }{ 1 + \eta_x^2 }\,, 
\qquad 
V := V(\eta,\psi) := \psi_x - B \eta_x.
\end{equation}
The vector $ (V,B) = \nabla_{x,y} \Phi $ is the velocity field  evaluated at the free surface  $ (x,  \eta (x)) $. 

Note also  that $G(\eta)$ is an even operator according to Definition \ref{def:even}.  

\smallskip

The Dirichlet-Neumann operator is a {\it pseudo-differential} operator of the form 
\be\label{sviluppo Geta}
G(\eta) = |D| + \mR_{G}(\eta) 
\ee
where $ G(0) = |D| $ and the  remainder $ \mR_{G}(\eta) \in OPS^{-\infty} $. 
The explicit representation of  
the integral Kernel of $ \mR_{G}(\eta) $ given by \eqref{Geta intermedia}, \eqref{int repres}, \eqref{kernel DN}, 
has been taught to us by Baldi \cite{Baldi}.
We use it  to estimate the pseudo-differential norm $ \norma \mR_{G}(\eta) \norma^{k_0, \gamma}_{-m, s,\a} $. 
Note that the free profile $ \eta (x) := \eta (\omega, \kappa, \vphi, x) $ as well as the potential
$ \psi (\omega, \kappa, \vphi, x ) $
may depend also on the angles $ \vphi \in \T^\nu$ and the parameters 
$ \lambda := ( \om, \kappa ) \in  \R^\nu \times [\kappa_1, \kappa_2 ] $. 
For simplicity of notation we sometimes omit to write the dependence on $\vphi$, $\omega$, $\kappa$.

\begin{proposition}\label{Prop DN}
Assume that $\partial_\lambda^k \eta(\lambda, \cdot, \cdot)$ is ${\mathcal C}^\infty$ for all $|k| \leq k_0$. There exists $\delta := \delta(s_0, k_0) > 0$ such that, if 
\begin{equation}\label{condizione di piccolezza lemma dirichlet neumann}
\| \eta\|_{2 s_0 + 2 k_0 + 1}^{k_0, \gamma} \leq \delta \, , 
\end{equation} 
then the Dirichlet-Neumann operator $ G(\eta) $ may be written as in \eqref{sviluppo Geta} 
where $ {\mathcal R}_G (\eta ) $ is an integral operator with  $ {\mathcal C}^\infty $ 
Kernel  $ K_G $ 
(see \eqref{integral operator}) 
which satisfies, for all $m, s, \alpha \in  \N$, the estimate 
\be\label{estimate DN}
\begin{aligned}
\norma {\mathcal R}_G (\eta) 
\norma_{-m, s, \a}^{k_0, \gamma} & 
\leq C(s,m,\a, k_0) \| K_G \|_{{\mathcal C}^{s+m + \alpha}}^{k_0,\gamma} \\ 
& \leq
 C(s, m, \a, k_0 ) \| \eta \|_{s + s_0 + 2 k_0 + m + \alpha + 3 }^{k_0, \gamma} \, .
 \end{aligned}
\ee
Let $s_1 \geq 2 s_0 + 1$. There exists $\delta(s_1) > 0$  such that,  
the map $ \{ \| \eta \|_{s_1 + 6} < \delta(s_1)  \}  \to H^{s_1}(\T^\nu \times \T \times \T) $, $\eta \mapsto K_G ( \eta ) $,  is ${\mathcal C}^1$.  
\end{proposition}

\begin{remark}
Note that the assumption \eqref{condizione di piccolezza lemma dirichlet neumann} in low norm 
$ \| \ \|_{2 s_0 + 2 k_0 + 1}^{k_0, \gamma} $ implies
the estimate \eqref{estimate DN} for any $ s \in \N $. 
The estimate
$\|\partial_\eta K_G [\widehat \eta] \|_{s_1} \leq_{s_1} \| \widehat \eta\|_{s_1 + 6}$
is used in section \ref{linearizzato siti normali} 
(in particular in section \ref{sec:changes}) with a Sobolev index $ s_1 $ which 
has to be considered fixed, see \eqref{vincolo s1 derivate i}. 
A sharper tame version of this estimate could be proved, but it is not needed.
Note also that it does not  involve the $ \| \ \|_{s_1}^{k_0, \gamma} $ norm. 
\end{remark}

The rest of this section is devoted to the proof of Proposition \ref{Prop DN}.

\noindent
In order to analyze the Dirichlet-Neumann  operator it is convenient to  transform the boundary value problem   
\eqref{BVP} defined in the free domain $ \{(x,y):y<\eta(x)\}$
into an elliptic problem in the lower half-plane $ \Sigma_0 := \{(X,Y) : Y<0\} $ via
a conformal diffeomorphism  
\be\label{conf-diffeo}
x = U(X,Y), \quad y = V(X,Y) \, . 
\ee
The following conformal transformation \eqref{defUV}, the  formulation of the  problem as the fixed point 
 equation \eqref{eq for p}, 
Lemma \ref{G=mH} and  \eqref{Geta intermedia} is due to Baldi \cite{Baldi}.
\\[1mm]
{\sc The conformal transformation.} Let $ p : \R \to \R $ be a smooth $ 2 \pi $-periodic  function with zero average
and  $ \| \pa_X^2 p \|_{L^2(\T)} \leq c_0 := 1/ (2 \sqrt{2\pi}) $. We  define  the functions   
\be\label{defUV}
\begin{aligned}
U( X, Y ) & := X + \sum_{k\neq0} p_k \, e^{|k|Y}\,e^{\ii k X},  \\
V( X, Y ) & := Y + \sum_{k\neq0} \ii \, \sign(k)\, p_k  \, e^{|k|Y}\,e^{\ii k X} + c 
\end{aligned}
\ee
with $ c  \in \R $.  The functions $ U $ and $ V $ are both harmonic on  $ \Sigma_0 $
and satisfy the Cauchy-Riemann equations
$ U_X = V_Y $, $ U_Y = - V_X $ 
so that  $ U + \ii V $ is holomorphic on $ \Sigma_0 $.
The gradient  $(U_X,U_Y) \to (1,0) $  as $ Y  \to -\infty$.  

Since, $  \forall Y \leq 0 $, $ \| U_{XX}(X,Y) \|_{L^2(\T)} \leq \|p_{XX} \|_{L^2(\T)} \leq c_0 $, 
it results  $ U_X \geq 1/2 $  on $ \Sigma_0 $, and, 
by $ V_Y = U_X \geq 1/2 $, we also get $ V(X,Y) < V(X,0) $ for $ Y < 0 $. 
The Jacobian 
$$
{\rm det} 
\left(
\begin{array}{cc}
U_X  & U_Y   \\
 V_X & V_Y  
\end{array}
\right) = {\rm det} 
\left(
\begin{array}{cc}
U_X  & U_Y   \\
- U_Y & U_X  
\end{array}
\right) = U_X^2 + U_Y^2 \geq \frac14 \, , \quad \forall (X,Y) \in \Sigma_0 \, ,
$$
so that $ U + \ii V $ is a global diffeomorphism from $ \Sigma_0 $ onto its image. Since $ U(X,Y)- X $ is 
$ 2 \pi $-periodic in $ X $ (see \eqref{defUV}) the map $ U + \ii V $
is the lift of a diffeomorphism from $ \T \times (-\infty,0 ] $ onto its image.  
The image of the map $ U + \ii V $ is the subset of $ \C \simeq \R^2 $ that is below the profile described parametrically by
\begin{equation}  \label{profile}
( U( X,0),V(X,0)) = (X+p(X), - \mH p(X) + c ) 
\end{equation}
where $ \mH $ is the Hilbert transform  in \eqref{Hilbert transform: complex}.
The profile \eqref{profile} coincides with the graph $ Y = \eta (X) $ if
\begin{equation}  \label{eq for p 0}
- \mH p(X) + c = \eta(X + p(X)) \, , \quad \forall X \in \R \, . 
\end{equation}
Since, by \eqref{Hilbert transform: complex},  the range of the Hilbert transform $ \mH $ is the 
space of functions with zero average and $ \mH^2=- \Pi $ 
where $ \Pi [f] := f - f_0 $,  
the equation \eqref{eq for p 0} is equivalent to 
$$
c  = \frac{1}{2\p} \intp \eta( X+p( X))\,d X 
$$
and 
\begin{equation}  \label{eq for p}
p( X ) =  \mH [\eta( X + p( X ))] \, . 
\end{equation}

\begin{lemma}\label{DN stima p eta}
Let $ \eta$ satisfy $\partial_\lambda^k \eta(\lambda, \cdot) \in {\mathcal C}^\infty(\T^{\nu + 1})$, for all $|k| \leq k_0$. 
There exists $\delta := \delta(s_0, k_0) > 0$, such that, if $\| \eta\|_{2 s_0 + k_0 + 1}^{k_0, \gamma} \leq \delta$,  then there exists a unique solution 
$ p = p (\lambda, \cdot ) $ of  \eqref{eq for p} satisfying the estimates
\begin{align}\label{stima p da eta}
& \|  p\|_s \leq_s \| \eta\|_s\,, \qquad \| p\|_{s}^{k_0, \gamma} \leq_s \| \eta \|_{s + k_0 }^{k_0, \gamma} \, , \qquad \forall s \geq s_0\,.
\end{align}
Let $ s_1 \geq 2 s_0 + 1 $. 
There exists $\delta(s_1) > 0$  such that 
the map $ \{ \| \eta \|_{s_1 + 2} < \delta(s_1)  \}  \to H^{s_1}$, $\eta \mapsto p(\eta)$,   is ${\mathcal C}^1$.
\end{lemma}

\begin{proof}
We find a solution of  \eqref{eq for p} as a  fixed point of the map
$$
p (\vphi, X) \mapsto \Phi ( p )(\vphi, X):= \mH [\eta(\vphi, X + p(\vphi, X))]\, . 
$$
For any $ n \in \N $, we consider the finite dimensional subspace 
$ E_n := {\rm span }\{ e^{\ii (\ell \cdot \vphi + j x)}: |(\ell, j)| \leq n \} $ 
and 
the regularized map 
$ \Phi_n := \Pi_n \Phi : E_n \to E_n  $ 
where $ \Pi_n $ denotes the $ L^2 $-orthogonal projector on $ E_n $.
We show that there is $ r >  0 $ small, such that, for any $n \in \N$, the map
$$
\Phi_n : {\mathcal B}_{2 s_0 + 1}( r ) \cap E_n \to {\mathcal B}_{2 s_0 + 1}( r) \cap E_n \,, 
\  {\mathcal B}_{2 s_0 + 1}( r ) := \big\{ p \in H^{2s_0+1} : \| p \|_{2 s_0 + 1} \leq r \big\}\,,
$$
is a contraction. We fix $ r >  0 $ such that 
$ \| p \|_{{\mathcal C}^{s_0 + 1}} \leq C(s_0)\|p \|_{2 s_0 + 1} \leq 1/2$, for all $  p \in {\mathcal B}_{2 s_0 + 1}( r ) $, i.e 
$ r := 1/ (2 C(s_0))$, so that  the hyphothesis \eqref{mille condizioni p} of Lemma \ref{lemma:utile} is fulfilled.
Then, using that $ {\mathcal H}$ is an isometry on the Sobolev spaces $ H^s $  (see \eqref{Hilbert transform: complex}), 
that $\| \Pi_n h\|_s \leq \| h \|_s$,  
and applying  \eqref{stima cambio di variabile dentro la dim}, we get  
$$
\| \Phi_n ( p )\|_{2 s_0 + 1} \leq \| \eta(\cdot + p(\cdot)) \|_{2 s_0 + 1} \leq C_1(s_0)  \| \eta\|_{2 s_0 + 1} \leq r 
$$
taking $\| \eta\|_{2 s_0 + 1} \leq r / C_1(s_0) $. Moreover for any $p_1\,,\,p_2 \in {\mathcal B}_{2 s_0 + 1}( r ) \cap E_n $, we have  
$$
\| \Phi_n(p_1) - \Phi_n(p_2) \|_{2 s_0 + 1} \leq C(s_0) \| \eta\|_{2 s_0 + 2} \| p_1 - p_2\|_{2 s_0 + 1} \leq 
\| p_1 - p_2\|_{2 s_0 + 1} / 2 \,,
$$
by taking $ C(s_0) \| \eta\|_{2 s_0 + 2} \leq 1/2$. 
Then, by the contraction mapping theorem there exists a unique fixed point solution $ p_n \in {\mathcal B}_{2 s_0 + 1}( r ) \cap E_n $ 
solving $ \Phi_n( p_n) = p_n $. Note that $p_n \in E_n \subset {\mathcal C}^\infty(\T^{\nu + 1})$.  
Using again that the Hilbert transform is a unitary operator, and 
the estimate \eqref{stima cambio di variabile dentro la dim}, we get, for all $s \geq s_0$ 
\be
\begin{aligned}
\| p_n \|_s & = \| \Phi_n( p_n )\|_s \\ 
& =  \|\Pi_n \mH \eta(\cdot + p_n(\cdot))  \| \leq C(s ) \| \eta\|_s + C(s_0 ) \| p_n \|_s \| \eta\|_{s_0 + 1} 
\end{aligned}
\ee
which implies $\| p_n \|_s \leq 2 C(s) \| \eta\|_s$ taking $C(s_0 ) \| \eta\|_{s_0 + 1} \leq 1/2 $. 
Since $H^s \hookrightarrow H^{s - 1}$ compactly, for any $s \geq s_0$, the sequence $ p_n $ converges
 strongly in $ H^s $ (up to subsequence) 
 to a function $p \in {\mathcal C}^\infty(\T^{\nu + 1})$ which satisfies $\| p \|_s \leq 2 C(s) \| \eta\|_s $ for any $s \geq s_0$. 
 The function $p$ solves the equation \eqref{eq for p} because 
\begin{align}
\| \Phi( p ) - \Phi_n ( p_n )\|_{s_0} & \leq \|\Pi_{n} \mH \eta(\cdot + p (\cdot)) - \Pi_{n} \mH \eta(\cdot + p_{n}(\cdot)) \|_{s_0} 
\nonumber
\\
& \ + \| ({\rm Id} - \Pi_{n})\mH \eta(\cdot + p(\cdot))\|_{s_0} \nonumber\\
& \leq_{s_0} \| \eta\|_{s_0 + 1} \| p - p_{n} \|_{s_0} + \frac{1}{n} \| \eta\|_{s_0 + 1}(1 + \| p \|_{s_0 + 1}) \nonumber \\
& \leq_{s_0} \| p - p_{n} \|_{s_0} + \frac{1}{n} \to 0 \nonumber
\end{align} 
as $ n  \to + \infty$. This implies that $\Phi( p ) = p$. Arguing as in Lemma \ref{lemma:utile} one can prove that if $\partial_\lambda^k \eta(\lambda, \cdot) \in {\mathcal C}^\infty$ for all $|k| \leq k_0$, then also $\partial_\lambda^k p(\lambda, \cdot) \in {\mathcal C}^\infty(\T^{\nu + 1})$, for all $|k| \leq k_0$. The second 
estimate in \eqref{stima p da eta} can be proved as the estimate \eqref{stime-lipschitz-q} in Lemma \ref{lemma:utile}, using the condition $\| \eta\|_{s_0 + k_0 + 1}^{k_0, \gamma} \leq \delta(s_0, k_0)$ for some $\delta(s_0, k_0) > 0$ small enough. 

The differentiability of $ \eta \mapsto p(\eta ) $ 
follows by the implicit function theorem using the $ {\mathcal C}^1$  map  
$$
F :  H^{s_1 + 2}  \times  H^{s_1} \to H^{s_1}\,, \quad F(\eta, p)(\vphi, X) := p(\vphi, X) - {\mathcal H}\big[ \eta(\vphi, X + p(\vphi, X)) \big] \, .
$$
Since $ F(0, 0 ) = 0 $ and $\partial_p F(0, 0) = {\rm Id} $,  
by the implicit function theorem there exists 
$ \delta(s_1) >  0  $ and 
a ${\mathcal C}^1 $ map  $ \{  \| \eta\|_{s_1 + 2} \leq \delta(s_1)  \} \ni 
\eta \mapsto p(\eta) \in H^{s_1} $, such that $ F(\eta, p(\eta)) = 0 $. 
\end{proof}

We transform \eqref{BVP} via the conformal diffeomorphism \eqref{defUV}.  Denote  
$$ 
P u (X) := u ( X + p(X)) \, . 
$$ 
The potential $ \phi( X,Y) := \Phi(U( X,Y),V( X,Y)) $  satisfies, using also \eqref{profile}-\eqref{eq for p 0}, 
\be\label{BVP-new}
\D \phi = 0   \ \text{in} \ \{Y<0\}\, ,  \quad  
\phi( X,0) = (P \psi)(X) \, , \quad 
\gr \phi \to (0,0)  \  \text{as}\ Y \to -\infty \, .
\ee
Recall that the Dirichlet-Neumann operator at the flat surface $ Y = 0 $ is $   \pa_X \mH $. 

\begin{lemma}\label{G=mH}
$ G(\eta) = \partial_x P^{-1} \mH P  $. 
\end{lemma}

\begin{proof}
Since $\eta( U( X,0))=V( X,0)$ (see \eqref{eq for p 0}) we derive $ - U_Y = V_X = \eta_x U_X $ on $ Y = 0 $. 
Moreover, by 
$$
\Phi_x = \frac{\phi_X U_X + \phi_Y U_Y}{U_X^2+U_X^2}\,,
\quad 
\Phi_y = \frac{\phi_Y U_X - \phi_X U_Y}{U_X^2+U_X^2}\,,
$$
and the definition  \eqref{D-N} of the Dirichlet-Neumann operator we get
\begin{align*}
G(\eta)\psi( x) & = \frac{1}{U_X^2+U_Y^2}\, 
\Big( \phi_X (-U_Y -\eta_x U_X) + \phi_Y (U_X - \eta_x U_Y ) \Big) \\ 
& = \frac{1}{U_X( X,0)}\,\phi_Y(X,0)   \\
& 
\stackrel{\eqref{defUV}, \eqref{BVP-new}} = \frac{1}{1+p_X( X)}\,\partial_X \mH (P \psi) (X)  
 = \Big\{ \frac{1}{1+p_X}\,\partial_X \mH P\psi \Big\}(x + \tilde p( x)) 
\end{align*}
where $ X = x + \tilde p( x) $ is the inverse diffeomorphism of $ x = X + p( X )$. 
In operatorial notation we have 
\begin{align} 
G(\eta) 
& = P^{-1} \frac{1}{1+p_X}\,\partial_X \mH P \notag  = \frac{1}{1+P^{-1} p_X}\, P^{-1} \partial_X P \, 
P^{-1} \mH P  \notag \\
& = \frac{1}{1+P^{-1} p_X}\, (1+P^{-1} p_X) \, \partial_x P^{-1} \mH P   = \partial_x P^{-1} \mH P 
\notag
\end{align}
by the  rule $P^{-1} \partial_X P = (1+P^{-1} p_X) \, \partial_x $ for the changes of coordinates.
\end{proof}

Lemma \ref{G=mH} provides the representation \eqref{sviluppo Geta} of the Dirichlet-Neumann operator with 
\be\label{Geta intermedia}
{\mathcal R}_G (\eta) := \partial_x \big( P^{-1} \mH P - \mH \big) \, .
\ee
By Lemma \ref{coniugio Hilbert}, in particular  by formula \eqref{kernel DN}, 
the operator $ {\mathcal R}_G (\eta) $ is an integral operator with kernel 
\begin{equation}\label{palla di lardo 0}
K_G  := K_G(\eta ) := - \frac{1}{\pi}\partial_{x z} \log \big( 1 + g ( \vphi, x, z) \big) 
\end{equation}
where 
\begin{equation}\label{palla di lardo 1}
\begin{aligned}
g( \vphi, x, z)  & := \cos \Big( \frac{q(\lambda, \vphi, x) - q( \lambda, \vphi, z)}{ 2}  \Big) - 1 \\ 
& \  + \cos \Big( \frac{x-z}{2}  \Big) 
\frac{\sin(\frac12[ q(  \lambda, \vphi, x) - q( \lambda, \vphi, z)])}
{\sin(\frac12[x-z])}
\end{aligned}
\end{equation}
and  $x \mapsto x + q(\vphi, x)$ is the inverse diffeomorphism of $X \mapsto X + p(\vphi, X) $ (the functions  
$ p, q $  depend on $ \eta $).
\\[1mm]
{\sc Proof of Proposition \ref{Prop DN} concluded.}
By \eqref{condizione di piccolezza lemma dirichlet neumann} we  apply Lemma \ref{DN stima p eta} and then 
\eqref{stima p da eta} implies 
$ \| p\|_{2 s_0 + k_0 + 1}^{k_0, \gamma} \leq_{s_0} \| \eta\|_{2 s_0 + 2 k_0 + 1}^{k_0, \gamma} $. 
Hence, by \eqref{condizione di piccolezza lemma dirichlet neumann},  the smallness assumption of Lemma 
\ref{coniugio Hilbert} is verified. 
Hence the estimate \eqref{estimate DN} follows by \eqref{stima pseudo diff in forma di nucleo derivate xi}, \eqref{stima kernel}, \eqref{stima p da eta}. 

We now prove that the function $\{ \| \eta \|_{s_1 + 6} \leq \delta(s_1)\} \mapsto H^{s_1}(\T^\nu \times \T \times \T)$, $\eta \mapsto K_G(\eta)$ is ${\mathcal C}^1$. Indeed, by applying Lemma \ref{DN stima p eta} (with $s_1 + 4$ instead of $s_1$), 
the map $\{ \| \eta \|_{s_1 + 6} < \delta(s_1)\} \mapsto H^{s_1 + 4}$, 
$\eta \mapsto p(\eta)$ is ${\mathcal C}^1$. Then, since $ q(\vphi, x) = - p (\vphi, x + q(\vphi, x) ) $, 
by the implicit function theorem, 
for $ p $ small  in $\| \cdot \|_{s_1 + 4}$-norm, also the map $ p \mapsto q ( p ) \in H^{s_1 + 2} $ is $ {\mathcal C}^1$. By composition, the claim follows by recalling \eqref{palla di lardo 0}, \eqref{palla di lardo 1}.
\hfill $\Box$

\smallskip

To conclude we provide the following tame estimates for the Dirichlet Neumann operator: 
\begin{lemma}\label{stime tame dirichlet neumann}
There is $\delta(s_0, k_0) > 0$ such that, if 
 $\| \eta\|_{2s_0 + 2 k_0 + 5}^{k_0, \gamma} \leq \delta(s_0, k_0)$,  then, for all $ s \geq s_0 $ 
 \begin{equation}\label{stima tame dirichlet neumann}
 \begin{aligned}
  \| \big(G(\eta) - |D| \big) \psi \|_s^{k_0, \gamma} & \leq_{s, k_0}  \| \eta \|^{k_0, \gamma}_{s + s_0 + 2 k_0 + 3} \| \psi\|_{s_0}^{k_0, \gamma}   \\
  & \quad \quad + \| \eta \|^{k_0, \gamma}_{2 s_0 + 2 k_0 + 3} \| \psi\|_{s}^{k_0, \gamma}\,, 
 \end{aligned}
 \end{equation}
 \begin{equation}\label{stima tame derivata dirichlet neumann}
 \begin{aligned}
  \|  G'(\eta)[\widehat \eta] \psi \|_s^{k_0, \gamma}  & \leq_{s, k_0} \| \psi\|_{s + 2}^{k_0, \gamma} \| \widehat \eta\|_{s_0 + 1}^{k_0, \gamma} + \| \psi\|_{s_0 + 2}^{k_0, \gamma} \| \widehat \eta\|_{s + 1}^{k_0, \gamma} \\
  & \quad  \quad + \| \eta\|_{s + s_0 + 2 k_0 + 4}^{k_0, \gamma} \| \widehat \eta\|_{s_0 + 1}^{k_0, \gamma} \| \psi\|_{s_0 + 2}^{k_0, \gamma}\,,
 \end{aligned}
 \end{equation}
 \begin{equation}\label{stima tame derivata seconda dirichlet neumann}
 \begin{aligned}
  \| G''(\eta)[\widehat \eta, \widehat \eta] \psi \|_s^{k_0, \gamma}  & \leq_{s, k_0} \| \psi\|_{s + 3}^{k_0, \gamma} \big( \| \widehat \eta\|_{s_0 + 2}^{k_0, \gamma} \big)^2  \\
  & \quad \quad + \| \psi\|_{s_0 + 3}^{k_0, \gamma} \| \widehat \eta\|_{s + 2}^{k_0, \gamma} \| \widehat \eta\|_{s_0 + 2}^{k_0, \gamma}  \\
& \quad \quad  + \| \eta\|_{s + s_0 + 2 k_0 + 5}^{k_0, \gamma} \| \psi\|_{s_0 + 3}^{k_0, \gamma} \big(\| \widehat \eta\|_{s_0 + 2}^{k_0, \gamma} \big)^2\,.  
 \end{aligned}
 \end{equation}
   \end{lemma}
\begin{proof}
The estimate \eqref{stima tame dirichlet neumann} follows by the formula \eqref{sviluppo Geta}, the bound
\eqref{estimate DN} (for $ m = \alpha = 0 $) and  Lemmata \ref{lemma: action Sobolev}, \ref{lemma operatore e funzioni dipendenti da parametro}.
The estimate \eqref{stima tame derivata dirichlet neumann} follows by the shape derivative formula \eqref{formula shape der}, applying  \eqref{stima tame dirichlet neumann}, \eqref{interpolazione C k0} and the fact that the functions
$ B , V $ defined in \eqref{def B V} satisfy  
$$
\| B\|_s^{k_0, \gamma},\, \| V\|_{s}^{k_0, \gamma} \leq_s \| \psi\|_{s + 1}^{k_0, \gamma} + \| \eta\|_{s + s_0 + 2 k_0 + 3}^{k_0, \gamma} \| \psi\|_{s_0 + 1}^{k_0, \gamma}\,.
$$
The estimate \eqref{stima tame derivata seconda dirichlet neumann} follows by differentiating the shape derivative formula \eqref{formula shape der} and by applying the same kind of arguments. 
\end{proof}

\chapter{Transversality properties of 
degenerate KAM theory}
\label{sec:degenerate KAM}

In this section we verify the weak {\it transversality} properties required by degenerate KAM theory that we shall use for proving the measure estimates. To this aim we follow the approach developed in \cite{BaBM}.
The main result of this section is Proposition \ref{Lemma: degenerate KAM}, which is derived by the non-degeneracy 
Lemma \ref{non degenerazione frequenze imperturbate}.

\begin{definition}\label{def:non-deg}
A function $ f := (f_1, \ldots, f_N ) :  [ \kappa_1, \kappa_2] \to \R^N $ is called non-degenerate 
if,  for any vector $ c := (c_1, \ldots, c_N ) \in \R^N \setminus \{0\}$ 
the function $ f \cdot c = f_1 c_1 + \ldots + f_N c_N $ is not identically zero on the whole interval $ [ \kappa_1, \kappa_2]  $. 
\end{definition}

From a geometric point of view, $ f $ non-degenerate means that the image of the curve $ f([\kappa_1,\kappa_2]) \subset \R^N $ is not contained in any  hyperplane of $ \R^N $. For such reason a curve $ f $ which satisfies 
the non-degeneracy property of Definition \ref{def:non-deg} is also referred as
an {\it essentially non-planar} curve, or a curve with {\it full torsion}.  
For a smooth degenerate function $ f  $, 
differentiating $ (N - 1)$ times  the identity $ f(\kappa) \cdot c = 0 $, we see that 
\be\label{function-degenerate}
\begin{aligned}
& f (\kappa)  \ {\rm degenerate} \quad \Longrightarrow \quad  \\
& f (\kappa), (\pa_\kappa f)(\kappa),  \ldots , (\pa_\kappa^{N-1} f)(\kappa)
{\rm \ are \ linearly \ dependent } \ \forall 
 \kappa \in [\kappa_1, \kappa_2] \, . 
 \end{aligned}
\ee
Given $ {\mathbb S}^+ \subset \N^+ $ we denote the unperturbed tangential and normal frequency vectors by 
\be\label{tangential-normal-frequencies}
{\vec \om} (\kappa) := ( \om_j (\kappa) )_{j \in {\mathbb S}^+} \, , \quad 
{\vec \Omega} (\kappa) := ( \Om_j (\kappa) )_{j \in \N^+ \setminus {\mathbb S}^+} := ( \om_j (\kappa) )_{j \in \N^+ \setminus {\mathbb S}^+} \,  .  
\ee
\begin{lemma}\label{non degenerazione frequenze imperturbate}
The frequency vectors $ {\vec \om} (\kappa) \in \R^\nu $, $ (\sqrt{\kappa}, {\vec \om} (\kappa)) \in \R^{\nu+1} $ and 
$$
\begin{aligned}
& ( {\vec \om} (\kappa), \Omega_j(\kappa)) \in \R^{\nu+1}, j \in \N^+ \setminus {\mathbb S}^+, \\
& ( {\vec \om} (\kappa), { \Omega}_j(\kappa), { \Omega}_{j'}(\kappa)) \in \R^{\nu+2} \,, 	\ \forall j, j' \in 
\N^+ \setminus {\mathbb S}^+ \,, \ j \neq j'\,,
\end{aligned}
$$
are non-degenerate.
\end{lemma}

\begin{proof}
Set $ \lambda_0(\kappa):= \sqrt{\kappa} $ and $   \lambda_j(\kappa) := \sqrt{j (1 + \kappa j^2)}$, $ j \geq 1 $.
The lemma follows by proving that, for any $ N $, for any 
$ \lambda_{j_1}(\kappa), \ldots, \lambda_{j_N}(\kappa) $, with 
$ j_1, \ldots, j_N \geq 1 $, $ j_i \neq j_k$ for all $ i \neq k$,
the function $  [\kappa_1, \kappa_2] \ni \kappa  \mapsto (\lambda_{j_1}(\kappa), \ldots, \lambda_{j_N}(\kappa)) \in \R^N $
is non-degenerate according to Definition \ref{def:non-deg}.
By \eqref{function-degenerate} it is sufficient to prove that the $ N \times N$-matrix 
$$
{\mathcal A}(\kappa) := \begin{pmatrix}
\lambda_{j_1}(\kappa) & \lambda_{j_2}(\kappa) & \dots & \lambda_{j_N}(\kappa) \\
\partial_\kappa \lambda_{j_1}(\kappa) & \partial_\kappa \lambda_{j_2}(\kappa) & \dots & \partial_\kappa \lambda_{j_N}(\kappa) \\
\vdots & \vdots & \ddots & \vdots \\
\partial_\kappa^{N - 1} \lambda_{j_1}(\kappa) & \partial_\kappa^{N - 1} \lambda_{j_2}(\kappa) & \dots & \partial_\kappa^{N - 1} \lambda_{j_N}(\kappa)
\end{pmatrix}
$$
is non-singular at some value of $ \kappa \in [\kappa_1, \kappa_2] $. Actually, it turns out to
be non-singular for all $ \kappa \in [\kappa_1, \kappa_2] $. 

Arguing by induction we get the following formula for the derivatives of $\lambda_j(\kappa)$: for all $ r \geq 1$
\be\label{formula utile lambda 0}
\partial_\kappa^r \lambda_0(\kappa)  = \frac{(- 1)^{r + 1}}{2^r}(2 r - 3)!! \kappa^{- \frac{2 r - 1}{2}}
= (- 1)^{r + 1} (2 r - 3)!! \lambda_0(\kappa) x_0^r \,,\  x_0 := \frac{1}{2 \kappa}\, , 
\ee
where $(- 1)!! := 1$, $ 1!! := 1$ and if $n > 1 $ is odd $n !! := \prod_{k = 0}^{\frac{n - 1}{2}} (n - 2 k) $.  For all $ j, r \geq 1$ 
\be\label{formula utile lambda j}
\begin{aligned}
\partial_{\kappa}^r \lambda_j(\kappa)  & = \frac{\sqrt{j} j^{2 r}}{2^r}(- 1)^{r + 1} (2 r - 3)!! (1 + \kappa j^2)^{- \frac{2 r - 1}{2}} \\
& = (- 1)^{r + 1} (2 r - 3)!! \lambda_j(\kappa) x_j^r\,, \  \ x_j := \frac{j^2}{2(1 + \kappa j^2)} \, . 
\end{aligned}
\ee
Using the previous formulas \eqref{formula utile lambda 0}-\eqref{formula utile lambda j} 
and  the multi-linearity of the determinant we get 
$$
{\rm det}({\mathcal A}(\kappa)) = \prod_{k = 1}^{N} \lambda_{j_k}(\kappa) \prod_{r = 1}^{N - 1}(- 1)^{r + 1} (2 r - 3)!! \, {\rm det}({\mathcal B}(\kappa))
$$
where the $ N \times N $ matrix 
$$
{\mathcal B}(\kappa) := \begin{pmatrix}
1 & 1 & \dots & 1 \\
x_{j_1} & x_{j_2} & \dots & x_{j_N} \\
\vdots & \vdots & \ddots & \vdots \\
x_{j_1}^{N - 1} & x_{j_2}^{N - 1}  & \dots & x_{j_N}^{N - 1}
\end{pmatrix}
$$
is the Vandermonde matrix. Its determinant is 
\be\label{determinante Van der Monde}
{\rm det}({\mathcal B}(\kappa)) = \prod_{1 \leq i < k \leq N} (x_{j_i} - x_{j_k} ) \, . 
\ee
By the definition of $x_j$ in 
\eqref{formula utile lambda 0}-\eqref{formula utile lambda j}, we have that, for all $ \kappa \in [ \kappa_1, \kappa_2] $,  
\begin{align*}
x_j - x_{j'}  & = \frac12 \frac{j^2 - j'^2}{ (1 + \kappa j^2)(1 + \kappa j'^2)} \neq 0\,,\,  \forall j \neq j' ,  j, j' \geq 1\, , \\
  x_j - x_0  & = - \frac{1}{2 \kappa (1 + \kappa j^2)} \neq 0\, , \  \forall j \geq 1 \, . 
\end{align*}
Thus, by \eqref{determinante Van der Monde} the determinant 
${\rm det}({\mathcal B}(\kappa)) \neq 0 $  and so  ${\rm det}({\mathcal A}(\kappa)) \neq 0 $,
$ \forall \kappa \in [\kappa_1, \kappa_2 ] $,  proving the lemma.
\end{proof}

In the next Proposition \ref{Lemma: degenerate KAM} we deduce, 
by the qualitative non-degeneracy condition proved in Lemma 
\ref{non degenerazione frequenze imperturbate}, the analyticity and the asymptotics of 
the linear frequencies $ \kappa \mapsto 
 \om_j (\kappa) = \sqrt{j (1+ \kappa j^2)} $, the quantitative bounds \eqref{0 Melnikov}-\eqref{2 Melnikov+}. 
The proof is similar to \cite{BaBM}. It does not follow immediately by \cite{BaBM} because the linear frequencies 
$  \om_j (\kappa) $
depend on the parameter $ \kappa $ also at the highest order $ O( \sqrt{\kappa} j^{3/2} )$.

\begin{proposition}\label{Lemma: degenerate KAM}
{\bf (Transversality)}
There exist $ k_0 \in \N $, $ \rho_0 > 0$ such that, for any $\kappa \in [ \kappa_1, \kappa_2] $,
\begin{equation}\label{0 Melnikov}
\begin{aligned}
{\rm max}_{k \leq k_0} |\partial_\kappa^{k}  \{{\vec \om} (\kappa) \cdot \ell   \} | & 
 \geq \rho_0 \langle \ell \rangle\,, \quad \forall \ell  \in \Z^\nu \setminus \{ 0 \}, 
\end{aligned}
\end{equation}
\begin{equation}\label{1 Melnikov}
\begin{aligned}
{\rm max}_{k \leq k_0} |\partial_\kappa^{k}  \{{\vec \om} (\kappa) \cdot \ell  + { \Omega}_j (\kappa) \} | 
 & \geq \rho_0 \langle \ell  \rangle\,, \quad \forall \ell  \in \Z^\nu, \, j \in \N^+ \setminus {\mathbb S}^+,
\end{aligned}
\end{equation}
\begin{equation}\label{2 Melnikov-}
\begin{aligned}
& {\rm max}_{k \leq k_0} |\partial_\kappa^{k}  \{{\vec \om} (\kappa) \cdot \ell  
 + { \Omega}_j (\kappa) - { \Omega}_{j'}(\kappa) \} | 
  \geq \rho_0 \langle \ell  \rangle\,,  \\
  & \quad \forall (\ell , j, j') \neq (0, j, j), \, \quad \ell \in \Z^\nu\,, \quad  j, j' \in \N^+ \setminus {\mathbb S}^+,
\end{aligned}
\end{equation}
\begin{equation} \label{2 Melnikov+}
\begin{aligned}
& {\rm max}_{k \leq k_0} |\partial_\kappa^{k}  \{{\vec \om} (\kappa) \cdot \ell 
 + { \Omega}_j (\kappa) + { \Omega}_{j'}(\kappa) \} |  
 \geq \rho_0 \langle \ell  \rangle\,,  \\
 & \quad  \forall \ell  \in \Z^{\nu}, \quad   j, j' \in \N^+ \setminus {\mathbb S}^+  \, .
\end{aligned}
\end{equation}
We call  (following  \cite{Ru1}) 
$ \rho_0 $  the ``{\it amount of non-degeneracy}" and $ k_0 $ the ``{\it index of nondegeneracy}".
\end{proposition}

\begin{proof} All the inequalities  \eqref{0 Melnikov}-\eqref{2 Melnikov+} are proved by contradiction.
\\[1mm]
{\sc Proof of \eqref{0 Melnikov}}.
Suppose 
that $\forall k_0 \in \N$, $\forall \rho_0 > 0$ there exist 
 $ \ell  \in \Z^\nu \setminus \{0\} $, $\kappa \in [ \kappa_1, \kappa_2]$  such that 
$ {\rm max}_{k \leq k_0} |\partial_\kappa^k\{ {\vec \om}(\kappa) \cdot \ell  \} | < \rho_0 \langle \ell  \rangle $. 
This implies 
that for all $ m \in \N $, taking $ \rho_0 = \frac{1}{1+m } $,
there exist $ \ell_m \in \Z^\nu \setminus \{0\} $, $\kappa^{(m)} \in [ \kappa_1, \kappa_2]$ such that 
$$
 {\rm max}_{k \leq m} |  \partial_\kappa^k \{ {\vec \om} (\kappa^{(m)}) \cdot \ell _m \}| <  \frac{1}{1 + m} 
 \langle \ell_m  \rangle
$$
and therefore 
\begin{equation}\label{rick 1}
\forall k \in \N, \quad m \geq k \, , \quad 
\Big| \partial_\kappa^k {\vec \om} (\kappa^{(m)}) \cdot \frac{\ell_m}{\langle \ell_m \rangle} \Big| < \frac{1}{1 + m}\,.
\end{equation}
The sequences $(\kappa^{(m)})_{m \in \N} \subset [ \kappa_1, \kappa_2] $
and $( \ell_m / \langle \ell_m \rangle)_{m \in \N} \subset \R^\nu \setminus \{0\} $ are bounded.
By compactness  there exists a sequence $ m_h \to + \infty $ such that 
$ \kappa^{(m_h)} \to \bar \kappa \in [ \kappa_1, \kappa_2], $ $ \ell_{m_h} / \langle \ell_{m_h} \rangle \to \bar c \neq 0 $. 
Passing to the limit in \eqref{rick 1} for $m_h \to + \infty$ we deduce that 
$ \partial_\kappa^k {\vec \om}(\bar \kappa) \cdot \bar c = 0 $, $  \forall k \in \N $. 
We conclude that the analytic function $ \kappa \mapsto {\vec \om} (\kappa)\cdot \bar c $ is identically zero.  
Since $ \bar c \neq 0 $, 
this is in contradiction with Lemma \ref{non degenerazione frequenze imperturbate}.
\\[1mm]
{\sc Proof of \eqref{1 Melnikov}}. 
Recalling that $ { \Omega}_j(\kappa) = \sqrt{j(1 + \kappa j^2)}$, we have the expansion 
\begin{equation}\label{espansione asintotica degli autovalori}
{ \Omega}_j(\kappa) = \sqrt{\kappa} j^{\frac32} +  \frac{c_j(\kappa)}{\sqrt{\kappa j}}\,,
\qquad c_j(\kappa) := \frac12 \int_0^1  \big( 1 + \frac{t}{\kappa j^2} \big)^{-1/2} d t
\end{equation}
where 
\begin{equation}\label{derivata cj}
\forall k \in \N \, , \quad \big| \partial_\kappa^k \frac{c_j(\kappa)}{\sqrt{\kappa }} \big| \leq C(k) 
\end{equation}
uniformly in $  j \in {\mathbb S}^c $, $ \kappa \in [\kappa_1, \kappa_2]  $. 

First of all note that $ \forall \kappa \in [\kappa_1, \kappa_2] $, we have 
$ |{\vec \om} (\kappa ) \cdot \ell  + { \Omega}_j (\kappa )  | \geq $ $ 
{ \Omega}_j (\kappa )   - |{\vec \om} (\kappa )  \cdot \ell  | \geq $ $ \sqrt{\kappa_1} j^{3/2} - C | \ell  | \geq  |\ell  |  $
if $ j^{3/2} \geq C_0 |\ell  | $ for some $ C_0 > 0 $. 
Therefore in \eqref{1 Melnikov} we can restrict to the indices 
$ (\ell  , j ) \in \Z^\nu \times (\N^+ \setminus {\mathbb S}^+) $ satisfying  
\be\label{prima restrizione}
j^{\frac32} < C_0 | \ell  | \, . 
\ee
Arguing by contradiction (as for proving \eqref{0 Melnikov}), we suppose that 
for all $m \in \N$ there exist  $ \ell_m \in \Z^\nu$, $ j_m \in \N^+ \setminus {\mathbb S}^+  $ and 
$ \kappa^{(m)} \in [ \kappa_1, \kappa_2]$, such that 
$$
\max_{k \leq m}\Big| 
\partial_\kappa^k \Big\{ {\vec \om}(\kappa^{(m)}) \cdot \frac{\ell_m}{\langle \ell_m \rangle} + 
\frac{{ \Omega}_{j_m}(\kappa^{(m)})}{\langle \ell_m \rangle} \Big\} \Big| 
< \frac{1}{1 + m} 
$$
and therefore
\begin{equation}\label{rick 2}
\forall k \in \N, \quad m \geq k \, , \quad 
\Big| 
\partial_\kappa^k \Big\{ {\vec \om}(\kappa^{(m)}) \cdot \frac{\ell_m}{\langle \ell_m \rangle} + 
\frac{{ \Omega}_{j_m}(\kappa^{(m)})}{\langle \ell_m \rangle} \Big\} \Big| 
< \frac{1}{1 + m} \, . 
\end{equation}
Since the sequences $(\kappa^{(m)})_{m \in \N} \subset [\kappa_1, \kappa_2] $ 
and $(\ell_{m} / \langle \ell_m \rangle)_{m \in \N} \in \R^\nu $ are bounded, there exist $m_h \to + \infty$ such that 
\be\label{up to subsequence}
\kappa^{(m_h)} \to \bar \kappa \in [ \kappa_1, \kappa_2]\,,\quad \frac{\ell_{m_h}}{\langle \ell_{m_h} \rangle} \to \bar c \in \R^\nu \,.
\ee
We now distinguish two cases:
\\[1mm]
{\sc Case 1: $ (\ell_{m_h}) \subset \Z^\nu $ is bounded}. In this case, 
up to subsequence,  $\ell_{m_h} \to \bar \ell \in \Z^\nu $, and since $|j_m| \leq C 
| \ell_m|^{\frac23}$ for all $m$ (see \eqref{prima restrizione}), 
we have $j_{m_h} \to \bar \jmath $. Passing to the limit for $m_h \to + \infty$ in \eqref{rick 2} we deduce, by \eqref{up to subsequence}, that 
$$
\partial_\kappa^k \big\{ {\vec \om}(\bar \kappa) \cdot \bar c + 
{ \Omega}_{\bar \jmath} (\bar \kappa)  \langle \bar \ell \rangle^{-1} \big\} = 0\,,\quad \forall k \in \N \, .
$$
Therefore the analytic function
$ \kappa \mapsto {\vec \om}(  \kappa) \cdot \bar c + {\langle \bar \ell \rangle}^{-1} { \Omega}_{\bar \jmath}( \kappa)  $
is identically zero. Since
$ ( \bar c, \langle \ell \rangle^{-1}) \neq 0 $  this is in contradiction with Lemma \ref{non degenerazione frequenze imperturbate}.
\\[1mm]
{\sc Case 2: $ (\ell_{m_h})$ is unbounded}. Up to subsequence $ | \ell_{m_h}| \to + \infty $.
In this case the constant $ \bar c \neq 0 $ in \eqref{up to subsequence}.  Moreover, 
by \eqref{prima restrizione}, we also have that, up to subsequences, 
\be\label{rapporto tende}
j_{m_h}^{\frac32}  \langle \ell_{m_h} \rangle^{-1} \to \bar d \in \R \, .
\ee
By \eqref{espansione asintotica degli autovalori}, \eqref{derivata cj},  \eqref{up to subsequence}, \eqref{rapporto tende}, 
we get
\be\label{convergenza Meln1}
\begin{aligned}
& \frac{{ \Omega}_{j_{m_h}}( \kappa^{(m_h)})}{\langle \ell_{m_h} \rangle} = 
\sqrt{\kappa^{(m_h)}} \frac{ j_{m_h}^{\frac32}}{\langle \ell_{m_h} \rangle} + 
\frac{c_{j_{m_h}}(\kappa^{(m_h)})}{\sqrt{\kappa^{(m_h)} j_{m_h}  }\langle \ell_{m_h}\rangle} 
\to \bar d \sqrt{\bar \kappa} \, , \\
& \partial_\kappa^k \frac{{ \Omega}_{j_{m_h}} (\kappa^{(m_h)})}{\langle \ell_{m_h} \rangle} 
\to  \bar d \, \partial_\kappa^k \sqrt{\bar \kappa} 
\end{aligned}
\ee
as $ m_h \to + \infty $. 
Passing to the limit in \eqref{rick 2}, by \eqref{convergenza Meln1}, \eqref{up to subsequence}
we deduce that 
$ \partial_\kappa^k \big\{ {\vec \om}(\bar \kappa) \cdot \bar c + \bar d \sqrt{\bar \kappa} \big\} = 0 $, $ \forall k \in \N $. 
Therefore the analytic function 
$ \kappa \mapsto {\vec \om}( \kappa) \cdot \bar c + \bar d \sqrt{ \kappa} = 0 $ is identically zero. 
Since $ (\bar c, \bar d ) \neq 0 $
this is in contradiction with  Lemma \ref{non degenerazione frequenze imperturbate}.
\\[1mm]
{\sc Proof of \eqref{2 Melnikov-}}. Notice that, for all $ \kappa \in [\kappa_1, \kappa_2] $, 
\begin{align*}
| {\vec \om} (\kappa) \cdot \ell  + { \Omega}_j (\kappa) - { \Omega}_{j'} (\kappa) | 
& \geq |{ \Omega}_j (\kappa) - { \Omega}_{j'} (\kappa) | - | {\vec \om} (\kappa) | |\ell  | \\
& \stackrel{\eqref{espansione asintotica degli autovalori}, \eqref{derivata cj}}
\geq  \sqrt{\kappa_1} |j^{\frac32} - j'^{\frac32}| -  
C - C |\ell| \geq \langle \ell \rangle 
\end{align*}
provided $ | j^{\frac32} - j'^{\frac32}|  \geq C_1  \langle \ell  \rangle $, for some $ C_1 > 0 $.
Therefore in  \eqref{2 Melnikov-} we can restrict to the indices such that 
\begin{equation}\label{rick 4}
|j^{\frac32} - j'^{\frac32}| < C_1 \langle \ell \rangle\,.
\end{equation}
Moreover in \eqref{2 Melnikov-} we can also assume that $ j \neq j' $ otherwise 
\eqref{2 Melnikov-} reduces to \eqref{0 Melnikov}, which is already proved. 

Now if, by contradiction, \eqref{2 Melnikov-} is false, 
we deduce, arguing as in the previous cases, 
that for all $m \in \N$, there exist $ \ell_m \in \Z^\nu$, $j_m, j'_m \in \N^+ \setminus {\mathbb S}^+ $, $ j_m \neq j'_m $, 
$ \kappa^{(m)} \in [ \kappa_1, \kappa_2]$, such that for all 
\begin{equation}\label{rick 5}
k \in \N \, , \  \forall m \geq k \, , \  
\Big| \partial_\kappa^k \Big\{ {\vec \om}(\kappa^{(m)}) \cdot \frac{\ell_m}{\langle \ell_m \rangle} + \frac{{ \Omega}_{j_m}(\kappa^{(m)})}{\langle \ell_m \rangle} - \frac{{ \Omega}_{j'_m}(\kappa^{(m)})}{\langle \ell_m \rangle} \Big\} \Big| 
<  \frac{1}{1 + m}\,.
\end{equation}
As in the previous cases, since the sequences 
$ (\kappa^{(m)})_{m \in \N}$, $(\ell_{m} / \langle \ell_{m}\rangle)_{m \in \N} $ 
are bounded, there exists $ m_h \to + \infty $ such that 
\be\label{conver2}
\kappa^{(m_h)} \to \bar \kappa \in [ \kappa_1, \kappa_2]\,,\quad \ell_{m_h} / \langle \ell_{m_h} \rangle \to \bar c 
\in \R^\nu \,.
\ee
We distinguish again two cases:
\\[1mm]
{\sc Case 1 : $ (\ell_{m_h} )$ is bounded}.
In this case, up to subsequence, $ \ell_{m_h} \to \bar \ell \in \Z^\nu $. 
Using that 
$$
|j^{\frac32} - j'^{\frac32}| \geq |j - j'|(\sqrt{j} + \sqrt{j'}) \geq \sqrt{j} + \sqrt{j'} \,, \quad \forall j \neq j' \, , 
$$ 
by \eqref{rick 4} we deduce that also $ j_{m_h} $, $ j'_{m_h}$ are 
bounded sequences and therefore, up to subsequence,  
\be\label{convjj'}
j_{m_h} \to \bar \jmath \,, \qquad j'_{m_h} \to \bar \jmath' \,, \quad \bar \jmath \neq \bar \jmath' \, . 
\ee
Hence passing to the limit in \eqref{rick 5} for $ m_h \to + \infty $, we deduce by \eqref{conver2}, \eqref{convjj'}  that 
$$
\partial_\kappa^k \big\{ {\vec \om}(\bar \kappa) \cdot \bar c + 
{ \Omega}_{\bar \jmath}(\bar \kappa)  \langle \bar \ell \rangle^{-1} - 
{ \Omega}_{\bar \jmath'}(\bar \kappa) \langle \bar \ell \rangle^{-1} \big\} = 0 \,, \quad \forall k \in \N \, . 
$$
Therefore the analytic function 
$ \kappa \mapsto  {\vec \om}( \kappa) \cdot \bar c + 
{ \Omega}_{\bar \jmath}(\kappa) \langle \bar \ell \rangle^{-1} - 
{ \Omega}_{\bar \jmath'}( \kappa) \langle \bar \ell \rangle^{-1} $
is identically zero.  This in contradiction with  Lemma \ref{non degenerazione frequenze imperturbate}.
\\[1mm]
{\sc Case 2 : $ (\ell_{m_h}) $ is unbounded}. Up to subsequence $ | \ell_{m_h}| \to + \infty$.  
In this case the constant $ \bar c \neq 0 $ in \eqref{conver2}. 
Using \eqref{espansione asintotica degli autovalori}-\eqref{derivata cj},  for all $k \in \N$, 
\begin{align*}
\partial_\kappa^k \frac{{ \Omega}_{j_{m_h}}(\kappa^{(m_h)}) - { \Omega}_{j'_{m_h}}( \kappa^{(m_h)} )}{\langle \ell_{m_h} \rangle} & = \partial_\kappa^k \sqrt{\kappa_{m_h}} \frac{j^{\frac32} -
j'^{\frac32}}{\langle \ell_{m_h} \rangle} + \frac{1}{\sqrt{j_{m_h}}\langle \ell_{m_h}\rangle} \partial_\kappa^k \frac{c_{j_{m_h}}(\kappa^{(m_h)})}{\sqrt{\kappa^{(m_h)}}}  \\
& \quad - \frac{1}{\sqrt{j'_{m_h}}\langle 
\ell_{m_h}\rangle} \partial_\kappa^k \frac{c_{j'_{m_h}}(\kappa^{(m_h)})}{\sqrt{\kappa^{(m_h)}}}
\end{align*}
and
\begin{align*}
& \Big|\frac{1}{\sqrt{j_{m_h}}\langle \ell_{m_h}\rangle} \partial_\kappa^k \frac{c_{j_{m_h}}(\kappa^{(m_h)})}{\sqrt{\kappa^{(m_h)}}} - \frac{1}{\sqrt{j'_{m_h}}\langle \ell_{m_h}\rangle} \partial_\kappa^k \frac{c_{j'_{m_h}}(\kappa^{(m_h)})}{\sqrt{\kappa^{(m_h)}}} \Big|  \\
&  \leq \frac{C}{\langle \ell_{m_h} \rangle} 
{\rm sup}_{j \in \N^+ \setminus {\mathbb S}^+, \kappa \in [\kappa_1, \kappa_2]} \Big| \partial_\kappa^k \frac{c_j(\kappa)}{\sqrt{\kappa}} \Big|   \leq \frac{C' (k)}{\langle \ell_{m_h} \rangle} \to 0
\end{align*}
as $ m_h \to + \infty $. Moreover, by \eqref{rick 4}, 
up to subsequences,  
$ |j_{m_h}^{\frac32} - {j'}_{m_h}^{\frac32}| \langle \ell_{m_h} \rangle^{-1} \to \bar d \in \R $. 
Therefore, for all $k \in \N$, 
$$
\partial_\kappa^k \frac{{ \Omega}_{j_{m_h}}(\kappa^{(m_h)}) -
 { \Omega}_{j'_{m_h}}(\kappa^{(m_h)}) }{\langle \ell_{m_h} \rangle} \to \bar d \,  \partial_\kappa^k \sqrt{\bar \kappa}\,.
$$
Passing to the limit in \eqref{rick 5} for $m_h \to + \infty$ we deduce that 
$ \partial_\kappa^k \big\{{\vec \om}(\bar \kappa) \cdot \bar c + \bar d \sqrt{\bar \kappa} \big\}  = 0 $, $  \forall k  \in \N $. 
In conclusion the analytic function 
$ \kappa \mapsto {\vec \om}( \kappa) \cdot \bar c + \bar d \sqrt{\kappa} $
is identically zero. Since $(\bar c, d ) \neq 0 $,
this is a contradiction  with Lemma \ref{non degenerazione frequenze imperturbate}.
\\[1mm]
{\sc Proof of \eqref{2 Melnikov+}.} The proof is similar to the previous ones and we omit it. 
\end{proof}

\chapter{Nash-Moser theorem and measure estimates}\label{sec:functional}

Instead of working in a shrinking neighborhood of the origin, it is a convenient devise to rescale
the  variable $ u \mapsto \e u  $ with $ u = O (1) $, writing  
\eqref{WW}-\eqref{HS} as
\be\label{WW-riscalato}
\pa_t u = J \Omega u + \e X_{P_\e} (u)  
\ee
where $ J \Omega $ is the linearized Hamiltonian vector  field in \eqref{definizione Omega} and
\begin{equation}\label{campo hamiltoniano X P epsilon}
\begin{aligned}
& X_{P_\e} (u) 
 := X_{P_\e}(\kappa, u) \\
& :=   \begin{pmatrix} 
 \e^{- 1}(G(\e \eta)- G(0))\psi \\
- \frac{1}{2} \psi_x^2  
+\frac{1}{2} \frac{ \big(G(\e \eta)\psi + \e \eta_x \psi_x \bigr)^2}{1+ (\e \eta_x)^2}
+ \e^{-1} \kappa \, \eta_{xx} \Big(  \big(1 + (\e \eta_x)^2 \big)^{-3/2} -  1 \Big)  \\
\end{pmatrix} \, .  
\end{aligned}
\end{equation}
Note that the dependence of the vector field $X_{P_\e}$ with respect to $\kappa$ is linear. 
System \eqref{WW-riscalato} is the Hamiltonian system generated by the Hamiltonian 
\be\label{Hamiltonian:rescaled}
{\mathcal H}_\e (u) := \e^{-2} H(\e u ) = H_L (u) + \e P_\e (u) 
\ee
where $ H $ is the water-waves Hamiltonian \eqref{Hamiltonian}, 
$ H_L $ is defined in \eqref{Hamiltonian linear} and
\begin{equation}\label{definizione P epsilon}
\begin{aligned}
P_\e (u) & := P_\e(\kappa, u) \\  
& := \frac{\e^{-1}}{2} (\psi, \big( G(\e \eta) - G(0) \big) \psi)_{L^2(\T_x)}  \\
& \quad +
\e^{-3} \kappa \int_{\T}\Big(  \sqrt{1 + (\e \eta_x)^2} - 1 - \frac{(\e \eta_x)^2}{2}   \Big) \, dx \, .
\end{aligned}
\end{equation}
We decompose the phase space
\begin{equation}
H_{0, {\rm even}}^1 := \Big\{ u := (\eta, \psi) \in H_0^1(\T_x) \times H_0^1(\T_x)\,, \quad u(x) = u(- x)\Big\}
\end{equation} 
as the direct sum of the symplectic subspaces
\begin{equation}\label{H S + H S + bot}
H_{0, {\rm even}}^1 = H_{{\mathbb S}^+} \oplus  H_{{\mathbb S}^+}^\bot  \qquad {\rm where} \qquad
H_{{\mathbb S}^+} := \Big\{ v := \sum_{j \in {\mathbb S}^+} \begin{pmatrix} 
 \eta_j  \\
\psi_j \\
\end{pmatrix} \cos (jx)  \Big\}
\end{equation}
and $ H_{{\mathbb S}^+}^\bot $ denotes the $ L^2$-orthogonal.

We now introduce action-angle variables on the tangential sites by setting
\be\label{Ac-An}
\begin{aligned}
\eta_j & := \sqrt{\frac{2}{\pi}} \, \Lambda_j^{1/2} \sqrt{\xi_j + I_j}  \cos (\theta_j ) , \\  
\psi_j & := - \sqrt{\frac{2}{\pi}} \, \Lambda_j^{- 1/2} \sqrt{\xi_j + I_j} \, \sin (\theta_j ) \, ,  \   
\Lambda_j := \sqrt{j (1 + \kappa j^2)^{-1}} \, , \,  j \in {\mathbb S}^+ \, ,
\end{aligned}
\ee
where $ \xi_j > 0  $, $ j=1, \ldots, \nu $, are positive constants, the variables $ | I_j  | \leq \xi_j $, 
and we  leave unchanged  the  normal component $ z $. 
The symplectic $ 2 $-form in \eqref{2form tutto} then reads (for simplicity of notation we denote it in the same way)
\begin{equation}\label{2form}
{\mathcal W} := \Big( {\mathop \sum}_{j \in {\mathbb S}^+}  d \theta_j \wedge d I_j \Big)  \oplus 
{\mathcal W}_{|H_{{\mathbb S}^+}^\bot} = d \Lambda 
\end{equation}
where $ \Lambda $ is the  Liouville  $1$-form
\begin{equation}\label{Lambda 1 form}
\Lambda_{(\theta, I, z)}[\widehat \theta, \widehat I, \widehat z] := 
- \sum_{j \in {\mathbb S}^+} I_j  \widehat \theta_j - \frac12 \big( J z\,,\,\widehat z \big)_{L^2_x}  \, . 
\end{equation}
 Hence the Hamiltonian system \eqref{WW-riscalato} 
transforms into the new Hamiltonian system 
\be\label{HS-theta-I-z}
\dot \theta = \partial_I H_\e ( \theta, I, z ) \, , \ 
\dot I = - \partial_\theta H_\e ( \theta, I, z ) \, , \quad 
z_t = J \nabla_z H_\e ( \theta, I, z ) 
\ee
generated by the Hamiltonian
\be\label{new Hamiltonian} 
H_\e :=  {\mathcal H}_\e \circ A = \e^{-2} H \circ \e A 
\ee
where
\begin{equation}\label{definizione A}
A (\theta, I, z) := v (\theta, I) +  z := 
\sum_{j \in {\mathbb S}^+} 
\sqrt{\frac{2}{\pi}}
\begin{pmatrix} 
\ \Lambda_j^{1/2} \sqrt{\xi_j + I_j} \, \cos (\theta_j ) \\
- \Lambda_j^{- 1/2} \sqrt{\xi_j + I_j} \, \sin (\theta_j )  \\
\end{pmatrix} \cos (jx) + z \, . 
\end{equation}
We denote by 
$$ 
X_{H_\e} := (\partial_I H_\e,  - \partial_\theta H_\e, J \nabla_z H_\e) $$  
the Hamiltonian vector field
in the variables $(\theta, I, z ) \in \T^\nu \times \R^\nu \times H_{{\mathbb S}^+}^\bot $. 
The involution $ \rho $ in \eqref{defS} becomes 
\be\label{involuzione tilde rho}
\tilde \rho : (\theta, I, z ) \mapsto (- \theta, I, \rho z ) \, .
\ee
By \eqref{Hamiltonian} and \eqref{new Hamiltonian}  the Hamiltonian $ H_\e  $ reads (up to a constant) 
\begin{equation}\label{definizione cal N P}
H_\e = {\mathcal N} + \e P \, , \quad 
{\mathcal N} := H_L \circ A  =    {\vec \om} (\kappa) \cdot  I  + 
\frac12 (z, \Om z)_{L^2_x} \, , \quad P :=  P_\e \circ A\,, 
\end{equation}
where $ {\vec \om} (\kappa) $ is defined in \eqref{tangential-normal-frequencies}
 and  $ \Om $  in \eqref{definizione Omega}. We look for an embedded invariant torus 
\be\label{def:embed}
i : \T^\nu \to \T^\nu \times \R^\nu \times H_{{\mathbb S}^+}^\bot, \quad  
\vphi \mapsto i (\vphi) := ( \theta (\vphi), I (\vphi), z (\vphi)) 
\ee
of the Hamiltonian vector field $ X_{H_\e}  $ 
filled by quasi-periodic solutions with diophantine frequency $ \om \in \R^\nu $
(and which satisfies also  first and second order Melnikov-non-resonance conditions
as in \eqref{Cantor set infinito riccardo}).

\section{Nash-Moser Th\'eor\'eme de conjugaison hypoth\'etique}\label{sec:conjugation-hyp}

The Hamiltonian  $ H_\e $ in 
\eqref{definizione cal N P} is a perturbation of the isochronous  
Hamiltonian $ {\mathcal N} $. The expected quasi-periodic solutions of the Hamiltonian system  \eqref{HS-theta-I-z} 
will have a shifted frequency
which depends on the nonlinear term $ P $. In view of that 
we 
introduce the family of Hamiltonians 
\begin{equation}\label{H alpha}
H_\a := {\mathcal N}_\a + \e P \, , \quad {\mathcal N}_\a :=  \a \cdot I + \frac12 (z, \Om z)_{L^2_x} \, , \quad \a \in \R^\nu \, ,  
\end{equation}
which depend on the constant vector $ \a  \in \R^\nu $. For the value $ \a = \vec \om (\kappa ) $ we have $ H_{\a} = H_\e $.  
Then we look for a zero $ (i, \a) $
of the nonlinear operator
\begin{equation}\label{operatorF}
\begin{aligned}
{\mathcal F} (i, \a ) 
& :=   {\mathcal F} (i, \a, \om, \kappa,  \e )  := \Dom i (\vphi) - X_{H_\a}  (i (\vphi)  )
 \\
 & =  \Dom i (\vphi) -  (X_{{\mathcal N}_\a}  +  \e X_{P})  (i(\vphi) )    \\
&   :=  \left(
\begin{array}{c}
\Dom \theta (\vphi) -  \a - \e \partial_I P ( i(\vphi)   )   \\
\Dom I (\vphi)  +  \e \partial_\teta P( i(\vphi)  )  \\
\Dom z (\vphi) -  J (\Om z(\vphi) + \e \nabla_z P ( i(\vphi) ))  
\end{array}
\right) 
\end{aligned}
\end{equation}
for some  diophantine vector $ \om \in \R^\nu $.
Thus 
$ \vphi \mapsto i (\vphi) $ is an embedded torus, invariant for the Hamiltonian vector field $ X_{H_\a } $,
filled by quasi-periodic solutions with  frequency $ \om  $. 

Each Hamiltonian $ H_\a $ in \eqref{H alpha} is reversible, i.e. $  H_\a \circ \tilde \rho = H_\a $
where the involution $ \tilde \rho $ is defined in \eqref{involuzione tilde rho}. 
We look for reversible solutions of $ {\mathcal F}(i, \a) = 0 $,  namely satisfying $ {\tilde \rho} i (\vphi ) = i (- \vphi) $ 
(see \eqref{involuzione tilde rho}), i.e.  
\begin{equation}\label{parity solution}
\theta(-\vphi) = - \theta (\vphi) \, , \quad 
I(-\vphi) = I(\vphi) \, , \quad 
z (- \vphi ) = ( \rho z)(\vphi) \, . 
\end{equation}
The weighted Sobolev norm of the periodic component of the embedded torus 
\begin{equation}\label{componente periodica}
{\mathfrak I}(\vphi)  := i (\vphi) - (\vphi,0,0) := ( {\Theta} (\ph), I(\ph), z(\ph))\,, \quad \Theta(\ph) := \teta (\vphi) - \vphi \, , 
\end{equation}
is 
\be\label{weighted-Sobo-three}
\|  {\mathfrak I}  \|_s^{k_0, \gamma} := \| \Theta \|_{H^s_\vphi}^{k_0, \gamma} +  \| I  \|_{H^s_\vphi}^{k_0, \gamma} 
+  \| z \|_s^{k_0, \gamma} 
\ee
where  $  \| z \|_s^{k_0, \gamma} :=  \| \eta \|_s^{k_0, \gamma} + \| \psi \|_s^{k_0, \gamma} $ and 
$\| \ \|_s^{k_0, \gamma} $ is the weghted Sobolev norm defined
 in \eqref{norma pesata derivate funzioni}.

For the next theorem, we recall that $ k_0 $ 
is the index of non-degeneracy provided by Proposition \ref{Lemma: degenerate KAM} and   
it depends only on the linear unperturbed frequencies. Therefore it is considered as 
an absolute constant and 
we will often omit to write explicitly the dependence of the constants with respect to  $ k_0 $.
We look for quasi periodic solutions with frequency $ \om $  
belonging to a $  \d $-neighborhood (independent of $ \e $)
\be\label{unperturbed-frequencies} 
\tOm := \Big\{  \om \in \R^\nu \, : \, {\rm dist} \big(\om,  {\vec \omega}[\kappa_1, \kappa_2]\big) < \d \, , \
\d > 0 \Big\} 
\ee
of the 
unperturbed linear frequencies $ {\vec \omega}[\kappa_1, \kappa_2] $ defined  in \eqref{tangential-normal-frequencies}.

\begin{theorem}\label{MAINTHEOREM}
{\bf (Nash-Moser)}
Fix finitely many tangential sites $ {\mathbb S}^+ \subset \N^+ $ and let $ \nu := |{\mathbb S}^+ | $. 
Let $ \tau \geq 1 $.  There exist constants $ \e_0 > 0 $, 
$ a_0 := a_0 (\nu, \tau, k_0) > 0 $ and $k_1 := k_1(\nu, k_0, \tau) > 0$ such that, for all 
$ \g = \e^a $, $ 0 < a < a_0 $, $ \e \in (0, \e_0) $,
there exist a ${k_0}$-times differentiable function
\be\label{mappa aep}
\begin{aligned}
& \a_\infty : \tOm \times [\kappa_1, \kappa_2] \mapsto \R^\nu \, , \\
&   \a_\infty (\om, \kappa ) = \om + r_\e (\om, \kappa ) \, , \quad {\text with} \quad  |r_\e|^{k_0, \gamma}  \leq C \e \g^{-(1 + k_1)} \, , 
\end{aligned}
\ee
a family of embedded tori  $ i_\infty  $ defined for all $ \om \in \tOm $ and $\kappa \in [\kappa_1, \kappa_2]$ satisfying the reversibility property \eqref{parity solution} and 
\be\label{stima toro finale}
\|  i_\infty (\vphi) -  (\vphi,0,0) \|_{s_0}^{k_0, \g} \leq C \e \g^{-(1 + k_1)}  \, , 
\ee
 a sequence of ${k_0}$-times differentiable functions $ \mu_j^\infty : \tOm \times [\kappa_1, \kappa_2] \to \R $, $  j \in \N^+ \setminus {\mathbb S}^+ $, of the form
\begin{equation}\label{autovalori infiniti}
\mu_j^\infty (\omega, \kappa) =
\mathtt m_3^\infty (\om, \kappa) j^\frac12 (1 + \kappa j^2)^{\frac12} +  \mathtt m_1^\infty (\om, \kappa)  j^{\frac12} 
+ r_j^\infty (\omega, \kappa)
\end{equation}
(defined in \eqref{autovalori finali riccardo}) 
satisfying 
\begin{equation}\label{stime autovalori infiniti}
| \mathtt m_3^\infty  - 1|^{k_0, \gamma} +  | \mathtt m_1^\infty|^{k_0, \gamma} \leq C \e\,, \qquad    \sup_{j \in {\mathbb S}^c} |r_j^\infty |^{k_0, \gamma}   \leq C  \e \gamma^{- k_1}\, ,
\end{equation}
such that  for all $(\omega, \kappa)$ in the  Borel set
\begin{equation}\label{Cantor set infinito riccardo}
\begin{aligned}
{\mathcal C}_\infty^{\gamma} & := \Big\{ (\omega, \kappa) \in \tOm \times [\kappa_1, \kappa_2] \, : 
 \, |\om \cdot \ell  | \geq \g \langle \ell \rangle^{-\tau}, \, \forall \ell \in \Z^{\nu} \setminus \{ 0 \},  \\
& \qquad   |\omega \cdot \ell  + \mu_j^\infty (\om, \kappa)  | \geq 4 \gamma j^{\frac32} \langle \ell  \rangle^{- \tau}, \,   \forall \ell   \in \Z^\nu, \, j \in \N^+ \setminus {\mathbb S}^+    \\
& \qquad \qquad  \qquad \qquad \text{$(1$-Melnikov conditions$)$}, \\
& \qquad |\omega \cdot \ell  + 
 \mu_j^\infty (\om, \kappa) - \varsigma \mu_{j'}^\infty (\om, \kappa) | \geq 
 \frac{4 \gamma | j^{\frac32} - \varsigma j'^{\frac32}|}{ \langle \ell  \rangle^{\tau}},      \\
 &  \qquad     \forall \ell   \in \Z^\nu,j, j' \in \N^+ \setminus {\mathbb S}^+,\varsigma = \pm 1 \quad \text{$(2$-Melnikov conditions$)$}
 \Big\} 
\end{aligned}
\end{equation}
the function $ i_\infty (\vphi) := i_\infty (\omega, \kappa,  \e)(\vphi) $ is a solution of 
$ {\mathcal F}( i_\infty, \a_\infty (\om, \kappa) , \om, \kappa, \e)  = 0 $. As a consequence the embedded torus 
$ \vphi \mapsto i_\infty (\vphi) $ is invariant for the Hamiltonian vector field $ X_{H_{\a_\infty (\om, \kappa)}  } $  
and it is filled by quasi-periodic solutions with frequency $ \om $.
\end{theorem}

Note that  the Borel set $  {\mathcal C}_\infty^\gamma $ in \eqref{Cantor set infinito riccardo} 
 for which a solution exists is  defined only in terms of the ``final" solution $ i_\infty $ and the 
 ``final" normal perturbed frequencies $  \mu_j^\infty $, $ j \in \N^+ \setminus {\mathbb S}^+ $. In Theorem 
 \ref{MAINTHEOREM} 
  we are not concerned about the measure of  $ {\mathcal C}_\infty^\gamma $, in particular in investigating if it is not empty
 (note that $ \a_\infty $, $ i_\infty $ and each $\mu_j^\infty $ are anyway defined for all $ (\om, \kappa) \in \tOm \times [\kappa_1, \kappa_2] $).

\section{Measure estimates}\label{sec:measure}

By  \eqref{mappa aep}, for any $\kappa \in [\kappa_1, \kappa_2]$,  the function $ \a_\infty(\cdot, \kappa) $ from $ \tOm $ into the image $\a_\infty( \tOm \times \{\kappa \})  $ is invertible:
\be\label{a-b}
\begin{aligned}
& \beta = \a_\infty (\om, \kappa) = \om + r_\e (\om, \kappa)  \quad \Longleftrightarrow \quad \\ 
& \om  = \a_\infty^{-1}(\b , \kappa ) = \beta + {\tilde r}_\e (\beta, \kappa) \quad {\rm with }  \quad
 |{\tilde r}_\e |^{k_0, \gamma} \leq C \e \g^{-(1 + k_1)} \, . 
 \end{aligned}
\ee
We underline that the function $\alpha_\infty^{- 1}(\cdot, \kappa)$ is the inverse of $\alpha_\infty(\cdot, \kappa)$, at any fixed value of $\kappa$ in $[\kappa_1, \kappa_2]$. 
{\sc Proof of \eqref{a-b}.}The inverse map $\beta \mapsto \alpha_\infty^{- 1}(\beta, \kappa) = \beta + {\tilde r}_\e (\beta, \kappa)$ satisfies the identities  $\tilde r_\e(\beta, \kappa ) + r_\e(\beta + \tilde r_\e(\beta, \kappa), \kappa) = 0$. By the implicit function theorem $\tilde r_\e$ is ${\mathcal C}^1$ with respect to $(\beta, \kappa)$ and it satisfies the identities
\begin{align*}
D_\beta \tilde r_\e(\beta, \kappa) & = 
- \big({\rm Id} + D_\omega r_\e(\beta + \tilde r_\e(\beta, \kappa), \kappa) \big)^{- 1} 
D_\omega r_\e(\beta + \tilde r_\e(\beta, \kappa), \kappa)\,, \\
\partial_\kappa \tilde r_\e(\beta, \kappa) & = - \big({\rm Id} + D_\omega r_\e(\beta + \tilde r_\e(\beta, \kappa), \kappa) \big)^{- 1} \partial_\kappa r_\e (\beta + \tilde r_\e(\beta, \kappa) , \kappa)
\end{align*}
where $D_\omega$, $D_\beta$ denote the Fr\'echet derivatives with respect to the variables $\omega$ and $\beta$.
Arguing by induction on $|k| \leq k_0$, $\tilde r_\e$ is $k_0$-times differentiable and the estimate \eqref{a-b} follows as the estimate \eqref{claim p q}.

\noindent
Then, for any $ \b \in \a_\infty ({\mathcal C}_\infty^\gamma) $, Theorem \ref{MAINTHEOREM} proves the existence of
 an embedded invariant torus  filled by quasi-periodic solutions with diophantine frequency 
$ \om =  \a_\infty^{-1}(\b , \kappa )  $ for the  Hamiltonian 
$$
H_\b =  \beta \cdot I + \frac12 ( z, \Om z)_{L^2_x} + \e P \, . 
$$
Consider the curve of the unperturbed linear frequencies 
$$ 
[\kappa_1, \kappa_2] \ni \kappa \mapsto \vec \om (\kappa ) := ( \sqrt{j(1+ \kappa j^2) } )_{j \in {\mathbb S}^+} \in \R^\nu \, . 
$$
In Theorem \ref{Teorema stima in misura} below, we prove that for ``most"  values of $ \kappa \in [\kappa_1, \kappa_2] $ 
the vector $(\alpha_\infty^{- 1}(\vec \om (\kappa ), \kappa), \kappa)$ is in ${\mathcal C}^\gamma_\infty$. Hence, for such 
values of $ \kappa $ we have found 
 an embedded invariant torus for the Hamiltonian 
$ H_\e $ in \eqref{definizione cal N P}, 
filled by quasi-periodic solutions with diophantine frequency $ \om =  \a_\infty^{-1}( \vec \om (\kappa ), \kappa ) $.
This implies Theorem \ref{thm:main0}.

\begin{theorem}\label{Teorema stima in misura} {\bf (Measure estimates)} 
Let
\begin{equation}\label{relazione tau k0}
 \gamma = \e^a \, , \ \ 0 < a < \min\{a_0, 1 / (1 + k_0 + k_1) \} \, , \quad \tau > k_0 ( \nu + 4) \,.
\end{equation} 
Then the measure of the set 
\be\label{defG-ep}
{\mathcal G}_\e := 
\big\{ \kappa \in [\kappa_1, \kappa_2] \, :  \big( \a_\infty^{-1} ( {\vec \om} (\kappa ), \kappa), \kappa \big) \in  {\mathcal C}^\gamma_\infty   \big\} 
\ee
satisfies  $ |{\mathcal G}_\e| \geq (\kappa_2 - \kappa_1) - C \e^{a/k_0} $ as $ \e \to 0 $.
\end{theorem}

Theorems \ref{MAINTHEOREM}-\ref{Teorema stima in misura} prove
Theorem \ref{thm:main0}  with the Borel set $ {\mathcal G} := {\mathcal G}_\e $ defined in \eqref{defG-ep} 
and frequency vector $ \tilde \om = \om_\e (\kappa ) $
defined in \eqref{omega epsilon kappa} below. 

\smallskip

The rest of this section is devoted to the proof of Theorem \ref{Teorema stima in misura}. 
By \eqref{a-b}
the vector
\begin{equation}\label{omega epsilon kappa}
\om_\e (\kappa) :=    \a_\infty^{-1}( {\vec \om} (\kappa ), \kappa  ) = {\vec \om} (\kappa) + {\mathtt r}_\e ( \kappa ) \, , \quad 
{\mathtt r}_\e ( \kappa ) := {\tilde r}_\e ({\vec \om } (\kappa), \kappa ) \, , 
\end{equation}
satisfies 
\begin{equation}\label{stima omega epsilon kappa}
| \pa_\kappa^k {\mathtt r}_\e  (\kappa) | \leq C \e \g^{-(1 + k_1 + k)} \, , \ \forall 0 \leq k \leq k_0 \, .
\end{equation}
We also denote, with a small abuse of notation,
\begin{equation}\label{mu j infty kappa}
\begin{aligned}
\mu_j^\infty (\kappa) & := \mu_j^\infty ( \om_\e(\kappa), \kappa) := \mathtt m_3^\infty (\kappa) j^\frac12 (1 + \kappa j^2)^{\frac12} +  \mathtt m_1^\infty (\kappa)  j^{\frac12} + r_j^\infty (\kappa) \, ,  \\
& \qquad  \forall  j \in \N^+ \setminus {\mathbb S}^+ \, , 
\end{aligned}
\end{equation}
where 
\begin{equation}\label{autovalori in kappa}
 \mathtt m_3^\infty (\kappa)  := \mathtt m_3^\infty(\omega_\e(\kappa), \kappa)\,,  \  \mathtt m_1^\infty(\kappa) := \mathtt m_1^\infty(\omega_\e(\kappa), \kappa)\,,  \
  r_j^\infty(\kappa) := r_j^\infty(\omega_\e(\kappa), \kappa)\,.
\end{equation}
By \eqref{stime autovalori infiniti}, \eqref{autovalori in kappa} and 
\eqref{omega epsilon kappa}, using that $ \e \gamma^{- ( 1 + k_1 + k_0)} \leq 1 $ 
(that by  \eqref{relazione tau k0} is satisfied for $ \e $ small), we get
\begin{equation}\label{stime coefficienti autovalori in kappa}
\begin{aligned}
& | \partial_\kappa^{k}[\mathtt m_3^\infty(\kappa) - 1]|\,,  
|\pa_\kappa^k \mathtt m_1^\infty(\kappa)| \leq C \e \gamma^{- k}\,, \quad {\rm sup}_{j \in {\mathbb S}^c} 
|\pa_\kappa^k r_j^\infty (\kappa)| \leq C \e \gamma^{- (k + k_1)}\,, \\
&  \  \forall 0 \leq k \leq k_0\,.
\end{aligned}
\end{equation}
 By \eqref{Cantor set infinito riccardo}, \eqref{omega epsilon kappa}, \eqref{mu j infty kappa} the set
$ {\mathcal G}_\e $  in \eqref{defG-ep} writes 
\begin{align*}
{\mathcal G}_\e  & = \Big\{  \kappa \in [\kappa_1, \kappa_2]  :  
|\om_\e(\kappa) \cdot \ell  | \geq \g \langle \ell \rangle^{-\tau} ,  \forall \ell \in \Z^{\nu}  \setminus  \{ 0 \}, \\
& \qquad \quad  |\omega_\e(\kappa) \cdot \ell  + \mu_j^\infty (\kappa)  | \geq 4 \gamma j^{\frac32} \langle \ell  \rangle^{- \tau}\!, 
 \forall \ell  \in \Z^\nu \!, j \in \N^+  \! \setminus \! {\mathbb S}^+,  \\
& \qquad \ \  \ |\omega_\e(\kappa) \cdot \ell  + 
 \mu_j^\infty (\kappa) - \varsigma \mu_{j'}^\infty (\kappa) | \geq 4 \gamma 
 | j^{\frac32} - \varsigma j'^{\frac32}| \langle \ell  \rangle^{-\tau}, \\ 
 & \qquad \quad \forall \ell   \in \Z^\nu, j, j' \in \N^+ \setminus {\mathbb S}^+, \varsigma \in \{ +, -\}   \Big\}\,. 
\end{align*}
We estimate the measure of the complementary set  
\begin{equation}\label{complementare insieme di cantor}
\begin{aligned}
{\mathcal G}_\e^c &  := [\kappa_1, \kappa_2] \setminus {\mathcal G}_\e \\
& := \Big(\bigcup_{\ell} R_{\ell}^{(0)} \Big) \bigcup \Big(\bigcup_{\ell , j} R_{\ell, j}^{(I)}  \Big)\bigcup  \Big(\bigcup_{\ell, j, j'} R_{\ell j j'}^{(II)} \Big) \bigcup  \Big(\bigcup_{\ell, j, j'} Q_{\ell j j'}^{(II)} \Big)
\end{aligned}
\end{equation}
where the ``resonant sets" are 
\begin{align}
& R_\ell^{(0)} := \big\{ \kappa \in [\kappa_1, \kappa_2] : |\omega_\e (\kappa) \cdot \ell| <  4 \gamma \langle \ell \rangle^{- \tau} \big\} 
\label{reso1} \\
& R_{\ell j}^{(I)} := \big\{ \kappa \in [\kappa_1, \kappa_2] : |\omega_\e (\kappa) \cdot \ell + \mu_j^\infty(\kappa)| < 4 \gamma j^{\frac32}\langle \ell \rangle^{- \tau} \big\} \label{reso2} \\
& 
R_{\ell j j'}^{(II)} := \big\{ \kappa \in [\kappa_1, \kappa_2] : |\omega_\e (\kappa) \cdot \ell + \mu_j^\infty(\kappa) - \mu_{j'}^\infty(\kappa)| < 4 \gamma |j^{\frac32} - j'^{\frac32}| \langle \ell \rangle^{- \tau} \big\}  \label{reso3} \\
& Q_{\ell j j'}^{(II)} := \big\{ \kappa \in [\kappa_1, \kappa_2] : |\omega_\e (\kappa) \cdot \ell + \mu_j^\infty(\kappa) + \mu_{j'}^\infty(\kappa)| < 4 \gamma |j^{\frac32} + j'^{\frac32}| \langle \ell \rangle^{- \tau} \big\}\,. \label{reso4}
\end{align}

\begin{lemma}\label{restrizione indici risonanti}
If $R_{\ell j}^{(I)} \neq \emptyset$ then $j^{\frac32} \leq C \langle \ell \rangle$. 
If $ R_{\ell j j'}^{(II)} \neq \emptyset$ then $|j^{\frac32} - j'^{\frac32}| \leq C \langle \ell \rangle$.
 If $Q_{\ell j j'}^{(II)} \neq \emptyset$ then $j^{\frac32} + j'^{\frac32} \leq C \langle \ell \rangle$.
\end{lemma}

\begin{proof}
We prove the lemma for $ R_{\ell j j'}^{(II)} $ . The other cases follow similarly. 
If $\kappa \in R^{(II)}_{\ell j j'}$ then 
\begin{equation}\label{santiago 0}
|\mu_j^\infty(\kappa) - \mu_{j'}^\infty(\kappa)| <
4 \gamma |j^{\frac32} - j'^{\frac32}| \langle \ell \rangle^{- \tau} + |\omega_\e(\kappa)| |\ell|  
\leq 4 \gamma |j^{\frac32} - j'^{\frac32}| + C |\ell|\,.
\end{equation}
Moreover \eqref{mu j infty kappa}  and \eqref{stime coefficienti autovalori in kappa} imply 
\begin{align}
|\mu_j^\infty - \mu_{j'}^\infty| & \geq |\mathtt m_3^\infty(\kappa)| |j^{\frac12} (1 + \kappa j^2)^{\frac12} - j'^{\frac12} (1 + \kappa j'^2)^{\frac12}|  \nonumber\\
& \quad - |\mathtt m_1^\infty(\kappa)| |j^{\frac12} - j'^{\frac12}| - 2 {\rm sup}_{j \in {\mathbb S}^c} |r_j^\infty(\kappa)| \nonumber\\
& \geq C_1 |j^{\frac32} - j'^{\frac32}| - 
C \e |j^{\frac12} - j'^{\frac12}| - C \e \gamma^{- k_1} \geq C_1 |j^{\frac32} - j'^{\frac32}| / 2 \, \label{santiago 1}
\end{align}
for $2 C \e \gamma^{- k_1} \leq C_1/2 $, which is fulfilled taking $\e$ small enough by \eqref{relazione tau k0}.   
The lemma follows by \eqref{santiago 0}, \eqref{santiago 1}, 
for $ C_1 / 4 \geq  4 \gamma  $. 
\end{proof}

The perturbed frequencies satisfy estimates  similar to 
\eqref{0 Melnikov}-\eqref{2 Melnikov+} 
 in Proposition \ref{Lemma: degenerate KAM}. 

\begin{lemma}\label{Lemma: degenerate KAM perturbato}
For $ \e $ small enough, for all $ \kappa \in [\kappa_1, \kappa_2] $, 
\begin{equation}\label{0 Melnikov perturbate}
\begin{aligned}
& {\rm max}_{k \leq k_0} |\partial_\kappa^{k}  \{\omega_\e (\kappa) \cdot \ell   \} |  
 \geq \rho_0 \langle \ell \rangle / 2 \,, \quad   \forall \ell  \in \Z^\nu \setminus \{ 0 \}, \,  \\
\end{aligned}
\end{equation}
\begin{equation}\label{1 Melnikov perturbate}
\begin{aligned}
& {\rm max}_{k \leq k_0} |\partial_\kappa^{k}  \{\omega_\e (\kappa) \cdot \ell  + \mu_j^\infty (\kappa) \} | 
  \geq \rho_0 \langle \ell  \rangle / 2 \,,   \quad \forall \ell  \in \Z^\nu, \, j \in  \N^+ \setminus {\mathbb S}^+,
\end{aligned}
\end{equation}
\begin{equation}\label{2 Melnikov- perturbate}
\begin{aligned}
& {\rm max}_{k \leq k_0} |\partial_\kappa^{k}  \{\omega_\e (\kappa) \cdot \ell  + \mu_j^\infty (\kappa) - \mu_{j'}^\infty(\kappa) \} |  \geq \rho_0 \langle \ell  \rangle / 2 \,,  \\
&  \quad \forall (\ell , j, j') \neq (0, j, j), \quad \ell \in \Z^\nu, \quad  \,  j, j' \in  \N^+ \setminus {\mathbb S}^+, 
\end{aligned}
\end{equation}
\begin{equation}\label{2 Melnikov+ perturbate}
\begin{aligned}
& {\rm max}_{k \leq k_0} |\partial_\kappa^{k}  \{\omega_\e (\kappa) \cdot \ell 
 + \mu_j^\infty (\kappa) + \mu_{j'}^\infty(\kappa) \} |  
 \geq \rho_0 \langle \ell  \rangle / 2 \,,  \\
 & \quad   \forall \ell  \in \Z^{\nu}, \quad  j, j' \in \N^+ \setminus {\mathbb S}^+ \, . 
\end{aligned}
\end{equation}
\end{lemma}

\begin{proof}
We prove  \eqref{2 Melnikov- perturbate}. The other estimates follow analogously. First of all, by Lemma \ref{restrizione indici risonanti} we may restrict to the set of indices satisfying
\begin{equation}\label{cardiff 0}
|j^{\frac32} - j'^{\frac32}| \leq C \langle \ell \rangle\,.
\end{equation}
Split  $ \mu_j^\infty (\kappa) = { \Omega}_j(\kappa) + (\mu_j^\infty - { \Omega}_j)(\kappa)$
where $ { \Omega}_j(\kappa) := j^\frac12 (1 + \kappa j^2)^{\frac12} $. 
A direct calculation shows that 
\begin{equation}\label{derivate Omega j - Omega j'}
| \pa_\kappa^k \{ { \Omega}_j(\kappa) - { \Omega}_{j'}(\kappa) \}| \leq C_k |j^{\frac32} - j'^{\frac32}|\,, \quad \forall \, k \geq 0\, .
\end{equation}
Then, for all $0 \leq k \leq k_0 $,  one has 
\begin{align}
|\pa_\kappa^k \big\{(\mu_j^\infty - \mu_{j'}^\infty)(\kappa) - ( { \Omega}_j - { \Omega}_{j'})(\kappa) \big\}| & \leq 
|\pa_\kappa^k \big\{ (\mathtt m_3^\infty(\kappa) - 1) ({ \Omega}_j(\kappa) - { \Omega}_{j'}(\kappa) \big)| \nonumber\\
& \quad + | \pa_\kappa^k \mathtt m_1^\infty(\kappa)| |j^{\frac12} - j'^{\frac12}| \nonumber\\
& \quad + 2 {\rm sup}_{j \in  \N^+ \setminus {\mathbb S}^+}| \pa_\kappa^k r_j^\infty(\kappa)| \nonumber\\
& \stackrel{\eqref{derivate Omega j - Omega j'}, \eqref{stime coefficienti autovalori in kappa}}{\leq} C \e \gamma^{- (k + k_1)} |j^{\frac32} - j'^{\frac32}|\,. \label{mu j - mu j' infty}
\end{align}
By \eqref{omega epsilon kappa}, \eqref{stima omega epsilon kappa} and \eqref{mu j - mu j' infty} we get 
\begin{align}
& {\rm max}_{k \leq k_0} | \pa_\kappa^k \{\omega_\e (\kappa) \cdot \ell + \mu_j^\infty (\kappa) - \mu_{j'}^\infty(\kappa)\}|   \nonumber\\
& \geq  {\rm max}_{k \leq k_0} | \pa_\kappa^k \{{\vec \om}(\kappa) \cdot \ell + { \Omega}_j(\kappa) 
-  { \Omega}_{j'}(\kappa) \}| \nonumber\\
& - C \e \gamma^{- (1 +k_0 + k_1)} |\ell|   - C \e \gamma^{- (k_0 + k_1)} |j^{\frac32} - j'^{\frac32}| \nonumber\\
& \stackrel{\eqref{cardiff 0}}{\geq} {\rm max}_{k \leq k_0} | \pa_\kappa^k \{{\vec \om}(\kappa) \cdot \ell + 
{ \Omega}_j(\kappa) - { \Omega}_{j'}(\kappa) \}| \nonumber\\
& \quad - C \e \gamma^{- (1 + k_0 + k_1)} \langle \ell \rangle \nonumber\\
& \stackrel{\eqref{2 Melnikov-}}{\geq} 
\rho_0 \langle \ell \rangle - C \e \gamma^{- (1 + k_0 + k_1)} \langle \ell \rangle \geq \rho_0 \langle \ell \rangle / 2 \nonumber 
\end{align}
provided $ \e \gamma^{- (1 + k_0 + k_1)} \leq \rho_0 \slash (2 C) $, that, by  \eqref{relazione tau k0}, is satisfied for $ \e $ small. 
\end{proof}

\begin{lemma}[{\bf Estimates of the resonant sets}]\label{stima risonanti Russman}
The measures of the sets in \eqref{reso1}-\eqref{reso4} satisfy 
\begin{align*}
& |R_{\ell}^{(0)}| \lessdot  \big(  \gamma \langle \ell \rangle^{- (\tau + 1)}\big)^{\frac{1}{k_0}} \,, 
|R_{\ell j}^{(I)}| \lessdot   \big(  \gamma j^{\frac32} \langle \ell \rangle^{- (\tau + 1)}\big)^{\frac{1}{k_0}} \, ,\\
&  
|R_{\ell j j'}^{(II)}| \lessdot   \big(  \gamma |j^{\frac32} - j'^{\frac32}| \langle \ell \rangle^{- (\tau + 1)}\big)^{\frac{1}{k_0}}, \
|Q_{\ell j j'}^{(II)}| \lessdot   \big(  \gamma |j^{\frac32} + j'^{\frac32}| \langle \ell \rangle^{- (\tau + 1)}\big)^{\frac{1}{k_0}} \, . 
\end{align*}
\end{lemma}
\begin{proof}
We prove the estimate of $ R_{\ell j j'}^{(II)} $. 
The other cases are simpler. 
We write 
$$
R_{\ell j j'}^{(II)} = \big\{ \kappa \in [\kappa_1, \kappa_2] : |g_{\ell j j'}(\kappa)| < 4 \gamma |j^{\frac32} - j'^{\frac32}| \langle \ell \rangle^{- (\tau + 1)}\big\}
$$
where
$ g_{\ell j j'}(\kappa) := ( \omega_\e(\kappa) \cdot \ell + \mu_j^\infty(\kappa) - \mu_{j'}^\infty(\kappa) ) 
\langle \ell\rangle^{-1} $. We apply Theorem 17.1 in \cite{Ru1}.  
We estimate the measure of  
$ R_{\ell j j'}^{(II)} $ only if 
$ 4 \gamma |j^{\frac32} - j'^{\frac32}| \langle \ell \rangle^{- (\tau + 1)}  \leq \frac{\rho_0}{4(1 + k_0)} $.
Otherwise, for  $\gamma$ small enough,  the set $ R_{\ell j j'}^{(II)} = \emptyset  $ is empty.  
By \eqref{2 Melnikov- perturbate} we derive that  
$$
{\rm max}_{k \leq k_0} | \pa_\kappa^k 
g_{\ell j j'}(\kappa)| \geq \rho_0 / 2 \,, \quad \forall \kappa \in [\kappa_1, \kappa_2]\,.
$$
In addition,  \eqref{omega epsilon kappa}-\eqref{autovalori in kappa} and Lemma \ref{restrizione indici risonanti} imply that
$ \max_{k \leq k_0} | \pa_\kappa^k g_{\ell j j'}(\kappa) | \leq C_1 $, 
$  \forall  \kappa \in [\kappa_1, \kappa_2] $, provided $\e \gamma^{- (1 + k_0 + k_1)}$ is small enough.
By Theorem 17.1 in \cite{Ru1}  the Lemma follows. 
\end{proof}

\noindent 
{\sc Proof of Theorem \ref{Teorema stima in misura} completed.}
The measure of the set ${\mathcal G}_\e^c$ in \eqref{complementare insieme di cantor} is estimated by
\begin{align}
|{\mathcal G}_\e^c| & \leq 
{\mathop \sum}_{\ell} |R_\ell^{(0)}| + {\mathop \sum}_{\ell, j} |R_{\ell j}^{(I)}| 
+ {\mathop \sum}_{\ell , j, j'} |R_{\ell j j'}^{(II)}| + 
{\mathop \sum}_{\ell , j, j'} |Q_{\ell j j'}^{(II)}| \nonumber\\
& \stackrel{{\text Lemma} \,\ref{restrizione indici risonanti}}\leq 
{\mathop \sum}_{\ell} |R_\ell^{(0)}| + {\mathop \sum}_{j \leq C \langle \ell \rangle^{2/3}} |R_{\ell j}^{(I)}| \nonumber \\
& \qquad \quad \ + {\mathop \sum}_{j, j' \leq C \langle \ell \rangle^2} |R_{\ell j j'}^{(II)}| + 
{\mathop \sum}_{ j, j' \leq C \langle \ell \rangle^{2/3}} |Q_{\ell j j'}^{(II)}| \nonumber\\
& \stackrel{{\text Lemma}\, \ref{stima risonanti Russman}}{\lessdot} 
\sum_{\ell} \big( \gamma \langle \ell \rangle^{- (\tau + 1)}\big)^{\frac{1}{k_0}} + \sum_{j \leq C \langle \ell \rangle^{2/3}} \big(  \gamma j^{\frac32} \langle \ell \rangle^{- (\tau + 1)}\big)^{\frac{1}{k_0}}  \nonumber \\ 
& \quad + \! \!\! \sum_{j, j' \leq C \langle \ell \rangle^2} \!\! \! \big(  \gamma |j^{\frac32} - j'^{\frac32}| \langle \ell \rangle^{- (\tau + 1)}\big)^{\frac{1}{k_0}} + \!\!  \sum_{ j, j' \leq C \langle \ell \rangle^{2/3}} \!\! \big( \gamma |j^{\frac32} + j'^{\frac32}| \langle \ell \rangle^{- (\tau + 1)}\big)^{\frac{1}{k_0}} \nonumber \\
& \stackrel{{\text Lemma} \,\ref{restrizione indici risonanti}} \leq C \gamma^{\frac{1}{k_0}} \sum_{\ell} \langle \ell \rangle^{4 - \frac{\tau}{k_0}} \stackrel{\eqref{relazione tau k0}}{\leq} C' \e^{\frac{a}{k_0}} \,. \nonumber
\end{align}
Hence $ |{\mathcal G}_\e| \geq \kappa_2 - \kappa_1 - C' \e^{a / k_0 } $ 
and the proof of Theorem \ref{Teorema stima in misura} is concluded. 

\chapter{Approximate inverse}  \label{costruzione dell'inverso approssimato}

\section{Estimates on the perturbation $P$}

We prove tame estimates for the composition operator induced
 by the Hamiltonian vector field $ X_P = (  \pa_I P, - \pa_\theta P, J \nabla_z P) $ 
 in \eqref{operatorF}. 

We first estimate the composition operator induced by $ v(\teta, y) $  defined in \eqref{definizione A}.
Since the functions $ I_j  \mapsto \sqrt{\xi_j + I_j}  $, $\theta \mapsto {\rm cos}(\theta)$, $\theta \mapsto {\rm sin}(\theta)$ 
are analytic  for $|I| \leq r $ small, 
the composition Lemma \ref{Moser norme pesate} implies that, for all  $ \Theta, y \in H^s(\T^\nu, \R^\nu )$,  
$   \| \Theta \|_{s_0},  \| y \|_{s_0} \leq r $, setting $\theta(\ph) := \ph + \Theta (\ph)$, 
\begin{equation}\label{stima k0 gamma v per il campo}
\| \partial_\theta^\alpha \partial_I^\beta v(\theta(\cdot), I(\cdot)) \|_s^{k_0, \gamma} \leq_s 1 + \| \fracchi\|_{s}^{k_0, \gamma}\,, \quad \forall \alpha, \beta \in \N^\nu\,, \ \  |\alpha| + |\beta| \leq 3\,.
\end{equation}

\begin{lemma}\label{lemma quantitativo forma normale}
Let $ \fracchi(\ph) $ in \eqref{componente periodica} satisfy
$ \| {\mathfrak I} \|_{2s_0 + 2 k_0 + 5}^{k_0, \gamma}  \leq 1$.
Then 
the following estimates hold:
\begin{equation}\label{stime XP}
\| X_P(i)\|_s^{k_0, \gamma}  \leq_s 1 + \| {\mathfrak I}\|_{s +s_0 + 2 k_0 + 3}^{k_0, \gamma}\,, 
\end{equation}
and for all $\widehat \imath := (\widehat \theta, \widehat I, \widehat z)$ 
\begin{align}\label{stima derivata XP}
 \| d_i X_P(i)[\widehat \imath]\|_s^{k_0, \gamma} & \leq_s \| \widehat \imath \|_{s + 2}^{k_0, \gamma} + \| \mathfrak I\|_{s + s_0 + 2 k_0 + 4}^{k_0, \gamma} \| \widehat \imath \|_{s_0 + 2}^{k_0, \gamma}\,, \\
 \label{stima derivata seconda XP}
\| d^2_i X_P(i)[\widehat \imath, \widehat \imath]\|_s^{k_0, \gamma} & \leq_s \| \widehat \imath\|_{s + 2}^{k_0, \gamma} \| \widehat \imath \|_{s_0 + 2}^{k_0, \gamma} + \| \mathfrak I\|_{s + s_0 + 2 k_0 + 5}^{k_0, \gamma} (\| \widehat \imath \|_{s_0 + 2}^{k_0, \gamma})^2\,.
\end{align}
\end{lemma}
\begin{proof}
By the definition \eqref{definizione cal N P}, $P = P_\e \circ A$, where $A$ is defined in \eqref{definizione A} and $P_\e$ is defined in \eqref{definizione P epsilon}. Hence 
\begin{equation}\label{campo hamiltoniano perturbazione P}
\begin{aligned}
& X_P =   \begin{pmatrix} [\partial_I v(\theta, I)]^T \nabla P_\e(A(\theta, I, z)) \\
 - [\partial_\theta v(\theta, I)]^T \nabla P_\e(A(\theta, I, z)) \\
  \Pi_{{\mathbb S}^+}^\bot J \nabla P_\e(A(\theta, I, z))  
  \end{pmatrix}\,
\end{aligned}
\end{equation}
where $\Pi_{{\mathbb S}_+^\bot}$ is the $L^2$-projector on the space $H_{{\mathbb S}_+}^\bot$ defined in \eqref{H S + H S + bot}.
Now  $ \nabla P_\e = - J X_{P_\e} $ (see \eqref{WW-riscalato}) where 
$  X_{P_\e} $ is the explicit Hamiltonian vector field in  \eqref{campo hamiltoniano X P epsilon}. 
The  smallness condition of Lemma \ref{stime tame dirichlet neumann} is fulfilled because
$$
\begin{aligned}
\| \eta\|_{2 s_0 + 2 k_0 + 5}^{k_0, \gamma} \leq \e \| A(\theta(\cdot), I(\cdot), z(\cdot, \cdot)) \|_{2 s_0 + 2 k_0 + 5}^{k_0, \gamma} 
& \leq C(s_0) \e (1 + \| \fracchi \|_{2 s_0 + 2 k_0 + 5}^{k_0, \gamma}) \\ 
& \leq C_1(s_0 ) \e  \leq \delta(s_0, k_0) 
\end{aligned}
$$
for $ \e $ small.
Thus by the tame estimate
 \eqref{stima tame dirichlet neumann} for the Dirichlet Neumann operator, 
 the interpolation inequality \eqref{interpolazione C k0},  and \eqref{stima k0 gamma v per il campo},  we get 
$$
\begin{aligned}
\| \nabla P_\e(A(\theta(\cdot), I(\cdot), z(\cdot, \cdot))) \|_s^{k_0, \gamma} 
& \leq_s  \| A(\theta(\cdot), I(\cdot), z(\cdot, \cdot))\|_{s + s_0 + 2 k_0 + 3}^{k_0, \gamma} \\
& \leq_s 1 + \| \fracchi\|_{s + s_0 + 2 k_0 + 3}^{k_0, \gamma}\, .
\end{aligned}
$$ 
Hence  \eqref{stime XP} 
follows by \eqref{campo hamiltoniano perturbazione P},
 interpolation and  \eqref{stima k0 gamma v per il campo}. 

 The estimates \eqref{stima derivata XP}, \eqref{stima derivata seconda XP} for $d_i X_P$ and $d_i^2 X_P$ follow
 by differentiating the expression of $X_P$ in \eqref{campo hamiltoniano perturbazione P} and applying the estimates \eqref{stima tame derivata dirichlet neumann}, \eqref{stima tame derivata seconda dirichlet neumann} on the Dirichlet Neumann operator, the estimate \eqref{stima k0 gamma v per il campo} on $v(\theta, y)$ and using the interpolation inequality \eqref{interpolazione C k0}. 
\end{proof}

\section{Almost approximate inverse}\label{sezione almost approximate inverse}

In order to implement a convergent Nash-Moser scheme that leads to a solution of 
$ \mF(i, \alpha) = 0 $ (the operator $ \mF(i, \alpha) $ is defined in \eqref{operatorF}) 
we 
linearize the nonlinear operator $ \mF(i, \alpha) $   at an arbitrary  torus 
$$
 i_0 (\vphi) = (\theta_0 (\vphi) , I_0 (\vphi), z_0 (\vphi) ) \, , 
$$
at a given value of $ \alpha_0 $,   obtaining 
$$
d_{i, \alpha} {\mathcal F}(i_0, \alpha_0 )[\widehat \imath \,, \widehat \alpha ] =
\Dom \widehat \imath - d_i X_{H_\alpha} ( i_0 (\vphi) ) [\widehat \imath ] - (\widehat \alpha,0, 0 ) \, .
$$
Note that 
$ d_{i, \alpha} {\mathcal F}(i_0, \alpha_0 ) = d_{i, \alpha} {\mathcal F}(i_0 ) $ is independent of $ \alpha_0 $, see \eqref{operatorF} and recall 
that the perturbation  $ P $ in \eqref{definizione cal N P} does not depend on $ \alpha $ (it depends on $ \kappa $).
In accordance with the notation introduced in \eqref{componente periodica} we denote by 
$$
{\mathfrak I}_0 (\vphi)  := i_0 (\vphi) - (\vphi,0,0) := 
( {\Theta}_0 (\ph), I_0(\ph), z_0(\ph))\,, \quad \Theta_0 (\ph) := \teta_0 (\vphi) - \vphi \, ,  
$$
the periodic component of the torus $ \vphi \mapsto i_0 (\vphi) $. 
In sections \ref{costruzione dell'inverso approssimato}-\ref{sec: reducibility} the torus $ i_0  $ and $ {\fracchi }_0 $ are fixed, satisfying the properties \eqref{ansatz 0} 
of the ansatz below. The main result of these sections is Theorem \ref{thm:stima inverso approssimato}
where we construct an \emph{almost-approximate right inverse} of $ d_{i, \alpha} {\mathcal F}(i_0, \alpha_0 ) $. 

In section  \ref{sec:NM} we shall apply Theorem \ref{thm:stima inverso approssimato} for obtaining the
invertibility of the linearized operators 
when $ i_0 $ is replaced by an arbitrary approximate torus
obtained  by the Nash-Moser iteration scheme. In section  \ref{sec:NM} we  shall also verify inductively that the
property \eqref{ansatz 0} is satisfied by the approximate solutions  defined by the Nash-Moser iteration. 

Let us make some comments about Theorem \ref{thm:stima inverso approssimato}. The main 
inversion assumption  \eqref{inversion assumption}-\eqref{tame inverse} 
required for the applicability of such a theorem 
(which concerns the linearized operator in the normal directions) 
is proved in sections \ref{linearizzato siti normali} and 
\ref{sec: reducibility}, see in particular Theorem \ref{inversione parziale cal L omega}.  
The reason why we call  $ {\bf T}_0 $ an ``almost-approximate" inverse of $  {\mathcal L}_\om $  is the following: the 
adjective ``approximate" refers to the presence of a remainder which is zero 
at an exact solution, i.e. when  $ {\mathcal F} (i_0, \a_0 ) = 0 $, like for example for the term \eqref{stima inverso approssimato 2}. This terminology is inspired by the notion of approximate inverse introduced by Zehnder \cite{Z1}.
The adjective ``almost" refers to the presence of  terms which are small as $ O( N_n^{-a} )$ or $ O( K_n^{-a} )$ for some $a > 0$, like \eqref{stima cal G omega} and which arise by requiring only finitely many non-resonance conditions
(of diophantine type) at each step.
  We find these words helpful to distinguish the different origin of the remainders. 

\smallskip
We  implement the general strategy proposed in \cite{BB13} and  \cite{BBM-auto}.
An invariant torus $ i_0 $ for the Hamiltonian vector field $X_{H_\alpha}$  with diophantine flow (i.e. $ \om $
satisfies \eqref{dioph})
is isotropic (see e.g. Lemma 1 in \cite{BB13}), 
namely the pull-back $ 1$-form $ i_0^* \Lambda $ is closed, 
where $ \Lambda $ is the Liouville 1-form defined in \eqref{Lambda 1 form}. 
This is tantamount to say that the 2-form 
$$ 
i_0^* {\mathcal W} =  i_0^* d \Lambda  = d i_0^* \Lambda = 0 
$$ 
where $ {\mathcal W} = d \Lambda $ is defined in \eqref{2form}. 
For an ``approximately invariant" torus $ i_0 $, which supports a linear flow which is 
only approximately diophantine, i.e. $ \om \in {\mathtt D \mathtt C}_{K_n}^{\g}  $ defined in \eqref{omega diofanteo troncato},  
the 1-form $ i_0^* \Lambda $ is only  ``approximately closed".
In order to make this statement quantitative we consider
\begin{equation}\label{coefficienti pull back di Lambda}
\begin{aligned}
& i_0^* \Lambda = {\mathop \sum}_{k = 1}^\nu a_k (\vphi) d \vphi_k \,,  \\
& a_k(\vphi) := - \big( [\pa_\ph \teta_0 (\vphi)]^T I_0 (\vphi)  \big)_k 
- \frac12 ( \partial_{\vphi_k} z_0(\ph), J z_0(\ph) )_{L^2(\T_x)}
\end{aligned}
\end{equation}
and we quantify how small is 
\begin{equation} \label{def Akj} 
\begin{aligned}
& i_0^* {\mathcal W} = d \, i_0^* \Lambda = {\mathop\sum}_{1 \leq k < j \leq \nu} A_{k j}(\vphi) d \vphi_k \wedge d \vphi_j\,, \\
& \qquad  A_{k j} (\vphi) := \partial_{\vphi_k} a_j(\ph) - \partial_{\vphi_j} a_k(\ph) \, , 
\end{aligned}
\end{equation} 
in terms of the ``error function"
\begin{equation} \label{def Zetone}
Z(\vphi) :=  (Z_1, Z_2, Z_3) (\vphi) := {\mathcal F}(i_0, \alpha_0) (\vphi) =
\om \cdot \pa_\vphi i_0(\vphi) - X_{H_{\alpha}}(i_0(\vphi), \alpha_0) \, , 
\end{equation}
and the ``ultra-violet" cut-off $ K_n = K_0^{\chi^n} $, $ \chi = 3/2 $, in \eqref{definizione Kn}, used in the definition 
\eqref{omega diofanteo troncato} of $  {\mathtt D \mathtt C}_{K_n}^{\g} $.
The main difference with respect to \cite{BB13} and  \cite{BBM-auto} is that we do not assume $ \om $ to be diophantine 
(i.e. \eqref{dioph}) but 
only  $ \om \in  {\mathtt D \mathtt C}_{K_n}^{\g} $. 

Along this section we will always assume the following hypothesis, 
which will be verified at each step of the Nash-Moser iteration of section \ref{sec:NM}:

\begin{itemize}
\item {\sc Ansatz.} 
The map $(\omega, \kappa) \mapsto \fracchi_0(\omega, \kappa) :=  i_0(\ph; \om, \kappa) - (\ph,0,0) $ is ${k_0} $-times differentiable with respect 
to the parameters $(\omega, \kappa) \in \R^\nu \times [\kappa_1, \kappa_2] $,  and for some $ \mu := \mu (\t, \nu) >  0 $,  $\gamma \in (0, 1)$,   
\begin{equation}\label{ansatz 0}
\| {\mathfrak I}_0  \|_{s_0+\mu}^{k_0, \gamma} +  |\alpha_0 - \omega|^{k_0, \gamma} \leq C\e \gamma^{-(1 + k_1)} \, ,
\end{equation}
where the constant $ k_1 = k_1(\nu, k_0) > 0$ is given in Theorem \ref{MAINTHEOREM}.
We shall always assume $ \e \g^{-(1 + k_1)} $ small enough
(in section \ref{sec:measure}  we have even required  the stronger condition $ \e \g^{- (1 + k_0 + k_1)} \ll 1 $).
\end{itemize}

We suppose that the torus $ i_0 (\omega, \kappa) $ is defined for all the values of  $ (\om, \kappa) \in \R^\nu \times [\kappa_1, \kappa_2]$ because, in the Nash-Moser
iteration of section \ref{sec:NM}, we construct a $ k_0 $-times differentiable extension of each approximate solution
on the whole $ \R^\nu \times [\kappa_1, \kappa_2] $, see Lemma \ref{lemma:extension torus}. 

\begin{lemma}\label{Lemma Z piccolezza}
$\| Z\|_{s}^{k_0, \gamma} \leq_s \e \gamma^{- (1 + k_1)} + \| \fracchi_0\|_{s + 2}^{k_0, \gamma} $.
\end{lemma}

\begin{proof}
By \eqref{operatorF}, \eqref{stime XP}, \eqref{ansatz 0}.
\end{proof}
In the following, we will assume that $ \om \in {\mathtt D \mathtt C}_{K_n}^{\g}$ (defined in \eqref{omega diofanteo troncato}) and 
we split the coefficients $ A_{k j} = A_{k j} (\vphi) $ in \eqref{def Akj} as 
\begin{equation}\label{splitting A kj}
A_{k j} = A_{k j}^{(n)} + A_{k j}^{(n), \bot}\,, \qquad A_{k j} := \Pi_{K_n} A_{k j}\,, \qquad  A_{k j}^{(n), \bot} := \Pi_{K_n}^\bot A_{k j}\,
\end{equation}
where $ K_n := K_0^{\chi^{n}} $, $ \chi := 3/ 2 $,  is defined in \eqref{definizione Kn}, the operator 
$\Pi_{K_n}$ is the orthogonal projection on the Fourier modes 
$|(\ell, j)| \leq K_n$ and $\Pi_{K_n}^\bot := {\rm Id} - \Pi_{K_n}$, see  \eqref{definizione smoothing operators}.
The ``ultra-violet" cut-off functions $ K_n $ are  introduced in view of the 
nonlinear Nash-Moser iteration of section \ref{sec:NM}.  

\begin{lemma} Assume that $ \om \in {\mathtt D \mathtt C}_{K_n}^{\g}$ defined in \eqref{omega diofanteo troncato}. 
Then the coefficients $ A_{kj}^{(n)} $ and $ A^{(n), \bot}_{k j}$ in \eqref{splitting A kj} satisfy the following tame estimates
\begin{align}
\| A_{k j}^{(n)} \|_s^{k_0, \gamma} & \leq_s \gamma^{-1}
\big(\| Z \|_{s+\t(k_0 + 1) + k_0+1}^{k_0, \gamma} + \| Z \|_{s_0+1}^{k_0, \gamma} \|  {\mathfrak I}_0 \|_{s+\t(k_0 + 1) + k_0+1}^{k_0, \gamma} \big)\,, \label{stima A ij} \\
\| A_{k j}^{(n), \bot}\|_s^{k_0, \gamma} &
 \leq_s  \| \fracchi_0\|_{s + 2}^{k_0, \gamma}\,,  \ \ \| A_{k j}^{(n), \bot}\|_{s_0 + \mathtt c}^{k_0, \gamma}  \leq_{s_0, b} K_n^{- b} \| \fracchi_0\|_{s_0 + b + \mathtt c}^{k_0, \gamma}\,, \ \  \forall b > 0 \, \, ,   \label{stima A kj bot}
\end{align}
and for any $\mathtt c > 0$ such that \eqref{ansatz 0} holds with $\mu \geq \mathtt \t(k_0 + 1) + k_0+1 + {\mathtt c} $. 
\end{lemma}

\begin{proof}
{\sc Proof of \eqref{stima A ij}.}
The coefficients $ A_{kj}  $  satisfy the identity (see \cite{BB13}, Lemma 5) 
$$  \omega \cdot \partial_\vphi A_{k j} 
=  {\mathcal W}\big( \pa_\ph Z(\vphi) \underline{e}_k ,  \pa_\ph i_0(\vphi)  \underline{e}_j \big) 
+  {\mathcal W} \big(\pa_\ph i_0(\vphi) \underline{e}_k , \pa_\ph Z(\vphi) \underline{e}_j \big)  
$$
where  $ \underline{e}_k  $ denote the $ k $-th versor of $ \R^\nu $. Therefore applying the projector $ \Pi_{K_n} $ we have 
$$
\omega \cdot \partial_\vphi A_{k j}^{(n)} = \Pi_{K_n} \big[ {\mathcal W}\big( \pa_\ph Z(\vphi) \underline{e}_k ,  \pa_\ph i_0(\vphi)  \underline{e}_j \big) + {\mathcal W} \big(\pa_\ph i_0(\vphi) \underline{e}_k , \pa_\ph Z(\vphi) \underline{e}_j \big) \big] \, . 
$$ 
Then by \eqref{interpolazione C k0} and \eqref{ansatz 0} we get 
\begin{equation} \label{bella trovata}
\| \omega \cdot \partial_\vphi A_{k j}^{(n)} \|_s^{k_0, \gamma} 
\leq_s \| Z \|_{s+1}^{k_0, \gamma} + \| Z \|_{s_0 + 1}^{k_0, \gamma} \| {\mathfrak I}_0 \|_{s + 1}^{k_0, \gamma}   \, 
\end{equation}
and  \eqref{stima A ij} follows  
applying $ (\om \cdot \pa_\vphi )^{-1}$, and using that, for all 
$\omega \in \mathtt{DC}_{K_n}^\gamma$ defined in \eqref{omega diofanteo troncato}, it results 
$ \| (\om \cdot \pa_\vphi)^{- 1} \Pi_{K_n} g \|_s^{k_0, \gamma} \leq_s \gamma^{- 1} \| g\|_{s + \tau(k_0 + 1) + k_0}^{k_0, \gamma} $.

\noindent
{\sc Proof of \eqref{stima A kj bot}.} Recalling \eqref{def Akj} and \eqref{splitting A kj}, 
the function 
$$
A_{kj}^{(n), \bot} (\vphi)  = \Pi_{K_n}^\bot \big( \partial_{\vphi_k} a_j(\ph) - \partial_{\vphi_j} a_k(\ph) \big)
$$ 
where $a_k (\vphi) $, $k = 1, \ldots, \nu$, are defined in \eqref{coefficienti pull back di Lambda}. Then  \eqref{stima A kj bot}  follows by the smoothing properties \eqref{smoothing-u1} and by \eqref{interpolazione C k0}, \eqref{ansatz 0}. 
\end{proof}
\begin{remark}
If the frequency $\omega$ is diophantine, i.e. $ \om $ satisfies \eqref{dioph}, then \eqref{stima A ij} holds 
with $A_{k j}$ instead of $A_{k j}^{(n)}$ (i.e. $A_{k j}^{(n), \bot} = 0$). Furthermore if $Z = {\mathcal F}(i_0, \alpha_0) = 0$, then $A_{k j} = 0$.  
\end{remark}

As in \cite{BB13}, \cite{BBM-auto} we first modify the approximate torus $ i_0 $ to obtain an isotropic torus $ i_\d $ which is 
still approximately invariant. We denote the Laplacian  $ \Delta_\vphi := \sum_{k=1}^\nu \partial_{\vphi_k}^2 $. 

\begin{lemma}\label{toro isotropico modificato} {\bf (Isotropic torus)} 
The torus $ i_\delta(\vphi) := (\theta_0(\vphi), I_\delta(\vphi), z_0(\vphi) ) $ defined by 
\begin{equation}\label{y 0 - y delta}
\begin{aligned}
& I_\d := I_0 +  [\pa_\ph \theta_0(\vphi)]^{- T}  \rho(\vphi) \, ,  \\
& \rho_j(\vphi) := \Delta_\vphi^{-1} {\mathop\sum}_{ k = 1}^\nu \partial_{\vphi_j} A_{k j}(\vphi) \, , \quad j = 1, \ldots, \nu \, , 
\end{aligned}
\end{equation}
 is {\it isotropic}. Moreover $ I_\delta $ admits the splitting $ I_\delta = I_\delta^{(n)} + I_\delta^{(n), \bot}$ where 
 \begin{align} \label{toro I delta (n)}
I_\delta^{(n)} 
& := 
I_0 + [\pa_\ph \theta_0(\vphi)]^{- T}  \rho^{(n)}(\vphi) \,, \quad \rho^{(n)}_j (\vphi) := \Delta_\vphi^{-1} {\mathop\sum}_{ k = 1}^\nu \partial_{\vphi_j} A_{k j}^{(n)}(\vphi) \,, \\
I_\delta^{(n), \bot} & :=  [\pa_\ph \theta_0(\vphi)]^{- T}  \rho^{(n), \bot}(\vphi) \,, \quad \ 
\rho^{(n), \bot}_j (\vphi) 
:= \Delta_\vphi^{-1} {\mathop\sum}_{ k = 1}^\nu \partial_{\vphi_j} A_{k j}^{(n), \bot}(\vphi) \, . \label{toro I delta (n) bot}
 \end{align}
There is $ \s := \s(\nu,\t, k_0) $ and $\mathtt c > 0$ such that if \eqref{ansatz 0} holds with $ \sigma + \mathtt c \leq \mu $, then 
\begin{align} \label{2015-2}
\| I_\delta - I_0\|_s^{k_0, \gamma} & \leq \| I_\delta^{(n)} - I_0 \|_s^{k_0, \gamma} + \| I_\delta^{(n), \bot} \|_s^{k_0, \gamma}  
\leq_s \| \fracchi_0 \|_{s + 1}^{k_0,\gamma} \\ \label{stima y - y delta}
\| I_\delta^{(n)} - I_0 \|_s^{k_0, \gamma} 
& \leq_s  \gamma^{-1} \big(\| Z \|_{s + \s}^{k_0, \gamma} + 
\| Z \|_{s_0 + \s}^{k_0, \gamma} \|  {\mathfrak I}_0 \|_{s + \s}^{k_0, \gamma} \big) \,,
\\
\label{stima y - y delta bot}
\| I_\delta^{(n), \bot} \|_{s_0 + \mathtt c}^{k_0, \gamma} 
& \leq_{s_0, b}  K_n^{- b} \| \fracchi_0\|_{s_0 + \mathtt c + b}^{k_0, \gamma}\,, \qquad \forall b > 0\,,
\\
\label{derivata i delta}
\| \pa_i [ i_\d][ \widehat \imath ] \|_s^{k_0, \gamma} & \leq_s \| \widehat \imath \|_s^{k_0, \gamma} +  
\| {\mathfrak I}_0\|_{s + \s}^{k_0, \gamma} \| \widehat \imath  \|_{s_0}^{k_0, \gamma} \, .
\end{align}
Moreover the ``error"  function $Z_\delta := {\mathcal F}(i_\delta, \alpha_0)$  of the isotropic torus $ i_\d $ 
(defined analogously to \eqref{def Zetone})
may be splitted as $ Z_\delta = Z_\delta^{(n)} + Z_\delta^{(n), \bot}$
 with
\begin{align}
\label{stima toro modificato}
\| Z_\delta^{(n)} \|_s^{k_0, \gamma}
& \leq_s  \| Z \|_{s + \s}^{k_0, \gamma}  +  \| Z \|_{s_0 + \s}^{k_0, \gamma} \|  {\mathfrak I}_0 \|_{s + \s}^{k_0, \gamma} \\
 \label{stima toro modificato bot}
\| Z_\delta^{(n), \bot}\|_s^{k_0, \gamma} & \leq_s \| \fracchi_0\|_{s + \sigma}^{k_0, \gamma}\,, \  \| Z_\delta^{(n), \bot}\|_{s_0 + \mathtt c}^{k_0, \gamma} \leq_{s_0, b} K_n^{- b} \| \fracchi_0\|_{s_0 +  \sigma + \mathtt c + b}^{k_0, \gamma}\,, \   \forall b > 0\,.
\end{align}
 \end{lemma}
In the paper we denote equivalently the differential by $ \partial_i $ or  $ d_i $. Moreover we denote 
by $ \s := \s(\nu, \tau, k_0 ) $ possibly different (larger) ``loss of derivatives"  constants. 

\begin{proof}
The isotropy of the torus $i_\delta$, defined by \eqref{y 0 - y delta}, is proved in Lemma 6 of \cite{BB13}. 
The estimate \eqref{2015-2} follows by \eqref{y 0 - y delta}, \eqref{coefficienti pull back di Lambda}, \eqref{def Akj}, 
\eqref{interpolazione C k0} and \eqref{ansatz 0}. The estimate \eqref{stima y - y delta} follows by 
\eqref{toro I delta (n)} and   \eqref{stima A ij}. The estimate \eqref{stima y - y delta bot} follows by 
\eqref{toro I delta (n) bot} and  \eqref{stima A kj bot}.
 The bound  \eqref{derivata i delta} follows by 
\eqref{y 0 - y delta}, \eqref{def Akj}, \eqref{coefficienti pull back di Lambda}, \eqref{ansatz 0}.
 We now prove  \eqref{stima toro modificato}, \eqref{stima toro modificato bot}. One has  
\begin{align}
{\mathcal F}(i_\delta, \alpha_0) & =  {\mathcal F}(i_0, \alpha_0) +  
\begin{pmatrix} 0 \\ \Dom (I_\delta  - I_0 )  \\ 0 \end{pmatrix}\, 
  \! \!  +  \!  \e \big( X_P(i_\delta ) - X_P(i_0) \big) \nonumber\\
& = {\mathcal F}(i_0, \alpha_0)  \! + \!   \begin{pmatrix} 0 \\ \Dom (I_\delta  - I_0 )  \\ 0 \end{pmatrix}\, 
 \!  \!  + \! \e  \!  \int_0^1 \! \! \partial_I X_P(t \, i_\delta + (1 - t) i_0) \cdot (I_\delta - I_0)\, d t \nonumber \\ 
& = Z_\delta^{(n)} + Z_\delta^{(n), \bot} \nonumber
\end{align}
where 
\begin{equation}\label{Z delta (n)}
\begin{aligned}
 Z_\delta^{(n)}  & := {\mathcal F}(i_0, \alpha_0) + \begin{pmatrix} 0 \\ \Dom (I_\delta^{(n)}  - I_0 )  \\ 0 \end{pmatrix} \\\, 
& \quad + \e \int_0^1 \partial_I X_P(t \, i_\delta + (1 - t) i_0) \cdot (I_\delta^{(n)} - I_0)\, d t\,, \\
\end{aligned}
\end{equation}
\begin{equation}\label{Z delta (n) bot}
\begin{aligned}
Z_\delta^{(n), \bot} := \begin{pmatrix} 0 \\ \Dom I_\delta^{(n), \bot}     \\ 0 \end{pmatrix}\, 
+ \e \int_0^1 \partial_I X_P(t \, i_\delta + (1 - t) i_0) \cdot I_\delta^{(n), \bot} \, d t\,.
\end{aligned}
\end{equation}
By differentiating \eqref{toro I delta (n)} and, arguing as in \cite{BB13}, \cite{BBM-auto}, we get 
\begin{align}
    \Dom (I_\delta^{(n)} - I_0 ) & =  [ \pa_\ph \theta_0(\vphi)]^{-T} \Dom \rho^{(n)}(\vphi)   \nonumber\\
& \quad - \big(  [ \pa_\ph \theta_0(\vphi)]^{-T} 
\big( \Dom [\pa_\ph \theta_0(\vphi)]^T \big) [\pa_\ph \theta_0(\vphi)]^{-T}  \big) \rho^{(n)}(\vphi)  \label{come1}  \\
&  \!   \Dom [\pa_\ph \theta_0(\vphi)]  
 = \e \pa_\vphi (\pa_I P)(i_0(\vphi)) + \partial_\vphi Z_1(\vphi) \, . \label{come2}
\end{align}
Then  \eqref{stima toro modificato} follows by  \eqref{Z delta (n)}, \eqref{come1}-\eqref{come2}, 
\eqref{stima derivata XP}, \eqref{interpolazione C k0}, \eqref{stima y - y delta}, 
\eqref{ansatz 0}, Lemma \ref{Lemma Z piccolezza}, \eqref{toro I delta (n)}, \eqref{bella trovata},  \eqref{stima A ij}. 
The estimates \eqref{stima toro modificato bot} follow by \eqref{Z delta (n) bot}, \eqref{toro I delta (n) bot}, \eqref{interpolazione C k0}, 
\eqref{stima A kj bot}, 
\eqref{stima derivata XP}, \eqref{2015-2}, \eqref{ansatz 0} and \eqref{stima y - y delta bot}. 
\end{proof}

In order to find an approximate inverse of the linearized operator $d_{i, \alpha} {\mathcal F}(i_\delta )$ 
we introduce  the symplectic diffeomorpshim 
$ G_\delta : (\phi, y, w) \to (\theta, I, z)$ of the phase space $\T^\nu \times \R^\nu \times H_{{\mathbb S}^+}^\bot$ defined by
\begin{equation}\label{trasformazione modificata simplettica}
\begin{pmatrix}
\theta \\
I \\
z
\end{pmatrix} := G_\delta \begin{pmatrix}
\phi \\
y \\
w
\end{pmatrix} := 
\begin{pmatrix}
\!\!\!\!\!\!\!\!\!\!\!\!\!\!\!\!\!\!\!\!\!\!\!\!\!\!\!\!\!\!\!\!\!
\!\!\!\!\!\!\!\!\!\!\!\!\!\!\!\!\!\!\!\!\!\!\!\!\!\!\!\!\!\!\!\!\!
\!\!\!\!\!\!\!\!\!\!\!\!\!\!\!\!\!\!\!\!\!\!\!\!\!\!\!\!\!\!\!\! \theta_0(\phi) \\
I_\delta (\phi) + [\pa_\phi \theta_0(\phi)]^{-T} y - \big[ (\pa_\teta \tilde{z}_0) (\theta_0(\phi)) \big]^T J w \\
\!\!\!\!\!\!\!\!\!\!\!\!\!\!\!\!\!\!\!\!\!\!\!\!\!\!\!\!\!
\!\!\!\!\!\!\!\!\!\!\!\!\!\!\!\!\!\!\!\!\!\!\!\!\!\!\!\!\!
\!\!\!\!\!\!\!\!\!\!\!\!\!\!\!\!\!\!\!\!\!\!\!\!\!\!\!\!\!  z_0(\phi) + w
\end{pmatrix} 
\end{equation}
where $ \tilde{z}_0 (\theta) := z_0 (\theta_0^{-1} (\theta))$. 
It is proved in \cite{BB13} that $ G_\delta $ is symplectic, because  the torus $ i_\d $ is isotropic 
(Lemma \ref{toro isotropico modificato}).
In the new coordinates,  $ i_\delta $ is the trivial embedded torus
$ (\phi , y , w ) = (\phi , 0, 0 ) $.  Under the symplectic change of variables $ G_\d $ the Hamiltonian 
vector field $ X_{H_\a} $ (the Hamiltonian $  H_\a $ is defined in \eqref{H alpha}) changes into 
\be\label{new-Hamilt-K}
X_{K_\a} = (D G_\d)^{-1} X_{H_\a} \circ G_\d \qquad {\rm where} \qquad K_\alpha := H_{\alpha} \circ G_\d  \, .
\ee
By \eqref{parity solution} the transformation $ G_\d $ is also reversibility preserving and so $ K_\alpha $ is reversible, 
$ K_\a \circ \tilde \rho = K_\a $.  

The Taylor expansion of $ K_\a $ at the trivial torus $ (\phi , 0, 0 ) $ is 
\begin{align} 
K_\alpha (\phi, y , w)
& =  K_{00}(\phi, \alpha) + K_{10}(\phi, \alpha) \cdot y + (K_{0 1}(\phi, \alpha), w)_{L^2(\T_x)} + 
\frac12 K_{2 0}(\phi) y \cdot y 
\nonumber \\ & 
\quad +  \big( K_{11}(\phi) y , w \big)_{L^2(\T_x)} 
+ \frac12 \big(K_{02}(\phi) w , w \big)_{L^2(\T_x)} + K_{\geq 3}(\phi, y, w)  
\label{KHG}
\end{align}
where $ K_{\geq 3} $ collects the terms at least cubic in the variables $ (y, w )$.
The Taylor coefficient $K_{00}(\phi, \alpha) \in \R $,  
$K_{10}(\phi, \alpha) \in \R^\nu $,  
$K_{01}(\phi, \alpha) \in H_{{\mathbb S}^+}^\bot$, 
$K_{20}(\phi) $ is a $\nu \times \nu$ real matrix, 
$K_{02}(\phi)$ is a linear self-adjoint operator of $ H_{{\mathbb S}^+}^\bot $ and 
$K_{11}(\phi) \in {\mathcal L}(\R^\nu, H_{{\mathbb S}^+}^\bot )$. 

Note that, by \eqref{H alpha} and \eqref{trasformazione modificata simplettica}, 
the only Taylor coefficients which depend on $ \a $ are
$ K_{00} $, $ K_{10} $, $ K_{01} $.

\smallskip
The Hamilton equations associated to \eqref{KHG}  are 
\begin{equation}\label{sistema dopo trasformazione inverso approssimato}
\begin{cases}
\dot \phi \hspace{-30pt} & = K_{10}(\phi, \alpha) +  K_{20}(\phi) y + 
K_{11}^T (\phi) w + \partial_{y} K_{\geq 3}(\phi, y, w)
\\
\dot y \hspace{-30pt} & = 
\partial_\phi K_{00}(\phi, \alpha) - [\partial_{\phi}K_{10}(\phi, \alpha)]^T  y - 
[\partial_{\phi} K_{01}(\phi, \alpha)]^T  w  - \partial_\phi \big(\frac12 K_{2 0}(\phi) y \cdot y \big)
\\
& \quad -
\partial_\phi \big(  ( K_{11}(\phi) y , w )_{L^2(\T_x)} + 
\frac12 ( K_{02}(\phi) w , w )_{L^2(\T_x)} + K_{\geq 3}(\phi, y, w) \big)
\\
\dot w \hspace{-30pt} & = J \big( K_{01}(\phi, \alpha) + 
K_{11}(\phi) y +  K_{0 2}(\phi) w + \nabla_w K_{\geq 3}(\phi, y, w) \big) 
\end{cases} 
\end{equation}
where $ \partial_{\phi}K_{10}^T $ is the $ \nu \times \nu $ transposed matrix and 
$ \partial_{\phi}K_{01}^T $,  $ K_{11}^T  : {H_{{\mathbb S}^+}^\bot \to \R^\nu} $ are defined by the 
duality relation $ ( \partial_{\phi} K_{01} [\hat \phi ],  w)_{L^2_x}  = \hat \phi \cdot [\partial_{\phi}K_{01}]^T w  $,
$ \forall \hat \phi \in \R^\nu, w \in H_{{\mathbb S}^+}^\bot $, 
and similarly for $ K_{11} $. 
Explicitly, for all  $ w \in H_{{\mathbb S}^+}^\bot $, 
and denoting $\underline{e}_k$ the $k$-th versor of $\R^\nu$, 
\begin{equation} \label{K11 tras}
K_{11}^T(\phi) w =  {\mathop \sum}_{k=1}^\nu \big(K_{11}^T(\phi) w \cdot \underline{e}_k\big) \underline{e}_k   =
{\mathop \sum}_{k=1}^\nu  
\big( w, K_{11}(\phi) \underline{e}_k  \big)_{L^2(\T_x)}  \underline{e}_k  \, \in \R^\nu \, .  
\end{equation}
In the next lemma we provide estimates of the coefficients $ K_{00} $, $ K_{10} $, $K_{01} $ in the Taylor expansion \eqref{KHG}. 
\begin{lemma} \label{coefficienti nuovi} 
There is $\sigma := \sigma (\tau, \nu, k_0)>0 $ and a decomposition 
\begin{equation}\label{splitting K 00 01 10}
\partial_\phi K_{00} =  \partial_\phi K_{00}^{(n)} + \partial_\phi K_{00}^{(n), \bot} \,,\ \  K_{10} = K_{10}^{(n)} + K_{10}^{(n), \bot}\,, \ \  K_{0 1} = K_{0 1}^{(n)} + K_{0 1}^{(n), \bot}\,,
\end{equation}
such that, 
if \eqref{ansatz 0} holds with $ \mu \geq \sigma + \mathtt c $, $\mathtt c > 0 $,  then  
\begin{equation}\label{K 00 10 01} 
\begin{aligned}
& \|  \partial_\phi K_{00}^{(n)}(\cdot, \alpha_0) \|_s^{k_0, \gamma} 
+ \| K_{10}^{(n)}(\cdot, \alpha_0) - \om  \|_s^{k_0, \gamma} +  \| K_{0 1}^{(n)}(\cdot, \alpha_0) \|_s^{k_0, \gamma}  \\
& \leq_s  \| Z \|_{s + \s}^{k_0, \gamma} +  \| Z \|_{s_0 + \s}^{k_0, \gamma} \| {\mathfrak I}_0 \|_{s + \s}^{k_0, \gamma} \, ,
\end{aligned}
\end{equation}
\begin{equation}\label{K 00 10 01 bot alta}
\begin{aligned}
& \|  \partial_\phi K_{00}^{(n), \bot}(\cdot, \alpha_0) \|_s^{k_0, \gamma} 
+ \| K_{10}^{(n), \bot}(\cdot, \alpha_0)   \|_s^{k_0, \gamma} +  \| K_{0 1}^{(n), \bot}(\cdot, \alpha_0) \|_s^{k_0, \gamma}   \\
& \leq_s \| \fracchi_0\|_{s + \sigma}^{k_0, \gamma}\,,
\end{aligned}
\end{equation}
\begin{equation}\label{K 00 10 01 bot bassa}
\begin{aligned}
& \|  \partial_\phi K_{00}^{(n), \bot}(\cdot, \alpha_0) \|_{s_0 + \mathtt c}^{k_0, \gamma} 
+ \| K_{10}^{(n), \bot}(\cdot, \alpha_0)   \|_{s_0 + \mathtt c}^{k_0, \gamma} +  \| K_{0 1}^{(n), \bot}(\cdot, \alpha_0) \|_{s_0 + \mathtt c}^{k_0, \gamma}  \\ & \leq_{s_0, b} K_n^{- b} \| \fracchi_0\|_{s_0 + \sigma + \mathtt c + b}^{k_0, \gamma}
\end{aligned}
\end{equation}
for all $ b > 0 $.
\end{lemma}

\begin{proof}
In Lemma 8 of \cite{BB13} or Lemma 6.4 of \cite{BBM-auto} 
the following identities are proved 
\begin{align*}
\partial_\phi K_{00}(\phi, \alpha_0) & =  
- [ \pa_\phi \teta_0 (\phi) ]^T \big( - Z_{2, \d} - 
[ \pa_\phi I_\d] [ \pa_\phi \teta_0]^{-1} Z_{1, \d}   
- [ (\pa_\theta {\tilde z}_0)( \teta_0 (\phi)) ]^T J Z_{3,\d} 
\\ 
&  \quad - [ (\pa_\theta {\tilde z}_0)(\teta_0 (\phi)) ]^T J \partial_\phi z_0 (\phi) [ \pa_\phi \teta_0 (\phi)]^{-1} Z_{1,\d} \big) \, ,  
\\
K_{10}(\phi, \alpha_0) & 
=  \omega -   [ \pa_\phi \theta_0(\phi)]^{-1} Z_{1,\d}(\phi) \,, 
\\
K_{01}(\phi, \alpha_0) 
& = J Z_{3,\d} - J \pa_\phi z_0(\phi) [\pa_\phi \theta_0(\phi)]^{-1} Z_{1,\d}(\phi)   
\end{align*}
where $  Z_\d = (Z_{1,\d}, Z_{2,\d}, Z_{3,\d}) := {\mathcal F}(i_\d, \alpha_0) $. According to the splitting  
$ Z_\delta = Z_\delta^{(n)} + Z_\delta^{(n), \bot}  $ 
given in Lemma \ref{toro isotropico modificato}, setting 
$$ Z_\delta^{(n)} = (Z_{1, \delta}^{(n)}, Z_{2, \delta}^{(n)}, Z_{3, \delta}^{(n)}), 
Z_\delta^{(n), \bot} = (Z_{1, \delta}^{(n), \bot}, Z_{2, \delta}^{(n), \bot}, Z_{3, \delta}^{(n), \bot}) \, , 
$$ we get the decomposition \eqref{splitting K 00 01 10}  with 
\begin{align*}
\partial_\phi K_{00}^{(n)}(\phi, \alpha_0) & =  
- [ \pa_\phi \teta_0 (\phi) ]^T \big( - Z_{2, \d}^{(n)} - 
[ \pa_\phi I_\d] [ \pa_\phi \teta_0]^{-1} Z_{1, \d}^{(n)}   \\
& \quad - [ (\pa_\theta {\tilde z}_0)( \teta_0 (\phi)) ]^T J Z_{3,\d}^{(n)} 
\\ 
&  \quad - [ (\pa_\theta {\tilde z}_0)(\teta_0 (\phi)) ]^T J \partial_\phi z_0 (\phi) [ \pa_\phi \teta_0 (\phi)]^{-1} Z_{1,\d}^{(n)} \big) \, ,  
\\
\partial_\phi K_{00}^{(n), \bot}(\phi, \alpha_0) & =  
- [ \pa_\phi \teta_0 (\phi) ]^T \big( - Z_{2, \d}^{(n), \bot} - 
[ \pa_\phi I_\d] [ \pa_\phi \teta_0]^{-1} Z_{1, \d}^{(n), \bot}    
\\ 
&  \quad  - [ (\pa_\theta {\tilde z}_0)( \teta_0 (\phi)) ]^T J Z_{3,\d}^{(n), \bot} \nonumber\\
& \quad  - [ (\pa_\theta {\tilde z}_0)(\teta_0 (\phi)) ]^T J \partial_\phi z_0 (\phi) [ \pa_\phi \teta_0 (\phi)]^{-1} Z_{1,\d}^{(n), \bot} \big) \, ,  
\\
K_{10}^{(n)}(\phi, \alpha_0) & 
=  \omega -   [ \pa_\phi \theta_0(\phi)]^{-1} Z_{1,\d}^{(n)}(\phi) \,, 
\\
K_{10}^{(n), \bot}(\phi, \alpha_0) & 
= -   [ \pa_\phi \theta_0(\phi)]^{-1} Z_{1,\d}^{(n), \bot}(\phi) \,, 
\\
K_{01}^{(n)}(\phi, \alpha_0) 
& = J Z_{3,\d}^{(n)} - J \pa_\phi z_0(\phi) [\pa_\phi \theta_0(\phi)]^{-1} Z_{1,\d}^{(n)}(\phi) \\
K_{01}^{(n), \bot}(\phi, \alpha_0) 
& = J Z_{3,\d}^{(n), \bot} - J \pa_\phi z_0(\phi) [\pa_\phi \theta_0(\phi)]^{-1} Z_{1,\d}^{(n), \bot}(\phi)\,.  
\end{align*}

Then the estimates \eqref{K 00 10 01}
-\eqref{K 00 10 01 bot bassa} follow by \eqref{2015-2}, \eqref{stima toro modificato}, \eqref{stima toro modificato bot}, using \eqref{interpolazione C k0} and \eqref{ansatz 0}. 
\end{proof}

We now estimate  the variation of the coefficients
$ K_{00} $, $ K_{10} $, $K_{01} $ with respect to $ \a $.
Note, in particular, that $ \partial_\alpha K_{10} \approx {\rm Id} $ says that the tangential frequencies vary with 
$ \a \in \R^\nu $. We also estimate  $ K_{20}$ and $ K_{11} $. 

\begin{lemma} \label{lemma:Kapponi vari} We have 
\begin{align*}
& \| \partial_\alpha K_{00}\|_s^{k_0, \gamma} + \| \partial_\alpha K_{10} - {\rm Id} \|_s^{k_0, \gamma} + \| \partial_\alpha K_{0 1}\|_s^{k_0, \gamma} \leq_s   \| \fracchi_0\|_{s + \sigma}^{k_0, \gamma} \, , \\
&   \|K_{20}  \|_s \leq_s \e \big( 1 + \| \fracchi_0\|_{s + \s}^{k_0, \gamma} \big) \, , \\ 
& \| K_{11} y \|_s^{k_0, \gamma} 
\leq_s \e \big(\| y \|_{s + 2}^{k_0, \gamma}
+ \| \fracchi_0 \|_{s + \sigma}^{k_0, \gamma}  
\| y \|_{s_0 + 2}^{k_0, \gamma} \big) \, , \\
&  \| K_{11}^T w \|_s^{k_0, \gamma}
\leq_s \e \big(\| w \|_{s + 2}^{k_0, \gamma}
+  \| \fracchi_0 \|_{s + \sigma}^{k_0, \gamma}
\| w \|_{s_0 + 2}^{k_0, \gamma} \big)\, . 
\end{align*}
\end{lemma}

\begin{proof}
By \cite{BB13}, \cite{BBM-auto} we have
\begin{align*}
\partial_\alpha K_{00}(\phi)  & = I_{\delta}(\phi)\,,\quad \partial_\alpha K_{10}(\phi) = [\partial_\phi \theta_0(\phi)]^{- 1}\,,\quad \partial_\alpha K_{01}(\phi) = J \partial_\theta \tilde z_0(\theta_0(\phi)) \, , \\
K_{2 0}(\vphi)  & = \e  [\pa_\ph \theta_0(\vphi)]^{-1} \partial_{II} P(i_\delta(\vphi)) 
[\pa_\ph \theta_0(\vphi)]^{-T} \, , \\
K_{11}(\vphi)  & = 
 \e \big(\partial_{I} \nabla_z P(i_\delta(\vphi)) [\pa_\ph \theta_0 (\vphi)]^{-T} \\
& \quad + \ J (\pa_\theta {\tilde z}_0) (\teta_0(\vphi)) (\partial_{II} P) (i_\delta(\vphi)) [\pa_\ph \theta_0 (\vphi)]^{-T} \big) \, . 
\end{align*}
Then \eqref{stime XP}, \eqref{ansatz 0}, \eqref{2015-2} imply the lemma
(the bound  for $K_{11}^T$ follows by \eqref{K11 tras}).
\end{proof}

Under the linear change of variables 
\begin{equation}\label{DGdelta}
D G_\delta(\vphi, 0, 0) 
\begin{pmatrix}
\widehat \phi \, \\
\widehat y \\
\widehat w
\end{pmatrix} 
:= 
\begin{pmatrix}
\pa_\phi \theta_0(\vphi) & 0 & 0 \\
\pa_\phi I_\delta(\vphi) & [\pa_\phi \theta_0(\vphi)]^{-T} & 
- [(\pa_\theta \tilde{z}_0)(\theta_0(\vphi))]^T J \\
\pa_\phi z_0(\vphi) & 0 & I
\end{pmatrix}
\begin{pmatrix}
\widehat \phi \, \\
\widehat y \\
\widehat w
\end{pmatrix} 
\end{equation}
the linearized operator  $d_{i, \alpha}{\mathcal F}(i_\delta )$ is transformed
 (approximately, see \eqref{verona 2} for the precise expression of the error) into the one  obtained when we linearize 
the Hamiltonian system \eqref{sistema dopo trasformazione inverso approssimato} at $(\phi, y , w ) = (\vphi, 0, 0 )$,
differentiating also in $ \a $ at $ \a_0 $, and changing $ \partial_t \rightsquigarrow \Dom $, 
namely 
\begin{equation}\label{lin idelta}
\begin{aligned}
& (\widehat \phi ,
\widehat y,
\widehat w,
\widehat \a ) \mapsto \\
& \begin{pmatrix}
\Dom \widehat \phi - \partial_\phi K_{10}(\vphi)[\widehat \phi \, ] - \partial_\alpha K_{10}(\vphi)[\widehat \alpha] - 
K_{2 0}(\vphi)\widehat y - K_{11}^T (\vphi) \widehat w \\
 \Dom  \widehat y + \partial_{\phi\phi} K_{00}(\vphi)[\widehat \phi] + 
 \partial_\phi \partial_\alpha  K_{00}(\vphi)[\widehat \alpha] + 
[\partial_\phi K_{10}(\vphi)]^T \widehat y + 
[\partial_\phi  K_{01}(\vphi)]^T \widehat w   \\ 
\Dom  \widehat w - J 
\{ \partial_\phi K_{01}(\vphi)[\widehat \phi] + \partial_\alpha K_{01}(\vphi)[\widehat \alpha] + K_{11}(\vphi) \widehat y + K_{02}(\vphi) \widehat w \}
\end{pmatrix} \! .  \hspace{-5pt}
\end{aligned}
\end{equation}
 As in \cite{BBM-auto},  
 by \eqref{DGdelta}, \eqref{ansatz 0}, \eqref{2015-2}, 
 the induced composition operator satisfies:   for all $ \widehat \imath := (\widehat \phi, \widehat y, \widehat w) $
 \begin{equation}\label{DG delta}
 \begin{aligned}
 \|DG_\delta(\vphi,0,0) [\widehat \imath] \|_s^{k_0, \gamma} + \|DG_\delta(\vphi,0,0)^{-1} [\widehat \imath] \|_s^{k_0, \gamma}&  \leq_s \| \widehat \imath \|_{s}^{k_0, \gamma}  \\
 & \quad +  \| {\mathfrak I}_0 \|_{s + \s}^{k_0, \gamma}  \| \widehat \imath \|_{s_0}^{k_0, \gamma}\,,
 \end{aligned}
 \end{equation}
 \begin{equation}\label{DG2 delta}
 \begin{aligned}
 \| D^2 G_\delta(\vphi,0,0)[\widehat \imath_1, \widehat \imath_2] \|_s^{k_0, \gamma} 
& \leq_s  \| \widehat \imath_1\|_s^{k_0, \gamma}  \| \widehat \imath_2 \|_{s_0}^{k_0, \gamma} 
+ \| \widehat \imath_1\|_{s_0}^{k_0, \gamma}  \| \widehat \imath_2 \|_{s}^{k_0, \gamma}  \\
& \quad +  \| {\mathfrak I}_0  \|_{s + \s}^{k_0, \gamma} \|\widehat \imath_1 \|_{s_0}^{k_0, \gamma}  \| \widehat \imath_2\|_{s_0}^{k_0, \gamma} \, .
 \end{aligned}
 \end{equation}
In order to construct an ``almost-approximate" inverse of \eqref{lin idelta} we need  that 
\be\label{Lomega def}
{\mathcal L}_\omega := \Pi_{{\mathbb S}^+}^\bot \big(\Dom   - J K_{02}(\vphi) \big)_{|{H_{{\mathbb S}^+}^\bot}} 
\ee
is ``almost invertible" up to the scales $ K_n := K_0^{\chi^{n}} $, $ \chi := 3/ 2 $, defined in \eqref{definizione Kn}, and 
used  for the nonlinear Nash-Moser iteration of section \ref{sec:NM}. 
Let $H^s_\bot(\T^{\nu + 1}) := H^s(\T^{\nu + 1}) \cap H_{{\mathbb S}^+}^\bot$.
\begin{itemize}
\item {\sc Almost-invertibility  assumption.} 
There exists a subset $ \tLm_o  \subset \tOm \times [\kappa_1, \kappa_2] $ such that, 
for all $ (\omega, \kappa) \in  \tLm_o  $ the operator $ {\mathcal L}_\omega $ in \eqref{Lomega def} may be decomposed as
\be\label{inversion assumption}
{\mathcal L}_\omega  
= {\mathbf L}_\omega + {\mathbf R}_\omega + {\mathbf R}_\omega^\bot 
\ee
where $ {\mathbf L}_\omega $ is invertible and  ${\mathbf R}_\omega$, ${\mathbf R}_\omega^\bot $ satisfy the estimates  
\eqref{stima R omega corsivo}-\eqref{stima R omega bot corsivo alta}. More precisely for every function
$ g \in H^{s+\sigma}_{\bot} (\T^{\nu+1}) $ and such that $ g(-\vphi) = - \rho g( \vphi)$,  there 
is a solution $ h :=  {\mathbf L}_\om^{- 1} g  \in H^{s}_{\bot} (\T^{\nu+1}) $ such that
$ h (-\vphi) =  \rho h ( \vphi) $,  
of the linear equation $ {\mathbf L}_\om h = g $
which satisfies for all $s_0 \leq s \leq S$ the tame estimate 
\begin{equation}\label{tame inverse}
\| {\mathbf L}_\om^{- 1} g \|_s^{k_0, \gamma} \leq_S  \g^{-1} 
\big(  \| g \|_{s + \sigma}^{k_0, \gamma} + 
 \| {\mathfrak I}_0 \|_{s + \mu({\mathtt b}) + \sigma}^{k_0, \gamma}  \|g \|_{s_0 + \sigma}^{k_0, \gamma}  \big) 
\end{equation}
for some $ \sigma := \sigma (\tau, \nu, k_0) >  0 $, and the constant $\mu(\mathtt b) > 0$ is defined in \eqref{definizione bf c (beta)}.
\end{itemize}

This assumption shall be verified 
by Theorem \ref{inversione parziale cal L omega} at each $ n$-th step of the Nash-Moser nonlinear iteration.
It is  obtained, in sections \ref{linearizzato siti normali} and \ref{sec: reducibility}, 
by the process of almost-diagonalization  of $ {\mathcal L}_\om $ 
up to remainders of size $ O(\e N_{n-1}^{{\mathtt a} -1}) $ where the larger scales $ N_n $ are 
 \be\label{NnKn}
 N_n := K_n^p \, , \quad i.e. \quad   N_0 = K_0^p  \, , 
 \ee
and the constant $ p > 1$  is large enough, i.e. it satisfies \eqref{cond-su-p}. 
The set of ``good" parameters $ \tLm_o$ is contained in particular in the set 
$ {\mathtt D \mathtt C}_{K_n}^{\g} \times [\kappa_1, \kappa_2] $ defined in \eqref{omega diofanteo troncato}.
Actually the parameters in $ (\omega, \kappa) \in \tLm_o$ have to 
satisfy also first and second order Melnikov non-resonance conditions, see \eqref{Melnikov-invert}.

\smallskip
In order to find an almost-approximate inverse of the linear operator in \eqref{lin idelta} 
(and so of $ d_{i, \a} {\mathcal F}(i_\d) $)
 it is sufficient to almost invert the operator
\begin{equation}\label{operatore inverso approssimato} 
{\mathbb D} [\widehat \phi, \widehat y, \widehat w, \widehat \alpha ] := 
  \begin{pmatrix}
\Dom \widehat \phi - \partial_\alpha K_{10}(\vphi)[\widehat \alpha] - 
K_{20}(\vphi) \widehat y  - K_{11}^T(\vphi) \widehat w\\
\Dom  \widehat y + \partial_\phi \partial_\alpha  K_{00}(\vphi)[\widehat \alpha] \\
{\mathbf L}_\omega \widehat w  - J \partial_\alpha K_{01}(\vphi)[\widehat \alpha] -J K_{11}(\vphi)\widehat y  
\end{pmatrix}
\end{equation}
which is obtained by neglecting in \eqref{lin idelta} 
the terms $ \partial_\phi K_{10} $, $ \partial_{\phi \phi} K_{00} $, $ \partial_\phi K_{00} $, $ \partial_\phi K_{01} $
(which vanish at an exact solution by Lemma \ref{coefficienti nuovi}), and the small  remainders
${\mathbf R}_\omega $, ${\mathbf R}_\omega^\bot $ which appear in \eqref{inversion assumption}. 
In addition, since we require only the finitely many non-resonance conditions \eqref{omega diofanteo troncato}, we 
also decompose 
$ \Dom $ as 
\begin{equation}\label{splitting cal D omega}
\begin{aligned}
& \qquad \qquad \qquad \quad \Dom = {\mathcal D}_\omega^{(n)} + {\mathcal D}_\omega^{(n), \bot}\,, \\
&   {\mathcal D}_\omega^{(n)} := \Pi_{K_n} \Dom \Pi_{K_n} + \Pi_{K_n}^\bot \,, \quad  {\mathcal D}_\omega^{(n), \bot} :=
 \Pi_{K_n}^\bot \Dom \Pi_{K_n}^\bot - \Pi_{K_n}^\bot 
 \end{aligned}
\end{equation}
and we further split the operator $ {\mathbb D} $ in \eqref{operatore inverso approssimato} as 
\begin{equation}\label{splitting mathbb D}
{\mathbb D} = {\mathbb D}_n + {\mathbb D}_n^\bot  \qquad {\rm where} \qquad 
{\mathbb D}_n^\bot [\widehat \phi, \widehat y, \widehat w, \widehat \alpha ]  := \begin{pmatrix}
{\mathcal D}_\omega^{(n), \bot} \widehat \phi \\
{\mathcal D}_\omega^{(n), \bot} \widehat y \\
0 
\end{pmatrix} 
\end{equation}
and 
\begin{equation}\label{definizione mathbb Dn}
{\mathbb D}_n [\widehat \phi, \widehat y, \widehat w, \widehat \alpha ] := 
  \begin{pmatrix}
{\mathcal D}_\om^{(n)} \widehat \phi - \partial_\alpha K_{10}(\vphi)[\widehat \alpha] - 
K_{20}(\vphi) \widehat y  - K_{11}^T(\vphi) \widehat w\\
{\mathcal D}_\om^{(n)}  \widehat y + \partial_\alpha \partial_\phi K_{00}(\vphi)[\widehat \alpha] \\
{\mathbf L}_\omega \widehat w  - J \partial_\alpha K_{01}(\vphi)[\widehat \alpha] -J K_{11}(\vphi)\widehat y  
\end{pmatrix}\, . 
\end{equation}
By the smoothing properties \eqref{smoothing-u1}, 
 the operator ${\mathcal D}_\omega^{(n), \bot}$ satisfies
\be\label{stime cal D omega (n) bot}
\| {\mathcal D}_\omega^{(n), \bot} h \|_{s_0}^{k_0, \gamma} \leq K_n^{- b} \| h \|_{s_0 + b + 1}^{k_0, \gamma}\,, \ \ \forall 
b  > 0\,, 
\qquad \|{\mathcal D}_\omega^{(n), \bot} h \|_s^{k_0, \gamma} \leq \| h \|_{s + 1}^{k_0, \gamma}\,. 
\ee
\begin{lemma}\label{lemma splitting cal D omega}
Assume that $\omega \in {\mathtt D \mathtt C}_{K_n}^\gamma$, see \eqref{omega diofanteo troncato}. Then, for all
$ g \in H^s $ with zero average, the linear equation $ {\mathcal D}_\omega^{(n)} h = g $ has a unique solution 
$ h := [{\mathcal D}_\omega^{(n)}]^{- 1} g $ with zero average, which satisfies  
\begin{equation}\label{stima D omega (n)}
\| [{\mathcal D}_\omega^{(n)}]^{- 1} g \|_s^{k_0, \gamma} \leq 
\gamma^{- 1} \| g \|_{s+\tau_1}^{k_0, \gamma}\,, \qquad \tau_1  :=  \tau + k_0 (\tau + 1)\,.
\end{equation}
\end{lemma}
We look for an exact inverse of ${\mathbb D}_n $ defined in \eqref{definizione mathbb Dn} by solving the system 
\begin{equation}\label{operatore inverso approssimato proiettato}
{\mathbb D}_n [\widehat \phi, \widehat y, \widehat w, \widehat \alpha] : = \begin{pmatrix}
{\mathcal D}_\om^{(n)} \widehat \phi - \partial_\alpha K_{10}(\vphi)[\widehat \alpha] - 
K_{20}(\vphi) \widehat y  - K_{11}^T(\vphi) \widehat w\\
{\mathcal D}_\om^{(n)}  \widehat y + \partial_\alpha \partial_\phi K_{00}(\vphi)[\widehat \alpha] \\
{\mathbf L}_\omega \widehat w  - J \partial_\alpha K_{01}(\vphi)[\widehat \alpha] -J K_{11}(\vphi)\widehat y  
\end{pmatrix} = \begin{pmatrix}
g_1  \\
g_2  \\
g_3 
\end{pmatrix}
\end{equation}
where $(g_1, g_2, g_3)$ satisfy the reversibility property 
\begin{equation}\label{parita g1 g2 g3}
g_1(\vphi) = g_1(- \vphi)\,,\quad g_2(\vphi) = - g_2(- \vphi)\,,\quad g_3(\vphi) = - (\rho g_3)(- \vphi)\,.
\end{equation}
We first consider the second equation in \eqref{operatore inverso approssimato proiettato}, namely 
$ {\mathcal D}_\om^{(n)}  \widehat y  = g_2  -  \partial_\alpha \partial_\phi K_{00}(\vphi)[\widehat \alpha] $. 
By reversibility, the $\vphi$-average of the right hand side of this equation is zero, and so, 
by Lemma \ref{lemma splitting cal D omega}, its solution is  
\begin{equation}\label{soleta}
\widehat y := [{\mathcal D}_\om^{(n)}]^{-1} \big(
g_2 - \partial_\alpha \partial_\phi K_{00}(\vphi)[\widehat \alpha] \big)\,.  
\end{equation}
Then we consider the third equation
$ {\mathbf L}_\om \widehat w = g_3 + J K_{11}(\vphi) \widehat y + J \partial_\alpha K_{0 1}(\vphi)[\widehat \alpha]  $
that, by the inversion assumption \eqref{tame inverse}, has a solution 
\begin{equation}\label{normalw}
\widehat w := {\mathbf L}_\om^{-1} \big( g_3 +J K_{11}(\vphi) \widehat y + J \partial_\alpha K_{0 1}(\vphi)[\widehat \alpha] \big) \, .  
\end{equation}
Finally, we solve the first equation in \eqref{operatore inverso approssimato proiettato}, 
which, substituting \eqref{soleta}, \eqref{normalw}, becomes
\begin{equation}\label{equazione psi hat}
{\mathcal D}_\om^{(n)} \widehat \phi  = 
g_1 +  M_1(\vphi)[\widehat \alpha] + M_2(\vphi) g_2 + M_3(\vphi) g_3\,,
\end{equation}
where
\begin{equation}\label{M1}
\begin{aligned}
M_1(\vphi) & :=  \partial_\alpha K_{10}(\vphi) - M_2(\vphi)\partial_\alpha \partial_\phi K_{00}(\vphi)  + M_3(\vphi) J \partial_\alpha K_{01}(\vphi)\,, 
\end{aligned}
\end{equation}
\begin{equation}\label{cal M2}
\begin{aligned}
& M_2(\vphi) :=  K_{20}(\vphi) [{\mathcal D}_\om^{(n)}]^{-1} + K_{11}^T(\vphi){\mathbf L}_\om^{- 1} J K_{11}(\vphi)[{\mathcal D}_\om^{(n)}]^{-1} \, , \\ 
& M_3(\vphi) :=  K_{11}^T (\vphi) {\mathbf L}_\om^{-1} \, .  
\end{aligned}
\end{equation}
In order to solve the equation \eqref{equazione psi hat} we have 
to choose $ \widehat \alpha $ such that the right hand side  has zero average.  
By Lemma \ref{lemma:Kapponi vari}, \eqref{ansatz 0}, \eqref{stima D omega (n)} the $\ph$-averaged matrix 
$ \langle M_1 \rangle = {\rm Id} + O( \e \gamma^{-(1 + k_1)}) $.  
Therefore, for $ \e \gamma^{- (1 + k_1)} $ small enough,  $\langle M_1 \rangle$ is invertible and $\langle M_1 \rangle^{-1} = {\rm Id} 
+ O(\e \gamma^{- (1 + k_1)})$. Thus we define 
\begin{equation}\label{sol alpha}
\widehat \alpha  := - \langle M_1 \rangle^{-1} 
( \langle g_1 \rangle + \langle M_2 g_2 \rangle + \langle M_3 g_3 \rangle ) \, .
\end{equation}
With this choice of $ \widehat \alpha$, by Lemma \ref{lemma splitting cal D omega},
 the equation \eqref{equazione psi hat} has the solution
\begin{equation}\label{sol psi}
\widehat \phi :=
[{\mathcal D}_\om^{(n)}]^{-1} \big( g_1 + M_1(\vphi)[\widehat \alpha] + M_2(\vphi) g_2 + M_3(\vphi) g_3 \big) \, . 
\end{equation}
In conclusion, we have obtained
a solution  $(\widehat \phi, \widehat y, \widehat w, \widehat \alpha)$ of the linear system \eqref{operatore inverso approssimato proiettato}. 

\begin{proposition}\label{prop: ai}
Assume \eqref{ansatz 0} (with $\mu = \mu(\mathtt b) + \sigma$) and \eqref{tame inverse}. 
Then, $\forall (\om, \kappa) \in \tLm_o $, $ \forall g := (g_1, g_2, g_3) $ satisfying \eqref{parita g1 g2 g3},
 the system \eqref{operatore inverso approssimato proiettato} has a solution 
$ {\mathbb D}_n^{-1} g := (\widehat \phi, \widehat y, \widehat w, \widehat \alpha ) $
where $(\widehat \phi, \widehat y, \widehat w, \widehat \alpha)$ are defined in 
\eqref{sol psi}, \eqref{soleta},  \eqref{normalw}, \eqref{sol alpha}, which satisfies \eqref{parity solution} and for any $s_0 \leq s \leq S$
\begin{equation} \label{stima T 0 b}
\| {\mathbb D}_n^{-1} g \|_s^{k_0, \gamma}
\leq_S \gamma^{-1} \big( \| g \|_{s + \sigma }^{k_0, \gamma} 
+  \| {\mathfrak I}_0  \|_{s + \mu(\mathtt b) + \sigma}^{k_0, \gamma}
 \| g \|_{s_0 + \sigma}^{k_0, \gamma}  \big).
\end{equation}
\end{proposition}

\begin{proof}
To shorten notation we write $\| \ \|_s$ instead of $\| \ \|_s^{k_0, \gamma}$. 
Recalling \eqref{cal M2}, by  Lemma \ref{lemma:Kapponi vari}, \eqref{tame inverse}, \eqref{ansatz 0}, \eqref{stima D omega (n)}, we get $ \| M_2 g \|_{s_0} + \| M_3 g \|_{s_0}  \leq C   \| g \|_{s_0 + \s} $. 
Then, by \eqref{sol alpha} and $\langle M_1 \rangle^{-1} = 1 + O(\e \gamma^{-(1 + k_1)}) = O(1)$, 
we deduce 
$ | \widehat \alpha|\leq C \| g \|_{s_0+ \s} $
and  \eqref{soleta}, \eqref{stima D omega (n)} imply
$ \| \widehat y \|_s \leq_s \gamma^{-1} \big( \| g \|_{s + \s}+ \| \fracchi_0 \|_{s + \mu ({\mathtt b}) + \s } \| g \|_{s_0}  \big)$.
The bound \eqref{stima T 0 b} is sharp for $ \widehat w $ because $ {\mathbf L}_\om^{-1} g_3 $ in  \eqref{normalw}
is estimated using \eqref{tame inverse}. Finally also $  \widehat \phi $ 
satisfies \eqref{stima T 0 b} using
\eqref{sol psi},  \eqref{cal M2}, \eqref{tame inverse}, \eqref{stima D omega (n)} and Lemma \ref{lemma:Kapponi vari}.
\end{proof}
Finally we prove that the operator 
\begin{equation}\label{definizione T} 
{\bf T}_0 := {\bf T}_0(i_0) := (D { \widetilde G}_\delta)(\vphi,0,0) \circ {\mathbb D}_n^{-1} \circ (D G_\delta) (\vphi,0,0)^{-1}
\end{equation}
is an almost-approximate right  inverse for $d_{i,\alpha} {\mathcal F}(i_0 )$ where
$$
 \widetilde{G}_\delta (\phi, y, w, \alpha) :=  \big( G_\delta (\phi, y, w), \alpha \big) 
 $$ 
 is the identity on the $ \alpha $-component. 
We denote the norm 
$$ 
\| (\phi, y, w, \alpha) \|_s^{k_0, \gamma} := 
  \max \{  \| (\phi, y, w) \|_s^{k_0, \gamma}, 
 | \alpha |^{k_0, \gamma}  \} \, .
$$

\begin{theorem}  \label{thm:stima inverso approssimato}
{\bf (Almost-approximate inverse)}
Assume the inversion assumption  \eqref{inversion assumption}-\eqref{tame inverse}. 
Then, there exists $ \bar \sigma := \bar \sigma(\tau, \nu, k_0) > 0 $ such that, 
if \eqref{ansatz 0} holds with $\mu = \mu(\mathtt b) + \bar \sigma $, then for all $ (\om, \kappa) \in \tLm_o $, 
for all $ g := (g_1, g_2, g_3) $ satisfying \eqref{parita g1 g2 g3},  
the operator $ {\bf T}_0 $ defined in \eqref{definizione T} satisfies,  for all $s_0 \leq s \leq S $, 
\begin{equation}\label{stima inverso approssimato 1}
\| {\bf T}_0 g \|_{s}^{k_0, \gamma} 
\leq_S  \gamma^{-1}  \big(\| g \|_{s + \bar \sigma}^{k_0, \gamma}  
+  \| {\mathfrak I}_0 \|_{s + \mu({\mathtt b}) +  \bar \sigma }^{k_0, \gamma}
\| g \|_{s_0 + \bar \sigma}^{k_0, \gamma}  \big)\, .
\end{equation}
 Moreover  ${\bf T}_0 $ is an almost-approximate inverse of $d_{i, \alpha} {\mathcal F}(i_0 )$, namely 
\begin{equation}\label{splitting per approximate inverse}
d_{i, \alpha} {\mathcal F}(i_0) \circ {\bf T}_0 - {\rm Id} = {\mathcal P} (i_0) + {\mathcal P}_\omega (i_0) + {\mathcal P}_\omega^\bot  (i_0)
\end{equation}
where, for all $ s_0 \leq s \leq S $, 
\begin{align}
 \| {\mathcal P} g \|_s^{k_0, \gamma} &
 \leq_S  \g^{-1} \Big( \| {\mathcal F}(i_0, \alpha_0) \|_{s_0 + \bar \sigma}^{k_0, \gamma} \| g \|_{s + \bar \sigma}^{k_0, \gamma}  
\nonumber\\
& \quad + \big\{ \| {\mathcal F}(i_0, \alpha_0) \|_{s + \bar \sigma}^{k_0, \gamma}
+\| {\mathcal F}(i_0, \alpha_0) \|_{s_0  + \bar \sigma}^{k_0, \gamma} \| {\mathfrak I}_0 \|_{s + \mu({\mathtt b}) + \bar \sigma}^{k_0, \gamma} \big\} \| g \|_{s_0 + \bar \sigma}^{k_0, \gamma} \Big) ,
\label{stima inverso approssimato 2} \\
\| {\mathcal P}_\omega g \|_s^{k_0, \gamma} & 
\leq_S  \e \g^{-2} N_{n - 1}^{- {\mathtt a}} \big( \| g \|_{s + \bar \sigma}^{k_0, \gamma} +\| \fracchi_0 \|_{s  + \mu({\mathtt b}) + \bar \sigma}^{k_0, \gamma}  \|  g \|_{s_0 + \bar \sigma}^{k_0, \gamma} \big)\,, \label{stima cal G omega} \\
\| {\mathcal P}_\omega^\bot g\|_{s_0}^{k_0, \gamma} & \leq_{S, b} 
\gamma^{- 1} K_n^{- b } \big( \| g \|_{s_0 + \bar \sigma + b }^{k_0, \gamma} +
\| \fracchi_0 \|_{s_0 + \mu({\mathtt b}) + b + \bar \sigma }^{k_0, \gamma} \big \| g \|_{s_0 + \bar \sigma}^{k_0, \gamma} \big)\,,\,\, \forall b > 0 \, ,   \label{stima cal G omega bot bassa} \\
\| {\mathcal P}_\omega^\bot g\|_s^{k_0, \gamma} &  \leq_S \gamma^{- 1}  
\big(\| g \|_{s + \bar \sigma}^{k_0, \gamma} + \| \fracchi_0 \|_{s + \mu({\mathtt b}) + \bar \sigma}^{k_0, \gamma}  \| g \|_{s_0 + \bar \sigma}^{k_0, \gamma} \big)\, . \label{stima cal G omega bot alta} 
\end{align}
\end{theorem}

\begin{proof}
The bound \eqref{stima inverso approssimato 1} 
follows from \eqref{definizione T}, \eqref{stima T 0 b},  \eqref{DG delta}.
By \eqref{operatorF}, since $ X_\mN $ does not depend on $ I $,  
and $ i_\d $ differs by $ i_0 $ only in the $ I $ component (see \eqref{y 0 - y delta}),  
we have \,
\be\label{verona 0}
\begin{aligned}
& d_{i, \alpha} {\mathcal F}(i_0 )   - d_{i, \alpha} {\mathcal F}(i_\delta )  \\
& = \e \int_0^1 \partial_I d_i X_P (\theta_0, I_\d + s (I_0 - I_\d), z_0) [I_0-I_\d,   \Pi [ \, \cdot \, ] \,  ] ds \\
& =: {\mathcal E}_0 =  {\mathcal E}_0^{(n)}   + {\mathcal E}_0^{(n), \bot}
\end{aligned}
\ee
where $ \Pi $ is the projection $ (\widehat \imath, \widehat \a ) \mapsto \widehat \imath $ and,
recalling \eqref{toro I delta (n)}, \eqref{toro I delta (n) bot},  
\begin{align}\label{verona 00}
{\mathcal E}_0^{(n)} &  := \e \int_0^1 \partial_I d_i X_P (\theta_0, I_\d + s (I_0 - I_\d), z_0) [I_0-I_\d^{(n)},   \Pi [ \, \cdot \, ] \,  ] ds\,, \\
{\mathcal E}_0^{(n), \bot} & := - \e \int_0^1 \partial_I d_i X_P (\theta_0, I_\d + s (I_0 - I_\d), z_0) [I_\d^{(n), \bot},   \Pi [ \, \cdot \, ] \,  ] ds\, .\label{verona 01}
\end{align}
Denote by $  {\mathtt u} := (\phi, y, w) $  the symplectic coordinates induced by $ G_\d $ in \eqref{trasformazione modificata simplettica}. 
Under the symplectic map $G_\delta $, the nonlinear operator ${\mathcal F}$ in \eqref{operatorF} is transformed into 
\be \label{trasfo imp}
{\mathcal F}(G_\delta(  {\mathtt u} (\vphi) ), \alpha ) 
= D G_\delta( {\mathtt u}  (\vphi) ) \big(  {\mathcal D}_\om {\mathtt u} (\vphi) - X_{K_\alpha} ( {\mathtt u} (\vphi), \alpha)  \big) 
\ee
where $ K_{\alpha} = H_{\alpha} \circ G_\delta $, see \eqref{new-Hamilt-K} and 
\eqref{sistema dopo trasformazione inverso approssimato}. 
Differentiating  \eqref{trasfo imp} at the trivial torus 
$ {\mathtt u}_\delta (\vphi) = G_\delta^{-1}(i_\delta) (\vphi) = (\ph, 0 , 0 ) $, 
at  $ \alpha = \alpha_0 $, 
we get
\begin{align} \label{verona 2}
& d_{i , \alpha} {\mathcal F}(i_\delta ) 
=   D G_\delta( {\mathtt u}_\delta) 
\big( \Dom 
- d_{\mathtt u, \alpha} X_{K_\alpha}( {\mathtt u}_\delta, \alpha_0) 
\big) D \widetilde G_\d ( {\mathtt u}_\d)^{-1}
+ {\mathcal E}_1 \,,
\\
& {\mathcal E}_1 
:=  
D^2 G_\delta( {\mathtt u}_\delta) \big[ D G_\delta( {\mathtt u}_\delta)^{-1} {\mathcal F}(i_\delta, \alpha_0), \,  D G_\d({\mathtt u}_\d)^{-1} 
 \Pi [ \, \cdot \, ] \,  \big] = {\mathcal E}_1^{(n)} + {\mathcal E}_1^{(n), \bot} \,   \label{verona 2 0} 
\end{align}
where, recalling the splitting ${\mathcal F}(i_\delta, \alpha_0) = Z_\delta = Z_\delta^{(n)} + Z_\delta^{(n), \bot}$ in 
Lemma 
\ref{toro isotropico modificato}, we have 
\begin{align}\label{verona 20}
{\mathcal E}_1^{(n)} & := D^2 G_\delta( {\mathtt u}_\delta) \big[ D G_\delta( {\mathtt u}_\delta)^{-1} Z_\delta^{(n)}, \,  D G_\d({\mathtt u}_\d)^{-1} \Pi [ \, \cdot \, ] \,  \big] \\
{\mathcal E}_1^{(n), \bot} & := D^2 G_\delta( {\mathtt u}_\delta) \big[ D G_\delta( {\mathtt u}_\delta)^{-1} Z_\delta^{(n), \bot}, \,  D G_\d({\mathtt u}_\d)^{-1} \Pi [ \, \cdot \, ] \,  \big] \label{verona 21} \, . 
\end{align}
In expanded form $ \Dom  - d_{\mathtt u, \alpha} X_{K_\a}( {\mathtt u}_\delta, \alpha_0) $ is provided in \eqref{lin idelta}.
By \eqref{operatore inverso approssimato}, \eqref{splitting mathbb D}, \eqref{definizione mathbb Dn}, \eqref{Lomega def}, 
\eqref{inversion assumption} and Lemma \ref{coefficienti nuovi} we split  
\begin{equation}\label{splitting linearizzato nuove coordinate}
\om \! \cdot \! \pa_\vphi 
- d_{\mathtt u, \alpha} X_K( {\mathtt u}_\delta, \alpha_0) 
= \mathbb{D}_n  + {\mathbb D}_n^\bot 
+ R_Z^{(n)} + R_Z^{(n), \bot}   + {\mathbb R}_\omega+ 
{\mathbb R}_\omega^\bot  
\end{equation}
where  
$$
\begin{aligned}
& R_Z^{(n)} [  \widehat \phi, \widehat y, \widehat w, \widehat \alpha]  := \\
&
\begin{pmatrix}
 - \partial_\phi K_{10}^{(n)}(\vphi, \alpha_0) [\widehat \phi ] \\
 \partial_{\phi \phi} K_{00}^{(n)} (\vphi, \alpha_0) [ \widehat \phi ] + 
 [\partial_\phi K_{10}^{(n)}(\vphi, \alpha_0)]^T \widehat y + 
 [\partial_\phi K_{01}^{(n)}(\vphi, \alpha_0)]^T \widehat w  \\
 - J \{ \partial_{\phi} K_{01}^{(n)}(\vphi, \alpha_0)[ \widehat \phi ] \}
 \end{pmatrix}\,,
 \end{aligned}
$$
$$
\begin{aligned}
& R_Z^{(n), \bot} [  \widehat \phi, \widehat y, \widehat w, \widehat \alpha]
:= \\ 
& \begin{pmatrix}
 - \partial_\phi K_{10}^{(n), \bot}(\vphi, \alpha_0) [\widehat \phi ] \\
 \partial_{\phi \phi} K_{00}^{(n), \bot} (\vphi, \alpha_0) [ \widehat \phi ] + 
 [\partial_\phi K_{10}^{(n), \bot}(\vphi, \alpha_0)]^T \widehat y + 
 [\partial_\phi K_{01}^{(n), \bot}(\vphi, \alpha_0)]^T \widehat w  \\
 - J \{ \partial_{\phi} K_{01}^{(n), \bot}(\vphi, \alpha_0)[ \widehat \phi ] \}
 \end{pmatrix}
 \end{aligned}
$$
and 
$$
{\mathbb R}_\omega[\widehat \phi, \widehat y, \widehat w, \widehat \alpha] := \begin{pmatrix}
0 \\
0 \\
{\mathbf R}_\omega [\widehat w]
\end{pmatrix}\,,\qquad {\mathbb R}_\omega^\bot[\widehat \phi, \widehat y , \widehat w, \widehat \alpha] := \begin{pmatrix}
0 \\
0 \\
{\mathbf R}_\omega^\bot[\widehat w]
\end{pmatrix}\,.
$$
By \eqref{verona 0}, \eqref{verona 2}, \eqref{verona 2 0}, \eqref{splitting linearizzato nuove coordinate} we get the decomposition
\begin{align} 
 d_{i, \alpha} {\mathcal F}(i_0 ) 
& = D G_\delta({\mathtt u}_\delta) \circ {\mathbb D}_n \circ D {\widetilde G}_\delta ({\mathtt u}_\delta)^{-1} + {\mathcal E}^{(n)} + {\mathcal E}_\omega + {\mathcal E}_\omega^\bot  \label{E2}
 \end{align}
where 
\begin{equation}\label{cal E (n) omega}
\begin{aligned}
{\mathcal E}^{(n)} & := {\mathcal E}_0^{(n)} + {\mathcal E}_1^{(n)} + D G_\delta ( {\mathtt u}_\delta)R_Z^{(n)} 
D {\widetilde G}_\delta ({\mathtt u}_\delta)^{-1}\,, \\
  {\mathcal E}_\omega & := D G_\delta ( {\mathtt u}_\delta)   {\mathbb R}_\omega    D {\widetilde G}_\delta ({\mathtt u}_\delta)^{-1}\,,
\end{aligned}
\end{equation}
\begin{equation}\label{cal E omega bot}
   {\mathcal E}_\omega^\bot := {\mathcal E}_0^{(n), \bot}  + {\mathcal E}_1^{(n), \bot} +  D G_\delta( {\mathtt u}_\delta)  [{\mathbb R}_\omega^\bot + {\mathbb D}_n^\bot + R_Z^{(n), \bot} ] 
D {\widetilde G}_\delta ({\mathtt u}_\delta)^{-1} \, . 
\end{equation}
Applying $ {\bf T}_0 $ defined in \eqref{definizione T} to the right in \eqref{E2} (recall that $ {\mathtt u}_\delta (\vphi) := (\vphi, 0, 0 ) $), 
since $ {\mathbb D}_n \circ  {\mathbb D}_n^{-1} = {\rm Id} $ (Proposition \ref{prop: ai}), 
we get 
\begin{align*}
& \qquad \qquad \quad d_{i, \alpha} {\mathcal F}(i_0 ) \circ {\bf T}_0  - {\rm Id} 
={\mathcal P} + {\mathcal P}_\omega + {\mathcal P}_\omega^\bot\,, \\
& 
 {\mathcal P} := {\mathcal E}^{(n)} \circ {\bf T}_0, \quad
 {\mathcal P}_\omega := \mE_\omega \circ {\bf T}_0 \, , \quad {\mathcal P}_\omega^\bot :=  \mE_\omega^\bot \circ {\bf T}_0 \, . 
\end{align*}
Lemma \ref{lemma quantitativo forma normale} and \eqref{ansatz 0}, \eqref{K 00 10 01}, \eqref{2015-2},  \eqref{stima y - y delta}, \eqref{stima toro modificato}, 
 \eqref{DG delta}-\eqref{DG2 delta},  imply the estimate 
\begin{equation}\label{stima parte trascurata 2}
\begin{aligned}
\| {\mathcal E}^{(n)} [\, \widehat \imath, \widehat \alpha \, ] \|_s^{k_0, \gamma} & 
\leq_s  \| Z \|_{s_0 + \s}^{k_0, \gamma} \| \widehat \imath \|_{s + \s}^{k_0, \gamma} +  
\| Z \|_{s + \s}^{k_0, \gamma} \| \widehat \imath \|_{s_0 + \s}^{k_0, \gamma} \\ 
 & \ \  + 
\| Z \|_{s_0 + \s}^{k_0, \gamma} \| \widehat \imath \|_{s_0 + \s}^{k_0, \gamma} \| \fracchi_0 \|_{s+\s}^{k_0, \gamma} 
\end{aligned}
\end{equation}
where $ Z := \mF(i_0, \alpha_0)$, recall \eqref{def Zetone}. 
Then \eqref{stima inverso approssimato 2} follows from 
\eqref{stima inverso approssimato 1}, \eqref{stima parte trascurata 2}, 
\eqref{ansatz 0}. The estimates \eqref{stima cal G omega}, 
\eqref{stima cal G omega bot bassa}, \eqref{stima cal G omega bot alta}
 follow by \eqref{stima R omega corsivo}-\eqref{stima R omega bot corsivo alta}, 
\eqref{stima inverso approssimato 1},  \eqref{DG delta}, \eqref{2015-2}, \eqref{stima y - y delta bot}, \eqref{stima toro modificato bot}, \eqref{K 00 10 01 bot alta}, \eqref{K 00 10 01 bot bassa}, \eqref{ansatz 0}, \eqref{stime cal D omega (n) bot}.  
 \end{proof}

\chapter{The linearized operator in the normal directions}\label{linearizzato siti normali}

In order to write an explicit  expression of the linear operator 
$\mL_\om$ defined in \eqref{Lomega def}
we compute the quadratic term $ \frac12 ( K_{02}(\phi) w, w )_{L^2(\T_x)} $ 
in the Taylor expansion 
of  the Hamiltonian $ K_\a (\phi, 0, w)$ in  \eqref{KHG}. 

\begin{lemma}  \label{thm:Lin+FBR}
The operator  $ K_{02}(\phi) $ reads 
\begin{equation}\label{K 02}
K_{02}(\phi) = \Pi_{{\mathbb S}^+}^\bot \partial_u \nabla_u H(T_\delta(\phi)) + \e R(\phi) 
\end{equation}
where $ H $ is the water-waves Hamiltonian defined in \eqref{Hamiltonian}, evaluated at the torus 
\begin{equation}\label{T delta}
T_\delta(\phi) := \e A(i_\delta(\phi)) = \e A(\theta_0(\phi), I_\delta(\phi), z_0(\phi) ) =
\e  v (\theta_0 (\phi), I_\d(\phi)) +  \e z_0 (\phi) 
\end{equation}
with $ A (\theta, I, z ) $, $ v (\teta, I )$ defined in \eqref{definizione A}.
The operator $ K_{02}(\phi) $ is even and reversible. 
The remainder $ R(\phi) $ has the ``finite dimensional" form 
\begin{equation}\label{forma buona resto}
R(\phi)[h] = {\mathop \sum}_{j = 1}^\nu \big(h\,,\,g_j \big)_{L^2_x} \chi_j\,, \quad \forall h \in H_{{\mathbb S}^+}^\bot \, ,  
\end{equation}
for functions $ g_j, \chi_j \in H_{{\mathbb S}^+}^\bot  $ which satisfy the tame estimates: 
for some $ \sigma:= \sigma(\tau, \nu) > 0 $, 
$ \forall s \geq s_0 $, 
\begin{equation}\label{stime gj chij}
\begin{aligned}
 \| g_j\|_s^{k_0, \gamma} +\| \chi_j\|_s^{k_0, \gamma} & \leq_s 1 + \| \fracchi_\delta\|_{s + \s}^{k_0, \gamma}\,, \\
   \| \partial_i g_j[\widehat \imath]\|_s +\| \partial_i \chi_j
[\widehat \imath]\|_s & \leq_s \| \widehat \imath \|_{s + \s}+ \| \fracchi_\delta\|_{s + \s} \| \widehat \imath\|_{s_0 + \s}\, . 
\end{aligned}
\end{equation}
\end{lemma}

\begin{proof}
The operator $ K_{02}(\phi)$ is
\begin{equation}\label{trasformata forma normale}
\begin{aligned}
K_{02}(\phi) = \partial_w \nabla_w K_\alpha(\phi, 0, 0) & 
= \partial_w \nabla_w (H_\alpha \circ G_\delta)(\phi, 0, 0) \\
& = \Omega_{| H_{{\mathbb S}^+}^\bot}  + \e
\partial_w \nabla_w (P \circ G_\delta)(\phi, 0, 0)  
\end{aligned}
\ee
where  $H_\alpha = {\mathcal N}_\alpha + \e P $ is defined in \eqref{H alpha} 
and  $\Omega$ in \eqref{definizione Omega}. 
Differentiating with respect to $w$ the Hamiltonian 
$$ 
(P \circ G_\delta)(\phi, y, w) =  P(\theta_0(\phi), I_\delta(\phi) + L_1(\phi) y + L_2(\phi) w, z_0(\phi) + w) 
$$
where (see \eqref{trasformazione modificata simplettica}) 
$ L_1(\phi) := [\partial_\phi \theta_0(\phi)]^{- T} $, $ L_2(\phi) := - [\partial_\theta \tilde z_0(\theta_0(\phi))]^{T} J $, 
we get 
$$ 
\nabla_w (P \circ G_\delta)(\phi, y, w) =
L_2(\phi)^T \partial_I P(G_\delta(\phi, y, w)) + \nabla_z P(G_\delta(\phi, y, w)) \, ,
$$ 
and therefore 
\begin{equation}\label{P G delta}
\begin{aligned}
& \partial_w \nabla_w (P \circ G_\delta)(\phi, 0, 0) = \partial_z \nabla_z P(i_\delta(\phi)) + R(\phi)  \qquad {\rm with} \qquad  \\
& R(\phi) := R_1(\phi) + R_2(\phi) + R_3(\phi)\,,  \\
& R_1(\phi) := L_2(\phi)^T \partial_{II} P(i_\delta(\phi)) L_2(\phi)\,,\quad R_2(\phi):=  L_2(\phi)^T 
\pa_z  \partial_{I} P(i_\delta(\phi)) \,, \\
& R_3(\phi) := \partial_I \nabla_z P(i_\delta(\phi))L_2(\phi) \, . 
\end{aligned}
\end{equation}
Each operator $ R_1, R_2, R_3 $ has the finite dimensional form \eqref{forma buona resto} because
it is the composition of at least one operator with finite rank $ \R^\nu $. For example,
writing the operator $L_2(\phi) : H_{{\mathbb S}^+}^\bot \to \R^\nu$ as  
$ L_2(\phi)[h] = \sum_{j = 1}^\nu \big(h\,,\, L_2(\phi)^T[ {\underline e}_j] \big)_{L^2_x} {\underline e}_j $, $ \forall h \in H_{{\mathbb S}^+}^\bot $, we get 
$$
R_1(\phi)[h] = {\mathop \sum}_{j = 1}^\nu \big(h\,,\, L_2(\phi)^T[{\underline e}_j] \big)_{L^2_x} 
A_1[{\underline e}_j]\,,\quad A_1 := L_2(\phi)^T \partial_{II} P(i_\delta(\phi))\,.
$$
Similarly $  R_3(\phi)[h] = 
{\mathop \sum}_{j = 1}^\nu \big(h, L_2(\phi)^T [{\underline e}_j] \big)_{L^2_x} A_3[{\underline e}_j] $ with 
$ A_3 := \partial_y \nabla_z P(i_\delta(\phi))$, and
since $ A_2 := \pa_z \pa_{I} P(i_\delta(\phi)) : H_{{\mathbb S}^+}^\bot \to \R^\nu $, we get 
$$
R_2(\phi)[h]  = {\mathop \sum}_{j = 1}^\nu \big(h, A_2^T[{\underline e}_j] \big)_{L^2_x} L_2(\phi)^T[{\underline e}_j] \, .
$$
 The estimate \eqref{stime gj chij} follows by Lemma \ref{lemma quantitativo forma normale}.

By \eqref{trasformata forma normale}, \eqref{P G delta}, and \eqref{definizione cal N P}, 
\eqref{definizione A},  
 \eqref{Hamiltonian:rescaled}, \eqref{Hamiltonian linear},   we get 
\begin{align*}
K_{02}(\phi) & = \Omega_{| H_{{\mathbb S}^+}^\bot}  + \e \partial_z \nabla_z P(i_\delta(\phi)) + \e R(\phi)  \\
& = \Omega_{| H_{{\mathbb S}^+}^\bot}  + \e \Pi_{{\mathbb S}^+}^\bot \partial_u \nabla_u P_\e(A(i_\delta(\phi))) + \e R(\phi) \\
& =  \Pi_{{\mathbb S}^+}^\bot \partial_u \nabla_u {\mathcal H}_\e(A(i_\delta(\phi))) + \e R(\phi) 
\end{align*}
which proves \eqref{K 02} because $ A(i_\delta(\phi)) = T_\delta(\phi) $, see \eqref{T delta}.
\end{proof}

By Lemma \ref{thm:Lin+FBR} the linear operator $ {\mathcal L}_\om $ defined in \eqref{Lomega def} has the form 
\be\label{representation Lom}
{\mathcal L}_\om =  \Pi_{{\mathbb S}^+}^\bot ( {\mathcal L} + \e R)_{| H_{{\mathbb S}^+}^\bot}  \qquad {\rm where} \qquad {\mathcal L} := 
\Dom {\mathbb I}_2  - J \partial_u  \nabla_u 
H (T_\delta(\vphi))
\ee
is obtained linearizing the original water waves system \eqref{WW}, \eqref{HS} at the torus
$ u = (\eta, \psi) = T_\d(\vphi) $ 
defined in \eqref{T delta}, changing $ \pa_t \rightsquigarrow \Dom $, and 
denoting the $ 2 \times 2 $-identity matrix  by 
$$
 {\mathbb I}_2 := \begin{pmatrix} 
{\rm Id} & 0 \\
0  &  {\rm Id} \\
\end{pmatrix} \,. 
$$
Using formula \eqref{formula shape der} the linearized operator $ {\mathcal L} $ is 
\begin{equation}  \label{linearized vero}
{\mathcal L} = \Dom {\mathbb I}_2 + 
\begin{pmatrix} 
\partial_x V + G(\eta) B & - G(\eta) \\
(1 + B V_x) + B G(\eta) B -  \kappa \pa_x c \pa_x \  &  V \partial_x - B G(\eta) \\
\end{pmatrix} 
\end{equation}
where the functions $ B:=  B(\vphi, x ) $, $ V := V(\vphi, x ) $ are defined by \eqref{def B V} 
with $ (\eta, \psi)  = (\eta (\vphi, x), \psi (\vphi, x) ) = T_\d(\vphi) $  defined in \eqref{T delta},  and 
\begin{equation} \label{def c}
c := c(\vphi, x) := (1 + \eta_x^2)^{-3/2}.
\end{equation}
By \eqref{T delta}, \eqref{definizione A}, \eqref{parity solution}
the function $ u = (\eta, \psi) = T_\d(\vphi) $ satisfies the parities 
($ \even(\vphi)$-$\even(x), \odd(\vphi)$-$\even(x) $),  
and  $c $ is $ \even(\vphi)$-$\even(x) $,
$B \in \odd(\vphi)$-$\even(x)$, $V = \odd(\ph), \odd(x)$. 
The operators $ {\mathcal L}_\om $ and $ {\mathcal L} $ are real, even and reversible. 
\\[1mm]
{\bf Notation.} In \eqref{linearized vero} and hereafter 
any function $a$ is identified with the corresponding multiplication operators $h \mapsto ah$, and, where there is no parenthesis, composition of operators is understood. For example, 
$\pa_x c \pa_x$ means: $h \mapsto \pa_x (c \pa_x h)$.

\smallskip

In the next sections we focus on reducing the linear operator $ {\mathcal L} $  in \eqref{linearized vero} 
to constant coefficients up to a pseudo-differential operator of order $ 0 $ (and up to 
a small remainder supported on the high modes). The finite dimensional remainder $ \e  R $
transforms under conjugation into an operator of the same form (Lemma \ref{lemma forma buona resto}) and therefore it will be dealt 
only once at the end of the section. 

\smallskip

For the sequel we will always assume the following ansatz in ``low norm" (that will be satisfied by the approximate solutions along the Nash-Moser iteration): 
for some $ \mu := \mu(\tau, \nu) > 0$, $\gamma \in (0, 1) $,  
\begin{equation}\label{ansatz I delta}
\| \fracchi_0 \|_{s_0 + \mu}^{k_0, \gamma} \leq 1   \ \, , 
\qquad \text{and  so, by } \eqref{2015-2}, \ \| \fracchi_\d \|_{s_0 + \mu}^{k_0, \gamma} \leq 2 \, . 
\end{equation}
Actually $\mu := \mu(\mathtt b) + \sigma_1$, where $\mu(\mathtt b)$ is defined in \eqref{definizione bf c (beta)} and $\sigma_1$ in \eqref{costanti nash moser 2}, is fixed in the Nash Moser iteration of section \ref{sec:NM} (see also \eqref{ansatz induttivi nell'iterazione}).
In order to estimate the  variation of the eigenvalues with respect to the approximate invariant torus, we need also to estimate the derivatives with respect to the torus $i(\vphi)$ in another low norm $\| \ \|_{s_1}$, for all the Sobolev indices $s_1$ such that  
\begin{equation}\label{vincolo s1 derivate i}
s_1 + \sigma \leq s_0 + \mu\,, \quad \text{for \,\,some} \quad \sigma := \sigma(\tau, \nu)> 0\,.
\end{equation}
Thus by \eqref{ansatz I delta} we have 
\be\label{estimate:low norm for der}
\|\fracchi_0 \|_{s_1 + \sigma}^{k_0, \gamma} \leq 1 
\qquad \text{and  so, by } \eqref{2015-2}, \ 
\|\fracchi_\d \|_{s_1 + \sigma}^{k_0, \gamma} \leq 2 \,.
\ee
 The constants $\mu$ and $\sigma$ represent the {\it loss of derivatives} at any step of the reduction procedure of this section and it  possibly
increases along the (finitely many)  steps of this reduction procedure. In  Lemma \ref{lem:tame iniziale} we fix the largest loss of derivatives $\sigma := \sigma(\mathtt b)$. 
 
\begin{remark}
Let us shortly motivate the role of the intermediate Sobolev index $ s_1 $. 
In the  reducibility scheme in section \ref{sec: reducibility} 
we require that  the remainders ${\bf R}_0$, ${\bf Q}_0$ satisfy the estimates \eqref{derivate i resti prima della riducibilita-s0}.
In Lemma \ref{lem:tame iniziale}  we take
$ {\bf R}_0 := {\bf R}_M^{(3)} $, $ {\bf Q}_0 := {\bf Q}_M^{(3)} $ defined 
in Proposition \ref{prop: sintesi linearized} and so we want that
\eqref{derivate i resti prima della riducibilita} holds with $ s_1 =  s_0 $. 
For that we need to estimate,  along  section \ref{linearizzato siti normali}, 
the derivatives $ \pa_i $ of functions, operators, etc, in  intermediate $ \| \ \|_{s_1} $ norms, i.e. 
for $ s_1 $ which satisfies \eqref{vincolo s1 derivate i}. 
\end{remark}


As a consequence of Moser composition Lemma \ref{Moser norme pesate}, 
the Sobolev norm of $ u = T_\d $ (see \eqref{T delta}) satisfies
\be\label{tame Tdelta}
\| u \|_s^{k_0,\gamma} = \| \eta \|_s^{k_0,\gamma}  + 
\| \psi \|_s^{k_0,\gamma} \leq \e C(s)  \big(  1 + \| \fracchi_0 \|_{s}^{k_0, \gamma})  \, , \quad  \forall s \geq s_0   
\ee
(the funtion $ A $  defined in \eqref{definizione A} is smooth). 
Similarly 
\begin{equation}\label{derivata i T delta}
\| \partial_i u [\hat \imath] \|_{s_1} \leq_{s_1} \e \| \hat \imath\|_{s_1} \, .
\end{equation}
We remark that it would be sufficient to give Lipschitz estimates of $ u $ (and of operators, transformations, eigenvalues) 
with respect to the variable $ i $, namely to estimate the finite difference $ \Delta_{12} u := u(i_1) - u (i_2)$
in terms of the difference $ \| i_1 - i_2 \|_{s_1 + \s } $, 
but for convenience we compute the derivative $ \partial_i $. 
We repeat that  it is sufficient to estimate the derivatives (or the finite difference) 
with respect to  $ i $ only in low norm $ s_1 $ is because this information 
 is only needed  to control the variation of the eigenvalues with respect to $ i $, see remark \ref{remark:dipendenza da omega non richiesta}.

\smallskip

Finally we recall that $ \fracchi_0 := \fracchi_0 (\om, \kappa)  $ 
is defined for all  $ \om \in \R^\nu $ and $ \kappa \in [\kappa_1, \kappa_2]$ 
by the extension procedure of section \ref{sec:NM}. 
Moreover all the functions appearing in 
$ {\mathcal L} $ in \eqref{linearized vero} are $ {\mathcal C}^\infty $ in $ (\vphi, x) $ as 
the approximate torus 
$ u = (\eta, \psi) =  T_\d (\vphi)  $. 
This enables to  use directly pseudo-differential operator theory as reminded in section \ref{sec:prelim}.

\section{Linearized good unknown of Alinhac} \label{sec:linearized operator}

We first conjugate the linearized operator $ {\mathcal L} $ in \eqref{linearized vero} by the change of variable
$$ 
\mZ := \begin{pmatrix}  1 & 0 \\ B & 1  \end{pmatrix}\, , \qquad \mZ^{-1} = \begin{pmatrix} 1 & 0 \\ - B & 1  \end{pmatrix} 
$$
obtaining
\begin{equation} \label{Con1}
{\mathcal L}_0 := \mZ^{-1} \mL \mZ = \Dom  {\mathbb I}_2 + \begin{pmatrix}
\partial_x V \  & \ - G(\eta) \\ 
a -  \kappa \pa_x c \pa_x \ &  V \partial_x 
\end{pmatrix}
\end{equation} 
where $ a $ is the function  
\begin{equation}  \label{a}
a := a(\vphi, x) =  1 + \Dom B + V B_x \, .
\end{equation} 
The matrix $ \mZ $ amounts to introduce  (a linearized version of)  the ``good unknown of Alinhac''.

\begin{lemma}  \label{lemma:remainder mR0}
The maps $\mZ^{\pm 1} -  {\rm Id} $ are even, reversibility preserving and ${\mathcal D}^{k_0}$-tame with tame constant satisfying, for all $s_0 \leq s \leq S $,  
\begin{equation}  \label{est Z-Id}
{\mathfrak M}_{\mZ^{\pm 1} - {\rm Id}} (s)\,,\,  {\mathfrak M}_{(\mZ^{\pm 1} - {\rm Id})^*} (s)
\leq_s \e \big( 1+ \| \fracchi_0 \|_{s + \sigma}^{k_0, \gamma} \big) \, . 
\end{equation}
The operator $ {\mathcal L}_0 $ is even and reversible. 
There is $\sigma := \sigma(\tau, \nu) > 0 $ such that    the functions 
\begin{equation}\label{stima V B a c}
\begin{aligned}
& \|a - 1 \|_s^{k_0, \gamma} + \| V \|_s^{k_0, \gamma} + \| B\|_s^{k_0, \gamma} \leq_s \e \big(1 + \| \fracchi_0 \|_{s + \sigma}^{k_0, \gamma} \big)\,, \\
&  \| c - 1 \|_s^{k_0, \gamma} \leq_s \e^2 \big(1 + \| \fracchi_0 \|_{s + \sigma}^{k_0, \gamma} \big)\, .
\end{aligned}
\end{equation}
Moreover 
\begin{align}\label{stima derivate i primo step}
& \| \partial_i a  [\hat \imath] \|_{s_1} + \| \partial_i V [\hat \imath] \|_{s_1} + \|\partial_i B [\hat \imath] \|_{s_1} \leq_{s_1}  \e \| \hat \imath \|_{s_1 + \sigma} \,, \  \| \partial_i c [\hat \imath] \|_{s_1} \leq_{s_1}\e^2  \| \hat \imath \|_{s_1 + \sigma} \\
& \label{derivate in i cal Z}
\| \partial_i ({\mathcal Z}^{\pm 1} [\hat \imath]) h \|_{s_1}\,,\,\| \partial_i (({\mathcal Z}^{\pm 1})^* [\hat \imath]) h \|_{s_1} \leq_{s_1}  \e \| \hat \imath\|_{s_1 + \sigma} \|  h \|_{s_1 }\,.
\end{align}
\end{lemma}

\begin{proof}
The estimate \eqref{stima V B a c}, follows by the explicit expressions of $a, V, B, c$ in \eqref{a}, \eqref{def B V}, \eqref{def c}, by applying Lemma \ref{Moser norme pesate} and the estimates \eqref{interpolazione C k0}, \eqref{estimate DN},  \eqref{interpolazione parametri operatore funzioni} and Lemma \ref{lemma operatore e funzioni dipendenti da parametro}.
The operators $ \mZ^{\pm 1}  $ are reversibility preserving because $ B $ is odd$ \vphi $. 
The estimate \eqref{est Z-Id} holds by \eqref{norma a moltiplicazione}, \eqref{interpolazione parametri operatore funzioni}, \eqref{stima V B a c} and since the adjoint  
$
{\mathcal Z}^* = \begin{pmatrix}
1 & B \\
0 & 1
\end{pmatrix} $. 
The estimates involving $\mZ^{- 1}$ follow similarly. The estimate \eqref{stima derivate i primo step} follows by differentiating the explicit expressions of $a$, $B$, $V$, $c$ in \eqref{a}, \eqref{def B V}, \eqref{def c}, by applying Lemma \ref{Moser norme pesate}, \eqref{formula shape der}, \eqref{estimate DN}, \eqref{interpolazione C k0} and \eqref{derivata i T delta}.
The estimates \eqref{derivate in i cal Z} follow
by the estimate of $\partial_i B$ in \eqref{stima derivate i primo step} and  \eqref{interpolazione C k0}. 
\end{proof}

\section{Symmetrization and space reduction of the highest order} 
\label{sec:changes}

The aim of this section is to conjugate the linear operator  $ \mL_0  $ 
in \eqref{Con1} to the operator $ {\mathcal L}_3 $
in \eqref{def mL4} whose  
coefficient $ m_3 (\vphi) $  of the  highest order
 is independent of the space variable.  By 
 \eqref{sviluppo Geta}  we first rewrite  
\be\label{forma L0}
\mL_0 = \Dom  {\mathbb I}_2 + \begin{pmatrix}
 V \partial_x + V_x \  & \ - |D_x| - \mR_G \\ 
a -  \kappa c \pa_{xx} -  \kappa c_x \pa_x \ &   V \partial_x 
\end{pmatrix} \, . 
\ee
{\bf Step 1}. 
We  first conjugate  $ \mL_0 $ 
with a change of  variable 
\be\label{change-of-variable}
(\mB h)(\vphi,x) := h(\vphi, x+\b(\vphi,x))  
\ee
induced by  a $ \vphi $-dependent family of diffeomorphisms of the torus 
\be\label{diffeo-torus} 
y = x + \b(\vphi,x)
\qquad   \Leftrightarrow  \qquad 
x = y + \tilde\b(\vphi,y) 
\ee
where $\b(\vphi,x)$ is a small periodic function to be determined.  
Under the change of variable \eqref{change-of-variable} the differential operators $ \pa_x $, $ \pa_{xx} $, 
$ \Dom $, and the multiplication operator by $ a $, transform into
\begin{equation}\label{conjB1}
\begin{aligned}
& \mB^{-1} \pa_x \mB  = \{ \mB^{-1}(1 + \b_x) \} \pa_y, \\
 & \mB^{-1} \pa_{xx} \mB  = \{ \mB^{-1}(1 + \b_x) \}^2 \pa_{yy} + (\mB^{-1} \b_{xx}) \pa_y, 
\end{aligned}
\end{equation}
\begin{equation}\label{conjB2}
\begin{aligned}
 \mB^{-1} \Dom \mB  = \Dom + (\mB^{-1} \Dom \b) \pa_y \, , \quad  \mB^{-1} a \mB = (\mB^{-1} a) 
\end{aligned}
\end{equation}
Moreover, using \eqref{conjB1}, 
\begin{align}
\mB^{-1} |D_x| \mB  & 
= \mB^{-1} \pa_x \mH \mB  = (\mB^{-1} \pa_x \mB) (\mB^{-1} \mH \mB) \nonumber \\
&   = \{ \mB^{-1}(1 + \b_x) \} \pa_y [ \mH + (\mB^{-1} \mH \mB - \mH) ] \nonumber \\ 
& = \{ \mB^{-1}(1 + \b_x) \} |D_y| + \mR_{\mB} \label{BmodDB}
\end{align}
where, by Lemma \ref{coniugio Hilbert},  
\be\label{Resto:RB}
\mR_{\mB} := \{ \mB^{-1}(1 + \b_x) \} \pa_y (\mB^{-1} \mH \mB - \mH) \in OPS^{-\infty} \,.
\ee
Thus, by \eqref{conjB1}-\eqref{BmodDB}, 
the operator $ {\mathcal L}_0 $ in \eqref{forma L0} transforms into
\be\label{Op L1}
\mL_1 := \mB^{-1} \mL_0 \mB  =  \Dom  {\mathbb I}_2 + 
\begin{pmatrix}
 a_1 \pa_y + a_2   &  - a_3 |D_y| + \mR_1 \\ 
- \kappa a_4 \pa_{yy} - \kappa a_5 \pa_y + a_6  &   a_1 \pa_y
\end{pmatrix}
\ee
where $ a_i = a_i(\vphi,y)$ are 
\begin{equation}\label{a1 a2 a3}
\begin{aligned}
& a_1  := \mB^{-1} [ \Dom \b + V(1 + \b_x)], \quad 
a_2  := \mB^{-1}(V_x),  \\
& a_3  := \mB^{-1}(1 + \b_x),
\end{aligned}
\end{equation}
\begin{equation}\label{a4 a5 a6}
\begin{aligned}
& a_4  := \mB^{-1} [c(1 + \b_x)^2], \quad 
a_5  := \mB^{-1} [c \b_{xx} + c_x (1 + \b_x)],  \\
& a_6  := \mB^{-1} a, 
\end{aligned}
\end{equation}
and 
\be\label{defR1}
\mR_1 := - \mR_\mB - \mB^{-1} \mR_G \mB \in OPS^{-\infty} \, . 
\ee
We look for  $\b(\vphi, x) $ such that 
\begin{equation}  \label{proportional}
(a_3 a_4) (\vphi, y)  = m(\vphi) 
\end{equation} 
for some function $m(\vphi)$, independent of the space variable $y$.
By \eqref{a1 a2 a3}-\eqref{a4 a5 a6},  the equation \eqref{proportional} is
$$
c(\vphi, x) ( 1 + \beta_x(\vphi, x) )^3 = m(\vphi) 
$$
which is solved by 
\begin{equation}\label{beta lambda3}
m(\vphi) := \Big( \frac{1}{2 \pi} \int_\T c(\vphi, x)^{- \frac13}\,dx \Big)^{- 3}\,,\ \  \beta(\vphi, x) := \partial_x^{- 1} \big( m(\vphi)^{\frac13} c(\vphi, x)^{- \frac13} - 1 \big)\,,
\end{equation}
where $\pa_x^{-1}$ is the Fourier multiplier 
$$
\partial_x^{-1 } e^{\ii j x} := \frac{e^{\ii j x}}{ \ii j}\,,\,\, \forall j \neq 0\,,\qquad \partial_x^{- 1} 1 := 0 \,.
$$

\begin{remark}\label{parities a1 - a6}
Since $c $ is $ \even(\vphi)$-$\even(x)$,  
it follows that $\b = \text{even}(\vphi), \text{odd}(x)$.
As a consequence, $\mB, \mB^{-1} $ are even and reversibility preserving.
Therefore 
$$ a_1 = \text{odd}(\vphi), \odd(x) \, ,  \ a_2 = \odd(\vphi),\even(x) \, , \
 a_3, a_4, a_6  = \even(\vphi),\even(x) ,
 $$ 
and $ a_5= \even(\vphi), \odd(x) $. 
\end{remark}

\noindent
{\bf Step 2}. 
We conjugate $ {\mathcal L}_1 $ in \eqref{Op L1} by the linear map 
$$
{\mathcal Q} := \begin{pmatrix}
1  & 0 \\
0 &  q \end{pmatrix},
\quad 
{\mathcal Q} ^{- 1} = \begin{pmatrix}
1 & 0 \\
0 & q^{- 1}
\end{pmatrix}\,,
$$
where $ q (\vphi, x)$ is a real valued function  close to $ 1 $ to be determined. We compute 
\begin{equation}  \label{mL3}
\begin{aligned}
\mL_2 & := {\mathcal Q} ^{-1} \mL_1 {\mathcal Q}  = \Dom  {\mathbb I}_2 +   \\
& 
\begin{pmatrix} 
 a_{1} \pa_y + a_{2} 
& \!\! \!\! \!\!
-a_3 q |D_y| - a_3 q_y  \mH + \mR_2
\vspace{4pt} \\
- \kappa q^{- 1}a_4  \pa_{yy} - \kappa q^{- 1} a_5 \pa_y + q^{- 1} a_6 \quad
& \!\! \!\! \!\!
 a_{1} \pa_y + q^{- 1}( \Dom q) + q^{- 1} a_1 q_y 
\end{pmatrix}
\end{aligned}
\end{equation}
where, by  Lemma \ref{lem: commutator aH} and \eqref{defR1}, the remainder
\begin{equation}\label{def mR3}
\mR_2 := {\mathcal R}_1 q - a_3 [{\mathcal H}, q] \partial_y - a_3 [\mH, q_y] \in OPS^{-\infty} \, . 
\end{equation}
We choose the function $ q $ so that 
the coefficients of the off diagonal highest order terms satisfy 
\begin{equation}\label{scelta q cambio di variabile}
a_3 q =  q^{- 1} a_4\,,\quad i.e.   \quad q := \sqrt{ a_4 / a_3 } 
\end{equation}
(note that $a_3$, $a_4$ are close to $1$). 
Thus by \eqref{scelta q cambio di variabile}, \eqref{proportional}, \eqref{beta lambda3}, \eqref{def c} we get  
\begin{equation} \label{formula m3}
a_3 q = q^{- 1} a_4 =m_3(\vphi)\,,\quad m_3(\vphi ) := \sqrt{m(\vphi)}  = 
\Big( \frac{1}{2 \pi}  \int_{\T} \sqrt{ 1 + \eta_x^2 } \, dx \Big)^{-3/2} \, , 
\end{equation}
and, by \eqref{mL3}, 
\begin{equation}\label{cal L3}
{\mathcal L}_2 = \Dom  {\mathbb I}_2 +  
\begin{pmatrix}
a_1 \partial_y + a_2 & -m_3(\vphi) |D_y| + a_7 {\mathcal H} + {\mathcal R}_2 \\
m_3(\vphi)(1 - \kappa \partial_{yy}) + a_8 \partial_y + b_9 & a_1 \partial_y + b_{10}
\end{pmatrix} 
\end{equation}
where 
\begin{equation}\label{a7 a8 a9 a10}
\begin{aligned}
& a_7 := - a_3 q_y\,,\quad a_8 := - \kappa q^{- 1} a_5\,, \\ 
&  b_9 := q^{- 1} a_6 - m_3(\vphi)\,,
\quad b_{10} := q^{- 1}\big( \Dom q + a_1 q_y   \big)\,.
\end{aligned}
\end{equation}

\begin{remark}\label{parities q a7 a8 a9 a10}
Since $a_4, a_3 $ is $ \even(\vphi), \even(x) $, the function $q $ is $ \even(\vphi), \even(x) $,
 hence the operator $ {\mathcal Q} $ is even and reversibility preserving. Moreover 
$ a_7, a_8= {\rm even}(\vphi) {\rm odd}(x) $, $ b_9 \in \even(\vphi), \even(x) $, $  b_{10} = {\rm odd}(\vphi) {\rm even}(x) $. 
\end{remark}

\begin{lemma} \label{lemma:BAPQ}
The operators ${\mathcal B}^{\pm 1}$ are ${\mathcal D}^{k_0}$-$(k_0 + 1)$-tame, ${\mathcal Q}^{\pm 1}$ are ${\mathcal D}^{k_0}$-tame with tame constants satisfying 
\begin{equation}  \label{est BAPQ}
{\mathfrak M}_{{\mathcal B}}(s)\,,\, {\mathfrak M}_{ {\mathcal Q}}(s) \leq_S 1 + \| \fracchi_0 \|_{s + \sigma}^{k_0, \gamma} \, , 
\quad \forall s_0 \leq s \leq S \,.
\end{equation}
The operators ${\mathcal B}^{\pm 1} - {\rm Id}$, $({\mathcal B}^{\pm 1} - {\rm Id})^*$ is ${\mathcal D}^{k_0}$-$(k_0 + 2)$-tame and ${\mathcal Q}^{\pm 1} - {\rm Id}$, $({\mathcal Q}^{\pm 1} - {\rm Id})^*$ are ${\mathcal D}^{k_0}$-tame and, for all $ s_0 \leq s \leq S $, 
\begin{equation}\label{est BAPQ (1)}
\begin{aligned}
& {\mathfrak M}_{{\mathcal B}^{\pm 1} - {\rm Id}}(s)\,,\,{\mathfrak M}_{({\mathcal B}^{\pm 1} - {\rm Id})^*}(s)\,,\, {\mathfrak M}_{ {\mathcal Q}^{\pm 1} - {\rm Id}}(s)\,,\,{\mathfrak M}_{ ({\mathcal Q}^{\pm 1} - {\rm Id})^*}(s)  \\
& \leq_S \e (1 + \| \fracchi_0 \|_{s + \sigma}^{k_0, \gamma})\,.
\end{aligned}
\end{equation}
The functions $ m_3 $ satisfies
\begin{equation}\label{stima m3 - 1}
\| m_3 - 1 \|_s^{k_0, \gamma}  \leq_s \e \big(1  + \| \fracchi_0 \|_{s + \s}^{k_0, \gamma} \big)\,, \quad \|\partial_i m_3 [\widehat \imath] \|_{s_1} \leq_{s_1} \e \| \widehat \imath\|_{s_1 + \sigma}
\end{equation}
and the functions $ a_i $ satisfy
\begin{equation}\label{stima m3(vphi)}
\begin{aligned}
& \max\{ \| a_{1} \|_s^{k_0, \gamma},  \| a_{2} \|_s^{k_0, \gamma}, \| a_{7} \|_s^{k_0, \gamma}, \| a_{8} \|_s^{k_0, \gamma},
\| b_9 \|_s^{k_0, \gamma}, \| b_{10} \|_s^{k_0, \gamma}  \}  \\
& \leq_S \e \big(1  + \| \fracchi_0 \|_{s + \s}^{k_0, \gamma} \big).
\end{aligned}
\end{equation}
The remainder ${\mathcal R}_2 $  in \eqref{def mR3} is in $ OPS^{-\infty } $
and, for some $\sigma := \sigma(\tau, \nu) > 0 $, for all $ m \geq 0 $, $ s \geq 0 $, $ \a \in \N $,  
\begin{equation}\label{stima cal R3}
\norma {\mathcal R}_2  \norma_{-m, s, \alpha}^{k_0, \gamma} 
\leq_{m, S, \a} \e (1 + \| \fracchi_0 \|_{s + \sigma + m + \alpha}^{k_0, \gamma})\, .
\end{equation}
Moreover 
\begin{equation}\label{stime derivate i cal B e Q}
\| (\partial_iA [\hat \imath] ) h \|_{s_1} \leq_{S}  \e \| \hat \imath \|_{s_1 + \sigma} \| h \|_{s_1 + \sigma}\,, \quad A \in \{ \mB^{\pm 1},  {\mathcal Q}^{\pm 1}, (\mB^{\pm 1})^*,  ({\mathcal Q}^{\pm 1})^*\}\,,
\end{equation} 
\begin{equation}\label{derivate i coefficienti secondo step}
\begin{aligned}
&   \| \partial_i  a_{1} [\hat \imath] \|_{s_1},  \| \partial_i  a_{2} [\hat \imath] \|_{s_1}, \| \partial_i  a_{7}[\hat \imath]  \|_{s_1}, \| \partial_i  a_{8} [\hat \imath] \|_{s_1},
\| \partial_i  b_9 [\hat \imath] \|_{s_1}, \| \partial_i  b_{10} [\hat \imath] \|_{s_1}    \\
& \leq_{S} \e \| \hat \imath \|_{s_1 + \sigma}
\end{aligned}
\end{equation}
and for all $m \geq 0$, $\alpha \in \N$
\begin{equation}\label{derivate in i cal R2}
\norma \partial_i {\mathcal R}_2 [\hat \imath]  \norma_{-m, s_1, \alpha} \leq_{m, S, \alpha} \e \| \hat \imath \|_{s_1 + \sigma + m + \alpha}\,.
\end{equation}
\end{lemma}
\begin{proof}
The estimates \eqref{est BAPQ}, \eqref{stima m3(vphi)} follows by \eqref{formula m3}, \eqref{a1 a2 a3}, \eqref{a4 a5 a6}, \eqref{a7 a8 a9 a10}, using \eqref{interpolazione C k0} and Lemmata \ref{lemma:remainder mR0}, \ref{Moser norme pesate}, \ref{lemma:utile}, \ref{lemma: action Sobolev}.
The estimate \eqref{stima cal R3} follows by 
Lemmas \ref{lemma:utile}, \ref{lemma cio}, \ref{lem: commutator aH}, \ref{coniugio Hilbert}, 
Proposition \ref{Prop DN}, 
 \eqref{tame Tdelta},  and 
\eqref{interpolazione C k0}. The estimate \eqref{est BAPQ (1)} for ${\mathcal Q} = {\mathcal Q}^*$ follows since the function $q(\vphi, x)$ is close to $1$, and it satisfies 
$
\| q - 1 \|_s^{k_0, \gamma} \leq_s \e (1 + \| \fracchi_0\|_{s + \sigma}^{k_0, \gamma})\,,
$
for some $\sigma := \sigma(k_0, \tau, \nu) > 0$. The estimate for  ${\mathcal B} - {\rm Id}$ follows by 
$$
({\mathcal B} - {\rm Id}) h = \beta\, {\mathcal B}_\tau [h_x]\,, \qquad {\mathcal B}_\tau [h](\vphi, x) : = \int_0^1 h_x(\vphi, x + \tau \beta(\vphi, x))
\,d \tau 
$$
and the estimate for the adjoint $({\mathcal B} - {\rm Id})^*$ follows by the representation 
\begin{equation}\label{cambio di variabile aggiunto}
{\mathcal B}^* h(\vphi, y) = (1 + \tilde \beta(\vphi, y)) h(\vphi, y + \tilde \beta(\vphi, y)) 
\end{equation}
where $y \mapsto y + \tilde \beta(\vphi, y)$ is the inverse diffeomorphism of $x \mapsto x + \beta(\vphi, x)$. The expressions of ${\mathcal B}^{- 1} - {\rm Id}$ and $({\mathcal B}^{- 1})^*$ are similar.

Let us prove the estimate \eqref{stime derivate i cal B e Q} for ${\mathcal B}$ and ${\mathcal B}^{- 1}$. The other estimates follow analogously. By \eqref{beta lambda3} and using the estimates \eqref{stima V B a c}, \eqref{stima derivate i primo step} on $c$ we get 
\begin{equation}\label{stima derivate i beta}
\| \partial_i \beta [\hat \imath] \|_{s_1} \leq_{s_1} \e \| \hat \imath\|_{s_1 + \sigma}
\end{equation}
then the estimate \eqref{stime derivate i cal B e Q} for ${\mathcal B}$ follows since 
$ (\partial_i {\mathcal B} [\hat \imath]) h = \partial_i \beta[\hat \imath] {\mathcal B}[h_x] $. 
Since $y = x + \beta(x)$ if and only if $x = y + \tilde \beta(y)$, differentiating with respect to $i$ we get 
$\partial_i \tilde \beta[\hat \imath] = (1 + \beta_x)^{- 1} {\mathcal B}^{- 1} [\partial_i \beta[\hat \imath]]$, hence $\partial_i \tilde \beta$ satisfies \eqref{stima derivate i beta} (for a possibly larger $\sigma := \sigma(\tau, \nu) > 0$), and hence ${\mathcal B}^{- 1}$ satisfies \eqref{stime derivate i cal B e Q}. The estimates \eqref{derivate i coefficienti secondo step} follows by differentiating the explicit expressions of the coefficients and applying \eqref{interpolazione C k0}, the estimates of Lemma \ref{lemma:remainder mR0}, 
\eqref{stime derivate i cal B e Q} for ${\mathcal B}^{\pm 1}$ and Lemma \ref{Moser norme pesate}. By \eqref{scelta q cambio di variabile}, $\partial_i q$ satisfies \eqref{derivate i coefficienti secondo step}, therefore $ {\mathcal Q}  $ and $ {\mathcal Q}^{- 1}$ satisfy \eqref{stime derivate i cal B e Q}. For proving  
\eqref{derivate in i cal R2} for $\partial_i {\mathcal R}_2 [\hat \imath] $ we show that the derivative $ \pa_i $ 
of each term  in \eqref{def mR3} 
 satisfies the estimate \eqref{derivate in i cal R2}. 
For instance the term $\partial_i [{\mathcal H}, q][\widehat \imath] = [{\mathcal H}, \partial_i q[\widehat \imath]]$ can be estimated by applying Lemma \ref{lem: commutator aH} and using that 
$\partial_i q [\widehat \imath]$ (the function $ q $ is defined in  \eqref{scelta q cambio di variabile}) 
satisfies the same bound  \eqref{derivate i coefficienti secondo step}. 
For estimating $\partial_i{\mathcal R}_1 [\widehat \imath]$ we estimate separately the derivatives of the two terms ${\mathcal B}^{- 1}{\mathcal R}_G {\mathcal B}$ and ${\mathcal R}_{\mathcal B}$ in \eqref{defR1}.  
The operator $\partial_i ({\mathcal B}^{- 1}{\mathcal R}_G {\mathcal B})[\widehat \imath]$ satisfies the estimate \eqref{derivate in i cal R2} by \eqref{Geta intermedia}-\eqref{palla di lardo 0} 
by Lemmata  \ref{lem:Int}, \ref{lemma cio}, \ref{coniugio Hilbert}, Proposition \ref{Prop DN} and \eqref{est BAPQ}, \eqref{est BAPQ (1)}, \eqref{stime derivate i cal B e Q}, \eqref{derivata i T delta}. 
The estimate of the operator $\partial_i {\mathcal R}_{\mathcal B}[\widehat \imath]$ in \eqref{Resto:RB}, 
follows similarly.
\end{proof}

\noindent
{\bf Step 3}.  We ``symmetrize" the order of derivatives in the off-diagonal
terms of the operator $ \mL_2 $  in \eqref{cal L3}. 
We  conjugate $  \mL_2 $ by the vector valued Fourier multiplier
\be\label{definition Lambda} 
\mS = \begin{pmatrix} 1 & 0 \\ 
0 & G \end{pmatrix}, 
\quad
\mS^{-1} = \begin{pmatrix} 1 & 0 \\ 
0 & G^{-1} \end{pmatrix}, \quad G := {\rm Op}(g(\xi)) \in OPS^{1/2}
\ee
where $ g $ is a $ {\mathcal C}^\infty $ even function satisfying 
\begin{equation}\label{proprieta g(xi) simmetrizzazione}
g(0) = 1\,, \quad g > 0\,, \quad g(\xi) = |\xi|^{- \frac12} (1 + \kappa \xi^2)^{\frac12}\,, \quad \forall |\xi | \geq 1/ 3\,.
\end{equation}
Note that $ {\mathcal S} $ is a real and even operator, see Lemma \ref{even:pseudo}. 
Recalling the definition of the cut off function $\chi$ in \eqref{cut off simboli 1}, the symbols $g \in S^{\frac12}$ and $1/g \in S^{- \frac12}$ admit the expansions 
\be\label{espansione simbolo g(xi) 1}
\begin{aligned}
g(\xi) & = \chi(\xi) g(\xi) + (1 - \chi(\xi)) g(\xi)  \\  
& =  \chi(\xi) \frac{(1 + \kappa \xi^2)^{\frac12}}{ |\xi|^{ 1/2}} + (1 - \chi(\xi)) g(\xi)  = \sqrt{\kappa} \chi(\xi) |\xi|^{\frac12} + g_{- \frac32}(\xi)
\end{aligned}
\ee
where $ g_{- \frac32} \in S^{- \frac32}  $  and 
\be
\begin{aligned}
\frac{1}{g(\xi)} & = 
\frac{\chi(\xi)}{g(\xi)} + \frac{1 - \chi(\xi)}{g(\xi)} \\ 
& = 
\chi(\xi ) \frac{|\xi|^{\frac12}}{ (1 + \kappa \xi^2)^{ \frac12}} +  \frac{1 - \chi(\xi)}{g(\xi)} 
=  \frac{\chi(\xi)}{\sqrt{\kappa} |\xi|^{\frac12}} + g_{- \frac52}(\xi)\,, \  \ g_{- \frac52} \in S^{- \frac52}\,.\label{espansione simbolo g(xi) inverso 2}
\end{aligned}
\ee
Since  $\frac{1 - \chi(\xi)}{g(\xi)} = 0$, for $|\xi| \geq 1$, and $\frac{1 - \chi(0)}{g(0)} = 1 $, 
the operator   $ {\rm Op}\Big( \frac{1 - \chi(\xi)}{g(\xi)} \Big) = \pi_0  $ on the periodic functions,  where  $ \pi_0 $ is the projector  
\be\label{def:pi0}
\pi_0 (f) := \frac{1}{2 \pi} \int_{\T} f(x) \, dx \, .
\ee
By \eqref{espansione simbolo g(xi) 1}-\eqref{espansione simbolo g(xi) inverso 2} we get 
the expansions
\be\label{basic decomp}
 G = \sqrt{\kappa} |D|^{\frac12}  + G_{-3/2}   \, , 
\ \   G^{-1} =  |D|^{\frac12} ( 1 - \kappa \pa_{xx})^{- \frac12} + \pi_0 = 
\frac{1}{\sqrt{\kappa}} |D|^{- \frac12} +  G_{-5/2} \, , 
\ee
where $ G_{-3/2} = {\rm Op}(g_{- \frac32}) \in OPS^{-3/2} $ and $  G_{-5/2} = {\rm Op}(g_{- \frac52}) \in OPS^{-5/2} $. 
Using \eqref{definition Lambda}, \eqref{proprieta g(xi) simmetrizzazione}, \eqref{definizione |D| m}, \eqref{basic decomp} 
we get 
\begin{equation}\label{parte principale 2 simmetrizzazione}
|D| G = {\rm Op}(\chi(\xi) |\xi| g(\xi)) = T(D) \, , \quad  
G^{- 1} (1 - \kappa \partial_{xx}) =  T(D) + \pi_0
\end{equation}
where $  T(D)$ is the Fourier multiplier
\begin{equation}\label{def T}
T  := T(D) := |D|^{1/2} (1 - \kappa \pa_{xx})^{1/2}  = {\rm Op} \big(  \chi(\xi) |\xi |^{\frac12} (1+ \kappa \xi^2)^{\frac12} \big)
\in OPS^{3/2} \,.
\end{equation}
Hence using \eqref{basic decomp}-\eqref{parte principale 2 simmetrizzazione} (and renaming $ \pa_y $ as $ \pa_x $)
we get 
\begin{align} \label{def mL4}
& \mL_3  \stackrel{\eqref{cal L3}} {:=} \mS^{-1} \mL_2 \mS \stackrel{ \eqref{mL3}, \eqref{Op L1}} = 
\mS^{-1} {\mathcal Q}^{-1} \mB^{-1} \mL_0  \mB {\mathcal Q} \mS  
=  \Dom  {\mathbb I}_2 + \\ 
& + \begin{pmatrix} 
 a_{1} \pa_x + a_{2}  &  \!\! \!\!
 -  m_3 T(D) + \sqrt{\kappa} \, a_{7} \mH |D|^{\frac12}  + {\mathcal R}_{3,B} \\
m_3   T(D) 
	- \frac{a_{8}}{\sqrt{\kappa}} \,  |D|^{\frac12} \mH + m_3 \pi_0
	+ {\mathcal R}_{3,C}   &  \!\! \!\! a_{1} \pa_x + {\mathcal R}_{3,D} 
\end{pmatrix} \nonumber
\end{align} 
where the remainders are the pseudo-differential operators in $ OPS^{0} $
\begin{align}\label{defR3CR3D}
{\mathcal R}_{3,B} & :=  a_7 \mH G_{-3/2} + {\mathcal R}_2 \Lambda \, , 
\quad {\mathcal R}_{3,D}   := [G^{-1}, a_1] \pa_x G + G^{-1} b_{10} G \, ,  
\\
{\mathcal R}_{3,C}  & := a_8 G_{-5/2} \pa_x +  [G^{-1}, a_8] \pa_x +  G^{-1} b_9 \, . \label{defC3} 
\end{align}

\begin{lemma} \label{lemma:S}
Each  $ {\mathcal R} = {\mathcal R}_{3,B}, {\mathcal R}_{3,C}, {\mathcal R}_{3,D}  $
is in $  OPS^0 $ and satisfy, for all $s_0 \leq s \leq S $, 
\begin{equation}\label{stima cal R 3 B C D}
\norma {\mathcal R} \norma_{0, s, \alpha}^{k_0, \gamma} \leq_{S, \alpha} 
\e (1 + \| \fracchi_0 \|_{s + \sigma + \alpha}^{k_0, \gamma})\,, \quad \norma \partial_i {\mathcal R}[\hat \imath] \norma_{0, s_1, \alpha} \leq_{S, \alpha} \e \| \hat \imath \|_{s_1 + \sigma + \alpha} 
\end{equation}
 for some $\sigma := \sigma(\tau, \nu) > 0$.  
The real operator $ {\mathcal L}_3 $ is even and reversible. 
\end{lemma}

\begin{proof}
Use Lemma \ref{lemma tame norma commutatore} to estimate the commutators in \eqref{defR3CR3D}-\eqref{defC3}.
\end{proof}

\section{Complex variables}\label{complex-coordinates}

We now write the real operator ${\mathcal L}_3 $ in \eqref{def mL4}, which acts on the real variables
$ (\eta, \psi) \in \R^2 $, as an operator acting on the complex variables (see \eqref{complex-coor})
$$
h := \eta + \ii \psi \, , \quad   \bar h := \eta - \ii \psi \, , \qquad i.e. \  	\
\eta = (h + \bar h)/ 2 \, , \quad \psi = (h - \bar h) \slash (2 \ii ) \, .
$$
By  \eqref{operatori in coordinate complesse} 
we get the real, even and reversible operator (for simplicity of notation we still denote it by $ {\mathcal L}_3 $)
\begin{equation}\label{cal L3 coordinate complesse}
\begin{aligned}
{\mathcal L}_3 & = \Dom {\mathbb I}_2 + \ii m_3(\vphi) {\bf T}(D) + 
 {\bf A}_1(\vphi, x) \partial_x + \ii ({\bf A}_0^{(I)}(\vphi, x)  \\
 & \quad + {\bf A}_0^{(II)}(\vphi, x)) \mH |D|^{\frac12} + \ii m_3 (\vphi) \Pi_0 + {\bf R}_{3}^{(I)} + {\bf R}_3^{(II)} 
 \end{aligned}
\end{equation}
where 
\begin{equation}\label{definizione T3 A1}
{\bf T} := {\bf T}(D) := \begin{pmatrix}
T(D) & 0 \\
0 & - T(D)
\end{pmatrix}\,, \quad {\bf A}_1(\vphi, x) := \begin{pmatrix}
a_1(\vphi, x) & 0 \\
0 & a_1(\vphi, x)
\end{pmatrix}\,,
\end{equation}
\begin{equation}\label{defA0}
{\bf A}_0^{(I)} (\vphi, x) := \begin{pmatrix}
a_9 & 0 \\
0 & - a_9
\end{pmatrix}\,, 
\quad  a_9 := - \frac12 \Big( \sqrt{\kappa} a_7 + \frac{a_8}{\sqrt{\kappa}} \Big)\,,
\end{equation}
\be\label{defB0}
{\bf A}_0^{(II)}(\vphi, x) := \begin{pmatrix}
0 & a_{10} \\
- a_{10} & 0
\end{pmatrix} \, , 
 \quad a_{10} := \frac{1}{2} \Big( \sqrt{\kappa} a_7 - \frac{a_8}{ \sqrt{\kappa}} \Big)\,,  
\ee
\begin{equation}\label{definizione Pi 0}
 \Pi_0 := \frac12  \begin{pmatrix}
1 & 1 \\
- 1 & - 1
\end{pmatrix}
\pi_0\,, \quad 
\end{equation}
$$
\begin{aligned}
& {\bf R}_{3}^{(I)} := \begin{pmatrix}
r_3^{(I)}(x, D) & 0 \\
0 & \overline{r_3^{(I)} (x,D)}
\end{pmatrix} \in OPS^0 \,, \\
& r_3^{(I)}(x, D) := \frac12 \big( a_2 + {\mathcal R}_{3, D} - \ii {\mathcal R}_{3, B} + \ii {\mathcal R}_{3, C}  \big)\,, 
\end{aligned}
$$
$$
\begin{aligned}
& {\bf R}_3^{(II)} := \begin{pmatrix}
0 & r_3^{(II)}(x, D)  \\
\overline {r_3^{(II)} (x, D)} & 0
\end{pmatrix} \in OPS^0 \,,  \\
&  r_3^{(II)}(x, D) := \frac12 \big( a_2 - {\mathcal R}_{3, D} + \ii {\mathcal R}_{3, B} + \ii {\mathcal R}_{3, C}  \big)\,.
\end{aligned}
$$
 Lemma \ref{lemma:BAPQ} and  \eqref{stima cal R 3 B C D} imply for all $s_0 \leq s \leq S$, the estimates  
\begin{align}\label{stima r 3 (I) (II)}
& \norma r_3^{(I)}(x, D) \norma_{0, s, \alpha}^{k_0, \gamma}, \
\norma r_3^{(II)}(x, D) \norma_{0, s, \alpha}^{k_0, \gamma} 
\leq_{S, \alpha} \e \big( 1 + \| \fracchi_0 \|_{s + \alpha + \sigma }^{k_0, \gamma} \big) \, , \\
& \label{stima derivate i r 3 (I) (II)}
\norma \partial_i r_3^{(I)}(x, D) [\hat \imath] \norma_{0, s_1, \alpha}, \
\norma \partial_i r_3^{(II)}(x, D)[\hat \imath] \norma_{0, s_1, \alpha} \leq_{S, \alpha} \e \| \hat \imath \|_{s_1 + \alpha + \sigma}\,.
\end{align}
Note that $ {\mathcal L}_3 $ in \eqref{cal L3 coordinate complesse} is block-diagonal (in ($u, \bar u$)) up to order $ |D|^{1/2} $. 
The introduction of the complex formulation is convenient in section \ref{sec:decoupling} where  
we eliminate iteratively   the  off-diagonal terms of $ {\mathcal L}_3 $ up to very smoothing remainders,
see Proposition \ref{Lemma finale decoupling}.

In the next sections we reduce the real, even and reversible operator $ {\mathcal L}_3 $ neglecting   the term 
$ \ii m_3 (\vphi) \Pi_0$ in \eqref{definizione Pi 0}.
For simplicity of notation we denote it as  $ {\mathcal L}_3 $ as well. The projector $ m_3 (\vphi) \ii \Pi_0 $
transforms under conjugation  into a finite dimensional operator 
and we will conjugate it only once in section \ref{coniugio cal L omega}.

\section{Time-reduction of the highest order}\label{sec: time-reduction highest order}

The purpose of this section is to remove the dependence on $ \vphi $ from the highest order term 
$ \ii m_3(\vphi){\bf T}(D)$ in the operator $ \mL_3 $ defined in \eqref{cal L3 coordinate complesse} (without $\Pi_0$). 
Actually, since 
we only assume 
that 
the frequency $  \om $ belongs to   ${\mathtt D \mathtt C}_{K_n}^{\g}$ defined in \eqref{omega diofanteo troncato},  
we shall only transform $ \ii \Pi_{K_n} m_3(\vphi){\bf T}(D)$ (where $K_n $ is defined in \eqref{definizione Kn})
into a constant coefficient operator, and we keep  the term \eqref{cal RN bot (3)} which is 
 Fourier supported on the high harmonics, 
and thus contributes to \eqref{stima R omega bot corsivo bassa}-\eqref{stima R omega bot corsivo alta}. 

To this aim we perform a quasi periodic reparametrization of time
\be\label{QP-repa}
\vartheta := \vphi + \omega p(\vphi) \quad  \Leftrightarrow \quad 
\vphi = \vartheta + \omega \tilde p(\vartheta)
\ee
where $ p(\vphi) $ is a small periodic function to be determined. 
We conjugate  $\mL_3 $ by  the real operator
$$
\begin{aligned}
& {\mathcal P} \, \mathbb I_2 = \begin{pmatrix}
{\mathcal P} & 0 \\
0 & {\mathcal P}
\end{pmatrix} 
\quad {\rm where} \\
& ({\mathcal P}h)(\vphi, x) := h(\vphi + \omega p(\vphi), x) \, , \quad  
({\mathcal P}^{- 1} h)(\vartheta, x) := h(\vartheta + \omega \tilde p(\vartheta), x ) \, . 
\end{aligned}
$$
The differential operator $ \Dom $ and  
the multiplication operator by $ a $ transform into
\be\label{transformed Pcal}
\begin{aligned}
&  {\mathcal P}^{- 1} \Dom {\mathcal P} = \rho(\vartheta)  \Dom\,,  \quad  
 \rho (\vartheta) := \big({\mathcal P}^{- 1}[1 + \omega \cdot \partial_\vphi p] \big) \, , \\
& \qquad  \qquad \qquad   {\mathcal P}^{- 1} a {\mathcal P} = ({\mathcal P}^{- 1} a) \, ,
\end{aligned}
\ee
while a space Fourier multiplier $\phi(D)$ remains clearly unchanged $ {\mathcal P}^{- 1} \phi(D) {\mathcal P} = \phi(D) $. 
Thus 
\begin{align*}
{({\mathcal P}^{- 1} \, \mathbb I_2}) \mL_3 \, {({\mathcal P} \, \mathbb I_2)} & = ({\mathcal P}^{- 1}[1 + \omega \cdot \partial_\vphi p ])
\Dom {\mathbb I}_2 +  
( {\mathcal P}^{- 1} m_3) \ii  {\bf T}(D) + ({\mathcal P}^{- 1}{\mathbb I}_2 {\bf A}_1 ) \partial_x \\
&  + \ii  ({\mathcal P}^{- 1} {\mathbb I}_2)({\bf A}_0^{(I)} +{\bf A}_0^{(II)}) \mH |D|^{\frac12}  
+   ({\mathcal P}^{- 1} {\mathbb I}_2) \big(  {\bf R}_3^{(I)} + {\bf R}_3^{(II)}   \big)({\mathcal P} {\mathbb I}_2) \, .  \nonumber 
\end{align*}
Splitting  
$ m_3(\vphi) = \Pi_{K_n}m_3(\vphi) + \Pi_{K_n}^\bot m_3(\vphi) $ 
we  solve, for all $ \om \in {\mathtt D \mathtt C}_{K_n}^{\g} $ (see \eqref{omega diofanteo troncato}), the equation 
\be\label{solution p m3}
1 + \omega \cdot \partial_\vphi p = \mathtt m_3^{-1} \Pi_{K_n} m_3(\vphi)\,,
\ee
by defining (the function $ m_3 (\vphi) $ is even)
\be\label{lambda3 formula} 
\begin{aligned}
 \mathtt m_3 & := (2 \pi)^{-\nu} \int_{\T^\nu} \Pi_{K_n} m_3(\vphi)\,d \vphi \\
 & \stackrel{\eqref{formula m3}} = 
 (2\p)^{-\nu} \int_{\T^\nu} \Big( \frac{1}{2\p} \int_\T \sqrt{1 + \eta_x^2}\, dx \Big)^{-3/2} \, d\ph ,  
\end{aligned}
\ee
and
\be
 p := (\omega \cdot \partial_\vphi)^{- 1} \big( \mathtt m_3^{-1}  \Pi_{K_n} m_3(\vphi) - 1 \big) 
\text{\ which is odd in } \vphi  \,. \label{definizione p riparametrizzazione}
\ee
Dividing $ {({\mathcal P}^{- 1} \, \mathbb I_2}) \mL_3 \, {({\mathcal P} \, \mathbb I_2)}  $ 
by the even function $\rho := {\mathcal P}^{- 1}[1 + \omega \cdot \partial_\vphi p]$ we get the real, even and reversible operator
\be
\begin{aligned}
\mL_4  & := \rho^{-1} {({\mathcal P}^{- 1} \, \mathbb I_2}) \mL_3 \, {({\mathcal P} \, \mathbb I_2)}  \\
& =  \Dom {\mathbb I}_2 +\ii \mathtt m_3 {\bf T}(D) + {\bf B}_1(\vphi, x) \partial_x +  \ii 
\big({\bf B}_0^{(I)}(\vphi, x) + {\bf B}_0^{(II)}(\vphi, x) \big) \mH |D|^{\frac12} \\
& \quad +  {\bf R}_4^{(I)} + {\bf R}_4^{(II)} + {\bf R}^\bot_{4}  \label{cal L 4}
\end{aligned}
\ee
where 
\begin{equation} \label{bf B1}
 {\bf B}_1 := \rho^{- 1}{\mathcal P}^{- 1} {\mathbb I}_2 {\bf A}_1 = \begin{pmatrix}
 a_{11} & 0 \\
0 &  a_{11} 
\end{pmatrix}\,, \quad a_{11} := \rho^{- 1} {\mathcal P}^{- 1}(a_1) 
\end{equation}
\begin{equation}\label{bf B0 (I)}
{\bf B}_0^{(I)} := \rho^{- 1} {\mathcal P}^{- 1} {\mathbb I}_2 {\bf A}_0^{(I)} = \begin{pmatrix}
a_{12}  & 0 \\
0 & - a_{12} 
\end{pmatrix}\,, \quad a_{12} := \rho^{- 1} {\mathcal P}^{- 1}(a_9)
\end{equation} 
\begin{equation}\label{bf B0 (II)}
{\bf B}_0^{(II)} := \rho^{- 1} {\mathcal P}^{- 1} {\mathbb I}_2 {\bf A}_0^{(II)} = \begin{pmatrix}
0 & \rho^{- 1} {\mathcal P}^{- 1}(a_{10})  \\
 - \rho^{- 1} {\mathcal P}^{- 1}(a_{10}) & 0
\end{pmatrix} 
\end{equation}
  \begin{equation}\label{cal RN (3)0}
{\bf R}_4^{(I)}  := \begin{pmatrix}
r_4^{(I)} (x, D) & 0 \\
0 & \overline {r_4^{(I)} (x, D)}
\end{pmatrix} \, , \quad 
r_4^{(I)}(x, D)  := 
\rho^{- 1}  {\mathcal P}^{- 1} 
r_3^{(I)}(x, D) {\mathcal P}  \,,
\end{equation}
\begin{equation}\label{cal QN (3)}
\begin{aligned}
& {\bf  R}_4^{(II)}  := \begin{pmatrix}
0 & r_4^{(II)}(x, D) \\
\overline {r_4^{(II)}  (x, D)} & 0
\end{pmatrix}\,,  \\
& r_4^{(II)} (x, D)  := \rho^{- 1} {\mathcal P}^{- 1} 
r_3^{(II)}(x, D) {\mathcal P}  
\end{aligned}
\end{equation}
and 
\begin{equation}\label{cal RN bot (3)}
{\bf R}^\bot_4 := \ii \rho^{- 1} \Pi_{K_n}^\bot m_3(\vphi) {\bf T}(D) \, . 
\end{equation}

\begin{lemma}
 The maps ${\mathcal P}$, ${\mathcal P}^{- 1}$ are 
 ${\mathcal D}^{k_0}$-$(k_0 + 1)$-tame with tame constants satisfying  the estimates
\begin{equation}\label{stima cal A}
{\mathfrak M}_{{\mathcal P}^{\pm 1}} (s) \leq_S (1  + \| \fracchi_0 \|_{s + \sigma}^{k_0, \gamma}) \, , \quad 
\forall s_0 \leq s \leq S  \, .
\end{equation}
The maps ${\mathcal P} - {\rm Id}$, ${\mathcal P}^{- 1} - {\rm Id}$ are 
 ${\mathcal D}^{k_0}$-$(k_0 + 2)$-tame and  
 \begin{equation}\label{stima cal P - identita}
 {\mathfrak M}_{{\mathcal P}^{\pm 1} - {\rm Id}} (s) \leq_S \e \gamma^{- 1}(1  + \| \fracchi_0 \|_{s + \sigma}^{k_0, \gamma})\, , 
 \quad \forall s_0 \leq s \leq S \,  .
 \end{equation}
The coefficient $\mathtt m_3 $ defined in \eqref{lambda3 formula} and the
functions $ a_{11}  $, $ a_{12} $, $ \rho^{-1} {\mathcal P}^{-1} (a_{10} ) $  in \eqref{bf B1}-\eqref{bf B0 (II)} satisfy 
\begin{align}
& |\mathtt m_3 - 1 |^{k_0, \gamma} \leq C \e\,, \quad |\partial_i \mathtt m_3 [\hat \imath]  | \leq C \e \| \hat \imath\|_{ \sigma}\,, \label{stima lambda 3 - 1 nuova} \\
&\| a_{11} \|_s^{k_0, \gamma}, \| a_{12} \|_s^{k_0, \gamma}, 
\| \rho^{-1} {\mathcal P}^{-1} (a_{10} ) \|_s^{k_0, \gamma} \leq_S \e \big(1 + \| \fracchi_0 \|_{s + \sigma}^{k_0, \gamma} \big)\,,
\ \forall s_0 \leq s \leq S \, ,  \label{stima a 15}  
\end{align}
 and 
\begin{align}\label{r 4 (I) (II)}
& \norma r_4^{(I)}(x, D) \norma_{0, s, \a}^{k_0, \gamma}, 
\norma r_4^{(II)}(x,D) \norma_{0, s, \a}^{k_0, \gamma} \leq_{S, \a} \e (1 +  \| \fracchi_0 \|_{s + \alpha + \sigma}^{k_0, \gamma} ) \\
& \label{stima derivata i cal A}
\| (\partial_i {\mathcal P}^{\pm 1}[\hat \imath]) h \|_{s_1} \leq_{S} \e \gamma^{- 1}\| \hat \imath \|_{s_1 + \sigma} \| h \|_{s_1 + \sigma} \\
& \label{stima derivata i a 15}  
\| \partial_i a_{11}[\hat \imath] \|_{s_1}, \| \partial_i a_{12} [\hat \imath] \|_{s_1}
\| \partial_i \big\{\rho^{-1} {\mathcal P}^{-1} (a_{10} ) \big\} [\hat \imath] \|_{s_1} \leq_{S}  \e \| \hat \imath \|_{s_1 + \sigma} \\ 
& \label{stima derivata i r 4 (I) (II)}
\norma \partial_i r_4^{(I)}(x, D) [\hat \imath]\norma_{0, s_1, \a}, 
\norma \partial_i r_4^{(II)}(x,D)[\hat \imath] \norma_{0, s_1, \a} \leq_{S, \alpha}  \e \| \hat \imath\|_{s_1 + \alpha + \sigma}\,.
\end{align}
\end{lemma}

\begin{proof}
The estimates \eqref{stima cal A}, \eqref{stima a 15} follow
by Lemmata \ref{lemma:utile}, \ref{lemma operatore e funzioni dipendenti da parametro} 
and  \ref{lemma:BAPQ}. The bound \eqref{stima cal P - identita} follows since 
$$
({\mathcal P} - {\rm Id}) h = p\, \int_0^1 {\mathcal P}_\tau [\omega \cdot \partial_\vphi h]\, d \tau \,, 
\quad {\mathcal P}_\tau [h] (\vphi, x) := h(\vphi + \tau \omega p(\vphi), x)\,,
$$
and since by Lemma \ref{lemma:BAPQ}, using  \eqref{definizione p riparametrizzazione} and 
\eqref{stima m3(vphi)}, \eqref{stima a 15}, we have 
\begin{equation}\label{stima p nel lemma riparametrizzazione}
 \| p \|_s^{k_0, \gamma} \leq_s \e \gamma^{- 1} (1 + \| \fracchi_0 \|_{s + \sigma}^{k_0, \gamma})\,.
 \end{equation}
  The estimate for ${\mathcal P}^{- 1} - {\rm Id}$ follows similarly. 
Let us  prove \eqref{r 4 (I) (II)}. The conjugated operator 
\begin{equation}\label{simbolo r3 tilde}
{\mathcal P}^{- 1} r_3^{(I)}(x, D) {\mathcal P} = {\rm Op}(\tilde r_3)
\quad {\rm where} \quad  \tilde r_3(\vartheta, x, \xi) :=  r_3^{(I)}(\vartheta + \omega \tilde p(\vartheta), x, \xi) \, .
\end{equation}
Hence for all $\alpha \geq 0$, for all $|k| \leq k_0$, for all $\xi \in \R $ and for all $\omega$ we have by Lemma \ref{lemma:utile} 
\begin{align*}
\|\partial_\xi^\alpha \tilde r(\omega, \cdot, \xi)  \|_s^{k_0, \gamma} \leq_S \| \partial_\xi^\alpha r_3^{(I)}(\cdot, \xi) \|_{s + k_0}^{k_0, \gamma} + \| p \|_{s + \sigma}^{k_0, \gamma} \|  \partial_\xi^\alpha r_3^{(I)}(\cdot, \xi) \|_{s_0 + k_0}^{k_0, \gamma}\,,
\end{align*}
 thus using the estimate \eqref{stima p nel lemma riparametrizzazione} we get 
\begin{align*}
\norma {\mathcal P}^{- 1} r_3^{(I)}(x, D) {\mathcal P}  \norma_{0, s, \alpha}^{k_0, \gamma} 
& 
\leq_S \norma r_3^{(I)}(x, D) \norma_{0, s, \alpha}^{k_0, \gamma} + \| \fracchi \|_{s + \sigma}^{k_0, \gamma} 
 \norma r_3^{(I)}(x, D) \norma_{0, s_0, \alpha}^{k_0, \gamma} \\
& \stackrel{\eqref{stima r 3 (I) (II)}}{\leq_{S, \alpha}} \e 
 \big( 1 + \| \fracchi_0  \|_{s + \alpha + \sigma}^{k_0, \gamma}\big)\,,
\end{align*}
and the estimate \eqref{r 4 (I) (II)} for $r_4^{(I)}$ follows. The estimate for $r_4^{(II)}$ is analogous. 
The proof of  \eqref{stima derivata i cal A} is similar to the proof of the estimate for $\partial_i {\mathcal B}^{\pm 1}$ in Lemma \ref{lemma:BAPQ}. The estimate \eqref{stima derivata i a 15} follows 
by differentiating the explicit expressions in \eqref{lambda3 formula}, \eqref{bf B1}-\eqref{bf B0 (II)}, 
using \eqref{stima cal A}, \eqref{stima derivata i cal A}, the estimates of Lemma \ref{lemma:BAPQ} and
\eqref{interpolazione C k0}. The estimate \eqref{stima derivata i r 4 (I) (II)} follows since by \eqref{simbolo r3 tilde} 
$ \partial_i {\rm Op}(\tilde r_3) [\hat \imath] =  \partial_i \tilde p [\hat \imath] {\rm Op}\big( \partial_\vphi r_3^{(I)}(\vartheta + \omega \tilde p(\vartheta), x, \xi) \big) $. 
\end{proof}

In the next sections we reduce the real, even and reversible operator $ {\mathcal L}_4 $ neglecting   the term $ {\bf R}^\bot_4 $
(for simplicity of notation we denote it in the same way). 
Note that the term $ {\bf R}^\bot_4  $  is  in $ OPS^{3/2} $. However 
it is supported on the high Fourier frequencies and it will contribute to remainders in 
\eqref{stima R omega bot corsivo bassa}-\eqref{stima R omega bot corsivo alta}. 
In other words, these terms do not need to be treated in the KAM reducibility scheme of section \ref{sec: reducibility}
and the estimates \eqref{stima R omega bot corsivo bassa}-\eqref{stima R omega bot corsivo alta}
are yet sufficient for the convergence of Nash-Moser scheme of section \ref{sec:NM}. 

\section{Block-decoupling up to smoothing remainders}\label{sec:decoupling}

The goal of this section is to conjugate the operator $ {\mathcal L}_4 $ in \eqref{cal L 4} (without  $ {\bf R}^\bot_4 $) to the 
operator $ {\mathcal L}_M $ in \eqref{cal LN decoupling} which  is block-diagonal  up to the smoothing remainder
  $ {\bf R}_M^{(II)}  \in OPS^{\frac12 - M} $. This is achieved by applying iteratively $ M $-times  
a conjugation map which transforms the off-diagonal block operators into $ 1 $-smoother ones.  

We describe the generic inductive step. We have a real, even and reversible operator 
\begin{equation}\label{cal Ln disaccoppiamento}
{\mathcal L}_n := \Dom {\mathbb I}_2 + \ii \mathtt m_3 {\bf T}(D) + 
{\bf B}_1 \partial_x + \ii {\bf B}_0^{(I)} \mH |D|^{\frac12}  + {\bf R}_n^{(I)} + {\bf R}_n^{(II)} 
\end{equation}
with block-diagonal terms 
\be\label{cal Rn I}
{\bf R}_n^{(I)} := \begin{pmatrix}
r_n^{(I)}(x, D) & 0 \\
0 & \overline {r_n^{(I)}(x, D)}
\end{pmatrix}\,, \quad r_n^{(I)}(x, D) \in OPS^0 \, , 
\ee
and smoothing off-diagonal remainders 
\begin{equation}\label{cal Rn II}
{\bf R}_n^{(II)} := \begin{pmatrix}
0 & r_n^{(II)}(x, D) \\
 \overline {r_n^{(II)}(x, D)} & 0
\end{pmatrix}\,, \quad r_n^{(II)}(x, D) \in OPS^{\frac12 - n}\,,
\end{equation}
which satisfy 
\begin{align}\label{stima induttiva disaccoppiamento}
& \norma {\bf R}_n^{(I)} \norma_{0, s, \alpha}^{k_0, \gamma} + \norma {\bf R}_n^{(II)} 
\norma_{-n + \frac12, s, \alpha}^{k_0, \gamma} \leq_{n, S, \a} \e (1 + \| \fracchi_0 \|_{s + \sigma + 
\aleph_n(\alpha)  }^{k_0, \gamma})\,, \quad \forall s_0 \leq s \leq S\,,\\
& \label{stima induttiva disaccoppiamento derivate in i}
\norma \partial_i {\bf R}_n^{(I)} [\hat \imath] \norma_{0, s_1, \alpha} + \norma \partial_i {\bf R}_n^{(II)}  [\hat \imath]
\norma_{-n + \frac12, s_1 , \alpha} \leq_{n, S, \alpha}  \e \| \hat \imath \|_{s_1 + \sigma + \aleph_n(\alpha)}\,,
\end{align}
where the increasing constants $ \aleph_n(\alpha)$  are defined inductively by 
\begin{equation}\label{definizione kn(alpha)}
\aleph_0(\alpha) := \alpha\,, \qquad  \aleph_{n + 1}(\alpha) :=  \aleph_n(\alpha + 1) +  n + 2 \alpha + 4\,.
\end{equation}
{\bf Initialization.} The real, even and reversible operator $ {\mathcal L}_4 $ in \eqref{cal L 4} 
satisfies the assumptions
\eqref{cal Ln disaccoppiamento}-\eqref{stima induttiva disaccoppiamento derivate in i}
where the off diagonal remainder is $  \ii  {\bf B}_0^{(II)}(\vphi, x) \mH |D|^{\frac12} + {\bf R}_{4}^{(II)} \in OPS^{1/2 } $
(recall that we have neglected $ {\bf R}_4^\bot $). 
\\[1mm]
{\bf Inductive step.}  We conjugate $ {\mathcal L}_n $ in \eqref{cal Ln disaccoppiamento} by  a
real  operator of the form 
\be\label{defPhi-n}
\begin{aligned}
& \Phi_n := {\mathbb I}_2 + \Psi_n\,, \quad \Psi_n := \begin{pmatrix}
0 & \psi_n(x, D) \\
\overline {\psi_n(x, D)} & 0
\end{pmatrix}\,, \\
&   \qquad \qquad  \psi_n(x, D) \in OPS^{- n - 1}\,.
\end{aligned}
\ee
We compute 
\begin{equation}\label{coniugato-n-decoupling}
\begin{aligned}
{\mathcal L}_n \Phi_n & =  \Phi_n \big( \Dom {\mathbb I}_2 + \ii \mathtt m_3 {\bf T}(D) + {\bf B}_1 \partial_x + \ii {\bf B}_0^{(I)} \mH |D|^{\frac12} +  {\bf R}_n^{(I)} \big) \\
& \quad + [\ii  \mathtt m_3 {\bf T}(D) +  {\bf B}_1 \partial_x + 
\ii  {\bf B}_0^{(I)} \mH |D|^{\frac12} +  {\bf R}_n^{(I)}, \Psi_n]  \\
& \quad + \Dom \Psi_n 
+ {\bf R}_n^{(II)} + {\bf R}_n^{(II)} \Psi_n\,. 
\end{aligned}
\end{equation}
By  \eqref{definizione T3 A1} and \eqref{defPhi-n} the vector valued commutator
\be\label{VVcommutator}
\begin{aligned}
& \ii [\mathtt m_3 {\bf T}(D), \Psi_n ] = \\
&  \ii \mathtt m_3 \begin{pmatrix}
0 & \!\!\!\! T(D) \psi_n(x, D) + \psi_n(x, D) T(D) \\
- \big( T(D) \overline{\psi_n (x, D)} + \overline{\psi_n (x, D)} T(D) \big) & \!\! 0 
\end{pmatrix}
\end{aligned}
\ee
is  block off-diagonal.

We define a cut off function $ \chi_0 \in {\mathcal C}^\infty(\R, \R)$, even, $ 0 \leq  \chi_0  \leq 1 $, such that 
\begin{equation}\label{cut off simboli 2}
\chi_0(\xi) = \begin{cases}
0 & \quad \text{if } \quad |\xi| \leq \frac12 \\
1 & \quad \text{if} \quad \ |\xi| \geq \frac34\,.
\end{cases}
\end{equation}
\begin{lemma}\label{lemma-bassan-dec}
Let 
\be\label{def:psin} 
\psi_n(x, \xi) := \begin{cases} - \dfrac{\chi_0(\xi) r_n^{(II)}(x, \xi) }{2 \ii \mathtt m_3 T(\xi)} \ \,   
\quad \rm{if} \quad |\xi| > \frac13 \, , \\
0 \quad \qquad \qquad\qquad \, \quad \ \rm{if}  \quad   |\xi| \leq \frac13\,,
\end{cases} \qquad \psi_n \in S^{- n - 1}\,.
\ee 
Then the operator $ \Psi_n $ in \eqref{defPhi-n} solves 
\be\label{ourgoaln}
\ii [\mathtt m_3 {\bf T}(D), \Psi_n] + {\bf R}_n^{(II)}  =   {\bf R}_{T, \psi_n} 
\ee
where
\begin{equation}\label{R T psi n definizione}
 {\bf R}_{T, \psi_n} :=  \ii \begin{pmatrix}
0 & r_{T, \psi_n}(x, D) \\
- \overline{r_{T, \psi_n}(x, D)} 
\end{pmatrix},\quad  r_{T, \psi_n} \in S^{- n - \frac12}  \, , 
\end{equation}
satisfies for all $s_0 \leq s \leq S$
\begin{equation}\label{R T psi n}
\norma r_{T, \psi_n}(x, D)  \norma_{ - n - \frac12, s, \alpha}^{k_0, \gamma} 
\leq_{S, \alpha}   \e \big( 1 + \| \fracchi_0 \|_{s + \sigma + \aleph_n(\alpha) + \alpha + 4}^{k_0, \gamma} \big)\,.
\end{equation}
The map $ \Psi_n $ is real, even, reversibility preserving and
\begin{align}\label{stima psi n}
& \norma \psi_n(x, D) 
\norma_{ - n -1 , s, \alpha}^{k_0, \gamma}  
\leq_{n, S, \alpha} \e (1 + \| \fracchi_0 \|_{s + \sigma + \aleph_n(\alpha)}^{k_0, \gamma})\,, \quad \forall s_0 \leq s \leq S\,,\\
& \label{derivate i psi n}  \norma \partial_i \psi_n(x, D) [\hat \imath] \norma_{- n - 1, s_1, \alpha} \leq_{n, S, \alpha} 
\e \| \hat \imath\|_{s_1 + \sigma + \aleph_n(\alpha)}\,, \\
& \label{derivate i R T psi n}
 \norma \partial_i r_{T, \psi_n}(x, D) [\hat \imath] \norma_{- n - \frac12 , s_1, \alpha} \leq_{n, S, \alpha} \e \| \hat \imath \|_{s_1 + \sigma + \aleph_n(\alpha) + \alpha + 4 }\,.
\end{align}
\end{lemma}

\begin{proof}
 By \eqref{VVcommutator} and  \eqref{cal Rn II}, in order to solve \eqref{ourgoaln}
 with a  remainder $ {\bf R}_{T, \psi_n} \in OPS^{- n - \frac12} $  as in \eqref{R T psi n definizione},  
we have to solve the equation
\be\label{sol-voluta}
\begin{aligned}
& \ii \mathtt m_3 \Big( T(D) \psi_n(x, D) + \psi_n(x, D) T(D) \Big) \\
& \quad + r_n^{(II)}(x, D) = r_{T, \psi_n}(x, D) \in OPS^{- n - \frac{1}{2}} \, . 
\end{aligned}
\ee
By \eqref{expansion symbol}, \eqref{rNTaylor} (applied with $N = 1$),  we have  
\be\label{primo-svi1}
\begin{aligned}
& T(D) \psi_n (x, D) + \psi_n (x, D) T(D) =   {\rm Op}(2 T(\xi) \psi_n (x, \xi)) + {\rm Op}({\mathfrak r}_{T, \psi_n}(x, \xi))  \\
& \quad {\rm where} \quad \mathfrak r_{T, \psi_n} \in S^{- n - \frac12}
\end{aligned}
\ee
because $ T(\xi ) \in S^{3/2} $ and $  \psi_n (x, \xi) \in S^{-n-1} $. The symbol  $ \psi_n ( x, \xi )  $ in \eqref{def:psin} is the solution of  
\be\label{equazione-per-decrescere}
2\ii \mathtt m_3 T(\xi) \psi_n(x, \xi) + \chi_0(\xi) r_n^{(II)}(x, \xi) = 0 
\ee
where the cut-off $ \chi_0 $ is defined in \eqref{cut off simboli 2}. 
Note that $ T( \xi ) = 0 $ for all $ |\xi| \leq 1/3 $ 
(see \eqref{def T}, \eqref{cut off simboli 1}) and that is why we 
do not include in \eqref{equazione-per-decrescere} the symbol $(1- \chi_0 (\xi) ) r_n^{(II)}(x, \xi) \in S^{-\infty} $.
Note also that $ |T(\xi)| \geq c > 0 $ 
for all $ | \xi | \geq 1 / 2 $.  
By \eqref{def:psin} and Lemma \ref{lemma composizione multiplier} and \eqref{stima induttiva disaccoppiamento}, we have, for all $s_0 \leq s \leq S$, 
$$
\norma \psi_n(x, D)
\norma_{-n-1, s, \alpha}^{k_0, \gamma} \leq_{ n, \alpha}  
\norma {\bf R}_n^{(II)} 
\norma_{-n + \frac12, s, \alpha}^{k_0, \gamma} {\leq_{n, S, \alpha}} 
\e (1 + \| \fracchi_0 \|_{s + \sigma + \aleph_n(\alpha)}) 
$$
proving \eqref{stima psi n}.
By  \eqref{primo-svi1} and \eqref{equazione-per-decrescere} 
the remainder $ r_{T, \psi_n}(x, \xi) $ in \eqref{sol-voluta} is 
\begin{equation}\label{puffo 0}
r_{T, \psi_n}(x, \xi)  = \ii \mathtt m_3  \mathfrak r_{T, \psi_n}(x, \xi) + (1 - \chi_0(\xi)) r_n^{(II)}(x, \xi) \in S^{-n-\frac12}\,.
\end{equation}
By \eqref{stima RN derivate xi parametri} (applied  
with $A = T(D)$, $B = \psi_n(x, D)$, $N = 1$, $m = 3/2$, $m' = - n - 1$) we have 
\begin{equation}\label{puffo 2}
\begin{aligned}
\norma \mathfrak r_{T, \psi_n}(x, D) \norma_{- n - \frac12, s, \alpha}^{k_0, \gamma} 
& \leq_{n, s, \alpha} \norma {\bf R}_n^{(II)} 
\norma_{- n + \frac12 , s + 2 + \frac32 + \alpha, \alpha }^{k_0, \gamma} \\
& \stackrel{\eqref{stima induttiva disaccoppiamento}}{\leq_{n, S , \alpha}} \e \big( 1 + \| \fracchi_0 \|_{s + \sigma + \aleph_n(\alpha) + \alpha + 4} \big)
\end{aligned}
\end{equation} 
and the estimate \eqref{R T psi n} for $r_{T, \psi_n}(x, D)$  follows by \eqref{puffo 0} 
 using also \eqref{stima lambda 3 - 1 nuova}, \eqref{stima induttiva disaccoppiamento}. 
The bound \eqref{derivate i psi n} is obtained differentiating the symbol \eqref{def:psin} 
and using \eqref{stima lambda 3 - 1 nuova},  \eqref{stima induttiva disaccoppiamento}, 
\eqref{stima induttiva disaccoppiamento derivate in i}. 
Let us prove the estimate \eqref{derivate i R T psi n}. By differentiating \eqref{puffo 0} with respect to $i$ we get 
\begin{equation}\label{puffo 1}
\begin{aligned}
\partial_i r_{T, \psi_n}(x, \xi) [\hat \imath] & : = \ii \partial_i \mathtt m_3 [\hat \imath]  \mathfrak r_{T, \psi_n}(x, \xi) + \ii  \mathtt m_3   \partial_i \mathfrak r_{T, \psi_n}(x, \xi) [\hat \imath]   \\
& \quad +  (1 - \chi_0(\xi)) \partial_i r_n^{(II)}(x, \xi)[\hat \imath]\,.
\end{aligned}
\end{equation}
Note that, since $T(\xi)$ does not depend on $i$, 
by formulae \eqref{expansion symbol}, \eqref{rNTaylor} (with $A = T(D)$, $B = \psi_n(x, D)$, $N = 1$), we get  
$ \partial_i \mathfrak r_{T, \psi_n}(x, D)[\hat \imath] = \mathfrak r_{T, \partial_i \psi_n [\hat \imath]}(x, D) $
and hence by \eqref{stima RN derivate xi parametri} 
(for $A = T(D) $, $B = \partial_i \psi_n(x, D) [\hat \imath]$, $N = 1$, $m = 3/2$, $m' = - n - 1$) we get 
$$ 
\begin{aligned}
\norma \partial_i \mathfrak r_{T, \psi_n}(x, D)[\hat \imath]  \norma_{- n - \frac12 ,s_1, \alpha } & 
\leq_{n, S, \alpha} 
 \norma \partial_i \psi_n(x, D)[\hat \imath] \norma_{- n - 1, s_1 + 2 + \frac32 + \alpha, \alpha  }  \\ 
& \stackrel{\eqref{derivate i psi n}}{\leq_{n, S, \alpha}}  \e \| \hat \imath \|_{s_1 + \sigma + \aleph_n(\alpha) + \alpha + 4} \,.
 \end{aligned}
 $$ 
The estimate \eqref{derivate i R T psi n} for $\partial_i r_{T, \psi_n}(x, D)[\hat \imath]$ then follows by recalling \eqref{puffo 1} and 
\eqref{stima lambda 3 - 1 nuova}, \eqref{stima induttiva disaccoppiamento derivate in i}, \eqref{puffo 2}. 

Finally, using Lemma
\ref{R-op:rev} and Lemma \ref{even:pseudo}
we see that the map $ \Psi_n $ defined by the symbol \eqref{def:psin}  is even and reversibility preserving 
because $ r_n $ is even and reversible.
\end{proof}

By \eqref{coniugato-n-decoupling} and \eqref{R T psi n definizione} the conjugated operator is 
\be\label{defin:Ln+1}
\begin{aligned}
{\mathcal L}_{n + 1} 
& := \Phi_n^{- 1} {\mathcal L}_n \Phi_n \\ 
& = \Dom {\mathbb I}_2 + \ii \mathtt m_3 {\bf T}(D) + {\bf B}_1 \partial_x + \ii {\bf B}_0^{(I)} \mH |D|^{\frac12}  + {\bf R}_n^{(I)} + {\bf R}_{n + 1} 
\end{aligned}
\ee
where $ {\bf R}_{n + 1} := \Phi_n^{- 1} {\bf R}^*_{n + 1} $ and 
\be \label{bf R n + 1-star}
\begin{aligned}
{\bf R}^*_{n + 1} & := 
{\bf R}_{T, \psi_n} + [{\bf B}_1 \partial_x , \Psi_n] + \ii [{\bf B}_0^{(I)} \mH |D|^{\frac12}, \Psi_n]   \\
& \quad + [{\bf R}_n^{(I)}, \Psi_n] + \Dom \Psi_n  + {\bf R}_n^{(II)} \Psi_n \, . 
\end{aligned}
\ee
Note that $ {\bf R}_{n + 1} $ is the only operator in \eqref{defin:Ln+1} containing off-diagonal terms. 

\begin{lemma}\label{stima-Rn+1-dec}
The operator ${\bf R}_{n + 1} \in OPS^{- n - \frac12}$ satisfies  
\begin{align}\label{stima cal R n+1}
& \norma {\bf R}_{n + 1} \norma_{ - n - \frac12, s, \alpha}^{k_0, \gamma} \leq_{n, S, \alpha} \e ( 1 + \| \fracchi_0 \|_{s + \sigma + \aleph_{n + 1}(\alpha)}^{k_0, \gamma}) \,, \quad \forall s_0 \leq s \leq S\,,\\
& \label{stima derivata i cal R n+1}
\norma \partial_i {\bf R}_{n + 1} [\hat \imath] \norma_{ - n - \frac12, s_1 , \alpha} \leq_{n, S, \alpha} \e \| \hat \imath\|_{s_1 + \sigma + \aleph_{n + 1}(\alpha)}
\end{align}
where the  constant $ \aleph_{n+1}(\alpha)$  is defined  in \eqref{definizione kn(alpha)}. 
\end{lemma}

\begin{proof}
{\sc Proof of \eqref{stima cal R n+1}}. We first estimate separately all the terms of $ {\bf R}^*_{n + 1} $ in \eqref{bf R n + 1-star}.
The operator  $ {\bf R}_{T, \psi_n} \in OPS^{- n - \frac12} $ in \eqref{R T psi n definizione}  satisfies  \eqref{R T psi n}. 
By \eqref{bf B1} and since $  \psi_n(x, D ) \in OPS^{-n-1} $, see \eqref{def:psin}, we have  
$$
[{\bf B}_1 \partial_x, \Psi_n] = \begin{pmatrix}
0 & [a_{11} \partial_x, \psi_n(x, D)] \\
[a_{11} \partial_x, \overline{\psi_n(x, D)}] & 0
\end{pmatrix} \in  OPS^{- n - 1} \subset OPS^{- n - \frac12} \, . 
$$
Moreover Lemma \ref{lemma tame norma commutatore} (with $m = 1$, $m' = - n - 1$) implies 
\begin{align*}
\norma [a_{11} \partial_x, \psi_n(x, D)]  
\norma_{-n- \frac12, s, \alpha}^{k_0, \gamma} & \leq \norma [a_{11} \partial_x, \psi_n(x, D)]  
\norma_{-n-1, s, \alpha}^{k_0, \gamma} \\
&  \leq_{n, s, \alpha} \| a_{11} \|_{s + n + 3 + \alpha}^{k_0, \gamma} 
\norma \psi_n(x, D)  \norma_{-n-1, s_0 + 3 + \alpha, \alpha + 1}^{k_0, \gamma} \\
& \quad + \| a_{11} \|_{s_0 + n + 3 + \alpha}^{k_0, \gamma} \norma \psi_n(x, D) 
\norma_{-n-1, s  + 3 + \alpha, \alpha + 1}^{k_0, \gamma} \\
& \stackrel{\eqref{stima a 15}, \eqref{stima psi n}, \eqref{ansatz I delta}}{\leq_{n, S, \alpha}} \e \big( 1 + 
\| \fracchi_0 \|_{s + \sigma + \aleph_n(\alpha + 1) + n + \alpha + 3 }^{k_0, \gamma} \big)\,.
\end{align*}
We also claim that 
$[{\bf B}_0^{(I)} \mH |D|^{\frac12}, \Psi_n] \in OPS^{- n- \frac12}$. Indeed by \eqref{bf B0 (I)} we have 
$$
\begin{aligned}
& [{\bf B}_0^{(I)} \mH |D|^{\frac12}, \Psi_n] = \\
&  \begin{pmatrix}
0 & \!\!\!\! \!\!\!\!\!\!\!  a_{12} \mH |D|^{\frac12} \psi_n(x, D) + \psi_n(x, D) a_{12} \mH |D|^{\frac12} \\
- a_{12} \mH |D|^{\frac12} \overline {\psi_n(x, D)} - \overline {\psi_n(x, D)} a_{12} \mH |D|^{\frac12}  & 
\!\!\!\! 0
\end{pmatrix}
\end{aligned}
$$
and  \eqref{estimate composition parameters}, \eqref{stima a 15}, \eqref{stima psi n} imply 
$$ \norma [{\bf B}_0^{(I)} \mH |D|^{\frac12}, \Psi_n] 
\norma_{- n - \frac12 , s, \alpha}^{k_0, \gamma} \leq_{n, S, \alpha} 
\e (1 + \|\fracchi_0 \|_{s + \sigma + \aleph_n(\alpha ) + n + \alpha + 1  }^{k_0, \gamma}).$$ 
In addition the operator $[{\bf R}_n^{(I)},  \Psi_n] \in OPS^{- n - 1} \subset OPS^{- n - \frac12}$ because  (see \eqref{cal Rn I}, \eqref{defPhi-n})
$$
\begin{aligned}
& [{\bf R}_n^{(I)},  \Psi_n] = \\
& \begin{pmatrix}
 0 & \!\!\!\!\!\!\!\!  r_n^{(I)}(x, D) \psi_n(x, D) - \psi_n(x, D) \overline {r_n^{(I)}(x, D)} \\
 \overline {r_n^{(I)}(x, D)} \overline {\psi_n(x, D)} - \overline {\psi_n(x, D)} r_n^{(I)}(x, D) & \!\!\!\!\!\! 0
\end{pmatrix}
\end{aligned}
$$
and \eqref{estimate composition parameters}, \eqref{stima induttiva disaccoppiamento}, \eqref{stima psi n} imply 
$$
\norma [{\bf R}_n^{(I)}, \Psi_n] 
\norma_{-n-\frac12,s, \alpha}^{k_0, \gamma} \leq \norma [{\bf R}_n^{(I)}, \Psi_n] 
\norma_{-n-1,s, \alpha}^{k_0, \gamma} \leq_{n, S, \alpha} \e 
\big(1 + \| \fracchi_0  \|_{s +\sigma + \aleph_n(\alpha)+ n  + \alpha + 1}^{k_0, \gamma} \big) \, .
$$ 
Moreover
$ \Dom \Psi_n \in OPS^{- n - 1} \subset OPS^{- n - \frac12} $ satisfies 
$$
\begin{aligned}
\norma \Dom \Psi_n 
 \norma_{-n- \frac12,s, \alpha}^{k_0, \gamma}  \leq \norma \Dom \Psi_n 
 \norma_{-n-1,s, \alpha}^{k_0, \gamma} & \lessdot \norma  \Psi_n 
 \norma_{-n-1, s + 1, \alpha}^{k_0, \gamma} \\
&  \leq_{n, S, \alpha} \e (1 + \| \fracchi_0 \|_{s + \sigma + 
\aleph_n(\alpha) + 1}^{k_0, \gamma})
\end{aligned}
$$
by \eqref{stima psi n}. Finally 
 $ {\bf R}_n^{(II)} \Psi_n \in OPS^{- 2 n - \frac12} \subset OPS^{- n - \frac12} $ and by \eqref{estimate composition parameters} 
 (applied with $m = \frac12 - n $, $m' = - n - 1$), \eqref{stima induttiva disaccoppiamento}, \eqref{stima psi n} we have 
$$
\norma {\bf R}_n^{(II)} \Psi_n  \norma_{-n- \frac12, s, \alpha}^{k_0, \gamma} \leq  
\norma {\bf R}_n^{(II)} \Psi_n  \norma_{-2 n- \frac12, s, \alpha}^{k_0, \gamma}
\leq_{n, S, \alpha} \e (1 + \| \fracchi_0 \|_{s + \sigma + \aleph_n(\alpha) + n + \alpha + \frac12 }^{k_0, \gamma})\,.
$$
Collecting all the previous estimates we deduce that  ${\bf R}_{n + 1}^* $ defined in  \eqref{bf R n + 1-star} 
is in $ OPS^{- n - \frac12} $   and
\be\label{finale-Rn+1-star}
\norma {\bf R}_{n + 1}^* \norma_{- n - \frac12, s, \alpha}^{k_0, \gamma} 
\leq_{n, S, \alpha} \e (1 + \| \fracchi_0 \|_{s + \sigma +  \aleph_n(\alpha + 1) + n + \alpha + 4 }^{k_0, \gamma}) \, . 
\ee
Now \eqref{estimate composition parameters} (applied with $m = 0$, $m' = - n - \frac12 $), 
Lemma \ref{Neumann pseudo diff},  \eqref{stima psi n},  \eqref{finale-Rn+1-star} imply 
\begin{align*}
\norma {\bf R}_{n + 1} \norma_{-n - \frac12 , s, \alpha}^{k_0, \gamma} & =  \norma \Phi_n^{- 1} {\bf R}_{n + 1}^* 
 \norma_{-n - \frac12 , s, \alpha}^{k_0, \gamma} \\
& \leq_{n, s, \alpha} \norma \Phi_n^{- 1} \norma_{0, s, \alpha}^{k_0, \gamma} 
\norma {\bf R}_{n + 1}^* \norma_{-n - \frac12 , s_0 + \alpha , \alpha}^{k_0, \gamma} + 
\norma \Phi_n^{- 1} \norma_{0, s_0, \alpha}^{k_0, \gamma} 
\norma {\bf R}_{n + 1}^* \norma_{-n - \frac12, s + \alpha , \alpha}^{k_0, \gamma} \\
& \leq_{n, S, \alpha} \e \big(1 + \| \fracchi_0 \|_{s + \sigma +  \aleph_n(\alpha + 1) +  n + 2 \alpha + 4 }^{k_0, \gamma} \big)\,
\end{align*}
 which is \eqref{stima cal R n+1}, recalling \eqref{definizione kn(alpha)}. 
 \\[1mm]
{\sc  Proof of  \eqref{stima derivata i cal R n+1}}.  
First we estimate $\partial_i {\bf R}_{n + 1}^*$ in \eqref{bf R n + 1-star}. 
The operator  $ \pa_i {\bf R}_{T, \psi_n} $  satisfies  \eqref{derivate i R T psi n}. 
Then we have 
 $$
 \partial_i [{\bf B}_1 \partial_x, \Psi_n][\hat \imath] = [\partial_i {\bf B}_1[\hat \imath] \partial_x, \Psi_n] + [{\bf B}_1 \partial_x , \partial_i \Psi_n [\hat \imath]]\, .
 $$
 Hence  Lemma \ref{lemma tame norma commutatore} (with $m = 1$, $m' = - n - 1$), the estimates 
 of  $ a_{11} $ in 
\eqref{stima a 15}, \eqref{stima derivata i a 15},  \eqref{stima psi n}, \eqref{derivate i psi n}, imply 
 $$
 \begin{aligned}
 \norma \partial_i [{\bf B}_1 \partial_x, \Psi_n] [\hat \imath] \norma_{- n - \frac12, s_1, \alpha} 
 & 
 \leq \norma \partial_i [{\bf B}_1 \partial_x, \Psi_n] [\hat \imath] \norma_{- n - 1, s_1, \alpha} \\
&  \leq_{n, S, \alpha} \e \| \hat \imath \|_{s_1 + \sigma + 
\aleph_n(\alpha + 1) + n + \alpha + 3}\,.
\end{aligned}
 $$
 The terms $\partial_i [{\bf B}_0^{(I)} \mH |D|^{\frac12}, \Psi_n]$, $\partial_i [{\bf R}_n^{(I)}, \Psi_n]$ may be estimated similarly. In addition 
 \begin{align}
 \norma \partial_i \big(\Dom \Psi_n \big) [\hat \imath]\norma_{- n - \frac12, s_1, \alpha} & \leq \norma \partial_i \big(\Dom \Psi_n \big) [\hat \imath]\norma_{- n - 1, s_1, \alpha} \lessdot  \norma \partial_i  \Psi_n [\hat \imath]\norma_{- n - 1, s_1 + 1 , \alpha} \nonumber\\
 & \stackrel{\eqref{derivate i psi n}}{\leq_{n, S, \alpha}} \e \| \hat \imath\|_{s_1 + \sigma + \aleph_n(\alpha) + 1} \, . \nonumber 
 \end{align}
Finally $\norma \partial_i ({\bf R}_n^{(II)} \Psi_n) [\hat \imath] \in OPS^{- 2 n - \frac12} \subset OPS^{- n - \frac12}$.
Hence applying \eqref{estimate composition parameters}  with $m =  - n + \frac12 $, $m' =  - n - 1$,  and using 
\eqref{stima induttiva disaccoppiamento}, \eqref{stima induttiva disaccoppiamento derivate in i}, \eqref{stima psi n},\eqref{derivate i psi n} we get 
$$
\begin{aligned} 
\norma \partial_i ({\bf R}_n^{(II)} \Psi_n) [\hat \imath] \norma_{- n - \frac12, s_1, \alpha}  
& \leq \norma \partial_i ({\bf R}_n^{(II)} \Psi_n) [\hat \imath] \norma_{- 2 n - \frac12, s_1, \alpha} \\
& \leq_{n, S, \alpha} \e \|\hat \imath \|_{s_1 + \sigma + \aleph_n(\alpha) +n+ \alpha + \frac12} \, .
\end{aligned}
$$
Collecting the previous bounds 
we conclude that 
$$
\norma  \partial_i {\bf R}_{n + 1}^* [\hat \imath] \norma_{- n - \frac12, s_1, \alpha} 
\leq_{n, S, \alpha} \e \| \hat \imath \|_{s_1 + \sigma +  \aleph_n(\alpha + 1) + n + \alpha + 4} 
$$ 
and  
the estimate \eqref{stima derivata i cal R n+1} follows by 
 $$
 \begin{aligned}
&  \partial_i {\bf R}_{n + 1}[\hat \imath] = \partial_i \big(\Phi_n^{- 1} {\bf R}_{n + 1}^* \big) [\hat \imath]  
 = 
 \partial_i \Phi_n^{- 1} [\hat \imath] {\bf R}_{n + 1}^* + \Phi_n^{- 1} \partial_i {\bf R}_{n + 1}^* [\hat \imath] 
 \quad \text{and} \\
 &  \partial_i \Phi_n^{- 1}[\hat \imath] = - \Phi_n^{- 1} \partial_i \Phi_n [\hat \imath] \Phi_n
 \end{aligned}
 $$
applying \eqref{estimate composition parameters} (with $m = 0$, $m' = - n - \frac12$), Lemma \ref{Neumann pseudo diff} and the estimates \eqref{stima psi n}, \eqref{derivate i psi n}.
\end{proof}

By \eqref{defin:Ln+1} and \eqref{stima cal R n+1}-\eqref{stima derivata i cal R n+1}   
the operator $ {\mathcal L}_{n+1} $ has the same form \eqref{cal Ln disaccoppiamento}-\eqref{cal Rn II} 
with $ {\bf R}_{n+1}^{(I)} $, $ {\bf R}_{n+1}^{(II)} $ 
that satisfy the estimates \eqref{stima induttiva disaccoppiamento}-\eqref{stima induttiva disaccoppiamento derivate in i}
at the step $ n+  1$. Hence we can repeat iteratively the procedure of 
Lemmata \ref{lemma-bassan-dec} and \ref{stima-Rn+1-dec}. Applying it $ M $-times ($ M $ will be fixed in
\eqref{relazione mathtt b N}) we derive the following proposition.

\begin{proposition}\label{Lemma finale decoupling}
The real invertible map $ {\bf \Phi}_M := \Phi_4 \circ \ldots \circ \Phi_{M+4} $  satisfies  the estimate 
\begin{equation}\label{definizione bf PhiN}
\begin{aligned}
& \norma {\bf \Phi}_M^{\pm 1} - {\mathbb I}_2 \norma_{0,s,0}^{k_0, \gamma}\,,\, \norma ({\bf \Phi}_M^{\pm 1} - {\mathbb I}_2)^* \norma_{0,s,0}^{k_0, \gamma}  \leq_{S, M } 
\e (1 + \|\fracchi_0 \|_{s + \sigma + \aleph_M (0)}^{k_0, \gamma}) , \\
&  \qquad \qquad \forall s_0 \leq s \leq S 
\end{aligned} 
\end{equation}
and conjugate $ {\mathcal L}_4 $ to the real, even and reversible operator 
\begin{equation}\label{cal LN decoupling}
\begin{aligned}
{\mathcal L}_M & := {\bf \Phi}_M^{- 1} {\mathcal L}_4 {\bf \Phi}_M =  \Dom {\mathbb I}_2 + \ii \mathtt m_3
 {\bf T}(D) + {\bf B}_1(\vphi, x) \partial_x  \\
 & + \ii {\bf B}_0^{(I)}(\vphi, x) \mH |D|^{\frac12}  +  {\bf R}_M^{(I)} + {\bf R}_M^{(II)}
  \end{aligned}
\end{equation}
where the remainders 
\be\label{resti prima Egorov}
\begin{aligned}
& {\bf R}_M^{(I)} := \begin{pmatrix}
r_M^{(I)}(\vphi, x, D) & 0 \\
0 & \overline {r_M^{(I)}(\vphi, x, D)}
\end{pmatrix} \in OPS^0  \,, \\
&  
{\bf R}_M^{(II)} := \begin{pmatrix}
0 & {\mathcal R}_M^{(II)} \\
 \overline {\mathcal R}_M^{(II)} & 0
\end{pmatrix}  \in OPS^{\frac12 - M} 
\end{aligned}
\ee
satisfy the estimates 
\begin{equation}\label{stima resti prima Egorov}
\norma {\bf R}_M^{(I)} \norma_{0, s, \alpha}^{k_0, \gamma} + \norma {\bf R}_M^{(II)} 
 \norma_{ - M + \frac12 , s, \alpha}^{k_0, \gamma} \leq_{S,\a}  
\e \big(1 + \| \fracchi_0 \|_{s + \sigma + \aleph_M(\alpha)}^{k_0, \gamma} \big)\,, \ \  \forall s_0 \leq s \leq S\,,
\end{equation}
and the constant $ \aleph_M (\alpha)$ is defined recursively by \eqref{definizione kn(alpha)}. 
Moreover 
\begin{align}\label{stima derivate resti prima Egorov}
& \norma \partial_i {\bf R}_M^{(I)} [\hat \imath] \norma_{0, s_1, \alpha} + \norma  \partial_i {\bf R}_M^{(II)} [\hat \imath]
 \norma_{ - M + \frac12 , s_1 , \alpha} \leq_{M, S,\a}  
\e \| \hat \imath \|_{s_1 + \sigma + \aleph_M(\alpha)}  \\
& \label{bf PhiN derivate in i}
\norma \partial_i {\bf \Phi}_M^{\pm 1} [\hat \imath] \norma_{0, s_1, 0} \,,\,\norma \partial_i ({\bf \Phi}_M^{\pm 1})^* [\hat \imath] \norma_{0, s_1, 0}  \leq_{M, S}  
\e \| \hat \imath \|_{s_1 + \sigma + \aleph_M(0)}\,. 
\end{align}
\end{proposition}

\begin{proof}
Let us prove \eqref{definizione bf PhiN}. 
For all $4 \leq n \leq M + 4$,  $s_0 \leq s \leq S$, we have 
$$
\begin{aligned}
\norma \Phi_n - {\mathbb I}_2 \norma_{0,s,0}^{k_0, \gamma} 
\stackrel{\eqref{defPhi-n}} = 
\norma \Psi_n \norma_{0,s,0}^{k_0, \gamma} & \stackrel{\eqref{stima psi n}}{\leq_S} \e \big(1 + \| \fracchi_0 \|_{s + \sigma + \aleph_n(0)}^{k_0, \gamma} \big) \\
& \leq_S \e \big(1 + \| \fracchi_0 \|_{s + \sigma + \aleph_M (0)}^{k_0, \gamma} \big) 
\end{aligned}
$$
and  \eqref{definizione bf PhiN} follows as in the proof of Corollary 4.1 in \cite{BBM-Airy}. The estimate on the adjoint operator 
$ ( {\bf \Phi}_M^{\pm 1} - {\mathbb I}_2)^* $ follows as well since  Lemma \ref{stima pseudo diff aggiunto} implies 
$\norma (\Phi_n^{\pm 1} - {\mathbb I}_2)^* \norma_{0, s, 0}^{k_0, \gamma} \leq_M 
\norma \Phi_n^{\pm 1} - {\mathbb I}_2 \norma_{0, s + s_0, 0}^{k_0, \gamma} $.
Also \eqref{bf PhiN derivate in i} is proved analogously. 
\end{proof}

The operator $ {\mathcal L}_M $ in \eqref{cal LN decoupling} is block-diagonal  up to the smoothing remainder
  $ {\bf R}_M^{(II)}  \in OPS^{\frac12 - M} $. The prize which has been paid is that $  {\bf R}_M^{(II)} $ 
  depends on  $ \aleph_M (\alpha) $-derivatives of the approximate  solution $ \fracchi  $, i.e.   on 
  $ \| \fracchi \|_{s + \sigma + \aleph_M (\alpha)}^{k_0, \gamma} $ in \eqref{stima resti prima Egorov}.
In any case, the number of regularizing steps $ M $ is fixed (independently on $ s $, see \eqref{relazione mathtt b N}, \eqref{alpha beta}), 
 determined by the KAM reducibility scheme in section \ref{sec: reducibility}. 

\section{Elimination of order $ \pa_x $: Egorov method}\label{egorov}

The goal of this section is 
to remove  ${\bf B}_1(\vphi, x) \partial_x$ from the operator 
$\mL_M $ defined in \eqref{cal LN decoupling}. We rewrite 
\begin{equation}\label{cal L4 pseudo}
\mL_M = \Dom {\mathbb I}_2 + {\bf P}_0 (\vphi, x, D) + {\bf R}_M^{(II)} 
\end{equation}
where we denote the whole block-diagonal part by 
\begin{equation}\label{defP0}
\begin{aligned}
 {\bf P}_0(\vphi, x, D) & :=  \ii \mathtt m_3 {\bf T} (D) + {\bf B}_1(\vphi, x) \pa_x 
+ \ii {\bf B}_0^{(I)} (\vphi, x) \mH | D|^{\frac12} + {\bf R}_M^{(I)}  \\
&  =
\begin{pmatrix}
{\rm Op}(p_0) & 0 \\
0 & \overline {{\rm Op}( p_0)}
\end{pmatrix} 
\end{aligned}
\end{equation}
and, by \eqref{definizione T3 A1}, \eqref{def T}, \eqref{bf B1}, \eqref{bf B0 (I)}, \eqref{resti prima Egorov}, the 
associated symbol is
\be \label{p0 Egorov iniziale}
\begin{aligned}
\quad p_0(\vphi, x, \xi)&  :=  \ii \big( \mathtt m_3 T( \xi)  + a_{11}(\vphi, x) \xi  \big)  \\
& \quad + a_{12}(\vphi, x) \chi(\xi) {\rm sign}(\xi) |\xi|^{\frac12}   +  r_M^{(I)}(\vphi, x, \xi) \in S^{3/2} 
\end{aligned}
\ee
where $ T(\xi) =  \chi(\xi)|\xi |^{1/2} (1+\kappa \xi^2)^{1/2} $.
\\[1mm]
{\bf Egorov approach.}
We transform  ${\mathcal L}_M $ in \eqref{cal L4 pseudo}  by the  {\it flow} of the system of pseudo-PDEs 
\begin{equation}\label{flusso Phi}
\begin{aligned}
& \partial_t \begin{pmatrix}
u\\
\overline u 
\end{pmatrix}
 = \ii {\bf a}(\vphi, x) | D |^{\frac12} \begin{pmatrix}
 u \\
 \overline u
 \end{pmatrix} \qquad {\rm where}  \\
 & \qquad {\bf a}(\vphi, x) := \begin{pmatrix}
a(\vphi, x) & 0 \\ 
0 & - a(\vphi, x)
\end{pmatrix}
\end{aligned}
\end{equation}
and $ a(\vphi, x ) $ is a {\it real} valued function to be determined, see \eqref{definizione a}. 
The flow ${\bf \Phi}(\vphi, t) $ of \eqref{flusso Phi} has the block-diagonal form 
\be\label{flusso-diagonale}
{\bf \Phi}(\vphi, t) := \begin{pmatrix}
\Phi(\vphi, t) & 0 \\
0 & \overline \Phi(\vphi, t)
\end{pmatrix} 
\ee
where $ \Phi(\vphi, t) $ is the flow of the scalar linear pseudo-PDE 
\be\label{pseudo-PDE} 
\partial_t u = \ii a(\vphi, x) |D|^{\frac12} u \, . 
\ee
In the Appendix we prove that its flow $ \Phi (\vphi, t) : H^s \mapsto H^s  $  
is well defined in the Sobolev spaces $ H^s $,  see Propositions \ref{Prop0-flow}, \ref{Prop1-flow}. 
The flow $ \Phi (\vphi, t)   $ solves 
\be \label{flow-propagator}
\begin{aligned}
& \begin{cases}
\pa_t \Phi (\vphi, t ) = \ii A (\vphi) \Phi (\vphi, t ) \\ 
\Phi (\vphi, 0 ) = {\rm Id} \, , 
\end{cases} \\
&    A(\vphi)  :=   {\mathfrak a}(\vphi, x, D) \, ,   \ \   {\mathfrak a} (\vphi, x, \xi ) := a(\vphi, x) \chi(\xi) | \xi |^{\frac12} \, , 
\end{aligned}
\ee
and, since \eqref{pseudo-PDE} is autonomous, it  satisfies the group property 
\be\label{group-flow}
\Phi (\vphi, t_1 + t_2 ) =  \Phi (\vphi, t_1  ) \circ  \Phi (\vphi, t_2 ) \, ,  \qquad 
\Phi (\vphi, t )^{-1} = \Phi (\vphi, - t ) \, . 
\ee
Moreover, assuming that $ a  (\omega, \kappa, \cdot) $ is $ {k_0} $-times differentiable smooth with respect to the parameters 
$ \om  $ and $\kappa$, the flow $ \Phi (\vphi, t, \om, \kappa )  $ is also $ k_0 $-times differentiable with 
respect to $ \om $ and $\kappa$ see Proposition \ref{lemma:tame derivate flusso}. 
If $ a(\vphi, x)   $ is odd$ (\vphi)$-\even{$(x)$} 
then the flow $ {\bf \Phi} (\vphi, t ) $ is even and reversibility preserving. 

\smallskip
We denote for simplicity $\Phi:= \Phi(\vphi) := \Phi(\vphi, 1)$ the time-$ 1 $ flow map of \eqref{pseudo-PDE}
 and ${\bf \Phi} := {\bf \Phi}(\vphi) := {\bf \Phi}(\vphi, 1)$ the time-$ 1 $ flow map of the system \eqref{flusso Phi}. 
The transformed operator is
\begin{equation}\label{coniugio flusso di una PDE iperbolica}
\begin{aligned}
{\mathcal L}_M^{(1)}  & :={\bf \Phi} {\mathcal L}_M {\bf \Phi}^{- 1}  = 
\Dom {\mathbb I}_2 + {\bf \Phi}(\vphi) {\bf P}_0(\vphi, x, D) {\bf \Phi}(\vphi)^{- 1}   \\
& \quad +  
{\bf \Phi}(\vphi) \Dom \{{\bf \Phi}(\vphi)^{- 1}\}   +  {\bf \Phi} {\bf R}_M^{(II)} {\bf \Phi}^{- 1}\,. 
\end{aligned}
\end{equation}
The  terms $ {\bf \Phi}(\vphi) {\bf P}_0(\vphi, x, D) {\bf \Phi}(\vphi)^{- 1} $ and 
$ {\bf \Phi}(\vphi) \, \Dom \{{\bf \Phi}(\vphi)^{- 1}\} $ are  block-diagonal. 
They are classical pseudo-differential operators and 
shall be analyzed by an Egorov type argument.
On the other hand the off-diagonal 
term $ {\bf \Phi} {\bf R}_M^{(II)} {\bf \Phi}^{- 1} $ 
is very regularizing  and satisfy tame estimates.  The contents of this section are summarized in 
Proposition \ref{Prop:Egorov}. 
\\[1mm]
{ \bf Analysis of $ {\bf \Phi}(\vphi) {\bf P}_0(\vphi, x, D) {\bf \Phi}(\vphi)^{- 1} $ in \eqref{coniugio flusso di una PDE iperbolica}.} 
\\[1mm]
We first consider  ${\bf P} (\vphi, t) := {\bf \Phi}(\vphi, t) {\bf P}_0 {\bf \Phi}(\vphi, t)^{- 1} $. By
\eqref{defP0} and  \eqref{flusso-diagonale} it reads
  \be\label{defPt}
  \begin{aligned}
  & {\bf P} (\vphi, t) := \begin{pmatrix}
   P(\vphi, t) & 0 \\
   0 & \overline P(\vphi, t)
   \end{pmatrix}\,,  \\
   & P(\vphi, t) := \Phi(\vphi, t) p_0(\vphi, x, D) \Phi^{- 1}(\vphi, t) \, .
   \end{aligned}
    \ee 
The operator $ {\bf P} (\vphi, t) $  solves the vector valued  Heisenberg equation
$$
\begin{cases}
\partial_t {\bf P} (\vphi, t) = \ii [ {\bf a}(\vphi, x) |D|^{\frac12}, {\bf P} (\vphi, t)] \cr 
{\bf P} (\vphi, 0) = {\bf P}_0(\vphi) \, , 
\end{cases}
$$
namely  the operator $P(\vphi, t)$ solves the usual Heisenberg equation 
\begin{equation}\label{equazione Egorov}
\begin{aligned}
& \begin{cases}
\partial_t P (\vphi, t) = \ii [ A(\vphi) , P (\vphi, t)]    \cr 
P (\vphi, 0) = P_0 := p_0(\vphi, x, D)  
\end{cases}  \\
&{\rm where} \qquad  A (\vphi ) := {\mathfrak a}(\vphi, x, D) = a(\vphi, x) |D|^{\frac12} \, . 
\end{aligned}
\end{equation}
We use the notation $ |D|^{\frac12} := {\rm Op}(\chi(\xi) |\xi|^{\frac12})$ as in \eqref{definizione |D| m}.

We look for an approximate solution $ Q(\vphi, t) := q(t, \vphi, x, D)  $ of  \eqref{equazione Egorov}  with a symbol of the form (expanded in decreasing symbols)
\begin{equation}\label{espansione simbolo P(t)}
\begin{aligned}
& q(t, \vphi, x, \xi) = \sum_{n = 0}^M  q_n (t, \vphi, x, \xi) \, , \\
& q_n (t,\vphi, x, \xi) \in S^{\frac12 (3 - n)} \, , \quad \forall n  = 0, \ldots, M \, . 
\end{aligned}
\end{equation}
The order of the commutator $ [ A(\vphi), Q(\vphi) ] $ is  strictly less than the order of $ Q(\vphi) $. 
Let $   {\mathfrak a} \star q $ denote the symbol of 
the commutator, i.e.   
$  [ A(\vphi), Q(\vphi )] := {\rm Op} ( {\mathfrak a} \star q) $, see \eqref{symbol commutator}.

\begin{lemma}\label{proprieta sigma(q,A)}
{\bf (Commutator symbol)} 
If $ q \in S^{ m}$, $ m \in \R $, then $    {\mathfrak a} \star q \in S^{ m - \frac{1}{2}}$ and  
\begin{align*}
\norma [ A, Q ]  \norma_{m- \frac12, s, \alpha}^{k_0, \gamma} = 
\norma {\rm Op}(  {\mathfrak a} \star q )  \norma_{m- \frac12, s, \alpha}^{k_0, \gamma} & 
\leq_{m, s, \alpha} \norma {\rm Op}(q)  \norma_{m, s  + \alpha + 3, \alpha + 1}^{k_0, \gamma} 
\| a\|_{s_0 + |m| + \alpha  + 2 }^{k_0, \gamma}  \\
& \qquad \ + \norma {\rm Op}(q) \norma_{m, s_0  + \alpha + 3, \alpha + 1}^{k_0, \gamma} 
\| a\|_{s +  |m| + \alpha  + 2 }^{k_0, \gamma} \, . 
\end{align*}
\end{lemma}

\begin{proof}
By Lemma \ref{lemma tame norma commutatore} with  $m' = 1/2 $.
\end{proof}

We solve approximately the equation \eqref{equazione Egorov} in decreasing orders. 
We define  $ q_0 $ as the solution of 
\begin{equation}\label{equazione per q0}
\begin{cases}
\partial_t q_0(t, \vphi, x, \xi) = 0 \\
q_0(0, \vphi, x, \xi) = p_0(\vphi, x, \xi)\, , 
\end{cases}
\end{equation}
namely 
\begin{equation}\label{definizione q0}
q_0(t, \vphi,  x, \xi) =  p_0(\vphi, x, \xi) \in S^{\frac32}\,, \qquad \forall t \in [0, 1]\,.
\end{equation}
Then we define inductively the symbols $q_n (t, \vphi, x, \xi)$, $ n \geq 1 $,  as the solutions of
\begin{equation}\label{equazione per qk}
\begin{cases}
\partial_t q_n = \ii  {\mathfrak a} \star q_{n - 1}  \\
q_n (0, \vphi, x, \xi) = 0\, , 
\end{cases}
\end{equation}
namely
\begin{equation}\label{definizione qk}
q_n (t, \vphi, x, \xi) = \ii \int_0^t (  {\mathfrak a} \star q_{n - 1})(\tau, \vphi, x, \xi) \, d \tau \, . 
\end{equation} 
Each symbol $ q_n \in S^{\frac12(3 - n)} $, $ \forall n = 0, \ldots, M $. Actually
$  q_0 \in S^{3/2} $ by \eqref{definizione q0}. 
Then, by induction, if $ q_{n - 1} \in S^{\frac12 (3 - (n - 1))}$ we deduce that 
$   {\mathfrak a} \star q_{n - 1} \in S^{\frac12(3 - n)}  $ by Lemma \ref{proprieta sigma(q,A)}. The 
quantitative estimate is given in \eqref{stima qk Egorov}.

\smallskip

We now expand the symbol $ q $  in \eqref{espansione simbolo P(t)}
writing explicitly the terms of order greater than $ 0 $. They come from $q_0 \in S^{\frac32}$, $q_1 \in S^{1}$ and $q_2 \in S^{\frac12}$
(all the symbols $ q_n $, $ n \geq 2 $, are yet in $ S^0 $). 
For that we further expand as in \eqref{Expansion Moyal bracket} 
the symbol of the commutator  as 
\be\label{decomposizione compound}
 ( {\mathfrak a} \star q)(t, \vphi, x, \xi)  = - \ii \{  {\mathfrak a}, q \}(t, \vphi, x,\xi) + {\mathtt r}_{\mathtt 2} ( {\mathfrak a}, q)(t, \vphi, x,\xi) 
\ee
where 
$ \{  {\mathfrak a}, q \} = (\partial_x q)  (\partial_\xi   {\mathfrak a}) - (\partial_\xi q) (\partial_x   {\mathfrak a}) $ is the Poisson bracket 
and $ {\mathtt r}_{\mathtt 2} ( {\mathfrak a}, q)  $ is a lower order symbol. 

\begin{lemma}\label{proprieta r2(q,A)}
{\bf (Lower order commutator symbol)} 
If $ q \in S^{m}$, $ m \in \R $, then $ {\mathtt r}_{\mathtt 2} ( {\mathfrak a}, q) \in S^{m - \frac{3}{2}} $ and  
\begin{align*}
\norma {\rm Op}\big( {\mathtt r}_{\mathtt 2} ( {\mathfrak a}, q) \big) \norma_{m - \frac32, s, \alpha}^{k_0, \gamma} & 
\leq_{m, s, \alpha}  \norma {\rm Op}(q) \norma_{m, s + \alpha + 5, \alpha + 2}^{k_0, \gamma}  \|a\|_{s_0 + |m| + \alpha + 4}^{k_0, \gamma}  \\
& \qquad \ + \norma {\rm Op}(q)
 \norma_{m, s_0 + \alpha + 5, \alpha + 2}^{k_0, \gamma}  \| a\|_{s + |m| + \alpha + 4}^{k_0, \gamma}\,.
\end{align*}
\end{lemma}

\begin{proof}
Apply  \eqref{stima RN derivate xi parametri} to ${\rm Op}(q) \circ {\rm Op}( {\mathfrak a})$ 
and to ${\rm Op}( {\mathfrak a}) \circ {\rm Op}(q)$ with $ N = 2 $ and $ m' = 1/2 $ (and use \eqref{norm-increa}). 
\end{proof}

We now get the expansion of the symbol 
$ q_{\leq 2} (\vphi, x, \xi) := q_{\leq 2}(1, \vphi, x, \xi)  = (q_0 + q_1 + q_2) (1, \vphi, x, \xi) $.

\begin{lemma}\label{ordini maggiori di 0 Egorov} 
{\bf (Expansion of approximate solution)}
The symbol $ q_{\leq 2}  = q_0 + q_1 + q_2 $ has the 
 expansion
\be\label{simboli<2}
q_{\leq 2} = \ii \mathtt m_3 T(\xi) + \ii \big(a_{11} - 
\frac32 \mathtt m_3 \sqrt{\kappa} \, a_x \big) \xi +
\big( \ii  a_{13} + a_{12} \, {\rm sign}(\xi) \big) \chi(\xi)  |\xi|^{\frac12} + r_{q_{\leq 2}}
\ee
where the symbol
\be\label{defrleq2}
r_{q_{\leq 2}}  := r_{q_{\leq 2}}(\vphi,x, \xi)  = 
r_M^{(I)} + r_{{\mathfrak a} p_0 }^{(0)}  +  r_{{\mathfrak a} p_0 }^{(1)}  + r_{{\mathfrak a} p_0 }^{(2)}  \in S^0 
\ee 
is defined  in  
\eqref{rq1q2}, \eqref{r p0 A (1)}, \eqref{r p0 A (3)}, and $ r_M^{(I)} $ in Proposition 
\ref{Lemma finale decoupling},  and  the function 
\begin{equation}\label{definizione a 10}
a_{13} := a_{13}(\vphi, x) :=  \frac12 
(a_{11})_x a -  a_{11} 
a_x - \frac38 \mathtt m_3 \sqrt{\kappa}  a_{xx} a + \frac{3 }{4} \mathtt m_3 \sqrt{\kappa} a_x^2  \,.
\end{equation}
\end{lemma}

\begin{proof}
By \eqref{definizione qk},  \eqref{definizione q0}, \eqref{decomposizione compound} we have 
\begin{align}
q_1 (t, \vphi, x, \xi) & = \ii  \int_0^t (  {\mathfrak a} 
\star q_0)(\tau, \vphi, x , \xi)\, d \tau = \ii \,t\, ( {\mathfrak a} \star p_0)(\vphi, x, \xi) \nonumber\\
& =  t\, \{   {\mathfrak a}, p_0 \}(\vphi, x, \xi) + \ii \,t\, {\mathtt r}_{\mathtt 2}( {\mathfrak a}, p_0)( \vphi, x, \xi) \in S^1 \label{espansione q1} 
 \end{align}
and note that $ \mathtt r_{\mathtt 2} ( {\mathfrak a}, p_0) \in S^0 $.  Similarly, using also \eqref{espansione q1}, the  symbol 
\be\label{prima espressione q2}
\begin{aligned}
q_2(1, \vphi,x , \xi) & = \ii \int_0^1 (   {\mathfrak a} \star q_1)(\tau, \vphi,  x, \xi)  d \tau \\
&  = \int_0^1  \{  {\mathfrak a}, q_1 \}(\tau, \vphi, x, \xi) d \tau 
+ \ii  \int_0^1 \mathtt r_{\mathtt 2} (	 {\mathfrak a}, q_1)(\tau, \vphi, x, \xi) d \tau   \\
&  = \frac{1}{2}\big( \{  {\mathfrak a}, \{ {\mathfrak a}, p_0 \} \} + \ii \{ {\mathfrak a}, \mathtt r_{\mathtt 2}( {\mathfrak a}, p_0)\} \big) \\
& \quad + \ii  \int_0^1 \mathtt r_{\mathtt 2}( {\mathfrak a}, q_1)(\tau, \vphi, x, \xi) d \tau \in S^{1/2}
\end{aligned}
\ee
where $ \{ {\mathfrak a}, \mathtt r_{\mathtt 2}({\mathfrak a}, p_0) \} $ and $  \mathtt r_{\mathtt 2} ( {\mathfrak a}, q_1) \in S^{-1/2} $. 
By \eqref{definizione q0}, \eqref{espansione q1} at $ t =1 $, and \eqref{prima espressione q2} we get 
\begin{equation}\label{espansione q leq 2}
 q_{\leq 2}  = q_0 + q_1 + q_2  =   p_0 + 
 \{ {\mathfrak a}, p_0 \} + \frac{1}{2} \{  {\mathfrak a}, \{ {\mathfrak a}, p_0\} \} + r_{{\mathfrak a} p_0 }^{(0)}
\ee
where 
\be\label{rq1q2}
r_{{\mathfrak a} p_0 }^{(0)} 
:= \ii  \mathtt r_{\mathtt 2}( {\mathfrak a}, p_0) +  
\frac{\ii}{2} \{  {\mathfrak a}, \mathtt r_{\mathtt 2}( {\mathfrak a}, p_0) \}  + 
 \ii \int_0^1 \mathtt r_{\mathtt 2} ( {\mathfrak a}, q_1)(\tau, \vphi, x, \xi)\, d \tau \in S^0 \, .
\ee 
By \eqref{p0 Egorov iniziale} and
 $  \partial_\xi T (\xi) = \frac32  \sqrt \kappa \, {\rm sign}(\xi) \chi(\xi) |\xi|^{\frac12} + O(|\xi|^{- \frac32}) $,   we get 
\begin{equation}\label{albertone 0}
\begin{aligned}
 \{ {\mathfrak a}, p_0\} & = \ii \{ a \chi(\xi)|\xi|^{\frac12},  \mathtt m_3 T(\xi) 
+ a_{11} \xi \} + {\tilde r}_{{\mathfrak a}p_0}  \\
&  = - \ii  \mathtt m_3 \partial_\xi T(\xi)  a_x  \chi(\xi) |\xi|^{\frac12} +
 \ii \Big( \frac12 (a_{11})_x a -  a_{11} 
a_x \Big) \chi(\xi) |\xi|^{\frac12}  \\
 & + \ii  (a_{11})_x a (\partial_\xi \chi (\xi)) |\xi|^{\frac12} \xi+  {\tilde r}_{{\mathfrak a}p_0}  \\
& =  -  \ii \frac{3}{2} \mathtt m_3 \sqrt{\kappa} \, a_x \xi 
+  \ii \Big( \frac12 (a_{11})_x a -  a_{11} a_x \Big) \chi(\xi) |\xi|^{\frac12} + r_{{\mathfrak a} p_0  }^{(1)}   
\end{aligned}
\end{equation}
where  
$ {\tilde r}_{{\mathfrak a}p_0} :=
\{  a \chi(\xi) |\xi|^{\frac12},  a_{12} {\rm sign}(\xi) \chi(\xi) |\xi|^{\frac12}   +  r_M^{(I)} \} \in S^0 $ and 
\begin{align} \label{r p0 A (1)}
r_{{\mathfrak a} p_0 }^{(1)} & := {\tilde r}_{{\mathfrak a}p_0} - 
\ii \mathtt m_3 \Big( \partial_\xi T (\xi) - \frac32  \sqrt \kappa \, {\rm sign}(\xi) \chi(\xi) |\xi|^{\frac12}\Big)a_x \chi |\xi|^{1/2} \\
& \quad + \ii \frac32  \mathtt m_3 \sqrt{\kappa} a_x (1 - \chi^2(\xi)) \xi +  \ii  (a_{11})_x \, a \, (\partial_\xi \chi(\xi)) |\xi|^{\frac12} \xi \in S^0 \, . \nonumber
 \end{align}
Furthermore, using  \eqref{albertone 0}, we compute
\begin{align}
\frac12 \{  {\mathfrak a}, \{ {\mathfrak a},  p_0 \} \}  
 & =  -  \ii \frac{ 3}{4} \mathtt m_3 \sqrt{\kappa} \Big( \frac12  a_{xx} a - a_x^2  \Big) \chi(\xi) |\xi|^{\frac12} + r_{{\mathfrak a} p_0 }^{(2)}
\label{albertone 1}
\end{align} 
where 
\begin{equation} \label{r p0 A (3)}
\begin{aligned}
 r_{ {\mathfrak a}  p_0 }^{(2)}  &  := \Big\{  a \chi(\xi)|\xi |^{1/2},
 \ii \Big( \frac12 (a_{11})_x a -  a_{11} a_x \Big) \chi(\xi) |\xi|^{1/2}+ r_{{\mathfrak a} p_0  }^{(1)} \Big\}  \\
 & \quad  - \ii \frac{ 3}{4} \sqrt{\kappa} \mathtt m_3 a_{xx} a (\partial_\xi \chi(\xi)) |\xi|^{\frac12} \xi   \in S^0 \, . 
 \end{aligned}
\end{equation}
 Finally \eqref{espansione q leq 2}, \eqref{p0 Egorov iniziale}, \eqref{albertone 0}, \eqref{albertone 1} imply  
\eqref{simboli<2}-\eqref{defrleq2}. 
\end{proof}

\noindent
{\bf Choice of the function $ a(\vphi, x) $.}
We now choose the function $ a(\vphi, x) $ so that the first order term in \eqref{simboli<2} vanishes, namely such that 
$ a_{11}(\vphi, x) - \frac32 \mathtt m_3 \sqrt{\kappa} a_x(\vphi, x) = 0 $. 
Since the function $a_{11}(\vphi, x) $ is odd in $ x $ (see \eqref{bf B1} and remark \ref{parities a1 - a6}) 
such equation may be solved. 
Its solution is 
\begin{equation}\label{definizione a}
a(\vphi, x) := \tilde a(\vphi, x ) + a_0 (\vphi )  \quad {\rm where} \quad 
\tilde a(\vphi, x ) := \frac{2}{3 \mathtt m_3 \sqrt{\kappa}} \partial_x^{- 1} a_{11}(\vphi, x)  
\end{equation}
and the function $ a_0 (\vphi ) $ will be determined later, see \eqref{choicea0}. 
In this way (by \eqref{simboli<2})
\begin{equation}\label{q leq 2 espansione finale}
q_{\leq 2} = \ii \mathtt m_3 T ( \xi) + 
\big( \ii  a_{13} + a_{12} \, {\rm sign}(\xi) \big) \chi(\xi) |\xi|^{\frac12} + r_{q_{\leq 2}}  
\end{equation} 
where $ r_{q_{\leq 2}}  \in S^0 $.  
The next lemma proves that we have found  an approximate solution of \eqref{equazione Egorov}. 

\begin{lemma}\label{lemma algebrico soluzione approssimata Q}
{\bf (Approximate solution of \eqref{equazione Egorov})}
The operator $Q(\vphi, t) = q(t, \vphi, x, D) $ where $ q = \sum_{n=0}^M q_n $  
with $ q_0 $ defined in \eqref{definizione q0} and $ q_n $, $ n = 1, \ldots, M $ 
in \eqref{definizione qk}, solves the  approximate Heisenberg equation
\begin{equation}\label{equazione approssimata}
\begin{cases}
\partial_t Q(\vphi, t) = \ii [A(\vphi), Q(\vphi , t)] + R_M (\vphi, t) \cr 
Q(0) = P_0
\end{cases}
\end{equation}
where $R_M (\vphi, t) := - \ii {\rm Op} ( {\mathfrak a} \star q_M  ) \in OPS^{1 - \frac{M}{2}}$. 
The quantitative estimate is given in \eqref{stima RN diagonale Egorov}. 
\end{lemma}

\begin{proof}
By \eqref{equazione per q0} and  \eqref{equazione per qk} the initial symbol  
$ q(0, \vphi, x, \xi) = q_0(0, \vphi, x, \xi) + \sum_{n = 1}^M q_n (0, \vphi, x, \xi) = p_0(\vphi, x , \xi) $. 
Hence $Q(0) = P_0$. Moreover \eqref{equazione per q0} and  \eqref{equazione per qk} imply 
$$
\begin{aligned}
\partial_t q = \sum_{n = 0}^M \partial_t q_n 
= \ii \sum_{n = 1}^{M}  {\mathfrak a} \star q_{n-1}   
& = \ii \sum_{n = 0}^{M - 1}  {\mathfrak a} \star q_n   \\
& = \ii \sum_{n = 0}^M  {\mathfrak a} \star q_n - \ii  {\mathfrak a} \star q_M 
 = \ii  {\mathfrak a} \star q - \ii  {\mathfrak a} \star q_M  
 \end{aligned}
$$
because $  {\mathfrak a} \star q  $ is linear in $ q $. Since 
$  [A(\vphi), Q ] = {\rm Op}( {\mathfrak a} \star q ) $ 
we get \eqref{equazione approssimata} with 
$ R_M (\vphi, t) := - \ii {\rm Op}( {\mathfrak a} \star q_M) $. The operator $ R_M \in OPS^{1 - \frac{M}{2} } 
$ since $q_M \in S^{\frac12 (3- M)}$, see after \eqref{equazione per qk}-\eqref{definizione qk}.
\end{proof}

The next lemma expresses  
the difference between $P(\vphi, t) $ and 
 the approximate solution $ Q(\vphi, t) $ 
 of \eqref{equazione Egorov} in terms of the remainder $ R_M $ in  \eqref{equazione approssimata} 
and  the flow $ \Phi (\vphi, t) $ of \eqref{pseudo-PDE}. 

\begin{lemma}\label{soluzione vera Egorov algebrico} We have 
\begin{equation}\label{forma finale P - Q}
W(\vphi, t) := Q(\vphi, t) - P(\vphi, t) = 
\int_0^t \Phi(\vphi, t - \tau) R_M (\vphi, \tau) \Phi(\vphi, \tau - t)\, d \tau \, . 
\end{equation}
\end{lemma}

\begin{proof}
Recalling  \eqref{defPt} we write 
$$ W(\vphi, t)  = 
 \big( Q(\vphi, t) \Phi(\vphi, t) - \Phi(\vphi, t) P_0 \big) \Phi(\vphi, t)^{- 1} \, . 
 $$
By \eqref{flow-propagator} and \eqref{equazione approssimata} we deduce that 
$ V(\vphi, t) := Q(\vphi, t) \Phi(\vphi, t) - \Phi(\vphi, t) P_0$ solves  the non-homogeneous equation 
$$
\partial_t V (\vphi,t) = 
\ii A(\vphi) V(\vphi, t)   + R_M (\vphi, t) \Phi(\vphi, t) \, , \quad 
V(\vphi, t) (\vphi, 0) = 0 \, . 
$$ 
By Duhamel principle (variation of constants method)
and  \eqref{group-flow} we get 
$$ 
V (\vphi, t) :=  \int_0^t  \Phi(\vphi, t - \tau) R_M (\vphi, \tau) \Phi(\vphi, \tau)\, d \tau 
$$ 
and thus  \eqref{forma finale P - Q} using again \eqref{group-flow}. 
\end{proof}

\noindent
{\bf Analysis of $ {\bf \Phi}(\vphi) \Dom \{ {\bf \Phi}(\vphi)^{- 1} \} $ in \eqref{coniugio flusso di una PDE iperbolica}.}
\\[1mm]
Set  for brevity (recall \eqref{flusso-diagonale})
$$
{\bf \Psi}(\vphi, t) := 
{\bf \Phi}(\vphi, t) \Dom \{ {\bf \Phi}(\vphi, t)^{- 1} \}  
 = \begin{pmatrix}
\Psi(\vphi, t) & 0 \\
0 & \overline \Psi(\vphi, t)
\end{pmatrix} 
$$
where
$$
 \Psi(\vphi, t) := \Phi(\vphi, t ) \Dom \{ \Phi(\vphi, t )^{- 1} \}\,.
$$
The term $ \Psi(\vphi, t) $ can be computed in terms of the flow $ \Phi $ of  \eqref{pseudo-PDE}  and 
$  A(\vphi) = a(\vphi, x ) | D |^{\frac12} $. 

\begin{lemma}\label{lemma:pezzo dal tempo} The operator 
$$ 
\Psi(\vphi, t) = - \ii \int_0^t S_\omega(\vphi, \tau)\, d \tau  \quad {\rm where} \quad 
S_\omega(\vphi, t) :=  \Phi(\vphi, t) (\Dom A(\vphi)) \Phi(\vphi,t)^{- 1} \, . 
$$
\end{lemma}

\begin{proof}
By \eqref{group-flow} the flow $\Phi^{- 1}(t) = \Phi(- t)$ and  $\partial_t \Phi(t)^{- 1} = - \ii A \Phi(t)^{- 1}$. Thus 
$ \Psi(\vphi, t) $ solves
\begin{align}
\partial_t \Psi(\vphi, t) & = (\partial_t \Phi)\Dom \Phi^{- 1} + \Phi \Dom (\partial_t \Phi^{- 1}) \nonumber \\ 
& =  - \Phi (\partial_t \Phi^{-1}) \Phi \Dom \Phi^{- 1} - \ii \Phi \Dom ( A \Phi^{- 1}) \nonumber\\
& =  \ii \Phi A \Dom \Phi^{- 1} - \ii \Phi A \Dom \Phi^{- 1} - \ii \Phi (\Dom A) \Phi^{- 1} 
=  - \ii \Phi (\Dom A) \Phi^{- 1} \, . \nonumber
\end{align}
Moreover $\Psi(\vphi, 0) = 0 $  (as $ \Phi (\vphi, 0) = {\rm Id} $, $ \forall \vphi \in \T^\nu $, see \eqref{flow-propagator}).  
The lemma follows by integration.  
\end{proof}

The operator $ S_\omega(\vphi, t) $ has the same conjugation structure of $ P (\vphi, t )$ in
\eqref{defPt} and therefore it solves the Heisenberg equation
\begin{equation}\label{equazione egorov S omega}
\begin{cases}
\partial_t S_\omega (\vphi, t) = \ii [A(\vphi), S_\omega(\vphi, t)]   \\
 S_\omega(\vphi, 0) = 
 (\omega \cdot \partial_\vphi a) |D|^{\frac12}\,.
\end{cases}
\end{equation}
Following the same procedure used for $ P (\vphi, t ) $, 
we look for an approximate solution  of  \eqref{equazione egorov S omega} of the form 
(expansion in decreasing symbols) 
\begin{equation}\label{s1 sM}
S_{\omega, M}(\vphi, t) := s(t,\vphi, x, D) \, ,  \qquad  
s = {\mathop \sum}_{n=0}^M s_n \, , \qquad 
s_n \in S^{\frac12(1-n)} \, . 
\end{equation}
We define the principal  symbol $s_0$ to be the solution of 
\begin{equation}\label{s0 soluzione egorov}
\begin{aligned}
& \begin{cases}
\partial_t s_0(t, \vphi, x, \xi) = 0 \\
s_0(0, \vphi, x, \xi) = (\omega \cdot \partial_\vphi a) \chi(\xi) |\xi |^{\frac12} \, , 
\end{cases} \\
& \ i.e.  \
s_0(t, \vphi, x, \xi) =  (\omega \cdot \partial_\vphi a) \chi(\xi)|\xi|^{\frac12} \in S^{1/2}\,.
\end{aligned}
\end{equation}
Then we define inductively  the symbols $s_n $, $ n \geq 1 $, as the solutions of 
\begin{equation}\label{sk soluzione}
\begin{aligned}
& \begin{cases}
\partial_t s_n = \ii  {\mathfrak a} \star s_{n - 1}  \\
s_n (0, \vphi, x, \xi) = 0\,, 
\end{cases} \\
& i.e. \quad s_n (t, \vphi, x, \xi) = \ii \int_0^t ( {\mathfrak a} \star s_{n - 1})(\tau, \vphi, x, \xi) \,d \tau\,.
\end{aligned}
\end{equation}
It turns out that $ s_n \in S^{\frac12(1-n)} $,  in particular each $ s_n \in S^0 $, $ \forall  n \geq 1 $. 

\begin{lemma}\label{soluzione vera Egorov S omega algebrico}
{\bf (Approximate solution of \eqref{equazione egorov S omega})}
The pseudo-differential operator $ S_{\omega, M}(\vphi, t) =  s(\vphi, t, x, D)  $
in \eqref{s1 sM} with $ s_0 \in S^{\frac12} $ defined  in \eqref{s0 soluzione egorov} and 
$ s_n \in S^{\frac12(1-n)} $, $ n 
= 1, \ldots, M $ in  \eqref{sk soluzione},  solves the approximate Heisenberg equation
\begin{equation}\label{equazione approssimata S omega M}
\begin{cases}
\partial_t S_{\omega, M}(\vphi, t) = \ii [A(\vphi), S_{\omega, M}(\vphi , t)] + R_{\omega, M}(\vphi, t) \cr 
S_{\omega, M}(\vphi, 0) =  
(\omega \cdot \partial_\vphi a) |D|^{\frac12}
\end{cases}
\end{equation}
where $R_{\omega, M}(\vphi, t) := - \ii {\rm Op}( {\mathfrak a} \star s_M ) \in OPS^{ - \frac{M}{2}}$.  Moreover
\be\label{forma finale S omega - S M}
W_\omega (\vphi, t) := S_{\omega, M} (\vphi, t)  - S_\omega (\vphi, t) =
\int_0^t \Phi(\vphi, t - \tau) R_{\omega, M}(\vphi, \tau) \Phi(\vphi, \tau - t)\, d \tau 
\ee
where  $ \Phi (\vphi, t) $ denotes the flow of \eqref{pseudo-PDE}.
\end{lemma} 

\begin{proof}
The equation \eqref{equazione approssimata S omega M} follows as in Lemma \ref{lemma algebrico soluzione approssimata Q}.
Then \eqref{forma finale S omega - S M}  follows as in Lemma \ref{soluzione vera Egorov algebrico}. 
\end{proof}

\noindent
{\bf Sub-principal symbol of $ {\mathcal L}_M^{(1)} $.} 
By Lemma  \ref{ordini maggiori di 0 Egorov} 
and the choice of $a(\vphi, x) $ in \eqref{definizione a}, 
the principal and subprincipal symbols of  $ {\bf \Phi}(\vphi) {\bf P}_0(\vphi, x, D) {\bf \Phi}(\vphi)^{- 1} $ are given by 
 \eqref{q leq 2 espansione finale}. 
Also  $ {\bf \Phi}(\vphi) \Dom \{{\bf \Phi}(\vphi)^{- 1}\} $ contributes to the  subprincipal symbol   of $ {\mathcal L}_M^{(1)}  $, i.e 
to  $ OPS^{1/2}$.
By Lemmata \ref{lemma:pezzo dal tempo}, \ref{soluzione vera Egorov S omega algebrico} and the expression of 
$ s_0 =  (\omega \cdot \partial_\vphi a) \chi(\xi)|\xi|^{\frac12} $ in \eqref{s0 soluzione egorov} we find that 
the conjugated operator ${\mathcal L}_M^{(1)}$ in \eqref{coniugio flusso di una PDE iperbolica} has the expansion 
\begin{equation}\label{cal LN (1)0}
{\mathcal L}_M^{(1)} = \Dom {\mathbb I}_2 + \ii \mathtt m_3 {\bf T}(D) + 
\ii \big({\bf C}_1(\vphi, x) + {\bf C}_0(\vphi, x) \mH \big) |D|^{\frac12} + \ldots 
\end{equation}
where 
\begin{equation}\label{bf A1}
\begin{aligned}
& {\bf C}_1(\vphi, x) := \begin{pmatrix}
a_{14} & 0 \\
0 & - a_{14}
\end{pmatrix}, \ a_{14} := a_{13} -  \omega \cdot \partial_\vphi a \, , \\
& {\bf C}_0(\vphi, x) := \begin{pmatrix}
a_{12}  & 0 \\
0 & - a_{12} 
\end{pmatrix}\,,
\end{aligned}
\end{equation}
and the functions $a_{13} $, $ a_{12}$  are defined respectively in \eqref{definizione a 10}, \eqref{bf B0 (I)}. 

In the next sections we reduce the operator $ {\mathcal L}_M^{(1)} $ neglecting  the term 
 \begin{equation}\label{bf RN (1) bot}
{\bf R}_{M}^{(1), \bot} :=  \ii \Pi_{K_n}^\bot {\bf C}_1  |D|^{\frac12} := \ii \begin{pmatrix}
 \Pi_{K_n}^\bot a_{14}(\vphi, x) & 0 \\
 0 & - \Pi_{K_n}^\bot a_{14}(\vphi, x)
 \end{pmatrix}  |D|^{\frac12}
 \end{equation}
which  is supported on the high Fourier frequencies and which will contribute to the remainders in 
\eqref{stima R omega bot corsivo bassa}-\eqref{stima R omega bot corsivo alta} 
(as we did with the similar terms at the end of section \ref{sec: time-reduction highest order}).
For simplicity of notation we still denote it by $ {\mathcal L}_M^{(1)} $. 
\\[1mm]
{\bf Choice of the function $ a_0(\vphi ) $.}
In view of  the reduction of $ \ii \Pi_{K_n} {\bf C}_1  |D|^{\frac12} $
in section \ref{sec:lineare},  
we  choose the function $ a_0 (\vphi )$ in \eqref{definizione a}
in such a way that, for all $ \vphi \in \T^\nu $, the integral  
\begin{equation}\label{media a 14}
\frac{1}{2 \pi} \int_\T \Pi_{K_n} a_{14}(\vphi, x)\, d x = \mathtt m_{1, K_n}  \, , \quad \forall \vphi \in \T^\nu \, , 
\end{equation}
is a constant. Since $ a = \tilde a + a_0 $ (see \eqref{definizione a}) we write the function  $ a_{14} $ in \eqref{bf A1} as 
\begin{equation}\label{a tilde 14}
 a_{14} (\vphi, x) = \tilde a_{14} (\vphi, x) - \om \cdot \pa_\vphi a_0 (\vphi ) \qquad {\rm where} \qquad 
\tilde a_{14}   = a_{13}   -  \omega \cdot \partial_\vphi \tilde a    \, . 
\end{equation}
The function $ a_{13} (\vphi, x) $ in   \eqref{definizione a 10} depends on $ a $,  
and thus also on $a_0 (\vphi) $, but  the integral 
$  \int_{\T} a_{13} (\vphi, x) dx $, and thus $  \int_{\T} \tilde a_{14} (\vphi, x) dx $, does not depend on $a_0 (\vphi) $.
For solving \eqref{media a 14}
we look for $ a_0 (\vphi) = \Pi_{K_n} a_0 (\vphi) $ such that  
$ \frac{1}{2 \pi} \int_{\T} \Pi_{K_n} \tilde a_{14}(\vphi, x) \, dx -   (\om \cdot \pa_\vphi  a_0) (\vphi)  = \mathtt m_{1, K_n} $. 
For all $ \omega \in \DC_{K_n}^\g $ (see \eqref{omega diofanteo troncato}) such equation is solved by  
\begin{equation}\label{lambda1 w} 
\begin{aligned}
 {\mathtt m}_{1, K_n} & := (2 \pi)^{-(\nu + 1)} \int_{\T^{\nu +1}} \Pi_{K_n} \tilde a_{14}(\vphi, x) 
\,d \vphi \, dx  \\
& = (2 \pi)^{-(\nu + 1)} \int_{\T^{\nu+1}} \tilde a_{14} (\vphi, x) \,d \vphi\, d x \, , 
\end{aligned} 
\end{equation}
\begin{equation} \label{choicea0}
\begin{aligned}
& a_0 (\vphi) := - (\omega \cdot \partial_\vphi)^{- 1} 
\Big( {\mathtt m}_{1, K_n} -  \frac{1}{2 \pi}\int_{\T} \Pi_{K_n} \tilde a_{14} (\vphi, x) \, dx   \Big) \, . 
\end{aligned}
\end{equation}
Note that $ a_0 (\vphi )$ is odd in $ \vphi $. 
Since also  $ \tilde a(\vphi, x )  $ defined in \eqref{definizione a} is odd in $ \vphi $, and even in $ x $, 
the flow $ {\bf \Phi} (\vphi, t) $ of \eqref{flusso Phi} is even and reversibility preserving.

\begin{lemma} {\bf (Coefficient $\mathtt m_{1, K_n} $)} The coefficient   
\begin{equation} \label{lm1 formula}
\mathtt m_{1, K_n} \! = \! - 
\frac{ (2\p)^{-\nu - \frac52} }{2  \sqrt{\kappa}} 
\int_{\T^{\nu + 1}} \!
(1 + \b_x) [ \om \cdot \pa_\vphi \beta + V (1+\b_x) ]^2 
\Pi_{K_n}  \Big( \int_\T \sqrt{1 + \eta_y^2} \, dy \Big)^{3/2} \! \! d\vphi  \, d x
\end{equation}
where the function $V$ is defined in \eqref{def B V} and $ \b$  in \eqref{beta lambda3}.  
The coefficient $ {\mathtt m}_{1, K_n} $ satisfies  
\be\label{stima R7lambda1} 
| \mathtt m_{1, K_n}  |^{k_0, \gamma} \leq C \e  \, , \qquad | \partial_i \mathtt m_{1, K_n}[\hat \imath] | \leq C \e \| \hat \imath \|_\sigma\,.
\ee
\end{lemma}

\begin{proof}
By \eqref{lambda1 w}, \eqref{a tilde 14}, \eqref{definizione a 10}, \eqref{definizione a} the coefficient
\begin{align}
\mathtt m_{1, K_n} & = 
\frac{1}{(2 \pi)^{\nu + 1}} \int_{\T^{\nu+1}}  \tilde a_{14} (\vphi, x) \,d \vphi d x = 
\frac{1}{(2 \pi)^{\nu + 1}} \int_{\T^{\nu+1}}  a_{13} (\vphi, x) \,d \vphi d x \nonumber \\ 
& = \frac{1}{(2 \pi)^{\nu + 1}} \int_{\T^{\nu+1}}  \frac12  (a_{11})_x \tilde a -  a_{11} 
\tilde a_x - \frac{3 }{8}  \mathtt m_3 \sqrt{\kappa}  \tilde a_{xx}  \tilde a + 
\frac{3 }{4} \mathtt m_3 \sqrt{\kappa}  \tilde a_x^2 \,d \vphi d x \nonumber \\
& = - \frac{(2 \pi)^{-\nu - 1}}{2 \mathtt m_3 \sqrt{\kappa} }  \int_{\T^{\nu+1}} a_{11}^2 (\vphi, x) \, d \vphi dx \, . \label{svil-lambda1}
\end{align}
By \eqref{bf B1}, \eqref{transformed Pcal}, $ d \vartheta = (1+ \om \cdot \pa_\vphi p ) d \vphi $ (by \eqref{QP-repa}),
\eqref{solution p m3}, \eqref{a1 a2 a3}, \eqref{diffeo-torus}  we have 
\be\label{ulter-svilu}
\begin{aligned}
 \int_{\T^{\nu+1}} a_{11}^2 (\vphi, x) \, d \vphi dx & =  \int_{\T^{\nu+1}} 
 \frac{a_1^2(\vphi, x)}{1+ \om \cdot \pa_\vphi p }  \, d \vphi dx \\ 
 & =
 \mathtt m_3 \int_{\T^{\nu+1}} \frac{ (\Dom \b + V(1 + \b_x))^2 }{\Pi_{K_n} m_3 (\vphi)} (1+ \b_x ) \, d \vphi dx \, . 
 \end{aligned}
\ee
By  \eqref{svil-lambda1}, \eqref{ulter-svilu},  \eqref{formula m3} 
 we deduce 
\eqref{lm1 formula}. 
\end{proof}

Lemmata \ref{ordini maggiori di 0 Egorov}, \ref{soluzione vera Egorov algebrico}, 
\ref{lemma:pezzo dal tempo}, \ref{soluzione vera Egorov S omega algebrico},   imply that 
$$
{\mathcal L}_M^{(1)} = \Dom {\mathbb I}_2 + \ii \mathtt m_3 {\bf T}(D) + 
\ii \big({\bf C}_1(\vphi, x) + {\bf C}_0(\vphi, x) \mH \big) |D|^{\frac12} + {\bf R}_M^{(1)} + {\bf Q}_M^{(1)}
$$
with remainders 
\begin{align} \nonumber
& \qquad \qquad \qquad \qquad {\bf R}_M^{(1)} := \begin{pmatrix}
{\mathcal R}_M^{(1)} & 0 \\
0 & \overline {\mathcal R}_M^{(1)}
\end{pmatrix}\,, \quad 
{\bf Q}_M^{(1)} := \begin{pmatrix}
  0 & {\mathcal Q}_M^{(1)} \\
\overline {\mathcal Q}_M^{(1)} & 0  
\end{pmatrix} \\
& {\mathcal R}_M^{(1)}  := {\rm Op}(r_M^{(1)}) - W(\vphi, 1) + 
\int_0^1 W_\omega (\vphi, \tau)d \tau  \, , 
\qquad {\mathcal Q}_M^{(1)} := \Phi {\mathcal R}_M^{(II)} \overline \Phi^{- 1} \, , \label{cal QN (1)}   \\
& r_M^{(1)}(\vphi, x, \xi) := r_{q_{\leq 2}}(\vphi,x, \xi) + 
{\mathop \sum}_{n = 3}^M  
q_n (1, \vphi, x, \xi)  \nonumber\\ 
& \qquad \qquad \qquad + \ii {\mathop \sum}_{n = 1}^M \int_0^1 
s_n (\tau, \vphi, x, \xi)d \tau \in S^0 \nonumber 
\end{align}
where $ r_{q_{\leq 2}} $ is defined in \eqref{defrleq2},   
$ q_n $  in \eqref{definizione qk}, $ s_n $ in  \eqref{sk soluzione},  
the operator $W$ is defined in \eqref{forma finale P - Q}, $W_\omega$ 
in \eqref{forma finale S omega - S M} and $ {\mathcal R}_M^{(II)} $  in Proposition \ref{Lemma finale decoupling}. 

In the final part of this section we prove that  $ {\bf R}_M^{(1)} $ and $ {\bf Q}_M^{(1)} $ 
are tame operators and  \eqref{tame resto diagonale Egorov} holds.

\begin{lemma}
For all $s_0 \leq s \leq S$, we have 
\begin{equation}\label{stima coefficienti a egorov}
\begin{aligned}
& \| a_{12}\|_s^{k_0, \gamma}, \| a_{13}\|_s^{k_0, \gamma}, 
\| a_{14}\|_s^{k_0, \gamma}, \| \tilde a\|_s^{k_0, \gamma} \leq_S \e (1 + \| \fracchi_0 \|_{s + \sigma}^{k_0, \gamma})\,,  \\
&  \| a_0 \|_s^{k_0, \gamma} \leq_S \e \gamma^{- 1} (1 + \| \fracchi_0 \|_{s + \sigma}^{k_0, \gamma})\,, 
\end{aligned}
\end{equation}
\begin{equation}\label{stime derivate coefficienti a Egorov}
\begin{aligned}
& \| \partial_i a_{12}[\hat \imath]\|_{s_1}, \| \partial_i a_{13} [\hat \imath] \|_{s_1}, 
\| \partial_i a_{14}[\hat \imath ]\|_{s_1}, \|  \partial_i \tilde a[\hat \imath] \|_{s_1} \leq_{S} \e \| \hat \imath \|_{s_1 + \sigma}, \\
&  \| \partial_i a_0[\hat \imath] \|_{s_1} \leq_{S} \e \gamma^{- 1} \| \hat \imath \|_{s_1 + \sigma}\,.
\end{aligned}
\end{equation}
\end{lemma}

\begin{lemma}\label{ordini maggiori di 0 Egorov1}
The remainder $ r_{q_{\leq 2}} \in S^0 $ in  \eqref{q leq 2 espansione finale} (see \eqref{defrleq2}) satisfies,
for some  $ \sigma := \sigma(\tau, \nu) > 0 $, 
\begin{equation}\label{stima resto q leq 2}
\norma r_{q_{\leq 2}}(x, D)  \norma_{0, s, \alpha}^{k_0, \gamma} \leq_{S, \alpha} \e 
\big(1 + \| \fracchi_0  \|_{s + \sigma + \aleph_M(\alpha + 4) + 2 \alpha
}^{k_0, \gamma} \big)\, , \quad \forall s_0 \leq s \leq S\,.
\end{equation}
Moreover, if the constant $\mu$ in \eqref{ansatz I delta} satisfies 
\begin{equation}\label{ansatz 1 indici egorov}
s_1 + \sigma + \aleph_M(\alpha + 4) + 2 \alpha \leq s_0 + \mu\,,
\end{equation}
then 
\begin{equation}\label{stima derivate i resto q leq 2}
\norma \partial_i r_{q_{\leq 2}}(x, D)[\hat \imath] \norma_{0, s_1, \alpha} \leq_{S, \alpha} \e \| \hat \imath \|_{s_1 + \sigma + \aleph_M(\alpha + 4) + 2 \alpha}\,.
\end{equation}
\end{lemma}

\begin{proof}
We rely on the Lemmata \ref{proprieta sigma(q,A)} and \ref{proprieta r2(q,A)}. 
 We prove that each term  of 
$ r_{q_{\leq 2}} = r_M^{(I)} + r_{{\mathfrak a} p_0 }^{(0)} +  r_{{\mathfrak a} p_0}^{(1)}  + r_{{\mathfrak a} p_0 }^{(2)} $ defined in \eqref{rq1q2}, 
\eqref{r p0 A (1)}, \eqref{r p0 A (3)} satisfies \eqref{stima resto q leq 2}. The term $ {\rm Op} ( r_M^{(I)}) $
satisfies \eqref{stima resto q leq 2}, \eqref{stima derivate i resto q leq 2} by Proposition  \ref{Lemma finale decoupling}. 
Then we consider $ r_{{\mathfrak a} p_0 }^{(0)} $ in \eqref{rq1q2}. 
Lemma \ref{proprieta r2(q,A)} (with $m = 3/2$),  
the definition of $ p_0 $ in \eqref{p0 Egorov iniziale},
the estimates of Proposition \ref{Lemma finale decoupling}, and  \eqref{stima coefficienti a egorov},
imply
\begin{equation}\label{stima r q1}
\norma \mathtt r_{\mathtt 2} ( {\mathfrak a}, p_0)(x, D)  \norma_{0, s, \alpha}^{k_0, \gamma} \leq_{S, \alpha}
\e \big( 1 + \| \fracchi_0 \|_{s + \sigma + \aleph_M( \alpha + 2) + \alpha}^{k_0, \gamma} \big) \, . 
\end{equation}
 In the same way, using  
$ \partial_i \mathtt r_{\mathtt 2} ( {\mathfrak a}, p_0)[\hat \imath] = $ $ 
\mathtt r_2 (\partial_i {\mathfrak a} [\hat \imath], p_0) + \mathtt r_2({\mathfrak a}, \partial_i p_0 [\hat \imath]) $
and \eqref{ansatz I delta}, \eqref{ansatz 1 indici egorov}, we deduce that 
\begin{equation}\label{stima derivata r q1}
\norma \partial_i \mathtt r_2({\mathfrak a}, p_0)(x, D)[\hat \imath] \norma_{0, s_1, \alpha} \leq_{S, \alpha} \e \| \hat \imath \|_{s_1 + \sigma + \aleph_M(\alpha + 2) + \alpha } \, . 
\end{equation} 
Lemma \ref{stima Poisson}, \eqref{stima r q1} and \eqref{stima coefficienti a egorov} imply 
\begin{align}
\norma \{  {\mathfrak a}, \mathtt r_{\mathtt 2} ( {\mathfrak a}, p_0) \}(x, D) \norma_{0, s, \alpha}^{k_0, \gamma} & 
\leq_{s, \alpha} \norma \mathtt r_{\mathtt 2} ( {\mathfrak a}, p_0)(x, D) \norma_{0, s + 1, \alpha + 1}^{k_0, \gamma} 
\| a \|_{s_0 + 1}^{k_0, \gamma}  \nonumber\\
& \qquad + \norma \mathtt r_{\mathtt 2}( {\mathfrak a}, p_0)(x, D) \norma_{0, s_0 + 1, \alpha + 1}^{k_0, \gamma} 
\| a \|_{s + 1}^{k_0, \gamma} \nonumber\\
& \leq_{S, \alpha}  
\e \big( 1 + \| \fracchi_0 \|_{s + \sigma + \aleph_M( \alpha + 2) + \alpha}^{k_0, \gamma} \big)  \label{r2(p0,A) poisson A}
\end{align}
for some $\sigma := \sigma(\tau, \nu) > 0$. Moreover 
$ \partial_i  \{  {\mathfrak a}, \mathtt r_{\mathtt 2} ( {\mathfrak a}, p_0) \}[\hat \imath]  = $
$\{ \partial_i {\mathfrak a}[\hat \imath], \mathtt r_{\mathtt 2} ( {\mathfrak a}, p_0) \} +$ $
 \{  {\mathfrak a}, \partial_i \mathtt r_{\mathtt 2} ( {\mathfrak a}, p_0)[\hat \imath] \} $.
Hence \eqref{stima Poisson}, \eqref{stima coefficienti a egorov}, \eqref{stime derivate coefficienti a Egorov}, \eqref{stima r q1}, \eqref{stima derivata r q1}, \eqref{ansatz I delta}, \eqref{ansatz 1 indici egorov} imply that 
\begin{equation}\label{derivata r2(p0,A) poisson A}
\norma \partial_i \{  {\mathfrak a}, \mathtt r_{\mathtt 2} ( {\mathfrak a}, p_0) \}(x, D) [\hat \imath] \norma_{0, s_1, \alpha} \leq_{S, \alpha} \e \| \hat \imath \|_{s_1 + \sigma + \aleph_M(\alpha + 2) + \alpha} \,.
\end{equation}
Moreover by \eqref{espansione q1}, \eqref{stima r q1}, \eqref{stima derivata r q1}, 
\eqref{stima Poisson} and Proposition \ref{Lemma finale decoupling} (and \eqref{ansatz I delta}, \eqref{ansatz 1 indici egorov}) we get
\begin{align}
& \norma q_1(x, D) \norma_{1, s, \alpha}^{k_0, \gamma} \leq_{S, \alpha} 
\e \big( 1 + \| \fracchi_0 \|_{s + \sigma + \aleph_M ( \alpha + 2) + \alpha}^{k_0, \gamma} \big)\,, \label{stima simbolo q1 Egorov} \\
& \norma \partial_i q_1(x, D) [\hat \imath] \norma_{1, s_1, \alpha} \leq_{S, \alpha} 
\e \| \hat \imath \|_{s_1 + \sigma + \aleph_M(\alpha + 2) + \alpha}\,, \label{stima derivata i simbolo q1 Egorov}
\end{align}  
and using Lemma \ref{proprieta r2(q,A)} (with $m = 1$), by the same 
arguments used to deduce \eqref{stima r q1}, \eqref{stima derivata r q1}, we get 
\begin{align}
& \norma \mathtt r_{\mathtt 2}( {\mathfrak a}, q_1)(x, D) \norma_{0, s, \alpha}^{k_0, \gamma} \leq_{S, \alpha} 
\e \big( 1 +  \| \fracchi_0 \|_{s + \sigma + \aleph_M (\alpha + 4) + 2 \alpha }^{k_0, \gamma} \big)\,, \label{r2(q1,A)} \\
& \norma \partial_i \mathtt r_{\mathtt 2}( {\mathfrak a}, q_1)(x, D)[\hat \imath] \norma_{0, s_1 , \alpha} \leq_{S , \alpha} 
\e  \| \hat \imath \|_{s + \sigma + \aleph_M (\alpha + 4) + 2 \alpha }  \label{derivata i r2(q1,A)} 
\end{align}
for some $\sigma := \sigma(\tau, \nu) > 0$. 
The estimates \eqref{stima r q1}, \eqref{stima derivata r q1}, \eqref{r2(p0,A) poisson A}, \eqref{derivata r2(p0,A) poisson A}, \eqref{r2(q1,A)}, \eqref{derivata i r2(q1,A)} imply 
\begin{align*}
& \norma r_{{\mathfrak a} p_0 }^{(0)}(x, D) \norma_{0, s, \alpha}^{k_0, \gamma} 
\leq_{S, \alpha} \e \big( 1 +  \| \fracchi_0 \|_{s + \sigma + \aleph_M(\alpha + 4) + 2 \alpha }^{k_0, \gamma} \big) \\
& \norma \partial_i r_{{\mathfrak a} p_0 }^{(0)}(x, D)[\hat \imath] \norma_{0, s_1 , \alpha} 
\leq_{S, \alpha} \e \| \hat \imath \|_{s_1 + \sigma + \aleph_M(\alpha + 4) + 2 \alpha } 
\end{align*}
for some $\sigma := \sigma(\tau, \nu) > 0$. 
The symbol $  {\tilde r}_{{\mathfrak a}p_0} $ defined in \eqref{r p0 A (1)} satisfies 
\begin{align}\label{stima r p0 A (1)}
& \norma  {\tilde r}_{{\mathfrak a}p_0} (x, D) \norma_{0, s, \alpha}^{k_0, \gamma} \leq_{S, \alpha} 
\e \big( 1 +  \| \fracchi_0 \|_{s + \sigma + \aleph_M(\alpha + 1)}^{k_0, \gamma} \big)\,, \\
& \label{stima derivata r p0 A (1)}
\norma  \partial_i {\tilde r}_{{\mathfrak a}p_0} (x, D)[\hat \imath] \norma_{0, s_1, \alpha} \leq_{S, \alpha} 
\e   \| \hat \imath \|_{s_1 + \sigma + \aleph_M(\alpha + 1)} \,,
\end{align}
by \eqref{stima resti prima Egorov}, \eqref{stima derivate resti prima Egorov}, Lemma \ref{lemma:BAPQ} and \eqref{defA0}.
Also the symbols $ r^{(1)}_{{\mathfrak a} p_0  }$ in \eqref{r p0 A (1)} and $ r_{{\mathfrak a} p_0 }^{(2)} $ in \eqref{r p0 A (3)} 
satisfy \eqref{stima r p0 A (1)}, \eqref{stima derivata r p0 A (1)}.
\end{proof}


\begin{lemma} \label{soluzione approssimata Egorov}
For all $ n \in \{ 1, \ldots, M \} $ the symbols  $ q_n \in S^{\frac12 (3 - n)} $ defined 
in \eqref{definizione qk} satisfy  
\begin{equation}\label{stima qk Egorov}
\norma {\rm Op}(q_n )   \norma_{\frac12 (3 - n) ,s, \alpha}^{k_0, \gamma} \leq_{n, S, \a} \e 
\big(1 + \| \fracchi_0 \|_{s + \sigma + \beth_n (M, \alpha)}^{k_0, \gamma} \big)\,, \quad \forall s_0 \leq s \leq S\,,
\end{equation}
where the constants $ \beth_n (M, \alpha)$, $ n \in \{3, \ldots , M \}$ are defined inductively by
\begin{equation}\label{c k + 1 c k}
\beth_1(M, \alpha) := \aleph_M(\alpha + 2) + \alpha \,, \quad \beth_{n + 1}(M, \alpha) := \alpha + \frac{n}{2} + \frac32 + \beth_n(M, \alpha + 1)\,.
\end{equation}
The operator $R_M (\vphi, t) := - \ii {\rm Op} ( {\mathfrak a} \star q_M ) \in OPS^{1 - \frac{M}{2}}$ 
 satisfies 
 \begin{equation}\label{stima RN diagonale Egorov}
\norma R_M (\vphi, t)  \norma_{ 1 - \frac{M}{2}, s, \alpha}^{k_0, \gamma} 
\leq_{M, S, \a} \e \big( 1 + \| \fracchi_0 \|_{s + \sigma + \beth_{M + 1}(M, \alpha) }^{k_0, \gamma} \big)\,, \quad \forall s_0 \leq s \leq S\,.
\end{equation}
Moreover if the constant $\mu$ in \eqref{ansatz I delta} satisfies 
\begin{equation}\label{ansatz indici Egorov 2}
s_1 + \sigma + \beth_{M + 1}(M, \alpha) \leq s_0 + \mu\,,
\end{equation}
then for all $n \in \{3, \ldots, M \}$
\begin{align} \label{stima derivata i qk Egorov}
& \norma \partial_i {\rm Op}(q_n )[\hat \imath]   \norma_{\frac12 (3 - n) ,s_1 , \alpha} \leq_{n, S , \a} \e \| \hat \imath \|_{s_1 + \sigma + \beth_{n}(M, \alpha)}\,, \\
& \label{stima derivata i RN diagonale Egorov}
\norma \partial_i R_M (\vphi, t)[\hat \imath] \norma_{1 - \frac{M}{2}, s_1, \alpha} \leq_{M, S, \alpha} \e \| \hat \imath \|_{s_1 + \sigma + \beth_{M + 1}(M, \alpha)}\,.
\end{align}
\end{lemma}

\begin{proof}
For $n = 1$ 
the estimates \eqref{stima qk Egorov}, \eqref{stima derivata i qk Egorov} for ${\rm Op}(q_1)$
have been proved in \eqref{stima simbolo q1 Egorov}, \eqref{stima derivata i simbolo q1 Egorov} 
 in Lemma \ref{ordini maggiori di 0 Egorov1}. 
Then we argue by induction supposing that $ q_n \in S^{\frac12(3 - n)} $ satisfies  \eqref{stima qk Egorov}, 
\eqref{stima derivata i qk Egorov}. 
Then, recalling \eqref{definizione qk}, Lemma \ref{proprieta sigma(q,A)} and 
\eqref{stima coefficienti a egorov} imply 
\begin{align*}
\norma {\rm Op}(q_{n + 1}) \norma_{\frac12 (3 -  (n+1)), s, \alpha}^{k_0, \gamma} & 
\leq_{n, S, \alpha}
\e \big(1 + \| \fracchi_0\|_{s + \sigma + \beth_{n + 1}(M, \alpha)}^{k_0, \gamma} \big) 
\end{align*}
where $\beth_{n + 1}(M, \alpha) $ is defined in \eqref{c k + 1 c k}. By \eqref{definizione qk} 
$$
\begin{aligned}
\partial_i {\rm Op}(q_{n + 1}) [\hat \imath] & = \ii {\rm Op}\Big(\int_0^t (  \partial_i {\mathfrak a}[\hat \imath] \star q_{n - 1})(\tau, \vphi, x, \xi) \, d \tau \Big)  \\& \quad + \ii {\rm Op}\Big(\int_0^t (  {\mathfrak a}\star  \partial_i q_{n - 1})(\tau, \vphi, x, \xi) [\hat \imath]  \, d \tau \Big)\, .
\end{aligned}
$$
Then \eqref{stima coefficienti a egorov}, \eqref{stime derivate coefficienti a Egorov}, \eqref{stima qk Egorov}, \eqref{stima derivata i qk Egorov}, \eqref{ansatz I delta}, \eqref{ansatz indici Egorov 2} imply 
$$
\norma \partial_i {\rm Op}(q_{n + 1})[\hat \imath] \norma_{\frac12 (3 -  (n+1)), s_1, \alpha} \leq_{n, S, \alpha} \e \| \hat \imath \|_{s_1 + \sigma + \beth_{n + 1}(M, \alpha) }\,.
$$
In the same way 
\eqref{stima RN diagonale Egorov}, \eqref{stima derivata i RN diagonale Egorov} follow. 
\end{proof}

\begin{remark}
We need \eqref{stima RN diagonale Egorov} only for $\alpha = 0$.
\end{remark}

We now estimate the difference $ W (\vphi, t) $ in \eqref{forma finale P - Q} 
between the approximate solution $ Q(\vphi, t) $  and the exact 
solution $  P(\vphi, t) $  of the equation \eqref{equazione Egorov}. 

\begin{lemma}\label{soluzione vera Egorov}
For all $\beta \in \N $ with  $ \beta  + k_0 + 4 \leq M $,  
 the operators $ \partial_{\vphi_j}^\beta  W(\vphi, t)  $,
$ \partial_{\vphi_j}^\beta [ W(\vphi, t), \partial_x] $,  $j = 1, \ldots, \nu$, are $ {\mathcal D}^{k_0} $-tame 
with tame constants 
\begin{equation}\label{stima W}
\begin{aligned}
& {\mathfrak M}_{ \partial_{\vphi_j}^\beta  W(\vphi, t) } (s), 
{\mathfrak M}_{ \partial_{\vphi_j}^\beta [ W(\vphi, t), \partial_x]  } (s) \leq_{S, M} \e 
(1 +  \| \fracchi_0 \|_{s + \sigma + \frac32 M + \daleth (M) 
+ \beta }^{k_0, \gamma})\,,  \\
& \qquad \forall s_0 \leq s \leq S\,,
\end{aligned}
\ee
for some $\sigma := \sigma(\tau, \nu, k_0) > 0 $ and 
(the constants $ \beth_n (M, \alpha)$  are defined in Lemma \ref{soluzione approssimata Egorov})
\begin{equation}\label{perdita pseudo-diff RN}
\daleth (M) := \beth_{M + 1}(M, 0) \, .
\end{equation}
Moreover if the constant $\mu$ in \eqref{ansatz I delta} satisfies 
\begin{equation}\label{ansatz 3 indici Egorov}
s_1 +  \sigma + \frac32 M + \daleth (M) 
+ \beta \leq s_0 + \mu\,,
\end{equation}
then 
\begin{equation}\label{stima derivata i W}
\begin{aligned}
& \|  \partial_{\vphi_j}^\b [ \partial_i W(\vphi, t)[\hat \imath], \partial_x]  \|_{{\mathcal L}(H^{s_1})} , \|  \partial_{\vphi_j}^\b \partial_i W(\vphi, t)[\hat \imath]  \|_{{\mathcal L}(H^{s_1})}  \\
& \leq_{M, S} \e   \|\hat \imath\|_{s_1 + \sigma + \frac32 M + \daleth (M) 
+ \beta}\,.
\end{aligned}
\end{equation}
\end{lemma}

\begin{proof}
To simplify  $\partial_\vphi := \partial_{\vphi_j}$, $j = 1, \ldots, \nu$.
We prove that  $ \partial_{\vphi}^\beta [W(\vphi, t), \partial_x] =  
\partial_{\vphi}^\beta W(\vphi, t) \partial_x - \partial_x \partial_{\vphi}^\beta W(\vphi, t) $
is $ {\mathcal D}^{k_0} $-tame. 
We first consider $ \partial_\vphi^\beta W(\vphi, t) \partial_x $. 
Recalling \eqref{forma finale P - Q} it is sufficient to estimate $ \forall t, \t \in [0,1] $ 
\begin{align*}
& \partial_\vphi^\beta \partial_\lambda^k \Big( \Phi(t - \tau ) R_M(\tau) \Phi( \tau - t) \Big)  \\
& = 
\!\! \!\! \! \sum_{\begin{subarray}{c}
\beta_1 + \beta_2 + \beta_3 = \beta \\
k_1 + k_2 + k_3 = k
\end{subarray}} \!\! \!\! \! C(\beta_1, \ldots, k_3) \partial_\vphi^{\beta_1} \partial_\lambda^{k_1}\Phi(t - \tau) \partial_\vphi^{\beta_2}\partial_\lambda^{k_2} R_M(\tau) \partial_\vphi^{\beta_3} \partial_\lambda^{k_3}\Phi(\tau - t)\, 
\end{align*}
where $\beta_1, \beta_2, \beta_3 \in \N$ and $k_1, k_2, k_3 \in \N^{\nu + 1}$.
We write each term as
\begin{align}
& \nonumber \partial_\vphi^{\beta_1} \partial_\lambda^{k_1} \Phi(t - \tau) \partial_\vphi^{\beta_2} \partial_\lambda^{k_2} R_M(\tau) \partial_\vphi^{\beta_3} \partial_\lambda^{k_3} \Phi(\tau - t) \partial_x   = \\	
& 
\partial_\vphi^{\beta_1} \partial_\lambda^{k_1} \Phi(t - \tau) \langle D \rangle^{- \frac{\beta_1 + |k_1|}{2}} 
\label{dec1} \\ 
&   \langle D \rangle^{\frac{\beta_1 + |k_1|}{2}} \partial_\vphi^{\beta_2} \partial_\lambda^{k_2} 
R_M(\tau) \langle D \rangle^{\frac{\beta_3 + |k_3|}{2} + 1} \label{dec2}  \\
&   \langle D \rangle^{- \frac{\beta_3 + |k_3|}{2} - 1}  \partial_\vphi^{\beta_3} \partial_\lambda^{k_3} \Phi(\tau - t) \partial_x \,. \label{dec3}
\end{align}
Propositions \ref{Teorema totale partial vphi beta k D beta k Phi}
and \ref{lemma:tame derivate flusso} and \eqref{stima coefficienti a egorov}
provide the estimates for \eqref{dec1} and \eqref{dec3}: 
for some $\sigma := \sigma(\tau, \nu , k_0) > 0 $, 
\begin{align}
& \| \partial_\vphi^{\beta_1} \partial_\lambda^{k_1} \Phi(t - \tau) \langle D \rangle^{- \frac{\beta_1 + |k_1|}{2}} h \|_s 
\leq_s  \gamma^{- |k_1|} \big(\| h \|_s + \| \fracchi_0 \|_{s + \beta_1 + \sigma}^{k_0, \gamma} \| h \|_{s_0} \big)\,,
 \label{D alpha 1 Phi} \\
& \| \langle D \rangle^{- \frac{\beta_3  + |k_3|}{2} - 1}  \partial_\vphi^{\beta_3} \partial_\lambda^{k_3}  \Phi(\tau - t) \partial_x  h  \|_s \leq_s \gamma^{- |k_3|} \big( \| h \|_s + \| \fracchi_0 \|_{s + \frac32 \beta_3 + \sigma}^{k_0, \gamma} \| h \|_{s_0} \big) \, . 
\label{D alpha 3 Phi |D|}
\end{align}
We now estimate the norm of the pseudo-differential operator in \eqref{dec2}  
where  $ R_M  \in OPS^{1 - \frac{M}{2}} $, see \eqref{stima RN diagonale Egorov}.  
By \eqref{norm-increa}, $ \b_0+ k_0 + 4 \leq M  $,
Lemmata \ref{lemma composizione multiplier}  and \ref{lemma stime Ck parametri},  
\eqref{Norm Fourier multiplier},  we get
\begin{align}
 \norma \langle D \rangle^{\frac{\beta_1 + |k_1|}{2}} \partial_\vphi^{\beta_2} \partial_\lambda^{k_2} R_M(\tau) 
\langle D \rangle^{\frac{\beta_3 + |k_3|}{2} + 1} \norma_{0, s, 0}  &  \nonumber\\
 \leq_s  \norma \langle D \rangle^{\frac{\beta_1 + |k_1|}{2}} \partial_\vphi^{\beta_2} \partial_\lambda^{k_2} R_M(\tau) 
\langle D \rangle^{\frac{\beta_3 + |k_3|}{2} + 1} \norma_{\frac{\b_1+|k_1|}{2} + 1- \frac{M}{2} +
\frac{\b_3 + |k_3|}{2} +1, s, 0} & \nonumber\\  
 \leq_s
\norma \langle D \rangle^{\frac{ \beta_1  + |k_1|}{2}} \partial_\vphi^{\beta_2} \partial_\lambda^{k_2} R_M(\tau) 
 \norma_{\frac{\b_1+|k_1|}{2} + 1- \frac{M}{2} , s, 0} & \nonumber \\ 
 \leq_s \norma  \partial_\vphi^{\beta_2} \partial_\lambda^{k_2} R_M(\tau) 
 \norma_{1- \frac{M}{2} , s + \frac{\b_1+|k_1|}{2}, 0}  &\nonumber\\ 
 \leq_{s,M} \gamma^{-|k_2|} 
\norma R_M(\tau) \norma_{1- \frac{M}{2}, s + \frac32  \beta  + \frac{k_0}{2}, 0}^{k_0,\gamma} & \nonumber\\ 
 \stackrel{ \eqref{stima RN diagonale Egorov} }{\leq_{S,M}}
\e \gamma^{- |k_2|} 
\big( 1 + \| \fracchi_0 \|_{s + \sigma + \daleth(M) + \frac32  \b + \frac{k_0}{2}}^{k_0, \gamma} \big) & \label{D alpha1 RM D alpha3}
\end{align}

where $\daleth(M) := \beth_{M + 1}(M, 0) $, see  \eqref{perdita pseudo-diff RN}.
Then \eqref{D alpha 1 Phi}, \eqref{D alpha 3 Phi |D|}, \eqref{D alpha1 RM D alpha3} and 
Lemma \ref{lemma: action Sobolev} imply that 
$ \pa_\vphi^\beta W (\vphi, t) \pa_x $ is $ {\mathcal D}^{k_0} $-tame with tame constant
$ \leq C(S) \e 
(1 +  \| \fracchi_0 \|_{s + \sigma + \frac32 M + \daleth (M) 
+ \beta }^{k_0, \gamma}) $. The operator $ \partial_x \partial_{\vphi}^\beta W(\vphi, t) $ satisfies a similar estimate and so
  \eqref{stima W} is proved.

 The estimate \eqref{stima derivata i W} follows by differentiating the operator $W(\vphi, t)$ with respect to the torus $i$, using the same strategy as above, applying 
\eqref{ansatz I delta}, \eqref{ansatz 3 indici Egorov}, the estimate  \eqref{stima derivata i RN diagonale Egorov} for $\partial_i R_M(\tau)[\hat \imath]$, Proposition \ref{lemma:tame derivate flusso} and 
the estimates for $ \pa_i \Phi $ in Propositions \ref{derivata del flusso rispetto al toro}-\ref{derivate i vphi omega flusso D destra}. 
\end{proof}

The following lemma can be proved as Lemmata \ref{soluzione approssimata Egorov} and \ref{soluzione vera Egorov}. 

\begin{lemma}\label{stime sk} 
For all $ n  \in \{1, \ldots, M \} $ the symbols 
$ s_n \in S^{\frac12(1 - n)}$ defined in \eqref{sk soluzione} satisfy  
\begin{equation}\label{stima sk Egorov}
\norma {\rm Op}(s_n)   \norma_{\frac12 (1 - n), s, \alpha}^{k_0, \gamma} 
\leq_{n, S, \a} \e \big(1 + \| \fracchi_0 \|_{s + \sigma + \beth_{n + 2}(M, \alpha)}^{k_0, \gamma} \big)\,, \quad \forall s_0 \leq s \leq S
\end{equation}
where the constants $\beth_n (M, \alpha)$ are defined in \eqref{c k + 1 c k}. 
The operator 
$$
R_{\omega, M}(\vphi, t) := - \ii {\rm Op}({\mathfrak a} \star  s_M) \in OPS^{ - \frac{M}{2}}
$$ 
 satisfies 
$$
\norma R_{\omega, M}(\vphi, t) 
\norma_{- \frac{M}{2}, s, \alpha}^{k_0, \gamma} \leq_{M, S, \a} 
\e \big( 1 + \| \fracchi_0 \|_{s + \sigma + \beth_{M + 3}(M, \alpha)}^{k_0, \gamma} \big)\,, \quad \forall s_0 \leq s \leq S\,.
$$
For all $\beta \in \N$, $\beta + k_0 + 4\leq M $, the operators 
$  \partial_{\vphi_j}^\beta W_\omega(\vphi, t)$, $  \partial_{\vphi_j}^\beta [W_\omega(\vphi, t), \partial_x]$, $j = 1, \ldots , \nu$  (recall  \eqref{forma finale S omega - S M}) are ${\mathcal D}^{k_0}$-tame where the tame constant satisfies 
\begin{equation}\label{stima W omega}
\begin{aligned}
& {\mathfrak M}_{\partial_{\vphi_j}^\beta [W_\omega(\vphi, t), \partial_x]}(s)\,, {\mathfrak M}_{ \partial_{\vphi_j}^\beta W_\omega(\vphi, t)}(s) \, \\
&  \leq_{M, S} \e ( 1 +  \| \fracchi_0 \|_{s + \sigma + \frac32 M + \daleth(M + 2)  + \beta}^{k_0, \gamma}) \,,   \quad \forall s_0 \leq s \leq S\,.
\end{aligned}
\end{equation}
Moreover if the constant $\mu$ in \eqref{ansatz I delta} satisfies 
$
s_1 + \sigma + \frac32 M + \daleth(M + 2)  + \beta \leq s_0 + \mu
$
then 
\begin{align}\label{stima derivata i sk Egorov}
& \norma \partial_i {\rm Op}(s_n)[\hat \imath]   \norma_{\frac12 (1 - n), s_1, \alpha}
\leq_{n, S, \a} \e  \| \hat \imath \|_{s_1 + \sigma + \beth_{n + 2}(M, \alpha)}\,, \\
& \norma \partial_i R_{\omega, M}(\vphi, t) [\hat \imath]
\norma_{- \frac{M}{2}, s_1 , \alpha} \leq_{M, S, \a} 
\e  \| \hat \imath \|_{s_1 + \sigma + \beth_{M + 3}(M, \alpha)}\,, 
\end{align}
and
\be
\begin{aligned}
& \|  \partial_{\vphi_j}^\b [ \partial_i W_\omega (\vphi, t)[\hat \imath], \partial_x ]   \|_{{\mathcal L}(H^{s_1})}\,,\, \|  \partial_{\vphi_j}^\b  \partial_i W_\omega (\vphi, t)[\hat \imath]  \|_{{\mathcal L}(H^{s_1})} \\
& \leq_{M, S} \e  \| \hat \imath \|_{s_1 + \sigma + \frac32 M + \daleth(M + 2)  + \beta} \,. \label{stima derivata i W omega}
\end{aligned}
\ee
\end{lemma} 

We summarize the whole section in the next proposition:
 
\begin{proposition}\label{Prop:Egorov} 
Let $ a ( \vphi , x ) $ be as in \eqref{definizione a} and $ a_0 (\vphi) $ in \eqref{choicea0}. Then 
the conjugated operator ${\mathcal L}_M^{(1)}$ in \eqref{coniugio flusso di una PDE iperbolica} is real, even, 
reversible and has the form  
\begin{equation}\label{cal LN (1)}
{\mathcal L}_M^{(1)} = \Dom {\mathbb I}_2 + \ii \mathtt m_3 {\bf T}(D) + 
\ii \big({\bf C}_1(\vphi, x) + {\bf C}_0(\vphi, x) \mH \big) |D|^{\frac12}  + {\bf R}_M^{(1)} + {\bf Q}_M^{(1)}
\end{equation}
where $ {\bf C}_1 (\vphi, x) $, ${\bf C}_0 (\vphi, x) $ are defined in \eqref{bf A1}, the function 
$ a_{14} $ satisfies \eqref{media a 14},  
and 
$$
 {\bf R}_M^{(1)} := \begin{pmatrix}
{\mathcal R}_M^{(1)} & 0 \\
0 & \overline {\mathcal R}_M^{(1)}
\end{pmatrix}\,, \quad 
{\bf Q}_M^{(1)} := \begin{pmatrix}
  0 & {\mathcal Q}_M^{(1)} \\
\overline {\mathcal Q}_M^{(1)} & 0  
\end{pmatrix} \, . 
$$
For all $\beta \in \N$, $ \beta  + k_0 + 4 \leq M $, 
the operators $ \pa_{\vphi_j}^\beta  {\mathcal R}_M^{(1)} $, $ \pa_{\vphi_j}^\beta  [ {\mathcal R}_M^{(1)}, \partial_x] $, $ \pa_{\vphi_j}^\beta  {\mathcal Q}_M^{(1)} $,
$  \pa_{\vphi_j}^\b  [ {\mathcal Q}_M^{(1)}, \pa_x] $, $j = 1, \ldots, \nu$ are $ {\mathcal D}^{k_0}$-tame with 
 tame constants satisfying for all $s_0 \leq s \leq S$
 \begin{equation}\label{tame resto diagonale Egorov}
\begin{aligned}
& {\mathfrak M}_{ \pa_{\vphi_j}^\beta  [ {\mathcal R}, \partial_x]} (s)\,,\,{\mathfrak M}_{ \pa_{\vphi_j}^\beta  {\mathcal R}} (s)\\
& \leq_{M, S} \e (1 +  \| \fracchi_0 \|_{s + \sigma + \frac32 M + \daleth(M + 2) + \beta}^{k_0, \gamma})\,,  \quad {\mathcal R} \in \{ {\mathcal R}_M^{(1)}, {\mathcal Q}_M^{(1)} \}
\end{aligned}
\end{equation}
where  the constant $ \daleth (M+2) $ is defined by \eqref{perdita pseudo-diff RN}. Moreover if the constant $\mu$ in \eqref{ansatz I delta} satisfies 
\begin{equation}\label{ansatz finale indici derivate Egorov}
s_1 + \sigma + \chi M + \daleth(M + 2)  + \b \leq s_0 + \mu\,,
\end{equation}
then each  ${\mathcal R} \in \{ {\mathcal R}_M^{(1)}, {\mathcal Q}_M^{(1)} \}$ satisfies 
\begin{equation}\label{tame resto diagonale Egorov derivata i}
\begin{aligned}
& \|   \pa_{\vphi_j}^\beta  [\partial_i {\mathcal R} [\hat \imath], \partial_x]   \|_{{\mathcal L}(H^{s_1})}\,, \|   \pa_{\vphi_j}^\beta   \partial_i {\mathcal R} [\hat \imath] \|_{{\mathcal L}(H^{s_1})}
\\
& \leq_{M, S} \e     \| \hat \imath \|_{s_1 + \sigma + \frac32 M + \daleth(M + 2) + \beta}\,.\quad 
\end{aligned}
\end{equation}
\end{proposition} 
 
\begin{proof} 
It remains only to prove \eqref{tame resto diagonale Egorov} and  \eqref{tame resto diagonale Egorov derivata i}.

\noindent
{\sc Proof of \eqref{tame resto diagonale Egorov}.}  
We estimate each term in  \eqref{cal QN (1)}. Let $\partial_\vphi := \partial_{\vphi_j}$, $j = 1, \ldots, \nu$. 
The estimates \eqref{stima resto q leq 2}, \eqref{stima qk Egorov}, \eqref{stima sk Egorov} imply 
$$ 
\norma r_M^{(1)}(x, D) \norma_{0, s, \alpha}^{k_0, \gamma} 
\leq_{S, \alpha} \e ( 1 + \| \fracchi_0 \|_{s + \sigma + \beth_{M + 2}(M, \alpha)}^{k_0, \gamma}) \, . 
$$ 
Now since  $ \partial_\vphi^\b [\partial_\lambda^k {\rm Op}(r_M^{(1)}), \partial_x ] = \partial_\lambda^k {\rm Op}(\partial_\vphi^\b \partial_x r_M^{(1)}) $, 
we get 
\begin{align*}
\norma \partial_\vphi^\b [\partial_\lambda^k {\rm Op}(r_M^{(1)}), \partial_x ]  \norma_{0, s, 0} & \lessdot \gamma^{- |k|} \norma{\rm Op}(\partial_\vphi^\b (\partial_x r_M^{(1)})) \norma_{0, s, 0}^{k_0, \gamma} \lessdot \gamma^{- |k|} \norma {\rm Op}(r_M^{(1)}) \norma_{0, s + \beta + 1, 0}^{k_0, \gamma}   \nonumber\\
& \leq_S \e (1 + \| \fracchi_0 \|_{s + \sigma + \daleth(M + 2) + \beta}^{k_0, \gamma}) \, .
\end{align*}
Hence the operator $r_M^{(1)}(\vphi, x, D)$ satisfies the estimate \eqref{tame resto diagonale Egorov}.

 The lemma follows by the estimates \eqref{stima W}, \eqref{stima W omega}.  The proof of \eqref{tame resto diagonale Egorov} for ${\mathcal Q}_M^{(1)}$ is similar.  It follows by
\eqref{stima resti prima Egorov} (for $\alpha = 0$) and  Lemma \ref{lemma:tame derivate flusso} using the same strategy for proving  \eqref{stima W} in Lemma \ref{soluzione vera Egorov}. 

\smallskip

\noindent
{\sc Proof of \eqref{tame resto diagonale Egorov derivata i}.} 
It follows by differentiating with respect to $i$ the expression of ${\mathcal R}_M^{(1)}$ in \eqref{cal QN (1)} and by applying the estimates \eqref{stima derivate i resto q leq 2}, \eqref{stima derivata i qk Egorov}, \eqref{stima derivata i W}, \eqref{stima derivata i sk Egorov}, \eqref{stima derivata i W omega}. 
\end{proof} 

\section{Space reduction of the order $|D|^{\frac12}$} \label{sec:lineare}
  
The aim of this section is to eliminate the $ x $-dependence of the coefficient in front of 
$ |D|^{\frac12} $ in the operator 
$ {\mathcal L}_M^{(1)} $ in \eqref{cal LN (1)} (where we have neglected the term \eqref{bf RN (1) bot})  and 
$ \Pi_{K_n} {\bf C}_1 := \begin{pmatrix}
\Pi_{K_n}a_{14} & 0 \\
0 & - \Pi_{K_n} a_{14}
\end{pmatrix} $. 

We conjugate $ {\mathcal L}_M^{(1)} $  by means of a real operator of the form 
 \begin{equation}\label{definition cal M}
 {\bf V} := \begin{pmatrix}
 {\mathcal V} & 0 \\
 0 & \overline{\mathcal V}
 \end{pmatrix}\,,\qquad 
 {\mathcal V} := {\rm Op}(v)\,,\qquad v := v(\vphi,x, \xi) \in S^0\,.
 \end{equation} 
Setting 
  $
  {\bf \Sigma} := \begin{pmatrix}
  1 & 0 \\
  0 & - 1
  \end{pmatrix}
  $
  and recalling that  $ \mathtt m_{1, K_n} $ is defined by \eqref{media a 14}, 
 we compute 
 \begin{equation}\label{svil-D12}
 \begin{aligned}  
& {\mathcal L}_M^{(1)} {\bf V} - {\bf V} \big(\Dom {\mathbb I}_2 + \ii \mathtt m_3 {\bf T}(D) + 
  \ii {\mathtt m}_{1, K_n}  {\bf \Sigma} |D|^{\frac12} \big)  \\
  &  =    \ii \mathtt m_3[{\bf T}(D), {\bf V}]   
  +  \ii   (\Pi_{K_n}{\bf C}_1 + {\bf C}_0 \mH) |D|^{\frac12} {\bf V}    \\ 
& \   \   - \ii {\mathtt m}_{1, K_n} {\bf V}   {\bf \Sigma} |D|^{\frac12} +  (\Dom {\bf V}) + \big( {\bf R}_M^{(1)} + {\bf Q}_M^{(1)} \big) {\bf V} \, . 
 \end{aligned} 
 \end{equation}
By \eqref{definizione T3 A1}, \eqref{def T} and \eqref{composition pseudo},  the commutator  has the expansion
 \begin{align*}
& \ii \mathtt m_3  [{\bf T}(D), {\bf V}] =   \begin{pmatrix}
   \ii \mathtt m_3 [T(D), {\mathcal  V}] & 0 \\
  0 & -  \ii \mathtt m_3 [T(D), \overline{\mathcal  V}]
  \end{pmatrix}, \\  
& \qquad  \qquad  \qquad \quad  \ \ii \mathtt m_3 [T(D), {\mathcal V}] = \mathtt m_3{\rm Op}\big( \partial_\xi T(\xi) v_x \big) 
 + r_{T, {\mathcal V}}(x, D)
 \end{align*}
with $  r_{T, {\mathcal V}}(x, D) \in OPS^{- \frac12} $. Similarly (recall \eqref{bf A1}) the operator 
$$
\begin{aligned}
&  \ii \big( \Pi_{K_n }{\bf C}_1+ {\bf C}_0 \mH \big) |D|^{\frac12} {\bf V}  \\
& = \begin{pmatrix}
 \ii  (\Pi_{K_n}a_{14} + a_{12} \mH) |D|^{\frac12} {\mathcal V} & 0 \\
  0 & - \ii  (\Pi_{K_n} a_{14} + a_{12} \mH) |D|^{\frac12} \overline{\mathcal V}
  \end{pmatrix}
  \end{aligned}
  $$
  has the expansion 
  \begin{equation}\label{pezzo-1}
  \begin{aligned}
 & \ii  (\Pi_{K_n} a_{14} + a_{12} \mH) |D|^{\frac12} {\mathcal V} \\
 & = {\rm Op}\big( 
  \big( \ii \Pi_{K_n} a_{14} + a_{12}  {\rm sign}(\xi) \big) |\xi|^{\frac12} \chi(\xi) v \big) +  
 \mathfrak r_{\mathcal V}(x, D) 
   \end{aligned} 
   \end{equation}
 with $ \mathfrak r_{\mathcal V}(x, D)\in OPS^{- \frac12} $.  
In addition 
  \begin{align}\label{pezzo-2}
   \ii{\mathtt m}_{1, K_n} {\bf V} {\bf \Sigma} |D|^{\frac12}  = \begin{pmatrix}
{\rm Op} \big(\ii  {\mathtt m}_{1, K_n} v \, \chi(\xi) |\xi|^{\frac12} \big) & 0 \\
  0 & \ov{ {\rm Op} \big(\ii  {\mathtt m}_{1, K_n} v \,  \chi(\xi)|\xi|^{\frac12}} \big) 
  \end{pmatrix}  \, . 
  \end{align}
By \eqref{pezzo-1}, \eqref{pezzo-2} and 
decomposing the cut-off function $ \chi(\xi) = \chi_0(\xi) + $ $ (\chi(\xi ) - \chi_0(\xi)) $  
where  $ \chi_0 $ is the cut-off function defined in \eqref{cut off simboli 2}, we get  
$$
\begin{aligned}
& \ii  \big( (\Pi_{K_n} a_{14} + a_{12} \mH) |D|^{\frac12} {\mathcal V} -  {\mathtt m}_{1, K_n} {\mathcal V} |D|^{\frac12}\big) 
= \\ 
& {\rm Op}\big( 
  \big( \ii (\Pi_{K_n} a_{14} -  {\mathtt m}_{1, K_n}) + a_{12}  {\rm sign}(\xi)  \big) |\xi|^{\frac12} \chi_0(\xi) v \big) +  
  r_{\mathcal V}(x, D)  	
  \end{aligned}
$$
where 
$$
\begin{aligned}
&r_{\mathcal V}(x, D) := \\
&  \mathfrak r_{\mathcal V}(x, D) + {\rm Op}\big( 
  \big( \ii \Pi_{K_n} a_{14} + a_{12}  {\rm sign}(\xi) - \ii {\mathtt m}_{1, K_n} \big) |\xi|^{\frac12} (\chi(\xi) - \chi_0(\xi)) v \big) \in OPS^{- \frac12}
  \end{aligned}
$$
 noting that $ \big( \ii \Pi_{K_n} a_{14} + a_{12}  {\rm sign}(\xi) - \ii {\mathtt m}_{1, K_n} \big) |\xi|^{\frac12} (\chi(\xi) - \chi_0(\xi)) v \in S^{- \infty}$ because $\chi(\xi) - \chi_0(\xi) = 0$ for $|\xi| \geq 3/4$.
 Therefore we  have to solve the equation 
  \begin{equation}\label{equazione omologica ordine 1/2}
  \mathtt m_3 \partial_\xi T (\xi)  v_x 
  + \big( \ii ( \Pi_{K_n} a_{14}  - {\mathtt m}_{1, K_n}) + a_{12}   {\rm sign}(\xi)  \big) \chi_0(\xi) |\xi|^{\frac12} 
  v  = 0\,.
  \end{equation}
  We look for a solution of \eqref{equazione omologica ordine 1/2} of the form 
  \begin{equation}\label{ansatz b p}
v :=  v(\vphi, x, \xi) := {\rm exp}(p(\vphi, x, \xi))\,, \qquad p := p(\vphi, x, \xi) \in S^0 \, . 
  \end{equation}
Thus,  from \eqref{equazione omologica ordine 1/2}, the symbol $ p $ has to solve 
\begin{equation}\label{equazione per p}
\mathtt m_3 \partial_\xi T(\xi) p_x (\vphi, x, \xi) =  
-  \big( \ii ( \Pi_{K_n} a_{14} (\vphi, x) - {\mathtt m}_{1, K_n}) + a_{12}(\vphi, x) {\rm sign}(\xi) \big) \chi_0(\xi) |\xi|^{\frac12} \,.
\end{equation}
The right hand side in  \eqref{equazione per p}  has zero average in $ x $  by \eqref{media a 14} and  
because $ a_{12} $ is odd in $ x $, by  \eqref{bf B0 (I)}, \eqref{defA0} and remark \ref{parities q a7 a8 a9 a10}.
By \eqref{def T}  the derivative 
$$ 
\pa_\xi T (\xi) = \begin{cases} 
\dfrac{\chi(\xi)\sign (\xi) ( 1 + 3 \kappa \xi^2 ) }{ 2 |\xi|^{1/2} (1+ \kappa \xi^2)^{1/2}} + \partial_\xi \chi(\xi) |\xi|^{\frac12} (1 + \kappa |\xi|^2)^{\frac12} \in S^{1/2}\, 
&  \text{if} \   |\xi| > \frac13  \\
0 &  \text{if}  \  |\xi| \leq \frac13  \, .
\end{cases}
$$
Since the symbol $ T(\xi) $ is even in $ \xi $, the derivative 
$ \pa_\xi T (\xi)  $  is odd. Moreover, by \eqref{cut off simboli 1},  
  $\partial_\xi \chi(\xi) > 0$ for all $1/3 < \xi< 2/3$, and so 
 $|\partial_\xi T(\xi)| > 0$ for all $|\xi| > 1/3$ and  $|\partial_\xi T(\xi)| > c > 0$ for all $|\xi| \geq 1/2$.
Therefore 
 \eqref{equazione per p} admits the solution  
  \begin{equation}\label{definizione p}
  \begin{aligned}
  & p (\vphi, x, \xi) \\
  & := \begin{cases}
  -  \frac{|\xi|^{\frac12}\chi_0(\xi) \partial_x^{- 1} \big( \ii ( \Pi_{K_n} a_{14} (\vphi, x) - {\mathtt m}_{1, K_n})  + 
  a_{12}(\vphi, x) {\rm sign}(\xi) \big)}{\mathtt m_3 
  \partial_\xi T(\xi)} & \ \text{if} \  |\xi| > \frac12 \\
  0 & \ \text{if} \ |\xi| \leq \frac12 \, .
  \end{cases} 
  \end{aligned}
  \end{equation}
   Since $ p (-\vphi, x, -\xi ) = \overline{p(\vphi,x, \xi)} $ 
and $  p (\vphi, - x, -\xi )  = p (\vphi, x, \xi ) $, then
   $ {\bf V} $ is reversibility preserving and  $ {\bf V} $ is even, by Lemma \ref{even:pseudo}. 
As a consequence  \eqref{svil-D12}-\eqref{equazione omologica ordine 1/2} imply that
\be\label{coniugazione67}
  {\bf V}^{- 1} {\mathcal L}_M^{(1)} {\bf V} = \Dom {\mathbb I}_2 + 
  \ii \mathtt m_3 {\bf T}(D) + \ii {\mathtt m}_{1, K_n}  {\bf \Sigma} |D|^{\frac12}  + {\bf R}_M^{(2)} + {\bf Q}_M^{(2)}
\ee
with  block-diagonal terms 
  \begin{align}\nonumber
& \qquad \qquad \qquad  {\bf R}_M^{(2)} := \begin{pmatrix}
  {\mathcal R}_M^{(2)} & 0 \\
  0 & \overline {\mathcal R}_M^{(2)}
  \end{pmatrix}\,, \quad {\bf Q}_M^{(2)} := \begin{pmatrix}
 0&  {\mathcal Q}_M^{(2)}  \\
   \overline{\mathcal Q}_M^{(2)} & 0
  \end{pmatrix} 
  \end{align}
  \begin{equation}\label{cal RN (3)}
  \begin{aligned}
& {\mathcal R}_M^{(2)} := {\mathcal V}^{- 1} 
 \big( r_{T, {\mathcal V}}(x, D) + r_{\mathcal V}(x, D) + \Dom {\mathcal V} + {\mathcal R}_M^{(1)} {\mathcal V}  \big) \, , \\ 
& {\mathcal Q}_M^{(2)} := {\mathcal V}^{- 1} {\mathcal Q}_M^{(1)} \overline{\mathcal V}   \, .
\end{aligned}
\end{equation}
Finally we define the real, even and reversible operator
\be
\label{cal LN (3)}
{\mathcal L}_M^{(2)} := \Dom {\mathbb I}_2 + 
  \ii \mathtt m_3 {\bf T}(D) + \ii  \mathtt m_{1}  {\bf \Sigma} |D|^{\frac12}  + {\bf R}_M^{(2)} + {\bf Q}_M^{(2)}
\ee
where the coefficient
\begin{equation}\label{lambda 1 senza proiettore}
{\mathtt m}_1 :=  - 
\frac{ (2\p)^{-\nu - \frac52} }{2  \sqrt{\kappa}} 
\int_{\T^{\nu + 1}} 
(1 + \b_x) [ \om \cdot \pa_\vphi \beta + V (1+\b_x) ]^2 
 \Big( \int_\T \sqrt{1 + \eta_y^2} \, dy \Big)^{3/2}d\vphi \, dx
\end{equation}
substitutes $ {\mathtt m}_{1, K_n}$ in \eqref{coniugazione67}, i.e. 
\begin{align}\label{resto modi alti lambda 1}
& {\bf V}^{- 1} {\mathcal L}_M^{(1)} {\bf V} =  {\mathcal L}_M^{(2)} + {\bf R}_{{\mathtt m}_1}^\bot\,, \quad {\bf R}_{{\mathtt m}_1}^\bot := \ii ({\mathtt m}_{1, K_n} - {\mathtt m}_1){\bf \Sigma} |D|^{\frac12}\, .
\end{align}
The term ${\bf R}_{{\mathtt m}_1}^\bot$ will contribute to the remainder ${\bf R}_\omega^\bot$ in the estimates \eqref{stima R omega bot corsivo bassa}-\eqref{stima R omega bot corsivo alta}. 
\begin{lemma}\label{lambda1 - lambda Kn}
$ |{\mathtt m}_1 - {\mathtt m}_{1, K_n}|^{k_0, \gamma} \leq C \e K_n^{- b} $, $ \forall b > 0 $. 
\end{lemma}

\begin{proof}
By \eqref{lm1 formula}, \eqref{lambda 1 senza proiettore} one has 
$$
\begin{aligned}
& {\mathtt m}_1 - {\mathtt m}_{1, K_n} =  \\ 
& \frac{ (2\p)^{-\nu -\frac52} }{2  \sqrt{\kappa}} 
\int_{\T^{\nu}} 
(1 + \b_x) [ \om \cdot \pa_\vphi \beta + V (1+\b_x) ]^2 
 \Pi_{K_n}^\bot\Big( \int_\T \sqrt{1 + \eta_y^2} \, dy \Big)^{3/2}d\vphi d x\,.
 \end{aligned}
$$
Then the lemma follows by \eqref{stima V B a c}, \eqref{beta lambda3}, \eqref{stima m3(vphi)}, \eqref{tame Tdelta}, \eqref{ansatz I delta}, using the smoothing property \eqref{smoothing-u1}. 
\end{proof}

\begin{lemma}\label{lemma:6.27}
The coefficient ${\mathtt m}_1$ defined in \eqref{lambda 1 senza proiettore} satisfies, for some $\sigma := \sigma(\tau, \nu, k_0) > 0$,
 the estimates 
\begin{equation}\label{stime lambda 1 senza proiettore}
|{\mathtt m}_1|^{k_0, \gamma} \leq C \e \,, \quad |\partial_i {\mathtt m}_1[\hat \imath]| \leq C \e \| \hat \imath \|_\sigma\, .
\end{equation}
The operator ${\bf V }$ defined in \eqref{definition cal M} is real, even, reversibility preserving and 
$ {\mathcal V} = {\rm Op}(v (\vphi, x, \xi )) \in OPS^0 $ with symbol $v(\vphi, x, \xi) \in S^0 $ 
defined in  \eqref{ansatz b p} and  \eqref{definizione p}, satisfies, for all $ s_0 \leq s \leq S $, 
\begin{equation}\label{estimate cal M} 
\norma {\mathcal V}^{\pm 1} - {\rm Id} \norma_{0,s,0}^{k_0, \gamma}\,, \, \norma ({\mathcal V}^{\pm 1} - {\rm Id})^* \norma_{0,s,0}^{k_0, \gamma} \leq_S \e (1 + \| \fracchi_0 \|_{s + \sigma}^{k_0, \gamma}) \, . 
\end{equation} 
For all $\beta \in \N$, $ \beta  + k_0 + 4 \leq M $, the operators $ \partial_{\vphi_j}^\beta {\mathcal R}_M^{(2)}$, $ \partial_{\vphi_j}^\beta [{\mathcal R}_M^{(2)}, \partial_x]$, $ \partial_{\vphi_j}^\beta {\mathcal Q}_M^{(2)}$, $ \partial_{\vphi_j}^\beta [{\mathcal Q}_M^{(2)}, \partial_x ]$  are 
${\mathcal D}^{k_0}$-tame and the tame constants 
${\mathfrak M}_{ \pa_{\vphi_j}^\beta  [ {\mathcal R}, \partial_x]} (s)\,,\,{\mathfrak M}_{ \pa_{\vphi_j}^\beta  {\mathcal R}} (s)$, 
${\mathcal R} \in \{ {\mathcal R}_M^{(2)} , {\mathcal Q}_M^{(2)} \}$   satisfy \eqref{tame resto diagonale Egorov} (with a
  possibly larger $\sigma := \sigma (\tau, \nu, k_0) > 0$). 
  
   \noindent
  Moreover if the  constant $\mu$ in \eqref{ansatz I delta} satisfies \eqref{ansatz finale indici derivate Egorov} (with a possibly larger $\sigma := \sigma(\tau, \nu, k_0) > 0$) then  
  \begin{equation}\label{stima derivate i cal V}
  \norma \partial_i {\mathcal V}^{\pm 1} [\hat \imath] \norma_{0, s_1, 0}\,,\, \norma \partial_i ({\mathcal V}^{\pm 1})^* [\hat \imath] \norma_{0, s_1, 0} \leq_{S} \e \| \hat \imath \|_{s_1 + \sigma}\,,
  \end{equation}
  and the remainders ${\mathcal R}_M^{(2)}$, ${\mathcal Q}_M^{(2)}$ satisfy the estimates \eqref{tame resto diagonale Egorov derivata i}.
The operators ${\mathcal R}_M^{(2)}$, ${\mathcal Q}_M^{(2)}$ are reversible.
\end{lemma}

\begin{proof}
The estimate \eqref{stime lambda 1 senza proiettore} follows by \eqref{lambda 1 senza proiettore}, \eqref{stima V B a c}, \eqref{stima derivate i primo step}, \eqref{beta lambda3}, \eqref{stima m3(vphi)}, \eqref{derivate i coefficienti secondo step}, \eqref{tame Tdelta}, \eqref{ansatz I delta}. 
The estimates \eqref{estimate cal M}, \eqref{stima derivate i cal V} for ${\mathcal V}^{\pm 1}$ follows by \eqref{definition cal M}, \eqref{ansatz b p}, \eqref{definizione p} and Lemma \ref{Neumann pseudo diff}. The estimates
for  $({\mathcal V}^{\pm 1} - {\rm Id})^*$ and $\partial_i ({\mathcal V}^{\pm 1})^*$ follow by 
Lemma \ref{stima pseudo diff aggiunto}.
Using Lemma \ref{lemma stime Ck parametri} we get 
$$ 
\norma r_{T, {\mathcal V}}(x, D) \norma_{0,s,0}^{k_0, \gamma} , 
\norma  r_{\mathcal V}(x, D)  \norma_{0, s,0}^{k_0, \gamma} \leq_S 
\e (1 + \|  \fracchi_0 \|_{s + \sigma}^{k_0, \gamma}), 
$$ 
and 
$$
\norma \partial_i r_{T, {\mathcal V}}(x, D)  [\hat \imath] \norma_{0, s_1, 0}, \norma \partial_i r_{ {\mathcal V}}(x, D)  [\hat \imath] \norma_{0, s_1, 0} \leq_{S} \e \| \hat \imath\|_{s_1 + \sigma}
$$
for some $\sigma : = \sigma(\tau, \nu, k_0) > 0$. The term ${\mathcal V}^{- 1} {\mathcal R}_M^{(1)}{ \mathcal V }$ 
in \eqref{cal RN (3)} is estimated following the same strategy of Lemma \ref{soluzione vera Egorov}. 
\end{proof}

\section{Conclusion: partial reduction of $ {\mathcal L}_\om $}\label{coniugio cal L omega}

By sections \ref{sec:linearized operator}-\ref{sec:lineare}
the linear operator ${\mathcal L}$ in \eqref{linearized vero} is semi conjugated 
to the real, even and reversible operator ${\mathcal L}_M^{(2)}$ defined in \eqref{cal LN (3)}, 
up to operators which are supported on high 
Fourier frequencies, namely 
\begin{align}\label{bf RN (3) bot}
& \qquad \qquad {\mathcal L}_M^{(2)} = {\mathcal W}_2^{- 1} {\mathcal L} {\mathcal W}_1 + {\bf R}_M^{(2), \bot} + {\bf R}_{\pi_0} \\
&  {\bf R}_M^{(2), \bot} :=  
- {\bf V}^{- 1} {\bf \Phi} {\bf \Phi}_M^{- 1} {\bf R}_4^\bot {\bf \Phi}_M {\bf \Phi}^{- 1} {\bf V} - 
{\bf V}^{- 1} {\bf R}_M^{(1), \bot} {\bf V} - {\bf R}_{{\mathtt m}_1}^\bot\,, \label{bf RN (3) bot-new}\\
& \label{pezzo con proiettore}
\  {\bf R}_{\pi_0} := - {\bf V}^{- 1} {\bf \Phi} {\bf \Phi}_M^{- 1} \rho^{- 1} ({\mathcal P}^{- 1} {\mathbb I}_2)(\ii m_3 (\vphi) \Pi_0) ({\mathcal P}{\mathbb I}_2) {\bf \Phi}_M {\bf \Phi}^{- 1} {\bf V}
\end{align}
where 
\begin{equation}\label{semiconiugio cal L8}
 {\mathcal W}_1 :=  \mZ \mB {\mathcal Q}  \mS ({\mathcal P} {\mathbb I}_2) 
 {\bf \Phi}_M {\bf \Phi}^{- 1} {\bf V} \,, \ \ 
 {\mathcal W}_2 := \mZ \mB {\mathcal Q}  \mS  ({\mathcal P}{\mathbb I}_2) \rho \, {\bf \Phi}_M {\bf \Phi}^{- 1} {\bf V} \, , 
\end{equation}
and 
${\bf R}_4^\bot, {\bf R}_M^{(1), \bot}, {\bf R}_{{\mathtt m}_1}^\bot$ are defined respectively 
in \eqref{cal RN bot (3)}, \eqref{bf RN (1) bot}, \eqref{resto modi alti lambda 1}
(they will contribute to the remainders in 
\eqref{stima R omega bot corsivo bassa}-\eqref{stima R omega bot corsivo alta}) and the operator $\Pi_0$ is defined in \eqref{definizione Pi 0}. 
The maps $ {\mathcal W}_1 $, ${\mathcal W}_2 $ are real, even and reversibility preserving. 

\noindent
Let ${\mathbb S} = {\mathbb S}^+ \cup (-{\mathbb S}^+)$ and ${\mathbb S}_0 := {\mathbb S} \cup \{ 0 \}$. We denote by $\Pi_{{\mathbb S}_0}$ the corresponding $L^2$-orthogonal projection and $\Pi_{{\mathbb S}_0}^\bot : = {\rm Id} - \Pi_{{\mathbb S}_0}$. We also denote by $H_{{\mathbb S}_0}^\bot$, the subspace of the even functions supported on the set $\mathbb S_0^c := \Z \setminus \mathbb S_0$, i.e.
\begin{equation}\label{definizione H S0 bot}
H_{{\mathbb S}_0}^\bot := \Big\{ u(x) = {\mathop \sum}_{j \in \mathbb S_0^c} u_j e^{\ii j x} : u_j=  u_{- j} \Big\}\,. 
\end{equation}

\begin{lemma}\label{Lemma:trasformazioni finali W}
Assume \eqref{ansatz I delta}. For $\e \gamma^{- 1}$ small enough, the operators 
\begin{equation}\label{Phi 1 Phi 2 proiettate}
{\mathcal W}_1^\bot := \Pi_{{\mathbb S}_0}^\bot {\mathcal W}_1 \Pi_{{\mathbb S}_0}^\bot\,,\qquad {\mathcal W}_2^\bot := \Pi_{{\mathbb S}_0}^\bot {\mathcal W}_2 \Pi_{{\mathbb S}_0}^\bot\, ,  
\end{equation}
are invertible and for all $s_0 \leq s \leq S$ they
satisfy the tame estimates
\begin{equation}\label{stima Phi 1 Phi 2 proiettate}
\| {\mathcal W}_n^\bot h \|_s^{k_0, \gamma} +
\| ({\mathcal W}_n^\bot)^{-1} h \|_s^{k_0, \gamma} \leq_{M, S} \| h \|_{s + \sigma}^{k_0, \gamma} + \| \fracchi_0 \|_{s + \sigma + \aleph_M(0)}^{k_0, \gamma} \| h \|_{s_0 + \sigma}^{k_0, \gamma}\,, \  n = 1,2\, , 
\end{equation}
for some $\sigma := \sigma(\tau, \nu) > 0$. 

\noindent
Moreover if the constant $\mu$ in \eqref{ansatz I delta} satisfies 
$ s_1 + \sigma + \aleph_M(0) \leq s_0 + \mu $ 
for some $\sigma := \sigma(\tau, \nu, k_0) > 0$, then 
\begin{equation}\label{stima derivate cal Wi}
\| \partial_i {\mathcal W}_n^{\pm 1} [\hat \imath] h \|_{s_1}\,,\, \| \partial_i ({\mathcal W}_n^\bot)^{\pm 1} [\hat \imath] h \|_{s_1} \leq_{M, S} \| \hat \imath\|_{s_1 + \sigma + \aleph_M(0)} \| h \|_{s_1 + \sigma}\,.
\end{equation}
\end{lemma}

\begin{proof}
By Lemmata \ref{composizione operatori tame AB}, \ref{lemma operatore e funzioni dipendenti da parametro} and by 
the estimates of sections \ref{sec:linearized operator}-\ref{sec:lineare}, 
the operators ${\mathcal W}_1, {\mathcal W}_2 $ are invertible and  satisfy tame estimates 
$ \|{\mathcal W}_1^{\pm 1} h \|_s^{k_0, \gamma} \leq_S
\| h \|_{s + \sigma}^{k_0, \gamma} + 
\| \fracchi_0 \|_{s + \sigma + \aleph_M (0)}^{k_0, \gamma} \| h \|_{s_0 + \sigma}^{k_0, \gamma} $ 
where $ \aleph_M(0)$ is given in Proposition \ref{Lemma finale decoupling}. 
In order to prove that ${\mathcal W}_1^\bot$ is invertible, it is sufficient to prove that $\Pi_{{\mathbb S}_0} {\mathcal W}_1 \Pi_{{\mathbb S}_0}$ is invertible. 
This follows by a perturbative argument, for $\e \gamma^{- 1}$ small, as 
in \cite{BBM-auto} using that $ \Pi_{{\mathbb S}_0} $ is a finite dimensional projector.
\end{proof}

Finally, the operator ${\mathcal L}_\omega $ defined in \eqref{Lomega def} (i.e. \eqref{representation Lom})
 is semi-conjugated to 
$$
({\mathcal W}_2^\bot)^{- 1} {\mathcal L}_\omega {\mathcal W}_1^\bot = \Pi_{{\mathbb S}_0}^\bot {\mathcal L}_M^{(2)} 
\Pi_{{\mathbb S}_0}^\bot - \Pi_{{\mathbb S}_0}^\bot  {\bf R}_M^{(2), \bot} \Pi_{{\mathbb S}_0}^\bot  + R_M 
$$
where $ \Pi_{{\mathbb S}_0}^\bot  {\bf R}_M^{(2), \bot} \Pi_{{\mathbb S}_0}^\bot$ is supported 
on the high Fourier modes and 
\begin{equation}\label{R8}
\begin{aligned}
R_M & := 
({\mathcal W}_2^\bot)^{- 1} \Pi_{{\mathbb S}_0}^\bot  
\big( {\mathcal W}_2 \Pi_{{\mathbb S}_0} {\mathcal L}_M^{(2)} \Pi_{{\mathbb S}_0}^\bot - 
 {\mathcal W}_2 \Pi_{{\mathbb S}_0}{\bf R}_M^{(2), \bot} \Pi_{{\mathbb S}_0}^\bot \\ 
 & \qquad \qquad \qquad - {\mathcal L} \Pi_{{\mathbb S}_0}{\mathcal W}_1 \Pi_{{\mathbb S}_0}^\bot - {\mathcal W}_2 {\bf R}_{\pi_0} \Pi_{{\mathbb S}_0}^\bot  +  
\e  R {\mathcal W}_1^\bot\big) 
\end{aligned}
\end{equation}
is a finite dimensional operator. 

\begin{lemma}\label{lemma forma buona resto}
The operator $R_M$  has the finite dimensional form \eqref{forma buona resto}-\eqref{stime gj chij}. 
\end{lemma}

\begin{proof}
We analyze the term 
$({\mathcal W}_2^\bot)^{- 1} R {\mathcal W}_1^\bot$ in \eqref{R8}. The others are similar. 
Since 
$R$ has the form \eqref{forma buona resto}, 
it is sufficient to prove that, given $ {\mathcal R} : h \to (h, g)_{L^2_x} \chi $, 
the operator $({\mathcal W}_2^\bot)^{- 1} {\mathcal R} {\mathcal W}_1^\bot$ 
has the form \eqref{forma buona resto} as well. We use the following  property: given 
a scalar function $ a : \T^\nu \to \C $ and $  \chi := \chi(\vphi, \cdot) \in H_{{\mathbb S}_0}^\bot $, we have 
\begin{equation}\label{proprieta trasparenza vphi}
({\mathcal W}_i^\bot)^{\pm 1}[a(\vphi) \chi ] = ({\mathcal P}^{\pm 1} a)(\vphi)  ({\mathcal W}_i^\bot)^{\pm 1} [\chi ] \, . 
\end{equation}
Let us prove \eqref{proprieta trasparenza vphi}  for $ {\mathcal W}_2^\bot $. We write (recall \eqref{Phi 1 Phi 2 proiettate} and
\eqref{semiconiugio cal L8}) 
$$
{\mathcal W}_2^\bot = \Pi_{{\mathbb S}_0}^\bot \big( {\bf \Gamma}_1 {\mathcal P}{\mathbb I}_2 \rho  {\bf \Gamma}_2 \big) \Pi_{{\mathbb S}_0}^\bot \qquad {\rm where} \qquad 
{\bf \Gamma}_1 := \mZ \mB {\mathcal Q}  \mS\,, \quad {\bf \Gamma}_2 :=   {\bf \Phi}_M {\bf \Phi}^{- 1} {\bf V}  \, , 
$$
are,  for any $\vphi \in \T^\nu $, linear operators
$ {\bf \Gamma}_i(\vphi) : H_{{\mathbb S}_0}^\bot \to H_{{\mathbb S}_0}^\bot $ of the phase space. Then 
\be\label{proprieta trasparenza vphi W1} 
\begin{aligned}
{\mathcal W}_2^\bot [a(\vphi) \chi ]  
& =  \Pi_{{\mathbb S}_0}^\bot \big( {\bf \Gamma}_1  {\mathcal P}{\mathbb I}_2 \rho {\bf \Gamma}_2 \big) 
\Pi_{{\mathbb S}_0}^\bot[a(\vphi) \chi ] \\
& = \Pi_{{\mathbb S}_0}^\bot {\bf \Gamma}_1  
{\mathcal P}{\mathbb I}_2 [a (\vphi)  \rho {\bf \Gamma}_2 \Pi_{{\mathbb S}_0}^\bot[ \chi ] ] \\
& = \Pi_{{\mathbb S}_0}^\bot {\bf \Gamma}_1 [({\mathcal P} a) (\vphi)  ({\mathcal P}{\mathbb I}_2 \rho {\bf \Gamma}_2 \Pi_{{\mathbb S}_0}^\bot[\chi] )]  \\
& = ({\mathcal P} a )(\vphi)  \Pi_{{\mathbb S}_0}^\bot {\bf \Gamma}_1 {\mathcal P}{\mathbb I}_2 
\rho {\bf \Gamma}_2 \Pi_{{\mathbb S}_0}^\bot [\chi ]  = ({\mathcal P} a) (\vphi) {\mathcal W}_2^\bot [\chi]\,. 
\end{aligned}
\ee
Then  \eqref{proprieta trasparenza vphi} follows also 
for $({\mathcal W}_2^\bot)^{- 1}$. 
Denoting $ \tilde a := {\mathcal P}^{- 1} a $ and $ \tilde \chi := ({\mathcal W}_2^\bot)^{- 1}[\chi ] $, we have  
$$
\begin{aligned}
({\mathcal W}_2^\bot)^{- 1}[a(\vphi)  \chi ] 
& = ({\mathcal W}_2^\bot)^{- 1} [({\mathcal P} \tilde a)(\vphi)
 ({\mathcal W}_2^\bot \tilde \chi)] \\
&  \stackrel{\eqref{proprieta trasparenza vphi W1}}{=} ({\mathcal W}_2^\bot)^{- 1} {\mathcal W}_2^\bot [\tilde a
 (\vphi)  \tilde \chi] = ({\mathcal P}^{- 1}a)(\vphi)  ({\mathcal W}_2^\bot)^{- 1}[\chi] \, .  
 \end{aligned}
$$ 
Now for any $h(\vphi, \cdot) \in H_{{\mathbb S}_0}^\bot $ one has 
\be\label{FBR-1}
({\mathcal W}_2^\bot)^{- 1} {\mathcal R} {\mathcal W}_1^\bot [h ]  = ({\mathcal W}_2^\bot)^{- 1} \big[ ( {\mathcal W}_1^\bot [h ], g)_{L^2_x} \chi  \big] 
\stackrel{\eqref{proprieta trasparenza vphi}} = \big({\mathcal P}^{-1} ( {\mathcal W}_1^\bot [h ], g)_{L^2_x} \big)   \chi_*
\ee
with $  \chi_* := ({\mathcal W}_2^\bot)^{- 1}[ \chi ] $ and
\begin{align}
 {\mathcal P}^{-1} ( {\mathcal W}_1^\bot [h ], g)_{L^2_x} & = 
{\mathcal P}^{-1} 
( \Pi_{{\mathbb S}_0}^\bot {\bf \Gamma}_1 {\mathcal P}{\mathbb I}_2  {\bf \Gamma}_2 \Pi_{{\mathbb S}_0}^\bot [h ], g)_{L^2_x} \nonumber \\
& =
{\mathcal P}^{-1} 
(  {\mathcal P}{\mathbb I}_2  {\bf \Gamma}_2 \Pi_{{\mathbb S}_0}^\bot [h ],  {\bf \Gamma}_1^* \Pi_{{\mathbb S}_0}^\bot g)_{L^2_x} \nonumber
 \\
& = 
(   {\bf \Gamma}_2 \Pi_{{\mathbb S}_0}^\bot [h ],  {\mathcal P}^{-1}  {\bf \Gamma}_1^* \Pi_{{\mathbb S}_0}^\bot g)_{L^2_x} 
\nonumber \\
& =
(   h ,   \Pi_{{\mathbb S}_0}^\bot {\bf \Gamma}_2^* {\mathcal P}^{-1}  {\bf \Gamma}_1^* \Pi_{{\mathbb S}_0}^\bot g)_{L^2_x} 
=(h,  g_* )_{L^2_x} \label{coefficiente in phi 2} 
\end{align}
with  
$ g_* := \Pi_{{\mathbb S}_0}^\bot {\bf \Gamma}_2^* {\mathcal P}^{-1}  {\bf \Gamma}_1^* \Pi_{{\mathbb S}_0}^\bot g $. By  
\eqref{FBR-1} and \eqref{coefficiente in phi 2} the lemma follows. 
\end{proof}

In conclusion we write 
\begin{equation}\label{final conjugation prima del KAM}
\begin{aligned}
& \qquad \quad   {\mathcal L}_\omega ={\mathcal W}_2^\bot  
{\mathcal L}_M^{(3)} ({\mathcal W}_1^\bot)^{-1} + {\bf R}_M^{(3), \bot} \, , \\ 
& 
 {\mathcal L}_M^{(3)} :=
 {\mathcal L}_M^{(2)}  + R_M \,, \quad {\bf R}_M^{(3), \bot} := - {\mathcal W}_2^\bot {\bf R}_M^{(2), \bot} ({\mathcal W}_1^{\bot})^{- 1}
 \end{aligned}
\end{equation}
where ${\mathcal L}_M^{(2)}$ is defined in \eqref{cal LN (3)}, $ {\bf R}_M^{(2), \bot}$ is defined in \eqref{bf RN (3) bot-new} and $R_M$ in \eqref{R8}. The remainder ${\bf R}^{(3), \bot}_M $ 
satisfies  tame estimates:  there is $\sigma := \sigma(\tau, \nu, k_0) > 0$ such that
\begin{align}
\| {\bf R}_M^{(3), \bot} h \|_{s_0}^{k_0, \gamma} & \leq_S \e K_n^{- b} 
\big( \| h \|_{s_0 + \sigma + b}^{k_0, \gamma}  + \| \fracchi_0 \|_{s_0 + \sigma + \aleph_M(0) 
+ b}^{k_0, \gamma} \| h \|_{s_0 + \sigma}^{k_0, \gamma} \big)\,, \  \forall b > 0\,, \label{stima bf R N (3) bot bassa} \\
\| {\bf R}_M^{(3), \bot} h\|_s^{k_0, \gamma} & \leq_S \e \big( \| h \|_{s + \sigma}^{k_0, \gamma} + \| \fracchi_0 \|_{s + \sigma + \aleph_M(0)}^{k_0, \gamma} \| h \|_{s_0 + \sigma}^{k_0, \gamma} \big)\, , \quad \forall s_0 \leq s \leq S\,. \label{stima bf R N (3) bot alta}
\end{align}
The estimates \eqref{stima bf R N (3) bot bassa}, \eqref{stima bf R N (3) bot alta} follow by \eqref{final conjugation prima del KAM}, \eqref{bf RN (3) bot}, \eqref{cal RN bot (3)}, \eqref{bf RN (1) bot}, \eqref{resto modi alti lambda 1}, using the estimates \eqref{stima m3(vphi)}, \eqref{stima coefficienti a egorov}, \eqref{stima Phi 1 Phi 2 proiettate}, \eqref{estimate cal M}, \eqref{definizione bf PhiN},
\eqref{smoothing-u1}, Lemma \ref{lambda1 - lambda Kn} and Proposition \ref{flussoCk0}. 

\begin{proposition}\label{prop: sintesi linearized}
Assume \eqref{ansatz I delta}. 
For all $ (\om, \kappa) \in \DC_{K_n}^\g \times [\kappa_1, \kappa_2] $ (see \eqref{omega diofanteo troncato})  the operator
 $ {\mathcal L}_\omega $ defined in \eqref{Lomega def} (i.e. \eqref{representation Lom}) is semiconjugated
 to the real, even and reversible operator $ {\mathcal L}_M^{(3)} $ in \eqref{final conjugation prima del KAM}  
 up to the remainder $ {\bf R}_M^{(3), \bot} $ which satisfies 
 \eqref{stima bf R N (3) bot bassa}-\eqref{stima bf R N (3) bot alta}.  
The operator  
\begin{equation}\label{forma esplicita cal LN (4)}
{\mathcal L}_M^{(3)} = \Pi_{{\mathbb S}_0}^\bot \big( \Dom {\mathbb I}_2 + \ii \mathtt m_3 {\bf T}(D) + \ii {\mathtt m}_1 
{\bf \Sigma} |D|^{\frac12} + {\bf R}_M^{(3)} + {\bf Q}_M^{(3)} \big) \Pi_{{\mathbb S}_0}^\bot 
\end{equation}
where the constant coefficients  $ \mathtt m_3 := \mathtt m_3 (\om, \kappa) \in \R $, $  {\mathtt m}_1 := {\mathtt m}_1 (\om, \kappa) \in \R  $, 
are defined in \eqref{lambda3 formula}, \eqref{lambda 1 senza proiettore} for all $ (\om, \kappa) \in \R^\nu \times [\kappa_1, \kappa_2]$, and 
satisfy \eqref{stima lambda 3 - 1 nuova}, \eqref{stime lambda 1 senza proiettore}. 
The operator $ {\bf T}(D) $ is defined in \eqref{definizione T3 A1}, \eqref{def T} and the matrix  $ 
 {\bf \Sigma} := \begin{pmatrix}
  1 & 0 \\
  0 & - 1
  \end{pmatrix}
  $. The remainders 
\be\label{RM3QM3}
{\bf R}_M^{(3)} := \begin{pmatrix}
{\mathcal R}_M^{(3)} & 0 \\
0 & \overline {\mathcal R}_M^{(3)}
\end{pmatrix}\,, \quad {\bf Q}_M^{(3)} := \begin{pmatrix}
0  & {\mathcal Q}_M^{(3)} \\
\overline{\mathcal Q}_M^{(3)} & 0
\end{pmatrix} 
\ee
satisfy the following tame properties:  for all $ \b \in \N $, $ \beta + k_0 + 4  \leq M  $,  the operators 
 $ \partial_{\vphi_j}^\beta {\mathcal R}_M^{(3)} $,
 $ \partial_{\vphi_j}^\beta [{\mathcal R}_M^{(3)}, \partial_x ] $,
 $ \partial_{\vphi_j}^\beta {\mathcal Q}_M^{(3)} $,
 $ \partial_{\vphi_j}^\beta [{\mathcal Q}_M^{(3)}, \partial_x ] $, $ j = 1, \ldots, \nu $, 
 are ${\mathcal D}^{k_0}$-tame and their tame constants satisfy, for all $s_0 \leq s \leq S $,  
 \begin{equation}\label{stima finale resti prima del KAM}
 \begin{aligned}
& \max_{{\mathcal R} \in \{ {\mathcal R}_M^{(3)}, {\mathcal Q}_M^{(3)}  \}}
\big\{  {\mathfrak M}_{\partial_{\vphi_j}^\beta {\mathcal R} } (s)\,,
{\mathfrak M}_{\partial_{\vphi_j}^\beta [{\mathcal R}, \partial_x ] } (s) \big\} \\ 
& \leq_{M, S} \e \gamma^{- 1}  
 \big( 1 + \| \fracchi_0 \|_{s + \sigma + \frac32 M + \daleth(M + 2) + \aleph_M(0)
+ \beta}^{k_0, \gamma} \big) 
\end{aligned}
\end{equation}
for some $\sigma := \sigma(\tau, \nu, k_0) > 0$ where the constant $ \aleph_M(0)$, $\daleth(M)$ 
are defined in \eqref{definizione kn(alpha)}, \eqref{perdita pseudo-diff RN}. 

\noindent
Moreover if the constant $ \mu $ in \eqref{ansatz I delta} satisfies 
\be\label{mu-grande}
s_1 + \sigma  +  \chi M + \daleth(M + 2) + \aleph_M(0)
+M - k_0 - 4 \leq s_0 + \mu \,,
\ee
then each ${\mathcal R} \in \{ {\mathcal R}_M^{(3)}, {\mathcal Q}_M^{(3)}  \}$ satisfies,  for all $ \b \in \N $, $ \beta + k_0 + 4  \leq M  $, 
\begin{equation}\label{derivate i resti prima della riducibilita}
\begin{aligned}
& \|  \partial_{\vphi_j}^\beta [ \partial_i {\mathcal R} [\hat \imath], \partial_x]  \|_{{\mathcal L}(H^{s_1})}, 
\|  \partial_{\vphi_j}^\beta  \partial_i {\mathcal R} [\hat \imath]   \|_{{\mathcal L}(H^{s_1})}   \\
& \leq_{M, S} \e \gamma^{- 1}
 \| \hat \imath\|_{s_1 + \sigma  +  \frac32 M + \daleth(M + 2) + \aleph_M(0) + \beta }\,.
\end{aligned}
\end{equation}
\end{proposition}

\begin{proof}
Note that the coefficients $ \mathtt m_3 $, $ {\mathtt m}_1 $ in  \eqref{lambda3 formula}, \eqref{lambda 1 senza proiettore} 
are actually defined for all the parameters $ (\om, \kappa) \in \R^\nu \times [\kappa_1, \kappa_2] $ 
since the approximate solution $ ( \eta, \psi) $ is defined for all $ (\om, \kappa) \in \R^\nu \times [\kappa_1, \kappa_2]$ at each step of 
the Nash-Moser iteration in section \ref{sec:NM},  
see the extension Lemma \ref{lemma:extension torus}. 

By \eqref{final conjugation prima del KAM}, \eqref{cal LN (3)} and Lemma \ref{lemma:6.27}, it is enough to prove 
the estimates \eqref{stima finale resti prima del KAM}, \eqref{derivate i resti prima della riducibilita}  for the 
operator $R_M$ defined in \eqref{R8}.
We estimate the term $ ({\mathcal W}_2^\bot)^{- 1} \Pi_{{\mathbb S}_0}^\bot {\mathcal W}_2 {\bf R}_{\pi_0} \Pi_{{\mathbb S}_0}^\bot$,
the others are  analogous. 
By \eqref{pezzo con proiettore}, setting 
$$
{\bf \Gamma}_2 :=  {\bf \Phi}_M {\bf \Phi}^{- 1} {\bf V}\, , \quad 
{\bf \Gamma}_3 := ({\mathcal W}_2^\bot)^{- 1} \Pi_{{\mathbb S}_0}^\bot {\mathcal W}_2 {\bf V}^{- 1} {\bf \Phi} {\bf \Phi}_M^{- 1} \rho^{- 1}\,, 
$$
and recalling \eqref{QP-repa} we write
$$
\begin{aligned}
 ({\mathcal W}_2^\bot)^{- 1} \Pi_{{\mathbb S}_0}^\bot {\mathcal W}_2 {\bf R}_{\pi_0} \Pi_{{\mathbb S}_0}^\bot & = 
{\bf \Gamma}_3  (\ii {\mathtt m}_3 \Pi_0) {\bf \Gamma}_2 \Pi_{{\mathbb S}_0}^\bot \quad {\rm where} \quad \\
& \qquad \quad {\mathtt m}_3 (\vartheta) := {\mathcal P}^{-1} m_3 (\vartheta) = m_3 ( \vartheta + \om \tilde p ( \vartheta)) \, . 
\end{aligned}
$$
Writing 
$ {\bf \Gamma}_m = \begin{pmatrix}
\Gamma_m^{(1)} & \Gamma_m^{(2)} \\
\overline{\Gamma}_m^{(2)} & \overline{\Gamma}_m^{(1)}
\end{pmatrix},  m = 2, 3 $, 
and recalling the definition \eqref{definizione Pi 0} of $\Pi_0$ and using that $\Pi_0 \Pi_{{\mathbb S}_0}^\bot = 0$,   we get 
\begin{align*}
{\bf R} :=  &   {\bf \Gamma}_3 (\ii {\mathtt m}_3 \Pi_0) {\bf \Gamma}_2  \Pi_{{\mathbb S}_0}^\bot = {\bf \Gamma}_3 (\ii {\mathtt m}_3 \Pi_0) \big({\bf \Gamma}_2 - {\rm Id}  \big)\Pi_{{\mathbb S}_0}^\bot\,
\end{align*}
and then for all $h \in H_{{\mathbb S}_0 }^\bot$ we get 
$$
\begin{aligned}
& {\bf R} h = \chi(\vphi, x) \big( h(\vphi, \cdot)\,,\, g(\vphi, \cdot) \big)_{L^2_x}\,, \\
&  \chi := \ii{\bf \Gamma}_3 [ {\mathtt m}_3 ] \in  H_{{\mathbb S}_0 }^\bot\,, \quad g := \Pi_{{\mathbb S}_0}^\bot({\bf \Gamma}_2 - {\rm Id})^*[1] \in  H_{{\mathbb S}_0 }^\bot\,.
\end{aligned}
$$
Lemma \ref{Lemma:trasformazioni finali W}, the estimates of sections 
\ref{sec:linearized operator}-\ref{sec:lineare} and of Propositions \ref{teorema stime aggiunto flusso}, \ref{teorema derivata aggiunto flusso} imply that for some $\sigma := \sigma(k_0, \tau, \nu) > 0 $, for all $ s \in [s_0, S ] $,  
\begin{align*}
& \| g \|_s^{k_0, \gamma} \leq_{S, M} \e \gamma^{- 1} (1 + \| \fracchi_0\|_{s + \aleph_M(0) + \sigma }^{k_0, \gamma} )\,, \quad \| \chi\|_s^{k_0, \gamma} \leq_{S, M} 1 + \| \fracchi_0\|_{s + \aleph_M(0) + \sigma }^{k_0, \gamma}\,, \\
& 
\| \partial_i g[\widehat \imath] \|_{s_1} \leq_{S, M} \e \gamma^{- 1} \| \widehat \imath\|_{s_1 + \aleph_M(0) + \sigma }\,, \quad \| \partial_i \chi[\widehat \imath]\|_{s_1} \leq_{S, M} \| \widehat \imath\|_{s_1 + \aleph_M(0) + \sigma },
\end{align*}
provided \eqref{mu-grande} is satisfied. Then the estimates \eqref{stima finale resti prima del KAM}, \eqref{derivate i resti prima della riducibilita} for the operator 
${\mathcal R} $ follow since for all $ j = 1, \ldots, \nu $,  $ \beta \in \N $, $ k \in \N^{\nu + 1} $, 
$$
\partial_{\vphi_j}^\beta \partial_\lambda^k [{\bf R}, \partial_x] h  =  - \!\! \!\!  \!\!  \!\! \!\! \sum_{
\beta_1 + \beta_2 = \beta, k_1 + k_2 = k}   \!\! \!\! \!\! \!\!
( \partial_\lambda^{k_1} \partial_{\vphi_j}^{\beta_1} \chi \big( h, \partial_{\lambda}^{k_2} \partial_{\vphi_j}^{\beta_2} g_x)_{L^2_x} + \partial_\lambda^{k_1} \partial_{\vphi_j}^{\beta_1} \chi_x ( h, \partial_{\lambda}^{k_2} \partial_{\vphi_j}^{\beta_2} g)_{L^2_x}\big)
$$
and the operators $\partial_{\vphi_j}^\beta \partial_\lambda^k {\bf R}$, $\partial_{\vphi_j}^\beta [\partial_i {\bf R}[\widehat \imath], \partial_x]$, $\partial_{\vphi_j}^\beta \partial_i {\bf R}[\widehat \imath]$ have similar expressions. 
\end{proof}

In the next section we diagonalize the operator $ {\mathcal L}_M^{(3) } $.
We neglect the term ${\bf R}_M^{(3), \bot}$ in \eqref{final conjugation prima del KAM}, which will contribute to the remainders in 
\eqref{stima R omega bot corsivo bassa}-\eqref{stima R omega bot corsivo alta}.

\chapter{Almost  diagonalization and invertibility of $ {\mathcal L}_\om $ } \label{sec: reducibility}

We have a  linear real operator acting on $ H_{{\mathbb S}_0}^\bot $, 
\be\label{defL0-red}
{\bf L}_0 := {\bf L}_0 (i) := \Dom  {\mathbb I}_2^\bot  + \ii {\bf D}_0 + {\bf R}_0 + {\bf Q}_0 \, , \qquad  
{\mathbb I}_2^\bot :=  {\mathbb I}_2 \Pi_{{\mathbb S}_0}^\bot \, , 
\ee
defined for all $ (\om, \kappa) \in \DC_{K_n}^\g \times [\kappa_1, \kappa_2] $(see \eqref{omega diofanteo troncato}), with 
diagonal part (with respect to the exponential basis)
\be\label{op-diago0}
\begin{aligned}
& \qquad \qquad \qquad {\bf D}_0 := \begin{pmatrix}
{\mathcal D}_0 & 0 \\
0 & - {\mathcal D}_0
\end{pmatrix}\,, \\ 
&  {\mathcal D}_0  := {\rm diag}_{j \in {\mathbb S}_0^c} \mu_j^{(0)} \,, \quad 
\mu_j^{(0)} :=\mathtt m_3 |j|^{\frac12} (1 + \kappa |j|^2)^{\frac12} + {\mathtt m}_1 |j|^{\frac12}\,, 
\end{aligned}
\ee
where ${\mathbb S}_0^c := \Z \setminus {\mathbb S}_0 $ (see \eqref{def:S0}), 
$ \mathtt m_3 := \mathtt m_3 (\om, \kappa) \in \R $, $ {\mathtt m}_1 := {\mathtt m}_1 (\om, \kappa) \in \R $ 
are defined for all $ (\om, \kappa) \in \R^\nu \times [\kappa_1, \kappa_2] $,  and
\begin{align}   \label{defRQ0}
{\bf R}_0, {\bf Q}_0 : H_{{\mathbb S}_0}^\bot \to H_{{\mathbb S}_0}^\bot\,, \qquad
{\bf R}_0 := \begin{pmatrix}
{\mathcal R}_0 & 0 \\
0 & \overline{\mathcal R}_0
\end{pmatrix}\,, \quad 
{\bf Q}_0 := \begin{pmatrix}
0 & {\mathcal Q}_0 \\
\overline{\mathcal Q}_0 & 0
\end{pmatrix}  
\end{align}
are real, even and reversible. The operators  $ {\bf R}_0 $, $ {\bf Q}_0  $ satisfy also the following tame estimates:
\begin{itemize}
\item 
{\bf (Smallness assumption on $  {\bf R}_0 $ and  $ {\bf Q}_0 $).} 
{\it The operators 
$$ 
\begin{aligned}
& {\mathcal R}_0 \, ,  \ [ {\mathcal R}_0, \pa_x ] \, , \ \partial_{\vphi_m}^{s_0} {\mathcal R}_0 \, , \
\partial_{\vphi_m}^{s_0}  [ {\mathcal R}_0, \pa_x ] \, , \\
& {\mathcal Q}_0, \ [ {\mathcal Q}_0, \pa_x ] \, , \  \partial_{\vphi_m}^{s_0} {\mathcal Q}_0 \, ,  
\ \partial_{\vphi_m}^{s_0}  [ {\mathcal Q}_0, \pa_x ] \, , \
\forall m =1, \ldots, \CS \, , 
\end{aligned}
$$ 
are $ {\mathcal D}^{k_0} $-tame with tame constants, 
defined for all  $ s_0 \leq s \leq S $, 
\be\label{tame cal R0 cal Q0}
\begin{aligned}
{\mathbb M}_0 (s) 
& := \max \Big\{   
{\mathfrak M}_{{\mathcal R} }(s),
{\mathfrak M}_{[ {\mathcal R}, \pa_x ]  }(s),
{\mathfrak M}_{ \partial_{\vphi_m}^{s_0} {\mathcal R} }(s),
{\mathfrak M}_{ \partial_{\vphi_m}^{s_0}  [ {\mathcal R}, \pa_x ]  }(s) \\
& \qquad \qquad m=1, \ldots, \CS, {\mathcal R} \in \{ {\mathcal R}_0, {\mathcal Q}_0 \} \Big\} \, . 
 \end{aligned}
\ee
In addition the operators 
$$ 
\partial_{\vphi_m}^{s_0 +  {\mathtt b}} {\mathcal R}_0, 
\partial_{\vphi_m}^{s_0 +  {\mathtt b}}  [ {\mathcal R}_0, \pa_x ], 
 \partial_{\vphi_m}^{s_0 +  {\mathtt b}} {\mathcal Q}_0, 
 \partial_{\vphi_m}^{s_0 +  {\mathtt b}}  [ {\mathcal Q}_0, \pa_x ],  m =1, \ldots, \CS \, , 
 $$ 
are $ {\mathcal D}^{k_0} $-tame with tame constants, defined for all $ s_0 \leq s \leq S $,    
\be\label{tame norma alta cal R0 cal Q0}
{\mathbb M}_0(s, {\mathtt b})  := \max_{m=1, \ldots, \CS, {\mathcal R} \in \{ {\mathcal R}_0, {\mathcal Q}_0 \}} \big\{ 
{\mathfrak M}_{ \partial_{\vphi_m}^{s_0 +  {\mathtt b}} {\mathcal R} }(s),
{\mathfrak M}_{ \partial_{\vphi_m}^{s_0 +  {\mathtt b}}  [ {\mathcal R}, \pa_x ]  }(s)  \big\} 
\ee}
{\it where  ${\mathtt b} \in \N  $ satisfies}
\begin{equation}\label{alpha beta}
{\mathtt b} := [{\mathtt a}] + 2 \in \N \,,\ \ {\mathtt a} := 3 \tau_1 \, , 
\ \  \chi = 3 / 2 \, ,  \ \  \tau_1 := \tau + (\tau + 1) k_0 \, .
\end{equation}
{\it We assume that the tame constants satisfy}
\be\label{def:costanti iniziali tame}
{\mathfrak M}_0 (s_0, {\mathtt b}) := {\rm max}\{ {\mathbb M}_0 (s_0), {\mathbb M}_0(s_0, {\mathtt b}) \} \leq C(S) \e \gamma^{- 1}  \, 
\ee 
{\it and moreover, there is $ \sigma(\mathtt b) > 0 $ 
(we take $ \sigma (\mathtt b) = \mu (\mathtt b) + \sigma $ in Lemma \ref{lem:tame iniziale}), 
such that, for all $ m = 1, \ldots, \CS $, $ \b \in \N $, $ \b    \leq 
{\mathtt b}  + s_0  $, 
\be\label{derivate i resti prima della riducibilita-s0}
\begin{aligned}
\max_{{\mathcal R} \in \{ {\mathcal R}_0, {\mathcal Q}_0 \} }
\big\{ \|  \partial_{\vphi_m}^\beta  \partial_i {\mathcal R} [\hat \imath]   \|_{{\mathcal L}(H^{s_0})} , \, & 
\|  \partial_{\vphi_m}^\beta [ \partial_i {\mathcal R} [\hat \imath], \partial_x]  \|_{{\mathcal L}(H^{s_0})} \big\} \\
& \leq C(S) \e \gamma^{- 1} 
 \| \hat \imath\|_{s_0 + \sigma(\mathtt b)} \, .  
 \end{aligned}
\ee}
\end{itemize}

In this section we use  $ \CS $ to denote the cardinality of the set of tangential sites
$ {\mathbb S}^+ $
(and thus the number of  components of the frequency vector $ \omega $) 
that elsewhere is denoted simply by $ \nu = \CS $. 

\begin{remark}\label{remark:ab}
The conditions $ {\mathtt b} >  {\mathtt a} + \chi^{-1} $ and 
$ {\mathtt a} > 3 \tau_1 = \tau_1 \chi  \slash (2 - \chi)  $ arise for the convergence of the 
iterative scheme \eqref{schema quadratico tame}-\eqref{M+Ms}, see  Lemma \ref{stima M nu + 1 K nu + 1}.  
We take an integer $ {\mathtt b} := [{\mathtt a}] + 2 \in \N $ so that 
$ \pa_{\vphi_m}^{s_0+ {\mathtt b}} $ are differential operators
(recall also that $ s_0 \in \N $ by \eqref{def:s0}). 
Note also that 
$ {\mathtt a} > \chi k_0  (\tau + 2 ) +1 $ (as $ \tau \geq 1 $) which is used in the extension procedure in 
${\bf(S2)_{\nu}}$, see e.g. \eqref{vicinanza autovalori estesi}. Moreover  
$ {\mathtt a} > \chi (\tau + k_0  (\tau + 2 )) $ which is used in Lemma \ref{lemma inclusione cantor riccardo 2}. 
\end{remark}

Proposition \ref{prop: sintesi linearized} implies that the operators 
 $  {\bf R}_M^{(3)} $, $ {\bf Q}_M^{(3)} $ in \eqref{RM3QM3}
satisfy the above tame estimates 
by fixing the constant  $ M $ in section  \ref{sec:decoupling} large enough 
(this means to perform sufficiently many regularizing steps  in Proposition  \ref{Lemma finale decoupling}), namely
\begin{equation}\label{relazione mathtt b N}
M := {\mathtt b}  + s_0 + k_0 + 4 \, . 
\end{equation}
Set  (recall \eqref{perdita pseudo-diff RN}, \eqref{definizione kn(alpha)}) 
 \begin{equation}\label{definizione bf c (beta)}
 \begin{aligned}
& {\bf c}({\mathtt b}) := \chi ({\mathtt b} +s_0 + k_0 + 4) + \daleth ( {\mathtt b} + s_0  + k_0 + 6) + 
\aleph_{{\mathtt b} + s_0 + k_0 + 4}(0)\, , \\
& \mu (\mathtt b) := s_0 + {\bf c}(\mathtt b) + \mathtt b \, . 
\end{aligned}
\end{equation}

\begin{lemma} {\bf (Tame estimates of $ {\bf R}_M^{(3)} $, $ {\bf Q}_M^{(3)} $)} \label{lem:tame iniziale}
Assume \eqref{ansatz I delta} with $ \mu \geq \mu (\mathtt b) + \sigma $.  
Then the operators $  {\bf R}_0 := {\bf R}_M^{(3)} $, $ {\bf Q}_0 := {\bf Q}_M^{(3)} $ in \eqref{RM3QM3}
satisfy, for all $ s_0 \leq s \leq S $, the tame estimates 
\eqref{tame cal R0 cal Q0}-\eqref{tame norma alta cal R0 cal Q0} with 
\begin{equation}\label{costanti resto iniziale}
\begin{aligned}
& {\mathbb M}_0(s) \leq_S \e \gamma^{- 1} \big( 1 + \| \fracchi_0 \|_{s + s_0 + \sigma +  {\bf c}({\mathtt b})  }^{k_0, \gamma} \big) \,,  \\ 
&  {\mathbb M}_0(s, \mathtt b) \leq_S  \e \gamma^{- 1}
\big( 1 + \| \fracchi_0 \|_{s +  \mu (\mathtt b)  + \sigma}^{k_0, \gamma} \big) 
\end{aligned}
\end{equation} 
and \eqref{def:costanti iniziali tame} holds. Moreover, for all $ m = 1, \ldots, \CS $, $ \b \in \N $, $ \b    \leq 
{\mathtt b}  + s_0  $, the operators $\partial_{\vphi_m}^\beta  \partial_i {\mathcal R} [\hat \imath]$, $\partial_{\vphi_m}^\beta [ \partial_i {\mathcal R} [\hat \imath], \partial_x]$, ${\mathcal R} \in \{{\mathcal R}_0, {\mathcal Q}_0 \}$ satisfy the bounds \eqref{derivate i resti prima della riducibilita-s0} with $\sigma (\mathtt b) = \mu (\mathtt b) + \sigma$. 
\end{lemma} 

\begin{proof}
The estimates \eqref{costanti resto iniziale} follow by \eqref{stima finale resti prima del KAM} and by the definitions \eqref{relazione mathtt b N}, \eqref{definizione bf c (beta)}. 
Moreover with the choice of $ \mu := \mu ({\mathtt b}) + \sigma $ in \eqref{definizione bf c (beta)} (see also \eqref{relazione mathtt b N}) 
the condition \eqref{mu-grande} holds with $ s_1 = s_0 $ and so \eqref{derivate i resti prima della riducibilita-s0}
holds by \eqref{derivate i resti prima della riducibilita}, with $\sigma(\mathtt b) = \mu(\mathtt b) + \sigma$.  
\end{proof}

By \eqref{costanti resto iniziale}, \eqref{definizione bf c (beta)}, 
we have verified that, for all $s_0 \leq s \leq S $,  
\begin{equation}\label{stima cal K0 I delta}
{\mathfrak M}_0(s, {\mathtt b}) := {\rm max}\{ {\mathbb M}_0 (s), {\mathbb M}_0(s, {\mathtt b}) \}
 \leq_{S} \e \gamma^{- 1} ( 1 + \| \fracchi_0 \|^{k_0, \gamma}_{s + \mu ({\mathtt b}) + \sigma}) \,.
\end{equation}
We perform the almost reducibility of $ {\bf L}_0 $ along the scale 
\be\label{def Nn}
N_{- 1} := 1\,,\quad N_\nu := N_0^{\chi^\nu}\,,\quad \forall \nu \geq 0\,,\quad \chi := 3/2\, ,
\ee
requiring inductively at each step the second order Melnikov non-resonance conditions in
\eqref{Omega nu + 1 gamma}. 

\begin{theorem}\label{ITERAZIONERIDUCIBILITA}
{\bf (Almost reducibility)} 
There exists $ \tau_0 := \tau_0 (\t, \CS ) > 0 $ such that, for all $ S > s_0 $, 
there is $ N_0 := N_0 (S, {\mathtt b})  \in \N$ such that, if 
\begin{equation}\label{KAM smallness condition1}
N_0^{\tau_0}  {\mathfrak M}_0(s_0, {\mathtt b}) \gamma^{- 1} \leq 1\,, 
\end{equation}
(see \eqref{def:costanti iniziali tame}), then, for all $ n \in \N $, $\nu = 0, 1 , \ldots, n$: 
\begin{itemize}
\item[${\bf(S1)_{\nu}}$] There exists a real, even and reversible operator 
\begin{equation}\label{cal L nu}
\begin{aligned}
& {\bf L}_\nu := \Dom {\mathbb I}_2^\bot + \ii {\bf D}_\nu + {\bf R}_\nu + {\bf Q}_\nu\,, \\
& {\bf D}_\nu := \begin{pmatrix}
{\mathcal D}_\nu & 0 \\
0 & - {\mathcal D}_\nu
\end{pmatrix}\,,\quad {\mathcal D}_\nu := {\rm diag}_{j \in {\mathbb S}_0^c} \mu_j^\nu\,,
\end{aligned}
\end{equation}
-which acts on the space of functions even in $ x $- 
defined for  $(\omega, \kappa) \in \DC_{K_n}^\gamma \times [\kappa_1, \kappa_2] $ for $ \nu = 0 $, and for all $ (\om, \kappa) $ in 
\be	\label{inclusione-insiemi-gamma2}
{\mathcal N}(\tLm_\nu^\gamma, \g N_{\nu-1}^{- \tau - 2}) \subset \tLm_\nu^{\gamma/2} \, , \quad {\rm for} \ \nu \geq 1 \, , 
\ee 
(recall the definition \eqref{definizione incicciottamento insieme})
where  $ \mu_j^\nu $ 
are $ k_0 $-times differentiable functions of the form 
\begin{equation}\label{mu j nu}
\mu_j^\nu(\omega, \kappa) := \mu_j^0(\omega, \kappa) + r_j^\nu(\omega, \kappa)\,,\ \  \,\mu_j^0 := 
\mathtt m_3 |j|^{\frac12} (1 + \kappa j^2)^{\frac12} + {\mathtt m}_1 |j|^{\frac12}\,, 
\end{equation}
satisfying 
\begin{equation}\label{stima rj nu}
\mu_j^\nu = \mu_{- j}^\nu\,, \quad i.e. \ r_j^\nu = r_{- j}^\nu\,, \quad |r_j^\nu|^{k_0, \gamma} \leq C(S) \e \gamma^{- 1}\,,\ \  \forall j \in {\mathbb S}_0^c\, .
\end{equation}
The sets $ \tLm_\nu^\gamma  $ are defined by
$ \tLm_0^\gamma := \tOm \times [\kappa_1, \kappa_2]$, and, for all $\nu \geq 1$, 
\be
\begin{aligned}
\tLm_\nu^\gamma  :=  \tLm_\nu^\gamma (i) & :=  
\Big\{ \lambda = (\omega, \kappa) \in \tLm_{\nu - 1}^\gamma \cap \big( [\DC_{K_n}^\gamma \cap \DC_{N_{\nu-1}}^\gamma] \times [\kappa_1, \kappa_2] \big) \, : \label{Omega nu + 1 gamma} \\ 
& \qquad |\omega \cdot \ell  + \mu_j^{\nu - 1} - \varsigma \mu_{j'}^{\nu - 1}| \geq  
\gamma |j^{\frac32} - \varsigma j'^{\frac32}| \langle \ell \rangle^{-\tau},  \\
& \qquad  \forall |\ell  | \leq N_{\nu - 1}, j, j' \in \N \setminus {\mathbb S}^+, \varsigma \in \{+, -  \} \Big\} 
\end{aligned}
\ee
(recall  
\eqref{omega diofanteo troncato}
and  that the tangential sites 
$ {\mathbb S}  = {\mathbb S}^+ \cup (- {\mathbb S}^+) \subset \Z $ with $ {\mathbb S}^+ \subset \N $). 
The remainders
\begin{equation}\label{forma cal R nu}
{\bf R}_\nu := \begin{pmatrix}
{\mathcal R}_{\nu} & 0 \\
0 & \overline {\mathcal R}_{\nu}
\end{pmatrix}\,,\qquad {\bf Q}_\nu := \begin{pmatrix}
0 & {\mathcal Q}_\nu \\
\overline{\mathcal Q}_\nu & 0
\end{pmatrix}
\end{equation}
are $ {\mathcal D}^{k_0} $-modulo-tame: 
more precisely the operators $ {\mathcal R}_\nu,  {\mathcal Q}_\nu $, respectively  
$ \langle\partial_\vphi \rangle^{\mathtt b }  {\mathcal R}_\nu, 
\langle\partial_\vphi \rangle^{\mathtt b }  {\mathcal Q}_\nu $, 
are $ {\mathcal D}^{k_0} $-modulo-tame  with modulo-tame constants respectively
\be\label{def:msharp}
\begin{aligned}
& {\mathfrak M}_\nu^\sharp (s) := \max \{ {\mathfrak M}_{{\mathcal R}_\nu}^\sharp (s), 
{\mathfrak M}_{{\mathcal Q}_\nu}^\sharp (s) \} \, , \\
& {\mathfrak M}_\nu^\sharp (s, {\mathtt b}) := 
\max\{ {\mathfrak M}_{ \langle \pa_\vphi \rangle^{\mathtt b} {\mathcal R}_\nu}^\sharp (s), 
{\mathfrak M}_{ \langle \pa_\vphi \rangle^{\mathtt b} {\mathcal Q}_\nu}^\sharp (s)  \}.
\end{aligned}
\ee 
There exists a contant $C_* (s_0, \mathtt b)$ such that  for all $s \in [s_0, S] $, 
\begin{equation}\label{stima cal R nu}
{\mathfrak M}_\nu^\sharp (s) \leq C_* (s_0, \mathtt b) {\mathfrak M}_0 (s, {\mathtt b}) N_{\nu - 1}^{- {\mathtt a}},\,
 {\mathfrak M}_\nu^\sharp ( s, \mathtt b) \leq  C_* (s_0, \mathtt b) {\mathfrak M}_0 (s, {\mathtt b}) N_{\nu - 1}\,.
\end{equation}
Moreover, for $ \nu \geq 1 $, there exists a real,  even and reversibility preserving map
\begin{equation}\label{Psi nu - 1}
 {\bf \Phi}_{\nu - 1} := {\mathbb I}_2^\bot + {\bf \Psi}_{\nu - 1}\,, \quad 
{\bf \Psi}_{\nu - 1} := \begin{pmatrix}
\Psi_{\nu - 1, 1} & \Psi_{\nu - 1, 2} \\
\overline \Psi_{\nu - 1, 2} & \overline \Psi_{\nu - 1, 1}
\end{pmatrix}\,, 
\end{equation}
such that 
\be\label{coniugionu+1}
{\bf L}_\nu := {\bf \Phi}_{\nu - 1}^{-1} {\bf L}_{\nu - 1} {\bf \Phi}_{\nu - 1}\, . 
\ee
The operators $ \Psi_{\nu - 1, m} $ and $ \langle \partial_\vphi \rangle^{\mathtt b}  \Psi_{\nu - 1, m} $,
$ m  = 1, 2 $, are $ {\mathcal D}^{k_0} $-modulo-tame with modulo-tame constants satisfying, for all $s \in [s_0, S] $,
($ \tau_1, \mathtt a $ are defined in \eqref{alpha beta})
\be\label{tame Psi nu - 1}
\begin{aligned}
& {\mathfrak M}_{\Psi_{\nu - 1, m}}^\sharp \! (s)   \leq \frac{C(k_0, s_0, \mathtt b)}{\g} 
N_{\nu - 1}^{\tau_1} N_{\nu - 2}^{- \mathtt a} {\mathfrak M}_0 (s, {\mathtt b}) \, , \\
& {\mathfrak M}_{\langle \partial_\vphi \rangle^{\mathtt b}  \Psi_{\nu - 1, m}}^\sharp \! (s) 
 \leq \frac{C(k_0, s_0, \mathtt b)}{\g} 
N_{\nu - 1}^{\tau_1} N_{\nu - 2}  {\mathfrak M}_0 (s, {\mathtt b}) \, . 
\end{aligned}
\ee
\item[${\bf(S2)_{\nu}}$] 
For all $ j \in {\mathbb S}_0^c $ there exists a ${k_0}$-times differentiable 
extension $ \tilde \mu_j^\nu : \tOm \times [\kappa_1, \kappa_2] \mapsto \R $  such that $  \tilde \mu_j^\nu =   \mu_j^\nu $   on  $ \tLm_\nu^\gamma $, 
and  
\begin{equation}\label{autovalori finali}
\begin{aligned}
& \tilde \mu_j^\nu (\omega, \kappa) 
:=   \mu_j^0(\omega, \kappa ) + \tilde r_j^\nu (\omega, \kappa) \in \R\,, \\
& \tilde r_j^\nu = \tilde r_{- j}^\nu\,, \,\,    |\tilde r^\nu_j|^{k_0, \gamma} \leq C(S)
\e \gamma^{- 1} N_0^{k_0(\tau+2)} \, , \,\, \forall j \in {\mathbb S}_0^c\,, 
\end{aligned}
\end{equation}
and  for all $\nu \geq 1$
\be\label{vicinanza autovalori estesi}
\begin{aligned}
 |\tilde \mu_j^\nu - \tilde \mu_j^{\nu - 1}|^{k_0, \gamma} & \leq  C(k_0) N_{\nu - 1}^{ k_0 (\tau + 2)} 
{\mathfrak M}_{\nu - 1}^\sharp (s_0) \\
& \leq  C(k_0, S) \e \gamma^{- 1} N_{\nu - 1}^{ k_0 (\tau + 2)} N_{\nu - 2}^{- {\mathtt a}}    \, . 
\end{aligned}
\ee
\item[${\bf(S3)_{\nu}}$] Let $ i_1(\omega, \kappa) $, $ i_2(\omega, \kappa) $ such that 
$ {\bf R}_0(i_1)$, ${\bf Q}_0(i_1)$,  ${\bf R}_0(i_2 )$, ${\bf Q}_0(i_2)$ satisfy \eqref{def:costanti iniziali tame}. Assume also \eqref{derivate i resti prima della riducibilita-s0}.
Then for all $\nu = 0, \ldots n$, for all $(\omega, \kappa) \in \tLm_\nu^{\gamma_1}(i_1) \cap \tLm_\nu^{\gamma_2}(i_2)$
with $\gamma_1, \gamma_2 \in [\gamma/2, 2 \gamma]$, there exists $\sigma := \sigma(\tau, \CS, k_0) > 0$ such that 
\be\label{stima R nu i1 i2}
\begin{aligned}
  \| |{\mathcal R}_\nu(i_1) - {\mathcal R}_\nu(i_2)| \|_{{\mathcal L}(H^{s_0})}, & 
\| |{\mathcal Q}_\nu(i_1) - {\mathcal Q}_\nu(i_2)| \|_{{\mathcal L}(H^{s_0})} \\ 
&  \leq_{S, \mathtt b}  \e \gamma^{- 1 } N_{\nu - 1}^{- \mathtt a}
 \| i_1 - i_2\|_{s_0 +  \mu(\mathtt b) + \sigma}, 
 \end{aligned}
\ee 
\be
\begin{aligned}
 \label{stima R nu i1 i2 norma alta}
 \|  |\langle \partial_\vphi \rangle^{\mathtt b}({\mathcal R}_\nu(i_1) - {\mathcal R}_\nu(i_2)) | \|_{{\mathcal L}(H^{s_0})}, & 
\|  | \langle \partial_\vphi \rangle^{\mathtt b}({\mathcal Q}_\nu(i_1) - {\mathcal Q}_\nu(i_2))| \|_{{\mathcal L}(H^{s_0})}\\
& 
 \leq_{S, \mathtt b} \frac{\e}{\gamma} N_{\nu - 1} \| i_1 - i_2\|_{ s_0 +  \mu(\mathtt b) + \sigma}\,.
 \end{aligned}
\ee
Moreover for all $\nu = 1, \ldots , n$, for all $j \in {\mathbb S}^c_0$, 
\begin{align}\label{r nu - 1 r nu i1 i2}
& \big|(r_j^\nu(i_1) - r_j^\nu(i_2)) - (r_j^{\nu - 1}(i_1) - r_j^{\nu - 1}(i_2))  \big| \leq C 
\| |{\mathcal R}_\nu(i_1) - {\mathcal R}_\nu(i_2)| \|_{{\mathcal L}(H^{s_0})} \,, \\
&  |r_j^{\nu}(i_1) - r_j^{\nu}(i_2)| \leq C(S) \e \gamma^{- 1} \| i_1 - i_2  \|_{ s_0  + \mu(\mathtt b) + \sigma}\,. \label{r nu i1 - r nu i2}
\end{align}
\item[${\bf(S4)_{\nu}}$]  Let $i_1$, $i_2$ be like in ${\bf(S3)_{\nu}}$ and $0 < \rho < \gamma/2$. Then 
$$
\e \gamma^{- 1} C(S) N_{\nu - 1}^\tau \|i_1 - i_2 \|_{ s_0 + \mu(\mathtt b) + \sigma} \leq \rho \ \ 
 \Longrightarrow \ \ 
\tLm_\nu^\gamma(i_1) \subseteq \tLm_\nu^{\gamma - \rho}(i_2) \, . 
$$
\end{itemize}
\end{theorem}

\begin{remark}\label{remark:dipendenza da omega non richiesta}
Note that \eqref{r nu - 1 r nu i1 i2}-\eqref{r nu i1 - r nu i2} are sufficient to prove $({\bf S4})_\nu$
about the inclusion of the sets $ \tLm_\nu^\gamma(i_1) $, $ \tLm_\nu^{\gamma-\rho}(i_2) $ 
corresponding to two nearby approximate solutions: a smallness condition in $ | \ |^{k_0, \gamma}$ is not required. 
This is sufficient to prove Lemma \ref{lemma inclusione cantor riccardo 1}, and thus Lemma \ref{lemma inclusione cantor riccardo 2}. 
The bounds \eqref{r nu - 1 r nu i1 i2}-\eqref{r nu i1 - r nu i2} are implied just by the estimate \eqref{stima R nu i1 i2}, which  is  in $ s_0 $ norm 
and  there is no control of the derivatives  with respect to  ($ \om , \kappa $).
This is why we do not need to estimate the  derivatives  with respect to  ($ \om , \kappa $) of the 
operators $ \pa_i {\mathcal R} $ in \eqref{derivate i resti prima della riducibilita-s0}.
\end{remark}

An important point of Theorem \ref{ITERAZIONERIDUCIBILITA} is to require only the bound \eqref{KAM smallness condition1} 
for $ {\mathfrak M}_0 (s_0, {\mathtt b}) $ in low norm, which is verified in Lemma \ref{lem:tame iniziale}, 
as well as the estimate \eqref{derivate i resti prima della riducibilita-s0} (which is still in low norm). 
On the other hand  Theorem \ref{ITERAZIONERIDUCIBILITA}  provides the smallness \eqref{stima cal R nu}  
of the tame constants $ {\mathfrak M}_\nu^\sharp (s) $ and proves that 
$ {\mathfrak M}_\nu^\sharp  (s, {\mathtt b})  $, $ \nu \geq 0 $, do not diverge too much. 
Theorem \ref{ITERAZIONERIDUCIBILITA} implies that the invertible operator
\be\label{defUn}
{\bf U}_n := {\bf \Phi}_0 \circ \ldots \circ {\bf \Phi}_n 
\ee
has almost diagonalized  $ {\bf L}_0 $, i.e. \eqref{cal L infinito} below holds. 
We have the following corollary:

\begin{theorem}\label{Teorema di riducibilita} {\bf (KAM almost-reducibility)}
Assume 
\eqref{ansatz I delta}  with $ \mu \geq   \mu (\mathtt b) + \sigma$. For all $S > s_0$ there exists $N_0 := N_0(S, \mathtt b) > 0$, $ \d_0 := \delta_0(S)  >  0 $ such that, if the smallness condition 
\begin{equation}\label{ansatz riducibilita}
N_0^{\tau_0} \e \gamma^{- 2}  \leq \d_0  
\end{equation}
holds, where the constant $ \tau_0 := \tau_0 (\tau, \CS)  $ is defined in Theorem \ref{ITERAZIONERIDUCIBILITA},
then, for all $ n \in \N$, for all $ \lambda = (\omega, \kappa) $ in 
\be\label{Cantor set}
{ \tLm}_{n + 1}^{ \gamma} := {\tLm}_{n + 1}^{ \gamma} (i)  =  \bigcap_{\nu = 0}^{n + 1 } \tLm_\nu^\gamma
\ee
where the sets $ \tLm_\nu^\gamma$ are defined in \eqref{Omega nu + 1 gamma},
the operator $ {\bf U}_n  $ in \eqref{defUn} is well defined  and 
\begin{equation}\label{cal L infinito}
{\bf  L}_n := {\mathbf  U}_n^{- 1} {\bf L}_0 {\mathbf  U}_n = 
\Dom {\mathbb I}_2^\bot + \ii {\bf D}_n + {\bf R}_n + {\bf Q}_n 
\end{equation}
where $ {\bf D}_n $ is defined in \eqref{cal L nu} and $  {\bf R}_n $, ${\bf Q}_n$
 in \eqref{forma cal R nu} (with $ \nu = n $).
The operators $ {\mathcal R}_n $, $ {\mathcal Q}_n $ are $ {\mathcal D}^{k_0} $-modulo-tame with modulo-tame constants
\be\label{stima resto operatore quasi diagonalizzato}
{\mathfrak M}^\sharp_{{\mathcal R}_n}(s) , {\mathfrak M}^\sharp_{{\mathcal Q}_n}(s)  \leq_S  \e \gamma^{- 1} N_{n - 1}^{- {\mathtt a}} (1 + 
\| \fracchi_0 \|_{s   + \mu (\mathtt b) + \sigma}^{k_0, \gamma}) \, , \qquad \forall s_0 \leq s \leq S\,.
\ee
Moreover the operators 
$ {\bf U}_{n}^{\pm 1} -  {\mathbb I}_2^\bot  $ are $ {\mathcal D}^{k_0}$-modulo-tame with modulo-tame constants 
\begin{equation}\label{stima Phi infinito}
{\mathfrak M}^\sharp_{{\bf U}_{n}^{\pm 1} -  {\mathbb I}_2^\bot} (s) \leq_S \e \g^{-2} N_0^{\tau_1} (1 + 
\| \fracchi_0 \|_{s  + \mu (\mathtt b) + \sigma}^{k_0, \gamma})\,, \quad \forall s_0 \leq s \leq S \, , 
\end{equation} 
where  $\tau_1$ is defined in \eqref{alpha beta}.  
The operators $ {\mathbf  U}_n$, $ {\mathbf  U}_n^{- 1} $ are real, even and reversibility preserving. 
$ {\bf L}_n $ is real, even and reversible. 
\end{theorem}

\begin{proof}
The assumption \eqref{KAM smallness condition1} of Theorem \ref{ITERAZIONERIDUCIBILITA} 
holds by \eqref{stima cal K0 I delta}, \eqref{ansatz I delta} with $ \mu \geq    \mu (\mathtt b) + \sigma $,  
and \eqref{ansatz riducibilita}.
The estimate \eqref{stima resto operatore quasi diagonalizzato} follows by 
 \eqref{stima cal R nu} (for $\nu = n$) and \eqref{stima cal K0 I delta}. 
It remains to prove \eqref{stima Phi infinito}.   
By Lemma \ref{interpolazione moduli parametri}  the composition of $ {\mathcal D}^{k_0} $-modulo-tame operators is 
$ {\mathcal D}^{k_0} $-modulo-tame.
To estimate the modulo-tame constant $ {\mathfrak M}_{{\bf U}_{\nu+1}}^\sharp(s)$ 
of 
$ {\bf U}_{\nu + 1} = {\bf U}_\nu \circ {\bf \Phi}_{\nu + 1} = $ $  {\bf U}_\nu \circ ({\mathbb I}_2^\bot + {\bf \Psi}_{\nu + 1}) $, we use the
following inductive inequalities, which are deduced by 
Lemma \ref{interpolazione moduli parametri} and   \eqref{tame Psi nu - 1}, 
\begin{align}\label{epsilon nu}
& {\mathfrak M}_{{\bf U}_{\nu + 1}}^\sharp(s_0) 
\leq  {\mathfrak M}_{{\bf U}_\nu}^\sharp(s_0) \big( 1+  C(k_0) \e_\nu (s_0) \big) \, , \\
& {\mathfrak M}_{{\bf U}_{\nu + 1}}^\sharp(s) \leq 
{\mathfrak M}_{{\bf U}_\nu}^\sharp(s) (1 + C(k_0) \e_\nu (s_0) ) + C(k_0) {\mathfrak M}_{{\bf U}_\nu}^\sharp(s_0)  \e_\nu(s) 
\label{epsilon nu-2}
\end{align}
where
$ \e_{\nu}(s) := {\mathfrak M}_0(s, {\mathtt b}) \gamma^{- 1} N_{\nu + 1}^{\tau_1} N_\nu^{- \mathtt a} $.

Iterating \eqref{epsilon nu}, setting $ \e_\nu :=  C(k_0) \e_\nu (s_0) $, 
and using \eqref{def:costanti iniziali tame}, \eqref{tame Psi nu - 1}, \eqref{ansatz riducibilita} we get
\begin{equation}\label{M U nu + 1}
\begin{aligned}
{\mathfrak M}_{{\bf U}_{\nu + 1}}^\sharp(s_0) & 
\leq {\mathfrak M}_{{\bf U}_0}^\sharp(s_0) \prod_{\nu \geq 0 } (1 + \e_\nu ) \\
& \leq 
{\mathfrak M}_{{\bf U}_0}^\sharp(s_0) {\rm exp}(C(S) \e \gamma^{- 2}) \leq 2 \, , \quad \forall \nu \geq 0 \, .
\end{aligned}
\end{equation}
Iterating  \eqref{epsilon nu-2}, using \eqref{M U nu + 1} and  $\prod_{\nu \geq 0} (1 + \e_\nu) \leq 2 $, we get 
\begin{equation}\label{M U nu + 1 (s)}
\begin{aligned}
{\mathfrak M}_{{\bf U}_{\nu + 1}}^\sharp(s) 
& \leq_{k_0}  {\mathop \sum}_{\nu \geq 0} \e_\nu(s) + 
{\mathfrak M}_{{\bf U}_0}^\sharp(s) \\ 
& \leq C(k_0) \big( 1 + N_0^{\tau_1}{\mathfrak M}_0(s, \mathtt b) \gamma^{- 1} \big)\,, 
\quad \forall \nu \geq 0 \, ,
\end{aligned}
\end{equation}
since ${\bf U}_0 = {\bf \Phi}_0 = {\mathbb I}_2^\bot + {\bf \Psi}_0$ and 
$ {\mathfrak M}_{{\bf U}_0}^\sharp(s) \leq 1 + 
C(k_0) N_0^{\tau_1} {\mathfrak M}_0(s, \mathtt b) \gamma^{- 1} $ by \eqref{tame Psi nu - 1}. 
Finally
$$
\begin{aligned}
{\bf U}_n - {\mathbb I}_2^\bot & = ({\bf U}_n - {\bf \Phi}_0 )+ ({\bf \Phi}_0 - {\mathbb I}_2^\bot) \\
& = \sum_{\nu = 0}^{n - 1} ({\bf U}_{\nu + 1} - {\bf U}_\nu ) + {\bf \Psi}_0 = \sum_{\nu = 0}^{n - 1} {\bf U}_\nu {\bf \Psi}_{\nu + 1} + {\bf \Psi}_0\, .
\end{aligned}
$$
Hence Lemma \ref{interpolazione moduli parametri}, 
\eqref{M U nu + 1}, \eqref{M U nu + 1 (s)}, 
\eqref{stima cal K0 I delta}, 
\eqref{ansatz I delta}, \eqref{tame Psi nu - 1},  \eqref{ansatz riducibilita},  
imply 
\eqref{stima Phi infinito} for ${\bf U}_n - {\mathbb I}_2^\bot $. 
The estimate for ${\bf U}_n^{- 1} - {\mathbb I}_2^\bot$ follows by Lemma \ref{serie di neumann per maggioranti}.  
\end{proof}

\section{Proof of Theorem \ref{ITERAZIONERIDUCIBILITA}}

{\sc Proof of ${\bf(S1)}_{0}$}. Properties \eqref{cal L nu}-\eqref{forma cal R nu} for $ \nu = 0 $
 follow by the assumptions  \eqref{defL0-red}-\eqref{defRQ0}
with  $r_j^0(\omega, \kappa) = 0$. 
We now prove that also \eqref{stima cal R nu} for $ \nu = 0 $ holds:

\begin{lemma}\label{lem: Initialization} 
$ {\mathfrak M}_0^\sharp (s) $,  $ {\mathfrak M}_0^\sharp ( s, \mathtt b) \leq_{s_0, \mathtt b}  {\mathfrak M}_0 (s, {\mathtt b})  $. 
\end{lemma}

\begin{proof}
Let $ {\mathcal R} \in \{ {\mathcal R}_0, {\mathcal Q}_0 \}$ and  set $\lambda := (\omega, \kappa)$. 
The matrix elements of the commutator  $ [ {\mathcal R}, \pa_x ]$ are $ \ii (j'-j) ({\mathcal R})^{j'}_j (\ell- \ell' ) $, of $\partial_{\vphi_m}^b {\mathcal R}$, 
$ m  = 1, \ldots , \CS $, are $\ii^b (\ell_m - \ell'_m)^b {\mathcal R}_j^{j'}(\ell - \ell')$, and
of   $  \pa_{\vphi_m}^b [ {\mathcal R}, \pa_x ] $ are $ \ii^{b +1} (\ell_m - \ell_m' )^b (j'-j) ({\mathcal R}_0)^{j'}_j (\ell- \ell' ) $. 
Then, recalling \eqref{tame-coeff} with $ \s = 0  $, the assumptions
\eqref{tame cal R0 cal Q0}-\eqref{tame norma alta cal R0 cal Q0} imply that
 $ \forall |k | \leq k_0 $, $ s_0 \leq s \leq S $, $\ell' \in \Z^{\CS},  j' \in {\mathbb S}_0^c$, 
\be
 \gamma^{2 |k|} {\mathop \sum}_{\ell , j} \langle \ell, j \rangle^{2 s} |\partial_\lambda^k {\mathcal R}_j^{j'}(\ell - \ell')|^2 \leq
 2 {\mathbb M}_0^2 (s_0) \langle \ell', j' \rangle^{2 s} 
 + 2  {\mathbb M}_0^2(s) \langle \ell', j' \rangle^{2 s_0} \label{monfalcone 1} 
 \ee
 \be
\begin{aligned}
 &       \gamma^{2 |k|} {\mathop \sum}_{\ell, j} \langle\ell, j \rangle^{2 s} |j - j'|^2 |\partial_\lambda^k {\mathcal R}_j^{j'}(\ell - \ell')|^2 \leq \\
& \qquad \qquad 2  {\mathbb M}_0^2(s_0) \langle \ell', j' \rangle^{2 s} + 2  {\mathbb M}_0^2(s) \langle \ell', j' \rangle^{2 s_0} \label{monfalcone 2} 
\end{aligned} 
\ee 
\be
\begin{aligned}
&      \gamma^{2 |k|} {\mathop \sum}_{\ell, j} \langle \ell, j \rangle^{2 s} |\ell_m - \ell'_m|^{2 s_0} 
|\partial_\lambda^k {\mathcal R}_j^{j'}(\ell - \ell')|^2 \leq \\
&  \qquad \qquad 2  {\mathbb M}_0^2(s_0) \langle \ell', j' \rangle^{2 s} 
+ 2  {\mathbb M}_0^2(s) \langle \ell', j' \rangle^{2 s_0} \label{monfalcone 3} 
\end{aligned} 
 \ee 
\be
\begin{aligned}
&   \gamma^{2 |k|} {\mathop \sum}_{\ell, j} \langle \ell, j \rangle^{2 s} |\ell_m - \ell'_m|^{2 s_0} |j - j'|^2  |\partial_\lambda^k {\mathcal R}_j^{j'}(\ell - \ell')|^2 \leq \\ 
& \qquad \qquad 2  {\mathbb M}_0^2(s_0) \langle \ell', j' \rangle^{2 s} + 2  {\mathbb M}_0^2(s) \langle \ell', j' \rangle^{2 s_0} \label{monfalcone 4} 
\end{aligned}
\ee
\be
\begin{aligned}
&    \gamma^{2 |k|} {\mathop \sum}_{\ell, j} \langle \ell, j \rangle^{2 s} |\ell_m - \ell'_m|^{2 (s_0 + \mathtt b) } |\partial_\lambda^k {\mathcal R}_j^{j'}(\ell - \ell')|^2 \leq \\ 
& \qquad \qquad 2  {\mathbb M}_0^2 (s_0, \mathtt b ) \langle \ell', j' \rangle^{2 s} 
+ 2  {\mathbb M}_0^2(s, \mathtt b) \langle \ell', j' \rangle^{2 s_0} \label{monfalcone 5} 
\end{aligned}
\ee
\be
\begin{aligned}
&     \gamma^{2 |k|} {\mathop \sum}_{\ell, j} \langle \ell, j \rangle^{2 s} |\ell_m - \ell'_m|^{2 (s_0 + \mathtt b)} |j - j'|^2 |\partial_\lambda^k {\mathcal R}_j^{j'}(\ell - \ell')|^2 \leq  \\ 
& \qquad \qquad 2 {\mathbb M}_0^2(s_0, \mathtt b) \langle \ell', j' \rangle^{2 s}  + 
2  {\mathbb M}_0^2(s, \mathtt b) \langle \ell', j' \rangle^{2 s_0}. \label{monfalcone 6} 
\end{aligned}
\ee
Using the inequality 
\be\label{inequality scema langle rangle}
\begin{aligned}
\langle \ell - \ell' \rangle^{2 s_1} \langle j - j' \rangle^2  \leq_{s_1} 1 + |j - j'|^2 & +
 \max_{m = 1, \ldots, \CS}|\ell_m - \ell'_m|^{2 s_1} \\ 
 & + |j - j'|^2\max_{m = 1, \ldots, \CS}|\ell_m - \ell'_m|^{2 s_1}
 \end{aligned}
\ee
for $ s_1 = s_0 $, $ s = s_0+{\mathtt b} $,    the estimates \eqref{monfalcone 1}-\eqref{monfalcone 6} 
imply, recalling also \eqref{def:costanti iniziali tame}, 
\be
\begin{aligned}
&   \gamma^{2|k|} {\mathop \sum}_{\ell, j} \langle \ell, j \rangle^{2 s} \langle \ell - \ell' \rangle^{2 s_0}
 \langle j - j' \rangle^{2} |\partial_\lambda^k {\mathcal R}_j^{j'}(\ell - \ell')|^2  
 \leq_{\mathtt b} \\ 
 & \qquad \qquad {\mathfrak M}_0^2(s_0, {\mathtt b}) \langle \ell', j'\rangle^{2 s} + 
{\mathfrak M}_0^2(s, {\mathtt b}) \langle \ell', j' \rangle^{2 s_0}  \label{carletto0} \\
\end{aligned}
\ee
\be
\begin{aligned}
& \gamma^{2|k|}  {\mathop \sum}_{\ell, j} \langle \ell, j \rangle^{2 s} 
 \langle \ell - \ell' \rangle^{2 ({s_0+{\mathtt b}})}   \langle j - j' \rangle^{2} 
  | \partial_\lambda^k {\mathcal R}_j^{j'}(\ell - \ell')|^2  
\leq_{\mathtt b} \\ 
& \qquad  \qquad {\mathfrak M}_0^2(s_0, {\mathtt b}) \langle \ell', j'\rangle^{2 s} + 
{\mathfrak M}_0^2(s, {\mathtt b}) \langle \ell', j' \rangle^{2 s_0}  \, .  \label{carletto}
\end{aligned}
\ee
We can now prove that $ \langle \partial_\vphi  \rangle^{\mathtt b} {\mathcal R} $ is $ {\mathcal D}^{k_0}$-modulo-tame. 
 $\forall |k | \leq k_0 $, by Cauchy-Schwartz inequality, we get
\begin{align}
\| |\langle \partial_\vphi  \rangle^{\mathtt b}\partial_\lambda^k {\mathcal R}| h \|_s^2 & \leq 
{\mathop \sum}_{\ell, j} \langle \ell, j \rangle^{2 s} 
\Big( {\mathop \sum}_{\ell', j'} | \langle \ell - \ell' \rangle^{\mathtt b}\partial_\lambda^k  {\mathcal R}_j^{j'}(\ell - \ell')| |h_{\ell', j'}| \Big)^2 
\nonumber \\
& = \sum_{\ell, j} \langle \ell, j \rangle^{2 s} 
\Big(  \sum_{\ell', j'} \langle \ell - \ell' \rangle^{s_0 + \mathtt b} \langle j' - j \rangle 
|\partial_\lambda^k {\mathcal R}_j^{j'}(\ell - \ell')| \times \nonumber \\
& \qquad \qquad \qquad \qquad \qquad \qquad
 \times  |h_{\ell', j'}| \frac{1}{\langle \ell - \ell' \rangle^{s_0} \langle j' - j \rangle } \Big)^2 \nonumber \\
& \leq_{s_0} {\mathop \sum}_{\ell, j} \langle \ell, j \rangle^{2 s} 
\sum_{\ell', j'} \langle \ell - \ell' \rangle^{2 (s_0 + \mathtt b)} 
\langle j' - j \rangle^{2 } |\partial_\lambda^k {\mathcal R}_j^{j'}(\ell - \ell')|^2 |h_{\ell', j'}|^2 \nonumber \\
& =  \sum_{\ell' , j'} |h_{\ell', j'}|^2 
 \sum_{\ell, j} \langle \ell, j \rangle^{2 s} \langle \ell - \ell' \rangle^{2 (s_0 + \mathtt b)} \langle j' - j \rangle^{2 } |\partial_\lambda^k {\mathcal R}_j^{j'}(\ell - \ell')|^2 \nonumber \\
& \stackrel{ \eqref{carletto}}{{\leq}_{s_0, \mathtt b}}  \gamma^{- 2 |k|}
 \sum_{\ell', j'} |h_{\ell', j'}|^2 
\big( {\mathfrak M}_0^2(s_0, \mathtt b) \langle \ell', j' \rangle^{2 s} + {\mathfrak M}_0^2(s, \mathtt b) \langle \ell', j' \rangle^{2 s_0}    \big)
\nonumber   \\
 & \leq_{s_0, \mathtt b}  \gamma^{- 2|k|} \big( {\mathfrak M}_0^2(s_0, \mathtt b) \| h \|_{s}^2 + {\mathfrak M}_0^2(s, \mathtt b) \| h \|_{s_0}^2 \big) \, .  \label{tame-inizio}
\end{align}
Therefore (recall 
 \eqref{CK0-tame}) the modulo-tame constant   
 $ {\mathfrak M}_{ \langle \pa_\vphi \rangle^{\mathtt b} {\mathcal R}}^\sharp (s) \leq_{s_0, \mathtt b} {\mathfrak M}_0(s, \mathtt b) $. Since
 $ {\mathcal R} $ is both $ \{ {\mathcal R}_0, {\mathcal Q}_0 \} $ we have proved that (see \eqref{def:msharp})
  $$ 
  {\mathfrak M}_0^\sharp ( s, \mathtt b) := \max\{ {\mathfrak M}_{ \langle \pa_\vphi \rangle^{\mathtt b} {\mathcal R}_0}^\sharp (s), 
{\mathfrak M}_{ \langle \pa_\vphi \rangle^{\mathtt b} {\mathcal Q}_0}^\sharp (s)  \} 
 \leq_{s_0, \mathtt b}  {\mathfrak M}_0 (s, {\mathtt b})  \, . 
  $$ 
  The inequality   $ {\mathfrak M}_0^\sharp ( s ) \leq_{s_0}  {\mathfrak M}_0 (s, {\mathtt b})  $ follows similarly by \eqref{carletto0}. 
\end{proof}
\noindent
{\sc Proof of ${\bf(S2)}_0$}.
It follows since the functions $ \mathtt m_3(\omega, \kappa)$ and ${\mathtt m}_1(\omega, \kappa)$ are $ k_0$-times differentiable
on all $ \tOm \times [\kappa_1, \kappa_2] $ (they depend on the torus $ i_\delta(\omega, \kappa) $ which is $ k_0 $-times differentiable  
with respect to $ (\om, \kappa) $ on all $ \tOm \times [\kappa_1, \kappa_2] $).

\noindent
{\sc Proof of ${\bf(S3)}_0$}. We prove \eqref{stima R nu i1 i2 norma alta} at $ \nu = 0 $, namely that, 
 for ${\mathcal R} \in \{ {\mathcal R}_0, {\mathcal Q}_0 \} $,  
\be\label{7.28-0}
\|   | \langle \partial_\vphi\rangle^{\mathtt b} \Delta_{1 2} {\mathcal R}| h \|_{s_0}^2  
\leq C(S, {\mathtt b}) \e^2 \gamma^{- 2} \| i_1 - i_2 \|_{s_0  + \mu(\mathtt b) + \sigma }^2 \| h \|_{s_0}^2 \, , \quad \forall h \in H^{s_0} \, , 
\ee 
where we denote $ \Delta_{12} {\mathcal R} := {\mathcal R}(i_1) - {\mathcal R}(i_2) $. 
By \eqref{derivate i resti prima della riducibilita-s0}  and the mean value theorem we get 
\begin{align*}
& \| \Delta_{12} {\mathcal R} \|_{{\mathcal L}(H^{s_0})}, \| [\Delta_{12} {\mathcal R}, \partial_x ]\|_{{\mathcal L}(H^{s_0})},
\| \partial_{\vphi_m}^{s_0 + \mathtt b} \Delta_{12} {\mathcal R} \|_{{\mathcal L}(H^{s_0})},
\| \partial_{\vphi_m}^{s_0 + \mathtt b} [ \Delta_{12} {\mathcal R}, \partial_x] \|_{{\mathcal L}(H^{s_0})} \\
&  \leq_{S, \mathtt b} 
\e \gamma^{- 1} \|i_1 - i_2\|_{ s_0  + \mu(\mathtt b) + \sigma}
\end{align*}
for all $ m = 1, \ldots, \CS $. 
We deduce as in \eqref{monfalcone 1}-\eqref{monfalcone 6} (with $ k = 0 $) 
and \eqref{inequality scema langle rangle} that,  for all $\ell' \in \Z^{\CS}, j' \in {\mathbb S}^c_0 $,   
$$
\begin{aligned}
& {\mathop \sum}_{\ell, j} \langle \ell, j \rangle^{2 s_0} \langle j - j' \rangle^2 \langle \ell - \ell' \rangle^{2(s_0 + \mathtt b)} 
| (\Delta_{12}{\mathcal R})_j^{j'}(\ell - \ell')|^2 
\leq \\
& \qquad \qquad \qquad \qquad \qquad \qquad 
C(S, \mathtt b) \e^2 \gamma^{- 2} \| i_1 - i_2 \|_{ s_0 + \mu(\mathtt b) + \sigma }^2 \langle \ell', j' \rangle^{2 s_0}  
\end{aligned}
$$
which, arguing as in \eqref{tame-inizio}, proves \eqref{7.28-0}.
The proof of  \eqref{stima R nu i1 i2} at $ \nu = 0 $ is analogous.

\noindent 
{\sc Proof of ${\bf(S4)}_{0}$}. It is trivial because by definition 
$ \tOm_0^\gamma(i_1) = \tOm = \tOm_0^{\gamma - \rho}(i_2)$. 
\subsection{The reducibility step}

In this section we describe the generic inductive step, showing how to define ${\bf L}_{\nu + 1}$ (and ${\bf \Phi}_\nu$, ${ \bf \Psi}_\nu$ etc). To simplify notation we drop the index $\nu$ and we write $+$ instead of $\nu + 1$, so that we write
${\bf  L} := {\bf  L}_{\nu } $, $ {\bf D} := {\bf D}_\nu $, $ {\bf R} := {\bf R}_\nu $, 
$ {\mathcal R} := {\mathcal R}_\nu $,
$ {\bf Q} := {\bf Q}_\nu $, $ {\mathcal Q} := {\mathcal Q}_\nu $, $ {\mathcal D} := {\mathcal D}_\nu $, $ \mu_j = \mu_j^\nu $, etc \ldots

We conjugate $ {\bf L} $ by a  transformation of the form (see \eqref{Psi nu - 1})
\be\label{defPsi}
{\bf \Phi} := {\mathbb I}_2^\bot + {\bf \Psi}\,, \qquad {\bf \Psi} := \begin{pmatrix}
\Psi_{1} & \Psi_2 \\
\overline \Psi_2 & \overline \Psi_1
\end{pmatrix}\,.
\ee
We have 
\be\label{Ltra1}
\begin{aligned}
{\bf L} {\bf \Phi}  = {\bf \Phi} \big( \Dom {\mathbb I}_2^\bot + \ii {\bf D} \big) & +
\big( \Dom{\bf \Psi} + \ii [{\bf D}, {\bf \Psi}]  
+ \Pi_{N} {\bf R}+ \Pi_{N} {\bf Q} \big) \\
& + \Pi_{N}^\bot {\bf R} + \Pi_{N}^\bot {\bf Q} + {\bf R} {\bf \Psi} + {\bf Q} {\bf \Psi}
\end{aligned}
\ee
where the projector $ \Pi_N $ is defined in \eqref{proiettore-oper} and $ \Pi_N^\bot := {\mathbb I}_2 - \Pi_N$.
We want to solve the homological equation 
\begin{equation}\label{equazione omologica}
\Dom{\bf \Psi} + \ii [{\bf D}, {\bf \Psi}] + \Pi_{N} {\bf R} + \Pi_{N} {\bf Q} = [{\bf R}] 
\ee
where
\be\label{def[R]1}
[{\bf R}] := \begin{pmatrix}
[{\mathcal R}] & 0 \\
0 & [\overline{\mathcal R}] 
\end{pmatrix}\, 
\ee
and the operator $[{\mathcal R}]$ is defined by 
\begin{equation}\label{def[R]}
\begin{aligned}
& [{\mathcal R}] u(x) = {\mathop \sum}_{j \in {\mathbb S}_0^c} 
 \big( {\mathcal R}_j^{- j} (0) u_{- j} + {\mathcal R}_j^j (0) u_j \big) e^{\ii j x}\,, \\
 & \qquad  \text{for any function} \quad  u (x) = \sum_{j \in {\mathbb S}_0^c} u_j e^{\ii j x}\,.
 \end{aligned} 
\end{equation}
By \eqref{cal L nu}, \eqref{forma cal R nu}, \eqref{defPsi} the 
equation \eqref{equazione omologica} is equivalent to the two scalar  homological equations 
\be\label{homo1r}
\begin{aligned}
& \Dom \Psi_1 + \ii [{\mathcal D}, \Psi_1] + \Pi_{N} {\mathcal R} = [{\mathcal R}]  \, , \\
& \Dom \Psi_2 + \ii \big( {\mathcal D} \Psi_2 + \Psi_2 {\mathcal D} \big) + \Pi_{N} {\mathcal Q} = 0 \, .
\end{aligned}
\ee
The solutions of \eqref{homo1r} are 
\begin{align}\label{shomo1}
& (\Psi_1)_j^{j'}(\ell  ) := \begin{cases}
- \dfrac{({\mathcal R})_j^{j'}(\ell  )}{\ii (\omega \cdot \ell  + \mu_j - \mu_{j'})} \quad \forall (\ell ,  j,  j') \neq (0, \pm j,  \pm j) \, , 
| \ell  | \leq N \, ,  \\
0 \qquad \text{otherwise}
\end{cases} \\
& \label{shomo2}
(\Psi_2)_j^{j'}(\ell  ) := 
- \dfrac{({\mathcal Q})_j^{j'}(\ell )}{\ii (\omega \cdot \ell  +\mu_j + \mu_{j'})}\, , \, \quad  
\forall (\ell  , j,  j') \in \Z^{\CS} \times {\mathbb S}_0^c \times {\mathbb S}_0^c \, , | \ell  | \leq N \, .
\end{align}
Note that, since $ \mu_j = \mu_{- j} $, $ \forall j \in {\mathbb S}_0^c $ (see \eqref{stima rj nu})
the denominators in \eqref{shomo1}, \eqref{shomo2} are different from zero 
for $ (\om, \kappa) \in \tLm_{\nu+1}^\gamma $ (see \eqref{Omega nu + 1 gamma}
with  $ \nu \rightsquigarrow \nu+1 $) and 
the maps  $ \Psi_1 $, $ \Psi_2 $ are well defined.

\begin{lemma} {\bf (Homological equations)}\label{Homological equations tame}
For all $(\omega, \kappa) \in \tLm_{\nu+1}^{\gamma /2}$ the solutions 
$ \Psi_1 $, $ \Psi_2 $  in \eqref{shomo1}, \eqref{shomo2} of the homological equations 
 \eqref{homo1r}  are $ {\mathcal D}^{k_0} $-modulo-tame operators with modulo-tame constants satisfying
\be\label{stima tame Psi}
\begin{aligned}
 {\mathfrak M}_{\Psi_1}^\sharp (s), {\mathfrak M}_{\Psi_2}^\sharp (s) & 
 \leq_{k_0} N^{\tau_1} \g^{-1} {\mathfrak M}^\sharp (s), \\
{\mathfrak M}_{\langle \partial_\vphi \rangle^{\mathtt b} \Psi_1 }^\sharp (s),
{\mathfrak M}_{\langle \partial_\vphi \rangle^{\mathtt b} \Psi_2}^\sharp (s)
&  \leq_{k_0} N^{\tau_1} \g^{-1} {\mathfrak M}^\sharp (s, {\mathtt b})
 \end{aligned}
\ee
where $ \tau_1 := \tau (k_0 + 1) + k_0$. 

Given $i_1 $, $i_2 $  denote $ \Delta_{12} \Psi_1 := \Psi_1 (i_2) - \Psi_1 (i_1) $. 
If $\gamma/2 \leq \gamma_1, \gamma_2 \leq 2 \gamma$  
then, for all $(\omega, \kappa) \in \tLm_{\nu + 1}^{\gamma_1}(i_1) \cap \tLm_{\nu + 1}^{\gamma_2}(i_2)$,
\be\label{stime delta 12 Psi bassa}
\begin{aligned}
&   \| |\Delta_{12} \Psi_1 | \|_{{\mathcal L}(H^{s_0})}\leq \\ 
& C 
N^{2 \tau} \gamma^{- 1} \big( \||{\mathcal R}(i_2)| \|_{{\mathcal L}(H^{s_0})} \| i_1 - i_2 \|_{2 s_0 + \sigma + \mu(\mathtt b)} + \| | \Delta_{1 2} {\mathcal R} | \|_{{\mathcal L}(H^{s_0})} \big) \, , 
\end{aligned}
\ee
\be\label{stime delta 12 Psi alta}
\begin{aligned}
& \| |\langle \partial_\vphi \rangle^{\mathtt b}\Delta_{12} \Psi_1 | \|_{{\mathcal L}(H^{s_0})} \leq_{\mathtt b} \\
& N^{2 \tau} \gamma^{- 1} \big(  \| |\langle \partial_\vphi \rangle^{\mathtt b}{\mathcal R}(i_2) | \|_{{\mathcal L}(H^{s_0})} \| i_1 - i_2 \|_{2 s_0  + \sigma + \mu(\mathtt b)} 
 +  \| | \langle \partial_\vphi \rangle^{\mathtt b}\Delta_{12} {\mathcal R}  | \|_{{\mathcal L}(H^{s_0})}\big)
\end{aligned}
\ee
and a similar estimate holds for $\Psi_2$, replacing ${\mathcal R}$ by ${\mathcal Q}$. 
Moreover  ${\bf \Psi}$ 
is real, even and reversibility preserving. 
\end{lemma}

In the sequel, for a quantity $ g(i) $ (an operator, a map, a scalar function) depending on the torus $ i $, given 
 $i_1 $, $i_2 $  we denote the difference 
 $$ 
 \Delta_{12} g := g (i_2) - g (i_1) \, .
$$ 
\begin{proof}
We make the proof for $\Psi := \Psi_1$, for $\Psi_2$ is analogous. 

\noindent
{\sc Proof of \eqref{stima tame Psi}}. 
Let $ (\om, \kappa) \in \tLm_{\nu+1}^{\gamma/2} $. 
By \eqref{Omega nu + 1 gamma} with  $ \nu \rightsquigarrow \nu+1 $,  
and the definition of $ \Psi_1 $ in \eqref{shomo1}, we have, for all $ (\ell, j, j')  \in \Z^{\CS} \times {\mathbb S}_0^c \times {\mathbb S}_0^c $, 
with $ |\ell | \leq N $, $ (\ell, j, j') \neq (0, \pm j,  \pm j) $, 
$ |\Psi_j^{j'}(\ell )|  
\leq C N^\tau \gamma^{-1} |{\mathcal R}_j^{j'}(\ell )| $. 
Moreover,  differentiating \eqref{shomo1} with respect to $\lambda = (\omega, \kappa)$, we get 
$$
\pa_\lambda^k \Psi_j^{j'}(\ell ) = {\mathop \sum}_{k_1+ k_2 = k} C(k_1, k_2) 
\big[ \pa_\lambda^{k_1} (\omega \cdot \ell  + \mu_j - \mu_{j'})^{-1} 
\big]
\pa_\lambda^{k_2} {\mathcal R}_j^{j'}(\ell  )  \, , 
$$
and since, by \eqref{mu j nu}, \eqref{stima rj nu}, \eqref{Omega nu + 1 gamma},  \eqref{stima lambda 3 - 1 nuova}, \eqref{stime lambda 1 senza proiettore},
$$
\sup_{|k_1| \leq k_0}  | \pa_\lambda^{k_1} (\omega \cdot \ell  + \mu_j - \mu_{j'})^{-1} | \leq
C(k_0) \langle \ell  \rangle^{\tau(k_0 + 1) + k_0} \gamma^{- 1 - |k_1|} \, , 
$$
we deduce that,  for all $0 < |k| \leq  k_0 $, 
\be\label{stima-Psi-R}
|\partial_\lambda^k \Psi_j^{j'}(\ell )| \leq C(k_0) \langle \ell  \rangle^{\tau(k_0 + 1) + k_0} \gamma^{- 1 - |k|} 
{\mathop \sum}_{|k_2| \leq |k|} \gamma^{|k_2|} |\partial_\lambda^{k_2} {\mathcal R}_j^{j'}(\ell )|\, . 
\ee
Therefore for all $0 \leq |k| \leq k_0$ we get 
\begin{align}
\| |\langle \partial_\vphi \rangle^{\mathtt b}\partial_\lambda^k \Psi| h  \|_s^2 & \leq \sum_{\ell, j} \langle \ell, j \rangle^{2 s} \Big(\sum_{ |\ell' - \ell| \leq N, j'} 
| \langle \ell - \ell' \rangle^{\mathtt b} \partial_\lambda^k \Psi_j^{j'}(\ell - \ell')| |h_{\ell', j'}| \Big)^2 \nonumber\\
&  \stackrel{\eqref{stima-Psi-R}} {\leq_{k_0}}  N^{2 \tau_1} \gamma^{- 2(1 + |k|)} \sum_{|k_2| \leq |k|}  \gamma^{2 |k_2|} \sum_{\ell, j} \langle \ell, j \rangle^{2 s} \times \nonumber \\
& 
\qquad \qquad \qquad \qquad 
\times \Big(\sum_{\ell', j'} | \langle \ell - \ell' \rangle^{\mathtt b} \partial_\lambda^{k_2}{\mathcal R}_j^{j'}(\ell - \ell')| |h_{\ell', j'}| \Big)^2 \nonumber\\
& = N^{2 \tau_1} \gamma^{- 2(1 + |k|)} {\mathop \sum}_{|k_2| \leq |k|} \gamma^{2 |k_2|} \| 
| \langle \partial_\vphi \rangle^{\mathtt b} \partial_\lambda^{k_2}{\mathcal R}|[\norma h \norma] \|_s^2 \nonumber\\
& \stackrel{\eqref{def:msharp}, \eqref{CK0-tame}} {\leq_{k_0}} N^{2 \tau_1} \gamma^{- 2(1 + |k|)} 
\big( {\mathfrak M}^\sharp (s, \mathtt b)^2 \| \norma h \norma\|_{s_0}^2 + {\mathfrak M}^\sharp (s_0, \mathtt b)^2 \| \norma h \norma\|_{s}^2  \big)  \nonumber\\
& \stackrel{ \eqref{Soboequals}} = C(k_0) N^{2 \tau_1} \gamma^{- 2(1 + |k|)} \big( {\mathfrak M}^\sharp (s, \mathtt b)^2 \|  h \|_{s_0}^2 + 
{\mathfrak M}^\sharp (s_0, \mathtt b)^2 \|  h \|_{s}^2  \big) \label{caserta}
\end{align}
and, recalling 
Definition \ref{def:op-tame}, the second inequality in \eqref{stima tame Psi} follows. The proof of the first inequality is analogous. 

\smallskip

\noindent
{\sc Proof of \eqref{stime delta 12 Psi bassa}-\eqref{stime delta 12 Psi alta}}. 
By \eqref{shomo1}, for all $(\omega, \kappa) \in \tLm_{\nu + 1}^{\gamma_1}(i_1) \cap 
\tLm_{\nu + 1}^{\gamma_2}(i_2)$, one has 
$$
\Delta_{1 2} \Psi_j^{j'}(\ell) = \dfrac{\Delta_{12} {\mathcal R}_j^{j'}(\ell)}{\delta_{\ell j j'}(i_1)} - {\mathcal R}_j^{j'}(\ell)(i_2) \dfrac{\Delta_{1 2} \delta_{\ell j j'}}{\delta_{\ell j j'}(i_1) \delta_{\ell j j'}(i_2)}\,,\ 
\delta_{\ell j j'} := \ii (\omega \cdot \ell + \mu_j - \mu_{j'})\,.
$$
By \eqref{mu j nu},  
\eqref{stima lambda 3 - 1 nuova}, \eqref{stime lambda 1 senza proiettore}, \eqref{r nu i1 - r nu i2} we get 
$$
|\Delta_{12} \delta_{\ell j j'}|  = |\Delta_{1 2} (\mu_j - \mu_{j'})|\leq 
 C \e \gamma^{- 1} ||j|^{\frac32} - |j'|^{\frac32}| \| i_1 - i_2 \|_{2 s_0 + \sigma + \mu(\mathtt b)}\,,
 $$
whence $\gamma_1^{- 1}, \gamma_2^{- 1} \leq \gamma^{- 1}$, $\e \gamma^{- 2}$ small enough, imply 
$$
| \Delta_{12} \Psi_j^{j'}(\ell)| \leq C N^{2 \tau} \gamma^{- 1} \big( |{\mathcal R}_j^{j'}(\ell)(i_2)| \| i_1 - i_2 \|_{2 s_0 + \sigma + \mu(\mathtt b)} +  |\Delta_{1 2} {\mathcal R}_j^{j'}(\ell)|\big)
$$
and \eqref{stime delta 12 Psi bassa}, \eqref{stime delta 12 Psi alta} follow arguing as in \eqref{caserta}.

Finally, since ${\bf R}, {\bf Q}$ are even and reversible, 
\eqref{shomo1}, \eqref{shomo2} imply that ${\bf \Psi}$ is even and reversibility preserving.  
\end{proof}

By \eqref{Ltra1}, \eqref{equazione omologica} we have 
$$
{\bf L}_+ = {\bf \Phi}^{- 1} {\bf L} {\bf \Phi} = \Dom {\mathbb I}_2^\bot + \ii {\bf D}_+ + {\bf R}_+ + {\bf Q}_+ 
$$
which  proves \eqref{coniugionu+1} and \eqref{cal L nu} at the step $ \nu + 1 $,  with 
\be\label{new-diag-new-rem} 
\ii {\bf D}_+ := \ii {\bf D} + [{\bf R}] \, , \quad 
{\bf R}_+ + {\bf Q}_+ = 
{\bf \Phi}^{- 1} \big( \Pi_{N}^\bot {\bf R} + \Pi_{N}^\bot {\bf Q} + {\bf R} {\bf \Psi} - {\bf \Psi} [{\bf R}] + {\bf Q} {\bf \Psi} \big) \, . 
\ee
The new operator $ {\bf L}_+  $ has the same form of $ \bf L $ with
$ {\bf R}_+ + {\bf Q}_+ $  which is the sum of a quadratic function of 
$ { \bf \Psi} $ and  $ ({ \bf R}, {\bf Q}) $ and a remainder supported on high frequencies. 
The new normal form $ {\bf D}_+  $ is diagonal:  

\begin{lemma}\label{nuovadiagonale} 
{\bf (New diagonal part).}  
The new normal form 
 is 
\be\label{new-NF}
\begin{aligned}
& \ii {\bf D}_+ = \ii {\bf D} + [{\bf R}] = \ii \begin{pmatrix}
 {\mathcal D}_+ & 0 \\
0 & - {\mathcal D}_+
\end{pmatrix}\,, \\
&  {\mathcal D}_+ := {\rm diag}_{j \in {\mathbb S}^c_0} \mu_j^+\,,\quad \mu_j^+ := 
\mu_j + {\mathtt r}_j\in \R\, , 
\end{aligned}
\ee
with $  {\mathtt r}_j = {\mathtt r}_{- j} $, $  \mu_j^+ = \mu_{- j}^+ $, $  \forall j \in {\mathbb S}_0^c $, and 
$ |\mu_j^+ - \mu_j|^{k_0, \gamma} \lessdot {\mathfrak M}^\sharp (s_0)$.

Moreover, given tori $ i_1(\omega, \kappa) $, $ i_2(\omega, \kappa)$ then,
 for all $(\omega, \kappa) \in \tLm_{\nu}^{\gamma_1}(i_1) \cap \tLm_\nu^{\gamma_2}(i_2)$, the  difference 
\be\label{diff:r1r2}
| {\mathtt r}_j (i_1) -{\mathtt r}_j (i_2)   |  
\leq C \| | \Delta_{1 2} {\mathcal R} | \|_{{\mathcal L}(H^{s_0})}\,.
\ee
\end{lemma}

\begin{proof}
The operator $ [ {\mathcal R}] $ in \eqref{def[R]} satisfies 
\begin{align*}
[{\mathcal R}] u  & = {\mathop \sum}_{j \in {\mathbb S}_0^c} 
 \big( {\mathcal R}_j^{- j} (0) u_{- j} + {\mathcal R}_j^j (0) u_j \big) e^{\ii j x} = 
 {\mathop \sum}_{j \in {\mathbb S}_0^c} 
 \big( {\mathcal R}_j^{- j} (0)  + {\mathcal R}_j^j (0) \big) u_j e^{\ii j x}  
\end{align*} 
since  $ [{\mathcal R}] $ acts on the space $H_{\mathbb S_0}^\bot$ of  functions even in $ x $, i.e. $ u_j = u_{- j} $ (see \eqref{definizione H S0 bot}).
Thus \eqref{new-NF} holds with 
$ {\mathcal R}_j^{- j} (0)  + {\mathcal R}_j^j (0) =: \ii {\mathtt r}_j  $. 
Since  $ {\mathcal R}$ is even, by \eqref{even operators Fourier}  we deduce 
$ {\mathtt r}_{-j} = {\mathtt r}_j  $.
In addition, since  $ {\mathcal R}  = A + \ii B $ is reversible we have $ {\mathcal R}(-\vphi) = - \overline{{\mathcal R}}(\vphi) $, and
so  
the maps $ \vphi \mapsto A_j^{j'} (\vphi) $ are odd and  
so the average $ A_j^j(0) := \int_{\T^{\CS}} A_j^j(\vphi)\, d \vphi = 0 $ as well as $A_j^{- j}(0) = 0 $.
Hence 
$  {\mathcal R}_j^j(0) + {\mathcal R}_j^{- j}(0) = \ii (B_j^j(0) + B_j^{- j}(0) ) \in \ii \R $ and  each $ {\mathtt r}_j  \in \R $. 

Recalling the definition of $ {\mathfrak M}^\sharp (s_0) $ in \eqref{def:msharp} (with $s = s_0 $) and 
Defintion \ref{def:op-tame},  we have, for $\lambda = (\omega, \kappa)$,
for all $0 \leq |k| \leq k_0$, 
$ \| |\partial_\lambda^k {\mathcal R}| h \|_{s_0} \leq 2 \gamma^{- |k|} {\mathfrak M}^\sharp (s_0) \| h \|_{s_0} $, 
which implies that (see \eqref{tame-coeff})
$$  
|\partial_\lambda^k {\mathcal R}_j^j(0)| + |\partial_\lambda^k {\mathcal R}_j^{- j}(0)|  \leq C \gamma^{- |k|}  {\mathfrak M}^\sharp (s_0) \, . 
$$
Hence  
$$ |\mu_j^+ - \mu_j|^{k_0, \gamma} 
\leq |{\mathcal R}_j^j(0)|^{k_0, \gamma} +  |{\mathcal R}_j^{- j}(0)|^{k_0, \gamma} \leq C {\mathfrak M}^\sharp (s_0)\,.$$ 
The estimate \eqref{diff:r1r2} 
follows analogously by 
$$
|\Delta_{1 2} ({\mathcal R}_j^j(0) + {\mathcal R}_j^{- j}(0)) | \leq C \| | \Delta_{1 2} {\mathcal R} | \|_{{\mathcal L}(H^{s_0})} \, .
$$ 
This completes the proof of the lemma. 
\end{proof}

\subsection{The iteration}

Let $ \nu \geq 0 $ and suppose that the statements $({\bf S1})_{\nu} $-$({\bf S4})_{\nu} $ are true. 
We prove $({\bf S1})_{\nu + 1}$-$({\bf S4})_{\nu + 1}$.
\\[1mm]
{\sc Proof of $({\bf S1})_{\nu + 1}$}. Since the eigenvalues $\mu_j^\nu$ are defined on 
${\mathcal N}(\tLm_\nu^\gamma, \g N_{\nu-1}^{- \tau - 2})$,
the set  
$ \tLm_{\nu + 1}^\gamma $ 
is  well-defined. Moreover
   $\mu_j^\nu$ are well defined also 
  on the set
 $$
 {\mathcal N}(\tLm_{\nu + 1}^\g, \gamma N_\nu^{- \tau - 2}) 
 \subseteq {\mathcal N}(\tLm_\nu^\gamma, \gamma N_{\nu - 1}^{- \tau - 2})
 $$
because $\tLm_{\nu + 1}^\gamma \subseteq \tLm_\nu^\gamma$. 
  Let us prove \eqref{inclusione-insiemi-gamma2} at the step $ \nu + 1 $, namely that 
$$ 
{\mathcal N}(\tLm_{\nu + 1}^\gamma, \g N_\nu^{- \tau - 2}) \subset \tLm_{\nu + 1}^{\gamma / 2 } \, .
$$ 
Indeed, 
let  $\lambda_0 = (\omega_0, \kappa_0) \in \tLm_{\nu+1}^\gamma$ and $\lambda = (\omega, \kappa)$ with $|\lambda - \lambda_0| \leq \gamma N_\nu^{- \tau - 2}$. 
Then, for all $| \ell | \leq N_\nu $, $ j \neq j' $  (consider the case $ \varsigma = 1 $),  
$$
\begin{aligned}
|\omega \cdot \ell  + \mu_j^\nu (\lambda)- \mu_{j'}^\nu(\lambda)| & \geq 
 |\omega_0 \cdot \ell  + \mu_j^\nu(\lambda_0) - \mu_{j'}^\nu(\lambda_0)|  \\
 &   - |\omega - \omega_0| |\ell | - |(\mu_j^\nu - \mu_{j'}^\nu)(\lambda) - (\mu_j^\nu - \mu_{j'}^\nu)(\lambda_0)| \\
& \stackrel{\eqref{stima a 15}, \eqref{stime lambda 1 senza proiettore}, \eqref{stima rj nu} , \e \gamma^{- 2} \leq 1}{
\geq}  |\omega_0 \cdot \ell  + \mu_j^\nu(\omega_0) - \mu_{j'}^\nu(\omega_0)| \\
& - \big( |\ell  | 
+ C(S)  |j^{\frac32} - j'^{\frac32}| \big) |\lambda - \lambda_0| \\
& \geq  \gamma |j^{\frac32} - j'^{\frac32}| \langle \ell  \rangle^{- \tau} - \gamma N_\nu^{- \tau - 1} - 
C(S) \gamma |j^{\frac32} - j'^{\frac32}| N_\nu^{- \tau - 2}   \\
& \geq 
\frac{\gamma}{2} |j^{\frac32} - j'^{\frac32}|  \langle \ell \rangle^{-\tau}  
\end{aligned}
$$
for $N_0  > 4 C(S)$ large enough. Thus $ \lambda = (\om, \kappa) \in \tLm_{\nu + 1}^{\gamma / 2 } $ 
(defined in \eqref{Omega nu + 1 gamma} with $ \nu \rightsquigarrow \nu + 1 $ and $ \g \rightsquigarrow \g / 2 $).

By \eqref{inclusione-insiemi-gamma2} at the step $ \nu + 1 $ 
 and Lemma \ref{Homological equations tame}, for all $ (\om, \kappa) \in 
{\mathcal N}(\tLm_{\nu + 1}^\gamma, \g N_\nu^{- \tau - 2}) $ the solutions $\Psi_{\nu, m}$, $ m =1,2 $, 
of the homological equations 
 \eqref{homo1r}, defined
 in \eqref{shomo1}, \eqref{shomo2},    are well defined and, 
 by  \eqref{stima tame Psi}, \eqref{stima cal R nu},  
 satisfy for all $0 \leq |k| \leq k_0$, the estimates \eqref{tame Psi nu - 1} at $ \nu + 1 $.
In particular \eqref{tame Psi nu - 1} at $ \nu + 1 $  with $ k  = 0 $, $ s = s_0 $ imply 
\begin{equation}\label{costante tame Psi nu i (s0)}
{\mathfrak M}^\sharp_{\Psi_{\nu, m}}(s_0 ) \leq_{k_0, \mathtt b} 
 N_{\nu}^{\tau_1} N_{\nu - 1}^{- \mathtt a} \gamma^{- 1} {\mathfrak M}_0(s_0, {\mathtt b})  \, , \ \ m = 1, 2 \, .
\end{equation}
Therefore, by  \eqref{alpha beta}, \eqref{KAM smallness condition1}, the smallness condition \eqref{piccolezza neumann tame} of Lemma \ref{serie di neumann per maggioranti} is verified for $ N_0 := N_0 (S, {\mathtt b}) $ large enough and 
the map ${\bf \Phi}_\nu = {\mathbb I}_2^\bot + {\bf \Psi}_\nu $ is invertible.
Its inverse has the form 
\begin{equation}\label{inverso Phi nu}
{\bf \Phi}_\nu^{- 1} = {\mathbb I}_2^\bot + \check{\bf \Psi}_\nu\,, \qquad  \check{\bf \Psi}_\nu :=  \begin{pmatrix}
\check{\Psi}_{\nu, 1} & \check{\Psi}_{\nu, 2} \\
\overline{\check{\Psi}}_{\nu, 2} & \overline{\check{\Psi}}_{\nu, 1}
\end{pmatrix} 
\end{equation}
and, by Lemma \ref{serie di neumann per maggioranti}, the $ {\check \Psi}_{\nu, m} $ 
$ m = 1, 2 $, are $ {\mathcal D}^{k_0} $-modulo-tame with the same modulo-tame constants
 of $ \Psi_{\nu,m} $  (see \eqref{tame Psi nu - 1} for $ \nu +1$),  i.e. 
\be\label{stima tame Psi tilde}
\begin{aligned}
 {\mathfrak M}_{{\check \Psi}_{\nu, m}}^\sharp (s)  & \leq_{k_0, \mathtt b} \g^{-1} 
N_{\nu }^{\tau_1} N_{\nu - 1}^{- \mathtt a} {\mathfrak M}_0(s, \mathtt b) \, , \\
 {\mathfrak M}_{\langle \partial_\vphi \rangle^{\mathtt b}  {\check \Psi}_{\nu , m}}^\sharp (s) &  \leq_{k_0, \mathtt b} \g^{-1} 
N_{\nu }^{\tau_1} N_{\nu - 1}  {\mathfrak M}_0(s, \mathtt b) \, . 
\end{aligned}
\ee
Since ${\bf \Psi}_\nu$ is even and reversibility preserving, also $ \check{\bf \Psi}_\nu$ is even and reversibility preserving. 

By Lemma \ref{nuovadiagonale} the operator $ {\bf D}_{\nu+1} $ is diagonal and its  eigenvalues  
$$ 
\mu_j^{\nu+1} : {\mathcal N}(\tLm_{\nu+1}^\gamma, \gamma N_{\nu}^{-\tau - 2}) \to \R 
$$ 
satisfy \eqref{stima rj nu} at $ \nu + 1 $. 

Now we estimate the remainder (see \eqref{new-diag-new-rem})
$$
\begin{aligned}
& \qquad \qquad {\bf R}_{\nu + 1} + {\bf Q}_{\nu + 1} := {\bf \Phi}_\nu^{- 1} {\bf H}_\nu\,, \\
&  {\bf H}_\nu := \Pi_{N_\nu}^\bot {\bf R}_\nu + \Pi_{N_\nu}^\bot {\bf Q}_\nu + {\bf R}_\nu {\bf \Psi}_\nu - {\bf \Psi}_\nu [{\bf R}_\nu] + {\bf Q}_\nu {\bf \Psi}_\nu\,.
\end{aligned}
$$
By \eqref{inverso Phi nu}, \eqref{forma cal R nu}, \eqref{defPsi} we get  
\begin{equation}\label{bf R nu + 1 bf Q nu + 1}
{\bf R}_{\nu + 1} = \begin{pmatrix}
{\mathcal R}_{\nu + 1} & 0 \\
0 & \overline{\mathcal R}_{\nu + 1}
\end{pmatrix}\,, \qquad {\bf Q}_{\nu + 1} := \begin{pmatrix}
0 & {\mathcal Q}_{\nu + 1} \\
\overline{\mathcal Q}_{\nu + 1} & 0
\end{pmatrix} 
\end{equation}
where 
\begin{align}
{\mathcal R}_{\nu + 1} & := ({\rm Id} + \check \Psi_{\nu , 1}) (\Pi_{N_\nu}^\bot {\mathcal R}_\nu + {\mathcal R}_\nu \Psi_{\nu, 1} - \Psi_{\nu, 1} [{\mathcal R}_\nu]  + {\mathcal Q}_\nu \overline{\Psi}_{\nu, 2} ) \nonumber\\
& \quad  + \check{\Psi}_{\nu, 2} (\Pi_{N_\nu}^\bot {\mathcal Q}_\nu + {\mathcal R}_\nu \Psi_{\nu, 2} - \Psi_{\nu, 2} [\overline{\mathcal R}_\nu] +  {\mathcal Q}_\nu \overline{\Psi}_{\nu, 1})\,, \label{cal R nu + 1 tame} \\
{\mathcal Q}_{\nu + 1} & := ({\rm Id} + \check{\Psi}_{\nu, 1})(\Pi_{N}^\bot {\mathcal Q}_\nu + {\mathcal R}_\nu \Psi_{\nu, 2} - \Psi_{\nu, 2} [\overline{\mathcal R}_{\nu}] +  {\mathcal Q}_\nu \overline{\Psi}_{\nu, 1}) \nonumber\\
& \qquad + \Pi_{N}^\bot \overline{\mathcal R}_\nu + \overline{\mathcal R}_\nu \overline{\Psi}_{\nu, 1} - \overline{\Psi}_{\nu, 1} [\overline{\mathcal R}_\nu] +  \overline{\mathcal Q}_\nu {\Psi}_{\nu, 2} \, . \label{cal Q nu + 1 tame}
\end{align}

\begin{lemma}\label{estimate in low norm} {\bf (Nash-Moser iterative scheme)}
The operators $  {\mathcal R}_{\nu + 1} $, $  {\mathcal Q}_{\nu + 1} $  are $ {\mathcal D}^{k_0} $-modulo-tame with 
modulo-tame constants satisfying
\begin{equation}\label{schema quadratico tame}
{\mathfrak M}_{\nu + 1}^\sharp (s ) \leq_{k_0} N_\nu^{- {\mathtt b}} {\mathfrak M}_\nu^\sharp (s, {\mathtt b}) 
+  N_\nu^{\tau_1} \gamma^{- 1} {\mathfrak M}_\nu^\sharp (s) {\mathfrak M}_\nu^\sharp (s_0)\,.
\end{equation}
The operators $ \langle \partial_\vphi \rangle^{ {\mathtt b}} {\mathcal R}_{\nu + 1}, 
\langle \partial_\vphi \rangle^{ {\mathtt b}} {\mathcal Q}_{\nu + 1} $ are $ {\mathcal D}^{k_0} $-modulo-tame with 
modulo-tame constants satisfying 
\be\label{M+Ms}
\begin{aligned}
{\mathfrak M}_{\nu + 1}^\sharp (s, {\mathtt b})  
\leq_{k_0, \mathtt b} 
 {\mathfrak M}_\nu^\sharp (s, \mathtt b) & + 
 N_\nu^{\tau_1} \gamma^{- 1} {\mathfrak M}_\nu^\sharp (s, \mathtt b) 
{\mathfrak M}_\nu(s_0) \\ 
& + N_\nu^{\tau_1} \gamma^{- 1} {\mathfrak M}_\nu^\sharp (s_0, \mathtt b) {\mathfrak M}_\nu^\sharp (s)  \,.
\end{aligned}
\ee
\end{lemma}

\begin{proof}
We estimate each term in \eqref{cal R nu + 1 tame}-\eqref{cal Q nu + 1 tame}.
The proof of \eqref{schema quadratico tame} follows by Lemmata \ref{lemma:smoothing-tame},
\ref{interpolazione moduli parametri},
\eqref{stima tame Psi}, \eqref{stima tame Psi tilde}. 
The proof of \eqref{M+Ms} follows by Lemma \ref{interpolazione moduli parametri}  
\eqref{stima tame Psi}, \eqref{stima tame Psi tilde},
\eqref{stima cal R nu} and Lemma \ref{lemma:smoothing-tame}. 
\end{proof}

The estimates \eqref{schema quadratico tame}, \eqref{M+Ms},  and  \eqref{alpha beta},  
allow to prove  that also \eqref{stima cal R nu} holds at the step $ \nu + 1 $. 

\begin{lemma}\label{stima M nu + 1 K nu + 1}
$$ 
\begin{aligned}
& {\mathfrak M}_{\nu + 1}^	\sharp (s) \leq C_*(s_0, \mathtt b) N_\nu^{- \mathtt a} {\mathfrak M}_0(s, \mathtt b) \, ,  \\
& {\mathfrak M}_{\nu + 1}^\sharp (s,\mathtt b) \leq C_*(s_0, \mathtt b) N_\nu {\mathfrak M}_0(s, \mathtt b)  \, .
\end{aligned}
$$ 
\end{lemma}

\begin{proof}
By \eqref{schema quadratico tame} and \eqref{stima cal R nu} we get 
\begin{align*}
{\mathfrak M}_{\nu + 1}^\sharp (s) & \leq_{k_0} 
C_*(s_0, \mathtt b)N_{\nu}^{- \mathtt b} N_{\nu - 1} {\mathfrak M}_0(s, \mathtt b) \\ 
& \quad + C_*(s_0, \mathtt b)^2N_\nu^{\tau_1} \gamma^{- 1} 
{\mathfrak M}_0(s, \mathtt b) {\mathfrak M}_0(s_0, \mathtt b) N_{\nu - 1}^{- 2 \mathtt a}  \\
 & \leq C_*(s_0, \mathtt b) N_\nu^{- \mathtt a} {\mathfrak M}_0(s, \mathtt b)
\end{align*}
by \eqref{alpha beta}, \eqref{KAM smallness condition1} and taking  $N_0 := N_0(S, \mathtt b) > 0$ large enough.  
Then by \eqref{M+Ms}, \eqref{stima cal R nu}  we get that 
\begin{align}
{\mathfrak M}_{\nu + 1}^\sharp (s, \mathtt b) 
& \leq_{k_0, \mathtt b}N_{\nu - 1} {\mathfrak M}_0(s, \mathtt b) + N_\nu^{\tau_1} N_{\nu - 1}^{1 - \mathtt a} \gamma^{- 1} {\mathfrak M}_0(s, \mathtt b) {\mathfrak M}_0(s_0, \mathtt b) \nonumber \\
& \leq C_*(s_0, \mathtt b) N_\nu {\mathfrak M}_0(s, \mathtt b) \nonumber
\end{align}
by \eqref{alpha beta}, \eqref{KAM smallness condition1} and taking  $N_0 := N_0(S, \mathtt b) > 0$ large enough.  
\end{proof}

The proof of $({\bf S1})_{\nu + 1} $ is concluded by noting that 
the operators ${\bf R}_{\nu + 1}$, ${\bf Q}_{\nu + 1}$ 
are even and reversible because $ {\bf  \Phi}_\nu $ is even and reversibility preserving (Lemma
\ref{Homological equations tame}).

\smallskip

\noindent {\sc Proof of $({\bf S2})_{\nu + 1}$.}
We now construct the smooth extension $\tilde \mu_j^{\nu + 1}$ on all the parameter space $ \mathtt \Omega \times [\kappa_1, \kappa_2] $. 
By the inductive hyphothesis there exists a ${k_0}$-times differentiable function 
 $\tilde \mu_j^\nu: \mathtt \Omega \times [\kappa_1, \kappa_2] \mapsto \R $  such that 
$ \mu_j^\nu = \tilde \mu_j^\nu $ on $  \tLm_\nu^\gamma $ and  
$ \tilde \mu_j^\nu = 0 $ outside $  {\mathcal N}(\tLm_\nu^\gamma, \g N_{\nu-1}^{- \tau - 2}) $.
Note that all the sets $ \tLm_\nu^\gamma  $  in \eqref{Omega nu + 1 gamma} 
are defined by only {\it finitely} many non-resonance conditions, namely (for brevity we omit to write the 
sets $ \DC_{K_n}^\gamma \cap  \DC_{N_{\nu-1}}^\gamma $)
$$
\begin{aligned}
\tLm_\nu^\gamma  = \!\! \!\! \bigcap_{| \ell  | \leq N_{\nu - 1}, |j|, |j'| \leq C N_{\nu-1}^2} \!\!\!\!\!\!
 \Big\{ (\omega, \kappa) \in \tLm_{\nu - 1}^\gamma & : |\omega \cdot \ell  + \mu_j^{\nu - 1} - \varsigma 
\mu_{j'}^{\nu - 1}| \geq
 \frac{\gamma |j^{\frac32} - \varsigma j'^{\frac32}|}{\langle \ell \rangle^\tau} , \\
& \quad  j, j' \in {\mathbb S}_0^c\,, \varsigma  \in \{+, -  \} 
\Big\} \,. 
\end{aligned}
$$
Actually, provided  $j^{\frac12} + j'^{\frac12} \geq C N_{\nu -1}$, $ j \neq j' $, 
for all $(\omega, \kappa) \in \tLm_{\nu - 1}^\gamma$ the functions  
$$
\begin{aligned}
|\omega \cdot \ell  + \mu_j^{\nu - 1} - \mu_{j'}^{\nu - 1}| 
& \geq  |\mu_j^{\nu - 1} - \mu_{j'}^{\nu - 1}| - |\omega | |\ell  | \\
& \geq \frac12 |j^{\frac32} - j'^{\frac32}| - C | \ell  |  \geq  C (j^{\frac12} + j'^{\frac12}) - C N_{\nu-1} \geq \frac12\,. 
 \end{aligned}
$$
Since $\mu_j^{\nu + 1} = \mu_j^\nu + {\mathtt r}_j^\nu$ (defined on 
${\mathcal N}(\tLm_{\nu + 1}^\gamma, \g N_\nu^{- \tau - 2})$) we need only to extend the function $ {\mathtt r}_j^\nu$. 

Let  $\psi_\nu \in {\mathcal C}^\infty : \R^{\CS + 1}  \to \R $ be 
a cut-off function satisfying:  $0 \leq \psi_\nu \leq 1$, 
$$
\begin{aligned}
& \psi_\nu(\lambda) = 1\,, \, \forall \lambda \in  \tLm_{\nu + 1}^\gamma \,, \, 
{\rm supp}(\psi_\nu) \subseteq {\mathcal N}(\tLm_{\nu + 1}^\gamma, \g N_\nu^{- \tau - 2}), \\
& |\partial^k_{\lambda} \psi_\nu(\lambda)| \leq C(k) \big( N_\nu^{\tau + 2} \gamma^{- 1} \big)^{|k|}, \,  
\forall k \in \N^\nu \,,
\end{aligned}
$$
and thus 
$ |\psi_\nu|^{k_0, \gamma} \leq C(k_0) N_\nu^{(\tau + 2) k_0} $.  Hence, defining 
$ \tilde {\mathtt r}_j^\nu := \psi_\nu  {\mathtt r}_j^\nu $ and $ \tilde \mu_j^{\nu + 1} := \tilde \mu_j^\nu + \tilde {\mathtt r}_j^\nu $,  
we get the estimate 
$$ 
\begin{aligned}
|\tilde \mu_j^{\nu + 1} - \tilde \mu_j^\nu|^{k_0, \gamma} 
& \leq |\psi_\nu|^{k_0, \gamma} |{\mathtt r}_j^\nu|^{k_0, \gamma} \\
& \leq C(k_0) N_\nu^{(\tau + 2) k_0} {\mathfrak M}_\nu^\sharp (s_0) \leq \e \gamma^{- 1} C(k_0, S, {\mathtt b}) 
 N_\nu^{(\tau + 2) k_0} N_{\nu-1}^{-{\mathtt a}}
 \end{aligned}
$$ 
by Lemma \ref{nuovadiagonale}, \eqref{stima cal R nu} and \eqref{stima cal K0 I delta}.
This is \eqref{vicinanza autovalori estesi} at $ \nu + 1 $.
Summing we also get \eqref{autovalori finali} at the step $ \nu + 1 $. 

\smallskip

\noindent {\sc Proof of $({\bf S3})_{\nu + 1}$.} At the $\nu$-th step we have already constructed the operators 
$$
{\mathcal R}_\nu(i_m)\,, {\mathcal Q}_\nu(i_m)\,, \Psi_{\nu - 1, 1}(i_m)\,, \Psi_{\nu - 1, 2}(i_m)\,,\quad m = 1, 2\,,
$$
which are defined on $\tLm_\nu^{\gamma_1}(i_1) \cap \tLm_\nu^{\gamma_2}(i_2)$ and they satisfies \eqref{stima cal R nu}, \eqref{tame Psi nu - 1}.  
 We now estimate the operator $\Delta_{1 2} {\mathcal R}_{\nu + 1}$. The estimate for $\Delta_{1 2} {\mathcal Q}_{\nu + 1}$ is analogous. By Lemma \ref{Homological equations tame} we may construct
 the operators $\Psi_{\nu, 1}(i_1)$, $\Psi_{\nu, 2}(i_1)$, $\Psi_{\nu, 1}(i_2)$, $\Psi_{\nu, 2}(i_2) $, 
defined for all $\omega \in \tLm_{\nu + 1}^{\gamma_1}(i_1) \cap \tLm_{\nu + 1}^{\gamma_2}(i_2)$ and  
\begin{align}
\| |\langle \partial_\vphi \rangle^{\mathtt b}\Delta_{1 2} \Psi_{\nu, 1}| \|_{{\mathcal L}(H^{s_0})} 
& \stackrel{\eqref{stime delta 12 Psi alta}}{\leq_{\mathtt b}} C N_\nu^{2 \tau} \gamma^{- 1} \big( \| | \langle \partial_\vphi \rangle^{\mathtt b}{\mathcal R}_\nu(i_2)| \|_{{\mathcal L}(H^{s_0})} \| i_1 - i_2 \|_{s_0  + \mu(\mathtt b) + \sigma} \nonumber
\\ 
& + \| | \langle \partial_\vphi \rangle^{\mathtt b}\Delta_{1 2} {\mathcal R}_\nu| \|_{{\mathcal L}(H^{s_0})} \big) \nonumber \\
& \stackrel{\eqref{norma operatoriale costante tame}, \eqref{stima cal R nu}, \eqref{stima cal K0 I delta}}{\leq_{S, \mathtt b} 
} N_\nu^{2 \tau} N_{\nu - 1} \e \gamma^{- 2} \| i_1 - i_2\|_{s_0  + \mu(\mathtt b) + \sigma} \nonumber \\ 
& + N_\nu^{2 \tau} \gamma^{- 1} \| | \langle \partial_\vphi \rangle^{\mathtt b}\Delta_{1 2} {\mathcal R}_\nu| \|_{{\mathcal L}(H^{s_0})} \nonumber \\
& \stackrel{\eqref{stima R nu i1 i2 norma alta}}{\leq_{S, \mathtt b}} N_\nu^{2 \tau} N_{\nu - 1} \e \gamma^{- 2} \| i_1 - i_2\|_{ s_0  + \mu(\mathtt b) + \sigma} \label{Delta 12 Psi nu alta}
\end{align}
and by \eqref{stime delta 12 Psi bassa}, \eqref{norma operatoriale costante tame}, \eqref{stima cal R nu}, \eqref{stima R nu i1 i2} we get 
\begin{align}
\| |\Delta_{12} \Psi_{\nu, 1}| \|_{{\mathcal L}(H^{s_0})} & \leq_{S, \mathtt b} N_\nu^{2 \tau} N_{\nu - 1}^{- \mathtt a} \e \gamma^{- 2} \| i_1 - i_2\|_{s_0  + \mu(\mathtt b) + \sigma}\,. \label{Delta 12 Psi nu bassa}
\end{align}
Similarly one can prove that $ \Delta_{12} \Psi_{\nu, 2}$ satisfies \eqref{Delta 12 Psi nu alta}, \eqref{Delta 12 Psi nu bassa}. By \eqref{costante tame Psi nu i (s0)}, for $\e \gamma^{- 2}$ small enough, the smallness condition \eqref{piccolezza Psi i1 i2} is verified.
Therefore by \eqref{Delta 12 Psi nu alta}, \eqref{Delta 12 Psi nu bassa}, Lemma \ref{Delta 12 Phi inverso} and \eqref{stima tame Psi tilde}, \eqref{norma operatoriale costante tame} we get
\be
\begin{aligned}
& \| |\Delta_{1 2} \check \Psi_{\nu, 1}| \|_{{\mathcal L}(H^{s_0})},  \| |\Delta_{1 2} \check \Psi_{\nu, 2}| \|_{{\mathcal L}(H^{s_0})}
\\ 
 & \qquad  \leq_{S, \mathtt b} N_\nu^{2 \tau} N_{\nu - 1}^{- \mathtt a} \e \gamma^{- 2} \| i_1 - i_2\|_{s_0  + \mu(\mathtt b) + \sigma}\,, \label{Delta 12 Psi tilde nu} 
 \end{aligned}
 \ee
 \be \begin{aligned}
& 
\| | \langle \partial_\vphi \rangle^{\mathtt b}\Delta_{1 2} \check \Psi_{\nu, 1}| \|_{{\mathcal L}(H^{s_0})},
\| | \langle \partial_\vphi \rangle^{\mathtt b}\Delta_{1 2} \check \Psi_{\nu, 2}| \|_{{\mathcal L}(H^{s_0})} \\ 
& \qquad \leq_{S, \mathtt b} N_\nu^{2 \tau} N_{\nu - 1} \e \gamma^{- 2} \| i_1 - i_2\|_{s_0  + \mu(\mathtt b) + \sigma}\,.\label{Delta 12 Psi tilde nu alta}
\end{aligned}
\ee
We now
 estimate $\Delta_{12}{\mathcal R}_{\nu + 1}$
 where ${\mathcal R}_{\nu + 1}$ is defined in \eqref{cal R nu + 1 tame}. We consider the term 
$ {\mathcal R}_{\nu + 1}^{\star} := $ $  ({\rm Id} + \check \Psi_{\nu , 1}) (\Pi_{N_\nu}^\bot {\mathcal R}_\nu + {\mathcal R}_\nu \Psi_{\nu, 1}) $. 
The other terms in \eqref{cal R nu + 1 tame} satisfy the same estimate. One has 
\begin{align}
\Delta_{12} {\mathcal R}_{\nu + 1}^\star & = \Delta_{12} \check \Psi_{\nu, 1} \big( \Pi_{N_\nu}^\bot {\mathcal R}_\nu(i_1) + {\mathcal R}_\nu(i_1) \Psi_{\nu, 1}(i_1) \big) \nonumber\\
& \quad + ({\rm I d} + \check \Psi_{\nu, 1}(i_2)) \big( \Pi_{N_\nu}^\bot \Delta_{1 2} {\mathcal R}_\nu + \Delta_{1 2} {\mathcal R}_\nu \Psi_{\nu, 1}(i_1 ) + {\mathcal R}_\nu(i_2) \Delta_{1 2} \Psi_{\nu, 1} \big)\, .
\end{align}
Hence by Lemma \ref{proprieta norme operatoriali}, \eqref{Delta 12 Psi tilde nu}, \eqref{stima tame Psi tilde}, \eqref{stime delta 12 Psi bassa}, \eqref{norma operatoriale costante tame}, \eqref{stima tame Psi}, taking $\e \gamma^{- 2}$ small enough, we get 
\begin{align}
\| |\Delta_{1 2 } {\mathcal R}_{\nu + 1}^\star| \|_{{\mathcal L}(H^{s_0})} & 
{\leq_{\mathtt b}} \big( N_\nu^{- \mathtt b} {\mathfrak M}_\nu^\sharp(s_0, \mathtt b) + N_\nu^{\tau_1} \gamma^{- 1} {\mathfrak M}_\nu^\sharp(s_0)^2   \big) \| i_1 - i_2\|_{s_0  + \mu(\mathtt b) + \sigma}  + \nonumber\\
&  \!\!\!\! + N_\nu^{- \mathtt b} \| |\langle \partial_\vphi \rangle^{\mathtt b}\Delta_{12} {\mathcal R}_\nu| \|_{{\mathcal L}(H^{s_0})} + N_\nu^{\tau_1} \gamma^{- 1} {\mathfrak M}_\nu^{\sharp}(s_0)\||\Delta_{1 2} {\mathcal R}_\nu| \|_{{\mathcal L}(H^{s_0})}\,. \label{Delta 12 cal R nu + 1 bassa}  
\end{align}
Moreover, using also \eqref{Delta 12 Psi tilde nu alta}, \eqref{stime delta 12 Psi alta} and since  
\eqref{stima cal R nu}, \eqref{KAM smallness condition1} imply   
$  N_\nu^{\tau_1} \gamma^{- 1} {\mathfrak M}_\nu^\sharp(s_0) \leq 1 $, 
we get 
\begin{align}
\||\langle \partial_\vphi \rangle^{\mathtt b} \Delta_{1 2} {\mathcal R}_{\nu + 1}^\star | \|_{{\mathcal L}(H^{s_0})} & 
\leq_{S, \mathtt b} \big( \e \gamma^{- 1} N_{\nu - 1}  + {\mathfrak M}_\nu^\sharp (s_0, \mathtt b) \big) \| i_1 - i_2 \|_{s_0  + \mu(\mathtt b) + \sigma} 
\nonumber\\
& \!\!\!\!\!\!\!\!\!\!\!\!\!\! + \||\langle \partial_\vphi \rangle^{\mathtt b} \Delta_{1 2} {\mathcal R}_{\nu}| \|_{{\mathcal L}(H^{s_0})} + 
N_\nu^{\tau_1} \gamma^{- 1} \| |\Delta_{1 2} {\mathcal R}_\nu| \|_{{\mathcal L}(H^{s_0})} {\mathfrak M}_\nu^\sharp (s_0, \mathtt b)\,. 
\label{Delta 12 cal R nu + 1 alta}
\end{align}
The other terms in \eqref{cal R nu + 1 tame} may be estimated in the same way, whence $\Delta_{12 } {\mathcal R}_{\nu + 1}$ satisfies \eqref{Delta 12 cal R nu + 1 bassa}, \eqref{Delta 12 cal R nu + 1 alta}.  

We now prove \eqref{stima R nu i1 i2}, \eqref{stima R nu i1 i2 norma alta} at the step $\nu + 1$.  
By \eqref{Delta 12 cal R nu + 1 bassa}, \eqref{stima cal R nu}, 
\eqref{def:costanti iniziali tame}, \eqref{stima R nu i1 i2}, \eqref{stima R nu i1 i2 norma alta} we get
\begin{align*}
\| |\Delta_{1 2} {\mathcal R}_{\nu + 1}| \|_{{\mathcal L}(H^{s_0})} 
& {\leq_{S, \mathtt b}} \big( \e \gamma^{- 1} N_{\nu - 1}N_\nu^{- \mathtt b} + N_\nu^{\tau_1}  \e^2 \gamma^{- 3} N_{\nu - 1}^{- 2 \mathtt a}  \big) \| i_1 - i_2 \|_{s_0  + \mu(\mathtt b) + \sigma} \\
& \stackrel{\eqref{alpha beta}} {\leq_{S, \mathtt b}}  \e \gamma^{- 1} N_\nu^{-\mathtt a} \| i_1 - i_2 \|_{s_0  + \mu(\mathtt b) + \sigma}\,.
\end{align*}
for  $\e \gamma^{- 2}  \leq 1 $ and $ N_0(S, {\mathtt b}) > 0 $ large.
 Hence  \eqref{stima R nu i1 i2} at the step $\nu + 1$ is proved. 
Similarly, by \eqref{Delta 12 cal R nu + 1 alta}, 
\eqref{stima cal R nu}, \eqref{def:costanti iniziali tame}, \eqref{stima R nu i1 i2}, \eqref{stima R nu i1 i2 norma alta},  we get 
$$
\begin{aligned}
\| |\langle \partial_\vphi \rangle^{\mathtt b} \Delta_{1 2} {\mathcal R}_{\nu + 1}| \|_{{\mathcal L}(H^{s_0})} 
& 
\leq_{S, \mathtt b}  \e \gamma^{- 1} N_{\nu - 1} \big( 1 + \e \gamma^{- 2} N_\nu^{\tau_1} 
N_{\nu - 1}^{- \mathtt a}  \big) \| i_1 - i_2 \|_{s_0  + \mu(\mathtt b) + \sigma}  \\
& \leq_{S, \mathtt b}\e \gamma^{- 1} N_{\nu } \| i_1 - i_2 \|_{s_0  + \mu(\mathtt b) + \sigma}
 \end{aligned}
$$
by \eqref{alpha beta},  $\e \gamma^{- 2} \leq 1 $ and taking $N_0 := N_0(S, \mathtt b) > 0$ large. Thus   
  \eqref{stima R nu i1 i2 norma alta} at the step $\nu + 1$ is proved. 

The proof of \eqref{r nu - 1 r nu i1 i2} at the step $\nu + 1$ follows by Lemma \ref{nuovadiagonale}. The estimate \eqref{r nu i1 - r nu i2} follows by a telescopic argument using \eqref{r nu - 1 r nu i1 i2} and \eqref{stima R nu i1 i2}. 

\smallskip

\noindent {\sc Proof of $({\bf S4})_{\nu + 1}$.} 
The proof is the same as that of $({\bf S4})_{\nu + 1}$ of Theorem 4.2 in \cite{BBM-Airy}. 
It uses $ ({\bf S3})_\nu $. 
 \qed

\section{Almost-invertibility  of  $ {\mathcal L}_\om $}\label{quasi invertibilita}

By \eqref{final conjugation prima del KAM} and  Theorem \ref{Teorema di riducibilita} (applied to
$ {\bf L}_0 = {\mathcal L}_M^{(3)} $) we obtain
\begin{equation}\label{final semi conjugation}
{\mathcal L}_\omega =  {\bf W}_{2, n}  {\bf L}_n  {\bf W}_{1, n}^{- 1} + {\bf R}_M^{(3), \bot}\,, 
\quad {\bf W}_{1, n} := {\mathcal W}_{1}^\bot {\bf U}_n\,, \quad {\bf W}_{2, n} := {\mathcal W}_{2}^\bot {\bf U}_n\,, 
\end{equation}
where 
 the operator ${\bf L}_n$ is defined in \eqref{cal L infinito} and ${\bf R}_M^{(3), \bot}$ 
 (defined in \eqref{final conjugation prima del KAM}) 
 satisfies the estimates \eqref{stima bf R N (3) bot bassa}, \eqref{stima bf R N (3) bot alta}. 
 Then \eqref{stima Phi 1 Phi 2 proiettate}, \eqref{stima Phi infinito}, 
 \eqref{relazione mathtt b N}, \eqref{definizione bf c (beta)},  
  imply that for all $s_0 \leq s \leq S$
\begin{equation}\label{stime W1 W2}
\| {\bf W}_1^{\pm 1} h \|_s^{k_0, \gamma}, \, 
\| {\bf W}_2^{\pm 1} h \|_s^{k_0, \gamma} \leq_S \| h \|_{s + \sigma}^{k_0, \gamma}  + 
\| \fracchi \|_{s + \mu(\mathtt b) + \sigma}^{k_0, \gamma} \| h \|_{s_0 + \sigma}^{k_0, \gamma} 
\end{equation}
for some $\sigma := \sigma(\tau, \nu, k_0) > 0 $ where $ \nu = \CS $ as used in the whole paper.

In order to verify the inversion assumption \eqref{inversion assumption}-\eqref{tame inverse} required to 
construct an approximate inverse (and thus define the successive approximate solution of the Nash-Moser non-linear iteration) 
we decompose the operator $ {\bf L}_n $ in \eqref{cal L infinito} as 
\begin{equation}\label{decomposizione bf Ln}
{\bf L}_n = {\bf D}_n^{<}  + {\bf R}_n^\bot + {\bf R}_n  + {\bf Q}_n
\end{equation}
where 
\be\label{Rn-bot}
\begin{aligned}
& {\bf D}_n^{<}  := \Pi_{K_n} \big( \Dom {\mathbb I}_2^\bot + \ii {\bf D}_n \big) \Pi_{K_n} 
+ \Pi_{K_n}^\bot \, , \\
&  {\bf R}_n^\bot := 
\Pi_{K_n}^\bot \big( \Dom {\mathbb I}_2^\bot + \ii {\bf D}_n \big) \Pi_{K_n}^\bot - \Pi_{K_n}^\bot \, , 
\end{aligned}
\end{equation}
the diagonal operator $ {\bf D}_n $ are defined in \eqref{cal L nu} (with $ \nu = n $), and 
the constant $ K_n $  in \eqref{definizione Kn}.

\begin{lemma} {\bf (First order Melnikov non-resonance conditions)}
For all $ \lambda = (\om, \kappa)  $ in 
\begin{equation}\label{prime di melnikov}
\begin{aligned}
& {\mathtt \Lambda}_{n + 1}^{\gamma, I}  :=  {\mathtt \Lambda}_{n + 1}^{\gamma, I} ( i ) := \\
& 
\big\{ \lambda \in  \tLm_{n + 1}^\gamma : |\omega \cdot \ell  
+  \mu_j^n| \geq  2\gamma j^{\frac32} \langle \ell  \rangle^{- \tau}\,,
\  \forall | \ell  | \leq K_n\,,\,\, j \in \N \setminus {\mathbb S}^+ \big\}
\end{aligned}
\end{equation}
(recall \eqref{Cantor set}), the operator 
$ {\bf D}_n^< $ in \eqref{Rn-bot} 
is invertible and 
\begin{equation}\label{stima tilde cal Dn}
\| ({\bf D}_n^<)^{- 1} g \|_s^{k_0, \gamma} \leq_{k_0} \gamma^{- 1} \| g \|_{s + \t_1}^{k_0, \gamma}\,, 
\quad \tau_1 := \tau + k_0 (\tau + 1) \, .
\end{equation}
\end{lemma}

\begin{proof}
The estimate \eqref{stima tilde cal Dn} follows by 
$$ 
\big|\partial_\lambda^k ( \omega \cdot \ell  + \mu_j^n(\lambda) )^{-1}  \big| \leq  C(k) 
\langle \ell  \rangle^{\tau (|k| + 1) + |k|} \gamma^{- (|k| + 1)}, 
$$ 
$ \forall |k| \leq k_0 $. 
\end{proof}

Standard smoothing properties imply that 
the operator ${\bf R}_n^\bot $ defined in \eqref{Rn-bot} satisfies, for all $ b  > 0$,   
\begin{equation}\label{stima tilde cal Rn}
\| {\bf R}_n^\bot h \|_{s_0}^{k_0, \gamma} \lessdot K_n^{- b} \| h \|_{s_0 + b + \frac32}^{k_0, \gamma}\,,\quad \| {\bf R}_n^\bot h\|_s^{k_0, \gamma} \lessdot \| h \|_{s + \frac32}^{k_0, \gamma} \, .
\end{equation}
By the decompositions \eqref{final semi conjugation}, \eqref{decomposizione bf Ln}, 
Theorem \ref{Teorema di riducibilita}, Proposition \ref{prop: sintesi linearized}, the estimates \eqref{stima tilde cal Dn}, \eqref{stima tilde cal Rn}, \eqref{stime W1 W2} we  deduce the following theorem:

\begin{theorem}\label{inversione parziale cal L omega}
{\bf (Almost invertibility of $ {\mathcal L}_\om $)}
Assume  \eqref{ansatz 0} and that, for all $S > s_0 $,  the smallness condition \eqref{ansatz riducibilita} holds.
Let $ {\mathtt a}, {\mathtt b} $ as in \eqref{alpha beta}. 
Then for all 
\begin{equation}\label{Melnikov-invert}
(\omega, \kappa) \in  {\bf \Lambda}_{n + 1}^{\g}  := {\bf \Lambda}_{n + 1}^{\g} (i) 
 :=  \tLm_{n + 1}^\gamma  \cap  {\mathtt \Lambda}_{n + 1}^{\gamma, I}
\end{equation}
(see \eqref{Cantor set}, \eqref{prime di melnikov}) 
the operator $ {\mathcal L}_\omega$ defined in \eqref{Lomega def} (see also \eqref{representation Lom}) 
can be decomposed as 
\be
\begin{aligned} \label{splitting cal L omega}
& \, \qquad {\mathcal L}_\omega  = {\mathbf L}_\omega + {\mathbf R}_\omega +  {\mathbf R}_\omega^\bot \,, \qquad 
{\mathbf L}_\omega := 
{\bf W}_{2, n} {\bf D}_n^< {\bf W}_{1, n}^{- 1}\,, \\
&  {\mathbf R}_\omega := {\bf W}_{2, n} ({\bf R}_n + {\bf Q}_n ) {\bf W}_{1, n}^{- 1}\,,\quad {\mathbf R}_\omega^\bot := {\bf W}_{2, n} {\bf R}_n^\bot {\bf W}_{1, n}^{- 1} + {\bf R}_M^{(3), \bot}\, , 
\end{aligned}
\ee
where $ {\mathbf L}_\omega $ is invertible and, for some $ \s := \s(\nu, \tau, k_0) >  0 $, for all $ s_0 \leq s \leq S $,  $ g \in H^{s+\sigma} $,   
\begin{equation}\label{stima L omega corsivo}
\| {\mathbf L}_\omega^{- 1} g \|_s^{k_0, \gamma} \leq_S  
\gamma^{- 1} \big( \| g \|_{s + \sigma}^{k_0, \gamma} + \| \fracchi_0 \|_ {s + \sigma + \mu (\mathtt b)}^{k_0, \gamma} \| g \|_{s_0 + \sigma}^{k_0, \gamma}\big) 
\end{equation}
(with $ \mu (\mathtt b) $  defined in \eqref{definizione bf c (beta)})  and 
\begin{align}  \label{stima R omega corsivo}
\|{\mathbf R}_\omega h \|_s^{k_0, \gamma} & 
\leq_S  \e \gamma^{- 1} N_{n - 1}^{- {\mathtt a}}\big( \|  h \|_{s + \sigma}^{k_0, \gamma} +  \| \fracchi_0 \|_ {s + \sigma + \mu (\mathtt b)}^{k_0, \gamma} \| h \|_{s_0 + \sigma}^{k_0, \gamma} \big)\, , \\
\label{stima R omega bot corsivo bassa}
\| {\mathbf R}_\omega^\bot h \|_{s_0}^{k_0, \gamma} & \leq_S
K_n^{- b} \big( \| h \|_{s_0 + b 
+ \sigma}^{k_0, \gamma} + 
\| \fracchi_0 \|_ {s + \sigma + \mu (\mathtt b) + b}^{k_0, \gamma}  \|  h \|_{s_0 + \sigma}^{k_0, \gamma}\big)\,, 
\qquad \forall b > 0\,, \\
\label{stima R omega bot corsivo alta}
\| {\mathbf R}_\omega^\bot h \|_s^{k_0, \gamma} & \leq_S 
 \|  h \|_{s + \sigma}^{k_0, \gamma} + \| \fracchi_0 \|_ {s + \sigma + \mu (\mathtt b)}^{k_0, \gamma} \| h \|_{s_0 + \sigma}^{k_0, \gamma} \,.
\end{align}
 \end{theorem}
We point out that the above Theorem proves the almost-invertibility assumption \eqref{inversion assumption}-\eqref{tame inverse} that we stated in section \ref{sezione almost approximate inverse} and from which we deduce in Theorem \ref{thm:stima inverso approssimato}  the existence of an almost-approximate inverse of the linearized operator $d_{i, \alpha} {\mathcal F}(i_0)$. 
 
\medskip
 
\noindent
We finally remark that the operators
\be\label{defW1W2}
{\bf W}_{1, \infty} := {\mathcal W}_1^\bot {\bf U}_\infty \, , 
\qquad 
{\bf W}_{2, \infty} := {\mathcal W}_2^\bot {\bf U}_\infty \quad  {\rm where} \quad  
{\bf U}_\infty := \lim_{n \to + \infty} {\bf U}_n 
\ee
see \eqref{defUn}, and  $ {\mathcal W}_1^\bot $, ${\mathcal W}_2^\bot $ are defined in 
\eqref{Phi 1 Phi 2 proiettate}, \eqref{semiconiugio cal L8} completely diagonalize the linearized operator
$ {\mathcal L}_\om $ defined in \eqref{Lomega def}. 
 We deduce that
  $ {\bf W}_{1, \infty}(\vphi ) $, $  {\bf W}_{2, \infty}(\vphi ) $ satisfy the tame estimates \eqref{W12-ben-poste}-\eqref{W12-ben-poste-inverse} by small modifications of the arguments of sections \ref{linearizzato siti normali}-\ref{sec: reducibility}. 
  
\chapter{The Nash-Moser iteration}\label{sec:NM}
 
In this section we prove Theorem \ref{MAINTHEOREM}. It will be a consequence of Theorem \ref{iterazione-non-lineare} below
where we construct iteratively a sequence of better and better approximate solutions of the operator
$ {\mathcal F} ( i, \alpha) $ defined in \eqref{operatorF}. 
We consider the finite-dimensional subspaces 
$$
E_n := \Big\{ \fracchi (\vphi ) = (\Theta , I , z) (\vphi) , \ \  
\Theta = \Pi_n \Theta, \ I = \Pi_n I, \ z = \Pi_n z \Big\}
$$
where $ \Pi_n  $ is the projector
\be\label{truncation NM}
\begin{aligned}
& \Pi_n := \Pi_{K_n} : \\ 
& z(\ph,x) = \sum_{\ell  \in \Z^\nu,  j \in {\mathbb S}_0^c} z_{\ell, j} e^{\ii (\ell  \cdot \ph + jx)} 
\ \mapsto \Pi_n z(\ph,x) 
:= \sum_{|(\ell ,j)| \leq K_n} 
z_{\ell,  j} e^{\ii (\ell  \cdot \ph + jx)}  
\end{aligned}
\ee
with $ K_n = K_0^{\chi^n} $ (see \eqref{definizione Kn} and \eqref{NnKn}) and we 
 denote with the same symbol also 
$ \Pi_n p(\ph) :=  \sum_{|\ell | \leq K_n}  p_\ell e^{\ii \ell  \cdot \ph} $. 

We also define $ \Pi_n^\bot := {\rm Id} - \Pi_n $.  
The projectors $ \Pi_n $, $ \Pi_n^\bot$ satisfy the smoothing properties \eqref{smoothing-u1} for the weighted Sobolev norm defined in \eqref{norma pesata derivate funzioni}.

In view of the Nash-Moser Theorem \ref{iterazione-non-lineare} we introduce the constants 
\be
\begin{aligned}\label{costanti nash moser}
&  {\mathtt a}_1 :=  {\rm max}\{6  \sigma_1 + 13, \chi (p k_0 (\tau + 2) + p \tau + \mu({\mathtt b}) + 2 \sigma_1 ) +1 \},  \\
&  \mathtt a_2 := \chi^{- 1} \mathtt a_1 - p k_0 (\tau + 2) - \mu(\mathtt b) - 2 \sigma_1 
\end{aligned}  
\ee
and
\begin{align}  
& \!\! \!  {\mathtt b}_1 := {\mathtt a}_1 + \mu({\mathtt b}) + 3 \sigma_1 + 3 +  \chi^{-1} \mu_1 \,,  
\quad  \mu_1 := 3( \mu({\mathtt b}) + 2\sigma_1  ) + 1 \, ,  \quad \chi = 3/ 2 \, , 
\label{costanti nash moser 1} \\
& \!\! \!  \label{costanti nash moser 2}
 \sigma_1 := \max \{ \bar \sigma\,, \sigma\,, s_0 + 2 k_0 + 5 \}\,,
\end{align}
where 
 $\bar \sigma := \bar \sigma(\tau, \nu, k_0) > 0$ is defined in Theorem \ref{thm:stima inverso approssimato}, $\sigma = \sigma(\tau, \nu, k_0) > 0$ is the constant which appears in Theorem \ref{ITERAZIONERIDUCIBILITA}-${\bf(S3)_{\nu}}$-${\bf(S4)_{\nu}}$, $ s_0 + 2 k_0 + 5 $ is the largest loss of regularity in the estimates of the Hamiltonian vector field $X_P$ in Lemma \ref{lemma quantitativo forma normale}, 
 $\mu(\mathtt b)$  in \eqref{definizione bf c (beta)},  the constant
 $ {\mathtt b} := [{\mathtt a}] + 2 \in \N $ where $ {\mathtt a} $ 
is defined in \eqref{alpha beta}, and the exponent 
$ p $ in \eqref{NnKn} satisfies  
\be\label{cond-su-p}
p {\mathtt a} > (\chi - 1 ) {\mathtt a}_1 + \chi \sigma_1 = \frac12 {\mathtt a}_1 + \frac32 \sigma_1 \, . 
\ee
By remark \ref{remark:ab} the constant $ {\mathtt a} \geq \chi k_0 ( \tau + 2 ) + 1 $. Hence, by the definition 
of $ {\mathtt a}_1 $ in \eqref{costanti nash moser}, there exists 
$ p := p(\tau, \nu, k_0) $ such that  \eqref{cond-su-p} holds. For example we fix
\be\label{choice:p}
p := \max \Big\{  \frac{5 \s_1+ 7}{ \chi k_0 (\tau+2) + 1 }, \frac{ \chi ( \mu ({\mathtt b}) + 2 \s_1)+1}{ \chi k_0  + 1 } \Big\} \, . 
\ee
\begin{remark}\label{choice-a-b}
The constant $ {\mathtt a}_1 $ is the exponent in \eqref{P2n}. 
The constant $ {\mathtt a}_2 $ is the exponent in \eqref{Hn}. 
The constant $ \mu_1 $ is the exponent in $({\mathcal P}3)_{n}$. 
The conditions 
$  {\mathtt a}_1 >   (2 \s_1 + 4)\chi \slash (2- \chi) = 6 \sigma_1 + 12 $,
$ {\mathtt b}_1 > {\mathtt a}_1 + \mu({\mathtt b}) + 3 \sigma_1 + 2 +  \chi^{-1} \mu_1 $,  as well 
as $ p {\mathtt a} > (\chi - 1 ) {\mathtt a}_1 + \chi \sigma_1 $,
$ \mu_1 >  ( \mu ({\mathtt b}) + 2 \s_1 ) \chi \slash (\chi - 1 ) = 3 ( \mu ({\mathtt b}) + 2 \s_1) $
arise for the convergence of the iterative scheme \eqref{F(U n+1) norma bassa}-\eqref{U n+1 alta}, 
see Lemma \ref{lemma:quadra}. 
In addition we require 
$\mathtt a_1 \geq \chi (p k_0 (\tau + 2) + \mu({\mathtt b}) + 2 \sigma_1) + \chi p \tau + 1 $ so that 
$\mathtt a_2 > p \tau $, more precisely  
$\mathtt a_2 \geq p \tau + \chi^{-1} $. 
This condition is used in the proof of Lemma \ref{lemma inclusione cantor riccardo 1}.
\end{remark}

In this section, given a function  
$$
W = ( \fracchi, \beta ): 
{\mathtt \Lambda}_0 \to \big(H^{s}_\vphi  \times H^{s}_\vphi \times H^{s}\big)  \times \R^\nu \, , 
\quad \lambda \mapsto W (\lambda) = ( \fracchi (\lambda), \beta (\lambda)) ,  
$$
where
$   \fracchi (\lambda) \in H^{s}_\vphi  \times H^{s}_\vphi \times H^{s} $ is defined as in \eqref{componente periodica},
we denote 
$$ 
\|  W \|_{s}^{k_0, \gamma} = \|  \fracchi \|_{s}^{k_0, \gamma} + |  \beta |^{k_0, \gamma} \, .
$$ 

\begin{theorem}\label{iterazione-non-lineare} 
{\bf (Nash-Moser)} 
There exist $ \d_0$, $ C_* > 0 $, such that, if
\begin{equation}\label{nash moser smallness condition}  
\begin{aligned}
& K_0^{\tau_2} \e \g^{-2} < \d_0 , \quad \tau_2 := {\rm max}\{ p \tau_0, 2 \sigma_1 + {\mathtt a}_1 + 4 \} \, , \\
&  K_0 := \gamma^{- 1}, \quad \gamma:= \e^a\,,\quad 0 < a < \frac{1}{2 + \tau_2}\,,
\end{aligned}
\end{equation}
where $ \tau_0 := \tau_0(\tau, \nu)$ is  defined in Theorem \ref{ITERAZIONERIDUCIBILITA}, 
 then, for all $ n \geq 0 $: 
\begin{itemize}
\item[$({\mathcal P}1)_{n}$] 
there exists a   ${k_0}$-times differentiable function $\tilde W_n : \R^\nu  \times [\kappa_1, \kappa_2] 
\to E_{n -1} \times \R^\nu $, $ \lambda = (\om, \kappa) \mapsto \tilde W_n (\lambda) 
:=  (\tilde \fracchi_n, \tilde \alpha_n - \om ) $, for  $ n \geq 1$,  
and  $\tilde W_0 := 0 $,  satisfying 
\begin{equation}\label{ansatz induttivi nell'iterazione}
\| \tilde W_n \|_{s_0 + \mu({\mathtt b}) + \sigma_1}^{k_0, \gamma} \leq C_*  K_0^{p k_0 (\tau+2)} \e  \gamma^{-1}\,. 
\end{equation}
Let $\tilde U_n := U_0 + \tilde W_n$ where $ U_0 := (\vphi,0,0, \om )$.
The difference $\tilde H_n := \tilde U_{n} - \tilde U_{n-1}$, $ n \geq 1 $,  satisfies
\begin{equation}  \label{Hn}
\begin{aligned}
& \|\tilde H_1 \|_{s_0 + \mu({\mathtt b}) + \sigma_1}^{k_0, \gamma} \leq	 
C_* \e \gamma^{- 1} K_0^{p k_0 (\tau+2)}\,, \\
&  \| \tilde H_{n} \|_{ s_0 + \mu({\mathtt b}) + \sigma_1}^{k_0, \gamma} \leq C_* \e \gamma^{-1} K_{n - 1}^{- \mathtt a_2} \,,\quad \forall n > 1\, . 
\end{aligned}
\end{equation}
\item[$({\mathcal P}2)_{n}$]   Setting  $ {\tilde \imath}_n := (\vphi, 0, 0) + \tilde \fracchi_n $ we define 
\be\label{def:cal-Gn}
{\mathcal G}_{0} := \tOm \times [\kappa_1, \kappa_2]\,,\quad {\mathcal G}_{n+1}  :=  {\mathcal G}_n \bigcap {\bf \Lambda}_{n + 1 }^{ \gamma}({\tilde \imath}_n)
\,, \quad n \geq 0 \, , 
\ee
where $  {\bf \Lambda}_{n + 1}^{ \gamma}({\tilde \imath}_n) $ is defined in \eqref{Melnikov-invert}. 

Then, for all $\lambda = (\omega, \kappa)$ in 
${\mathcal N}({\mathcal G}_{n}, \gamma  K_{n-1}^{-p(\tau+2)} ) $,  setting  $ \g_{-1} = \gamma $ and  $ K_{-1} := 1 $, we have 
\be\label{P2n}
\| {\mathcal F}(\tilde U_n) \|_{ s_{0}}^{k_0, \gamma}  \leq C_* \e K_{n - 1}^{- {\mathtt a}_1} \, .
\ee
\item[$({\mathcal P}3)_{n}$] \emph{(High norms).} 
$$ \| \tilde W_n \|_{ s_{0}+ {\mathtt b}_1}^{k_0, \gamma} 
\leq C_* \e \gamma^{-1}  K_{n - 1}^{\mu_1}
$$ 
for all $\lambda = (\omega, \kappa) \in $ 
$ {\mathcal N}({\mathcal G}_{n},  \gamma  K_{n-1}^{-p(\tau+2)})$.
\end{itemize}
\end{theorem}

\begin{proof}
To simplify notation, in this proof we denote $\| \, \|^{k_0, \gamma}$ by $\| \, \|$. 

\smallskip

{\sc Step 1:} \emph{Proof of} $({\mathcal P}1, 2, 3)_0$.
They follow by  $\|{\mathcal F}(U_0) \|_s = O(\e)$ and taking $C_* $ large enough.

\smallskip

{\sc Step 2:} \emph{Assume that $({\mathcal P}1,2,3)_n$ hold for some $n \geq 0$, and prove $({\mathcal P}1,2,3)_{n+1}$.}
We are going to define the successive approximation $ \tilde U_{n+1} $ by a modified Nash-Moser scheme.
For that we prove the almost-approximate invertibility of  the linearized operator 
$$
L_n := L_n(\lambda) := d_{i,\a} {\mathcal F}({\tilde \imath}_n(\lambda)) 
$$ 
 applying Theorem \ref{thm:stima inverso approssimato} to ${ L}_n(\lambda) $.
The verification of the inversion assumption \eqref{inversion assumption}-\eqref{tame inverse} is the purpose  
of  Theorem  \ref{inversione parziale cal L omega} that we apply 
with 
$ i = \tilde \imath_n $.   
By \eqref{nash moser smallness condition}   the smallness condition 
\eqref{ansatz riducibilita} 
holds for $ \e $ small enough. 
Therefore Theorem  \ref{inversione parziale cal L omega}  
applies, and we deduce  that the inversion assumption \eqref{inversion assumption}-\eqref{tame inverse} holds for all 
$ \lambda \in {\bf \Lambda}_{n + 1}^{\gamma / 2}(\tilde \imath_n)$, see \eqref{Melnikov-invert}.
Actually  the inversion assumption holds
for all $\lambda \in $ $ {\mathcal N}({\bf \Lambda}_{n + 1}^{\gamma}(\tilde \imath_n), 2 \gamma K_n^{- p(\tau + 2)})$ because
$$
 {\mathcal N}\big( {\bf \Lambda}_{n+1}^{ \gamma }({\tilde \imath}_{n }), 2 \gamma K_n^{- p (\tau +2)}\big)
 \subseteq {\bf \Lambda}_{n +1}^{\gamma / 2}(\tilde \imath_{n }) \, ,  \quad \forall n \geq 0  \, ,
$$
which is a consequence of \eqref{inclusione-insiemi-gamma2} and 
the similar inclusion
$$  
{\mathcal N} ({\mathtt \Lambda}_{n+1}^{ \gamma, I }({\tilde \imath}_{n }), 2 \gamma K_n^{- p (\tau +2)}\big)
 \subseteq {\mathtt \Lambda}_{n +1}^{\gamma / 2 , I}(\tilde \imath_{n }) \, .
 $$ 
Now we apply Theorem \ref{thm:stima inverso approssimato} to the linearized operator 
$ L_n(\lambda) $  with 
$$ 
\tLm_o = 
{\mathcal N}({\bf \Lambda}_{n + 1}^{\gamma}(\tilde \imath_n), 2 \gamma K_n^{- p(\tau + 2)})
$$
 and 
\begin{equation}\label{valore finalissimo S}
S := s_0 + \mathtt b_1 \quad \text{where} \  \mathtt b_1 \  \text{is  defined in  \eqref{costanti nash moser 1}} \, . 
\end{equation}
It implies the existence of
an almost-approximate inverse ${\bf T}_n  := { \bf T}_n (\lambda, {\tilde \imath}_n(\lambda))$ which 
satisfies
\begin{align}
& \| {\bf T}_n g  \|_s 
 \leq_{s_0 + \mathtt b_1} \gamma^{-1} \big( \| g \|_{s + \sigma_1} 
+ \| \tilde \fracchi_n \|_{s + \sigma_1 + \mu({\mathtt b})} \| g \|_{s_0+ \sigma_1}\big)\,, \,  \forall s_0 < s \leq s_0 + \mathtt b_1  \label{stima Tn} \\
& \| {\bf T}_n  g \|_{s_0}  \leq_{s_0 + \mathtt b_1} \gamma^{-1} \| g \|_{s_0 + \sigma_1}\,.   \label{stima Tn norma bassa}
\end{align}
For all 
\be\label{inclusioni:Cantor}
\lambda \in {\mathcal N}({\mathcal G}_{n+1}, 2 \gamma K_n^{- p(\tau + 2)}) \subset 
 {\mathcal N}({\mathcal G}_{n}, \gamma K_{n-1}^{- p(\tau + 2)} ) 
\, , \ n \geq 0 \, , 
\ee 
we define the successive approximation 
\begin{equation}\label{soluzioni approssimate}
\begin{aligned}
&  U_{n + 1} := \tilde U_n + H_{n + 1} \, , \\
& H_{n + 1} :=
( \widehat \fracchi_{n+1}, \widehat \alpha_{n+1}) :=  - {\bf \Pi}_{n } {\bf T}_n \Pi_{n } {\mathcal F}(\tilde U_n) 
\in E_n \times \R^\nu  
\end{aligned}
\end{equation}
where  $ {\bf \Pi}_n $ is defined by (see \eqref{truncation NM})
\be\label{proiettore modificato}
 {\bf \Pi}_n ({\fracchi}, \alpha) := (\Pi_n \fracchi, \alpha)\,,
 \quad {\bf \Pi}_n^\bot (\fracchi, \alpha) := (\Pi_n^\bot \fracchi, 0)\,,\quad \forall (\fracchi, \alpha)\,.
\ee
We now show that the iterative scheme in \eqref{soluzioni approssimate} is rapidly converging. 
We write  
$$ 
{\mathcal F}(U_{n + 1}) =  {\mathcal F}(\tilde U_n) + L_n H_{n + 1} + Q_n 
$$ 
where $ L_n := d_{i,\alpha} {\mathcal F}({\tilde \imath}_n) $ and  
\begin{equation}\label{def:Qn}
\begin{aligned}
& Q_n := Q(\tilde U_n, H_{n + 1}) \, , \\ 
& Q (\tilde U_n, H)  :=  {\mathcal F}(\tilde U_n + H ) - {\mathcal F}(\tilde U_n) - L_n H \,, \quad 
H \in E_n \times \R^\nu \,  . 
\end{aligned}
\end{equation}
Then, by the definition of $ H_{n+1} $ in \eqref{soluzioni approssimate}, 
we have (recall also \eqref{proiettore modificato})
\begin{align}
{\mathcal F}(U_{n + 1}) & = 
 {\mathcal F}(\tilde U_n) - L_n {\bf \Pi}_{n } {\bf T}_n \Pi_{n } {\mathcal F}(\tilde U_n) + Q_n \nonumber \\
 & = 
 {\mathcal F}(\tilde U_n) - L_n  {\bf T}_n \Pi_{n } {\mathcal F}(\tilde U_n) + L_n  {\bf \Pi}_n^\bot  {\bf T}_n \Pi_{n } {\mathcal F}(\tilde U_n)
 + Q_n \nonumber\\ 
& =  {\mathcal F}(\tilde U_n)  - \Pi_{n } L_n {\bf T}_n \Pi_{n }{\mathcal F}(\tilde U_n) 
+ ( L_n  {\bf \Pi}_n^\bot -  \Pi_n^\bot L_n ) {\bf T}_n \Pi_{n }{\mathcal F}(\tilde U_n) + Q_n \nonumber\\
 & = \Pi_{n }^\bot {\mathcal F}(\tilde U_n) + R_n + Q_n + P_n  
\label{relazione algebrica induttiva}
\end{align}
where 
\begin{equation}\label{Rn Q tilde n}
\begin{aligned}
& R_n := (L_n  {\bf \Pi}_n^\bot -  \Pi_n^\bot L_n) {\bf T}_n \Pi_{n }{\mathcal F}( \tilde U_n) \,, \\
& P_n := - \Pi_{n } ( L_n {\bf T}_n - {\rm Id}) \Pi_{n } {\mathcal F}( \tilde U_n)\,.
\end{aligned}
\end{equation}
We first note that,  for all $ \lambda \in \tOm \times [\kappa_1, \kappa_2] $,  $s \geq s_0 $, 
\begin{equation}\label{F tilde Un W tilde n}
\begin{aligned}
\| {\mathcal F}(\tilde U_n)\|_s & 
\leq_s \|{\mathcal F}(U_0)\|_s + \| {\mathcal F}(\tilde U_n) - {\mathcal F}(U_0)\|_s \\
& \stackrel{\eqref{operatorF}, \eqref{stima derivata XP}, \eqref{costanti nash moser 2}, \eqref{ansatz induttivi nell'iterazione}}{\leq_s}  \e + \| \tilde W_n\|_{s + \sigma_1}  
\end{aligned}
\end{equation}
and, by \eqref{ansatz induttivi nell'iterazione}, \eqref{nash moser smallness condition}, 
\begin{equation}\label{gamma - 1 F tilde Un}
\gamma^{- 1} \| {\mathcal F}(\tilde U_n)\|_{s_0} \leq 1\, . 
\end{equation}

\begin{lemma}
For all  $ \lambda \in {\mathcal N}({\mathcal G}_{n + 1}, 2 \gamma  K_{n}^{-p(\tau+2)}) $ we have, 
setting $ \mu_2 := 
\mu({\mathtt b}) +  3 \sigma_1 + 2 $,  
\be
\begin{aligned}
  \| {\mathcal F}(U_{n + 1})\|_{s_0}  \leq_{s_0 + {\mathtt b}_1}    \frac{1}{\g}K_{n }^{\mu_2 - {\mathtt b}_1} 
 (  \e + \| \tilde W_n \|_{s_0 + \mathtt b_1}) 
& + \frac{K_n^{2 \sigma_1 + 4}}{\gamma} \| {\mathcal F}(\tilde U_n)\|_{s_0}^2 \\
& + K_{n - 1}^{- p {\mathtt a} } K_n^{\sigma_1} \frac{\e}{\gamma^{ 2}} \| {\mathcal F}(\tilde U_n) \|_{s_0} \label{F(U n+1) norma bassa} 
\end{aligned}
\ee
\be
\begin{aligned}
& \label{U n+1 alta}
\| W_1 \|_{s_0+ {\mathtt b}_1} 
\leq_{s_0+ {\mathtt b}_1}  \e \gamma^{- 1} \, , \\
& \| W_{n + 1}\|_{s_0 + {\mathtt b}_1} \leq_{s_0 + {\mathtt b}_1} 
K_n^{\mu({\mathtt b}) + 2\sigma_1 } \g^{-1} (\e  +  \| \tilde W_n\|_{s_0 + {\mathtt b}_1}  )\, , \ n \geq 1 \, . 
\end{aligned}
\ee
\end{lemma}

\begin{proof} 
We first estimate $ H_{n +1} $ defined in  \eqref{soluzioni approssimate}.

\noindent
{\bf Estimates of $ H_{n+1} $.}
 By \eqref{soluzioni approssimate} and \eqref{smoothing-u1}, 
\eqref{stima Tn}, \eqref{stima Tn norma bassa}, \eqref{ansatz induttivi nell'iterazione},   we get 
\begin{align}
\|  H_{n + 1} \|_{s_0 + {\mathtt b}_1} 
& \leq_{s_0 + {\mathtt b}_1}  \gamma^{- 1} 
\big( K_n^{\sigma_1} \|{\mathcal F}(\tilde U_n) \|_{s_0 + {\mathtt b}_1} + 
K_n^{\mu({\mathtt b}) + 2 \sigma_1} \|\tilde \fracchi_n \|_{s_0 + {\mathtt b}_1}\| {\mathcal F}(\tilde U_n)\|_{s_0 }  \big)
\nonumber \\
& \stackrel{\eqref{F tilde Un W tilde n}, \eqref{gamma - 1 F tilde Un}}{\leq_{s_0 + {\mathtt b}_1}} 
K_n^{\mu(\mathtt b) + 2 \sigma_1  } \g^{-1} \big( \e  +  \| \tilde W_n \|_{s_0 + \mathtt b_1} \big)\, ,  \label{H n+1 alta} 
\\
\label{H n+1 bassa}
& \|  H_{n + 1}\|_{s_0} 
 \leq_{s_0 + \mathtt b_1} \gamma^{-1}K_{n}^{\sigma_1} \| {\mathcal F}(\tilde U_n)\|_{s_0} \, .
\end{align}
Now we  estimate the terms $ Q_n $ in \eqref{def:Qn} and $ P_n , R_n $ in \eqref{Rn Q tilde n} in $ \| \ \|_{s_0} $ norm. 
\\[1mm]
{\bf Estimate of $ Q_n $.}
By  \eqref{def:Qn}, \eqref{operatorF}, \eqref{stima derivata seconda XP} and \eqref{ansatz induttivi nell'iterazione}, \eqref{smoothing-u1}, 
we have the quadratic estimate
\be\label{stima parte quadratica norma bassa}
\| Q( \tilde U_n, H) \|_{s_0}  \leq_{s_0}\e K_n^4 \|  \widehat \fracchi \|_{s_0}^2   \, ,  \ \forall  \widehat \fracchi \in E_n \, . 
\ee
Then the term $ Q_n $ in \eqref{def:Qn} satisfies, 
by \eqref{stima parte quadratica norma bassa},  
 \eqref{H n+1 bassa},  $\e \gamma^{- 1} \leq 1$, 
\begin{align} 
\| Q_n \|_{s_0} 
&  \leq_{s_0 + \mathtt b_1}  K_n^{2 \sigma_1 + 4 }  \gamma^{-1} \| {\mathcal F}(\tilde U_n) \|_{s_0}^2\, . \label{Qn norma bassa}
\end{align}
{\bf Estimate of $ P_n $.} According to \eqref{splitting per approximate inverse}, 
we write the term $ P_n $ in \eqref{Rn Q tilde n} as
\begin{align*}
& 
P_n = - \Pi_n (L_n {\bf T}_n - {\rm Id}) \Pi_n {\mathcal F}(\tilde U_n) = - P_n^{(1)} - P_{n , \omega} - P_{n, \omega}^\bot \\
& P_n^{(1)} := \Pi_n {\mathcal P}({\tilde \imath}_n ) \Pi_n {\mathcal F}(\tilde U_n)\,, \,   \\
& P_{n, \omega} := \Pi_n {\mathcal P}_\omega( {\tilde \imath}_n ) \Pi_n {\mathcal F}(\tilde U_n)\,, \\
&   P_{n, \omega}^\bot := \Pi_n {\mathcal P}_\omega^\bot({\tilde \imath}_n ) \Pi_n {\mathcal F}(\tilde U_n)\,.
\end{align*}
By \eqref{ansatz induttivi nell'iterazione}, \eqref{nash moser smallness condition}, \eqref{gamma - 1 F tilde Un}, 
using that, by \eqref{smoothing-u1}, 
$$
\begin{aligned}
\| {\mathcal F}(\tilde U_n) \|_{s_0 + \sigma_1} & 
\leq \| \Pi_n {\mathcal F}(\tilde U_n)  \|_{s_0  + \sigma_1} + \|\Pi_n^\bot {\mathcal F}(\tilde U_n)  \|_{s_0  + \sigma_1} \\
& 
\leq K_n^{  \sigma_1} \big( \| {\mathcal F}(\tilde U_n)\|_{s_0} + K_n^{- \mathtt b_1} \|{\mathcal F}(\tilde U_n) \|_{s_0 + \mathtt b_1} \big)
\end{aligned}
$$
the bounds \eqref{stima inverso approssimato 2}-\eqref{stima cal G omega bot alta} imply the following estimates: 

\begin{align}
  \| P_n^{(1)} \|_{s_0}  & \leq_{s_0 + \mathtt b_1}
 \g^{-1} K_n^{2 \sigma_1 } \| {\mathcal F}(\tilde U_n)\|_{s_0}^2 \nonumber \\
 & \qquad + K_n^{2 \sigma_1 - {\mathtt b}_1} \| {\mathcal F}(\tilde U_n) \|_{s_0 + {\mathtt b}_1} \| {\mathcal F}(\tilde U_n)\|_{s_0}\,,\nonumber\\
 & \qquad \quad  \stackrel{\eqref{F tilde Un W tilde n}, \eqref{smoothing-u1}}{\leq_{s_0 + \mathtt b_1}}\g^{-1} K_n^{2 \sigma_1}  \| {\mathcal F}(\tilde U_n)\|_{s_0}^2 \nonumber \\
 & \ \ \quad \qquad \qquad \quad +\g^{- 1} K_n^{3 \sigma_1 - {\mathtt b}_1}(\e + \| \tilde W_n\|_{s_0 + \mathtt b_1} )\big) \| {\mathcal F}(\tilde U_n)\|_{s_0}\,, \label{Q n 1 bassa} 
 \end{align}

 \begin{align}
& \| P_{n, \omega} \|_{s_0}   \leq_{s_0 + \mathtt b_1} \e \gamma^{- 2} 
N_{n - 1}^{- {\mathtt a}} K_n^{\sigma_1}  \| {\mathcal F}(\tilde U_n) \|_{s_0}\,, \label{Q n omega bassa} \\
& \| P_{n, \omega}^\bot\|_{s_0}  \leq_{s_0 + \mathtt b_1} K_{n}^{\mu({\mathtt b}) + 2 \sigma_1 - {\mathtt b}_1} \gamma^{- 1} 
( \| {\mathcal F}(\tilde U_n) \|_{s_0 + {\mathtt b}_1}+ \e  \| \tilde \fracchi_n \|_{s_0 + {\mathtt b}_1})  \nonumber\\ 
& \qquad \qquad \stackrel{\eqref{F tilde Un W tilde n}, \eqref{smoothing-u1}}{\leq_{s_0 +\mathtt b_1}} K_{n}^{\mu({\mathtt b}) + 3 \sigma_1 - {\mathtt b}_1}
\gamma^{- 1} ( \e + \| \tilde W_n \|_{s_0+ \mathtt b_1 }) \, . \label{Q n omega bot bassa} 
\end{align}
{\bf Estimate of $ R_n $.} For  $ H := (\widehat \fracchi, \widehat \a ) $ we have 
$ (L_n  {\bf \Pi}_n^\bot -  \Pi_n^\bot L_n) H = $ $ \e [ d_i X_{P}( \tilde \imath_n), \Pi_n^\bot ] \widehat \fracchi = $
$ [ \Pi_n , d_i X_{P}( \tilde \imath_n)] 
\widehat \fracchi $ where $ X_{P} $ is the Hamiltonian vector field of the perturbation 
$ P $ in \eqref{definizione cal N P}, see \eqref{operatorF}. 
Thus, applying the estimate \eqref{stima derivata XP}, using \eqref{smoothing-u1} and recalling \eqref{costanti nash moser 2}, the following estimate holds:
\begin{align}
\| (L_n  {\bf \Pi}_n^\bot -  \Pi_n^\bot L_n) H \|_{s_0} & \leq_{s_0+ {\mathtt b}_1} 
\e K_{n }^{- {\mathtt b}_1  + \sigma_1 + 2} \big(\|  \widehat \fracchi \|_{s_0 + {\mathtt b}_1}  \nonumber\\
& \qquad + 
\| \tilde \fracchi_n \|_{s_0 + {\mathtt b}_1 } \|  \widehat \fracchi \|_{s_0 + 2}\big)\,. \label{stima commutatore modi alti norma bassa}
\end{align}
Hence, applying 
\eqref{stima Tn}, 
\eqref{stima commutatore modi alti norma bassa},  
\eqref{nash moser smallness condition}, 
\eqref{ansatz induttivi nell'iterazione}, 
\eqref{smoothing-u1}, \eqref{gamma - 1 F tilde Un}
the term $R_n$ defined in \eqref{Rn Q tilde n} satisfies
\begin{align} 
\| R_n\|_{s_0} 
& \leq_{s_0 + {\mathtt b}_1}  K_n^{ \mu({\mathtt b}) + 2 \sigma_1 + 2 - {\mathtt b}_1} ( \e \gamma^{-1} \| {\mathcal F}(\tilde U_n)\|_{s_0 + {\mathtt b}_1} + \e \| \tilde \fracchi_n  \|_{s_0 + {\mathtt b}_1} ) \nonumber\\
& \stackrel{\eqref{F tilde Un W tilde n}}{\leq_{s_0 + \mathtt b_1}} K_n^{ \mu({\mathtt b}) + 3 \sigma_1 + 2 - {\mathtt b}_1} ( \e + 
\| \tilde W_n\|_{s_0 + \mathtt b_1} )\, . \label{stima Rn norma bassa}
\end{align}
We can finally  estimate $ {\mathcal F}(U_{n + 1}) $ in $ \| \ \|_{s_0} $. 
By \eqref{relazione algebrica induttiva} and  
\eqref{Qn norma bassa}, \eqref{Q n 1 bassa}-\eqref{Q n omega bot bassa}, 
\eqref{stima Rn norma bassa},  \eqref{nash moser smallness condition}, 
\eqref{ansatz induttivi nell'iterazione}, 
we get \eqref{F(U n+1) norma bassa}. 
Moreover by \eqref{soluzioni approssimate} and \eqref{stima Tn} we have the bound  \eqref{U n+1 alta} for  
$$ 
\| W_1 \|_{s_0+ {\mathtt b}_1} = \| H_1 \|_{s_0+ {\mathtt b}_1}  \leq_{s_0+ {\mathtt b}_1} \g^{-1} \| {\mathcal F}(U_0)\|_{s_0+ {\mathtt b}_1 + \sigma_1} 
\leq_{s_0+ {\mathtt b}_1}  \e \gamma^{- 1} \, .
$$ 
The estimate  \eqref{U n+1 alta}  for  $ W_{n+1} := \tilde W_n + H_{n+1} $, $ n \geq 1 $,  
 follows by 
 \eqref{H n+1 alta}. 
\end{proof}

As a corollary we get

\begin{lemma}\label{lemma:quadra}
For all  $ \lambda \in {\mathcal N}({\mathcal G}_{n + 1}, 2 \gamma  K_{n}^{-p(\tau+2)}) $ we have
\begin{equation}\label{stima F u n + 1 induttiva}
\begin{aligned}
&  \| {\mathcal F}(U_{n + 1}) \|_{s_0}^{k_0, \gamma}  \leq C_* \e K_n^{- \mathtt a_1}  \, , \\
& \| W_{n + 1} \|_{s_0 + \mathtt b_1}^{k_0, \gamma} 
\leq C_* \e \gamma^{- 1} K_n^{\mu_1} \, ,  \\ 
\end{aligned}
\end{equation}
\begin{equation}\label{stima H n+1 lemma}
\begin{aligned}
& 
\| H_1 \|_{s_0 + \mu(\mathtt b) + \sigma_1}^{k_0, \gamma} \leq C \e \g^{-1} \, , \\
& \| H_{n + 1}\|_{s_0 + \mu(\mathtt b) + \sigma_1}^{k_0, \gamma} \leq_{s_0}   \e \gamma^{- 1} K_n^{\mu(\mathtt b) + 2 \sigma_1} K_{n - 1}^{- \mathtt a_1} \, ,
\   n \geq 1 \,.
\end{aligned}
\end{equation}
\end{lemma}

\begin{proof}
First note that, by  \eqref{inclusioni:Cantor},  if  
$ \lambda \in  {\mathcal N}({\mathcal G}_{n + 1}, 2 \gamma  K_{n}^{-p(\tau+2)}) $ then 
$ \lambda \in  {\mathcal N}({\mathcal G}_{n}, \gamma K_{n-1}^{- p(\tau + 2)} )  $ and so  \eqref{P2n} and $({\mathcal P}3)_{n}$ hold. 
Then the first inequality in \eqref{stima F u n + 1 induttiva} follows by 
\eqref{F(U n+1) norma bassa},
$ ({\mathcal P}2)_{n}$, $({\mathcal P}3)_{n}$,  $\gamma^{- 1} = K_0 \leq K_n $, $ \e \g^{-2} \leq c $ small, 
and by \eqref{costanti nash moser}, \eqref{costanti nash moser 1}, \eqref{cond-su-p}-\eqref{choice:p}
(see also remark \ref{choice-a-b}). 
For $ n = 0 $ we use also  \eqref{nash moser smallness condition}. 
The second inequality in \eqref{stima F u n + 1 induttiva} follows similarly 
by  \eqref{U n+1 alta}, $({\mathcal P}3)_{n} $, the choice of $ \mu_1 $ in \eqref{costanti nash moser 1} and 
$ K_0 $ large enough. 
Since  $ H_1  = W_1 $
the first inequality in \eqref{stima H n+1 lemma} follows by the first inequality in \eqref{U n+1 alta}.
For $ n \geq 1 $, the estimate \eqref{stima H n+1 lemma} follows 
by \eqref{smoothing-u1}, \eqref{H n+1 bassa}
and \eqref{P2n}.
\end{proof}

We now define a $ k_0 $-times differentiable extension of 
$ (H_{n + 1})_{|{\mathcal N}({\mathcal G}_{n+1}, \gamma K_n^{- p (\tau +2)})} $ to the whole $ \R^\nu \times [\kappa_1, \kappa_2] $.

\begin{lemma} {\bf (Extension)}\label{lemma:extension torus}
There is a $ k_0 $-times differentiable  function $ {\tilde H}_{n + 1} $  defined on the whole $ \R^\nu \times [\kappa_1, \kappa_2] $  such that
\be\label{Hn+1= tilde Hn+1}
\tilde H_{n + 1} = H_{n + 1} \, , \quad \forall \lambda \in {\mathcal N}({\mathcal G}_{n+1}, \gamma K_n^{- p (\tau +2)}) \, , 
\ee 
and \eqref{Hn} holds also at the step $n + 1$. 
\end{lemma}

\begin{proof}
The function $ H_{n +1}(\lambda)$ is defined for all $\lambda \in {\mathcal N}\big( {\mathcal G}_{n+1}, 2 \gamma K_n^{- p (\tau +2)}\big) $.
Then we define 
$$ 
\tilde H_{n + 1} (\lambda)  := 
\begin{cases}
\psi_{n + 1} (\lambda) H_{n + 1} (\lambda) \qquad \forall \lambda \in {\mathcal N}\big( {\mathcal G}_{n+1}, 2 \gamma K_n^{- p (\tau +2)}\big) \cr
0 \qquad\qquad\qquad\qquad \quad \forall \lambda \notin {\mathcal N} \big( {\mathcal G}_{n+1}, 2 \gamma K_n^{- p (\tau +2)}\big)
\end{cases}
$$ 
where $ \psi_{n + 1} $ is a $ {\mathcal C}^\infty $ cut-off function satisfying $ 0 \leq \psi_{n + 1} \leq 1 $, 
\begin{align*}
& 
\psi_{n + 1}(\lambda) = 1 \, , \ \forall \lambda \in  {\mathcal N}({\mathcal G}_{n+1}, 
\gamma K_n^{- p (\tau +2)}), \ \  {\rm supp}(\psi_{n + 1}) \subseteq
{\mathcal N} \big( {\mathcal G}_{n+1}, 2 \gamma K_n^{- p (\tau +2)}\big)\, , 
\\
& \qquad \qquad \qquad  |\partial^k_{\lambda} \psi_{n + 1}(\lambda)| \leq C(k) \big( K_n^{p(\tau + 2)} \gamma^{- 1} \big)^{|k|}, \,  \  
\forall k \in \N^{\nu + 1} \,. 
\end{align*}
Then 
\eqref{Hn+1= tilde Hn+1} holds
and we have the estimate 
$$ 
\| \tilde H_{n + 1}\|_{s_0 + \mu({\mathtt b}) + \sigma_1 }^{k_0, \gamma} \leq  K_n^{p (\tau + 2) k_0 }
\| H_{n + 1}\|_{s_0+ \mu({\mathtt b}) + \sigma_1}^{k_0, \gamma}.
$$ 
For $ n = 0 $ and \eqref{stima H n+1 lemma} we get the first inequality in  \eqref{Hn}. For $ n \geq 1 $ we deduce using   
\eqref{stima H n+1 lemma}  and  
the definition of $ \mathtt a_2 $ in \eqref{costanti nash moser}, the estimate  \eqref{Hn} also  at the step $n + 1$. 
\end{proof}

We now define 
$$
\tilde W_{n+1} = \tilde W_{n} + \tilde H_{n + 1} \, , \quad 
\tilde U_{n + 1} := \tilde U_n + \tilde H_{n + 1} = U_0 + \tilde W_n + \tilde H_{n + 1} = U_0 + \tilde W_{n + 1}\,,
$$
which are defined for all $\lambda \in \R^\nu \times [\kappa_1, \kappa_2] $ and satisfy 
$$
\tilde W_{n + 1} = W_{n + 1} \, , \  
\tilde U_{n + 1} = U_{n + 1} \, , \ \ \forall \lambda \in {\mathcal N}({\mathcal G}_{n + 1}, \gamma K_n^{- p (\tau +2)}) \, . 
$$ 
Therefore  $({\mathcal P}2)_{n + 1}$, $({\mathcal P}3)_{n + 1}$ are proved by Lemma \ref{lemma:quadra}.
Moreover by \eqref{Hn}, which has been proved up to the step $n + 1 $ in Lemma \ref{lemma:extension torus}, we have 
$$
\| \tilde W_{n + 1} \|_{s_0 + \mu({\mathtt b}) + \sigma_1}^{k_0, \gamma} 
\leq {\mathop \sum}_{k = 1}^{n + 1} \| \tilde H_k \|_{s_0 + \mu({\mathtt b}) + \sigma_1}^{k_0, \gamma} 
\leq C_*  K_0^{p k_0 (\tau+2)} \e  \gamma^{-1}
$$
and thus \eqref{ansatz induttivi nell'iterazione} holds also 
at the step $n + 1$. This completes the proof of Theorem \ref{iterazione-non-lineare}.
\end{proof}

\section{Proof of Theorem \ref{MAINTHEOREM}}\label{proof theorem 4.1}
 
Let $ \gamma = \e^a  $ with $ a \in (0, a_0) $ and $ a_0 := 1 / (2+ \tau_2 ) $.
Then the smallness condition \eqref{nash moser smallness condition} holds
for $ 0 < \e < \e_0 $ small enough and Theorem \ref{iterazione-non-lineare} holds.   
By \eqref{Hn} the  sequence of functions 
$$ 
\tilde W_n = {\tilde U}_n - (\vphi, 0, 0, \omega)  = 
\big(\tilde \fracchi_n , \tilde \a_n - \omega \big)  = \big(\tilde \imath_n - (\vphi, 0, 0), \tilde \a_n - \omega \big) 
$$ 
is a Cauchy sequence in $ \|  \ \|_{s_0}^{k_0, \gamma} $ and then it converges to a function 
$$ 
W_\infty := \lim_{n \to + \infty} {\tilde W}_n\,, \quad \text{with} \quad W_\infty : \tOm \times [\kappa_1, \kappa_2] \to H^{s_0}_\vphi  \times H^{s_0}_\vphi 
\times H^{s_0}
\times \R^\nu \, .
$$
We define 
$$
U_\infty := (i_\infty, \a_\infty) = (\vphi,0,0, \om) + 
W_\infty \,.  
$$
By \eqref{ansatz induttivi nell'iterazione} and \eqref{Hn} we also deduce 
\begin{equation}\label{U infty - U n}
\begin{aligned}
& \|  U_\infty -  U_0 \|_{s_0 + \mu(\mathtt b) + \sigma_1}^{k_0, \gamma} \leq C_* \e \gamma^{- 1} K_0^{p k_0 (\tau + 2)}\,, \\
& \| U_\infty - {\tilde U}_n \|_{s_0  + \mu({\mathtt b}) + \sigma_1}^{k_0, \gamma} \leq C \e \gamma^{-1} K_{n }^{- \mathtt a_2} \, , \ \  \forall n \geq 1\,.
\end{aligned}
\end{equation}
Moreover by Theorem \ref{iterazione-non-lineare}-$({\mathcal P}2)_n$, we deduce that 
$ {\mathcal F}(\lambda, U_\infty(\lambda)) = 0 $ for all $ \lambda  $ belonging to 
\be\label{defGinfty}
\begin{aligned}
\bigcap_{n \geq 0} {\mathcal G}_n & = 
\tLm \cap  \bigcap_{n \geq 1} 
 {\bf \Lambda}_{n}^{\gamma}(\tilde \imath_{n-1}) \\
 & \stackrel{\eqref{Melnikov-invert}, \eqref{Cantor set}, \eqref{prime di melnikov}}{=}
 \tLm \cap  \Big[ \bigcap_{n \geq 1}  \tLm_{n}^{\gamma}(\tilde \imath_{n-1}) \Big] \bigcap 
 \Big[ \bigcap_{n \geq 1}   \mathtt \Lambda_{n}^{\gamma, I}(\tilde \imath_{n-1}) \Big]\, \,,
 \end{aligned}
\ee
where $\mathtt \Lambda := \mathtt \Omega \times [\kappa_1, \kappa_2]$.
By \eqref{U infty - U n} for $ n = 0 $
and since $ K_0 = \g^{-1} $ (see \eqref{nash moser smallness condition})  we deduce
the estimates \eqref{mappa aep} and \eqref{stima toro finale} 
with $ k_1 :=  p k_0 ( \tau + 2 )$.  

In order to conclude the proof of Theorem \ref{MAINTHEOREM} we have to provide the characterization of
$ {\mathcal C}_\infty^{\gamma} $ in \eqref{Cantor set infinito riccardo}. 
We first consider the set 
\begin{equation}\label{cantor finale 1 riccardo}
{\mathcal  G}_\infty 
:= \tLm \cap \Big[ \bigcap_{n \geq 1} \tLm_n^{2 \gamma}( i_\infty) \Big] \bigcap \Big[ \bigcap_{n \geq 1} \mathtt \Lambda_n^{2 \gamma, I}(i_\infty)  \Big]\,.
\end{equation}

\begin{lemma}\label{lemma inclusione cantor riccardo 1}
$ {\mathcal G}_\infty  \subseteq  \bigcap_{n \geq 0 } {\mathcal G}_n $, where $ {\mathcal G}_n $ are  defined in \eqref{def:cal-Gn}.
\end{lemma}

\begin{proof}
By \eqref{U infty - U n}, \eqref{nash moser smallness condition},  we have
\begin{align*}
 \e \gamma^{- 1} C(S) N_{0}^\tau \| i_\infty - i_0 \|_{s_0+ {\mu}(\mathtt  b) + \s_1} & \! \leq  \!
\e \gamma^{- 1} C(S) K_{0}^{p\tau}C_* \e \gamma^{- 1} K_0^{p k_0 (\tau + 2)} \leq 
 \gamma  \\
  \e \gamma^{- 1} C(S) N_{n-1}^\tau \| i_\infty - {\tilde \imath}_{n-1} \|_{s_0+ {\mu}(\mathtt  b) + \s_1} & \! \leq \!
 \e \gamma^{- 1} C(S) K_{n-1}^{p\tau} C \e \gamma^{-1} K_{n }^{- \mathtt a_2} \leq
  \gamma   , \, \forall n \geq 2 \,  , 
\end{align*}
noting that the exponent $ \tau_2 $ in \eqref{nash moser smallness condition} satisfies 
$ \t_2 > {\mathtt a}_1 > 3 ( p k_0 (\tau + 2) + p \tau)/ 2 $ by \eqref{costanti nash moser} and 
that  $\mathtt a_2 \geq p \tau + \chi^{-1} $  (see \eqref{costanti nash moser} and remark \ref{choice-a-b}). Recall also that $S$ has been fixed in \eqref{valore finalissimo S} and that $\sigma_1 \geq \sigma$, see \eqref{costanti nash moser 2}.  
Therefore  Theorem \ref{ITERAZIONERIDUCIBILITA}-$({\bf S4})_\nu$ implies 
$$
\tLm_n^{2 \gamma}( i_\infty) \subset  \tLm_n^{\gamma}( \tilde \imath_{n-1} ) \, , \quad \forall n \geq 1 \, .
$$
By similar arguments we  deduce that 
$ \mathtt \Lambda_n^{2 \gamma, I}( i_\infty) \subset  \mathtt\Lambda_n^{\gamma, I}( \tilde \imath_{n-1} ) $ and the lemma is proved. 
\end{proof}

Then we define the ``final eigenvalues" 
\begin{equation}\label{autovalori finali riccardo}
\mu_j^\infty := \mathtt m_3^{\infty}  j^{\frac12} (1 + \kappa j^2)^{\frac12}  + {\mathtt m}_1^\infty j^{\frac12} + r_j^\infty\,, \quad j \in \N^+ \setminus {\mathbb S}^+ \, , 
\end{equation}
where
\begin{equation}\label{resti autovalori finali riccardo}
\mathtt m_3^\infty := \mathtt m_3(i_\infty)\,, \quad {\mathtt m}_1^\infty := {\mathtt m}_1(i_\infty)\,, \quad   r_j^\infty := \lim_{n \to + \infty} \tilde r_j^n(i_\infty)\,, \ \  j \in \N^+ \setminus {\mathbb S}^+\,,
\end{equation}
where $\mathtt m_3$, ${\mathtt m}_1$ are defined in \eqref{lambda3 formula}, \eqref{lambda 1 senza proiettore}
and $\tilde r_j^n$ are given in Theorem \ref{ITERAZIONERIDUCIBILITA}-$({\bf S2})_\nu$.
Note that the sequence $(\tilde r_j^n(i_\infty))_{n \in \N}$ is a Cauchy sequence 
in $ | \ |^{k_0, \gamma}$ by  \eqref{vicinanza autovalori estesi}. As  a consequence its limit function 
 $ r_j^\infty (\om, \kappa)  $ is  well defined, it is $ k_0 $-times differentiable  and satisfies 
\be\label{distanza-rnrinfty}
| r_j^\infty - \tilde r_j^n(i_\infty)|^{k_0, \gamma} \leq 
C \e \gamma^{- 1} N_{n }^{ k_0 (\tau + 2)} N_{n - 1}^{- {\mathtt a}} \, ,  \ n \geq 0 \, .
\ee 
In particular, since  $ \tilde r_j^0 (i_\infty) = 0 $ and $ K_0 = \g^{-1} $ we get  
$ | r_j^\infty |^{k_0, \gamma} \leq C \e \gamma^{- 1}  K_{0}^{ p k_0 (\tau + 2) + 1} $ and
 \eqref{stime autovalori infiniti} holds with $ k_1 = p k_0 (\tau + 2) + 1 $ (recall that the constant $C := C(S, k_0)$ with $S$ fixed in \eqref{valore finalissimo S}).  
 
Finally, we consider  the set $ {\mathcal C}_\infty^{\gamma} $ in \eqref{Cantor set infinito riccardo}.

\begin{lemma} \label{lemma inclusione cantor riccardo 2}
$ {\mathcal C}_\infty^\gamma \subseteq {\mathcal  G}_\infty  $ defined in \eqref{cantor finale 1 riccardo}. 
\end{lemma}

\begin{proof}
By \eqref{cantor finale 1 riccardo}, we have to prove that 
$ {\mathcal C}_\infty^\gamma \subseteq \tLm_n^{2 \gamma}(i_\infty) $, $ \forall n \in \N $. 
We argue by induction. For $n = 0$ the inclusion is trivial, since $ \tLm_0^{2 \gamma}(i_\infty) = \tOm \times [\kappa_1, \kappa_2] = \tLm$. 
Now  assume that ${\mathcal C}_\infty^\gamma \subseteq \tLm_n^{2 \gamma}(i_\infty)$. 
Theorem \ref{ITERAZIONERIDUCIBILITA}-$({\bf S2})_\nu$ implies
$ \tilde \mu_j^n(i_\infty)(\lambda) = \mu_j^n(i_\infty)(\lambda) $, $  \forall \lambda \in \tLm_n^{2 \gamma}(i_\infty) $. 
Hence $ \forall \lambda \in {\mathcal C}_\infty^\gamma \subseteq  \tLm_n^{2 \gamma}(i_\infty)$, 
by \eqref{mu j nu}, \eqref{autovalori finali riccardo}, \eqref{distanza-rnrinfty},  we get 
$$
|(\mu_j^n- \mu_{j'}^n)(i_\infty) - (\mu_j^\infty - \mu_{j'}^\infty) | \leq  C \e \gamma^{- 1} N_{n}^{k_0(\tau + 2)} N_{n - 1}^{- \mathtt a} \, , 
$$
and therefore (consider in \eqref{Cantor set infinito riccardo} the case $ \varsigma = 1 $ and $ j \neq  j' $) 
\begin{align*}
|\omega \cdot \ell + \mu_j^n(i_\infty) - \mu_{j'}^n(i_\infty)| & \geq |\omega \cdot \ell + \mu_j^\infty - \mu_{j'}^\infty| - C \e \gamma^{- 1} N_{n}^{k_0(\tau + 2)} N_{n - 1}^{- \mathtt a} \\
& \geq 4 \gamma |j^{\frac32} - j'^{\frac32}| \langle \ell \rangle^{- \tau} - C \e \gamma^{- 1} |j^{\frac32} - j'^{\frac32}|N_{n}^{k_0(\tau + 2)} N_{n - 1}^{- \mathtt a} \\
& \geq 2 \gamma |j^{\frac32} - j'^{\frac32}| \langle \ell \rangle^{- \tau} \, , \quad \forall |\ell| \leq N_n \, , 
\end{align*}
provided
$ \e \gamma^{- 2} \leq C N_{n - 1}^{\mathtt a} N_n^{- k_0(\tau + 2) - \tau} $, $ \forall n \geq 0 $,
which holds true by \eqref{alpha beta}, \eqref{nash moser smallness condition}, see also remark \ref{remark:ab}. 
We have proved that $ {\mathcal C}_\infty^\gamma \subseteq \tLm_{n + 1}^{2 \gamma}(i_\infty)$.
Similarly we prove that $ {\mathcal C}_\infty^{ \gamma} \subseteq  \mathtt \Lambda_{n}^{2 \gamma, I}(i_\infty)$, 
$ \forall n \in \N $. 
\end{proof}
Lemmata \ref{lemma inclusione cantor riccardo 1}, \ref{lemma inclusione cantor riccardo 2} imply that 
\begin{corollary}
$ {\mathcal C}_\infty^\gamma \subseteq \bigcap_{n\geq 0} {\mathcal G}_n $  defined in \eqref{def:cal-Gn}.
\end{corollary}

\appendix
\chapter{Tame estimates for the flow of pseudo-PDEs}\label{AppendiceA}

In this Appendix we prove  tame estimates  for the flow $ \Phi^t $ of the pseudo-PDE
\begin{equation}\label{pseudo PDE}
\begin{cases}
\partial_t u = \ii a(\vphi, x) |D|^{\frac12}  u  \\
u(0, x) = u_0(x) \, , 
\end{cases}\qquad \vphi \in \T^\nu\,, \quad  x \in \T\,,
\end{equation}
where $a(\vphi, x) = a(\lambda, \vphi, x) $ is a real valued function which is $ {\mathcal C}^\infty$ 
with respect to the variables $(\vphi, x)$ and $ k_0 $-times differentiable with respect to the parameters 
$\lambda =(\omega, \kappa) $.
The function $ a := a(i) $  may depend also on the ``approximate" torus $ i (\vphi )$.  
We look for the solution of \eqref{pseudo PDE} 
by a Galerkin approximation, as  limit of the solutions of the truncated equations 
\begin{equation}\label{pseudo PDE N}
\begin{cases}
\partial_t u = \ii \Pi_N \big( a(\vphi, x) |D|^{\frac12}  \Pi_N u  \big) \\
u(0, x) = \Pi_N u_0(x) \, , 
\end{cases}\qquad \vphi \in \T^\nu\,, \quad  x \in \T\, , 
\end{equation}
where, for any $N \in \N$,  we denote by 
$\Pi_N$  the $ L^2 $-orthogonal projector on the finite dimensional subspace
$$
E_N := \big\{ u \in L^2(\T) : u(x) = {\mathop \sum}_{|j| \leq N} u_j e^{\ii j x} \big\} \, .
$$
We denote by 
$\Phi_N(t) = \Phi_N(\lambda, t, \vphi) : E_N \to E_N $ the flow of \eqref{pseudo PDE N}. It solves
\begin{equation}\label{flow-propagator N}
\begin{cases}
\partial_t \Phi_N(t) = \ii \Pi_N  a(\vphi, x) |D|^{\frac12}  \Phi_N(t) \\
\Phi_N(0) = \Pi_N  \, , 
\end{cases}\qquad \vphi \in \T^\nu\,.
\end{equation}
We  introduce the  ``paraproduct" decomposition 
for the  product of two functions $ a, u : \T \to \C $,  
\be\label{paraproduct}
 a u = T_a u + R_u a 
\ee 
where
\be
\begin{aligned} \label{Ta Ru}
&  T_a u  := \sum_{k, \xi \in \Z\,,|k - \xi| \leq |\xi|} \widehat a(k - \xi) \widehat u(\xi) e^{\ii k x}\,, \\
&  R_u a  := \sum_{k, \xi \in \Z\,,|k - \xi| <  |\xi|} \widehat u(k - \xi) \widehat a(\xi) e^{\ii k x}\, .
\end{aligned}
\ee
Note that  
\be\label{simbolo-a0}
T_a = {\rm Op} ( a_0 (x, \xi) ) \quad {\rm with  } \quad
a_0 (x, \xi) := {\mathop \sum}_{|k| \leq |\xi |} \widehat a( k) e^{\ii k x } \, . 
\ee
For all $ s \geq 0   $, we have the following estimates 
\begin{equation}\label{stima Ta Ru}
\| T_a u  \|_{H^s_x} \leq C(s) \| a \|_{H^1_x} \| u \|_{H^s_x}\,, \quad 
\| R_u (a) \|_{H_x^s} \leq C(s) \| a \|_{H^{s+  (1/2)}_x} \| u \|_{H^{1/2}_x}  
\end{equation}
(the operator $ u \mapsto R_u (a) $  is  smoothing) which  follow 
arguing as in Lemma \ref{lemma: action Sobolev}. 

\begin{lemma}\label{super man}
$\big\| |D|^{\frac12} (T_a)^* - T_a |D|^{\frac12} \big\|_{{\mathcal L}(L^2_x)} \leq C \| a \|_{H^2_x} $ and 

\noindent
$ \big\| [\langle D \rangle^s, T_a |D|^{\frac12}] u \big\|_{L^2_x} 
\leq_s \| a \|_{H^2_x}  \|  u \|_{H^s_x} $, $\forall s \geq 0$. 
\end{lemma}

\begin{proof}
By \eqref{simbolo aggiunto senza tempo} 
the adjoint of $ T_a =  {\rm Op} ( a_0  ) $ 
is the pseudo-differential operator $  (T_a)^* = {\rm Op}(a_0^*) $ with symbol 
$$ 
\begin{aligned}
a_0^*(x, \xi) = 
\overline{{\mathop\sum}_{k \in \Z} \widehat a_0(k, \xi - k) e^{\ii k x}} & \stackrel{\eqref{simbolo-a0}}= 
\overline{{\mathop\sum}_{|k| \leq |\xi - k|} \widehat a(k) e^{\ii k x}} \\ 
& = {\mathop\sum}_{|k| \leq |\xi + k|} \widehat a(k) e^{\ii k x }
\end{aligned}
$$
since $ \overline{\widehat a(k)} = \widehat a(-k) $ because $ a(x) $ is real valued. 
Thus
\be  \label{priP}
|D|^{\frac12} (T_a)^* u  = {\mathop \sum}_{\xi} 
{\mathop \sum}_{|k| \leq |\xi + k|} |\xi + k|^{\frac12} \widehat a(k) \widehat u(\xi) 
e^{\ii (k + \xi) x } = R_1 + R_2 
\ee
where, writing 
\begin{equation}\label{razionalizzazione utile}
\vartheta(\xi, k) := | \xi + k |^{\frac12} - |\xi|^{\frac12} =   \frac{|\xi + k| - |\xi|}{| \xi + k |^{\frac12} + |\xi|^{\frac12}}\,, 
\quad \text{for } \quad (\xi, k) \neq (0, 0)\,, 
\end{equation} 
we split 
\be\label{defR1R2}
\begin{aligned}
& R_1  := \sum_{\xi} \sum_{|k| \leq |\xi + k|} |\xi|^{\frac12} \widehat a(k) 
\widehat u(\xi) e^{\ii (k + \xi) x }\,, \\
&  R_2 := \sum_{\xi} \sum_{|k| \leq |\xi + k|} \vartheta(\xi, k)  \widehat a(k) \widehat u(\xi) e^{\ii (k + \xi) x }\,.
\end{aligned}
\ee
In addition, by \eqref{Ta Ru}, 
\begin{align}
& 
T_a |D|^{\frac12} u(x)  = 
{\mathop \sum}_{\xi} 
{\mathop \sum}_{|k| \leq |\xi|} |\xi|^{\frac12} \widehat a(k) \widehat u(\xi) e^{\ii (k + \xi) x }\,. \label{secP}
\end{align}
We estimate
\begin{equation}\label{R1 R2 paraprodotto}
\big( |D|^{\frac12} (T_a)^* - T_a |D|^{\frac12} \big) u  = ( R_1 - T_a |D|^{\frac12} u )  + R_2  \, . 
\ee
{\sc Estimate of $ R_2 $.} By \eqref{razionalizzazione utile} the triangular inequality implies $ | \vartheta(\xi, k) | \leq |k| $,
for any $ k, \xi \in \Z $. Then by the Cauchy-Schwartz inequality we get 
\begin{align}
\| R_2 \|_{L^2_x}^2 & \leq	   \sum_{j} \Big( \sum_{|j - \xi| \leq |j|} |\vartheta(\xi, j - \xi)| |\widehat a(j - \xi)|  |\widehat u(\xi)| \Big)^2 \nonumber\\
& \leq \sum_{j}  \Big(\sum_{|j - \xi| \leq |j|} |j - \xi|  |\widehat a(j - \xi)| |\widehat u(\xi)| \frac{\langle j - \xi \rangle}{\langle j - \xi \rangle} \Big)^2 \nonumber\\
& \leq C \sum_{j} \sum_{|j - \xi| \leq |j|} \langle j - \xi \rangle^4 |\widehat a(j - \xi)|^2 |\widehat u(\xi)|^2 \nonumber\\
& \leq C  \sum_\xi |\widehat u(\xi)|^2 \sum_j \langle j - \xi \rangle^4 |\widehat a(j - \xi)|^2 \leq C \| a \|_{H^2_x}^2 \| u \|_{L^2_x}^2\,. \label{stima R2 paraprodotto}
\end{align}
{\sc Estimate of $ R_1 - T_a |D|^{\frac12} u $}. 
By \eqref{defR1R2} and \eqref{secP}  we write
\begin{equation}\label{T1 T2 paraprodotto} 
\begin{aligned}
&  R_1  - T_a |D|^{\frac12} u  = T_1  - T_2 \\
&  T_1 := \sum_{\xi} \sum_{|\xi| < |k| \leq |\xi + k|} |\xi|^{\frac12} 
 \widehat a(k) \widehat u(\xi) e^{\ii (\xi + k) x},  \\
 & T_2  
 := \sum_{\xi}  \sum_{|\xi + k| < |k| \leq |\xi|}  |\xi|^{\frac12}  \widehat a(k) \widehat u(\xi) e^{\ii (\xi + k)x}\,. 
\end{aligned}
\end{equation}
We estimate the $L^2_x $ norm of $T_2$. The estimate for $T_1$ is analogous. We have
$$
\| T_2 \|_{L^2_x}^2  \leq \sum_j \Big( \sum_{|j| 
\leq |j - \xi| \leq |\xi|} |\xi|^{\frac12} |\widehat a(j - \xi)| |\widehat u(\xi)| \Big)^2
$$
and, since in the sum
$ |\xi| \leq |j|+ | \xi - j | \leq 2 | j - \xi | $, the Cauchy-Schwartz inequality implies 
\begin{align}
\| T_2\|_{L^2_x}^2 & \leq 4 \sum_j \Big( \sum_{| j  | \leq | j - \xi| \leq |\xi|} | j - \xi|^{\frac12} 
|\widehat a(j - \xi)| |\widehat u(\xi)| \frac{\langle j - \xi \rangle}{\langle j - \xi\rangle} \Big)^2 \nonumber\\
& \leq C \sum_j \sum_{|j | \leq | j - \xi| \leq |\xi|} \langle j - \xi \rangle^{3}
|\widehat a(j - \xi)|^2 |\widehat u(\xi)|^2 \nonumber\\
& \leq C \sum_\xi |\widehat u(\xi)|^2 \sum_j \langle j - \xi \rangle^{3}|
\widehat a(j - \xi)|^2 \leq C \| a\|^2_{H^{\frac32}_x} \| u \|_{L^2_x}^2\, . 
\label{stima T2 paraprodotto}
\end{align}
The first estimate of Lemma \ref{super man} follows 
by \eqref{R1 R2 paraprodotto}, \eqref{stima R2 paraprodotto},  \eqref{T1 T2 paraprodotto}, 
\eqref{stima T2 paraprodotto} (and the similar bound for $ T_1 $).

Let us prove the second estimate of Lemma \ref{super man}. 
By \eqref{secP}  the commutator
$$
[\langle D\rangle^s, T_a |D|^{\frac12}] u  = {\mathop \sum}_\xi 
{\mathop \sum}_{|j - \xi| \leq |\xi|} \psi(\xi, j) 
\widehat a(j - \xi) \widehat u(\xi)  e^{\ii  j x } 
$$
where 
$ \psi(\xi, j) := (\langle j \rangle^s  - \langle\xi \rangle^s ) |\xi|^{\frac12} $. 
Since $|j - \xi| \leq |\xi|$ we have  
$ |\psi(\xi, j)| \leq_s \langle \xi \rangle^s |j - \xi| $. 
Hence using as before the Cauchy-Schwartz inequality we get
\begin{align*}
\| [\langle D \rangle^s, T_a |D|^{\frac12}] u\|_{L^2_x}^2 & \leq_s
\sum_j \Big( \sum_{|j - \xi| \leq |\xi|} |\psi(\xi, j)| |\widehat a(j - \xi)| |\widehat u(\xi)| \Big)^2 \\
& \leq _s \Big( \sum_{|j - \xi| \leq |\xi|} \langle \xi \rangle^s |j - \xi| |\widehat a(j - \xi)| 
|\widehat u(\xi)| \frac{\langle j - \xi \rangle}{\langle j - \xi \rangle} \Big)^2  \\
& \leq_s \sum_\xi \langle \xi \rangle^{2 s} |\widehat u(\xi)|^2 \sum_j \langle j - 
\xi \rangle^4 |\widehat a(j - \xi)|^2 \leq_s
 \| a \|_{H^2_x}^2 \| u \|_{H^s_x}^2 	\, . 
\end{align*}
The lemma is proved. 
\end{proof}

\begin{proposition}\label{Prop0-flow}
Assume  $ \| a \|_{s_0 + \frac52} \leq 1   $. Then, $\forall \vphi \in \T^\nu $, for all $s \geq 0$   
the flow $ \Phi^t_N(\vphi) $ of \eqref{pseudo PDE N} satisfies
\begin{align}\label{stima flusso PDE s 0 1 N}
& {\rm sup}_{t \in [0, 1] } \| \Phi^t_N(\vphi) (u_0) \|_{H^s_x} \leq C \| u_0\|_{H^s_x}\,, 
\qquad \qquad  \forall 0 \leq s \leq 1 \\   
& \label{stima tame Phi t N}
{\rm sup}_{t \in [0, 1] } \| \Phi^t_N(\vphi) (u_0) \|_{H^s_x} \leq  C(s)  \big( \| u_0 \|_{H^s_x}  
+ \| a \|_{H^{s + \frac12}_x}  \| u_0\|_{H^1_x  }\big) \, , \ \  \forall s \geq 1 \, ,
\end{align}
uniformly for all $ N \in \N $.
The flow of  \eqref{pseudo PDE} is a linear bounded  operator  $ \Phi^t(\vphi) : H^s_x (\T) \to H^s_x (\T)$ satisfying 
\begin{align}\label{stima flusso PDE s 0 1}
& {\rm sup}_{t \in [0, 1] } \| \Phi^t(\vphi) (u_0) \|_{H^s_x} \leq C \| u_0\|_{H^s_x}\,, 
\qquad \qquad  \forall 0 \leq s \leq 1 \\   
& \label{stima tame Phi t}
{\rm sup}_{t \in [0, 1] } \| \Phi^t(\vphi) (u_0) \|_{H^s_x} \leq  C(s)  \big( \| u_0 \|_{H^s_x}  
+ \| a \|_{H^{s + \frac12}_x}  \| u_0\|_{H^1_x  }\big) \, , \quad \forall s \geq 1 \,.
\end{align}
\end{proposition}

\begin{proof}

\noindent
{\sc Proof of \eqref{stima flusso PDE s 0 1 N}, \eqref{stima tame Phi t N}}.
\\[1mm]
{\sc Step 1. $ s = 0 $.} 
For any $N \in \N$, the equation \eqref{pseudo PDE N} is an ODE on the finite dimensional space $E_N$
which admits a unique solution $u_N(t) = $ $ u_N(\lambda, t, \vphi, \cdot ) = $ $ \Phi_N^t(u_0) \in E_N $.
The $ L^2_x $-norm of the solution $ u_N(t) $  satisfies 
 (using that $\Pi_N$ is $L^2$ self-adjoint)  
\begin{align}
\partial_t \| u_N(t)\|_{L^2_x}^2 & = 
(\ii \Pi_N a |D|^{\frac12} u_N , u_N)_{L^2_x} + ( u_N , \ii \Pi_N a |D|^{\frac12} u_N  )_{L^2_x}  \nonumber\\
& = (\ii  a |D|^{\frac12} u_N , u_N)_{L^2_x} + ( u_N , \ii  a |D|^{\frac12} u_N  )_{L^2_x} 
= (\ii[a , |D|^{\frac12}] u_N,u_N )_{L^2_x} 
\end{align}
because $ a $ is real. 
Lemma \ref{lemma tame norma commutatore}, 
\eqref{lemma: action Sobolev}, \eqref{norma a moltiplicazione}, \eqref{Norm Fourier multiplier}, 
and  $ \| a \|_{s_0 + \frac52} \leq 1  $, imply the commutator estimate
$ \| [a, |D|^{\frac12}] \|_{{\mathcal L}(L^2_x)}  \leq C $. Hence 
$ \partial_t \| u_N (t)\|_{L^2_x}^2 \leq C \| u_N(t)\|_{L^2_x}^2 $ and 
Gronwall inequality implies \eqref{stima flusso PDE s 0 1 N} for $s = 0$. 
\\[1mm]
{\sc Step 2. $ s \geq 1 $.}
The Sobolev norm $ \| u_N \|_{H^s_x}^2 = \| \langle D \rangle^s u_N \|_{L^2_x}^2 $ 
satisfies 
\begin{align}
\partial_t \| \langle D \rangle^s u_N \|_{L^2_x}^2 & = 
 \big( \langle D \rangle^s \Pi_N \ii a |D|^{\frac12} u_N , \langle D \rangle^s u_N \big)_{L^2_x} + 
 \big(\langle D \rangle^s  u_N ,\langle D \rangle^s \Pi_N  \ii a |D|^{\frac12} u_N \big)_{L^2_x} \nonumber  \\ 
 & =  \big( \langle D \rangle^s\ii a |D|^{\frac12} u_N , \langle D \rangle^s u_N \big)_{L^2_x} + 
 \big(\langle D \rangle^s  u_N ,\langle D \rangle^s   \ii a |D|^{\frac12} u_N \big)_{L^2_x} \nonumber\\
& = 
 \big(\langle D \rangle^s \ii T_a  (|D|^{\frac12} u_N), \langle D \rangle^s u_N \big)_{L^2_x} + 
 \big(\langle D \rangle^s  u_N, \langle D \rangle^s \ii T_a (|D|^{\frac12} u_N) \big)_{L^2_x} \label{termine-prin} \\
& \quad \, +  \big(\langle D \rangle^s \ii  R_{|D|^{\frac12} u_N} a, \langle D \rangle^s u_N \big)_{L^2_x} + 
\big(\langle D \rangle^s  u_N, \langle D \rangle^s \ii  R_{|D|^{\frac12} u_N} a \big)_{L^2_x} \label{stima di energia} 
\end{align}
by  the paraproduct decomposition \eqref{paraproduct} of  $ a |D|^{\frac12} u_N  = T_a |D|^{\frac12} u_N + R_{|D|^{\frac12} u_N} a $.
\\[1mm]
{\sc Estimate of \eqref{termine-prin}.} We write
\begin{align}
\eqref{termine-prin} & = 
\big(\ii T_a |D|^{\frac12} \langle D \rangle^s u_N, \langle D \rangle^s u_N \big)_{L^2_x} + \big(\ii [\langle D \rangle^s, T_a |D|^{\frac12}]  u_N, \langle D \rangle^s u_N \big)_{L^2_x} \nonumber \\
& \, + \big(\langle D \rangle^s  u_N ,  \ii T_a |D|^{\frac12} \langle D \rangle^s u_N \big)_{L^2_x} + \big(\langle D \rangle^s  u_N, \ii [\langle D \rangle^s , T_a |D|^{\frac12}] u_N \big)_{L^2_x} \nonumber \\
& =  \big(\ii [\langle D \rangle^s, T_a |D|^{\frac12}]  u_N, \langle D \rangle^s u_N \big)_{L^2_x} + 
\big(\langle D \rangle^s  u_N, \ii [\langle D \rangle^s , T_a |D|^{\frac12}] u_N \big)_{L^2_x} \nonumber \\
&  \, + 
\big( \ii (T_a |D|^{\frac12} - |D|^{\frac12} (T_a)^* ) \langle D \rangle^s u_N, \langle D \rangle^s u_N \big)_{L^2_x} \, . 
\label{secondo pezzo stima di energia}
\end{align}
Thus \eqref{secondo pezzo stima di energia} and Lemma \ref{super man} imply
that the term in \eqref{termine-prin} satisfies 
\begin{equation}\label{super man 2}
\big| \big(\langle D \rangle^s \ii T_a  |D|^{\frac12} u_N, \langle D \rangle^s u_N \big)_{L^2_x} + \big(\langle D \rangle^s  u_N, \langle D \rangle^s \ii T_a |D|^{\frac12} u_N \big)_{L^2_x}  
\big| \leq_s \| a \|_{H^2_x} \| u_N \|_{H^s_x}^2 \, .
\end{equation}
{\sc Estimate of \eqref{stima di energia}.} Cauchy-Schwartz inequality and \eqref{stima Ta Ru} imply
\begin{equation}\label{super man 0}
\begin{aligned}
& \big| \big(\langle D \rangle^s \ii  R_{|D|^{\frac12} u_N} a, \langle D \rangle^s u_N \big)_{L^2_x}  +  
\big(\langle D \rangle^s  u_N, \langle D \rangle^s \ii  R_{|D|^{\frac12} u_N} a \big)_{L^2_x} \big| \\
& \leq_s
\| \langle D \rangle^s u_N \|_{L^2_x} \| a \|_{H^{s + \frac12}_x}  \| u_N \|_{H^1_x}\,.
\end{aligned}
\end{equation}
By \eqref{termine-prin}-\eqref{stima di energia},  \eqref{super man 2}, \eqref{super man 0},
$  \| a \|_{H^2_x}  \leq 1  $,  we deduce the differential inequality: $ \forall s \geq 1 $
\be
\begin{aligned}
\partial_t \| u_N \|_{H^s_x}^2 
& \leq_s  \| a \|_{H^{s + (1/2)}_x } 
\| u_N \|_{H^s_x} \| u_N \|_{H^{1}_x } + \| a \|_{H^2_x} \|  u_N \|_{H^s_x}^2 \\
& \leq_s \| a\|_{H^{s + (1/2)}_x}^2 \| u_N \|_{H^{1}_x }^2 + 
\|  u_N \|_{H^s_x}^2  \,  .\label{diff-ine}
\end{aligned}
\ee
For $ s = 1 $ and since $  \| a \|_{H^2_x}  \leq 1 $, \eqref{diff-ine} reduces to 
$ \partial_t \| u_N \|_{H^1_x}^2  \leq C \|  u_N \|_{H^1_x}^2 $, which implies  
$ \| \Phi^t_N (u_0) \|_{H^1_x  } \leq C' \| u_0 \|_{H^1_x }$, $  \forall t \in [0, 1] $.  
For $ s > 1 $,  \eqref{diff-ine} reduces to 
$ \partial_t \| u_N \|_{H^s_x}^2 \leq 
C(s) \big( \| a\|_{H^{s + (1/2)}_x}^2  \| u_0 \|_{H^1_x }^2 + \|  u_N \|_{H^s_x}^2 \big) $ and 
the estimate \eqref{stima tame Phi t N} follows by the  Gronwall inequality in differential form.

Since $\Phi^t_N : H^0_x(\T) \to H^0_x(\T)$ and $\Phi^t_N : H^1_x(\T) \to H^1_x(\T) $ are linear bounded
operators,  
a classical  interpolation result implies  that 
$\Phi^t_N : H^s_x(\T) \to H^s_x(\T)$ is also bounded  $ \forall s \in [0, 1]$ and 
 \eqref{stima flusso PDE s 0 1 N} holds. 
 
 \smallskip
 
 \noindent
 {\sc Proof of \eqref{stima flusso PDE s 0 1}, \eqref{stima tame Phi t}.} 
 Now we  pass to the limit $ N \to + \infty $. 
By \eqref{stima flusso PDE s 0 1 N} the sequence of functions $u_N( t, \cdot )$ is bounded in $L^\infty_{ t} H^s_x$ and,  
 up to subsequences,  
 \begin{equation}\label{disuguaglianza liminf u uN}
 u_N \stackrel{w^*}{\rightharpoonup} u \ \  \text{in } \  L^\infty_{ t} H^s_x \, , 
\qquad  \| u \|_{L^\infty_{ t} H^s_x} \leq \liminf_{N \to + \infty} \| u_N \|_{L^\infty_{ t} H^s_x}\,.
 \end{equation} 
  {\sc Claim}: {\it the sequence $ u_N \to u $ in ${\mathcal C}^0_{ t} H^{s}_x  \cap {\mathcal C}^1_{ t} H^{ s - \frac12}_x  $, 
 and $ u (t, x)$ solves the equation \eqref{pseudo PDE}}. 
  \\[1mm]
We first prove that $u_N$ is a Cauchy sequence in ${\mathcal C}^0_{ t}L^2_x$. Indeed, by \eqref{pseudo PDE N},  
    the difference $h_N := u_{N + 1} - u_N$ solves 
    $$
    \partial_t h_N = \ii \Pi_{N+ 1} ( a |D|^{\frac12}h_N ) + \ii (\Pi_{N + 1} - \Pi_N) a |D|^{\frac12} u_N\,, 
    \ \  h_N(0) = (\Pi_{N + 1} - \Pi_N) u_0 \, , 
    $$ 
    and therefore
    \begin{align}
    \partial_t \| h_N(t)\|_{L^2_x}^2 & = (\partial_t h_N\,,\,h_N )_{L^2_x} + ( h_N\,,\,\partial_t h_N )_{L^2_x}  \nonumber\\
    & = (\ii \Pi_{N+ 1} ( a |D|^{\frac12}h_N ), h_N )_{L^2_x} +  ( h_N, \ii \Pi_{N+ 1} ( a |D|^{\frac12}h_N ) )_{L^2_x}  \nonumber\\
    & \quad + (\ii (\Pi_{N + 1} - \Pi_N) a |D|^{\frac12} u_N, h_N )_{L^2_x} \nonumber\\
    & \quad + 
    ( h_N, \ii (\Pi_{N + 1} - \Pi_N) a |D|^{\frac12} u_N )_{L^2_x} \,. \label{coccodrillo 0}
    \end{align}
Since  $\Pi_{N + 1}$ is self-adjoint with respect to the $L^2$ scalar product  
  \begin{align}
 (\ii \Pi_{N+ 1} ( a |D|^{\frac12}h_N ), h_N )_{L^2_x} + 
 ( h_N, \ii \Pi_{N+ 1} ( a |D|^{\frac12}h_N ) )_{L^2_x}  
 & =  (\ii a |D|^{\frac12}h_N ), h_N  )_{L^2_x} \nonumber\\
 & \quad + ( h_N, \ii a |D|^{\frac12}h_N )_{L^2_x} \nonumber\\
 & = (\ii [a, |D|^{\frac12}]h_N )\,,\,h_N )_{L^2_x}\nonumber\\
 &  \leq C \| h_N(t)\|_{L^2_x}^2 \,. \label{coccodrillo1}
  \end{align}
  Moreover 
  \begin{align}
  &(\ii (\Pi_{N + 1} - \Pi_N) a |D|^{\frac12} u_N, h_N )_{L^2_x} + ( h_N ,\,\ii (\Pi_{N + 1} - \Pi_N) a |D|^{\frac12} u_N )_{L^2_x} \nonumber\\
  & \leq 2 \| (\Pi_{N + 1} - \Pi_N) a |D|^{\frac12} u_N  \|_{L^2_x} \| h_N\|_{L^2_x}  \nonumber \\
&  \leq 
  \| h_N\|_{L^2_x}^2 + \| (\Pi_{N + 1} - \Pi_N) a |D|^{\frac12} u_N  \|_{L^2_x}^2  \nonumber\\
  & \lessdot  \| h_N\|_{L^2_x}^2 + \big( N^{-2} \|  a |D|^{\frac12} u_N  \|_{H^2_x} \big)^2  \nonumber\\
&  \lessdot  \| h_N\|_{L^2_x}^2  + \big( N^{-2} \| u_0  \|_{H^{5/2}_x} \big)^2  \label{coccodrillo 2}
  \end{align}
using that $  \| a \|_{H^2_x} \leq 1 $.  Hence \eqref{coccodrillo 0}-\eqref{coccodrillo 2} imply that 
  $$  
  \partial_t \| h_N(t)\|_{L^2_x}^2 \lessdot \| h_N(t)\|_{L^2_x}^2 + N^{-4}\| u_0\|_{H^{5/2}_x}^2 
  $$ 
 and, since $ \| h_N(0)\|_{L^2_x} \leq N^{- 2} \| u_0\|_{H^2_x}$,  by Gronwall lemma we deduce that
$$
\| u_{N + 1} - u_N\|_{{\mathcal C}^0_{ t} L^2_x} = \sup_{t\in [0, 1]} \| u_{N + 1}( t, \cdot) - u_{N}(t,\cdot)\|_{L^2_x} 
\lessdot N^{-2} \| u_0\|_{H^{5/2}_x}\,.
 $$
 The above inequality implies that $ u_N $ is a Cauchy sequence in 
 $ {\mathcal C}^0_{ t} L^2_x $. 
Hence $ u_N \to \tilde u \in {\mathcal C}^0_{ t} L^2_x $. By \eqref{disuguaglianza liminf u uN} 
we have $ u = \tilde u  \in {\mathcal C}^0_{ t} L^2_x \cap L^\infty_t H^s_x $. 
 Next, for any $ \bar s \in [0, s) $ we use the interpolation inequality 
  $$
  \| u_N - u \|_{L^\infty_{t} H^{\bar s}_x } \leq 
  \| u_N - u \|_{L^\infty_{ t} L^2_x }^{1 - \lambda} \| u_N - u \|_{L^\infty_{ t} H^{\bar s}_x }^\lambda \, ,
$$
and, since  $ u_N $ is bounded in $ L^\infty_t H^s_x $ (see \eqref{stima flusso PDE s 0 1 N}, \eqref{stima tame Phi t N}),
$ u  \in L^\infty_t H^s_x $, 
and $ u_N \to u \in {\mathcal C}^0_{ t} L^2_x $, we deduce that  $u_N \to u$ in each $ L^\infty_t H^{\bar s}_x$.
Since $u_N \in {\mathcal C}^0_t H^{\bar s}_x$  are continuous  in $ t$, the limit function 
$u \in {\mathcal C}^0_t H^{\bar s}_x $ is continuous as well. 
  Moreover  we also deduce that 
$$
  \partial_t u_N = \ii \Pi_N ( a |D|^{\frac12} u_N ) \to \ii a |D|^{\frac12} u \quad  \text{in}\ \  {\mathcal C}^0_{ t} H^{\bar s - 1/2}_x  
  \, , \quad \forall \bar s \in [0, s) \, . 
$$
 As a consequence $ u \in 
 {\mathcal C}^1_t H^{\bar s - \frac12}_x $ 
 and $ \pa_t u = \ii a |D|^{\frac12} u $  
  solves \eqref{pseudo PDE}. 
 
Finally,  
arguing as in \cite{Tay}, Proposition 5.1.D, 
it follows that the function $t \to \| u(t)\|_{H^s_x}^2$ is Lipschitz.  Furthermore,  if $ t_n \to t $ then 
$ u(t_n ) \rightharpoonup u(t) $ weakly in $  H^s_x $, because 
 $u(t_n) \to u(t)$ in $H^{\bar s}_x$ for any $\bar s \in [0, s)$.
 As a consequence the sequence $u(t_n) \to u(t) $ strongly in $H^s_x$. 
 This proves that  $u \in {\mathcal C}^0_t H^s_x$ and therefore $\partial_t u = \ii a |D|^{\frac12} u \in {\mathcal C}^0_t H^{s - \frac12}_x$. 
 \\[1mm]
  {\sc Uniqueness.} If $u_1, u_2 \in {\mathcal C}^0_t H^{s}_x \cap {\mathcal C}^1_t H^{s - \frac12}_x $, $s \geq 1/2 $,  
  are solutions of \eqref{pseudo PDE}, then $h: = u_1 - u_2$ solves 
  $$
  \partial_t h = \ii a |D|^{\frac12} h\,, \qquad h(0) = 0\,.
  $$
  Arguing as in the proof of \eqref{diff-ine} we deduce the energy inequality 
  $  \partial_t \| h (t)\|_{L^2_x}^2 \leq C \| h(t)\|_{L^2_x}^2 $. Since $ h(0)= 0 $, 
Gronwall lemma implies that $\| h(t)\|_{L^2_x}^2 = 0$, for any $t \in [0, 1]$, i.e.  $h(t) = 0$. Therefore the problem \eqref{pseudo PDE} has a unique solution $u(t)$ that we denote by  $\Phi^t(u_0)$. 
  The estimate \eqref{stima flusso PDE s 0 1}, \eqref{stima tame Phi t} then follows by
  \eqref{disuguaglianza liminf u uN}, \eqref{stima flusso PDE s 0 1 N}, \eqref{stima tame Phi t N}, since $u_N(t) = \Phi_N^t(u_0)$. 
\end{proof}

In the next lemma we prove the smooth dependence of the flow with respect to parameters.

\begin{lemma}\label{dipendenza liscia dai parametri}
Let $a(z, \cdot ) \in {\mathcal C}^\infty(\T)$ and 
$ p_0 $-times differentiable, resp. ${\mathcal C}^{p_0} $,  with respect to $z \in {\mathcal B}_X $, where ${\mathcal B}_X$ is 
an open subset of a  Banach space $ X $. 
Then, for any $ p \leq p_0 $, the flow map $ \Phi(z, t)$, $ t \in [0, 1] $,  
  is smooth in $ z $, more precisely,  the map 
$$
{\mathcal B}_X \ni z \mapsto \Phi(z, t) 
\in {\mathcal L}(H^{s}_x, H^{s - \frac{p}{2} - \frac12}_x) \, , \quad \forall s  \geq (p/2) + (1/2)  \, ,
$$
is $ p $-times differentiable, resp. $ {\mathcal C}^p $.
 Moreover, for any $z \in {\mathcal B}_X $,  the derivative $\partial_z^p \Phi(z, t)$ is a 
 multilinear form from $X^p$ in ${\mathcal L}(H_x^s, H^{s - \frac{p}{2}}_x)$.
\end{lemma} 

\begin{proof}
We denote for simplicity $ \| \ \|_{{\mathcal L}(H^s_x)} := \| \ \|_{{\mathcal L}(H^s_x, H^s_x)} $. 
We argue by induction on $ p $. We first prove the statement for $ p = 0 $. Let $ s \geq 1/2 $. 
By \eqref{flow-propagator}, we have that  
$\Delta_z \Phi(z, t) := \Phi(z + z_1 , t) - \Phi(z , t ) $ solves 
$$
\partial_t \Delta_z \Phi(t) = \ii a(z + z_1, x) |D|^{\frac12} \Delta_z \Phi(t) + 
\ii \Delta_z a |D|^{\frac12} \Phi(z, t)\,, \quad \Delta_z \Phi( 0) = 0\,,
$$
where $\Delta_z a := a (z + z_1, x ) - a(z, x)$. By  Duhamel principle 
$$ 
\Delta_z \Phi(z, t) = \int_0^t \Phi(z + z_1 , t - \tau) \ii \Delta_z a |D|^{\frac12} \Phi(z, \tau) d \tau \, .
$$ 
Hence
\be\label{Lip-flow}
\begin{aligned}
& \sup_{t \in [0, 1]} \| \Delta_z \Phi(z, t) \|_{{\mathcal L}(H^s_x, H^{s - \frac12}_x)}   \\
& \leq \sup_{t \in [0, 1]} \|\Phi(z + z_1, t) \|_{{\mathcal L}(H^{s - \frac12}_x)} \| \Delta_z a \|_{{\mathcal C}^{s - \frac12}_x} \sup_{t \in [0, 1]} 
\|\Phi(z, t) \|_{{\mathcal L}(H^{s }_x)} \to 0  
\end{aligned}
\ee
as $ z_1 \to 0 $,  because $ \| a (z + z_1 ) - a(z )  \|_{{\mathcal C}^{s - \frac12}_x} \to 0 $ 
by continuity. 

Now we assume that for all $ 0 \leq q \leq p < p_0 $,  the flow  
$$ z \mapsto \Phi(z, t) \in {\mathcal L}(H^s_x, H^{s - \frac{q}{2} - \frac12}_x), \quad s \geq q/2 + 1/2, $$ 
is $ q $-times differentiable, with 
$ \partial_z^q \Phi : X^q \to {\mathcal L}(H^s_x, H^{s - \frac{q}{2} }_x)  $
and we prove that 
$ z \mapsto  \Phi(z,t ) \in {\mathcal L}(H^{s}_x, H^{s - \frac{p + 1}{2} - \frac12}_x) $, $s \geq (p + 1)/2 + 1/2 $, 
 is $ (p+1) $-times differentiable with 
 $ \partial_z^{p + 1} \Phi(z, t): X^{p+1} \to  {\mathcal L}(H^s_x, H^{s - \frac{p + 1}{2}}_x) $.
   
The derivate $ \partial_z^p \Phi(z, t) $ solves the equation, for any $z_1, \ldots, z_p \in X $, 
\be\label{equazione derivata in lambda p del flusso}
\begin{aligned}
 &  \partial_t \big( \partial_z^p \Phi(z, t)[z_1, \ldots, z_p] \big)  = \\ 
& \ii a(z, x) 
 |D|^{\frac12} \partial_\lambda^p \Phi(z, t)[z_1, \ldots, z_p]  
 +  F_p(z, t)[z_1, \ldots,z_p]\,, \,  
 \partial_z^p \Phi(z, 0) = 0 
 \end{aligned}
\ee
where $ F_0 := 0 $ and,  for any $ 1 \leq q \leq p + 1 $,    
\begin{align}
 F_q(z, t)[z_1, \ldots, z_q] := & \label{definizione Fq lemma parametri}\\
 \sum_{
0 \leq q_1 \leq q -1, \s \in {\mathcal P}_q} \!\!\! \! \!\!\! \!  \ii \partial_z^{q - q_1} a (z)[z_{\s(1)}, \ldots, z_{\s(q - q_1)}] |D|^{\frac12} \partial_z^{q_1} \Phi(z, t)[z_{\s(q - q_1 + 1)}, \ldots, z_{\s(q)}]  & \nonumber
\end{align}
denoting by $ {\mathcal P}_q $  the set of permutations of the indices $ \{ 1, \ldots , q \} $. 
For $ 0 \leq q \leq p$  we have 
\begin{equation}\label{formula utile derivata lambda flusso}
\quad F_{q + 1}(z, t) = \partial_z F_q(z, t) + \ii \partial_z a(z, x)[ \cdot ] |D|^{\frac12} \partial_z^q \Phi(z, t)\, .
\end{equation}
The candidate $ (p + 1)$-derivative of the operator $ \Phi(z, t)$ is the multilinear $ (p + 1)$-form 
\be \label{formula utile derivata lambda flusso-2}
{\mathcal A}_p(z, t)[z_1, \ldots, z_{p + 1}] := \int_0^t \Phi(z, t - \tau) F_{p + 1}(z, \tau)[z_1, \ldots, z_{p + 1}]\, d \tau 
\ee
obtained by differentiating formally the equation \eqref{equazione derivata in lambda p del flusso} and using the Duhamel principle. We now estimate 
$ \partial_z^p \Phi(z + z_{p + 1}, t) - \partial_z^p \Phi(z, t) 
- {\mathcal A}_p(z , t)[z_{p + 1}] $.
Note that, since ${\mathcal A}_p(z, t)$ is a multilinear $(p + 1)$-form, then ${\mathcal A}_p(z, t)[z_{p + 1}]$ is a 
multilinear $p$-form. 
Taking the difference of \eqref{equazione derivata in lambda p del flusso} evaluated at $ z+ z_{p+1} $ and 
$ z $, and using the Duhamel principle 
we get  that 
$$
\begin{aligned}
\Delta_z \partial_z^p \Phi(z, t) & := \partial_z^p \Phi(z + z_{p + 1}, t) - \partial_z^p \Phi(z, t) \\
& = \int_0^t \Phi( z + z_{p + 1}, t - \tau) 
\big(\ii  \Delta_z a 
|D|^{\frac12} \partial_z^p \Phi(z, t) +  \Delta_z F_p \big)  
d \tau
\end{aligned}
$$
where
$ \Delta_z a := a(z+ z_{p + 1}, x) - a (z, x) $ and  
$ \Delta_z F_p := F_p(z + z_{p + 1}, t) - F_p(z, t) $.
Hence, by 
\eqref{formula utile derivata lambda flusso-2} and 
 \eqref{formula utile derivata lambda flusso} with $ q = p  $,  we get
\begin{align}
& \Delta_z \partial_z^p \Phi(z, t) - {\mathcal A}_p(z, t)[z_{p + 1}] =  \int_0^t \Big( {\mathcal R}_\Phi^{(1)}(t, \tau, z) + {\mathcal R}_\Phi^{(2)}(t, \tau, z) \Big) \, d \tau \nonumber 
\end{align}
\begin{align}
 {\mathcal R}_\Phi^{(1)}(t, \tau, z) & := \int_0^t 
\Phi(z + z_{p + 1}, t - \tau) \ii   \Delta_z a |D|^{\frac12} \partial_z^p \Phi(z, \tau) d \tau \nonumber\\
& \quad - \int_0^t \Phi(z, t - \tau) \ii \partial_z a (z)[ z_{p + 1}] 
|D|^{\frac12} \partial_z^p \Phi(z, \tau) d \tau \label{rapporto incrementale p 1} 
\end{align}
\begin{align}
  {\mathcal R}_\Phi^{(2)}(t, \tau, z) & :=  \int_0^t \Phi(z + z_{p + 1}, t - \tau) \Delta_z F_p  d \tau \nonumber\\
  & \quad - \int_0^t \Phi(z, t - \tau) \partial_z F_p(z, \tau)[z_{p + 1}] d \tau \, .  \label{rapporto incrementale p 3} 
\end{align}
\noindent
{\sc Estimate of \eqref{rapporto incrementale p 1}.} Set  
$ \Delta_z \Phi (t) := \Phi( z + z_{p + 1}, t) - \Phi( z , t) $. For all $ 0 \leq \tau \leq t $, we have
\begin{align}
 \big\|   {\mathcal R}_\Phi^{(1)}(t, \tau, z)[z_1, \ldots, z_p]   \Big\|_{{\mathcal L}(H^s_x, H^{s - \frac{p + 1}{2} - \frac12}_x)} \nonumber& \\
 \leq \Big\|   \Phi(z + z_{p + 1}, t - \tau) \ii (\Delta_z a - \partial_z a[z_{p + 1}])   |D|^{\frac12} \partial_z^p \Phi(z, \tau)[z_1, \ldots, z_p]   \Big\|_{{\mathcal L}(H^s_x, H^{s - \frac{p + 1}{2} - \frac12}_x)} \nonumber & \\
 + \Big\|   \Delta_z \Phi (t - \tau ) \ii  \partial_z a (z)[z_{p + 1}] |D|^{\frac12} \partial_z^p \Phi(z, \tau)[ z_1, \ldots, z_p]   \Big\|_{{\mathcal L}(H^s_x, H^{s - \frac{p + 1}{2} - \frac12}_x)} \nonumber & \\ 
 \leq \sup_{t \in [0, 1]} 
\| \Phi( z + z_{p + 1}, t)\|_{{\mathcal L}(H^{s - \frac{p + 1}{2} - \frac12}_x)} \| \Delta_z a - \partial_z a[ z_{p + 1}] \|_{{\mathcal C}_x^{s - \frac{p + 1}{2} - \frac12}} \times \nonumber& \\
 \times\sup_{t \in [0, 1]} \| \partial_z^p \Phi(z, \tau)[z_1, \ldots, z_p] \|_{{\mathcal L}(H^s_x, H^{s - \frac{p}{2} - \frac12}_x)} \nonumber& \\
 + \sup_{t \in [0, 1]} \| \Delta_z \Phi (t) \|_{{\mathcal L}(H^{s - \frac{p + 1}{2} }_x,H^{s - \frac{p + 1}{2} - \frac12}_x )} \| \partial_z a (z)[z_{p + 1}]\|_{{\mathcal C}_x^{s - \frac{p + 1}{2}}} \times \nonumber & \\
 \times \sup_{t \in [0, 1]} 
\| \partial_z^p \Phi(z, \tau)[ z_1, \ldots, z_p] \|_{{\mathcal L}(H^s_x, H^{s - \frac{p}{2}}_x)} \nonumber& \\
 \leq_{s,p} \Big(\| \Delta_z a - \partial_z a[z_{p + 1}] \|_{{\mathcal C}_x^{s - \frac{p + 1}{2}}}  \nonumber & \\
   + \sup_{t \in [0, 1]} \| \Delta_z \Phi (t) \|_{{\mathcal L}(H^{s - \frac{p + 1}{2}}_x,H^{s - \frac{p + 1}{2} - \frac12}_x )} \| z_{p + 1}\|  \Big) \| z_1\|\ldots \| z_p\|  \label{st1-d} & 
\end{align}
using the inductive assumption on $ \partial_z^p \Phi(z, \tau) $.
\\[1mm]
{\sc Estimate of \eqref{rapporto incrementale p 3}.}
By the expression in  \eqref{definizione Fq lemma parametri} (with $ q = p $),  
the fact that $ z \mapsto a(z) $ is $ (p+1)$-times differentiable, 
the inductive differentiability properties of the flow, 
 the map
$ z \mapsto F_p(z, t)[z_1, \ldots, z_p] \in {\mathcal L}(H^s_x, H^{s - \frac{p}{2} - \frac12}_x) $ 
is differentiable. 
Arguing as above, we have, for all $ 0 \leq \tau \leq t $, 
\begin{align}
 \label{st2-d}  \Big\|  {\mathcal R}_\Phi^{(2)}(t, \tau, z)[z_1, \ldots,z_p]  \Big\|_{{\mathcal L}(H^s_x, H_x^{s - \frac{p + 1}{2} - \frac12})} 
  & \\
 \leq_{s,p}
 \sup_{t \in [0, 1]} \| \big( \Delta_z F_p(z, \tau)  -  \partial_z F_p( z , \tau)[z_{p + 1}] \big)[z_1, \ldots, z_p] \|_{{\mathcal L}(H^s_x, H^{s - \frac{p + 1}{2} - \frac12}_x)} \nonumber & \\
 +  \sup_{t \in [0, 1]}
\| \Delta_z  \Phi(z, t )\|_{{\mathcal L}(H_x^{s - \frac{p + 1}{2} }, H_x^{s - \frac{p + 1}{2} - \frac12})} 
\| \partial_z F_p (z)[z_{p+1}]
 [z_1, \ldots, z_p] \|_{{\mathcal L}(H_x^s , H_x^{s - \frac{p + 1}{2} })} 
 \, . \nonumber 
\end{align}
In conclusion, by \eqref{rapporto incrementale p 1}, \eqref{rapporto incrementale p 3}, \eqref{st1-d}, \eqref{st2-d},
the differentiability of $ a(z) $ and  \eqref{Lip-flow}, 
we deduce that 
$$
\sup_{t \in [0, 1]} \sup_{\|z_1\|, \ldots, \|z_p\| \leq 1} \!\!\!\!\!\!
\frac{\|  \big(\Delta_z \partial_z^p \Phi(z, t) -  {\mathcal A}_p(z, t)[z_{p + 1}] \big)[z_1, \ldots, z_p] \|_{{\mathcal L}(H^s_x, H^{s - \frac{p + 1}{2} - \frac12}_x)}}{\| z_{p + 1}\|} \to 0\,,
$$
for $ z_{p + 1} \to 0 $,  namely $ \partial_z^{p} \Phi(z, t) $ is differentiable and 
$ \partial_z^{p + 1} \Phi(z, t) = {\mathcal A}_p(z, t) $. 
 Moreover, by \eqref{formula utile derivata lambda flusso-2}, \eqref{definizione Fq lemma parametri} for
 $ q = p +1 $,  
 the continuity of $ z \mapsto \pa_z^p a (z) $ and 
 the inductive differentiability properties of the flow,   we have that 
  $ z \mapsto \partial_z^{p + 1} \Phi( z , t) $ is continuous and 
$ \partial_z^{p + 1} \Phi(z , t)[ z_1, \ldots, z_{p + 1}] \in  {\mathcal L}(H^s_x, H^{s - \frac{p + 1}{2}}_x) $. 
\end{proof}

We now want to prove tame estimates for the flow operator 
$ \Phi^t := \Phi ( t) := \Phi (\lambda, \vphi, t) $ acting  in the Sobolev spaces $ H^s $ of functions  $ u(\vphi, x) $. 
Recall that the Sobolev norm  $ \| \ \|_s $  in \eqref{unified norm} 
is equivalent  to $ \| \ \|_s \simeq \| \ \|_{H^s_\ph L^2_x} + \| \ \|_{L^2_\ph H^s_x} $, see \eqref{Sobolev norm}.
Note also the continuous embeddings
\be\label{continuous embed}
H^{s + s_0}(\T^{\nu + 1}) \hookrightarrow H^{s_0} (\T^\nu, H^s_x ) \hookrightarrow 
L^\infty(\T^\nu, H^s_x ) \, . 
\ee

\begin{lemma}\label{lemma finitezza norme stime flusso}
For any $|\beta| \leq \beta_0, |k| \leq k_0, t \in [0, 1], h \in {\mathcal C}^\infty(\T^{\nu + 1})$, 
the function $ \partial_\lambda^k \partial_\vphi^\beta \Phi^t (\vphi )  h $ is ${\mathcal C}^\infty(\T^{\nu + 1})$. 
\end{lemma}

\begin{proof}
Since $ h (\vphi, x ) \in {\mathcal C}^\infty (\T^\nu \times \T) $ then $ \T^\nu \ni \vphi \mapsto h (\vphi, \cdot ) \in H^s_x 
$ is a $ {\mathcal C}^\infty $
map for  any $ s >  0 $.  
By Lemma \ref{dipendenza liscia dai parametri}, the map
$ \T^\nu \ni \vphi \mapsto  \partial_\lambda^k \partial_\vphi^\beta \Phi^t (\vphi)  [ h(\vphi) ] \in H^s_x $ 
is $ {\mathcal C}^\infty $ and,  for any $ \alpha \in \N^\nu $, 
$ \partial_\vphi^\alpha \big\{ \partial_\lambda^k \partial_\vphi^\beta \Phi^t (\vphi)  h \big\} = 
{\mathop \sum}_{\alpha_1 + \alpha_2 = \alpha} C_{\alpha_1, \alpha_2} 
\partial_\lambda^k \partial_\vphi^{\beta + \alpha_1} \Phi^t (\vphi)  [\partial_\vphi^{\alpha_2} h] $.  
By Lemma \ref{dipendenza liscia dai parametri} each function $  \partial_\lambda^k \partial_\vphi^{\beta + \alpha_1} \Phi^t (\vphi)  [\partial_\vphi^{\alpha_2} h] \in {\mathcal C}^\infty_x $. 
\end{proof}

\begin{proposition} \label{Prop1-flow}
Assume that  
\begin{equation}\label{ansatz derivate vphi flusso}
 \quad \| a \|_{2 s_0 + \frac32} \leq 1 \, , \quad  \| a \|_{2 s_0 +  1 } \leq \delta (s)  
\end{equation}
for some $ \d(s) > 0 $ small. Then the following tame estimates hold: 
\begin{align}\label{stima flusso norma bassa}
& {\rm sup}_{t \in [0, 1]}\| \Phi(t) u_0 \|_s \leq C(s) \| u_0\|_s\,, \qquad \quad  \forall s \in [0, s_0 + 1 ] \, , \\
& \label{stima flusso norme unite}
{\rm sup}_{t \in [0, 1]} \|  \Phi( t) u_0 \|_{s} \leq C(s) \big( \| u_0 \|_{s} + \| a \|_{s + s_0 + \frac12 } \| u_0\|_{s_0} \big)\,, \quad \forall s \geq s_0  \,.
\end{align}
\end{proposition}

\begin{proof}
We take $u_0 \in {\mathcal C}^\infty(\T^{\nu + 1})$, so that $ \Phi u_0$ is ${\mathcal C}^\infty(\T^{\nu + 1})$. 

\noindent
{\sc Proof of \eqref{stima flusso norma bassa}.} 
\noindent
For $ s = 0 $, integrating  \eqref{stima flusso PDE s 0 1} in $ \vphi $, we have
\be
\begin{aligned}
\| \Phi(t) u_0 \|_0^2 & = \| \Phi(t) u_0 \|_{L^2_\vphi L^2_x}^2 \\
& = \int_{\T^\nu} \| \Phi(\vphi, t) u_0 \|_{L^2_x}^2\,d \vphi 
\leq C \int_{\T^\nu} \| u_0 \|_{L^2_x}^2\,d \vphi = C \| u_0 \|_{L^2_\vphi L^2_x}^2\,. \label{stima flusso L 2 vphi L 2 x}
\end{aligned}
\ee
Now we suppose that \eqref{stima flusso norma bassa} holds for
$ s \in \N $, $  s \leq s_0 $, and we prove it for  $ s + 1 $. By \eqref{Sobolev norm}
\be\label{split norma s + 1}
\| \Phi(t) u_0\|_{s + 1} \simeq \|\Phi(t) u_0 \|_{L^2_\vphi H^{s + 1}_x} + \|\Phi(t) u_0 \|_{H^{s + 1}_\vphi L^2_x} \, . 
\ee
The first term in \eqref{split norma s + 1} is estimated, using   \eqref{stima tame Phi t},  \eqref{continuous embed}, 
\eqref{ansatz derivate vphi flusso},  by 
\be
\begin{aligned}
\|\Phi(t) u_0 \|_{L^2_\vphi H^{s + 1}_x} & \leq_s \| u_0 \|_{L^2_\vphi H^{s + 1}_x} + \| a \|_{L^\infty_\vphi  
H^{s + \frac32}_x} \| u_0 \|_{L^2_\vphi H^1_x} \\
&  \leq_s \| u_0 \|_{s + 1} + \| a \|_{s +  s_0 + \frac32} \| u_0 \|_{1} \label{stima tame Phi t-1}     \leq_s \| u_0\|_{s + 1}  \, . 
\end{aligned}
\ee
The second term in \eqref{split norma s + 1} is estimated, using \eqref{stima flusso L 2 vphi L 2 x} and 
\eqref{stima flusso norma bassa},  by
\begin{align}
\| \Phi(t) u_0\|_{H^{s + 1}_\vphi L^2_x} & \simeq 
\| \Phi(t) u_0 \|_{L^2_\vphi L^2_x} + 
\sup_{m=1, \ldots, \nu} \| \partial_{\vphi_m} (\Phi(t) u_0) \|_{H^s_\vphi L^2_x} \nonumber\\
&  
\leq_s  \|u_0 \|_{L^2_\vphi L^2_x} + \sup_{m=1, \ldots, \nu} \big( \| \Phi(t) [ \partial_{\vphi_m} u_0] \|_{s} +  
\| \partial_{\vphi_m} \Phi(t) u_0 \|_{s}\big) \label{albertone1bis} 
\\
&  \leq_s  \| u_0\|_{ s+ 1} + \| \partial_{\vphi_m} \Phi(t) u_0 \|_{s}\,. \label{mario 2}
\end{align}
For estimating the last term in \eqref{mario 2} note that, differentiating \eqref{flow-propagator}, the operator $\partial_{\vphi_m} \Phi(t) $ solves 
$$
\partial_t (\partial_{\vphi_m} \Phi(t)) = \ii a |D|^{\frac12} (\partial_{\vphi_m} \Phi(t)) + 
\ii (\partial_{\vphi_m} a) |D|^{\frac12} \Phi(t)\,, \quad \partial_{\vphi_m} \Phi(0) = 0\,,
$$
and then, by Duhamel principle (variation of constants method),  it has the representation
\begin{equation}\label{Duhamel principle flusso}
\partial_{\vphi_m} \Phi(t) = \ii  \int_0^t \Phi(t - \tau) (\partial_{\vphi_m} a) |D|^{\frac12} \Phi(\tau)\, d \tau\, .
\end{equation}
By the inductive assumption \eqref{stima flusso norma bassa} up to $ s  \leq s_0 $, and \eqref{continuous embed}, we get
\begin{align}\label{mario 3} 
\| \Phi(t - \tau) (\partial_{\vphi_m} a) |D|^{\frac12} \Phi(\tau) [u_0] \|_s  & \leq_s \| (\partial_{\vphi_m} a) 
|D|^{\frac12} \Phi(\tau) [u_0] \|_s \\
& \leq \| a \|_{{\mathcal C}^{s + 1}} \| \Phi(\tau) u_0 \|_{s + \frac12} \nonumber \\
& \leq_s \| a \|_{2 s_0 + 1} {\rm sup}_{t \in [0, 1]} \| \Phi(t) u_0 \|_{s + 1}
\,. \nonumber 
\end{align}
Therefore  \eqref{split norma s + 1}-\eqref{mario 3} imply
$$ 
\| \Phi(t) u_0 \|_{s + 1} \leq C(s) \big( \| u_0 \|_{s + 1} + \| a \|_{2 s_0 + 1} 
{\rm sup}_{t \in [0, 1]} \| \Phi(t) u_0 \|_{s + 1} \big)
$$ 
and, for $ C(s) \| a \|_{2 s_0 + 1} \leq 1/2 $,  
we deduce \eqref{stima flusso norma bassa} for $s + 1$. After $ s_0 $-steps we prove 
\eqref{stima flusso norma bassa} at $ s_0 + 1 $. 
Then  a classical  interpolation result implies  that $ \Phi(t) $ 
satisfies  the estimate \eqref{stima flusso norma bassa} also for all $ s \in (0, s_0 + 1 ) $.  
\\[1mm]
{\sc Proof of \eqref{stima flusso norme unite}.}
We argue again by induction on $ s $. For $ s \in [s_0, s_0 + 1] $ 
the tame estimate \eqref{stima flusso norme unite} is trivially implied  by \eqref{stima flusso norma bassa}.  
Then we suppose that \eqref{stima flusso norme unite} 
holds up to $  s \geq s_0 $ and we prove it  at $ s + 1 $. 

We estimate $ \| \Phi(t) u_0\|_{s + 1} $ as in \eqref{split norma s + 1}-\eqref{albertone1bis}. Then
we estimate  the last terms in \eqref{albertone1bis} in a tame way. 
The inductive hyphothesis \eqref{stima flusso norme unite} and  Lemma \ref{interpolazione fine} (with 
$a_0 = 2 s_0  + \frac12 $, $ b_0 = s_0 $, $p =s - s_0$, $q = 1$) imply  
\begin{align}
\| \Phi(t) [\partial_{\vphi_m} u_0] \|_{s} 
  & \leq_s \| u_0 \|_{s + 1} + \| a \|_{s + s_0 + \frac12} \| u_0 \|_{s_0 + 1} \nonumber\\
 & \leq_s \| u_0 \|_{s + 1} + \| a\|_{s + s_0 + \frac32} \| u_0\|_{s_0} + \| a\|_{2 s_0 + \frac12} 
\| u_0\|_{s + 1} \nonumber\\
& \leq_s \| u_0 \|_{s + 1} + \| a \|_{s + s_0 + \frac32} \| u_0 \|_{s_0}  \label{albertone 2}
\end{align}
since  $ \| a\|_{2 s_0 + \frac12}  \leq 1 $. Then we estimate
 $ \| \partial_{\vphi_m} \Phi(t) u_0 \|_s $. 
By the inductive assumption \eqref{stima flusso norme unite}, the 
tame estimates for the product of functions, \eqref{ansatz derivate vphi flusso} and
\eqref{stima flusso norma bassa}, we get, for all $t, \tau \in [0, 1]$,
\begin{align}
\|\Phi(t - \tau) (\partial_{\vphi_m} a) |D|^{\frac12} \Phi(\tau) [u_0] \|_s 
& \leq_s \| (\partial_{\vphi_m} a) |D|^{\frac12} \Phi(\tau) [u_0] \|_s \nonumber\\
& \quad + \| a \|_{s + s_0 + \frac12} 
\| (\partial_{\vphi_m} a) |D|^{\frac12} \Phi(\tau) [u_0] \|_{s_0} \nonumber\\
& \leq_s  \| a \|_{s + s_0 + \frac12}  \| u_0 \|_{s_0 + \frac12} 
+ \| a \|_{s_0 + 1} \| \Phi(\tau) u_0 \|_{s + \frac12} \,.  \label{albertone 3} 
\end{align}
Then \eqref{split norma s + 1}, \eqref{stima tame Phi t-1}, \eqref{albertone1bis}, 
\eqref{Duhamel principle flusso}, \eqref{albertone 2}, \eqref{albertone 3} imply
\begin{align}\nonumber
{\rm sup}_{t \in [0, 1]} \| \Phi(t) u_0 \|_{s + 1} & 
\leq_s \| u_0 \|_{s + 1} + \| a \|_{s +  s_0 + \frac32}  \| u_0 \|_{s_0} 
\\ 
& \quad + \| a \|_{s_0 + 1} {\rm sup}_{\tau \in [0, 1]}\| \Phi(\tau) u_0 \|_{s +1}  +  \| a \|_{s + s_0 + \frac12}  \| u_0 \|_{s_0 + 1}  \, . \nonumber 
\end{align}
Then, using \eqref{ansatz derivate vphi flusso} and 
Lemma \ref{interpolazione fine}  (with $a_0 = 2 s_0  + \frac12 $, $b_0 = s_0$, $p =s - s_0$, $q = 1$), we get 
\begin{align} 
{\rm sup}_{t \in [0, 1]} \| \Phi(t) u_0 \|_{s + 1}  & \leq_s \| u_0 \|_{s + 1} + \| a \|_{s + s_0 
+ \frac32} \| u_0 \|_{s_0} 
+ \| a \|_{s + s_0 + \frac12}  \| u_0 \|_{s_0 + 1} \nonumber \\
& \leq_s  \| u_0 \|_{s + 1} + \| a \|_{s + s_0 + \frac32} \| u_0 \|_{s_0}  \nonumber 
\end{align}
which is  \eqref{stima flusso norme unite} for $s + 1$. 

We have proved  \eqref{stima flusso norma bassa}, \eqref{stima flusso norme unite}, for $u_0 \in {\mathcal C}^\infty(\T^{\nu + 1})$. The estimates for $u_0 \in H^s $ follow by density. 
\end{proof}
We also prove the following tame estimates. 

\begin{lemma}
For all $n \geq 1$, if 
$ \| a \|_{s_0 + \frac{n}{2} + 2} \leq \d(s) $ small, then the following tame estimates hold: $ \forall s \geq s_0 $
\begin{equation} \label{prima proprieta flusso}
\begin{aligned}
& \| \langle D \rangle^{- \frac{n}{ 2}} \Phi(t) \langle D \rangle^{\frac{n}{2}} h\|_s\,,
\| \langle D \rangle^{ \frac{n}{ 2}} \Phi(t) \langle D \rangle^{- \frac{n}{2}} h\|_s \,  \\
& \leq_s \| h \|_s + \| a \|_{s + s_0 +  \frac{n}{2} + 2}  
\| h \|_{s_0} \, . 
\end{aligned}
\end{equation}
\end{lemma}

\begin{proof}
Let $ \Phi_n(t) := \langle D \rangle^{- \frac{n}{2}} \Phi(t) \langle D \rangle^{\frac{n}{2}} $. We consider $h \in {\mathcal C}^\infty$ so that $\Phi_n(t) h \in {\mathcal C}^\infty$.

\noindent
We have $ \Phi_n(0) = {\rm Id} $ and 
\begin{align*}
\partial_t \Phi_n(t)  = 
 \langle D \rangle^{- \frac{n}{2}} \ii a |D|^{\frac12} \Phi(t) \langle D \rangle^{\frac{n}{2}}  
  = \ii a |D|^{\frac12} \Phi_n(t) + \ii [\langle D \rangle^{- \frac{n}{2}}, a |D|^{\frac12}] \langle D \rangle^{\frac{n}{2}} \Phi_n(t)\, .
\end{align*}
Therefore by Duhamel principle we get 
\begin{equation}\label{Phi n Psi n}
\begin{aligned}
& \Phi_n(t) = \Phi(t) + \Psi_n(t)\,, \\ 
& \Psi_n(t) := \int_0^t \Phi(t - \tau) A_n \Phi_n(\tau) \, d \tau \ \
{\rm where} \ \
A_n := \ii [\langle D \rangle^{- \frac{n}{2}}, a |D|^{\frac12}] \langle D \rangle^{\frac{n}{2}} \,.
\end{aligned}
\end{equation}
By Lemmata \ref{lemma composizione multiplier}, 
\ref{lemma tame norma commutatore}, and  
\eqref{Norm Fourier multiplier}, \eqref{norma a moltiplicazione}, we get the estimate
\begin{equation}\label{stima An}
\norma A_n \norma_{0,s,0} \leq_s \| a \|_{s + \frac{n}{2} + 2}\, .
\end{equation}
Then by \eqref{Phi n Psi n},  using \eqref{stima flusso norma bassa} (for $s = s_0$) and 
Lemma \ref{lemma: action Sobolev}, we get
$$
{\rm sup}_{t \in [0, 1]}\|\Phi_n(t) h \|_{s_0} \leq 
C \| h \|_{s_0} + C \| a \|_{s_0 + \frac{n}{2} + 2} \, {\rm sup}_{t \in [0, 1]} \| \Phi_n(t) h\|_{s_0} \, .
$$
For $ C \| a \|_{s_0 + \frac{n}{2} + 2} \leq 1/2  $, we deduce
$ {\rm sup}_{t \in [0, 1]} \|\Phi_n(t) h \|_{s_0} \leq C \| h \|_{s_0} $. Then
\eqref{stima flusso norme unite}, 
\eqref{stima An}  and Lemma \ref{lemma: action Sobolev},  imply, for all $s > s_0 $, 
\begin{align}
\| \Psi_n(t) h \|_s & \leq_s  
{\rm sup}_{t \in [0, 1]} \big( \| A_n \Phi_n(t) h \|_s + \| a \|_{s + s_0 + \frac12} \| h \|_{s_0} \big) \nonumber \\
& \leq_s   \| a \|_{s + s_0 
+ \frac{n}{2} + 2} \| h \|_{s_0} +  \| a \|_{s_0 + \frac{n}{2} + 2}  \| h\|_s \nonumber\\
& \quad + 
\| a \|_{s_0 + \frac{n}{2} + 2} \, {\rm sup}_{t \in [0, 1]} \| \Psi_n(t) h\|_s  \,  . 
\label{stimaPsin}
\end{align}
Hence, for $ \| a \|_{s_0 + \frac{n}{2} + 2} \leq \delta (s) $ small, we deduce 
the estimate \eqref{prima proprieta flusso} 
by  \eqref{Phi n Psi n}, \eqref{stima flusso norme unite}, \eqref{stimaPsin}.  

If $h \in H^s$, the estimate \eqref{prima proprieta flusso} follows by density.
\end{proof}

Now  we prove similar tame estimates for  the operator  $ \partial_\lambda^k \pa_{\vphi}^{\beta} \Phi $ 
when the vector field $ \ii a(\lambda, \vphi, x) |D|^{1/2} $ depends also on  $ \lambda $.
The operator  $ \partial_\lambda^k \pa_{\vphi}^{\beta} \Phi $ 
loses $ | D_x |^{\frac{|\b| + |k|}{2}} $ 
derivatives which  are compensated  by  applying $ \langle D \rangle^{- \frac{|\beta| + |k|}{2}} $.

\begin{proposition}\label{Teorema totale partial vphi beta k D beta k Phi}
Assume that 
\begin{equation}\label{piccolezza a partial vphi beta k D beta k}
\| a\|_{2 s_0 + \beta_0 + 1} \leq \delta(s)\,, \quad \| a \|_{2 s_0 + \frac52 + \beta_0 + k_0 }^{k_0, \gamma} \leq 1
\end{equation}
with $\delta(s) > 0$ small enough.
Then, for all $|k| \leq k_0$, $|\beta| \leq \beta_0$, the following tame estimates hold:
\begin{align}\label{copenaghen A omega}
& \| \partial_\lambda^k \partial_\vphi^\beta \Phi \langle D \rangle^{- \frac{|\beta| + |k|}{2}} h \|_s \leq_s \gamma^{- |k|} \| h \|_s \,, \quad \forall s \in [0, s_0 + 1] \, , \\
& \label{copenaghen B omega}
\| \partial_\lambda^k \partial_\vphi^\beta \Phi \langle D \rangle^{- \frac{|\beta| + |k|}{2}} h \|_s \leq_s \gamma^{- |k|} 
\big(\| h \|_s + \| a \|_{s + s_0  + |\beta| + |k| + 1}^{k_0, \gamma} \| h \|_{s_0} \big)\,,  \forall s \geq s_0\, , 
\end{align}
and
\begin{equation}\label{copenaghen A omega con D}
\begin{aligned}
& \| \langle D \rangle \partial_\lambda^k \partial_\vphi^\beta \Phi \langle D \rangle^{- \frac{|\beta| + |k|}{2} - 1} h \|_s \leq_s \gamma^{- |k|} \| h \|_s \,, \qquad \forall s \in [0, s_0 + 1] \, ,  
\end{aligned}
\end{equation}
\begin{equation}\label{copenaghen B omega con D}
\begin{aligned}
& \| \langle D \rangle \partial_\lambda^k \partial_\vphi^\beta \Phi \langle D \rangle^{- \frac{|\beta| + |k|}{2} - 1} h \|_s  \\
& \leq_s \gamma^{- |k|} 
\big(\| h \|_s   +  \| a \|_{s + s_0  + |\beta| + |k| + 2}^{k_0, \gamma} \| h \|_{s_0} \big)\,, \  \forall s \geq s_0\, .
\end{aligned}
\end{equation}
\end{proposition}

\noindent
We prove Proposition \ref{Teorema totale partial vphi beta k D beta k Phi} by induction. We introduce the following notation 
\begin{itemize}
\item{\bf Notation} : Given $k_1, k \in \N^{\nu + 1}$, we say that $ k_1 \prec k $ if each component
$ k_{1,m} \leq k_m $,  $ \forall m = 1, \ldots, \nu + 1 $, and there exists $\bar m \in \{ 1, \ldots, \nu + 1 \}$ such that 
$ k_{1, \bar m} \neq k_{\bar m} $. Given $ (k_1, \b_1), (k, \b) \in \N^{\nu + 1} \times \N^\nu $
we say that $ (k_1, \b_1) \prec (k, \b) $ if each component
$ k_{1,m} \leq k_m $, $ \b_{1, n} \leq \b_n $,  $ \forall m = 1, \ldots, \nu + 1 $, $\forall n = 1, \ldots, \nu$ and $ (k_1, \b_1) \neq (k, \b) $. 
\end{itemize}
We first estimate $\| \partial_\lambda^k\partial_\vphi^\beta \Phi \langle D \rangle^{-\frac{ |\beta| + |k|}{2}} h \|_{L^2_\vphi H^s_x} $. 

\begin{lemma}
Assume \eqref{piccolezza a partial vphi beta k D beta k}. 
Then, for all $\vphi \in \T^\nu $, $ |k| \leq k_0$, $|\beta| \leq \beta_0$,  
\begin{equation}\label{partial beta partial k D Phi norme basse}
\begin{aligned}
& \|\partial_\lambda^k \partial_\vphi^\beta \Phi (\vphi) \langle D \rangle^{- \frac{|\beta| + |k|}{2}} h \|_{H^s_x} 
\leq_s \gamma^{- |k|} \| h \|_{H^s_x}\,, 
 \quad \forall s \in [0, 1] \, ,  
 \end{aligned}
 \end{equation}
 \begin{equation}\label{partial beta partial k D Phi norme alte}
 \begin{aligned}
& \| \partial_\lambda^k \partial_\vphi^\beta \Phi(\vphi) 
\langle D \rangle^{- \frac{|\beta|+ |k|}{2}} h \|_{H^s_x} \\
&  \leq_s \gamma^{- |k|} 
\big(\| h \|_{H^s_x} + \| a\|_{s + s_0 + |\beta| + \frac{|k|}{2} + \frac12 }^{k_0, \gamma} \| h \|_{H^1_x} \big)\,, \ \forall s \geq 1\, .  
\end{aligned}
\end{equation}
\end{lemma}

\begin{proof}
We take $h \in {\mathcal C}^\infty$, so that $ \partial_\lambda^k \partial_\vphi^\beta \Phi (\vphi) \langle D \rangle^{- \frac{|\beta| + |k|}{2}} h$ is ${\mathcal C}^\infty$. 

\noindent
We argue by induction on $ (k, \b ) $. For $  k = \b =  0 $ 
the estimates \eqref{partial beta partial k D Phi norme basse}-\eqref{partial beta partial k D Phi norme alte}  
are proved by \eqref{stima flusso PDE s 0 1}-\eqref{stima tame Phi t}.
Then supposing that 
\eqref{partial beta partial k D Phi norme basse}-\eqref{partial beta partial k D Phi norme alte} 
hold for all  $(k_1, \b_1) \prec (k, \b) $, $ |k| \leq k_0 $, $ |\beta| \leq \beta_0$,  
 we prove them for $ \partial_\lambda^{k} \partial_\vphi^\beta \Phi \langle D \rangle^{- \frac{|\beta| + |k|}{2}}$. 
 Differentiating  \eqref{flow-propagator} and using the Duhamel principle we get 
\begin{equation}\label{Duhamel prop 9.6}
\partial_\lambda^k \partial_\vphi^{\beta} \Phi(t) = \int_0^t \Phi(t - \tau) F_{\beta, k}(\tau)\, d \tau 
\end{equation}
where 
\begin{align}\label{F beta k prop 9.6}
& F_{\beta, k} (\tau) :=  \sum_{\begin{subarray}{c}
k_1 + k_2 = k \\ \b_1 + \b_2 = \b \\  
(k_1, \b_1) \prec (k, \b) \end{subarray}} 
C(k_1, k_2, \b_1, \b_2) (\partial_\lambda^{k_2} \partial_\vphi^{\b_2} a) |D|^{\frac12} \partial_\lambda^{k_1} \partial_\vphi^{\beta_1} \Phi(\tau) \, .
\end{align}
We now  prove  \eqref{partial beta partial k D Phi norme alte}.
For all $ (k_1, \b_1) \prec (k, \b) $,  
$ k_1 + k_2 = k $,  $ \b_1 + \b_2 = \b $, 
for all $t, \tau \in [0, 1] $, using \eqref{stima tame Phi t}, tame estimates for the product, 
\eqref{piccolezza a partial vphi beta k D beta k},  
we deduce
\begin{align}
 \| \Phi(t - \tau) (\partial_\lambda^{k_2} \pa_\vphi^{\b_2} a) |D|^{\frac12}  \partial_\lambda^{k_1} \partial_\vphi^{\beta_1}  \Phi(\tau) \langle D \rangle^{- \frac{|\beta| + |k|}{2}}  h \|_{H^s_x}  & \nonumber \\
 \leq_s \| (\partial_\lambda^{k_2} \partial_\vphi^{\beta_2} a) |D|^{\frac12}   \partial_\lambda^{k_1} \partial_\vphi^{\beta_1} \Phi(\tau) \langle D \rangle^{- \frac{|\beta| + |k| }{2}}  h \|_{H^s_x} & \nonumber \\
 + \| a \|_{s + s_0 + \frac12} \| (\partial_\lambda^{k_2} \partial_\vphi^{\beta_2} a) |D|^{\frac12} \partial_\lambda^{k_1}  \partial_\vphi^{\beta_1}  \Phi(\tau) \langle D \rangle^{- \frac{|\beta| + |k|}{2}}  h \|_{H^1_x} & \nonumber\\
 \leq_s \gamma^{- |k_2|} \| a \|_{s + s_0 + |\b | + 1}^{k_0, \gamma} \|  \partial_\lambda^{k_1}\partial_\vphi^{\beta_1} \Phi(\tau) \langle D \rangle^{- \frac{|\beta| + |k|}{2}}  h \|_{H^{\frac{3}{2}}_x}  & \nonumber\\
 + \gamma^{- |k_2|}  \|  \partial_\lambda^{k_1} \partial_\vphi^{\beta_1} 
\Phi(\tau) \langle D \rangle^{- \frac{|\beta| + |k| }{2}}  h \|_{H^{s + \frac12}_x} \, . &  \label{intermedia-Dx}
\end{align}
Now, since $ (k_1, \b_1) \prec (k, \b) $,   
$$
\begin{aligned}
 \partial_\lambda^{k_1} \partial_\vphi^{\beta_1} 
\Phi(\tau) \langle D \rangle^{- \frac{|\beta| + |k| }{2}} & = \partial_\lambda^{k_1} \partial_\vphi^{\beta_1} 
\Phi(\tau) \langle D \rangle^{- \frac{|\beta_1| + |k_1| }{2}} \langle D \rangle^{- \frac{m}{2}}  \, , \\
&  \qquad m := |\b| - |\beta_1| + |k| - |k_1| \geq 1 \, , 
\end{aligned}
$$
 and, applying  the inductive estimates \eqref{partial beta partial k D Phi norme alte}  for 
 $ \partial_\lambda^{k_1} \partial_\vphi^{\beta_1} 
\Phi(\tau) \langle D \rangle^{- \frac{|\beta_1| + |k_1| }{2}} $, \eqref{piccolezza a partial vphi beta k D beta k}, we get 
$$
\eqref{intermedia-Dx}  \leq_s \gamma^{- |k|} \big(  \| h \|_{H^s_x} + \| a \|_{s + s_0 + \frac12 + |\beta|  + \frac{|k|}{2}}^{k_0, \gamma} \| h\|_{H^1_x}\big) 
$$
which, by \eqref{Duhamel prop 9.6}, \eqref{F beta k prop 9.6},  proves  \eqref{partial beta partial k D Phi norme alte} for $h$ which is ${\mathcal C}^\infty$. The estimate \eqref{partial beta partial k D Phi norme alte} for $h \in H^s$ follows by density. 
The estimates \eqref{partial beta partial k D Phi norme basse}  follow in the same way using \eqref{stima flusso PDE s 0 1}. 
\end{proof}

Then, integrating in $ \vphi $ we get the following corollary

\begin{lemma}
Assume \eqref{piccolezza a partial vphi beta k D beta k}. 
Then, for all $ \vphi \in \T^\nu $, $|k| \leq k_0$, $|\beta| \leq \beta_0 $,  we have 
\begin{equation}\label{partial beta partial omega D omega D vphi Phi norme basse L2vphi}
 \|\partial_\lambda^k \partial_\vphi^\beta \Phi (\vphi) \langle D \rangle^{- \frac{|\beta| + |k|}{2}} h \|_{L^2_\vphi H^s_x} \leq_s \gamma^{- |k|}
\| h \|_{L^2_\vphi H^s_x}\,, 
 \quad \forall 
s \in [0, 1] \, ,  
\end{equation}
\begin{equation}\label{stima partial vphi partial omega L2vphi Hs x}
\begin{aligned}
& \| \partial_\lambda^k \partial_\vphi^\beta \Phi(\vphi) \langle D \rangle^{- \frac{|\beta| + |k|}{2}} h \|_{L^2_\vphi H^s_x}  \\
& \leq_s \gamma^{- |k|}
\big(\| h \|_{L^2_\vphi H^s_x} + \| a\|_{s + s_0 + \frac12  + |\beta| + \frac{|k|}{2}}^{k_0, \gamma} \| h \|_{L^2_\vphi H^1_x} \big)\,, \quad \forall s \geq 1\, ,   
\end{aligned}
\end{equation}
and 
\begin{equation}\label{partial beta partial k D Phi norme basse + D} 
 \| \langle D \rangle \partial_\lambda^k \partial_\vphi^\beta \Phi (\vphi) \langle D \rangle^{- \frac{|\beta| + |k| }{2} - 1} h \|_{L^2_\vphi H^s_x} \leq_s \gamma^{- |k|} \| h \|_{L^2_\vphi H^s_x}\,, 
 \quad \forall 
s \in [0, 1] \, , 
\end{equation}
\begin{equation}\label{partial beta partial k D Phi norme alte + D} 
\begin{aligned}
& \| \langle D \rangle \partial_\lambda^k \partial_\vphi^\beta \Phi(\vphi) \langle D \rangle^{- \frac{|\beta|+ |k| }{2} - 1} h \|_{L^2_\vphi H^s_x}  \\
& \leq_s \gamma^{- |k|} \big(\| h \|_{L^2_\vphi H^s_x} + \| a\|_{s + s_0 + |\beta| + \frac{|k|}{2} + \frac32 }^{k_0, \gamma} \| h \|_{L^2_\vphi H^1_x} \big)\,, \quad \forall s \geq 1\, . 
\end{aligned}
\end{equation}
\end{lemma}

\noindent
{\bf Proof of Proposition \ref{Teorema totale partial vphi beta k D beta k Phi}.}
Let $h \in {\mathcal C}^\infty$.
We argue by induction. For $ k= 0,  \b = 0 $ the estimates 
\eqref{copenaghen A omega}-\eqref{copenaghen B omega} follow by 
\eqref{stima flusso norma bassa}-\eqref{stima flusso norme unite}.  
We first argue by induction on $ k $ assuming 
that we have already proved \eqref{copenaghen A omega}-\eqref{copenaghen B omega} 
for all $ k_1 \prec k $,  $ |\b| \leq \beta_0$. Then we prove the tame estimates 
\eqref{copenaghen A omega}-\eqref{copenaghen B omega} 
for the operator $\partial_\lambda^k \partial_\vphi^\beta \Phi \langle D \rangle^{- \frac{|\beta| + |k|}{2}}$, for all $|\beta| \leq \beta_0$. To do this 
we argue by induction on $|\beta|$,
assuming 
\eqref{copenaghen A omega}-\eqref{copenaghen B omega} for all $|\beta| < n $  and we prove them for $|\beta| = n$
(also $ n = 0 $). To estimate $\| \partial_\lambda^k \partial_\vphi^\beta \Phi \langle D \rangle^{- \frac{|\beta| + |k|}{2}} h \|_s$ we argue by induction on $s$. 
\\[1mm]
{\sc Proof of  \eqref{copenaghen A omega} for $ |\b| = n $.} 
For $s = 0$, by \eqref{partial beta partial omega D omega D vphi Phi norme basse L2vphi}, we have 
 \begin{equation}\label{beta k Phi 0 sobolev}
 \begin{aligned}
 \| \partial_\lambda^k \partial_\vphi^\beta \Phi \langle D \rangle^{- \frac{|\beta| + |k|}{2}} h\|_0 
 & =  \| \partial_\lambda^k \partial_\vphi^\beta \Phi \langle D \rangle^{- \frac{|\beta| + |k|}{2}} h\|_{L^2_\vphi L^2_x} \\
 & \leq C \gamma^{- |k|} \| h \|_{L^2_\vphi L^2_x} = C \gamma^{- |k|} \| h \|_0\,.
 \end{aligned}
 \end{equation}
Now we suppose to have proved \eqref{copenaghen A omega} with
$ |\beta| = n $, up to the Sobolev index $ s < s_0 + 1$ and we prove it for $s + 1 \leq s_0 + 1$. We have 
 \begin{equation}\label{pallina A}
 \begin{aligned}
 \|  \partial_\lambda^k \partial_\vphi^\beta \Phi \langle D \rangle^{- \frac{|\beta| + |k|}{2}} h \|_{s + 1} & \simeq \|  \partial_\lambda^k 
 \partial_\vphi^\beta \Phi \langle D \rangle^{- \frac{|\beta| + |k|}{2}} h \|_{L^2_\vphi H^{s + 1}_x} \\ 
 & \quad + 
 \|   \partial_\lambda^k \partial_\vphi^\beta \Phi \langle D \rangle^{- \frac{|\beta| + |k|}{2}} h \|_{H^{s + 1}_\vphi 	L^2_x}\,.
 \end{aligned}
 \end{equation}
The first term in \eqref{pallina A} is estimated, using \eqref{stima partial vphi partial omega L2vphi Hs x},  
 $ s \leq s_0 $,  \eqref{piccolezza a partial vphi beta k D beta k}, by 
 $$
 \begin{aligned}
 \|  \partial_\lambda^k \partial_\vphi^\beta \Phi \langle D \rangle^{- \frac{|\beta| + |k|}{2}} h \|_{L^2_\vphi H^{s + 1}_x}  
 & \leq_s \gamma^{- |k|} \big(  \| h \|_{s + 1} + \| a \|^{k_0, \gamma}_{s + 1 + s_0 + \frac12 + |\beta| + \frac{|k|}{2}} \| h \|_1\big)
 \\
& \leq_s \gamma^{- |k|} \| h \|_{s + 1}\,.
 \end{aligned}
$$
 Now we estimate the second term in \eqref{pallina A}. By the inductive hyphothesis
 \begin{align}
  \|  \partial_\lambda^k \partial_\vphi^\beta \Phi \langle D \rangle^{- \frac{|\beta| + |k|}{2}} h \|_{H^{s + 1}_\vphi L^2_x} 
 & \simeq \|  \partial_\lambda^k\partial_\vphi^\beta  \Phi \langle D \rangle^{- \frac{|\beta| + |k|}{2}} h \|_{L^2_\vphi L^2_x}  \, + \nonumber \\
 & \quad +  \sup_{\alpha \in \N^\nu, |\alpha| = 1} \|  \partial_\lambda^k \partial_\vphi^{\beta} \Phi \langle D \rangle^{- \frac{|\beta| + |k|} {2}} [\partial_\vphi^\alpha h] \|_{H^s_\vphi L^2_x }  \nonumber\\
 & \quad + \sup_{\alpha \in \N^\nu, |\alpha| = 1} 
 \|  \partial_\lambda^k \partial_\vphi^{\beta + \alpha} \Phi \langle D \rangle^{- \frac{|\beta| + |k|}{2}} h \|_{H^s_\vphi L^2_x } \nonumber \\
 & \stackrel{\eqref{beta k Phi 0 sobolev}}{\lessdot} \gamma^{- |k|} \| h \|_0  \nonumber\\
 & \quad +\sup_{\alpha \in \N^\nu, |\alpha| = 1} 
 \|  \partial_\lambda^k \partial_\vphi^{\beta} \Phi \langle D \rangle^{- \frac{|\beta| + |k|}{2}} [\partial_\vphi^\alpha h] \|_{s} \nonumber\\ 
 & \quad +  \sup_{\alpha \in \N^\nu, |\alpha| = 1} \|  \partial_\lambda^k \partial_\vphi^{\beta + \alpha} \Phi \langle D \rangle^{- \frac{|\beta| + |k|}{2}} h \|_{s}  \label{penultima-Ap}\\
 & \leq_s \gamma^{- |k|} \| h \|_{s + 1} \nonumber\\
 & \quad + \sup_{\alpha \in \N^\nu, |\alpha| = 1} \|  \partial_\lambda^k \partial_\vphi^{\beta + \alpha} \Phi \langle D \rangle^{- \frac{|\beta| + |k|}{2}} h \|_{s}\,. \label{pallina B}
 \end{align}
Now, differentiating \eqref{flow-propagator} and using Duhamel principle,  we get 
 \begin{equation}\label{pallina Z1}
 \begin{aligned}
&  \partial_\lambda^k \partial_\vphi^{\beta + \alpha} \Phi(t) = \int_0^t \Phi(t - \tau) F_{\beta, k}(\tau)\, d \tau\,, \\
&  F_{\beta, k}(\tau) := F_{\beta, k}^{(1)}(\tau) + F_{\beta, k}^{(2)}(\tau) + F_{\beta, k}^{(3)}(\tau)\,,
\end{aligned}
\end{equation}
where 
\begin{align}
& F_{\beta, k}^{(1)}(\tau)  := 
\sum_{ \begin{subarray}{c}
\b_1+ \b_2 = \b + \a \\ k_1 + k _2 = k \\ k_1 \prec k
\end{subarray}} C(k_1, k_2, \beta_1, \beta_2) 
\pa_{\lambda}^{k_2} \pa_{\vphi}^{\b_2}  a  |D|^{1/2} \pa_{\lambda}^{k_1}  \pa_\vphi^{\b_1} \Phi(\tau) \nonumber \\
& F_{\beta, k}^{(2)}(\tau) := \sum_{\begin{subarray}{c}
\b_1+ \b_2 = \b + \a \\ |\b_1| \leq n-1
\end{subarray}} C(\beta_1, \beta_2) 
\pa_{\vphi}^{\b_2}  a  |D|^{1/2}  \pa_{\lambda}^{k} \pa_\vphi^{\b_1} \Phi(\tau) \nonumber \\
& F_{\beta, k}^{(3)}(\tau) :=  \sum_{\begin{subarray}{c}
\b_1+ \b_2 = \b + \a \\
 |\b_1|=n
\end{subarray}} C(\beta_1, \beta_2)
\pa_{\vphi}^{\b_2} a |D|^{1/2} \pa_{\lambda}^{k} \pa_\vphi^{\b_1}  \Phi(\tau) \,. \label{pallina E}
\end{align}
Note that if $ n = 0 $ the same formula applies, just without the second line.  Therefore
\be
\begin{aligned}
& \|  \pa_\lambda^k \partial_\vphi^{\beta + \alpha} \Phi \langle D \rangle^{- \frac{|\beta| + |k|}{2}} h \|_{s} \\
& \lessdot  \sup_{\begin{subarray}{c}
k_1 \prec k \\ k_1 + k_2 = k \\
\beta_1 + \beta_2 = \beta + \alpha \\
 \end{subarray}} \sup_{t, \tau \in [0, 1]} \|\Phi(t - \tau) ( \partial_\lambda^{k_2} \partial_\vphi^{\beta_2} a) |D|^{\frac12} 
 \partial_\lambda^{k_1} \partial_\vphi^{\beta_1} \Phi(\tau) \langle D \rangle^{- \frac{|\beta| + |k|}{2}} h \|_s  \\
&  + \sup_{\begin{subarray}{c}
\beta_1 + \beta_2 = \beta + \alpha \\
|\beta_1| \leq n - 1 
\end{subarray}} \sup_{t, \tau \in [0, 1]} \| \Phi(t - \tau) (\partial_\vphi^{\beta_2} a) |D|^{\frac12} 
\partial_\lambda^k \partial_\vphi^{\beta_1}  \Phi(\tau) \langle D \rangle^{- \frac{|\beta| + |k|}{2}} h  \|_s\, \\
& + \sup_{\begin{subarray}{c}
\beta_1 + \beta_2 = \beta + \alpha \\
|\beta_1| = n
\end{subarray}} \sup_{t, \tau \in [0, 1]} \| \Phi(t - \tau) (\partial_\vphi^{\beta_2} a) |D|^{\frac12}
\partial_\lambda^k  \partial_\vphi^{\beta_1}  \Phi(\tau) \langle D \rangle^{- \frac{|\beta| + |k|}{2}} h  \|_s  . \label{pallina G}
\end{aligned}
\ee
We estimate separately the three terms in the above inequality.
By the estimate \eqref{stima flusso norma bassa} for $\Phi$,
the inductive hyphothesis  for $ k_1 + k_2 = k $, $ k_1 \prec k $, $ \beta_1 + \beta_2 = \beta + \alpha $, 
$ t, \tau \in [0, 1] $,  and using \eqref{piccolezza a partial vphi beta k D beta k},  we get 
\begin{align}
\| \Phi(t - \tau) ( \partial_\lambda^{k_2} \partial_\vphi^{\beta_2} a) |D|^{\frac12} 
\partial_\lambda^{k_1}  \partial_\vphi^{\beta_1} \Phi(\tau) \langle D \rangle^{- \frac{|\beta| + |k|}{2}} h \|_s & \nonumber\\
\leq_s \| ( \partial_\lambda^{k_2} \partial_\vphi^{\beta_2} a) |D|^{\frac12}  \partial_\lambda^{k_1} \partial_\vphi^{\beta_1} 
\Phi(\tau) \langle D \rangle^{- \frac{|\beta| + |k|}{2}} h \|_s & \nonumber \\ 
 \leq_s  \gamma^{- |k_2| } \| a \|_{2 s_0 + |\beta| + 1}^{k_0, \gamma} \|  \partial_\lambda^{k_1} \partial_\vphi^{\beta_1} \Phi(\tau) \langle D \rangle^{- \frac{|\beta| + |k|}{2}} h \|_{s + 1} & \nonumber\\
 \leq_s \gamma^{- |k|} \| h \|_{s + 1}\,.& \label{pallina H}
\end{align}  
The  second term in \eqref{pallina G} is estimated as in \eqref{pallina H}. Then we 
consider  the last term in \eqref{pallina G}.
For $\beta_1 + \beta_2 = \beta + \alpha$, $|\beta_1| = n$, $ s \leq s_0 $, 
\begin{align}
\| \Phi(t - \tau) (\partial_\vphi^{\beta_2} a) |D|^{\frac12} \partial_\lambda^k \partial_\vphi^{\beta_1} \Phi(\tau) \langle D \rangle^{- \frac{|\beta| + |k|}{2}} h  \|_s  
 &  \nonumber\\
 \leq_s \| a \|_{2 s_0 + |\beta| + 1} \| \partial_\lambda^k \partial_\vphi^{\beta_1}  \Phi(\tau) \langle D \rangle^{- \frac{|\beta| + |k|}{2}} h \|_{s + 1}\,. &\label{pallina I}
\end{align}
By \eqref{pallina A}-\eqref{pallina I} we get  
\begin{align}
& \sup_{|\beta| = n} \sup_{t \in [0, 1]} \|  \pa_\lambda^k  \partial_\vphi^\beta \Phi(t) \langle D \rangle^{- \frac{|\beta|+|k|}{2}} h \|_{s + 1}  \nonumber\\
& \leq_s \gamma^{- |k|} \| h \|_{s + 1}  \nonumber\\
&  \quad + \| a \|_{2 s_0 + |\beta| + 1}  \sup_{|\beta| = n} \sup_{t \in [0, 1]} \|  \pa_\lambda^k 
\partial_\vphi^\beta \Phi(t) \langle D \rangle^{- \frac{|\beta|+|k|}{2}} h \|_{s + 1}   \nonumber
\end{align}
which implies  \eqref{copenaghen A omega} for $|\beta| = n$  at $ s+ 1 $, because 
$ \| a \|_{2 s_0 + |\beta| + 1} \leq \delta(s)$ is small enough (see \eqref{piccolezza a partial vphi beta k D beta k}).
\\[1mm]
\noindent
{\sc Proof of  \eqref{copenaghen B omega}  for $ |\b| = n $.}
The estimate \eqref{copenaghen B omega} 
for  $ s = s_0 $  follows  by \eqref{copenaghen A omega}. 
Then we assume to have proven \eqref{copenaghen B omega} with $ | \beta | = n $, 
up to the Sobolev index $ s $ and we prove it 
$ \|  \partial_\lambda^k \partial_\vphi^\beta \Phi \langle D \rangle^{- \frac{|\beta| + |k|}{2}} h \|_{s + 1} $.
The first term in \eqref{pallina A} is estimated, using
\eqref{stima partial vphi partial omega L2vphi Hs x}, by 
\begin{equation}\label{polizia - 1}
\|  \pa_\lambda^k \partial_\vphi^\beta \Phi \langle D \rangle^{- \frac{|\beta|+|k|}{2}} h  \|_{L^2_\vphi H^{s + 1}_x} \leq_s \gamma^{- |k|} 
\big( \| h \|_{s + 1} + \| a\|_{s + s_0 + 1  + |\beta| + |k| + 1 }^{k_0, \gamma} \| h \|_{1} \big)\, .
\end{equation}
Now we estimate the second term in \eqref{pallina A}. 
We have as in  \eqref{penultima-Ap} that
\begin{equation}\label{penultima-Ap1}
\begin{aligned}
  \|  \partial_\lambda^k \partial_\vphi^\beta \Phi \langle D \rangle^{- \frac{|\beta| + |k|}{2}} h \|_{H^{s + 1}_\vphi L^2_x} & \simeq  \gamma^{- |k|} \| h \|_0 \\
  & \quad +\sup_{\alpha \in \N^\nu, |\alpha| = 1} 
 \|  \partial_\lambda^k \partial_\vphi^{\beta} \Phi \langle D \rangle^{- \frac{|\beta| + |k|}{2}} [\partial_\vphi^\alpha h] \|_{s} \\
& \quad  +  \sup_{\alpha \in \N^\nu, |\alpha| = 1} \|  \partial_\lambda^k \partial_\vphi^{\beta + \alpha} \Phi \langle D \rangle^{- \frac{|\beta| + |k|}{2}} h \|_{s}   \, . 
 \end{aligned}
 \end{equation}
By the inductive hypothesis (on $ s $), we estimate the term in \eqref{penultima-Ap1} 
\begin{align}
\|  \pa_\lambda^k \partial_\vphi^\beta  \Phi \langle D \rangle^{- \frac{|\beta|+|k|}{2}} [\partial_\vphi^{\alpha} h] \|_s 
& \leq_s \gamma^{- |k|} 
\big( \| h \|_{s + 1} + \| a \|_{s + s_0 + 1 + |\beta| + |k| }^{k_0, \gamma} \| h \|_{s_0 + 1} \big) \nonumber \\
&  \leq_s \gamma^{- |k|} 
\big( \| h \|_{s + 1} + \| a \|_{s + s_0 + 1 + |\beta| + |k| + 1}^{k_0, \gamma} \| h \|_{s_0} \big) \label{copenaghen 3 omega}
\end{align}
using \eqref{piccolezza a partial vphi beta k D beta k} and  the  interpolation 
inequality \eqref{interpolation estremi fine Ck0}
with $a_0 = 2 s_0 + |\beta| + |k| +1 $, $ b_0 = s_0 $, $ p = s - s_0 $, $ q = 1$, $ \epsilon = 1 $.

Now we estimate the last term in \eqref{penultima-Ap1}. By \eqref{pallina Z1}-\eqref{pallina E} one has
\begin{align}
& \|  \pa_\lambda^k \partial_\vphi^{\beta + \alpha} \Phi \langle D \rangle^{- \frac{|\beta| + |k|}{2}} h \|_{s}  \nonumber\\
& \lessdot \sup_{\begin{subarray}{c}
k_1 \prec k, k_1 + k_2 = k,  \\
\beta_1 + \beta_2 = \beta + \alpha 
\end{subarray}} \sup_{t, \tau \in [0, 1]} \|\Phi(t - \tau) (\partial_\lambda^{k_2} \partial_\vphi^{\beta_2}  a)
 |D|^{\frac12} \partial_\lambda^{k_1}  \partial_\vphi^{\beta_1} \Phi(\tau) \langle D \rangle^{- \frac{|\beta| + |k|}{2}} h \|_s  \nonumber\\
&  + \sup_{\begin{subarray}{c}
\beta_1 + \beta_2 = \beta + \alpha \\
|\beta_1| \leq n -1 
\end{subarray}} \sup_{t, \tau \in [0, 1]} \| \Phi(t - \tau) (\partial_\vphi^{\beta_2} a) 
|D|^{\frac12}  \partial_\lambda^k \partial_\vphi^{\beta_1} \Phi(\tau) \langle D \rangle^{- \frac{|\beta| + |k|}{2}} h  \|_s\, \nonumber\\
&  + \sup_{\begin{subarray}{c}
\beta_1 + \beta_2 = \beta + \alpha \\
|\beta_1| = n
\end{subarray}} \sup_{t, \tau \in [0, 1]} \| \Phi(t - \tau) (\partial_\vphi^{\beta_2} a) |D|^{\frac12} 
 \partial_\lambda^k \partial_\vphi^{\beta_1} \Phi(\tau) \langle D \rangle^{- \frac{|\beta| + |k|}{2}} h  \|_s  . \label{polizia 0 omega}
\end{align}
Note that if $ n = 0 $ the same formula applies, just without the second line. 
We estimate separately the terms in \eqref{polizia 0 omega}.
By the estimate \eqref{stima flusso norme unite} on $ \Phi $, \eqref{copenaghen A omega},  
and the inductive hyphothesis  for $k_1 + k_2 = k$, $k_1 \prec k$, $\beta_1 + \beta_2 = \beta + \alpha$, $t, \tau \in [0, 1]$, 
we get  
\begin{align}
& \| \Phi(t - \tau) ( \partial_\lambda^{k_2} \partial_\vphi^{\beta_2} a) |D|^{\frac12}  \partial_\lambda^{k_1} \partial_\vphi^{\beta_1} \Phi(\tau) \langle D \rangle^{- \frac{|\beta| + |k|}{2}} h \|_s  \nonumber \\
& \leq_s \gamma^{- |k_2|} \| \partial_\lambda^{k_1} \partial_\vphi^{\beta_1}  \Phi(\tau) \langle D \rangle^{- \frac{|\beta| + |k|}{2}} h \|_{s + \frac12} 
\nonumber\\
& \quad + \gamma^{- |k_2|}  \| a \|_{s + s_0  + |\beta| + 1}^{k_0,\gamma} \|  \partial_\lambda^{k_1}\partial_\vphi^{\beta_1}  \Phi(\tau) \langle D \rangle^{- \frac{ |\beta| + |k| }{2}} h \|_{s_0 + \frac12} \nonumber\\
& \leq_s \gamma^{- |k|} \big( \| h \|_{s + 1} + \| a \|_{s + s_0 + 1 + |\beta| + |k|  + 1}^{k_0, \gamma} \| h \|_{s_0} \big)  \label{polizia 1 omega}
\end{align}
using \eqref{piccolezza a partial vphi beta k D beta k} and since  \eqref{interpolation estremi fine Ck0} 
with $ a_0 = 2 s_0 + |\beta| + 1 $, $ b_0 = s_0 $, $ p = s - s_0 $, $ q = 1 $, $\epsilon = 1 $,  implies 
\begin{equation}\label{municipio uffa}
\| a \|_{s + s_0 + |\beta| + 1}^{k_0, \gamma} \| h \|_{s_0 + 1} \leq  \| a \|_{2 s_0 + |\beta| + 1}^{k_0, \gamma} \| h \|_{s + 1} + \| a \|_{s + s_0 + |\beta| + 2}^{k_0, \gamma} \| h \|_{s_0} \, . 
\end{equation}
The second term in \eqref{polizia 0 omega} is estimated similarly by \eqref{polizia 1 omega}.
Then we consider the third term in \eqref{polizia 0 omega}.
For $\beta_1 + \beta_2 = \beta + \alpha$, $|\beta_1| = n$, by \eqref{stima flusso norme unite}, \eqref{copenaghen A omega}
\begin{align}
& \| \Phi(t - \tau) (\partial_\vphi^{\beta_2} a) |D|^{\frac12} \partial_\lambda^k  \partial_\vphi^{\beta_1} \Phi(\tau) \langle D \rangle^{- \frac{|\beta| + |k|}{2}} h  \|_s \nonumber \\
& 
\leq_s  \| a\|_{s + s_0 + |\beta| + 1} \| \partial_\lambda^k \partial_\vphi^{\beta_1}  \Phi(\tau) \langle D \rangle^{- \frac{|\beta| + |k|}{2}} h \|_{s_0 + 1}  \nonumber\\
& \quad + \| a \|_{2 s_0 + |\beta| + 1} \| \partial_\lambda^k \partial_\vphi^{\beta_1}  \Phi(\tau) \langle D \rangle^{- \frac{|\beta| + |k|}{2}} h \|_{s+ 1}  \nonumber\\
& \leq_s \gamma^{- |k|}\| a \|_{s + s_0 + |\beta| + 1} \| h \|_{s_0 + 1} \nonumber\\
& \quad + \| a \|_{2 s_0 + |\beta| + 1} 
\|  \partial_\lambda^k \partial_\vphi^{\beta_1} \Phi(\tau) \langle D \rangle^{- \frac{|\beta| + |k|}{2}} h \|_{s+ 1} \nonumber\\
& \stackrel{\eqref{municipio uffa}}{\leq_s} \gamma^{- |k|} (\| h \|_{s + 1} + \| a \|_{s + s_0 + |\beta| + 2}^{k_0, \gamma} \| h \|_{s_0}) \nonumber\\
& \quad +  \| a \|_{2 s_0 + |\beta| + 1} \| \partial_\lambda^k \partial_\vphi^{\beta_1}  \Phi(\tau) \langle D \rangle^{- \frac{|\beta| + |k|}{2}} h  \|_{s+ 1}\,. \label{polizia 2 omega}
\end{align}
By \eqref{polizia - 1}, \eqref{penultima-Ap1},  \eqref{copenaghen 3 omega}, \eqref{polizia 0 omega}, \eqref{polizia 1 omega}, \eqref{polizia 2 omega} we get  
\begin{align}
& \sup_{|\beta| = n} \sup_{t \in [0, 1]} \|  \pa_\lambda^k \partial_\vphi^\beta \Phi(t) \langle D \rangle^{- \frac{|\beta|+|k|}{2}} h \|_{s + 1} \nonumber\\
&  \leq_s \gamma^{- |k|} \big( \| h \|_{s + 1} + \| a\|_{s + s_0  + |\beta| + |k| + 2}^{k_0, \gamma} \| h \|_{s_0} \big) \nonumber\\
& \quad + \| a \|_{2 s_0 + |\beta| + 1} \sup_{|\beta| = n} \sup_{t \in [0, 1]} \|  \pa_\lambda^k \partial_\vphi^\beta
\Phi(t) \langle D \rangle^{- \frac{|\beta|+|k|}{2}} h \|_{s + 1} \nonumber
\end{align}
which implies  \eqref{copenaghen B omega} at $ s+ 1 $ 
for $|\beta| = n$, because 
 $\| a \|_{2 s_0 + |\beta| + 1} \leq \delta(s)$ is small enough (see \eqref{piccolezza a partial vphi beta k D beta k}). 

\smallskip

\noindent
{\sc Proof of  \eqref{copenaghen A omega con D}-\eqref{copenaghen B omega con D}.} 
We argue by induction on  $ s $.  
The estimate  \eqref{copenaghen A omega con D} for $ s = 0 $ is proved by 
\eqref{partial beta partial k D Phi norme basse + D} for $ s  = 0 $. 
Now let us suppose to have estimated the operator $ \langle D \rangle  \partial_\lambda^k \partial_\vphi^\beta \Phi \langle D \rangle^{- \frac{|\beta| + |k|}{2} - 1} $ up to the Sobolev index $s$ and let us prove it for $s + 1$. We have to estimate 
$$
\begin{aligned}
 \| \langle D \rangle \partial_\lambda^k \partial_\vphi^\beta \Phi \langle D \rangle^{- \frac{|\beta| + |k|}{2} - 1} h \|_{s + 1} 
 & \simeq \| \langle D \rangle  \partial_\lambda^k \partial_\vphi^\beta \Phi \langle D \rangle^{- \frac{|\beta| + |k|}{2} - 1} h \|_{L^2_\vphi H^{s + 1}_x} \\ 
 & \quad + \| \langle D \rangle  \partial_\lambda^k \partial_\vphi^\beta \Phi \langle D \rangle^{- \frac{|\beta| + |k|}{2} - 1} h \|_{H^{s + 1}_\vphi 	L^2_x}\,.
 \end{aligned}
$$
 The first term is estimated by   \eqref{partial beta partial k D Phi norme alte + D} as 
\be
 \begin{aligned}
 & \| \langle D \rangle  \partial_\lambda^k \partial_\vphi^\beta \Phi \langle D \rangle^{- \frac{|\beta| + |k|}{2} - 1} h \|_{L^2_\vphi H^{s + 1}_x} \\ 
 & \leq_s \gamma^{- |k|} 
 \big(  \| h \|_{s + 1} + \| a \|^{k_0, \gamma}_{s + 1 + s_0  + |\beta| + \frac{|k|}{2} + \frac32} \| h \|_1\big) \, ,  
 \end{aligned}
 \ee
and the second term,  
 using \eqref{partial beta partial k D Phi norme basse + D}, as
 \begin{align}
&  \| \langle D \rangle  \partial_\lambda^k \partial_\vphi^\beta \Phi \langle D \rangle^{- \frac{|\beta| + |k|}{2} - 1} h \|_{H^{s + 1}_\vphi L^2_x} \nonumber\\ & \simeq \| \langle D \rangle   \partial_\lambda^k \partial_\vphi^\beta \Phi \langle D \rangle^{- \frac{|\beta| + |k|}{2} - 1} h \|_{L^2_\vphi L^2_x}
  \nonumber\\ 
& \quad + \sup_{\alpha \in \N^\nu, |\alpha| = 1} \|  \langle D \rangle  \partial_\lambda^k \partial_\vphi^{\beta + \alpha} \Phi \langle D \rangle^{- \frac{|\beta| + |k|}{2} - 1} h \|_{H^s_\vphi L^2_x } \nonumber\\
& \quad + \sup_{\alpha \in \N^\nu, |\alpha| = 1} \| \langle D \rangle  \partial_\lambda^k \partial_\vphi^\beta \Phi \langle D \rangle^{- \frac{|\beta| + |k|}{2} - 1} \partial_\vphi^\alpha h \|_{s} \nonumber\\
 & \lessdot \gamma^{- |k|} \| h \|_{0} + \sup_{\alpha \in \N^\nu, |\alpha| = 1} \| \langle D \rangle 
 \partial_\lambda^k  \partial_\vphi^{\beta + \alpha}  \Phi \langle D \rangle^{- \frac{|\beta| + |k|}{2} - 1} h \|_{s} \nonumber\\
 & + \sup_{\alpha \in \N^\nu, |\alpha| = 1} \| \langle D \rangle  \partial_\lambda^k \partial_\vphi^\beta \Phi \langle D \rangle^{- \frac{|\beta| + |k|}{2} - 1} \partial_\vphi^\alpha h \|_{s}\, .  \label{pallina B (1)}
 \end{align}
 By the inductive hyphothesis, for all $ \alpha \in \N^\nu $, $ | \alpha | = 1 $,
 \be
 \| \langle D \rangle  \partial_\lambda^k 
 \partial_\vphi^\beta \Phi \langle D \rangle^{- \frac{|\beta| + |k|}{2} - 1} 
 \partial_\vphi^\alpha h \|_{s}  \leq_s \gamma^{- |k|} \| h \|_{s + 1}\,, \quad  \forall s \leq s_0  
\ee
and 
\be
\begin{aligned}
&  \| \langle D \rangle  \partial_\lambda^k \partial_\vphi^\beta \Phi \langle D \rangle^{- \frac{|\beta| + |k|}{2} - 1} \partial_\vphi^\alpha h \|_{s}  \\ 
& \leq_s   
\gamma^{- |k|} \big(  \| h \|_{s + 1} 
   +  \| a \|^{k_0, \gamma}_{s + s_0 + |\beta| + |k| + 2} \| h \|_{s_0 + 1}\big), \quad  \forall s > s_0
  \\
& \leq_s \gamma^{- |k|} \big( \| h \|_{s + 1} + \| a \|^{k_0, \gamma}_{s + 1 + s_0 + |\beta| + 1 + |k| + 1} \| h \|_{s_0} \big)\,,  
 \end{aligned}
 \ee
since \eqref{interpolation estremi fine Ck0} with $a_0 = 2 s_0 + |\beta| + |k| + 2$, $b_0 = s_0$, $p = s - s_0$, $q = 1$,
$ \epsilon = 1 $,  
and  \eqref{piccolezza a partial vphi beta k D beta k} 
imply 
$$
\| a \|_{s + s_0 + |\beta| + |k| + 2}^{k_0, \gamma} \| h \|_{s_0 + 1} \leq   \| h \|_{s + 1} + \| a \|_{s + 1 + s_0 + |\beta| + |k| + 2}^{k_0, \gamma} \| h \|_{s_0} \, .
$$
Finally 
$$
\begin{aligned}
 \|\langle D \rangle  \partial_\lambda^k \partial_\vphi^{\beta + \alpha} \Phi \langle D \rangle^{- \frac{|\beta| + |k|}{2} - 1} h \|_{s} 
&  \leq \| \partial_\lambda^k \partial_\vphi^{\beta + \alpha}  \Phi \langle D \rangle^{- \frac{|\beta| + |k|}{2} - 1} h \|_{s + 1} \\
&  = \|  \partial_\lambda^k \partial_\vphi^{\beta + \alpha} \Phi \langle D \rangle^{- \frac{|\beta| + |k| + 1}{2} }
  [ \langle D \rangle^{- \frac12} h] \|_{s + 1}
  \end{aligned}
$$
and   \eqref{copenaghen A omega}-\eqref{copenaghen B omega} imply
 \begin{align*}
 \|  \partial_\lambda^k \partial_\vphi^{\beta + \alpha} \Phi \langle D \rangle^{- \frac{|\beta| + |k| + 1}{2} }[ \langle D \rangle^{- \frac12} h] \|_{s + 1} & \leq_s \gamma^{- |k|} \| h \|_{s + 1}\,, \quad \forall s \leq s_0 \, , \\
 \|  \partial_\lambda^k \partial_\vphi^{\beta + \alpha} \Phi \langle D \rangle^{- \frac{|\beta| + |k| + 1}{2} }[ \langle D \rangle^{- \frac12} h] \|_{s + 1} & \leq_s \gamma^{- |k|} \big( \| h \|_{s + 1} \\
 & + \| a \|^{k_0, \gamma}_{s + 1 + s_0 + |\beta| + 1 + |k| + 1} \| h \|_{s_0} \big)\, , \  \forall s \geq s_0 \, .
 \end{align*}
 Collecting all the above estimates we have proved 
  \eqref{copenaghen A omega con D}-\eqref{copenaghen B omega con D} with Sobolev index $s + 1$. 
  
  We have then proved the estimates \eqref{copenaghen A omega}-\eqref{copenaghen B omega con D} for $h \in {\mathcal C}^\infty$. If $h \in H^s$ they follow by density.
  The proof of Proposition \ref{Teorema totale partial vphi beta k D beta k Phi} is completed.
\hfill $\Box$

\begin{proposition} \label{lemma:tame derivate flusso}
For  $ \beta_0 \in \N $ assume that
 \begin{equation}\label{ansatz a Lemma tecnico flusso secondo}
 \| a \|_{2 s_0 + \frac{\beta_0 +k_0}{2} + 3} \leq \delta (s) \,, \quad  \| a \|_{2s_0  + 3 + \frac32 \beta_0 + \frac{k_0}{2}}^{k_0, \gamma} \leq 1\,,
  \end{equation}
   for $ \d (s) > 0$ small. Then, for all 
   $ \b \in \N^\nu,  k \in \N^{\nu + 1} $ with $ | \b | \leq \b_0 $, $|k| \leq k_0$, $ s \geq s_0 $,   we have 
\begin{equation}\label{D - 1/2 D vphi flusso N}
\begin{aligned}
& 
{\rm sup}_{t \in [0, 1]} \| \langle D \rangle^{- \frac{|\beta| + |k|}{2}} 
\partial_\lambda^k \partial_\vphi^\beta  \Phi(\vphi, t) h  \|_s  \\
& \leq_s \gamma^{- |k|} \big( \| h\|_s + \| a \|_{s + s_0 + 2 +\frac32 |\beta| + \frac12 |k|  }^{k_0, \gamma} \| h\|_{s_0} \big)\, ,
\end{aligned}
\end{equation}
\begin{equation}\label{D - 1/2 D vphi flusso N + Dx}
\begin{aligned}
& 
{\rm sup}_{t \in [0, 1]} \| \langle D \rangle^{- \frac{|\beta| + |k|}{2}-1} 
\partial_\lambda^k \partial_\vphi^\beta  \Phi(\vphi, t) \langle D \rangle  h \|_s  \\
& \leq_s \gamma^{- |k|} \big( \| h\|_s + \| a \|_{s + s_0 + 3 +\frac32 |\beta| + \frac12 |k|  }^{k_0, \gamma} \| h\|_{s_0} \big)\, .
  \end{aligned}
  \end{equation}
\end{proposition}

\begin{proof}
We prove only \eqref{D - 1/2 D vphi flusso N + Dx}.  The proof of  \eqref{D - 1/2 D vphi flusso N} is the same (easier).
We take $h \in {\mathcal C}^\infty$ and we argue by induction on $ (k, \b ) $.
For $  k = 0\,, \b =  0 $ the estimate \eqref{D - 1/2 D vphi flusso N + Dx}  is  proved by \eqref{prima proprieta flusso} with $ n = 2 $.
Then supposing that \eqref{D - 1/2 D vphi flusso N + Dx}  
holds for all  $(k_1, \b_1) \prec (k, \b) $, $ |k| \leq k_0 $, $ |\beta| \leq \beta_0$,  
we prove it for $ \langle D \rangle^{- \frac{|\beta| + |k|}{2}-1} \partial_\lambda^{k} \partial_\vphi^\beta \Phi \langle D \rangle  $
for which we use  the integral representation  \eqref{Duhamel prop 9.6}-\eqref{F beta k prop 9.6}. 
For all $\beta_1 + \beta_2 = \beta $, $k_1 + k_2 = k$, $ (k_1, \b_1 ) \prec ( k, \b ) $, $t, \tau \in [0, 1]$, one has 
\begin{align}
& \langle D \rangle^{- \frac{|\beta|  + |k|}{2} -1 }\Phi(t - \tau) 
( \partial_\lambda^{k_2} \partial_\vphi^{\beta_2} a) |D|^{\frac12}  \partial_\lambda^{k_1} \partial_\vphi^{\beta_1}  \Phi(\tau) 
\langle D \rangle  \label{una-per-tutte-bis} \\
& = 
 \langle D \rangle^{- \frac{|\beta|  + |k|}{2} -1 }\Phi(t - \tau) \langle D \rangle^{\frac{|\beta|  + |k|}{2} + 1} \nonumber\\
& \quad \, \langle D \rangle^{- \frac{|\beta| + |k|}{2} -1}  (\partial_\lambda^{k_2} \partial_{\vphi}^{\beta_2} a ) 
\langle D \rangle^{\frac{|\beta|  + |k|}{2}+1} \nonumber \\
& \quad \, |D|^{\frac12} \langle D \rangle^{- \frac{m}{2}} \langle D \rangle^{- \frac{|\beta_1| + |k_1|}{2} -1 } 
\partial_\lambda^{k_1}  \partial_\vphi^{\beta_1}  \Phi(\tau)  \langle D \rangle  \nonumber
\end{align}
where $ m := |\b|-|\b_1|+|k|- |k_1| \geq 1 $. 
These three terms satisfy tame estimates. 
By \eqref{prima proprieta flusso} (which can be applied because of \eqref{ansatz a Lemma tecnico flusso secondo}) we have 
\begin{align}
\| \langle D \rangle^{- \frac{|\beta|  + |k|}{2}-1}\Phi(t - \tau) \langle D \rangle^{\frac{|\beta|  + |k|}{2}+1} h \|_s 
& \leq_s \| h \|_s  + \| a \|_{s+ s_0  + 2 + \frac{|\beta|  + |k|}{2}} \| h\|_{s_0}\, . \label{riccardone 2-bis}
\end{align}
Lemma \ref{lemma stime Ck parametri}, \ref{lemma composizione multiplier}, 
and  \eqref{norma a moltiplicazione}, \eqref{Norm Fourier multiplier}, imply
\begin{align}
 \norma \langle D \rangle^{- \frac{|\beta |  + |k|}{2}-1} 
\partial_\lambda^{k_2} \partial_\vphi^{\beta_2} a \langle D \rangle^{\frac{|\beta|  + |k|}{2} + 1} 
\norma_{0,s,0} & \leq_s \! \|  \partial_\lambda^{k_2} \partial_\vphi^{\beta_2} a \|_{s + \frac{|\beta |  + |k|}{2}} \nonumber\\
&  \leq_s \gamma^{- |k_2|} \| a \|^{k_0, \gamma}_{s + \frac{3}{2}|\beta| + \frac{|k|}{2}}\,. \label{riccardone 1 secondo-bis}
\end{align}
Since $ (k_1, \b_1) \prec (k, \b) $, using the inductive estimates \eqref{D - 1/2 D vphi flusso N} for 
$$ \langle D \rangle^{- \frac{|\beta_1| + |k_1|}{2}-1} \partial_\lambda^{k_1} \partial_\vphi^{\beta_1}  \Phi(\tau) \langle D \rangle \, ,  
$$ 
we get  
\begin{align}
& \|  
|D|^{\frac12} \langle D \rangle^{- \frac{m}{2}} \langle D \rangle^{- \frac{|\beta_1| + |k_1|}{2} -1} 
\partial_\lambda^{k_1}  \partial_\vphi^{\beta_1}  \Phi(\tau)  \langle D \rangle h \|_s \nonumber\\
& \leq_s  \| \langle D \rangle^{- \frac{|\beta_1| + |k_1|}{2}-1} 
\partial_\lambda^{k_1} \partial_\vphi^{\beta_1}  \Phi(\tau)  \langle D \rangle h \|_s   \nonumber \\
& \leq_s \gamma^{- |k_1|} 
\big( \| h \|_s \nonumber\\
& \quad + \| a\|^{k_0, \gamma}_{s + s_0 + 2 + \frac{3}{2}|\beta| + \frac{|k|}{2}} \| h \|_{s_0} \big)\,.   \label{riccardone 3 secondo-bis}
\end{align}
In conclusion, \eqref{una-per-tutte-bis}-\eqref{riccardone 3 secondo-bis} 
imply  \eqref{D - 1/2 D vphi flusso N + Dx}. If $h \in H^s$ the estimate \eqref{D - 1/2 D vphi flusso N + Dx} follows by density.
\end{proof}

As a corollary we get  
\begin{proposition}\label{flussoCk0}
Assume \eqref{piccolezza a partial vphi beta k D beta k}. 
Then the flow $\Phi(t, \lambda)$ of 
\eqref{pseudo PDE} is ${\mathcal D}^{k_0}$-{$\frac{ k_0}{2}$}-tame (Definition \ref{def:Ck0}), more precisely,  
for all $ k \in \N^{\nu + 1} $, $ |k| \leq k_0 $,  $ s \geq s_0 $, 
\begin{equation}\label{stima-rozzina}
\begin{aligned}
 & \sup_{t \in [0, 1]}\|  \partial_\lambda^{k} \Phi(\vphi, t) h \|_s  \\
 &  \leq_s  \gamma^{- |k|} 
\big( \| h \|_{s+ \frac{ |k|}{2}} + \| a \|_{s + s_0   +  |k| + 1}^{k_0, \gamma} 
\|  h \|_{s_0 + \frac{ |k|}{2} }  \big)  \, , \\
\end{aligned}
\end{equation}
\begin{equation}\label{stima flusso - identita}
\begin{aligned}
& \sup_{t \in [0, 1]} \| \partial_\lambda^k (\Phi(t )- {\rm I d}) h \|_s  \\
& \leq_s \gamma^{- |k|}\big( \| a \|_{s_0}^{k_0, \gamma} \|  h\|_{s + \frac{|k| + 1}{2}} + \| a \|_{s + s_0 + k_0 + \frac32}^{k_0, \gamma} \| h \|_{s_0 + \frac{|k| + 1}{2}}  \big) \, .  
\end{aligned}
\end{equation}
\end{proposition}

\begin{proof}
By \eqref{copenaghen B omega} (with $ \b = 0 $) we have 
\begin{align*}
\|  \partial_\lambda^{k} \Phi(\vphi, t) h \|_s 
 & =  
\|  \partial_\lambda^{k} \Phi(\vphi, t) \langle D \rangle^{- \frac{ |k|}{2}}  \langle D \rangle^{\frac{ |k|}{2}} h \|_s \\ 
&  \leq_s \gamma^{- |k|} 
\big( \| \langle D \rangle^{\frac{ |k|}{2}} h \|_s + \| a \|_{s + s_0 + |k| + 1 }^{k_0, \gamma} 
\| \langle D \rangle^{\frac{ |k|}{2}} h \|_{s_0 }  \big) \nonumber \\ 
& \leq_s  \gamma^{- |k|} 
\big( \| h \|_{s+ \frac{ |k|}{2}} + \| a \|_{s + s_0   +  |k| + 1}^{k_0, \gamma} 
\|  h \|_{s_0 + \frac{ |k|}{2} }  \big) 
\end{align*}
which proves \eqref{stima-rozzina}.  

\noindent
{\sc Proof of \eqref{stima flusso - identita}.}
By \eqref{pseudo PDE}, i.e. \eqref{flow-propagator}, we write 
$ \Phi(t) - {\rm Id} = \int_0^t \ii a |D|^{\frac12} \Phi(\tau )\, d \tau  $. Then 
 \eqref{stima flusso - identita} for $ k = 0 $  follows by \eqref{interpolazione C k0} and  \eqref{stima flusso norme unite}. 
For $|k| > 0$, \eqref{stima flusso - identita} follows by interpolation 
and  using \eqref{stima-rozzina}. 
\end{proof}

Finally we consider also the dependence of the flow $\Phi$  
with respect to the torus $ i := i(\vphi ) := (\vphi, 0, 0) + \fracchi (\vphi ) $ (recall the notation \eqref{componente periodica}). 
Assuming that  there exists $\sigma > 0$ such that for any $s\geq 0$, the map 
$$
\begin{aligned}
& {\fracchi}(\lambda) \in {\mathcal Y}^{s + \sigma}    \mapsto a(\lambda, i(\lambda)) \in H^s\,, \\
&  {\mathcal Y}^s := H^{s }(\T^\nu, \R^\nu) \times  H^{s }(\T^\nu, \R^\nu) \times \big(H^{s}(\T^{\nu + 1}, \R^2) \cap H_{{\mathbb S}^+}^\bot \big)
\end{aligned}
$$
is differentiable, then,  by Lemma \ref{dipendenza liscia dai parametri},  the flow $ \Phi ( t ) $ is differentiable with respect to $ i $. 
Note that in the lemma below 
we do not estimate the derivatives of $  \partial_i \Phi(t)$ with respect to 
$ \lambda $ 
since it is not required, see remark \ref{remark:dipendenza da omega non richiesta}.  We state an analogous version of Lemma 
\ref{lemma finitezza norme stime flusso} (the proof is similar) which takes into account the dependence with respect to the torus $i$. 
\begin{lemma}\label{lemma finitezza norme stime flusso derivate i}
For any $|\beta| \leq \beta_0$, $h$, $i$, $\widehat \imath$ which are ${\mathcal C}^\infty (\T^{\nu+1}) $, the function 
$  \partial_\vphi^\beta \partial_i \Phi^t(i)[\widehat \imath] h \in {\mathcal C}^\infty (\T^{\nu+1})$.
\end{lemma}

\begin{proposition}\label{derivata del flusso rispetto al toro}
Let $s_1 > s_0$  and  assume  the condition 
\begin{equation}\label{piccolezza a derivate i D sinistra}
 \| a \|_{2 s_0 + \frac{\beta_0  + 1}{2} + 3} \leq \delta (s_1)  \,, 
 \qquad \| a \|_{s_1 + s_0  + 3 + \frac32 \beta_0 } \leq 1
\end{equation}
for $ \d (s_1) > 0$ small enough. Then, for all 
$ \beta  \in \N^\nu$ with $ | \b | \leq \b_0 $, for all $ s \in [s_0, s_1] $  
\begin{align}
\| \langle D \rangle^{- \frac{|\beta |  + 1}{2}}   \partial_\vphi^\beta \big( \partial_i \Phi(t)[\hat \imath] \big) h  \|_s & 
\leq_s  \| \partial_i a [\hat \imath] \|_{ s + \frac32 |\beta|  + \frac12} \| h \|_s  
\label{derivate i omega vphi flusso} \\
\| \langle D \rangle^{- \frac{|\beta |  + 1}{2} - 1}   \partial_\vphi^\beta \big( \partial_i \Phi(t)[\hat \imath] \big) \langle D \rangle h  \|_s & 
\leq_s  \| \partial_i a [\hat \imath] \|_{ s + \frac32 |\beta| + \frac32} \| h \|_s \, . 
\label{derivate i omega vphi flusso + D destra}
\end{align}
\end{proposition}

\begin{proof}
We prove \eqref{derivate i omega vphi flusso + D destra}. The proof of \eqref{derivate i omega vphi flusso} is similar. We take $h, \widehat \imath$ in ${\mathcal C}^\infty$ with respect to $\vphi$ and $x$, so that $\langle D \rangle^{- \frac{|\beta |  + 1}{2} - 1}   \partial_\vphi^\beta \big( \partial_i \Phi(t)[\hat \imath] \big) \langle D \rangle h$ is ${\mathcal C}^\infty$.  
Differentiating  \eqref{flow-propagator} and using  Duhamel principle we get 
\begin{equation}\label{ancora una volta Duhamel derivate i A}
  \partial_\vphi^{\beta} \partial_i \Phi [\widehat \imath]  = \int_0^t \Phi(t - \tau) F_{\beta}(\tau)\, d \tau\,,
\quad F_{\beta} := F_{\beta}^{(1)} + F_{\beta}^{(2)}  
\end{equation}
where
\begin{align}
& F_{\beta}^{(1)}(\tau) :=  \sum_{
\beta_1 + \beta_2 = \beta,  |\b_1| < |\beta|} 
C(\beta_1, \beta_2) ( \partial_\vphi^{\beta_2} a) |D|^{\frac12} 
\partial_\vphi^{\beta_1} \partial_i \Phi [\widehat \imath](\tau) \label{ancora una volta 2 A}\\
& F_{\beta}^{(2)}(\tau) :=  \sum_{\beta_1 + \beta_2 = \beta} 
C(\beta_1, \beta_2) (  \partial_\vphi^{\beta_2}  \partial_i a [\widehat \imath]) |D|^{\frac12} 
 \partial_\vphi^{\beta_1}  \Phi(\tau)\,. \label{ancora una volta 3 A}
\end{align}
We argue by induction on $  \b $. The proof of \eqref{derivate i omega vphi flusso + D destra} 
for $  \b = 0 $ follows as a particular case of the  estimate below for the term in \eqref{ancora una volta 3 A}.
\\[1mm]
{\sc Estimate of \eqref{ancora una volta 2 A}.} For any $\beta_1 + \beta_2 = \beta $, $  |\b_1| < |\beta| $ we have 
\begin{align}
& \langle D \rangle^{- \frac{|\beta|   + 1}{2} - 1}\Phi(t - \tau)( \partial_\vphi^{\beta_2}  a) 
|D|^{\frac12}   \partial_\vphi^{\beta_1} \partial_i \Phi [\widehat \imath](\tau) \langle D \rangle \nonumber\\
& = \big( \langle D \rangle^{- \frac{|\beta| + 1}{2} - 1}\Phi(t - \tau) \langle D \rangle^{ \frac{|\beta|   + 1}{2} + 1} \big)
\big( \langle D \rangle^{- \frac{|\beta|   + 1}{2} - 1} (  \partial_\vphi^{\beta_2} a) 
\langle D \rangle^{ \frac{|\beta|   + 1}{2} + 1} \big) \nonumber\\
& \qquad |D|^{\frac12} \langle D \rangle^{- \frac12} \langle D \rangle^{- \frac{|\beta|    }{2} - 1}
 \partial_\vphi^{\beta_1} \partial_i \Phi [\widehat \imath](\tau) \langle D \rangle\,. \label{varese 14 A}
\end{align}
By \eqref{prima proprieta flusso}, $ s_0 \leq s \leq s_1 $, \eqref{piccolezza a derivate i D sinistra} one has 
\be
\begin{aligned}
\| \langle D \rangle^{- \frac{|\beta| + 1}{2} - 1} \Phi(t - \tau)  \langle D \rangle^{\frac{|\beta|   + 1}{2} + 1} h  \|_s 
&  \leq_s \|  h\|_s + \| a \|_{s + s_0 + 2 + \frac{|\beta|  + 1}{2} + 1} \| h \|_{s_0} \\
& \leq_s \|  h \|_s\,. \label{varese 1 A}
 \end{aligned}
\ee
Lemma \ref{lemma stime Ck parametri}, \ref{lemma composizione multiplier}, 
and  \eqref{norma a moltiplicazione}, \eqref{Norm Fourier multiplier}, imply
\be
\begin{aligned}
 \norma \langle D \rangle^{- \frac{|\beta|   + 1}{2} - 1} (  \partial_\vphi^{\beta_2} a)  
 \langle D \rangle^{ \frac{|\beta|   + 1}{2} + 1}  \norma_{0, s, 0} &
 \leq_s   \|  \partial_\vphi^{\beta_2} a\|_{s +  \frac{|\beta|  + 1}{2} + 1 }  \label{varese 15 A} \\
&  \leq_s \| a \|_{s + \frac32 |\beta|  + \frac32 } \stackrel{s \leq s_1\,,\,
 \eqref{piccolezza a derivate i D sinistra}}{\leq_s} 1 \,. 
 \end{aligned}
\ee
Since $ |\beta_1| < |\beta|$ the inductive hyphothesis implies 
\begin{align}
 \| |D|^{\frac12} \langle D \rangle^{- \frac12} \langle D \rangle^{- \frac{|\beta|   }{2} - 1}  
 \partial_\vphi^{\beta_1} \partial_i \Phi [\widehat \imath](\tau) \langle D \rangle h \|_s & \leq_s  \| \langle D \rangle^{- \frac{|\beta_1|   + 1 }{2} - 1}  \partial_\vphi^{\beta_1} \partial_i \Phi [\widehat \imath](\tau) \langle D \rangle h \|_s  \nonumber\\
 & \leq_s  \| \partial_i a [\widehat \imath]\|_{s + \frac32 |\beta |  + \frac32 } \| h \|_s\,.  \label{varese 16 A}
\end{align}
Then \eqref{ancora una volta 2 A},  \eqref{varese 14 A}, \eqref{varese 1 A}, \eqref{varese 15 A}, \eqref{varese 16 A} imply
\begin{equation}\label{varese 17 A}
\| \langle D \rangle^{- \frac{|\beta|   + 1}{2} - 1} \Phi(t - \tau) F_{\beta}^{(1)}(\tau) \langle D \rangle h \|_s \leq_s \| \partial_i a [\widehat \imath]\|_{s + \frac32 |\beta|  + \frac32}\| h \|_s\,.
\end{equation}
{\sc Estimate of \eqref{ancora una volta 3 A}.}
For any $\beta_1 + \beta_2 = \beta $, $t, \tau \in [0, 1]$, we have 
\begin{align}
& \langle D \rangle^{- \frac{|\beta|  + 1}{2} - 1} \Phi(t - \tau) (\partial_\vphi^{\beta_2} \partial_i a [\widehat \imath]) |D|^{\frac12}   \partial_\vphi^{\beta_1}  \Phi(\tau) \langle D \rangle =  \nonumber\\
&  \big( \langle D \rangle^{- \frac{|\beta| + 1}{2} - 1} \Phi(t - \tau)  \langle D \rangle^{\frac{|\beta|  + 1}{2} + 1} \big)
\big( \langle D \rangle^{- \frac{|\beta|   + 1}{2} - 1} (  \partial_\vphi^{\beta_2} \partial_i a [\widehat \imath]) \langle D \rangle^{ \frac{|\beta|  + 1}{2} + 1}\big)  \nonumber\\
&  |D|^{\frac12}\langle D \rangle^{- \frac12} \langle D \rangle^{- \frac{|\beta|   }{2} - 1} 
 \partial_\vphi^{\beta_1}  \Phi(\tau) \langle D \rangle \,. \label{varese 0 A}
\end{align}
 Lemma \ref{lemma stime Ck parametri}, \ref{lemma composizione multiplier}, 
and  \eqref{norma a moltiplicazione}, \eqref{Norm Fourier multiplier}, imply (as for \eqref{varese 15 A})
\begin{align}
\norma \langle D \rangle^{- \frac{|\beta|  + 1}{2} - 1} ( \partial_\vphi^{\beta_2}  \partial_i a [\widehat \imath]) \langle D \rangle^{ \frac{|\beta|   + 1}{2} + 1}  \norma_{0, s, 0} 
&  \leq_s  \| \partial_i a[\widehat \imath] \|_{s + \frac{3}{2}|\beta| + \frac32}\,. \label{varese 2 A}
\end{align}
By \eqref{D - 1/2 D vphi flusso N + Dx}, $ s_0 \leq s \leq s_1 $,  and  \eqref{piccolezza a derivate i D sinistra} we get  
\begin{align}
\| |D|^{\frac12}\langle D \rangle^{- \frac12} \langle D \rangle^{- \frac{|\beta| }{2} - 1}   
\partial_\vphi^{\beta_1}\Phi(\tau) \langle D \rangle h \|_s \leq_s  \| h \|_s\,. \label{varese 3 A}
\end{align}
Finally  \eqref{ancora una volta 3 A}, \eqref{varese 0 A}, \eqref{varese 1 A}, \eqref{varese 2 A}, \eqref{varese 3 A} imply 
\begin{equation}\label{varese 20 A}
\| \langle D \rangle^{- \frac{|\beta|   + 1}{2} - 1}\Phi(t - \tau ) F_{\beta}^{(2)}(\tau) \langle D \rangle h \|_s 
\leq_s  \| \partial_i a [\widehat \imath] \|_{s + \frac32 |\beta|  + \frac32} \| h \|_s\,. 
\end{equation}
In conclusion the estimate \eqref{derivate i omega vphi flusso + D destra}  follows by \eqref{varese 17 A}, \eqref{varese 20 A}. If $h \in H^s$, $\widehat \imath \in {\mathcal Y}^{s + \frac32 |\beta| + \frac32 + \sigma}$, then  \eqref{derivate i omega vphi flusso + D destra} follows by density. 
\end{proof}

\begin{proposition}\label{derivate i vphi omega flusso D destra}
Let $s_1 > s_0$ and  assume 
\begin{equation}\label{piccolezza a derivate i vphi omega}
 \| a \|_{ s_1 + s_0 + \frac52 + \beta_0 } \leq 1\, ,
\quad \| a\|_{ s_1 + s_0 + \beta_0 + 1} \leq \delta(s_1)\,, 
\end{equation}
for some $ \d (s_1) > 0 $ small. Then for all $|\beta| \leq \beta_0$, 
\begin{align}\label{stima norme unite derivate i Phi D destra}
 \|  \partial_\vphi^\beta  \partial_i \Phi[\widehat \imath] \langle D \rangle^{- \frac{|\beta|  + 1}{2}} h \|_s 
 & \leq_s  \| \partial_i a [\widehat \imath] \|_{s + s_0 + \frac12 + |\beta|} \| h \|_s\,, \quad \forall s \in [0, s_1]\,, \\
\label{stima norme unite derivate i Phi D destra + D}
\|\langle D \rangle  \partial_\vphi^\beta \partial_i \Phi[\widehat \imath] \langle D \rangle^{- \frac{|\beta|  + 1}{2} - 1} h \|_s 
& \leq_s  \| \partial_i a [\widehat \imath] \|_{s + s_0 + \frac32 + |\beta|} \| h \|_s\,, \quad \forall s \in [0, s_1 - 1]\, .
\end{align}
\end{proposition}

We first provide the estimate in $\| \cdot \|_{L^2_\vphi H^s_x}$ for all $s \in [0, s_1]$. 
\begin{lemma}
Assume \eqref{piccolezza a derivate i vphi omega}.
Then for all $\vphi \in \T^\nu$, the following estimate holds
\be\label{derivate in sezione H_x}
\|   \partial_\vphi^\beta \partial_i \Phi[\widehat \imath] \langle D \rangle^{- \frac{|\beta|  + 1}{2}} h \|_{H^s_x} \leq_s  \| \partial_i a [\widehat \imath]\|_{s + s_0 + \frac12 + |\beta| } \| h \|_{H^s_x}\,, \quad \forall s\in [0, s_1]\,.
\ee
\end{lemma}

\begin{proof}
Let us suppose that $\widehat \imath$ and $h$ are ${\mathcal C}^\infty$. We argue by induction on $  \b $, supposing that we have already proved \eqref{derivate in sezione H_x}
for $  |\b_1 | < |\b| $. We use the integral representation of 
$  \partial_\vphi^{\beta}  \partial_i \Phi [\widehat \imath] $ in \eqref{ancora una volta Duhamel derivate i A}. 
For all  $ \beta_1 + \beta_2 = \beta $, $|\beta_1| < |\beta| $, $ t, \tau \in [0, 1] $, 
by \eqref{stima flusso PDE s 0 1}, \eqref{stima tame Phi t}, \eqref{piccolezza a derivate i vphi omega}, 
and  the inductive hyphothesis, 
\begin{align}
& \| \Phi(t - \tau) ( \partial_\vphi^{\beta_2}  a) |D|^{\frac12}  \partial_\vphi^{\beta_1} \partial_i \Phi [\widehat \imath] \langle D \rangle^{- \frac{|\beta|  + 1}{2}} h \|_{H^s_x}  \label{ancora una volta 6}  \\
& \leq_s \|  
 a\|_{{\mathcal C}^{s + |\beta| }} \|   \partial_\vphi^{\beta_1} \partial_i \Phi [\widehat \imath] \langle D \rangle^{- \frac{|\beta| + 1}{2}} h \|_{H^{s + \frac12}_x}    
\leq_s   \| \partial_i a [\widehat \imath]\|_{s + s_0 + \frac12 + |\beta| + 1  }\| h \|_{H^s_x} \, . 
\nonumber
\end{align}
Similarly, for all $\beta_1 + \beta_2 = \beta $, 
by \eqref{stima flusso PDE s 0 1}, \eqref{stima tame Phi t}, \eqref{piccolezza a derivate i vphi omega}
\begin{align}
& \| \Phi(t - \tau) (\partial_\vphi^{\beta_2}  \partial_i a [\widehat \imath]) |D|^{\frac12} 
 \partial_\vphi^{\beta_1}  \Phi(\tau) \langle D \rangle^{- \frac{|\beta| + 1}{2}} h \|_{H^s_x} \label{ancora una volta 4} \\
& \leq_s \| \partial_i a [\widehat \imath] \|_{{\mathcal C}^{s + |\beta| }} \| \partial_\vphi^{\beta_1}  \Phi(\tau) \langle D \rangle^{- \frac{|\beta|  + 1}{2}} h \|_{H^{s + \frac12}_x} \nonumber \\
&  \stackrel{\eqref{partial beta partial k D Phi norme basse}, 
 \eqref{partial beta partial k D Phi norme alte}, \eqref{piccolezza a derivate i vphi omega}}{\leq_s}
 \| \partial_i a [\widehat \imath] \|_{s + s_0 + |\beta| }  \|h \|_{H^{s}_x} \,. \nonumber
\end{align}
By \eqref{ancora una volta Duhamel derivate i A}, \eqref{ancora una volta 6}, \eqref{ancora una volta 4} 
we deduce \eqref{derivate in sezione H_x}. If $h \in H^s_x$ and 

\noindent
$\widehat \imath \in {\mathcal Y}^{s + s_0 + \frac12 + |\beta| + \sigma}$ it follows by density. 
\end{proof}

Then, integrating in $ \vphi $, we get the following corollary

\begin{lemma}
Let $s_1 > s_0$ and assume \eqref{piccolezza a derivate i vphi omega}.
Then for all $|\beta| \leq \beta_0$
\begin{equation}\label{stima partial i Phi D destra L2 vphi Hs x}
\begin{aligned}
&\!  \|  \partial_\vphi^\beta \partial_i \Phi [\widehat \imath] \langle D \rangle^{- \frac{|\beta|  + 1}{2}} h \|_{L^2_\vphi H^s_x} \\
& \leq_s \| \partial_i a [\widehat \imath] \|_{s + s_0 + \frac12 + |\beta| }^{k_0, \gamma} \| h \|_{L^2_\vphi H^s_x}\,, \quad \forall s \in [0, s_1]\, , 
\end{aligned}
\end{equation}
\begin{equation}\label{stima partial i Phi D destra L2 vphi Hs x +D}
\begin{aligned}
& 
\| \langle D \rangle  \partial_\vphi^\beta \partial_i \Phi [\widehat \imath] \langle D \rangle^{- \frac{|\beta|  + 1}{2} - 1} h \|_{L^2_\vphi H^s_x} \\
& \leq_s  \| \partial_i a [\widehat \imath] \|_{s + s_0 + \frac32 + |\beta|  } \| h \|_{L^2_\vphi H^s_x}\,, \forall s \in [0, s_1 - 1]\,.
\end{aligned}
\end{equation}
\end{lemma}

\noindent
{\bf Proof of Proposition \ref{derivate i vphi omega flusso D destra}.} Let $h$ and $\widehat \imath$ be ${\mathcal C}^\infty$ with respect to the variables $\vphi$ and $x$. 
\\[1mm]
{\sc Proof of \eqref{stima norme unite derivate i Phi D destra}.}
We argue by induction $|\beta|$.  For $  \b = 0 $ the proof of \eqref{stima norme unite derivate i Phi D destra}
is a particular case of the estimate of \eqref{Duhamel deri}, \eqref{bruxelles D} (with $ k = 0, \b + \a = 0 $)
 in \eqref{bruxelles 6}.  Assume that we have proved \eqref{stima norme unite derivate i Phi D destra} for  
$  \partial_\vphi^\beta \partial_i \Phi [\widehat \imath] \langle D \rangle^{- \frac{|\beta|  + 1}{2}}$ for all $|\beta| < n$, and let us prove it for $|\beta| = n$. Then we 
estimate $\|   \partial_\vphi^\beta \partial_i \Phi [\widehat \imath] \langle D \rangle^{- \frac{|\beta|  + 1}{2}} h  \|_s$ 
for all $|\beta| = n$, for all $s \in [0, s_1]$. For $s = 0$ one has 
$$
\begin{aligned}
\|   \partial_\vphi^\beta \partial_i \Phi [\widehat \imath] \langle D \rangle^{- \frac{|\beta| + 1}{2}} h \|_{0} & 
= \| 
 \partial_\vphi^\beta \partial_i \Phi [\widehat \imath] \langle D \rangle^{- \frac{|\beta|  + 1}{2}} h \|_{L^2_\vphi L^2_x} \\ 
 &  \stackrel{\eqref{stima partial i Phi D destra L2 vphi Hs x}}{\lessdot}  \| \partial_i a [\widehat \imath]\|_{s_0 + \frac12 + |\beta|} \| h \|_0\,.
 \end{aligned}
$$
Then, assume that \eqref{stima norme unite derivate i Phi D destra} holds
up to the Sobolev index $ s  < s_1 $ and we  prove it for $s + 1 \leq s_1$. We have 
$$
\begin{aligned}
\|   \partial_\vphi^\beta \partial_i \Phi [\widehat \imath] \langle D \rangle^{- \frac{|\beta|  + 1}{2}} h \|_{s + 1} & 
\simeq \|   \partial_\vphi^\beta \partial_i \Phi [\widehat \imath] \langle D \rangle^{- \frac{|\beta| + 1}{2}} h \|_{L^2_\vphi H^{s + 1}_x} \\ 
& + \|   \partial_\vphi^\beta \partial_i \Phi [\widehat \imath] \langle D \rangle^{- \frac{|\beta|  + 1}{2}} h \|_{H^{s + 1}_\vphi L^2_x}\,.
\end{aligned}
$$
By \eqref{stima partial i Phi D destra L2 vphi Hs x} we have 
\begin{equation}\label{bruxelles 0}
\|  \partial_\vphi^\beta \partial_i \Phi [\widehat \imath] \langle D \rangle^{- \frac{|\beta|  + 1}{2}} h \|_{L^2_\vphi H^{s + 1}_x} \leq_s \| \partial_i a [\widehat \imath] \|_{s + 1 + s_0 + \frac12 + |\beta| }  \| h \|_{s + 1}\,.
\end{equation}
Then  
\begin{align}
& \|  \partial_\vphi^\beta \partial_i \Phi [\widehat \imath] \langle D \rangle^{- \frac{|\beta|  + 1}{2}} h \|_{H^{s + 1}_\vphi L^2_x}  \simeq \|   \partial_\vphi^\beta \partial_i \Phi [\widehat \imath] \langle D \rangle^{- \frac{|\beta|  + 1}{2}} h \|_0
  \nonumber\\
& + \sup_{\alpha \in \N^\nu, |\alpha| = 1} \|   \partial_\vphi^\beta \partial_i \Phi [\widehat \imath] \langle D \rangle^{- \frac{|\beta| + 1}{2}} \partial_\vphi^\alpha h\|_{H^s_\vphi L^2_x} \nonumber\\
& + \sup_{\alpha \in \N^\nu, |\alpha| = 1} \|   \partial_\vphi^{\beta + \alpha} \partial_i \Phi [\widehat \imath] \langle D \rangle^{- \frac{|\beta|  + 1}{2}} h\|_{H^s_\vphi L^2_x}\,. \label{bruxelles 1} 
\end{align}
The inductive hyphothesis implies 
\begin{equation}\label{bruxelles 2}
\|  \partial_\vphi^\beta \partial_i \Phi [\widehat \imath] \langle D \rangle^{- \frac{|\beta|  + 1}{2}} \partial_\vphi^\alpha h \|_{s} \leq_s  \| \partial_i a [\widehat \imath]\|^{k_0, \gamma}_{s + s_0 + \frac12 +|\beta| } \| h \|_{s + 1}\,.
\end{equation}
We estimate the last term in \eqref{bruxelles 1}.
Differentiating  \eqref{flow-propagator} and using the Duhamel principle we get
\be\label{Duhamel deri}
 \partial_\vphi^{\beta + \alpha} \partial_i \Phi [\widehat \imath] 
  = \int_0^t \Phi(t - \tau) F_{\beta}(\tau)\, d \tau\,, \quad 
F_{\beta} := F_{\beta}^{(1)} + F_{\beta}^{(2)} + F_{\beta}^{(3)} + F_{\beta}^{(4)} \, ,  
\ee
with 
\begin{align}
& F_{\beta}^{(1)}(\tau) :=  \sum_{\beta_1 + \beta_2 = \beta + \alpha,
|\beta_1| = |\beta|} C(\beta_1, \beta_2) \partial_\vphi^{\beta_2} a |D|^{\frac12}   \partial_\vphi^{\beta_1} \partial_i \Phi[\widehat \imath](\tau) \label{bruxelles A}\\
& F_{\beta}^{(2)}(\tau) := \sum_{\beta_1 + \beta_2 = \beta + \alpha, |\beta_1| < |\beta|} 
C(\beta_1, \beta_2) \partial_\vphi^{\beta_2} a |D|^{\frac12}  \partial_\vphi^{\beta_1}  \partial_i \Phi[\widehat \imath](\tau)  \label{bruxelles B} \\
& F_{\beta}^{(3)}(\tau) := \sum_{
\beta_1 + \beta_2 = \beta + \alpha} C(\beta_1, \beta_2) (  \partial_\vphi^{\beta_2} \partial_i a [\widehat \imath]) |D|^{\frac12}
   \partial_\vphi^{\beta_1} \Phi(\tau)\,. \label{bruxelles D}
\end{align}   
We estimate separately the terms $\Phi(t - \tau) F_{\beta}^{(m)}(\tau)$,  $m = 1,2,3$. 
We use that by \eqref{stima flusso norma bassa}, \eqref{stima flusso norme unite},  \eqref{piccolezza a derivate i vphi omega}
\begin{equation}\label{riccardo bruxelles 1}
\sup_{t \in [0, 1]} \| \Phi(t) h \|_s \leq_s \| h \|_s \qquad \forall s \in [0, s_1]\, . 
\end{equation}
For all $t, \tau \in [0, 1]$, $\beta_1 + \beta_2 = \beta + \alpha$, $|\beta_1| = |\beta|$, one has by \eqref{riccardo bruxelles 1}
\begin{align}
& \| \Phi(t - \tau) \partial_\vphi^{\beta_2} a |D|^{\frac12}   \partial_\vphi^{\beta_1} \partial_i \Phi[\widehat \imath](\tau) \langle D \rangle^{- \frac{|\beta|  + 1}{2}} h   \|_s  \nonumber\\
& \leq_s  \| a \|_{s + s_0 + |\beta| + 1} \|   \partial_\vphi^{\beta_1} \partial_i \Phi[\widehat \imath](\tau) \langle D \rangle^{- \frac{|\beta|  + 1}{2}} h  \|_{s + 1}\,.  \label{bruxelles 3}
\end{align}
For all $t, \tau \in [0, 1] $, $\beta_1 + \beta_2 = \beta + \alpha$, $|\beta_1| < |\beta|$, by \eqref{riccardo bruxelles 1}, the inductive hyphothesis, 
and \eqref{piccolezza a derivate i vphi omega} we get 
\begin{align}
& \| \Phi(t - \tau) \partial_\vphi^{\beta_2} a |D|^{\frac12} \partial_\vphi^{\beta_1} \partial_i \Phi[\widehat \imath](\tau) \langle D \rangle^{- \frac{|\beta| + 1}{2}} h   \|_s  \nonumber \\
& \leq_s 
\| a \|_{s + s_0 + |\beta| + 1} \|  \partial_\vphi^{\beta_1} \partial_i \Phi[\widehat \imath](\tau) 
\langle D \rangle^{- \frac{|\beta|  + 1}{2}} h  \|_{s + 1} \nonumber\\
& \leq_s \| \partial_i a [\widehat \imath]\|_{s + 1 + s_0 + \frac12 + |\beta| - 1  } \| h \|_{s + 1} \,.  \label{bruxelles 4}
\end{align}
 For all $t, \tau \in [0, 1]$, $\beta_1 + \beta_2 = \beta + \alpha$, we have, by \eqref{riccardo bruxelles 1}, 
\begin{align}
&  \|\Phi(t - \tau) (\partial_\vphi^{\beta_2}  \partial_i a [\widehat \imath]) |D|^{\frac12} 
\partial_\vphi^{\beta_1}  \Phi(\tau) \langle D \rangle^{- \frac{|\beta|  + 1 }{2}} h  \|_s  \nonumber\\
& \leq_s \|  \partial_i a [\widehat \imath] \|_{{\mathcal C}^{s + |\beta| + 1}} \| 
 \partial_\vphi^{\beta_1}  \Phi(\tau) \langle D \rangle^{- \frac{|\beta|  + 1 }{2}} h \|_{s + 1} \nonumber\\
& \leq_s  \|\partial_i a [\widehat \imath] \|_{s + s_0 + |\beta| + 1} \| h \|_{s + 1} . \label{bruxelles 6}  
\end{align}
using  \eqref{copenaghen A omega}, \eqref{copenaghen B omega}, \eqref{piccolezza a derivate i vphi omega}. 
Collecting \eqref{bruxelles 0}-\eqref{bruxelles 6} we get 
\begin{align} 
& \sup_{|\beta| = n} \sup_{t \in [0, 1]} \|   \partial_\vphi^\beta  \partial_i \Phi[\widehat \imath]  
\langle D \rangle^{- \frac{|\beta| +1}{2}} h \|_{s + 1} \leq_s \| \partial_i a [\widehat \imath] \|_{s + 1+ s_0 + \frac12 + |\beta| } \| h \|_{s + 1} \nonumber\\
& \qquad + \| a \|_{s + 1 + s_0 + |\beta|} \sup_{|\beta| = n} \sup_{t \in [0, 1]} \|  \partial_\vphi^\beta  \partial_i \Phi[\widehat \imath]   \langle D \rangle^{- \frac{|\beta|  +1}{2}} h \|_{s + 1} \nonumber
\end{align}
which, by  \eqref{piccolezza a derivate i vphi omega}, implies \eqref{stima norme unite derivate i Phi D destra} with Sobolev index $ s + 1 $.
  \hfill $\Box$

\smallskip

\noindent
{\sc Proof of \eqref{stima norme unite derivate i Phi D destra + D}.} 
We argue by induction on $s$. For $s = 0$ it follows by  \eqref{stima partial i Phi D destra L2 vphi Hs x +D}.
Then assuming that \eqref{stima norme unite derivate i Phi D destra + D} holds 
up to the Sobolev index $s < s_1 - 1$ and we prove it for $s + 1$. We have 
\begin{align}
\| \langle D \rangle \partial_\vphi^\beta  (\partial_i \Phi [\widehat \imath]) \langle D \rangle^{- \frac{|\beta|  + 1}{2} - 1} h \|_{s + 1} & \simeq \| \langle D \rangle   \partial_\vphi^\beta (\partial_i \Phi [\widehat \imath]) \langle D \rangle^{- \frac{|\beta|  + 1}{2} - 1} h\|_{L^2_\vphi H^{s + 1}_x} \nonumber\\
& \quad + \|  \langle D \rangle   \partial_\vphi^\beta (\partial_i \Phi [\widehat \imath]) \langle D \rangle^{- \frac{|\beta| + 1}{2} - 1} h \|_{H^{s + 1}_\vphi L^2_x}\,. \label{paoletto 0}
\end{align}
By \eqref{stima partial i Phi D destra L2 vphi Hs x +D} we have 
\begin{align}
 \|  \langle D \rangle   \partial_\vphi^\beta (\partial_i \Phi [\widehat \imath]) \langle D \rangle^{- \frac{|\beta|  + 1}{2} - 1} h \|_{L^2_\vphi H^{s + 1}_x} \leq_s  \| \partial_i a [\widehat \imath] \|_{s + 1+ s_0 + \frac32 + |\beta| } \|  h \|_{s + 1}\,. \label{paoletto 1}
\end{align}
We estimate the second term in \eqref{paoletto 0}. By the inductive hyphothesis  
and  \eqref{stima partial i Phi D destra L2 vphi Hs x +D} one has  
\begin{align}
& \| \langle D \rangle   \partial_\vphi^\beta (\partial_i \Phi [\widehat \imath]) \langle D \rangle^{- \frac{|\beta|  + 1}{2} - 1} h \|_{H^{s + 1}_\vphi L^2_x} \nonumber\\
& \simeq \| \langle D \rangle   \partial_\vphi^\beta (\partial_i \Phi [\widehat \imath]) \langle D \rangle^{- \frac{|\beta|  + 1}{2} - 1} h \|_{L^2_\vphi L^2_x}  \nonumber \\
& \quad + \sup_{\alpha \in \N^\nu,  |\alpha| = 1}\| \langle D \rangle  \partial_\vphi^\beta (\partial_i \Phi [\widehat \imath]) \langle D \rangle^{- \frac{|\beta|  + 1}{2} - 1} \partial_\vphi^\alpha h \|_{H^s_\vphi L^2_x} \nonumber\\
& \quad + \sup_{\alpha \in \N^\nu, |\alpha| = 1}\| \langle D \rangle   
\partial_\vphi^{\beta + \alpha} (\partial_i \Phi [\widehat \imath]) \langle D \rangle^{- \frac{|\beta|  + 1}{2} - 1} h \|_{H^s_\vphi L^2_x} \nonumber\\
& \leq_s  \| h \|_{s + 1}  + \sup_{\alpha \in \N^\nu, |\alpha| = 1}\| \langle D \rangle 
 \partial_\vphi^{\beta + \alpha} (\partial_i \Phi [\widehat \imath]) \langle D \rangle^{- \frac{|\beta| + 1}{2} - 1} h \|_{s}\,. \label{paoletto 2}
\end{align}
Finally, for all $ \alpha \in \N^\nu $, $ | \alpha | = 1 $, we have, by \eqref{stima norme unite derivate i Phi D destra}, 
\begin{align}
\| \langle D \rangle   \partial_\vphi^{\beta + \alpha} (\partial_i \Phi [\widehat \imath]) \langle D \rangle^{- \frac{|\beta| + 1}{2} - 1} h \|_{s} & \leq_s \|   \partial_\vphi^{\beta + \alpha}  (\partial_i \Phi [\widehat \imath]) \langle D \rangle^{- \frac{|\beta| + 2 }{2} } [\langle D \rangle^{- \frac12} h]\|_{s + 1} \nonumber\\
& 
\leq_s  \| \partial_i a [\widehat \imath] \|_{s + 1 + s_0 + \frac32 + |\beta| } \| h \|_{s + 1}\,. \label{paoletto 3}
\end{align}
Hence \eqref{paoletto 0}-\eqref{paoletto 3} imply the estimate \eqref{stima norme unite derivate i Phi D destra + D} with Sobolev index $s + 1$. 
If $h \in H^s$ and $\widehat \imath \in {\mathcal Y}^{s + s_0 + |\beta| + \frac12 + \sigma}$ (resp. $\widehat \imath \in {\mathcal Y}^{s + s_0 + |\beta| + \frac32 + \sigma}$ ), the estimate \eqref{stima norme unite derivate i Phi D destra} (resp. \eqref{stima norme unite derivate i Phi D destra + D}) follows by density. 
 \hfill $\Box$

\smallskip
 
We now estimate the adjoint $ \Phi^* $ of the time-$ 1 $ flow   $ \Phi = \Phi(\vphi, 1) $.  
As in \cite{BBM-auto} (Lemma 8.2) we 
represent the adjoint $ \Phi^* = \Psi = \Psi(\vphi, 0)$ with the 
backward flow   $\Psi (\vphi, t)$ of 
\begin{equation}\label{propagatore aggiunto}
\partial_t \Psi(\vphi, t) =  \ii |D|^{\frac12} a \Psi(\vphi, t)\,, \qquad \Psi(\vphi, 1) = {\rm Id}\,.
\end{equation}
Indeed, since $ \Phi (\vphi, t) $ solves \eqref{flow-propagator} and $ \Psi (\vphi, t) $ solves \eqref{propagatore aggiunto}, we have,
for all $ u_0, v_0 \in L^2_x (\T) $, that  
$$
\partial_t \big( \Phi(\vphi, t) [u_0] \,,\, \Psi(\vphi, t)[v_0]  \big)_{L^2_x} = 0\,, \quad \forall t \in [0, 1] \, .
$$
Therefore $ ( \Phi(\vphi, 1) [u_0] \,,\, v_0 )_{L^2_x} = ( u_0 \,,\, \Psi(\vphi, 0)[v_0] )_{L^2_x} $, namely 
\be\label{rappresentazione aggiunto}
\Psi(\vphi, 0) = \Phi(\vphi, 1)^* = \Phi(\vphi)^* \, .
\ee 
The adjoint operator, since it is the flow of   \eqref{propagatore aggiunto}, satisfies properties like those stated in Lemma
\ref{dipendenza liscia dai parametri}. 

\begin{proposition} {\bf (Adjoint)}\label{teorema stime aggiunto flusso}
Assume that 
\begin{equation}\label{piccolezza stime aggiunto flusso}
\| a \|_{2 s_0 + \frac52 + k_0}^{k_0, \gamma} \leq 1 \,, \qquad \| a \|_{2 s_0 + 1} \leq \delta (s) 
\end{equation}
for some $  \delta(s) > 0 $ small enough. Then for any 
$ k \in \N^{\nu + 1}$, $|k| \leq k_0$, for all $s \geq s_0 $, 
\begin{align}
 \| (\partial_\lambda^k\Phi^* )  h \|_s & \leq_s \gamma^{- |k|} \big( \| h \|_{s + \frac{|k|}{2}} + \| a \|_{s + s_0 + |k| + \frac32}^{k_0, \gamma} \| h \|_{s_0 + \frac{|k|}{2}}\big) \label{stima aggiunto flusso} 
 \end{align}
 \begin{equation} \label{aggiunto flusso - identita}
 \begin{aligned}
 & \| \partial_\lambda^k(\Phi^* - {\rm Id}) h \|_s  \\
 & \leq_s \gamma^{- |k|} \big(\| a \|_{s_0}^{k_0, \gamma} \| h \|_{s + \frac{|k| + 1}{2}} + \| a \|_{s + s_0 + |k| + 2}^{k_0, \gamma} \| h \|_{s_0 + \frac{|k| + 1}{2}} \big) \, .
\end{aligned}
\end{equation}
\end{proposition}

\begin{proof}
First we take $h \in {\mathcal C}^\infty$. 

\noindent
{\sc Proof of \eqref{stima aggiunto flusso}.} 
The equation  \eqref{propagatore aggiunto} can be written as 
$$
\partial_t \Psi(\vphi, t) = \ii a |D|^{\frac12} \Psi(\vphi, t) + \ii [|D|^{\frac12}, a] \Psi(\vphi, t) \,, \qquad \Psi(\vphi, 1) = {\rm Id}\, ,
$$
and, by Duhamel principle,  one gets 
\begin{equation}\label{formula Duhamel Psi aggiunto}
\Psi( t)  = \Phi(t) \Phi(1)^{- 1}  - \ii \int_t^1 \Phi(t - \tau) [|D|^{\frac12}, a] \Psi(\tau)\, d \tau\,.
\end{equation}
By \eqref{rappresentazione aggiunto} 
the estimate \eqref{stima aggiunto flusso} follows by proving that,  for all $|k| \leq k_0 $, $ s \geq s_0 $,  
\begin{equation}\label{argomento induttivo Psi aggiunto}
\sup_{t \in [0, 1]}\| \partial_\lambda^k \Psi(t ) h \|_s \leq_s \gamma^{- |k|} \big( \| h \|_{s + \frac{|k|}{2}} +
\|a  \|_{s + s_0 + |k| + \frac32}^{k_0, \gamma} \| h \|_{s_0 + \frac{|k|}{2}} \big) \, . 
\end{equation}
For $ k = 0 $, the estimate  \eqref{argomento induttivo Psi aggiunto} follows by the same proof below
(using only \eqref{formula Duhamel Psi aggiunto}, \eqref{stima flusso norme unite}, and 
\eqref{stima-pseudo-diff-solita} with $ k_1 = k_2 = 0 $). 
Then we argue by induction.
 We assume  that  \eqref{argomento induttivo Psi aggiunto} 
holds for $k_1 \prec k$ with $|k| \leq k_0$ and we prove it for $ k $. 
Differentiating \eqref{formula Duhamel Psi aggiunto}  we get 
\be\label{derivata omegak Psi}
\partial_\lambda^k \Psi(t) = F_1^{(k)}(t) + F_2^{(k)}(t) 
\ee
where 
\begin{align}
F_1^{(k)}(t) & := \partial_\lambda^k (\Phi(t) \Phi(1)^{- 1}) \nonumber\\
& \quad  - \ii \sum_{
k_1 + k_2 + k_3 = k, k_3 \prec k}\int_t^1 \partial_\lambda^{k_1}\Phi(t - \tau) [|D|^{\frac12}, \partial_\lambda^{k_2} a]  \partial_\lambda^{k_3} \Psi(\tau)\, d \tau\,, \label{F k 1 aggiunto}\\
\label{F k 2 aggiunto}
F_2^{(k)}(t) & := - \ii \int_t^1 \Phi(t - \tau) [|D|^{\frac12},  a]  \partial_\lambda^{k} \Psi(\tau)\, d \tau\,.
\end{align}
{\sc Estimate of $F_1^{(k)}(t)$.} 
By 
\eqref{stima-rozzina}, \eqref{stima flusso norme unite} (for $ \Phi (1)^{-1} $), and \eqref{piccolezza stime aggiunto flusso}, 
we get 
\begin{equation}\label{stima rick aggiunto 0}
\| \partial_\lambda^k (\Phi(t) \Phi(1)^{- 1}) h \|_s \leq_s \gamma^{- |k|} \big( \| h \|_{s + \frac{|k|}{2}} + \| a \|_{s + s_0 + |k| + 1}^{k_0, \gamma} \| h \|_{s_0 + \frac{|k|}{2}}  \big) 
\end{equation}
and, for all $k_1 + k_2 + k_3 = k$, $k_3 \prec k $, 
\begin{align}
& \| \partial_\lambda^{k_1} \Phi(t - \tau) [|D|^{\frac12}, \partial_\lambda^{k_2} a] \partial_\lambda^{k_3} \Psi(\tau) h \|_s \nonumber\\
& {\leq_s}  \gamma^{- |k_1|} \| [|D|^{\frac12}, \partial_\lambda^{k_2} a] \partial_\lambda^{k_3} \Psi(\tau) h \|_{s + \frac{|k_1|}{2}} \nonumber\\
& \quad +\gamma^{- |k_1|}  \| a \|^{k_0, \gamma}_{s + s_0 + |k_1| + 1} \| [|D|^{\frac12}, \partial_\lambda^{k_2} a] \partial_\lambda^{k_3} \Psi(\tau) h \|_{s_0 + \frac{|k_1|}{2}}\,. \label{stima rick aggiunto 1}
\end{align}
By \eqref{stima commutator parte astratta} we have 
\be\label{stima-pseudo-diff-solita}
\norma [|D|^{\frac12}, \partial_\lambda^{k_2} a] \norma_{-\frac12, s + \frac{|k_1|}{2}, 0}  
\leq_s  \| \partial_\lambda^{k_2} a \|_{s + \frac{|k_1|}{2} + \frac52} \leq_s \gamma^{- |k_2|} \| a \|^{k_0, \gamma}_{s + \frac{|k_1|}{2} + \frac52}
\ee 
and, by  \eqref{piccolezza stime aggiunto flusso}, and the inductive hypothesis for $ k_3 \prec k $, we get
\begin{align}
\| [|D|^{\frac12}, \partial_\lambda^{k_2} a] \partial_\lambda^{k_3} \Psi(\tau) h \|_{s + \frac{|k_1|}{2}}  & \leq_s \gamma^{- (|k_2| + |k_3|)} \big( \| h \|_{s + \frac{|k_1| + |k_3|}{2}} \nonumber\\
& \quad + \| a \|_{s + s_0 + |k| + \frac32}^{k_0, \gamma} \| h \|_{s_0 + \frac{|k_1| + |k_3|}{2}} \big)\,. \label{stima rick aggiunto 2}
\end{align}
Hence \eqref{F k 1 aggiunto}, \eqref{stima rick aggiunto 0},  \eqref{stima rick aggiunto 1}, \eqref{stima rick aggiunto 2} imply 
\begin{equation}\label{stima rick aggiunto 3}
\| F_1^{(k)}(t) h \|_s \leq_s \gamma^{- |k|} \big( \| h \|_{s + \frac{|k|}{2}} + 
\| a \|_{s + s_0 + |k| + \frac32} \| h \|_{s_0 + \frac{|k|}{2}} \big)\,.
\end{equation}
{\sc Estimate of $F_2^{(k)}(t)$}. For all $t, \tau \in [0, 1]$, using  \eqref{stima flusso norme unite}, the bound
$$ 
\norma [|D|^{\frac12}, a] \norma_{- \frac12, s, 0} \leq_s \| a \|_{s + \frac52} 
$$
(see \eqref{stima-pseudo-diff-solita} with $ k_1 = k_2 = 0 $),  and \eqref{piccolezza stime aggiunto flusso}
we get 
\begin{align}
\| F_2^{(k)}(t) h\|_s \leq_s \| a \|_{s_0 + \frac52} \sup_{\tau \in [0, 1]}\| \partial_\lambda^k \Psi(\tau) h\|_s + \| a \|_{s + s_0 + 1} \sup_{\tau \in [0, 1]} \|\partial_\lambda^k \Psi(\tau) h \|_{s_0}\,. \label{stima F 2 k (t)}
\end{align}
{\sc Estimate of $ \partial_\lambda^k \Psi(t) $}.
By  \eqref{derivata omegak Psi},  
\eqref{stima rick aggiunto 3}, \eqref{stima F 2 k (t)} we get  
\begin{align}
\|\partial_\lambda^k \Psi(t) h \|_s & \leq_s \gamma^{- |k|} \big( \| h \|_{s + \frac{|k|}{2}} + \| a \|_{s + s_0 + |k| + \frac32} \| h \|_{s_0 + \frac{|k|}{2}} \big) \nonumber\\
& \quad + \| a \|_{s_0 + \frac52} \sup_{\tau \in [0, 1]}\| \partial_\lambda^k \Psi(\tau) h\|_s + \| a \|_{s + s_0 + 1} \sup_{\tau \in [0, 1]} \|\partial_\lambda^k \Psi(\tau) h \|_{s_0}\,. \label{pa-omega-da}
\end{align}
Then, for $ s = s_0 $, using that, by \eqref{piccolezza stime aggiunto flusso}, 
$ \| a \|_{2 s_0 + 1} \leq \delta(s)$ is small enough, we get 
$$
\begin{aligned}
\sup_{t \in [0, 1]} \|\partial_\lambda^k \Psi(t) h \|_{s_0} 
& \lessdot 
\gamma^{- |k|} \big( \| h \|_{s_0 + \frac{|k|}{2}} + \| a \|_{2s_0 + |k| + \frac32} \| h \|_{s_0 + \frac{|k|}{2}} \big) \\
& \stackrel{\eqref{piccolezza stime aggiunto flusso}} \lessdot
\gamma^{- |k|} \| h \|_{s_0 + \frac{|k|}{2}}\,,
\end{aligned}
$$
and therefore, by \eqref{pa-omega-da},  for all $ s \geq s_0 $, 
$$
\begin{aligned}
\sup_{t \in [0, 1]} \|\partial_\lambda^k \Psi(t) h \|_s   \leq_s \gamma^{- |k|} \big( \| h \|_{s + \frac{|k|}{2}} 
& + \| a \|_{s + s_0 + |k| + \frac32} \| h \|_{s_0 + \frac{|k|}{2}} \big)\\
& + \| a \|_{s_0 + \frac52} \sup_{t \in [0, 1]}\| \partial_\lambda^k \Psi(t) h\|_s
 \end{aligned} 
$$
which yields the estimate \eqref{argomento induttivo Psi aggiunto} for $\partial_\lambda^k \Psi(t)$  (using again \eqref{piccolezza stime aggiunto flusso} and $\delta(s)$ small enough).

\noindent
{\sc Proof of \eqref{aggiunto flusso - identita}.} By  \eqref{propagatore aggiunto} we have 
$ \Psi(\vphi, t) - {\rm Id} = - \ii \int_t^1  |D|^{\frac12} a \Psi(\vphi, \tau)\, d \tau $, 
then it is enough to apply \eqref{argomento induttivo Psi aggiunto}. 
If $h \in H^{s + \frac{|k|}{2}}$ (resp. $h \in H^{s + \frac{|k| + 1}{2}}$), the estimate \eqref{stima aggiunto flusso} (resp. \eqref{aggiunto flusso - identita}) follows by density. 
\end{proof}

Finally we  estimate the variation of the adjoint operator $ \Phi^* $ with respect to the torus 
$ i(\vphi) $. 
\begin{proposition}\label{teorema derivata aggiunto flusso}
Let $s_1  > s_0$ and assume the condition 
\begin{equation}\label{piccolezza stime derivata i aggiunto flusso}
\| a \|_{s_1 + s_0 + 3} \leq 1\,, \quad \| a \|_{s_1 + s_0 + 1} \leq \delta (s_1) \,, 
\end{equation}
for some $ \delta(s_1) > 0 $ small. Then, for all $s \in [s_0, s_1] $,  
\begin{equation}\label{derivata i aggiunto}
\| \partial_i \Phi^*[\widehat \imath] h \|_s \leq_s \| \partial_i a[\widehat \imath ] \|_{s + s_0 + \frac12} \| h \|_{s + \frac12}\,.
\end{equation} 
\end{proposition} 

\begin{proof}

First, we prove that  the map $\Psi(t)$ defined in \eqref{formula Duhamel Psi aggiunto} satisfies 
\eqref{derivata i aggiunto} for $h$ and $\widehat \imath$ which are ${\mathcal C}^\infty$ with respect to $\vphi$ and $x$. By differentiating  \eqref{formula Duhamel Psi aggiunto} we get 
\begin{align}
\partial_i \Psi(t)[\widehat \imath] & = \partial_i (\Phi(t) \Phi(1)^{- 1})[\widehat \imath]  \nonumber\\
& \quad - \ii \int_t^1 \partial_i \Phi(t - \tau)[\widehat \imath] [|D|^{\frac12}, a] \Psi(\tau)\, d \tau \nonumber\\
& \quad - \ii \int_t^1 \Phi(t - \tau) [|D|^{\frac12}, \partial_i a[\widehat \imath]] \Psi(\tau)\, d \tau \nonumber\\
& \quad - \ii \int_t^1 \Phi(t - \tau) [|D|^{\frac12}, a] \partial_i \Psi(\tau)[\widehat \imath]\, d \tau \label{formula di Psi aggiunto flusso} \, . 
\end{align}
By \eqref{stima norme unite derivate i Phi D destra} applied with $ \beta  = 0 $ we get 
\begin{equation}\label{stima d i Phi Lemma aggiunto}
\| \partial_i \Phi(t)[\widehat \imath] h \|_s \leq_s \| \partial_i a[\widehat \imath] \|_{s + s_0 + \frac12}  \|  h\|_{s + \frac12}\, . 
\end{equation}
Moreover by \eqref{stima commutator parte astratta}
\begin{equation}\label{stima commutatore a D 12 lemma aggiunto}
\norma [|D|^{\frac12}, a] \norma_{- \frac12, s, 0} \stackrel{}{\leq_s} \| a \|_{s + \frac52}\,, \quad \norma [|D|^{\frac12}, \partial_i a[\widehat \imath]] \norma_{- \frac12, s, 0} \stackrel{}{\leq_s} \| \partial_i a[\widehat \imath] \|_{s + \frac52}\,.
\end{equation}
Then for all $ t \in [0, 1] $,  by \eqref{stima d i Phi Lemma aggiunto}, 
\eqref{stima flusso norme unite}, \eqref{piccolezza stime derivata i aggiunto flusso}, 
\begin{align}
\| \partial_i (\Phi(t) \Phi(1)^{- 1})[\widehat \imath] h \|_s \leq_s  \| \partial_i a[\widehat \imath] \|_{s + s_0 + \frac12}  \|  h\|_{s + \frac12}  \label{stima derivata i aggiunto 1}
\end{align}
and for all $ t, \tau \in [0, 1] $, by \eqref{argomento induttivo Psi aggiunto} (applied for $ k = 0 $), 
\eqref{stima d i Phi Lemma aggiunto}, \eqref{stima norme unite derivate i Phi D destra}, \eqref{stima commutatore a D 12 lemma aggiunto}, \eqref{stima flusso norme unite} and \eqref{piccolezza stime derivata i aggiunto flusso} we get, for any $ s \in [s_0, s_1 ] $, 
\begin{align}
& \| \partial_i \Phi(t - \tau)[\widehat \imath] [|D|^{\frac12}, a] \Psi(\tau) h \|_s\,,\, \| \Phi(t - \tau) [|D|^{\frac12}, \partial_i a[\widehat \imath]] \Psi(\tau) h \|_s \nonumber\\
&  \leq_s \| \partial_i a[\widehat \imath] \|_{s + s_0 + \frac12}  \|  h\|_{s + \frac12}\,, \label{stima derivata i aggiunto 2}\\
& \| \Phi(t - \tau) [|D|^{\frac12}, a] \partial_i \Psi(\tau)[\widehat \imath] h \|_s \nonumber\\
& \leq_s \|  a\|_{s + \frac52} \| \partial_i \Psi(\tau)[\widehat \imath] h \|_s \leq_s \delta (s_1)  \| \partial_i \Psi(\tau)[\widehat \imath] h \|_s\,. \label{stima derivata i aggiunto 3}
\end{align}
Therefore \eqref{formula di Psi aggiunto flusso}, \eqref{stima derivata i aggiunto 1}, \eqref{stima derivata i aggiunto 2}, \eqref{stima derivata i aggiunto 3} imply, for all $ s \in [s_0, s_1 ] $,  
$$
\sup_{t \in [0, 1]} \| \partial_i \Psi(t)[\widehat \imath] h \|_s \leq_s 
\| \partial_i a[\widehat \imath ] \|_{s + s_0 + \frac12} \| h \|_{s + \frac12} + 
\delta (s_1) \sup_{t \in [0, 1]}\| \partial_i \Psi(t)[\widehat \imath] h \|_s 
$$
and therefore, taking $ \delta (s_1) $ small, 
$ \sup_{t \in [0, 1]} \| \partial_i \Psi(t)[\widehat \imath] h \|_s \leq_s \| \partial_i a[\widehat \imath ] \|_{s + s_0 + \frac12} \| h \|_{s + \frac12} $, proving \eqref{derivata i aggiunto}. If $h \in H^{s + \frac12}$ and $\widehat \imath \in {\mathcal Y}^{s + s_0 +\frac12 + \sigma}$, then the estimate follows by density. 
\end{proof}

\backmatter

\bibliographystyle{amsalpha}

\begin{thebibliography}{12}

\bibitem{AB}
Alazard T.,  Baldi P.
\newblock Gravity capillary standing water waves.
\newblock  {\em Arch. Rat. Mech.  Anal}, 217, 3, 741-830,  2015.

\bibitem{ABZ1}
Alazard T., Burq N., Zuily C.
\newblock On the water-wave equations with surface tension.
\newblock {\em Duke Math. J.}, 158, 413-499, 2011.

\bibitem{AlDe}
Alazard T., Delort J-M.
\newblock Sobolev estimates for two dimensional gravity water waves.
\newblock {\em Ast\'erisque}, 374, viii + 241, 2015. 


\bibitem{AlDe1}
Alazard T., Delort J-M. 
\newblock  Global solutions and asymptotic behavior for two dimensional gravity water waves.
\newblock {\em Ann. Sci. \'Ec. Norm. Sup\'er.}, 48, no. 5, 1149-1238, 2015. 

\bibitem{Baldi}  Baldi P. \newblock {\em Private Communication}. 

\bibitem{Baldi-Benj-Ono} Baldi P. 
\newblock  Periodic solutions of fully nonlinear autonomous equations of Benjamin-Ono type. 
\newblock {\em Ann. I. H. Poincar\'e (C) Anal. Non Lin\'eaire},  30, no.\,1, 33-77, 2013. 

\bibitem{BBMLincei} Baldi P., Berti M., Montalto R. 
\emph{A note on KAM theory for quasi-linear and fully nonlinear KdV}, 
Rend. Lincei Mat. Appl. 24, 437-450, 2013 

\bibitem{BBM-Airy}
Baldi P., Berti M., Montalto R.
\newblock K{AM} for quasi-linear and fully nonlinear forced perturbations of  {A}iry equation.
\newblock {\em Math. Annalen}, 359, 1-2, 471-536, 2014.
  
  \bibitem{BBMCRAS} Baldi P., Berti M., Montalto R., 
\emph{KAM for quasi-linear KdV},  C. R. Acad. Sci. Paris, Ser. I 352, 603-607, 2014.
  
  \bibitem{BBM-auto}
Baldi P., Berti M.,  Montalto R.
\newblock K{AM} for autonomous quasi-linear perturbations of {K}d{V}.
\newblock {\em Ann. I. H. Poincar\'e (C) Anal. Non Lin\'eaire}, AN 33, 1589-1638, 2016.

\bibitem{BBM-mKdV}
Baldi P., Berti M., Montalto R.
\newblock K{AM} for autonomous quasi-linear perturbations of m{K}d{V}.
 \newblock {\em Bollettino Unione Matematica Italiana},   9:143-188, 2016.



\bibitem{BaBM}
Bambusi D., Berti M., Magistrelli E.
\newblock Degenerate KAM theory for partial differential equations.
\newblock {\em Journal Diff. Equations}, 250, 8, 3379-3397, 2011.


\bibitem{BBi10}  Berti M., Biasco L.
\newblock  Branching of Cantor manifolds of elliptic tori and applications to  PDEs. 
\newblock  {\em Comm. Math. Phys.}, 305, 3,  741-796, 2011.

\bibitem{Berti-Biasco-Procesi-Ham-DNLW}
Berti M., Biasco L., Procesi M.
\newblock K{AM} theory for the {H}amiltonian derivative wave equation.
\newblock {\em Ann. Sci. \'Ec. Norm. Sup\'er. (4)}, 46(2):301-373,  2013.

\bibitem{Berti-Biasco-Procesi-rev-DNLW}
Berti M., Biasco L., Procesi M.
\newblock K{AM} for {R}eversible {D}erivative {W}ave {E}quations.
\newblock {\em Arch. Rat. Mech. Anal.}, 212(3):905-955, 2014.

\bibitem{BB06} Berti M., Bolle P. 
\newblock  Cantor families of periodic solutions for completely resonant nonlinear wave.
equations. 
\newblock {\em Duke Mathematical Journal}, 134, issue 2, 359-419, 2006.


\bibitem{BB13} Berti M., Bolle P. 
\newblock  A Nash-Moser approach to KAM theory. 
\newblock  {\em Fields Institute Communications}, 255-284,
special volume ``Hamiltonian PDEs and Applications'', 2015. 



\bibitem{BB14} Berti M., Bolle P. 
\newblock  
Quasi-periodic solutions for nonlinear wave equations with a multiplicative potential, 
\newblock  in preparation.

\bibitem{BCP} Berti M., Corsi L., Procesi M. 
\newblock 
An Abstract NashÐMoser Theorem and Quasi-Periodic Solutions for NLW and NLS on Compact Lie Groups and Homogeneous Manifolds.
\newblock {\em Comm. Math. Phys.} 334, no. 3, 1413-1454, 2015.

\bibitem{BM16} 
Berti, M., Montalto, R. 
\newblock Quasi-periodic water waves.
\newblock  {\em J. Fixed Point Theory Appl.} (2016), 
doi:10.1007/s11784-016-0375-z \, . 


\bibitem{B5}  Bourgain J.
\newblock
Green's function estimates for lattice Schr\"odinger 
operators and applications.
\newblock {\em Annals of Mathematics Studies} 158, 
Princeton University Press, Princeton, 2005.


\bibitem{CN}
Craig  W., Nicholls D.
\newblock Travelling two and three dimensional capillary gravity water waves.
\newblock {\em SIAM J. Math. Anal.}, 32(2):323--359 (electronic), 2000.



\bibitem{CrSu}
Craig W., Sulem C.
\newblock Numerical simulation of gravity waves.
\newblock {\em J. Comput. Phys.}, 108(1):73-83, 1993.

\bibitem{CS15}
Craig W., Sulem C.
\newblock  Normal form transformations for
capillary-gravity water waves. 
\newblock {\em Field Institute Communications}, 73-110, special volume ``Hamiltonian PDEs and Applications'', 2015. 

\bibitem{Craig-Worfolk}
Craig W., Worfolk P.
\newblock An integrable normal form for water waves in infinite depth.
\newblock { \em Phys. D}, 84, no. 3-4, 513-531, 1995. 

\bibitem{DS}
Delort, J.-M., Szeftel, J.
\newblock  Long-time existence for small data nonlinear Klein-Gordon equations on tori and spheres. 
\newblock  {\em Int. Math. Res. Not.} 2004, no. 37, 1897-1966. 

\bibitem{EK} Eliasson L.H., Kuksin S..
\newblock  KAM for non-linear Schr\"odinger equation. 
\newblock  {\it Annals of Math.},
172, 371-435, 2010.


\bibitem{HF}
Fejoz J. 
\newblock D\'emonstration du th\'eor\'eme d' Arnold sur la stabilit\'e du syst\'eme plan\'etaire (d' apr\'es Herman).
\newblock {\em Ergodic Theory Dynam. Systems} 24 (5), 1521-1582, 2004.

\bibitem{FP} Feola R., Procesi M.
\newblock
Quasi-periodic solutions for fully nonlinear forced
reversible Schr\"odinger equations. 
\newblock {\em J. Diff. Eq.}, 259, no. 7, 3389-3447, 2015.


\bibitem{Ho1}
H{\"o}rmander L. 
\newblock The analysis of linear partial differential operators {III}.  
\newblock {\em Springer-Verlag}, Berlin, 1990.


\bibitem{IP-SW2}
Iooss G., Plotnikov P.
\newblock Existence of multimodal standing gravity waves.
\newblock { \em J. Math. Fluid Mech.}, 7, 349--364, 2005.

\bibitem{IP-SW1}
Iooss G., Plotnikov P.
\newblock Multimodal standing gravity waves: a completely resonant system.
\newblock { \em J. Math. Fluid Mech.}, 7(suppl. 1), 110-126, 2005.

\bibitem{IP-Mem-2009}
Iooss G.,  Plotnikov P.
\newblock Small divisor problem in the theory of three-dimensional water   gravity waves.
\newblock {\em Mem. Amer. Math. Soc.}, 200(940):viii+128, 2009.

\bibitem{IP2}
Iooss G.,  Plotnikov P.
\newblock Asymmetrical tridimensional travelling gravity waves.
\newblock { \em Arch. Rat. Mech. Anal.}, 200(3):789-880, 2011.

\bibitem{IPT}
Iooss G., Plotnikov P., Toland J.
\newblock Standing waves on an infinitely deep perfect fluid under gravity.
\newblock { \em Arch. Rat. Mech. Anal.}, 177(3):367--478, 2005.


\bibitem{LannesLivre}
Lannes D.
\newblock The water waves problem: mathematical analysis and asymptotics.
\newblock { \em Mathematical Surveys and Monographs}, 188, 2013.

\bibitem{LC}
Levi-Civita  T.
\newblock D\'etermination rigoureuse des ondes permanentes d' ampleur finie. 
\newblock {\em Math. Ann.}, 93 , pp. 264-314, 1925.

\bibitem{Liu-Yuan}
Liu J., Yuan X.
\newblock A {KAM} theorem for {H}amiltonian partial differential equations with   unbounded perturbations.
\newblock {\em Comm. Math. Phys.}, 307(3), 629--673, 2011.


\bibitem{KaP} 
Kappeler T., P\"{o}schel J.
\newblock KAM and KdV, 
\newblock {\em Springer}, 2003.

\bibitem{K1} Kuksin S., {\it Hamiltonian perturbations
of infinite-dimensional linear systems with imaginary spectrum},
Funktsional Anal. i Prilozhen. 2, 22-37, 95, 1987.


\bibitem{Kuksin-Oxford}
Kuksin S.
\newblock { Analysis of {H}amiltonian {PDE}s}, volume~19 of { Oxford
  Lecture Series in Mathematics and its Applications}.
\newblock {\em Oxford University Press}, Oxford, 2000.


\bibitem{Met}
M\'etivier G.
\newblock 
Para-differential Calculus and Applications to the Cauchy Problem for Nonlinear Systems.
\newblock 
{\em Pubblicazioni Scuola Normale Pisa}, 5, 2008. 

\bibitem{riccardo-kirchhoff} Montalto R.
\newblock {Quasi-periodic solutions of forced Kirchhoff equation}.
\newblock {\em NoDEA, Nonlinear Differ. Equ. Appl.} 24:9, 2017.  DOI: 10.1007/s00030-017-0432-3.

\bibitem{PlTo}
Plotnikov P.,  Toland J.
\newblock Nash-{M}oser theory for standing water waves.
\newblock { \em Arch. Rat. Mech. Anal.}, 159(1):1--83, 2001.



\bibitem{Po82} 
\newblock  P\"oschel J. Integrability of Hamiltonian systems on Cantor sets. 
\newblock  {\em Comm. Pure Applied. Math.}, XXXV,  653-695, 1982. 


\bibitem{Po2} P\"oschel J.,
{\it A KAM-Theorem for some nonlinear partial differential equations},  Ann. Scuola Norm. Sup.
Pisa Cl. Sci.(4), 23, 119-148, 1996.


\bibitem{Py}
Pyartli. A.S.  
\newblock  Approximations diophantiennes sur les sous-vari\'et\'es de lÕespace
euclidien (en russe).
\newblock
{\em Funkcional. Anal. i Prilozen.} 3 (1969) 59-62 (trad. anglaise
{\em Functional Anal. Appl.} 3 (1969), 303-306.



\bibitem{Ru1}
R\"ussmann H.
\newblock Invariant tori in non-degenerate nearly integrable Hamiltonian systems. 
\newblock {\em Regul. Chaotic Dyn.} 6 (2), 119-204, 2001.

\bibitem{SV}
Saranen J., Vainikko G. 
\newblock Periodic Integral and Pseudodifferential Equations with Numerical Approximation. 
\newblock {\em Springer Monographs in Mathematics}, 2002.


\bibitem{Tay}  Taylor M. E. 
\newblock
Pseudodifferential Operators and Nonlinear PDEs, 
\newblock {\it Progress in Mathematics}, Birkh\"auser, 1991.


\bibitem{Zakharov1968}
Zakharov V.
\newblock Stability of periodic waves of finite amplitude on the surface of a   deep fluid.
\newblock { \em Journal of Applied Mechanics and Technical Physics},  9(2):190--194, 1968.

\bibitem{Z1} Zehnder E., {\it Generalized implicit function theorems with applications to some small divisors problems
I-II}, Comm. Pure Appl. Math. 28 (1975), 91-140, and 29 (1976), 49-113.

\bibitem{Zhang-Gao-Yuan}
Zhang J., Gao M.,  Yuan X.
\newblock K{AM} tori for reversible partial differential equations.
\newblock {\em Nonlinearity}, 24(4):1189-1228, 2011.

\end{thebibliography}
{\footnotesize

\def\cprime{$'$}

\end{document}